\documentclass[numbers=enddot,12pt,final,onecolumn,notitlepage]{scrartcl}%
\usepackage[headsepline,footsepline,manualmark]{scrlayer-scrpage}
\usepackage[all,cmtip]{xy}
\usepackage{amsfonts}
\usepackage{amssymb}
\usepackage{framed}
\usepackage{amsmath}
\usepackage{comment}
\usepackage{needspace}
\usepackage{color}
\usepackage[breaklinks]{hyperref}
\usepackage[sc]{mathpazo}
\usepackage[T1]{fontenc}
\usepackage{amsthm}
\usepackage{tikz}
\providecommand{\U}[1]{\protect\rule{.1in}{.1in}}
\theoremstyle{definition}
\newtheorem{theo}{Theorem}[section]
\newenvironment{theorem}[1][]
{\begin{theo}[#1]\begin{leftbar}}
{\end{leftbar}\end{theo}}
\newtheorem{lem}[theo]{Lemma}
\newenvironment{lemma}[1][]
{\begin{lem}[#1]\begin{leftbar}}
{\end{leftbar}\end{lem}}
\newtheorem{prop}[theo]{Proposition}
\newenvironment{proposition}[1][]
{\begin{prop}[#1]\begin{leftbar}}
{\end{leftbar}\end{prop}}
\newtheorem{defi}[theo]{Definition}
\newenvironment{definition}[1][]
{\begin{defi}[#1]\begin{leftbar}}
{\end{leftbar}\end{defi}}
\newtheorem{remk}[theo]{Remark}
\newenvironment{remark}[1][]
{\begin{remk}[#1]\begin{leftbar}}
{\end{leftbar}\end{remk}}
\newtheorem{coro}[theo]{Corollary}
\newenvironment{corollary}[1][]
{\begin{coro}[#1]\begin{leftbar}}
{\end{leftbar}\end{coro}}
\newtheorem{conv}[theo]{Convention}

\newtheorem{quest}[theo]{TODO}

\newtheorem{warn}[theo]{Warning}

\newtheorem{conj}[theo]{Conjecture}

\newtheorem{exmp}[theo]{Example}
\newenvironment{example}[1][]
{\begin{exmp}[#1]\begin{leftbar}}
{\end{leftbar}\end{exmp}}
\newenvironment{statement}{\begin{quote}}{\end{quote}}
\newenvironment{verlong}{}{}
\newenvironment{vershort}{}{}

\excludecomment{verlong}
\includecomment{vershort}
\excludecomment{noncompile}

\let\sumnonlimits\sum
\let\prodnonlimits\prod
\let\bigcapnonlimits\bigcap
\let\bigcupnonlimits\bigcup
\renewcommand{\sum}{\sumnonlimits\limits}
\renewcommand{\prod}{\prodnonlimits\limits}
\renewcommand{\bigcap}{\bigcapnonlimits\limits}
\renewcommand{\bigcup}{\bigcupnonlimits\limits}
\ihead{Note on NBC sets and chromatic polynomial}
\ohead{\today}
\begin{document}

\title{Generalized Whitney formulas for broken circuits in ambigraphs and
matroids\thanks{This article was formerly titled \textquotedblleft A note on
non-broken-circuit sets and the chromatic polynomial\textquotedblright.}}
\author{Darij Grinberg}
\date{version 2.1, \today}
\maketitle

\begin{abstract}
\textbf{Abstract.} We explore several generalizations of Whitney's theorem --
a classical formula for the chromatic polynomial of a graph. Following
Stanley, we replace the chromatic polynomial by the chromatic symmetric
function. Following Dohmen and Trinks, we exclude not all but only an
(arbitrarily selected) set of broken circuits, or even weigh these broken
circuits with weight monomials instead of excluding them. Following Crew and
Spirkl, we put weights on the vertices of the graph. Following Gebhard and
Sagan, we lift the chromatic symmetric function to noncommuting variables. In
addition, we replace the graph by an \textquotedblleft
ambigraph\textquotedblright, an apparently new concept that includes both
hypergraphs and multigraphs as particular cases.

We show that Whitney's formula endures all these generalizations, and a fairly
simple sign-reversing involution can be used to prove it in each setting.
Furthermore, if we restrict ourselves to the chromatic polynomial, then the
graph can be replaced by a matroid.

We discuss an application to transitive digraphs (i.e., posets), and reprove
an alternating-sum identity by Dahlberg and van Willigenburg.

\end{abstract}
\tableofcontents

\section*{***}

The purpose of this paper is to demonstrate several generalizations of
Whitney's Broken-Circuit theorem \cite[\S 7]{Whitne32} -- a classical formula
for the chromatic polynomial of a graph $\left(  V,E\right)  $ as an
alternating sum over subsets of $E$ that contain no broken circuits. We shall
generalize this formula in the following directions:

\begin{itemize}
\item Instead of summing over the sets that contain no broken circuits, we can
sum over the sets that are \textquotedblleft$\mathfrak{K}$%
-free\textquotedblright\ (i.e., contain no element of $\mathfrak{K}$ as a
subset), where $\mathfrak{K}$ is some fixed set of broken circuits (in
particular, $\mathfrak{K}$ can be $\varnothing$, yielding another well-known
formula for the chromatic polynomial). In other words, instead of excluding
all broken circuits, we can choose to exclude any given set of broken circuits.

This generalization has already been proposed by Dohmen and Trinks in
\cite[\S 3.1]{DohTri14}; however, we give a new and self-contained proof that
does not rely on Whitney's original formula.

\item Even more generally, instead of summing over $\mathfrak{K}$-free
subsets, we can form a weighted sum over all subsets, where the weight depends
on the broken circuits contained in the subset.

\item We can replace the graph by an \emph{ambigraph}: a more general notion
in which the edges are replaced by packages of edges (\textquotedblleft
edgeries\textquotedblright), and a proper coloring has to leave at least one
edge in each such package dichromatic (i.e., color its two endpoints
differently). The concept of ambigraph generalizes both multigraphs and
hypergraphs. We will discuss this concept in Sections \ref{sec.ambi} and
\ref{sec.weigh}.

\item Analogous (and more general) results hold for Stanley's chromatic
symmetric functions \cite{Stanley-chsym} along with two of their more recent
variants: the weighted chromatic symmetric functions of Crew and Spirkl
\cite{CreSpi19} and the noncommutative chromatic symmetric functions of
Gebhard and Sagan \cite{GebSag01}. The latter variants will be studied (and
generalized to ambigraphs) in Section \ref{sec.weigh}.

\item Analogous (and more general) results hold for matroids instead of
graphs. These will be discussed in Section \ref{sec.matroid}.
\end{itemize}

Note that, to my knowledge, the last two generalizations cannot be combined:
Unlike graphs, matroids do not seem to have a well-defined notion of a
chromatic symmetric function.

We will explore these generalizations in the work that follows. We shall also
use them to prove an apparently new formula for the chromatic polynomial of a
graph obtained from a transitive digraph by forgetting the orientations of the
edges (Proposition \ref{prop.digraph.2pf-chrom}). This latter formula was
suggested to me as a conjecture by Alexander Postnikov, during a discussion on
hyperplane arrangements on a space with a bilinear form; it is this formula
which gave rise to this whole paper. The topic of hyperplane arrangements,
however, will not be broached here.

As a further application, we will generalize and reprove an alternating-sum
identity for chromatic polynomials found by Dahlberg and van Willigenburg
(Section \ref{sec.alt-sum}), as well as an analogous identity for
characteristic polynomials of matroids (Subsection
\ref{subsec.matroid.alt-sum}).

\subsection*{Acknowledgments}

I thank Alexander Postnikov and Richard P. Stanley for discussions on
hyperplane arrangements that led to the results described here.

\section{\label{sec.defs-and-main}Definitions and a main result}

\subsection{Graphs and colorings}

We begin by recalling some basic features of finite graphs. Let us start with
the definition of a graph that we shall be using:

\begin{definition}
\label{def.graph}\textbf{(a)} If $V$ is any set, then $\mathcal{P}\left(
V\right)  $ will denote the powerset of $V$. This is the set of all subsets of
$V$.

\textbf{(b)} If $V$ is any set, then $\dbinom{V}{2}$ will denote the set of
all $2$-element subsets of $V$. In other words, if $V$ is any set, then we set%
\begin{align*}
\dbinom{V}{2}  &  =\left\{  S\in\mathcal{P}\left(  V\right)  \ \mid
\ \left\vert S\right\vert =2\right\} \\
&  =\left\{  \left\{  s,t\right\}  \ \mid\ s\in V,\ t\in V,\ s\neq t\right\}
.
\end{align*}

\textbf{(c)} A \emph{graph} means a pair $\left(  V,E\right)  $, where $V$ is
a set, and where $E$ is a subset of $\dbinom{V}{2}$. A graph $\left(
V,E\right)  $ is said to be \emph{finite} if the set $V$ is finite. If
$G=\left(  V,E\right)  $ is a graph, then the elements of $V$ are called the
\emph{vertices} of the graph $G$, while the elements of $E$ are called the
\emph{edges} of the graph $G$. If $e$ is an edge of a graph $G$, then the two
elements of $e$ are called the \emph{endpoints} of the edge $e$. If
$e=\left\{  s,t\right\}  $ is an edge of a graph $G$, then we say that the
edge $e$ \emph{connects the vertices }$s$ \emph{and }$t$ of $G$.
\end{definition}

Comparing our definition of a graph with some of the other definitions used in
the literature, we thus observe that our graphs are undirected (i.e., their
edges are sets, not pairs), loopless (i.e., the two endpoints of an edge must
always be distinct), edge-unlabelled (i.e., their edges are just $2$-element
sets of vertices, rather than objects with \textquotedblleft their own
identity\textquotedblright), and do not have multiple edges (or, more
precisely, there is no notion of several edges connecting two vertices, since
the edges form a set, not a multiset, and do not have labels). Such graphs are
commonly known as \emph{simple graphs}.

\begin{definition}
\label{def.coloring}Let $G=\left(  V,E\right)  $ be a graph. Let $X$ be a set.

\textbf{(a)} An $X$\emph{-coloring} of $G$ is defined to mean a map
$V\rightarrow X$.

\textbf{(b)} An $X$-coloring $f$ of $G$ is said to be \emph{proper} if every
edge $\left\{  s,t\right\}  \in E$ satisfies $f\left(  s\right)  \neq f\left(
t\right)  $.
\end{definition}

If $f$ is an $X$-coloring of a graph $G=\left(  V,E\right)  $, then the value
$f\left(  v\right)  $ for a given vertex $v\in V$ is called the \emph{color}
of this vertex $v$ under the coloring $f$. We shall not use this terminology
here, but we are mentioning it since it allows for a rather intuitive mental
model and explains the word \textquotedblleft coloring\textquotedblright. An
$X$-coloring of $G$ is then proper if and only if each edge of $G$ has two
endpoints of different colors.

\subsection{Symmetric functions}

We shall now briefly introduce the notion of symmetric functions. We shall not
use any nontrivial results about symmetric functions; we will merely need some
notations.\footnote{For an introduction to symmetric functions, see any of
\cite[Chapter 7]{Stanley-EC2}, \cite[Chapter 9]{Martin22} and \cite[Chapter
2]{Reiner} (and a variety of other texts).}

In the following, $\mathbb{N}$ means the set $\left\{  0,1,2,\ldots\right\}
$. Also, $\mathbb{N}_{+}$ shall mean the set $\left\{  1,2,3,\ldots\right\}  $.

A \emph{partition} will mean a sequence $\left(  \lambda_{1},\lambda
_{2},\lambda_{3},\ldots\right)  \in\mathbb{N}^{\infty}$ of nonnegative
integers such that $\lambda_{1}\geq\lambda_{2}\geq\lambda_{3}\geq\cdots$ and
such that all sufficiently high integers $i\geq1$ satisfy $\lambda_{i}=0$. If
$\lambda=\left(  \lambda_{1},\lambda_{2},\lambda_{3},\ldots\right)  $ is a
partition, and if a positive integer $n$ is such that all integers $i\geq n$
satisfy $\lambda_{i}=0$, then we shall identify the partition $\lambda$ with
the finite sequence $\left(  \lambda_{1},\lambda_{2},\ldots,\lambda
_{n-1}\right)  $. Thus, for example, the sequences $\left(  3,1\right)  $ and
$\left(  3,1,0\right)  $ and the partition $\left(  3,1,0,0,0,\ldots\right)  $
are all identified. Every weakly decreasing finite list of positive integers
thus is identified with a unique partition.

Let $\mathbf{k}$ be a commutative ring with unity. We shall keep $\mathbf{k}$
fixed throughout the paper. The reader will not be missing out on anything if
she assumes that $\mathbf{k}=\mathbb{Z}$.

We consider the $\mathbf{k}$-algebra $\mathbf{k}\left[  \left[  x_{1}%
,x_{2},x_{3},\ldots\right]  \right]  $ of (commutative) power series in
countably many distinct indeterminates $x_{1},x_{2},x_{3},\ldots$ over
$\mathbf{k}$. It is a topological $\mathbf{k}$-algebra\footnote{See
\cite[Section 2.6]{Reiner} or \cite[\S 2]{Grinbe16} for the definition of its
topology. This topology makes sure that a sequence $\left(  P_{n}\right)
_{n\in\mathbb{N}}$ of power series converges to some power series $P$ if and
only if, for every monomial $\mathfrak{m}$, all sufficiently high
$n\in\mathbb{N}$ satisfy%
\[
\left(  \text{the }\mathfrak{m}\text{-coefficient of }P_{n}\right)  =\left(
\text{the }\mathfrak{m}\text{-coefficient of }P\right)
\]
(where the meaning of \textquotedblleft sufficiently high\textquotedblright%
\ can depend on the $\mathfrak{m}$).}. A power series $P\in\mathbf{k}\left[
\left[  x_{1},x_{2},x_{3},\ldots\right]  \right]  $ is said to be
\emph{bounded-degree} if there exists an $N\in\mathbb{N}$ such that every
monomial of degree $>N$ appears with coefficient $0$ in $P$. A power series
$P\in\mathbf{k}\left[  \left[  x_{1},x_{2},x_{3},\ldots\right]  \right]  $ is
said to be \emph{symmetric} if and only if $P$ is invariant under any
permutation of the indeterminates. We let $\Lambda$ be the subset of
$\mathbf{k}\left[  \left[  x_{1},x_{2},x_{3},\ldots\right]  \right]  $
consisting of all symmetric bounded-degree power series $P\in\mathbf{k}\left[
\left[  x_{1},x_{2},x_{3},\ldots\right]  \right]  $. This subset $\Lambda$ is
a $\mathbf{k}$-subalgebra of $\mathbf{k}\left[  \left[  x_{1},x_{2}%
,x_{3},\ldots\right]  \right]  $, and is called the $\mathbf{k}$\emph{-algebra
of symmetric functions} over $\mathbf{k}$.

We shall now define the few families of symmetric functions that we will be
concerned with in this work. The first are the \emph{power-sum symmetric
functions}:

\begin{definition}
\label{def.powersum}Let $n$ be a positive integer. We define a power series
$p_{n}\in\mathbf{k}\left[  \left[  x_{1},x_{2},x_{3},\ldots\right]  \right]  $
by%
\begin{equation}
p_{n}=x_{1}^{n}+x_{2}^{n}+x_{3}^{n}+\cdots=\sum_{j\geq1}x_{j}^{n}.
\label{eq.def.powersum.pn}%
\end{equation}
This power series $p_{n}$ lies in $\Lambda$, and is called the $n$\emph{-th
power-sum symmetric function}.

We also set $p_{0}=1\in\Lambda$. Thus, $p_{n}$ is defined not only for all
positive integers $n$, but also for all $n\in\mathbb{N}$.
\end{definition}

\begin{definition}
\label{def.powersum2}Let $\lambda=\left(  \lambda_{1},\lambda_{2},\lambda
_{3},\ldots\right)  $ be a partition. We define a power series $p_{\lambda}%
\in\mathbf{k}\left[  \left[  x_{1},x_{2},x_{3},\ldots\right]  \right]  $ by%
\[
p_{\lambda}=\prod_{i\geq1}p_{\lambda_{i}}.
\]
This is well-defined, because the infinite product $\prod_{i\geq1}%
p_{\lambda_{i}}$ converges (indeed, all but finitely many of its factors are
$1$ (because every sufficiently high integer $i$ satisfies $\lambda_{i}=0$ and
thus $p_{\lambda_{i}}=p_{0}=1$)).
\end{definition}

We notice that every partition $\lambda=\left(  \lambda_{1},\lambda_{2}%
,\ldots,\lambda_{k}\right)  $ (written as a finite list of nonnegative
integers) satisfies%
\begin{equation}
p_{\lambda}=p_{\lambda_{1}}p_{\lambda_{2}}\cdots p_{\lambda_{k}}.
\label{eq.def.powersum2.finite-expression}%
\end{equation}

\subsection{Chromatic symmetric functions}

The next symmetric functions we introduce are the actual subject of this work;
they are the \emph{chromatic symmetric functions} and have been introduced by
Stanley in \cite[Definition 2.1]{Stanley-chsym}:

\Needspace{4cm}

\begin{definition}
\label{def.chromsym}Let $G=\left(  V,E\right)  $ be a finite graph.

\textbf{(a)} For every $\mathbb{N}_{+}$-coloring $f:V\rightarrow\mathbb{N}%
_{+}$ of $G$, we let $\mathbf{x}_{f}$ denote the monomial $\prod_{v\in
V}x_{f\left(  v\right)  }$ in the indeterminates $x_{1},x_{2},x_{3},\ldots$.

\textbf{(b)} We define a power series $X_{G}\in\mathbf{k}\left[  \left[
x_{1},x_{2},x_{3},\ldots\right]  \right]  $ by%
\[
X_{G}=\sum_{\substack{f:V\rightarrow\mathbb{N}_{+}\text{ is a}\\\text{proper
}\mathbb{N}_{+}\text{-coloring of }G}}\mathbf{x}_{f}.
\]

This power series $X_{G}$ is called the \emph{chromatic symmetric function} of
$G$.
\end{definition}

We have $X_{G}\in\Lambda$ for every finite graph $G=\left(  V,E\right)  $;
this will follow from Theorem \ref{thm.chromsym.empty} further below (but is
also rather obvious).

We remark that $X_{G}$ is denoted by $\Psi\left[  G\right]  $ in
\cite[\S 7.3.3]{Reiner}.

\subsection{Connected components}

\begin{vershort}
We shall now briefly recall the notion of connected components of a graph.
\end{vershort}

\begin{verlong}
We shall now recall the notion of connected components of a graph. This notion
is a particular case of the notion of a quotient set by an equivalence
relation; thus, we shall briefly recall the latter first.

A \emph{binary relation} on a set $X$ means a subset of $X\times X$. If $\sim$
is a binary relation on a set $X$, and if $x$ and $y$ are two elements of $X$,
then one says that $x$ \emph{is related to }$y$ \emph{by }$\sim$ if and only
if $\left(  x,y\right)  $ belongs to the set $\sim$. An \emph{equivalence
relation} is a binary relation (on a set) that is reflexive, symmetric and transitive.

\begin{definition}
\label{def.relquot}Let $X$ be a set. Let $\sim$ be an equivalence relation on
$X$. We shall write $\sim$ infix (that is, for any $x\in X$ and $y\in X$, we
shall abbreviate \textquotedblleft$\left(  x,y\right)  \in\left.  \sim\right.
$\textquotedblright\ as \textquotedblleft$x\sim y$\textquotedblright). For
every $x\in X$, we let $\left[  x\right]  _{\sim}$ be the subset $\left\{
y\in X\ \mid\ x\sim y\right\}  $ of $X$; this subset $\left[  x\right]
_{\sim}$ is called the $\sim$\emph{-equivalence class} of $x$ (or the
\emph{equivalence class of }$x$ \emph{with respect to the relation }$\sim$). A
$\sim$\emph{-equivalence class} means a subset of $X$ which is the $\sim
$-equivalence class of some $x\in X$. It is well-known that any two $\sim
$-equivalence classes are either identical or disjoint, and that $X$ is the
union of all $\sim$-equivalence classes.

We let $X/\left(  \sim\right)  $ denote the set of all $\sim$-equivalence
classes. This set $X/\left(  \sim\right)  $ is called the \emph{quotient of
the set }$X$ \emph{by the equivalence relation }$\sim$. We define a map
$\pi_{X}:X\rightarrow X/\left(  \sim\right)  $ by setting%
\[
\left(  \pi_{X}\left(  x\right)  =\left[  x\right]  _{\sim}%
\ \ \ \ \ \ \ \ \ \ \text{for every }x\in X\right)  .
\]
Thus, the map $\pi_{X}$ sends every $x\in X$ to its $\sim$-equivalence class.
\end{definition}

Let us introduce a few more notations:

\begin{definition}
\label{def.relquot.maps}Let $X$ and $Y$ be two sets.

\textbf{(a)} Then, $Y^{X}$ denotes the set of all maps $X\rightarrow Y$.

\textbf{(b)} Let $\sim$ be any binary relation on $X$. We shall write $\sim$
infix (that is, for any $x\in X$ and $y\in X$, we shall abbreviate
\textquotedblleft$\left(  x,y\right)  \in\left.  \sim\right.  $%
\textquotedblright\ as \textquotedblleft$x\sim y$\textquotedblright). Then, we
let $Y_{\sim}^{X}$ denote the subset%
\[
\left\{  g\in Y^{X}\ \mid\ g\left(  x\right)  =g\left(  y\right)  \text{ for
any }x\in X\text{ and }y\in X\text{ satisfying }x\sim y\right\}
\]
of $Y^{X}$.
\end{definition}

The map $\pi_{X}$ defined in Definition \ref{def.relquot} has an important
universal property:

\begin{proposition}
\label{prop.relquot.uniprop}Let $X$ be a set. Let $\sim$ be an equivalence
relation on $X$.

\textbf{(a)} The map $\pi_{X}:X\rightarrow X/\left(  \sim\right)  $ is
surjective and belongs to $\left(  X/\left(  \sim\right)  \right)  _{\sim}%
^{X}$.

\textbf{(b)} Let $Y$ be any set. Then, the map%
\[
Y^{X/\left(  \sim\right)  }\rightarrow Y_{\sim}^{X}%
,\ \ \ \ \ \ \ \ \ \ f\mapsto f\circ\pi_{X}%
\]
is a bijection.
\end{proposition}

\begin{proof}
[Proof of Proposition \ref{prop.relquot.uniprop}.]\textbf{(a)} The map
$\pi_{X}$ is surjective\footnote{\textit{Proof.} Let $u\in X/\left(
\sim\right)  $. Thus, $u$ is a $\sim$-equivalence class (since $X/\left(
\sim\right)  $ is the set of all $\sim$-equivalence classes). In other words,
there exists some $x\in X$ such that $u$ is the $\sim$-equivalence class of
$x$ (by the definition of a \textquotedblleft$\sim$-equivalence
class\textquotedblright). Consider this $x$.
\par
We know that $\left[  x\right]  _{\sim}$ is the $\sim$-equivalence class of
$x$. In other words, $\left[  x\right]  _{\sim}$ is $u$ (since $u$ is the
$\sim$-equivalence class of $x$). Thus, $\left[  x\right]  _{\sim}=u$. Now,
the definition of $\pi_{X}$ yields $\pi_{X}\left(  x\right)  =\left[
x\right]  _{\sim}=u$. Thus, $u=\pi_{X}\left(  \underbrace{x}_{\in X}\right)
\in\pi_{X}\left(  X\right)  $.
\par
Now, let us forget that we fixed $u$. We thus have shown that $u\in\pi
_{X}\left(  X\right)  $ for every $u\in X/\left(  \sim\right)  $. In other
words, $X/\left(  \sim\right)  \subseteq\pi_{X}\left(  X\right)  $. In other
words, the map $\pi_{X}$ is surjective. Qed.}.

The definition of $\left(  X/\left(  \sim\right)  \right)  _{\sim}^{X}$ shows
that%
\begin{align}
&  \left(  X/\left(  \sim\right)  \right)  _{\sim}^{X}\nonumber\\
&  =\left\{  g\in\left(  X/\left(  \sim\right)  \right)  ^{X}\ \mid\ g\left(
x\right)  =g\left(  y\right)  \text{ for any }x\in X\text{ and }y\in X\text{
satisfying }x\sim y\right\}  . \label{pf.prop.relquot.uniprop.a.def}%
\end{align}

Now, we claim that
\begin{equation}
\pi_{X}\left(  x\right)  =\pi_{X}\left(  y\right)
\ \ \ \ \ \ \ \ \ \ \text{for any }x\in X\text{ and }y\in X\text{ satisfying
}x\sim y. \label{pf.prop.relquot.uniprop.a.eq}%
\end{equation}

\textit{Proof of (\ref{pf.prop.relquot.uniprop.a.eq}):} Let $x\in X$ and $y\in
X$ be any two elements satisfying $x\sim y$. We shall show that $\pi
_{X}\left(  x\right)  =\pi_{X}\left(  y\right)  $.

Recall that $\pi_{X}\left(  x\right)  =\left[  x\right]  _{\sim}$. Thus,
$\pi_{X}\left(  x\right)  $ is the $\sim$-equivalence class of $x$ (since
$\left[  x\right]  _{\sim}$ is the $\sim$-equivalence class of $x$). The same
argument (applied to $y$ instead of $x$) shows that $\pi_{X}\left(  y\right)
$ is the $\sim$-equivalence class of $y$.

We have $x\sim y$. Thus, the elements $x$ and $y$ of $X$ lie in the same
$\sim$-equivalence class. In other words, the $\sim$-equivalence class of $x$
is the $\sim$-equivalence class of $y$. In other words, $\pi_{X}\left(
x\right)  $ is $\pi_{X}\left(  y\right)  $ (since $\pi_{X}\left(  x\right)  $
is the $\sim$-equivalence class of $x$, and since $\pi_{X}\left(  y\right)  $
is the $\sim$-equivalence class of $y$). In other words, $\pi_{X}\left(
x\right)  =\pi_{X}\left(  y\right)  $. This proves
(\ref{pf.prop.relquot.uniprop.a.eq}).

Now, $\pi_{X}$ is a map $g\in\left(  X/\left(  \sim\right)  \right)  ^{X}$
which satisfies $g\left(  x\right)  =g\left(  y\right)  $ for any $x\in X$ and
$y\in X$ satisfying $x\sim y$ (according to
(\ref{pf.prop.relquot.uniprop.a.eq})). In other words,%
\begin{align*}
\pi_{X}  &  \in\left\{  g\in\left(  X/\left(  \sim\right)  \right)  ^{X}%
\ \mid\ g\left(  x\right)  =g\left(  y\right)  \text{ for any }x\in X\text{
and }y\in X\text{ satisfying }x\sim y\right\} \\
&  =\left(  X/\left(  \sim\right)  \right)  _{\sim}^{X}%
\ \ \ \ \ \ \ \ \ \ \left(  \text{by (\ref{pf.prop.relquot.uniprop.a.def}%
)}\right)  .
\end{align*}
This completes the proof of Proposition \ref{prop.relquot.uniprop}
\textbf{(a)}.

\textbf{(b)} We have%
\begin{equation}
f\circ\pi_{X}\in Y_{\sim}^{X}\ \ \ \ \ \ \ \ \ \ \text{for every }f\in
Y^{X/\left(  \sim\right)  } \label{pf.prop.relquot.uniprop.wd}%
\end{equation}
\footnote{\textit{Proof of (\ref{pf.prop.relquot.uniprop.wd}):} Let $f\in
Y^{X/\left(  \sim\right)  }$. Then, clearly, $f\circ\pi_{X}\in Y^{X}$ (since
$\pi_{X}\in\left(  X/\left(  \sim\right)  \right)  ^{X}$). Now, let $x\in X$
and $y\in X$ be any two elements satisfying $x\sim y$. We shall show that
$\left(  f\circ\pi_{X}\right)  \left(  x\right)  =\left(  f\circ\pi
_{X}\right)  \left(  y\right)  $.
\par
From (\ref{pf.prop.relquot.uniprop.a.eq}), we obtain $\pi_{X}\left(  x\right)
=\pi_{X}\left(  y\right)  $. Now, $\left(  f\circ\pi_{X}\right)  \left(
x\right)  =f\left(  \underbrace{\pi_{X}\left(  x\right)  }_{=\pi_{X}\left(
y\right)  }\right)  =f\left(  \pi_{X}\left(  y\right)  \right)  =\left(
f\circ\pi_{X}\right)  \left(  y\right)  $.
\par
Now, let us forget that we fixed $x$ and $y$. We thus have shown that $\left(
f\circ\pi_{X}\right)  \left(  x\right)  =\left(  f\circ\pi_{X}\right)  \left(
y\right)  $ for any $x\in X$ and $y\in X$ satisfying $x\sim y$. Thus,
$f\circ\pi_{X}$ is a map $g\in Y^{X}$ which satisfies $g\left(  x\right)
=g\left(  y\right)  $ for any $x\in X$ and $y\in X$ satisfying $x\sim y$. In
other words,%
\begin{align*}
f\circ\pi_{X}  &  \in\left\{  g\in Y^{X}\ \mid\ g\left(  x\right)  =g\left(
y\right)  \text{ for any }x\in X\text{ and }y\in X\text{ satisfying }x\sim
y\right\} \\
&  =Y_{\sim}^{X}%
\end{align*}
(because $Y_{\sim}^{X}$ was defined to be $\left\{  g\in Y^{X}\ \mid\ g\left(
x\right)  =g\left(  y\right)  \text{ for any }x\in X\text{ and }y\in X\text{
satisfying }x\sim y\right\}  $). This proves (\ref{pf.prop.relquot.uniprop.wd}%
).}. Hence, we can define a map $\Phi:Y^{X/\left(  \sim\right)  }\rightarrow
Y_{\sim}^{X}$ by%
\begin{equation}
\left(  \Phi\left(  f\right)  =f\circ\pi_{X}\ \ \ \ \ \ \ \ \ \ \text{for
every }f\in Y^{X/\left(  \sim\right)  }\right)  .
\label{pf.prop.relquot.uniprop.Phi}%
\end{equation}
Consider this map $\Phi$. We shall now show that the map $\Phi$ is a bijection.

The map $\Phi$ is injective\footnote{\textit{Proof.} Let $f$ and $g$ be two
elements of $Y^{X/\left(  \sim\right)  }$ such that $\Phi\left(  f\right)
=\Phi\left(  g\right)  $. We shall show that $f=g$.
\par
Let $u\in X/\left(  \sim\right)  $. Thus, $u\in X/\left(  \sim\right)
=\pi_{X}\left(  X\right)  $ (since the map $\pi_{X}$ is surjective). In other
words, there exists some $x\in X$ such that $u=\pi_{X}\left(  x\right)  $.
Consider this $x$.
\par
Now, the definition of $\Phi$ yields $\Phi\left(  f\right)  =f\circ\pi_{X}$.
Hence, $\left(  \underbrace{\Phi\left(  f\right)  }_{=f\circ\pi_{X}}\right)
\left(  x\right)  =\left(  f\circ\pi_{X}\right)  \left(  x\right)  =f\left(
\underbrace{\pi_{X}\left(  x\right)  }_{=u}\right)  =f\left(  u\right)  $. The
same argument (but applied to $g$ instead of $f$) yields $\left(  \Phi\left(
g\right)  \right)  \left(  x\right)  =g\left(  u\right)  $. Now, $f\left(
u\right)  =\left(  \underbrace{\Phi\left(  f\right)  }_{=\Phi\left(  g\right)
}\right)  \left(  x\right)  =\left(  \Phi\left(  g\right)  \right)  \left(
x\right)  =g\left(  u\right)  $.
\par
Let us now forget that we fixed $u$. We thus have shown that $f\left(
u\right)  =g\left(  u\right)  $ for every $u\in X/\left(  \sim\right)  $. In
other words, $f=g$.
\par
Now, let us forget that we fixed $f$ and $g$. We thus have proven that if $f$
and $g$ are two elements of $Y^{X/\left(  \sim\right)  }$ such that
$\Phi\left(  f\right)  =\Phi\left(  g\right)  $, then $f=g$. In other words,
the map $\Phi$ is injective. Qed.}. We shall now show that the map $\Phi$ is surjective.

Indeed, let $h\in Y_{\sim}^{X}$.

We shall now define a map $f\in Y^{X/\left(  \sim\right)  }$ as follows: Let
$u\in X/\left(  \sim\right)  $. Thus, $u\in X/\left(  \sim\right)  =\pi
_{X}\left(  X\right)  $ (since the map $\pi_{X}$ is surjective). Thus, there
exists some $x\in X$ satisfying $u=\pi_{X}\left(  x\right)  $. Pick such an
$x$. Then, $h\left(  x\right)  $ is independent of the choice of
$x$\ \ \ \ \footnote{\textit{Proof.} Let $x_{1}$ and $x_{2}$ be any two $x\in
X$ satisfying $u=\pi_{X}\left(  x\right)  $. We shall show that $h\left(
x_{1}\right)  =h\left(  x_{2}\right)  $.
\par
Indeed, $x_{1}$ is an $x\in X$ satisfying $u=\pi_{X}\left(  x\right)  $. In
other words, $x_{1}$ is an element of $X$ and satisfies $u=\pi_{X}\left(
x_{1}\right)  $. But the definition of $\pi_{X}$ shows that $\pi_{X}\left(
x_{1}\right)  =\left[  x_{1}\right]  _{\sim}$. Thus, $u=\pi_{X}\left(
x_{1}\right)  =\left[  x_{1}\right]  _{\sim}$. In other words, $u$ is $\left[
x_{1}\right]  _{\sim}$. In other words, $u$ is the $\sim$-equivalence class of
$x_{1}$ (since $\left[  x_{1}\right]  _{\sim}$ is the $\sim$-equivalence class
of $x_{1}$ (by the definition of $\left[  x_{1}\right]  _{\sim}$)).
\par
The same argument (applied to $x_{2}$ instead of $x_{1}$) shows that $u$ is
the $\sim$-equivalence class of $x_{2}$. Thus, the $\sim$-equivalence classes
of $x_{1}$ and $x_{2}$ are both $u$. Hence, these two $\sim$-equivalence
classes are identical. In other words, $x_{1}$ and $x_{2}$ belong to the same
$\sim$-equivalence class. In other words, $x_{1}\sim x_{2}$.
\par
But%
\[
h\in Y_{\sim}^{X}=\left\{  g\in Y^{X}\ \mid\ g\left(  x\right)  =g\left(
y\right)  \text{ for any }x\in X\text{ and }y\in X\text{ satisfying }x\sim
y\right\}
\]
(by the definition of $Y_{\sim}^{X}$). In other words, $h$ is a $g\in Y^{X}$
which satisfies $g\left(  x\right)  =g\left(  y\right)  $ for any $x\in X$ and
$y\in X$ satisfying $x\sim y$. In other words, $h$ is an element of $Y^{X}$
and satisfies%
\begin{equation}
h\left(  x\right)  =h\left(  y\right)  \text{ for any }x\in X\text{ and }y\in
X\text{ satisfying }x\sim y. \label{pf.prop.relquot.uniprop.b.surj.fn1.1}%
\end{equation}
\par
Now, we can apply (\ref{pf.prop.relquot.uniprop.b.surj.fn1.1}) to $x=x_{1}$
and $y=x_{2}$ (since $x_{1}\sim x_{2}$). As a result, we obtain $h\left(
x_{1}\right)  =h\left(  x_{2}\right)  $.
\par
Let us now forget that we fixed $x_{1}$ and $x_{2}$. We thus have shown that
if $x_{1}$ and $x_{2}$ are any two $x\in X$ satisfying $u=\pi_{X}\left(
x\right)  $, then $h\left(  x_{1}\right)  =h\left(  x_{2}\right)  $. In other
words, $h\left(  x\right)  $ is independent of the choice of $x$ (when $x$ is
chosen as an element of $X$ satisfying $u=\pi_{X}\left(  x\right)  $). Qed.}.
Hence, we can set $f\left(  u\right)  =h\left(  x\right)  $.

Thus, we have defined a map $f\in Y^{X/\left(  \sim\right)  }$. This map has
the following property: If $u\in X/\left(  \sim\right)  $, and if $x\in X$ is
such that $u=\pi_{X}\left(  x\right)  $, then%
\begin{equation}
f\left(  u\right)  =h\left(  x\right)  \label{pf.prop.relquot.uniprop.surj.1}%
\end{equation}
(because this is how $f\left(  u\right)  $ was defined).

Now, $f\circ\pi_{X}=h$\ \ \ \ \footnote{\textit{Proof.} Every $x\in X$
satisfies $\left(  f\circ\pi_{X}\right)  \left(  x\right)  =f\left(  \pi
_{X}\left(  x\right)  \right)  =h\left(  x\right)  $ (by
(\ref{pf.prop.relquot.uniprop.surj.1}), applied to $u=\pi_{X}\left(  x\right)
$). In other words, $f\circ\pi_{X}=h$, qed.}. The definition of $\Phi$ now
yields $\Phi\left(  f\right)  =f\circ\pi_{X}=h$. Hence, $h=\Phi\left(
\underbrace{f}_{\in Y^{X/\left(  \sim\right)  }}\right)  \in\Phi\left(
Y^{X/\left(  \sim\right)  }\right)  $.

Now, let us forget that we fixed $h$. We thus have proven that $h\in
\Phi\left(  Y^{X/\left(  \sim\right)  }\right)  $ for every $h\in Y_{\sim}%
^{X}$. In other words, the map $\Phi$ is surjective.

Thus, we know that the map $\Phi$ is injective and surjective. In other words,
$\Phi$ is bijective. In other words, $\Phi$ is a bijection. Since $\Phi$ is
the map
\[
Y^{X/\left(  \sim\right)  }\rightarrow Y_{\sim}^{X}%
,\ \ \ \ \ \ \ \ \ \ f\mapsto f\circ\pi_{X}%
\]
(because $\Phi$ is a map $Y^{X/\left(  \sim\right)  }\rightarrow Y_{\sim}^{X}$
and satisfies (\ref{pf.prop.relquot.uniprop.Phi})), this shows that the map
\[
Y^{X/\left(  \sim\right)  }\rightarrow Y_{\sim}^{X}%
,\ \ \ \ \ \ \ \ \ \ f\mapsto f\circ\pi_{X}%
\]
is a bijection. This proves Proposition \ref{prop.relquot.uniprop}
\textbf{(b)}.
\end{proof}

We shall now recall what connected components are:
\end{verlong}

\begin{definition}
\label{def.walk}Let $G=\left(  V,E\right)  $ be a graph. Let $u$ and $v$ be
two elements of $V$ (that is, two vertices of $G$). A \emph{walk} from $u$ to
$v$ in $G$ will mean a sequence $\left(  w_{0},w_{1},\ldots,w_{k}\right)  $ of
elements of $V$ such that $w_{0}=u$ and $w_{k}=v$ and%
\[
\left(  \left\{  w_{i},w_{i+1}\right\}  \in E\ \ \ \ \ \ \ \ \ \ \text{for
every }i\in\left\{  0,1,\ldots,k-1\right\}  \right)  .
\]
We say that $u$ and $v$ are \emph{connected (in }$G$\emph{)} if there exists a
walk from $u$ to $v$ in $G$.
\end{definition}

\begin{definition}
\label{def.connectedness}Let $G=\left(  V,E\right)  $ be a graph.

\textbf{(a)} We define a binary relation $\sim_{G}$ (written infix) on the set
$V$ as follows: Given $u\in V$ and $v\in V$, we set $u\sim_{G}v$ if and only
if $u$ and $v$ are connected (in $G$). It is well-known that this relation
$\sim_{G}$ is an equivalence relation. The $\sim_{G}$-equivalence classes are
called the \emph{connected components} of $G$.

\textbf{(b)} Assume that the graph $G$ is finite. We let $\lambda\left(
G\right)  $ denote the list of the sizes of all connected components of $G$,
in weakly decreasing order. (Each connected component should contribute only
one entry to the list.) We view $\lambda\left(  G\right)  $ as a partition
(since $\lambda\left(  G\right)  $ is a weakly decreasing finite list of
positive integers).
\end{definition}

Now, we can state a formula for chromatic symmetric functions:

\begin{theorem}
\label{thm.chromsym.empty}Let $G=\left(  V,E\right)  $ be a finite graph.
Then,%
\[
X_{G}=\sum_{F\subseteq E}\left(  -1\right)  ^{\left\vert F\right\vert
}p_{\lambda\left(  V,F\right)  }.
\]
(Here, of course, the pair $\left(  V,F\right)  $ is regarded as a graph, and
the expression $\lambda\left(  V,F\right)  $ is understood according to
Definition \ref{def.connectedness} \textbf{(b)}.)
\end{theorem}

This theorem is not new; it appears, e.g., in \cite[Theorem 2.5]%
{Stanley-chsym}. We shall show a far-reaching generalization of it (Theorem
\ref{thm.chromsym.varis}) soon.

\subsection{Circuits and broken circuits}

Let us now define the notions of cycles and circuits of a graph:

\begin{definition}
\label{def.cycle}Let $G=\left(  V,E\right)  $ be a graph. A \emph{cycle} of
$G$ denotes a list $\left(  v_{1},v_{2},\ldots,v_{m+1}\right)  $ of elements
of $V$ with the following properties:

\begin{itemize}
\item We have $m>2$.

\item We have $v_{m+1}=v_{1}$.

\item The vertices $v_{1},v_{2},\ldots,v_{m}$ are pairwise distinct.

\item We have $\left\{  v_{i},v_{i+1}\right\}  \in E$ for every $i\in\left\{
1,2,\ldots,m\right\}  $.
\end{itemize}

If $\left(  v_{1},v_{2},\ldots,v_{m+1}\right)  $ is a cycle of $G$, then the
set $\left\{  \left\{  v_{1},v_{2}\right\}  ,\left\{  v_{2},v_{3}\right\}
,\ldots,\left\{  v_{m},v_{m+1}\right\}  \right\}  $ is called a \emph{circuit}
of $G$.
\end{definition}

For instance, if $\left(  1,3,5,7,1\right)  $ is a cycle of a graph $G$, then
the corresponding circuit is $\left\{  \left\{  1,3\right\}  ,\ \left\{
3,5\right\}  ,\ \left\{  5,7\right\}  ,\ \left\{  7,1\right\}  \right\}  $.

\begin{definition}
\label{def.BC}Let $G=\left(  V,E\right)  $ be a graph. Let $X$ be a totally
ordered set. Let $\ell:E\rightarrow X$ be a function. We shall refer to $\ell$
as the \emph{labeling function}. For every edge $e$ of $G$, we shall refer to
$\ell\left(  e\right)  $ as the \emph{label} of $e$.

A \emph{broken circuit} of $G$ means a subset of $E$ having the form
$C\setminus\left\{  e\right\}  $, where $C$ is a circuit of $G$, and where $e$
is the unique edge in $C$ having maximum label (among the edges in $C$). Of
course, the notion of a broken circuit of $G$ depends on the function $\ell$;
however, we suppress the mention of $\ell$ in our notation, since we will not
consider situations where two different $\ell$'s coexist.
\end{definition}

Thus, if $G$ is a graph with a labeling function $\ell$, then any circuit $C$
of $G$ gives rise to a broken circuit provided that among the edges in $C$,
only one attains the maximum label. (If more than one of the edges of $C$
attains the maximum label, then $C$ does not give rise to a broken circuit.)
Notice that two different circuits may give rise to one and the same broken circuit.

For instance, if $\left\{  a,b,c,d\right\}  $ is a circuit of a graph $G$ such
that $\ell\left(  a\right)  \leq\ell\left(  b\right)  \leq\ell\left(
c\right)  <\ell\left(  d\right)  $, then it gives rise to the broken circuit
$\left\{  a,b,c\right\}  $, since its unique edge having maximum label is $d$.
On the other hand, a circuit of the form $\left\{  a,b,c,d\right\}  $ with
$\ell\left(  a\right)  \leq\ell\left(  b\right)  \leq\ell\left(  c\right)
=\ell\left(  d\right)  $ (and $c\neq d$) does not give rise to any broken
circuit, since its edge with maximum label is not unique.

The notion of a broken circuit always depends on a labeling function
$\ell:E\rightarrow X$. Any time we speak about broken circuits, we shall
tacitly understand that the function $\ell:E\rightarrow X$ is used as the
labeling function.

\begin{example}
\label{exa.BC.graph1}Let $G$ be the graph $\left(  V,E\right)  $, where
$V=\left\{  1,2,3,4\right\}  $ and $E=\left\{  a,b,c,d,e\right\}  $ with%
\[
a=\left\{  1,2\right\}  ,\qquad b=\left\{  2,3\right\}  ,\qquad c=\left\{
1,3\right\}  ,\qquad d=\left\{  1,4\right\}  ,\qquad e=\left\{  3,4\right\}
.
\]
According to the standard conventions of graph theory, this graph $G$ can be
drawn as follows:%
\[%
\begin{tikzpicture}[scale=2]
\begin{scope}[every node/.style={circle,thick,draw=green!60!black}]
\node(1) at (0,1) {$1$};
\node(2) at (1,0) {$2$};
\node(3) at (0,-1) {$3$};
\node(4) at (-1,0) {$4$};
\end{scope}
\begin{scope}[every edge/.style={draw=black,very thick}, every loop/.style={}]
\path[-] (1) edge node[above] {$a$} (2);
\path[-] (1) edge node[left] {$c$} (3);
\path[-] (1) edge node[above] {$d$} (4);
\path[-] (2) edge node[below] {$b$} (3);
\path[-] (4) edge node[below] {$e$} (3);
\end{scope}
\end{tikzpicture}%
\ \ .
\]
Let $\ell:E\rightarrow X$ be a labeling function satisfying $\ell\left(
a\right)  <\ell\left(  b\right)  <\ell\left(  c\right)  <\ell\left(  d\right)
<\ell\left(  e\right)  $. Then, the circuits of $G$ are%
\[
\left\{  a,b,c\right\}  ,\qquad\left\{  a,b,d,e\right\}  ,\qquad\left\{
c,d,e\right\}  .
\]
The broken circuits of $G$ are therefore%
\[
\left\{  a,b\right\}  ,\qquad\left\{  a,b,d\right\}  ,\qquad\left\{
c,d\right\}  .
\]

\end{example}

\subsection{The main results}

We are now ready to state one of our main results:

\begin{theorem}
\label{thm.chromsym.varis}Let $G=\left(  V,E\right)  $ be a finite graph. Let
$X$ be a totally ordered set. Let $\ell:E\rightarrow X$ be a labeling
function. Let $\mathfrak{K}$ be some set of broken circuits of $G$ (not
necessarily containing all of them). Let $a_{K}$ be an element of $\mathbf{k}$
for every $K\in\mathfrak{K}$. Then,%
\[
X_{G}=\sum_{F\subseteq E}\left(  -1\right)  ^{\left\vert F\right\vert }\left(
\prod_{\substack{K\in\mathfrak{K};\\K\subseteq F}}a_{K}\right)  p_{\lambda
\left(  V,F\right)  }.
\]
(Here, of course, the pair $\left(  V,F\right)  $ is regarded as a graph, and
the expression $\lambda\left(  V,F\right)  $ is understood according to
Definition \ref{def.connectedness} \textbf{(b)}.)
\end{theorem}

Before we come to the proof of this result, let us explore some of its
particular cases. First, a definition is in order:

\begin{definition}
\label{def.K-free}Let $E$ be a set. Let $\mathfrak{K}$ be a subset of the
powerset of $E$ (that is, a set of subsets of $E$). A subset $F$ of $E$ is
said to be $\mathfrak{K}$\emph{-free} if $F$ contains no $K\in\mathfrak{K}$ as
a subset. (For instance, if $\mathfrak{K}=\varnothing$, then every subset $F$
of $E$ is $\mathfrak{K}$-free.)
\end{definition}

Here is a slightly more substantial example: If $E=\left\{  1,2,3,4\right\}  $
and $\mathfrak{K}=\left\{  \left\{  1,2\right\}  ,\ \left\{  2,3\right\}
\right\}  $, then the subset $\left\{  1,3\right\}  $ of $E$ is $\mathfrak{K}%
$-free whereas the subset $\left\{  2,3,4\right\}  $ is not (since it contains
$\left\{  2,3\right\}  \in\mathfrak{K}$ as a subset).

\begin{corollary}
\label{cor.chromsym.K-free}Let $G=\left(  V,E\right)  $ be a finite graph. Let
$X$ be a totally ordered set. Let $\ell:E\rightarrow X$ be a labeling
function. Let $\mathfrak{K}$ be some set of broken circuits of $G$ (not
necessarily containing all of them). Then,%
\[
X_{G}=\sum_{\substack{F\subseteq E;\\F\text{ is }\mathfrak{K}\text{-free}%
}}\left(  -1\right)  ^{\left\vert F\right\vert }p_{\lambda\left(  V,F\right)
}.
\]

\end{corollary}

\begin{corollary}
\label{cor.chromsym.NBC}Let $G=\left(  V,E\right)  $ be a finite graph. Let
$X$ be a totally ordered set. Let $\ell:E\rightarrow X$ be a labeling
function. Then,%
\[
X_{G}=\sum_{\substack{F\subseteq E;\\F\text{ contains no broken}%
\\\text{circuit of }G\text{ as a subset}}}\left(  -1\right)  ^{\left\vert
F\right\vert }p_{\lambda\left(  V,F\right)  }.
\]

\end{corollary}

Corollary \ref{cor.chromsym.NBC} appears in \cite[Theorem 2.9]{Stanley-chsym},
at least in the particular case in which $\ell$ is supposed to be injective.

\begin{example}
Let $G=\left(  V,E\right)  $ be the graph from Example \ref{exa.BC.graph1},
and let $\ell:E\rightarrow X$ be a labeling function as in Example
\ref{exa.BC.graph1}. Then, the subsets of $E$ that contain no broken circuits
of $G$ as subsets are the $18$ sets%
\begin{align*}
&  \varnothing,\qquad\left\{  a\right\}  ,\qquad\left\{  b\right\}
,\qquad\left\{  c\right\}  ,\qquad\left\{  d\right\}  ,\qquad\left\{
e\right\}  ,\qquad\left\{  a,c\right\}  ,\\
&  \left\{  a,d\right\}  ,\qquad\left\{  a,e\right\}  ,\qquad\left\{
b,c\right\}  ,\qquad\left\{  b,d\right\}  ,\qquad\left\{  b,e\right\}
,\qquad\left\{  c,e\right\}  ,\\
&  \left\{  d,e\right\}  ,\qquad\left\{  a,c,e\right\}  ,\qquad\left\{
a,d,e\right\}  ,\qquad\left\{  b,c,e\right\}  ,\qquad\left\{  b,d,e\right\}  .
\end{align*}
Thus, the sum on the right-hand side of Corollary \ref{cor.chromsym.NBC} has
$18$ addends. In contrast, the sums on the right-hand sides of Theorem
\ref{thm.chromsym.varis} and of Theorem \ref{thm.chromsym.empty} have $32$
addends. The number of addends in the sum on the right-hand side of Corollary
\ref{cor.chromsym.K-free} depends on the choice of $\mathfrak{K}$.
\end{example}

Let us now see how Theorem \ref{thm.chromsym.empty}, Corollary
\ref{cor.chromsym.K-free} and Corollary \ref{cor.chromsym.NBC} can be derived
from Theorem \ref{thm.chromsym.varis}:

\begin{vershort}
\begin{proof}
[Proof of Corollary \ref{cor.chromsym.K-free} using Theorem
\ref{thm.chromsym.varis}.]For every subset $F$ of $E$, we have%
\begin{equation}
\prod_{\substack{K\in\mathfrak{K};\\K\subseteq F}}0=%
\begin{cases}
1, & \text{if }F\text{ is }\mathfrak{K}\text{-free;}\\
0, & \text{if }F\text{ is not }\mathfrak{K}\text{-free}%
\end{cases}
\label{pf.cor.chromsym.K-free.short.prod0}%
\end{equation}
(because if $F$ is $\mathfrak{K}$-free, then the product $\prod
_{\substack{K\in\mathfrak{K};\\K\subseteq F}}0$ is empty and thus equals $1$;
otherwise, the product $\prod_{\substack{K\in\mathfrak{K};\\K\subseteq F}}0$
contains at least one factor and thus equals $0$). Now, Theorem
\ref{thm.chromsym.varis} (applied to $0$ instead of $a_{K}$) yields%
\begin{align*}
X_{G}  &  =\sum_{F\subseteq E}\left(  -1\right)  ^{\left\vert F\right\vert
}\underbrace{\left(  \prod_{\substack{K\in\mathfrak{K};\\K\subseteq
F}}0\right)  }_{\substack{=%
\begin{cases}
1, & \text{if }F\text{ is }\mathfrak{K}\text{-free;}\\
0, & \text{if }F\text{ is not }\mathfrak{K}\text{-free}%
\end{cases}
\\\text{(by (\ref{pf.cor.chromsym.K-free.short.prod0}))}}}p_{\lambda\left(
V,F\right)  }\\
&  =\sum_{F\subseteq E}\left(  -1\right)  ^{\left\vert F\right\vert }%
\begin{cases}
1, & \text{if }F\text{ is }\mathfrak{K}\text{-free;}\\
0, & \text{if }F\text{ is not }\mathfrak{K}\text{-free}%
\end{cases}
\ \ p_{\lambda\left(  V,F\right)  }=\sum_{\substack{F\subseteq E;\\F\text{ is
}\mathfrak{K}\text{-free}}}\left(  -1\right)  ^{\left\vert F\right\vert
}p_{\lambda\left(  V,F\right)  }.
\end{align*}
This proves Corollary \ref{cor.chromsym.K-free}.
\end{proof}
\end{vershort}

\begin{verlong}
\begin{proof}
[Proof of Corollary \ref{cor.chromsym.K-free} using Theorem
\ref{thm.chromsym.varis}.]We can apply Theorem \ref{thm.chromsym.varis} to $0$
instead of $a_{K}$. As a result, we obtain%
\begin{equation}
X_{G}=\sum_{F\subseteq E}\left(  -1\right)  ^{\left\vert F\right\vert }\left(
\prod_{\substack{K\in\mathfrak{K};\\K\subseteq F}}0\right)  p_{\lambda\left(
V,F\right)  }. \label{pf.cor.chromsym.K-free.0}%
\end{equation}
Now, if $F$ is any subset of $E$, then%
\begin{equation}
\prod_{\substack{K\in\mathfrak{K};\\K\subseteq F}}0=%
\begin{cases}
1, & \text{if }F\text{ is }\mathfrak{K}\text{-free;}\\
0, & \text{if }F\text{ is not }\mathfrak{K}\text{-free}%
\end{cases}
\label{pf.cor.chromsym.K-free.1}%
\end{equation}
\footnote{\textit{Proof of (\ref{pf.cor.chromsym.K-free.1}):} Let $F$ be any
subset of $E$. Recall that the set $F$ is $\mathfrak{K}$-free if and only if
$F$ contains no $K\in\mathfrak{K}$ as a subset (by the definition of
\textquotedblleft$\mathfrak{K}$-free\textquotedblright). Taking the
contrapositive of this equivalence statement, we obtain the following: The set
$F$ is \textbf{not} $\mathfrak{K}$-free if and only if $F$ contains some
$K\in\mathfrak{K}$ as a subset.
\par
We are in one of the following two cases:
\par
\textit{Case 1:} The set $F$ is $\mathfrak{K}$-free.
\par
\textit{Case 2:} The set $F$ is not $\mathfrak{K}$-free.
\par
Let us first consider Case 1. In this case, the set $F$ is $\mathfrak{K}%
$-free. In other words, $F$ contains no $K\in\mathfrak{K}$ as a subset
(because the set $F$ is $\mathfrak{K}$-free if and only if $F$ contains no
$K\in\mathfrak{K}$ as a subset (by the definition of \textquotedblleft%
$\mathfrak{K}$-free\textquotedblright)). In other words, there exists no
$K\in\mathfrak{K}$ such that $F$ contains $K$ as a subset. In other words,
there exists no $K\in\mathfrak{K}$ such that $K\subseteq F$. Hence, the
product $\prod_{\substack{K\in\mathfrak{K};\\K\subseteq F}}0$ is empty. Thus,
$\prod_{\substack{K\in\mathfrak{K};\\K\subseteq F}}0=\left(  \text{empty
product}\right)  =1$. Comparing this with
\[%
\begin{cases}
1, & \text{if }F\text{ is }\mathfrak{K}\text{-free;}\\
0, & \text{if }F\text{ is not }\mathfrak{K}\text{-free}%
\end{cases}
=1\ \ \ \ \ \ \ \ \ \ \left(  \text{since }F\text{ is }\mathfrak{K}%
\text{-free}\right)  ,
\]
we obtain $\prod_{\substack{K\in\mathfrak{K};\\K\subseteq F}}0=%
\begin{cases}
1, & \text{if }F\text{ is }\mathfrak{K}\text{-free;}\\
0, & \text{if }F\text{ is not }\mathfrak{K}\text{-free}%
\end{cases}
$. Thus, (\ref{pf.cor.chromsym.K-free.1}) is proven in Case 1.
\par
Let us now consider Case 2. In this case, the set $F$ is \textbf{not}
$\mathfrak{K}$-free. In other words, $F$ contains some $K\in\mathfrak{K}$ as a
subset (since the set $F$ is \textbf{not} $\mathfrak{K}$-free if and only if
$F$ contains some $K\in\mathfrak{K}$ as a subset). Let $L$ be such a $K$.
Thus, $L$ is an element of $\mathfrak{K}$, and the set $F$ contains $L$ as a
subset.
\par
Now, $L\subseteq F$ (since $F$ contains $L$ as a subset). Hence, the product
$\prod_{\substack{K\in\mathfrak{K};\\K\subseteq F}}0$ has at least one factor
(namely, the factor for $K=L$). This factor is clearly $0$. Therefore, the
whole product $\prod_{\substack{K\in\mathfrak{K};\\K\subseteq F}}0$ equals $0$
(because if a product contains a factor which is $0$, then the whole product
must equal $0$). In other words, $\prod_{\substack{K\in\mathfrak{K}%
;\\K\subseteq F}}0=0$. Comparing this with
\[%
\begin{cases}
1, & \text{if }F\text{ is }\mathfrak{K}\text{-free;}\\
0, & \text{if }F\text{ is not }\mathfrak{K}\text{-free}%
\end{cases}
=0\ \ \ \ \ \ \ \ \ \ \left(  \text{since }F\text{ is \textbf{not}
}\mathfrak{K}\text{-free}\right)  ,
\]
we obtain $\prod_{\substack{K\in\mathfrak{K};\\K\subseteq F}}0=%
\begin{cases}
1, & \text{if }F\text{ is }\mathfrak{K}\text{-free;}\\
0, & \text{if }F\text{ is not }\mathfrak{K}\text{-free}%
\end{cases}
$. Thus, (\ref{pf.cor.chromsym.K-free.1}) is proven in Case 2.
\par
We now have proven (\ref{pf.cor.chromsym.K-free.1}) in each of the two Cases 1
and 2. Since these two Cases cover all possibilities, this shows that
(\ref{pf.cor.chromsym.K-free.1}) always holds.}.

Thus, (\ref{pf.cor.chromsym.K-free.0}) becomes%
\begin{align*}
X_{G}  &  =\sum_{F\subseteq E}\left(  -1\right)  ^{\left\vert F\right\vert
}\underbrace{\left(  \prod_{\substack{K\in\mathfrak{K};\\K\subseteq
F}}0\right)  }_{\substack{=%
\begin{cases}
1, & \text{if }F\text{ is }\mathfrak{K}\text{-free;}\\
0, & \text{if }F\text{ is not }\mathfrak{K}\text{-free}%
\end{cases}
\\\text{(by (\ref{pf.cor.chromsym.K-free.1}))}}}p_{\lambda\left(  V,F\right)
}\\
&  =\sum_{F\subseteq E}\left(  -1\right)  ^{\left\vert F\right\vert }%
\begin{cases}
1, & \text{if }F\text{ is }\mathfrak{K}\text{-free;}\\
0, & \text{if }F\text{ is not }\mathfrak{K}\text{-free}%
\end{cases}
\ \ p_{\lambda\left(  V,F\right)  }\\
&  =\sum_{\substack{F\subseteq E;\\F\text{ is }\mathfrak{K}\text{-free}%
}}\left(  -1\right)  ^{\left\vert F\right\vert }\underbrace{%
\begin{cases}
1, & \text{if }F\text{ is }\mathfrak{K}\text{-free;}\\
0, & \text{if }F\text{ is not }\mathfrak{K}\text{-free}%
\end{cases}
}_{\substack{=1\\\text{(since }F\text{ is }\mathfrak{K}\text{-free)}%
}}\ \ p_{\lambda\left(  V,F\right)  }\\
&  \ \ \ \ \ \ \ \ \ \ +\sum_{\substack{F\subseteq E;\\F\text{ is not
}\mathfrak{K}\text{-free}}}\left(  -1\right)  ^{\left\vert F\right\vert
}\underbrace{%
\begin{cases}
1, & \text{if }F\text{ is }\mathfrak{K}\text{-free;}\\
0, & \text{if }F\text{ is not }\mathfrak{K}\text{-free}%
\end{cases}
}_{\substack{=0\\\text{(since }F\text{ is not }\mathfrak{K}\text{-free)}%
}}\ \ p_{\lambda\left(  V,F\right)  }\\
&  \ \ \ \ \ \ \ \ \ \ \ \ \ \ \ \ \ \ \ \ \left(  \text{since each subset
}F\text{ of }E\text{ either is }\mathfrak{K}\text{-free or is not}\right) \\
&  =\sum_{\substack{F\subseteq E;\\F\text{ is }\mathfrak{K}\text{-free}%
}}\left(  -1\right)  ^{\left\vert F\right\vert }p_{\lambda\left(  V,F\right)
}+\underbrace{\sum_{\substack{F\subseteq E;\\F\text{ is not }\mathfrak{K}%
\text{-free}}}\left(  -1\right)  ^{\left\vert F\right\vert }0p_{\lambda\left(
V,F\right)  }}_{=0}\\
&  =\sum_{\substack{F\subseteq E;\\F\text{ is }\mathfrak{K}\text{-free}%
}}\left(  -1\right)  ^{\left\vert F\right\vert }p_{\lambda\left(  V,F\right)
}.
\end{align*}
This proves Corollary \ref{cor.chromsym.K-free}.
\end{proof}
\end{verlong}

\begin{vershort}
\begin{proof}
[Proof of Corollary \ref{cor.chromsym.NBC} using Corollary
\ref{cor.chromsym.K-free}.]Corollary \ref{cor.chromsym.NBC} follows from
Corollary \ref{cor.chromsym.K-free} when $\mathfrak{K}$ is set to be the set
of \textbf{all} broken circuits of $G$.
\end{proof}
\end{vershort}

\begin{verlong}
\begin{proof}
[Proof of Corollary \ref{cor.chromsym.NBC} using Corollary
\ref{cor.chromsym.K-free}.]Let $\mathfrak{K}$ be the set of all broken
circuits of $G$. Thus, the elements of $\mathfrak{K}$ are the broken circuits
of $G$.

Now, for every subset $F$ of $E$, we have the following equivalence of
statements:%
\begin{align*}
&  \ \left(  F\text{ is }\mathfrak{K}\text{-free}\right) \\
&  \Longleftrightarrow\ \left(  F\text{ contains no }K\in\mathfrak{K}\text{ as
a subset}\right) \\
&  \ \ \ \ \ \ \ \ \ \ \left(
\begin{array}
[c]{c}%
\text{because }F\text{ is }\mathfrak{K}\text{-free if and only if }F\text{
contains no }K\in\mathfrak{K}\text{ as a subset}\\
\text{(by the definition of \textquotedblleft}\mathfrak{K}%
\text{-free\textquotedblright)}%
\end{array}
\right) \\
&  \Longleftrightarrow\ \left(  F\text{ contains no element of }%
\mathfrak{K}\text{ as a subset}\right) \\
&  \Longleftrightarrow\ \left(  F\text{ contains no broken circuit of }G\text{
as a subset}\right)
\end{align*}
(since the elements of $\mathfrak{K}$ are the broken circuits of $G$). Hence,
\[
\sum_{\substack{F\subseteq E;\\F\text{ is }\mathfrak{K}\text{-free}}%
}=\sum_{\substack{F\subseteq E;\\F\text{ contains no broken}\\\text{circuit of
}G\text{ as a subset}}}
\]
(an equality between summation signs). Now, Corollary
\ref{cor.chromsym.K-free} yields%
\[
X_{G}=\underbrace{\sum_{\substack{F\subseteq E;\\F\text{ is }\mathfrak{K}%
\text{-free}}}}_{=\sum_{\substack{F\subseteq E;\\F\text{ contains no
broken}\\\text{circuit of }G\text{ as a subset}}}}\left(  -1\right)
^{\left\vert F\right\vert }p_{\lambda\left(  V,F\right)  }=\sum
_{\substack{F\subseteq E;\\F\text{ contains no broken}\\\text{circuit of
}G\text{ as a subset}}}\left(  -1\right)  ^{\left\vert F\right\vert
}p_{\lambda\left(  V,F\right)  }.
\]
This proves Corollary \ref{cor.chromsym.NBC}.
\end{proof}
\end{verlong}

\begin{vershort}
\begin{proof}
[Proof of Theorem \ref{thm.chromsym.empty} using Theorem
\ref{thm.chromsym.varis}.]Let $X$ be the totally ordered set $\left\{
1\right\}  $, and let $\ell:E\rightarrow X$ be the only possible map. Let
$\mathfrak{K}$ be the empty set. Clearly, $\mathfrak{K}$ is a set of broken
circuits of $G$. For every $F\subseteq E$, the product $\prod_{\substack{K\in
\mathfrak{K};\\K\subseteq F}}0$ is empty (since $\mathfrak{K}$ is the empty
set), and thus equals $1$. Now, Theorem \ref{thm.chromsym.varis} (applied to
$0$ instead of $a_{K}$) yields%
\[
X_{G}=\sum_{F\subseteq E}\left(  -1\right)  ^{\left\vert F\right\vert
}\underbrace{\left(  \prod_{\substack{K\in\mathfrak{K};\\K\subseteq
F}}0\right)  }_{=1}p_{\lambda\left(  V,F\right)  }=\sum_{F\subseteq E}\left(
-1\right)  ^{\left\vert F\right\vert }p_{\lambda\left(  V,F\right)  }.
\]
This proves Theorem \ref{thm.chromsym.empty}.
\end{proof}
\end{vershort}

\begin{verlong}
\begin{proof}
[Proof of Theorem \ref{thm.chromsym.empty} using Theorem
\ref{thm.chromsym.varis}.]Let $X$ be the totally ordered set $\left\{
1\right\}  $ (equipped with the only possible order on this set). Let
$\ell:E\rightarrow X$ be the function sending each $e\in E$ to $1\in X$. Let
$\mathfrak{K}$ be the empty set. Clearly, $\mathfrak{K}$ is a set of broken
circuits of $G$. Theorem \ref{thm.chromsym.varis} (applied to $0$ instead of
$a_{K}$) yields%
\begin{align*}
X_{G}  &  =\sum_{F\subseteq E}\left(  -1\right)  ^{\left\vert F\right\vert
}\underbrace{\left(  \prod_{\substack{K\in\mathfrak{K};\\K\subseteq
F}}0\right)  }_{\substack{=\left(  \text{empty product}\right)  \\\text{(since
}\mathfrak{K}\text{ is the empty set)}}}p_{\lambda\left(  V,F\right)  }\\
&  =\sum_{F\subseteq E}\left(  -1\right)  ^{\left\vert F\right\vert
}\underbrace{\left(  \text{empty product}\right)  }_{=1}p_{\lambda\left(
V,F\right)  }=\sum_{F\subseteq E}\left(  -1\right)  ^{\left\vert F\right\vert
}p_{\lambda\left(  V,F\right)  }.
\end{align*}
This proves Theorem \ref{thm.chromsym.empty}.
\end{proof}
\end{verlong}

\section{\label{sec.first-proofs}Proof of Theorem \ref{thm.chromsym.varis}}

We shall now prepare for the proof of Theorem \ref{thm.chromsym.varis} with
some notations and some lemmas. Our proof will imitate \cite[proof of
Whitney's theorem]{BlaSag86}. We note that Theorem \ref{thm.chromsym.varis}
can also be easily obtained as a consequence of \cite[\S 2 and \S 3.1]%
{DohTri14} using Theorem \ref{thm.chromsym.empty}, but our proof has the
advantage of not relying on Theorem \ref{thm.chromsym.empty} (so that it leads
to a new proof of Theorem \ref{thm.chromsym.empty}).

\subsection{$\operatorname*{Eqs}f$ and basic lemmas}

We introduce a simple notion that measures \textquotedblleft how
non-proper\textquotedblright\ a given coloring of a graph is:\footnote{If $V$
is a set, then $V^{2}$ denotes the Cartesian product $V\times V$, that is, the
set of all ordered pairs of elements of $V$.}

\begin{definition}
\label{def.Eqs}Let $V$ and $X$ be two sets. Let $f:V\rightarrow X$ be a map.
We let $\operatorname*{Eqs}f$ denote the subset%
\[
\left\{  \left\{  s,t\right\}  \ \mid\ \left(  s,t\right)  \in V^{2},\ s\neq
t\text{ and }f\left(  s\right)  =f\left(  t\right)  \right\}
\]
of $\dbinom{V}{2}$. (This is well-defined, because any two elements $s$ and
$t$ of $V$ satisfying $s\neq t$ clearly satisfy $\left\{  s,t\right\}
\in\dbinom{V}{2}$.)
\end{definition}

\begin{example}
Let $V=\left\{  1,2,3,4,5\right\}  $ and $X=\left\{  1,2,3\right\}  $, and let
$f:V\rightarrow X$ be the map that sends the three numbers $1,2,3$ to $1$ and
the remaining two numbers $4,5$ to $2$. Then,%
\[
\operatorname*{Eqs}f=\left\{  \left\{  1,2\right\}  ,\ \left\{  1,3\right\}
,\ \left\{  2,3\right\}  ,\ \left\{  4,5\right\}  \right\}  .
\]

\end{example}

We shall now state some first properties of this notion:

\begin{lemma}
\label{lem.Eqs.proper}Let $G=\left(  V,E\right)  $ be a graph. Let $X$ be a
set. Let $f:V\rightarrow X$ be a map. Then, the $X$-coloring $f$ of $G$ is
proper if and only if $E\cap\operatorname*{Eqs}f=\varnothing$.
\end{lemma}

\begin{vershort}
\begin{proof}
[Proof of Lemma \ref{lem.Eqs.proper}.]The set $E\cap\operatorname*{Eqs}f$ is
precisely the set of edges $\left\{  s,t\right\}  $ of $G$ satisfying
$f\left(  s\right)  =f\left(  t\right)  $; meanwhile, the $X$-coloring $f$ is
called proper if and only if no such edges exist. Thus, Lemma
\ref{lem.Eqs.proper} becomes obvious.
\end{proof}
\end{vershort}

\begin{verlong}
\begin{proof}
[Proof of Lemma \ref{lem.Eqs.proper}.]The definition of $\operatorname*{Eqs}f$
shows that%
\begin{align}
\operatorname*{Eqs}f  &  =\left\{  \left\{  s,t\right\}  \ \mid\ \left(
s,t\right)  \in V^{2},\ s\neq t\text{ and }f\left(  s\right)  =f\left(
t\right)  \right\} \nonumber\\
&  =\left\{  \left\{  x,y\right\}  \ \mid\ \left(  x,y\right)  \in
V^{2},\ x\neq y\text{ and }f\left(  x\right)  =f\left(  y\right)  \right\}
\label{pf.lem.Eqs.proper.0}%
\end{align}
(here, we renamed the index $\left(  s,t\right)  $ as $\left(  x,y\right)  $).

We shall first prove the logical implication%
\begin{equation}
\left(  \text{the }X\text{-coloring }f\text{ of }G\text{ is proper}\right)
\ \Longrightarrow\ \left(  E\cap\operatorname*{Eqs}f=\varnothing\right)  .
\label{pf.lem.Eqs.proper.im1}%
\end{equation}

\textit{Proof of (\ref{pf.lem.Eqs.proper.im1}):} Assume that the $X$-coloring
$f$ of $G$ is proper. We must show that $E\cap\operatorname*{Eqs}%
f=\varnothing$.

Recall that the $X$-coloring $f$ of $G$ is proper if and only if every edge
$\left\{  s,t\right\}  \in E$ satisfies $f\left(  s\right)  \neq f\left(
t\right)  $ (by the definition of \textquotedblleft proper\textquotedblright).
Thus,%
\begin{equation}
\text{every edge }\left\{  s,t\right\}  \in E\text{ satisfies }f\left(
s\right)  \neq f\left(  t\right)  \label{pf.lem.Eqs.proper.im1.pf.1}%
\end{equation}
(since the $X$-coloring $f$ of $G$ is proper).

Now, let $e\in E\cap\operatorname*{Eqs}f$. Thus,%
\[
e\in E\cap\operatorname*{Eqs}f\subseteq\operatorname*{Eqs}f=\left\{  \left\{
s,t\right\}  \ \mid\ \left(  s,t\right)  \in V^{2},\ s\neq t\text{ and
}f\left(  s\right)  =f\left(  t\right)  \right\}
\]
(by the definition of $\operatorname*{Eqs}f$). In other words, $e=\left\{
s,t\right\}  $ for some $\left(  s,t\right)  \in V^{2}$ satisfying $s\neq t$
and $f\left(  s\right)  =f\left(  t\right)  $. Consider this $\left(
s,t\right)  $. We have $\left\{  s,t\right\}  =e\in E\cap\operatorname*{Eqs}%
f\subseteq E$ and therefore $f\left(  s\right)  \neq f\left(  t\right)  $ (by
(\ref{pf.lem.Eqs.proper.im1.pf.1})). This contradicts $f\left(  s\right)
=f\left(  t\right)  $.

Now, let us forget that we fixed $e$. We thus have found a contradiction for
every $e\in E\cap\operatorname*{Eqs}f$. Therefore, no $e\in E\cap
\operatorname*{Eqs}f$ exists. In other words, the set $E\cap
\operatorname*{Eqs}f$ is empty. In other words, $E\cap\operatorname*{Eqs}%
f=\varnothing$. Thus, the implication (\ref{pf.lem.Eqs.proper.im1}) is proven.

Now, we shall prove the implication%
\begin{equation}
\left(  E\cap\operatorname*{Eqs}f=\varnothing\right)  \ \Longrightarrow
\ \left(  \text{the }X\text{-coloring }f\text{ of }G\text{ is proper}\right)
. \label{pf.lem.Eqs.proper.im2}%
\end{equation}

\textit{Proof of (\ref{pf.lem.Eqs.proper.im2}):} Assume that $E\cap
\operatorname*{Eqs}f=\varnothing$. We have to show that the $X$-coloring $f$
of $G$ is proper.

Let $\left\{  s,t\right\}  \in E$ be an edge. We shall now show that $f\left(
s\right)  \neq f\left(  t\right)  $.

Indeed, assume the contrary. Thus, $f\left(  s\right)  =f\left(  t\right)  $.
Now, $\left\{  s,t\right\}  \in E\subseteq\dbinom{V}{2}$. In other words,
$\left\{  s,t\right\}  $ is a $2$-element subset of $V$ (since $\dbinom{V}{2}$
is the set of all $2$-element subsets of $V$). Thus, $\left\vert \left\{
s,t\right\}  \right\vert =2$, so that $s\neq t$. Also, $\left\{  s,t\right\}
$ is a subset of $V$; thus, $s\in V$ and $t\in V$. Hence, $\left(  s,t\right)
\in V^{2}$. So we know that $\left\{  s,t\right\}  $ has the form $\left\{
x,y\right\}  $ for some $\left(  x,y\right)  \in V^{2}$ satisfying $x\neq y$
and $f\left(  x\right)  =f\left(  y\right)  $ (namely, this $\left(
x,y\right)  $ is $\left(  s,t\right)  $). In other words,%
\[
\left\{  s,t\right\}  \in\left\{  \left\{  x,y\right\}  \ \mid\ \left(
x,y\right)  \in V^{2},\ x\neq y\text{ and }f\left(  x\right)  =f\left(
y\right)  \right\}  =\operatorname*{Eqs}f
\]
(by (\ref{pf.lem.Eqs.proper.0})). Combining this with $\left\{  s,t\right\}
\in E$, we obtain $\left\{  s,t\right\}  \in E\cap\operatorname*{Eqs}%
f=\varnothing$. Thus, the set $\varnothing$ has an element (namely, $\left\{
s,t\right\}  $). This contradicts the fact that the set $\varnothing$ is
empty. Thus, we have obtained a contradiction. This shows that our assumption
was wrong. Hence, $f\left(  s\right)  \neq f\left(  t\right)  $ is proven.

Let us now forget that we fixed $\left\{  s,t\right\}  $. We thus have shown
that every edge $\left\{  s,t\right\}  \in E$ satisfies $f\left(  s\right)
\neq f\left(  t\right)  $. Therefore, the $X$-coloring $f$ of $G$ is proper
(since the $X$-coloring $f$ of $G$ is proper if and only if every edge
$\left\{  s,t\right\}  \in E$ satisfies $f\left(  s\right)  \neq f\left(
t\right)  $ (by the definition of \textquotedblleft proper\textquotedblright%
)). This proves the implication (\ref{pf.lem.Eqs.proper.im2}).

Now we have proven the two implications (\ref{pf.lem.Eqs.proper.im1}) and
(\ref{pf.lem.Eqs.proper.im2}). Combining these two implications, we obtain the
equivalence%
\[
\left(  \text{the }X\text{-coloring }f\text{ of }G\text{ is proper}\right)
\ \Longleftrightarrow\ \left(  E\cap\operatorname*{Eqs}f=\varnothing\right)
.
\]
This proves Lemma \ref{lem.Eqs.proper}.
\end{proof}
\end{verlong}

\begin{lemma}
\label{lem.Eqs.circuit}Let $G=\left(  V,E\right)  $ be a graph. Let $X$ be a
set. Let $f:V\rightarrow X$ be a map. Let $C$ be a circuit of $G$. Let $e\in
C$ be such that $C\setminus\left\{  e\right\}  \subseteq\operatorname*{Eqs}f$.
Then, $e\in E\cap\operatorname*{Eqs}f$.
\end{lemma}

\begin{vershort}
\begin{proof}
[Proof of Lemma \ref{lem.Eqs.circuit}.]The set $C$ is a circuit of $G$. Hence,
we can write $C$ in the form%
\[
C=\left\{  \left\{  v_{1},v_{2}\right\}  ,\left\{  v_{2},v_{3}\right\}
,\ldots,\left\{  v_{m},v_{m+1}\right\}  \right\}
\]
for some cycle $\left(  v_{1},v_{2},\ldots,v_{m+1}\right)  $ of $G$. Consider
this cycle $\left(  v_{1},v_{2},\ldots,v_{m+1}\right)  $. According to the
definition of a \textquotedblleft cycle\textquotedblright, the cycle $\left(
v_{1},v_{2},\ldots,v_{m+1}\right)  $ is a list of elements of $V$ having the
following properties:

\begin{itemize}
\item We have $m>2$.

\item We have $v_{m+1}=v_{1}$.

\item The vertices $v_{1},v_{2},\ldots,v_{m}$ are pairwise distinct.

\item We have $\left\{  v_{i},v_{i+1}\right\}  \in E$ for every $i\in\left\{
1,2,\ldots,m\right\}  $.
\end{itemize}

From the first three of these properties, we can easily conclude that the $m$
sets $\left\{  v_{1},v_{2}\right\}  ,\left\{  v_{2},v_{3}\right\}
,\ldots,\left\{  v_{m},v_{m+1}\right\}  $ are distinct.

Recall that $e\in C$. Therefore, $e=\left\{  v_{i},v_{i+1}\right\}  $ for some
$i\in\left\{  1,2,\ldots,m\right\}  $. We can thus WLOG assume that
$e=\left\{  v_{m},v_{m+1}\right\}  $ (since otherwise, we can simply relabel
the vertices along the cycle $\left(  v_{1},v_{2},\ldots,v_{m+1}\right)  $).
Assume this. Since $C=\left\{  \left\{  v_{1},v_{2}\right\}  ,\left\{
v_{2},v_{3}\right\}  ,\ldots,\left\{  v_{m},v_{m+1}\right\}  \right\}  $ and
$e=\left\{  v_{m},v_{m+1}\right\}  $, we have%
\[
C\setminus\left\{  e\right\}  =\left\{  \left\{  v_{1},v_{2}\right\}
,\left\{  v_{2},v_{3}\right\}  ,\ldots,\left\{  v_{m-1},v_{m}\right\}
\right\}
\]
(since the $m$ sets $\left\{  v_{1},v_{2}\right\}  ,\left\{  v_{2}%
,v_{3}\right\}  ,\ldots,\left\{  v_{m},v_{m+1}\right\}  $ are distinct). For
every $i\in\left\{  1,2,\ldots,m-1\right\}  $, we have $f\left(  v_{i}\right)
=f\left(  v_{i+1}\right)  $ (since
\[
\left\{  v_{i},v_{i+1}\right\}  \in\left\{  \left\{  v_{1},v_{2}\right\}
,\left\{  v_{2},v_{3}\right\}  ,\ldots,\left\{  v_{m-1},v_{m}\right\}
\right\}  =C\setminus\left\{  e\right\}  \subseteq\operatorname*{Eqs}f
\]
). Hence, $f\left(  v_{1}\right)  =f\left(  v_{2}\right)  =\cdots=f\left(
v_{m}\right)  $, so that $f\left(  v_{m}\right)  =f\left(  v_{1}\right)
=f\left(  v_{m+1}\right)  $ (because $v_{1}=v_{m+1}$). Thus, $\left\{
v_{m},v_{m+1}\right\}  \in\operatorname*{Eqs}f$. Thus, $e=\left\{
v_{m},v_{m+1}\right\}  \in\operatorname*{Eqs}f$. Combined with $e\in E$, this
yields $e\in E\cap\operatorname*{Eqs}f$. This proves Lemma
\ref{lem.Eqs.circuit}.
\end{proof}
\end{vershort}

\begin{verlong}
\begin{proof}
[Proof of Lemma \ref{lem.Eqs.circuit}.]The set $C$ is a circuit of $G$. In
other words, the set $C$ has the form $\left\{  \left\{  v_{1},v_{2}\right\}
,\left\{  v_{2},v_{3}\right\}  ,\ldots,\left\{  v_{m},v_{m+1}\right\}
\right\}  $, where $\left(  v_{1},v_{2},\ldots,v_{m+1}\right)  $ is a cycle of
$G$ (by the definition of a \textquotedblleft circuit\textquotedblright).
Consider this cycle $\left(  v_{1},v_{2},\ldots,v_{m+1}\right)  $. We thus
have
\begin{equation}
C=\left\{  \left\{  v_{1},v_{2}\right\}  ,\left\{  v_{2},v_{3}\right\}
,\ldots,\left\{  v_{m},v_{m+1}\right\}  \right\}  .
\label{pf.lem.Eqs.circuit.1}%
\end{equation}

The list $\left(  v_{1},v_{2},\ldots,v_{m+1}\right)  $ is a cycle of $G$.
According to the definition of a \textquotedblleft cycle\textquotedblright,
this means that this list is a list of elements of $V$ satisfying the
following four properties:

\begin{itemize}
\item We have $m>2$.

\item We have $v_{m+1}=v_{1}$.

\item The vertices $v_{1},v_{2},\ldots,v_{m}$ are pairwise distinct.

\item We have $\left\{  v_{i},v_{i+1}\right\}  \in E$ for every $i\in\left\{
1,2,\ldots,m\right\}  $.
\end{itemize}

Thus, $\left(  v_{1},v_{2},\ldots,v_{m+1}\right)  $ is a list of elements of
$V$ satisfying the four properties that we have just mentioned. In particular,
$v_{m+1}=v_{1}$. Also,%
\begin{equation}
\left\{  v_{i},v_{i+1}\right\}  \in E\ \ \ \ \ \ \ \ \ \ \text{for every }%
i\in\left\{  1,2,\ldots,m\right\}  . \label{pf.lem.Eqs.circuit.inE}%
\end{equation}

Any two distinct elements $p$ and $q$ of $\left\{  1,2,\ldots,m\right\}  $
satisfy%
\begin{equation}
\left\{  v_{p},v_{p+1}\right\}  \neq\left\{  v_{q},v_{q+1}\right\}
\label{pf.lem.Eqs.circuit.distinct-edges}%
\end{equation}
\footnote{\textit{Proof of (\ref{pf.lem.Eqs.circuit.distinct-edges}):} Let $p$
and $q$ be two distinct elements of $\left\{  1,2,\ldots,m\right\}  $. We must
prove (\ref{pf.lem.Eqs.circuit.distinct-edges}).
\par
Assume the contrary. Thus, $\left\{  v_{p},v_{p+1}\right\}  =\left\{
v_{q},v_{q+1}\right\}  $.
\par
The vertices $v_{1},v_{2},\ldots,v_{m}$ are pairwise distinct. In other words,
any two distinct elements $a$ and $b$ of $\left\{  1,2,\ldots,m\right\}  $
satisfy $v_{a}\neq v_{b}$. Applying this to $a=p$ and $b=q$, we obtain
$v_{p}\neq v_{q}$. Combining $v_{p}\in\left\{  v_{p},v_{p+1}\right\}
=\left\{  v_{q},v_{q+1}\right\}  $ with $v_{p}\neq v_{q}$, we obtain $v_{p}%
\in\left\{  v_{q},v_{q+1}\right\}  \setminus\left\{  v_{q}\right\}
\subseteq\left\{  v_{q+1}\right\}  $. Thus, $v_{p}=v_{q+1}$. The same argument
(with $p$ and $q$ replaced by $q$ and $p$) yields $v_{q}=v_{p+1}$.
\par
We have $p\neq q$ (since $p$ and $q$ are distinct). Thus, we can WLOG assume
that $p<q$ (since otherwise, we can simply switch $p$ with $q$). Assume this.
From $q\in\left\{  1,2,\ldots,m\right\}  $, we obtain $q\leq m$, so that
$p<q\leq m$. Since $p$ and $m$ are integers, this shows that $p\leq m-1$.
Thus, $p+1\leq m$. Hence, $p+1\in\left\{  1,2,\ldots,m\right\}  $.
\par
Recall again that any two distinct elements $a$ and $b$ of $\left\{
1,2,\ldots,m\right\}  $ satisfy $v_{a}\neq v_{b}$. In other words, if two
elements $a$ and $b$ of $\left\{  1,2,\ldots,m\right\}  $ satisfy $v_{a}%
=v_{b}$, then $a=b$. Applying this to $a=q$ and $b=p+1$, we obtain $q=p+1$
(since $q\in\left\{  1,2,\ldots,m\right\}  $ and $p+1\in\left\{
1,2,\ldots,m\right\}  $ and $v_{q}=v_{p+1}$). The same argument (but with the
roles of $p$ and $q$ switched) shows that $p=q+1$ if $q+1\in\left\{
1,2,\ldots,m\right\}  $. Since $p=q+1$ is impossible (because $q+1>q=p+1>p$),
we thus conclude that $q+1\in\left\{  1,2,\ldots,m\right\}  $ is impossible as
well. Thus, we have $q+1\notin\left\{  1,2,\ldots,m\right\}  $. In other
words, $q\notin\left\{  0,1,\ldots,m-1\right\}  $.
\par
Combining $q\in\left\{  1,2,\ldots,m\right\}  $ with $q\notin\left\{
0,1,\ldots,m-1\right\}  $, we find $q\in\left\{  1,2,\ldots,m\right\}
\setminus\left\{  0,1,\ldots,m-1\right\}  =\left\{  m\right\}  $. In other
words, $q=m$. Comparing this with $q=p+1$, we obtain $p+1=m$, so that $p=m-1$.
Hence, $v_{p}=v_{m-1}$, so that $v_{m-1}=v_{p}=v_{q+1}=v_{m+1}$ (since $q=m$).
Therefore, $v_{m-1}=v_{m+1}=v_{1}$.
\par
However, $m>2$, so that $m-1>1$ and thus $m-1\neq1$. In other words, $m-1$ and
$1$ are distinct. Recall again that any two distinct elements $a$ and $b$ of
$\left\{  1,2,\ldots,m\right\}  $ satisfy $v_{a}\neq v_{b}$. Applying this to
$a=m-1$ and $b=1$, we obtain $v_{m-1}\neq v_{1}$ (since $m-1$ and $1$ are
distinct). This contradicts $v_{m-1}=v_{1}$.
\par
We thus have found a contradiction. This contradiction proves that our
assumption was wrong. Hence, (\ref{pf.lem.Eqs.circuit.distinct-edges}) is
proven.}.

We have $e\in C=\left\{  \left\{  v_{1},v_{2}\right\}  ,\left\{  v_{2}%
,v_{3}\right\}  ,\ldots,\left\{  v_{m},v_{m+1}\right\}  \right\}  $. Thus,
$e=\left\{  v_{i},v_{i+1}\right\}  $ for some $i\in\left\{  1,2,\ldots
,m\right\}  $. Consider this $i$.

Now, we have%
\begin{equation}
f\left(  v_{j}\right)  =f\left(  v_{j+1}\right)  \ \ \ \ \ \ \ \ \ \ \text{for
every }j\in\left\{  1,2,\ldots,m\right\}  \setminus\left\{  i\right\}
\label{pf.lem.Eqs.circuit.2}%
\end{equation}
\footnote{\textit{Proof of (\ref{pf.lem.Eqs.circuit.2}):} Let $j\in\left\{
1,2,\ldots,m\right\}  \setminus\left\{  i\right\}  $. Thus, $j\in\left\{
1,2,\ldots,m\right\}  $ and $j\notin\left\{  i\right\}  $. From $j\notin%
\left\{  i\right\}  $, we obtain $j\neq i$. Therefore, the two elements $j$
and $i$ of $\left\{  1,2,\ldots,m\right\}  $ are distinct. Thus,
(\ref{pf.lem.Eqs.circuit.distinct-edges}) (applied to $p=j$ and $q=i$) shows
that $\left\{  v_{j},v_{j+1}\right\}  \neq\left\{  v_{i},v_{i+1}\right\}  =e$.
But from $j\in\left\{  1,2,\ldots,m\right\}  $, we obtain%
\begin{align*}
\left\{  v_{j},v_{j+1}\right\}   &  \in\left\{  \left\{  v_{k},v_{k+1}%
\right\}  \ \mid\ k\in\left\{  1,2,\ldots,m\right\}  \right\} \\
&  =\left\{  \left\{  v_{1},v_{2}\right\}  ,\left\{  v_{2},v_{3}\right\}
,\ldots,\left\{  v_{m},v_{m+1}\right\}  \right\}  =C
\end{align*}
(by (\ref{pf.lem.Eqs.circuit.1})). Combining this with $\left\{  v_{j}%
,v_{j+1}\right\}  \neq e$, we obtain%
\begin{align*}
\left\{  v_{j},v_{j+1}\right\}   &  \in C\setminus\left\{  e\right\}
\subseteq\operatorname*{Eqs}f\\
&  =\left\{  \left\{  s,t\right\}  \ \mid\ \left(  s,t\right)  \in
V^{2},\ s\neq t\text{ and }f\left(  s\right)  =f\left(  t\right)  \right\}  .
\end{align*}
In other words, $\left\{  v_{j},v_{j+1}\right\}  $ has the form $\left\{
s,t\right\}  $ for some $\left(  s,t\right)  \in V^{2}$ satisfying $s\neq t$
and $f\left(  s\right)  =f\left(  t\right)  $. Consider this $\left(
s,t\right)  $. Thus, $\left\{  v_{j},v_{j+1}\right\}  =\left\{  s,t\right\}
$.
\par
We have $f\left(  s\right)  =f\left(  t\right)  $. Therefore, set $g=f\left(
s\right)  =f\left(  t\right)  $. We have%
\[
f\left(  \left\{  s,t\right\}  \right)  =\left\{  \underbrace{f\left(
s\right)  }_{=g},\underbrace{f\left(  t\right)  }_{=g}\right\}  =\left\{
g,g\right\}  =\left\{  g\right\}  .
\]
Now, $v_{j}\in\left\{  v_{j},v_{j+1}\right\}  =\left\{  s,t\right\}  $, and
thus $f\left(  \underbrace{v_{j}}_{\in\left\{  s,t\right\}  }\right)  \in
f\left(  \left\{  s,t\right\}  \right)  =\left\{  g\right\}  $. In other
words, $f\left(  v_{j}\right)  =g$. Also, $v_{j+1}\in\left\{  v_{j}%
,v_{j+1}\right\}  =\left\{  s,t\right\}  $, and thus $f\left(
\underbrace{v_{j+1}}_{\in\left\{  s,t\right\}  }\right)  \in f\left(  \left\{
s,t\right\}  \right)  =\left\{  g\right\}  $. In other words, $f\left(
v_{j+1}\right)  =g$. Comparing this with $f\left(  v_{j}\right)  =g$, we
obtain $f\left(  v_{j}\right)  =f\left(  v_{j+1}\right)  $. This proves
(\ref{pf.lem.Eqs.circuit.2}).}. Hence,%
\begin{equation}
f\left(  v_{1}\right)  =f\left(  v_{i}\right)  \label{pf.lem.Eqs.circuit.3a}%
\end{equation}
\footnote{\textit{Proof of (\ref{pf.lem.Eqs.circuit.3a}):} Let $j\in\left\{
1,2,\ldots,i-1\right\}  $. Thus, $j\in\left\{  1,2,\ldots,i-1\right\}
\subseteq\left\{  1,2,\ldots,m\right\}  $. Combining this with $j\neq i$
(since $j<i$ (since $j\in\left\{  1,2,\ldots,i-1\right\}  $)), we obtain
$j\in\left\{  1,2,\ldots,m\right\}  \setminus\left\{  i\right\}  $. Hence,
$f\left(  v_{j}\right)  =f\left(  v_{j+1}\right)  $ (by
(\ref{pf.lem.Eqs.circuit.2})).
\par
Now, let us forget that we fixed $j$. We thus have proven that $f\left(
v_{j}\right)  =f\left(  v_{j+1}\right)  $ for every $j\in\left\{
1,2,\ldots,i-1\right\}  $. In other words, $f\left(  v_{1}\right)  =f\left(
v_{2}\right)  =\cdots=f\left(  v_{i}\right)  $. Hence, $f\left(  v_{1}\right)
=f\left(  v_{i}\right)  $. This proves (\ref{pf.lem.Eqs.circuit.3a}).}. Also,%
\begin{equation}
f\left(  v_{i+1}\right)  =f\left(  v_{m+1}\right)
\label{pf.lem.Eqs.circuit.3b}%
\end{equation}
\footnote{\textit{Proof of (\ref{pf.lem.Eqs.circuit.3b}):} Let $j\in\left\{
i+1,i+2,\ldots,m\right\}  $. Thus, $j\in\left\{  i+1,i+2,\ldots,m\right\}
\subseteq\left\{  1,2,\ldots,m\right\}  $. Combining this with $j\neq i$
(since $j>i$ (since $j\in\left\{  i+1,i+2,\ldots,m\right\}  $)), we obtain
$j\in\left\{  1,2,\ldots,m\right\}  \setminus\left\{  i\right\}  $. Hence,
$f\left(  v_{j}\right)  =f\left(  v_{j+1}\right)  $ (by
(\ref{pf.lem.Eqs.circuit.2})).
\par
Now, let us forget that we fixed $j$. We thus have proven that $f\left(
v_{j}\right)  =f\left(  v_{j+1}\right)  $ for every $j\in\left\{
i+1,i+2,\ldots,m\right\}  $. In other words, $f\left(  v_{i+1}\right)
=f\left(  v_{i+2}\right)  =\cdots=f\left(  v_{m+1}\right)  $. Hence, $f\left(
v_{i+1}\right)  =f\left(  v_{m+1}\right)  $. This proves
(\ref{pf.lem.Eqs.circuit.3b}).}. Now, (\ref{pf.lem.Eqs.circuit.3a}) yields%
\[
f\left(  v_{i}\right)  =f\left(  \underbrace{v_{1}}_{=v_{m+1}}\right)
=f\left(  v_{m+1}\right)  =f\left(  v_{i+1}\right)
\ \ \ \ \ \ \ \ \ \ \left(  \text{by (\ref{pf.lem.Eqs.circuit.3b})}\right)  .
\]

Moreover, $v_{i}$ and $v_{i+1}$ are elements of $V$ (since $\left(
v_{1},v_{2},\ldots,v_{m+1}\right)  $ is a list of elements of $V$). In other
words, $v_{i}\in V$ and $v_{i+1}\in V$. Hence, $\left(  v_{i},v_{i+1}\right)
\in V^{2}$.

Furthermore, $v_{i}\neq v_{i+1}$\ \ \ \ \footnote{\textit{Proof.} Assume the
contrary. Thus, $v_{i}=v_{i+1}$.
\par
Let us first assume (for the sake of contradiction) that $i=m$. Thus,
$v_{i}=v_{m}$. Also, from $i=m$, we obtain $v_{i+1}=v_{m+1}=v_{1}$. Hence,
$v_{m}=v_{i}=v_{i+1}=v_{1}$.
\par
The vertices $v_{1},v_{2},\ldots,v_{m}$ are pairwise distinct. In other words,
any two distinct elements $a$ and $b$ of $\left\{  1,2,\ldots,m\right\}  $
satisfy $v_{a}\neq v_{b}$. Applying this to $a=m$ and $b=1$, we obtain
$v_{m}\neq v_{1}$ (since $m$ and $1$ are distinct (since $m>2>1$)). This
contradicts $v_{m}=v_{1}$.
\par
This contradiction proves that our assumption (that $i=m$) was wrong. Hence,
we cannot have $i=m$. We thus have $i\neq m$. Combined with $i\in\left\{
1,2,\ldots,m\right\}  $, this yields $i\in\left\{  1,2,\ldots,m\right\}
\setminus\left\{  m\right\}  \subseteq\left\{  1,2,\ldots,m-1\right\}  $.
Thus, $i+1\in\left\{  2,3,\ldots,m\right\}  \subseteq\left\{  1,2,\ldots
,m\right\}  $.
\par
Now, recall that any two distinct elements $a$ and $b$ of $\left\{
1,2,\ldots,m\right\}  $ satisfy $v_{a}\neq v_{b}$. We can apply this to $a=i$
and $b=i+1$ (since $i+1\in\left\{  1,2,\ldots,m\right\}  $). Thus, we obtain
$v_{i}\neq v_{i+1}$. This contradicts $v_{i}=v_{i+1}$. This contradiction
shows that our assumption was wrong. Qed.}.

Now, the definition of $\operatorname*{Eqs}f$ shows that%
\begin{equation}
\operatorname*{Eqs}f=\left\{  \left\{  s,t\right\}  \ \mid\ \left(
s,t\right)  \in V^{2},\ s\neq t\text{ and }f\left(  s\right)  =f\left(
t\right)  \right\}  . \label{pf.lem.Eqs.circuit.Eqs-def}%
\end{equation}

However, we have $\left(  v_{i},v_{i+1}\right)  \in V^{2}$, $v_{i}\neq
v_{i+1}$ and $f\left(  v_{i}\right)  =f\left(  v_{i+1}\right)  $. Hence, the
set $\left\{  v_{i},v_{i+1}\right\}  $ has the form $\left\{  s,t\right\}  $
for some $\left(  s,t\right)  \in V^{2}$ satisfying $s\neq t$ and $f\left(
s\right)  =f\left(  t\right)  $ (namely, for $\left(  s,t\right)  =\left(
v_{i},v_{i+1}\right)  $). Thus,%
\[
\left\{  v_{i},v_{i+1}\right\}  \in\left\{  \left\{  s,t\right\}
\ \mid\ \left(  s,t\right)  \in V^{2},\ s\neq t\text{ and }f\left(  s\right)
=f\left(  t\right)  \right\}  =\operatorname*{Eqs}f
\]
(by (\ref{pf.lem.Eqs.circuit.Eqs-def})). Thus, $e=\left\{  v_{i}%
,v_{i+1}\right\}  \in\operatorname*{Eqs}f$.

But $e=\left\{  v_{i},v_{i+1}\right\}  \in E$ (by
(\ref{pf.lem.Eqs.circuit.inE})). Combining this with $e\in\operatorname*{Eqs}%
f$, we obtain $e\in E\cap\operatorname*{Eqs}f$. This proves Lemma
\ref{lem.Eqs.circuit}.
\end{proof}
\end{verlong}

\begin{verlong}
\begin{lemma}
\label{lem.Eqs.sum-aux}Let $\left(  V,B\right)  $ be a finite graph. Let
$\sim$ denote the equivalence relation $\sim_{\left(  V,B\right)  }$ (defined
as in Definition \ref{def.connectedness} \textbf{(a)}).

Let $Y$ be a set. A set $Y_{\sim}^{V}$ is defined (according to Definition
\ref{def.relquot.maps} \textbf{(b)}). Let $f:V\rightarrow Y$ be any map. Then,
we have the following logical equivalence of statements:%
\[
\left(  B\subseteq\operatorname*{Eqs}f\right)  \ \Longleftrightarrow\ \left(
f\in Y_{\sim}^{V}\right)  .
\]

\end{lemma}

\begin{proof}
[Proof of Lemma \ref{lem.Eqs.sum-aux}.]We have%
\begin{equation}
Y_{\sim}^{V}=\left\{  g\in Y^{V}\ \mid\ g\left(  x\right)  =g\left(  y\right)
\text{ for any }x\in V\text{ and }y\in V\text{ satisfying }x\sim y\right\}
\label{pf.lem.Eqs.sum.defN+Vsim}%
\end{equation}
(by the definition of $Y_{\sim}^{V}$).

The definition of $\operatorname*{Eqs}f$ shows that%
\begin{align}
\operatorname*{Eqs}f  &  =\left\{  \left\{  s,t\right\}  \ \mid\ \left(
s,t\right)  \in V^{2},\ s\neq t\text{ and }f\left(  s\right)  =f\left(
t\right)  \right\} \nonumber\\
&  =\left\{  \left\{  x,y\right\}  \ \mid\ \left(  x,y\right)  \in
V^{2},\ x\neq y\text{ and }f\left(  x\right)  =f\left(  y\right)  \right\}
\label{pf.lem.Eqs.sum.defEqs}%
\end{align}
(here, we renamed the index $\left(  s,t\right)  $ as $\left(  x,y\right)  $).

We shall now show the following logical implication:%
\begin{equation}
\left(  B\subseteq\operatorname*{Eqs}f\right)  \ \Longrightarrow\ \left(  f\in
Y_{\sim}^{V}\right)  \label{pf.lem.Eqs.sum.im1}%
\end{equation}

\textit{Proof of (\ref{pf.lem.Eqs.sum.im1}):} Assume that $B\subseteq
\operatorname*{Eqs}f$. We must show that $f\in Y_{\sim}^{V}$.

Let $x\in V$ and $y\in V$ be such that $x\sim y$. We have $x\sim y$. In other
words, $x\sim_{\left(  V,B\right)  }y$ (since $\sim$ is the equivalence
relation $\sim_{\left(  V,B\right)  }$). In other words, $x$ and $y$ are
connected in the graph $\left(  V,B\right)  $ (since $x\sim_{\left(
V,B\right)  }y$ holds if and only if $x$ and $y$ are connected in the graph
$\left(  V,B\right)  $ (by the definition of the relation $\sim_{\left(
V,B\right)  }$)). In other words, there exists a walk from $x$ to $y$ in
$\left(  V,B\right)  $ (since $x$ and $y$ are connected in $\left(
V,B\right)  $ if and only if there exists a walk from $x$ to $y$ in $\left(
V,B\right)  $ (by the definition of \textquotedblleft
connected\textquotedblright)). Let $\mathfrak{w}$ be this walk. Thus,
$\mathfrak{w}$ is a walk from $x$ to $y$ in $\left(  V,B\right)  $. In other
words, $\mathfrak{w}$ is a sequence $\left(  w_{0},w_{1},\ldots,w_{k}\right)
$ of elements of $V$ such that $w_{0}=x$ and $w_{k}=y$ and%
\begin{equation}
\left(  \left\{  w_{i},w_{i+1}\right\}  \in B\ \ \ \ \ \ \ \ \ \ \text{for
every }i\in\left\{  0,1,\ldots,k-1\right\}  \right)
\label{pf.lem.Eqs.sum.im1.pf.1}%
\end{equation}
(since a walk from $x$ to $y$ in $\left(  V,B\right)  $ is the same as a
sequence $\left(  w_{0},w_{1},\ldots,w_{k}\right)  $ of elements of $V$ such
that $w_{0}=x$ and $w_{k}=y$ and \newline$\left(  \left\{  w_{i}%
,w_{i+1}\right\}  \in B\ \ \ \ \ \ \ \ \ \ \text{for every }i\in\left\{
0,1,\ldots,k-1\right\}  \right)  $ (by the definition of a \textquotedblleft
walk\textquotedblright)). Consider this sequence $\left(  w_{0},w_{1}%
,\ldots,w_{k}\right)  $.

For every $i\in\left\{  0,1,\ldots,k-1\right\}  $, we have $f\left(
w_{i}\right)  =f\left(  w_{i+1}\right)  $\ \ \ \ \footnote{\textit{Proof.} Let
$i\in\left\{  0,1,\ldots,k-1\right\}  $. Thus, (\ref{pf.lem.Eqs.sum.im1.pf.1})
shows that
\[
\left\{  w_{i},w_{i+1}\right\}  \in B\subseteq\operatorname*{Eqs}f=\left\{
\left\{  s,t\right\}  \ \mid\ \left(  s,t\right)  \in V^{2},\ s\neq t\text{
and }f\left(  s\right)  =f\left(  t\right)  \right\}
\]
(by the definition of $\operatorname*{Eqs}f$). In other words, $\left\{
w_{i},w_{i+1}\right\}  =\left\{  s,t\right\}  $ for some $\left(  s,t\right)
\in V^{2}$ satisfying $s\neq t$ and $f\left(  s\right)  =f\left(  t\right)  $.
Consider this $\left(  s,t\right)  $.
\par
We have $f\left(  s\right)  =f\left(  t\right)  $. Therefore, set $g=f\left(
s\right)  =f\left(  t\right)  $. Then, $f\left(  \left\{  s,t\right\}
\right)  =\left\{  \underbrace{f\left(  s\right)  }_{=g},\underbrace{f\left(
t\right)  }_{=g}\right\}  =\left\{  g,g\right\}  =\left\{  g\right\}  $. Now,
$f\left(  \underbrace{w_{i}}_{\in\left\{  w_{i},w_{i+1}\right\}  =\left\{
s,t\right\}  }\right)  \in f\left(  \left\{  s,t\right\}  \right)  =\left\{
g\right\}  $, so that $f\left(  w_{i}\right)  =g$. Also, $f\left(
\underbrace{w_{i+1}}_{\in\left\{  w_{i},w_{i+1}\right\}  =\left\{
s,t\right\}  }\right)  \in f\left(  \left\{  s,t\right\}  \right)  =\left\{
g\right\}  $, so that $f\left(  w_{i+1}\right)  =g$. Hence, $f\left(
w_{i}\right)  =g=f\left(  w_{i+1}\right)  $, qed.}. In other words, $f\left(
w_{0}\right)  =f\left(  w_{1}\right)  =\cdots=f\left(  w_{k}\right)  $. Thus,
$f\left(  w_{0}\right)  =f\left(  w_{k}\right)  $. This rewrites as $f\left(
x\right)  =f\left(  y\right)  $ (since $w_{0}=x$ and $w_{k}=y$).

Now, let us forget that we fixed $x$ and $y$. We thus have shown that
$f\left(  x\right)  =f\left(  y\right)  $ for any $x\in V$ and $y\in V$
satisfying $x\sim y$. Hence, $f$ is a $g\in Y^{V}$ which satisfies $g\left(
x\right)  =g\left(  y\right)  $ for any $x\in V$ and $y\in V$ satisfying
$x\sim y$. In other words,%
\begin{align*}
f  &  \in\left\{  g\in Y^{V}\ \mid\ g\left(  x\right)  =g\left(  y\right)
\text{ for any }x\in V\text{ and }y\in V\text{ satisfying }x\sim y\right\} \\
&  =Y_{\sim}^{V}\ \ \ \ \ \ \ \ \ \ \left(  \text{by
(\ref{pf.lem.Eqs.sum.defN+Vsim})}\right)  .
\end{align*}
Thus, the implication (\ref{pf.lem.Eqs.sum.im1}) is proven.

Next, we shall show the following logical implication:%
\begin{equation}
\left(  f\in Y_{\sim}^{V}\right)  \ \Longrightarrow\ \left(  B\subseteq
\operatorname*{Eqs}f\right)  . \label{pf.lem.Eqs.sum.im2}%
\end{equation}

\textit{Proof of (\ref{pf.lem.Eqs.sum.im2}):} Assume that $f\in Y_{\sim}^{V}$.
We must show that $B\subseteq\operatorname*{Eqs}f$.

Since $\left(  V,B\right)  $ is a graph, we must have $B\subseteq\dbinom{V}%
{2}$.

Let $e\in B$. Then, $e\in B\subseteq\dbinom{V}{2}$. In other words, $e$ is a
$2$-element subset of $V$ (since $\dbinom{V}{2}$ is the set of all $2$-element
subsets of $V$). In other words, $e=\left\{  s,t\right\}  $ for two distinct
elements $s$ and $t$ of $V$. Consider these $s$ and $t$.

We have $\left\{  s,t\right\}  =e\in B$. Thus, $s\sim t$%
\ \ \ \ \footnote{\textit{Proof.} We have $s\in V$ and $t\in V$. Thus,
$\left(  s,t\right)  $ is a sequence of elements of $V$. Let us denote this
sequence $\left(  s,t\right)  $ by $\left(  p_{0},p_{1},\ldots,p_{\ell
}\right)  $. Thus, $\ell=1$, $p_{0}=s$ and $p_{1}=t$.
\par
We have $\ell=1$ and thus $p_{\ell}=p_{1}=t$.
\par
Let $i\in\left\{  0,1,\ldots,\ell-1\right\}  $. Thus, $i\geq0$ and
$i\leq\underbrace{\ell}_{=1}-1=1-1=0$. Combining $i\leq0$ and $i\geq0$, we
obtain $i=0$, so that $p_{i}=p_{0}=s$ and $p_{i+1}=p_{0+1}=p_{1}=t$. Now,
$\left\{  \underbrace{p_{i}}_{=s},\underbrace{p_{i+1}}_{=t}\right\}  =\left\{
s,t\right\}  \in B$.
\par
Now, let us forget that we fixed $i$. We thus have shown that $\left\{
p_{i},p_{i+1}\right\}  \in B$ for every $i\in\left\{  0,1,\ldots
,\ell-1\right\}  $. Thus, the sequence $\left(  p_{0},p_{1},\ldots,p_{\ell
}\right)  $ is a sequence of elements of $V$ satisfying $p_{0}=s$, $p_{\ell
}=t$ and $\left(  \left\{  p_{i},p_{i+1}\right\}  \in
B\ \ \ \ \ \ \ \ \ \ \text{for every }i\in\left\{  0,1,\ldots,\ell-1\right\}
\right)  $. In other words, the sequence $\left(  p_{0},p_{1},\ldots,p_{\ell
}\right)  $ is a sequence $\left(  w_{0},w_{1},\ldots,w_{k}\right)  $ of
elements of $V$ such that $w_{0}=s$ and $w_{k}=t$ and $\left(  \left\{
w_{i},w_{i+1}\right\}  \in B\ \ \ \ \ \ \ \ \ \ \text{for every }i\in\left\{
0,1,\ldots,k-1\right\}  \right)  $. In other words, the sequence $\left(
p_{0},p_{1},\ldots,p_{\ell}\right)  $ is a walk from $s$ to $t$ in $\left(
V,B\right)  $ (since a walk from $s$ to $t$ in $\left(  V,B\right)  $ is the
same as a sequence $\left(  w_{0},w_{1},\ldots,w_{k}\right)  $ of elements of
$V$ such that $w_{0}=s$ and $w_{k}=t$ and $\left(  \left\{  w_{i}%
,w_{i+1}\right\}  \in B\ \ \ \ \ \ \ \ \ \ \text{for every }i\in\left\{
0,1,\ldots,k-1\right\}  \right)  $ (by the definition of a \textquotedblleft
walk\textquotedblright)). Hence, there exists a walk from $s$ to $t$ in
$\left(  V,B\right)  $. In other words, $s$ and $t$ are connected in the graph
$\left(  V,B\right)  $ (since $s$ and $t$ are connected in $\left(
V,B\right)  $ if and only if there exists a walk from $s$ to $t$ in $\left(
V,B\right)  $ (by the definition of \textquotedblleft
connected\textquotedblright)). In other words, $s\sim_{\left(  V,B\right)  }t$
(since $s\sim_{\left(  V,B\right)  }t$ holds if and only if $s$ and $t$ are
connected in the graph $\left(  V,B\right)  $ (by the definition of the
relation $\sim_{\left(  V,B\right)  }$)). In other words, $s\sim t$ (since
$\sim$ is the equivalence relation $\sim_{\left(  V,B\right)  }$). Qed.}.

But $f\in Y_{\sim}^{V}=\left\{  g\in Y^{V}\ \mid\ g\left(  x\right)  =g\left(
y\right)  \text{ for any }x\in V\text{ and }y\in V\text{ satisfying }x\sim
y\right\}  $ (by (\ref{pf.lem.Eqs.sum.defN+Vsim})). In other words, $f$ is a
$g\in Y^{V}$ such that $g\left(  x\right)  =g\left(  y\right)  $ for any $x\in
V$ and $y\in V$ satisfying $x\sim y$. In other words, $f$ is an element of
$Y^{V}$ and satisfies%
\begin{equation}
f\left(  x\right)  =f\left(  y\right)  \ \ \ \ \ \ \ \ \ \ \text{for any }x\in
V\text{ and }y\in V\text{ satisfying }x\sim y. \label{pf.lem.Eqs.sum.im2.pf.3}%
\end{equation}

Applying (\ref{pf.lem.Eqs.sum.im2.pf.3}) to $x=s$ and $y=t$, we obtain
$f\left(  s\right)  =f\left(  t\right)  $ (since $s\sim t$).

We have $s\in V$ and $t\in V$. Thus, $\left(  s,t\right)  \in V^{2}$. Also,
$s$ and $t$ are distinct; thus, $s\neq t$. Hence, $\left(  s,t\right)  $ is an
element of $V^{2}$ satisfying $s\neq t$ and $f\left(  s\right)  =f\left(
t\right)  $. In other words, $\left(  s,t\right)  $ is an $\left(  x,y\right)
\in V^{2}$ satisfying $x\neq y$ and $f\left(  x\right)  =f\left(  y\right)  $.
Therefore, $\left\{  s,t\right\}  $ has the form $\left\{  x,y\right\}  $ for
some $\left(  x,y\right)  \in V^{2}$ satisfying $x\neq y$ and $f\left(
x\right)  =f\left(  y\right)  $ (namely, for $\left(  x,y\right)  =\left(
s,t\right)  $). In other words,%
\[
\left\{  s,t\right\}  \in\left\{  \left\{  x,y\right\}  \ \mid\ \left(
x,y\right)  \in V^{2},\ x\neq y\text{ and }f\left(  x\right)  =f\left(
y\right)  \right\}  =\operatorname*{Eqs}f
\]
(by (\ref{pf.lem.Eqs.sum.defEqs})). Thus, $e=\left\{  s,t\right\}
\in\operatorname*{Eqs}f$.

Now, let us forget that we fixed $e$. We thus have shown that $e\in
\operatorname*{Eqs}f$ for every $e\in B$. In other words, $B\subseteq
\operatorname*{Eqs}f$. This proves the implication (\ref{pf.lem.Eqs.sum.im2}).

Now, we can combine the two implications (\ref{pf.lem.Eqs.sum.im1}) and
(\ref{pf.lem.Eqs.sum.im2}). As a result, we obtain the equivalence $\left(
B\subseteq\operatorname*{Eqs}f\right)  \ \Longleftrightarrow\ \left(  f\in
Y_{\sim}^{V}\right)  $. This proves Lemma \ref{lem.Eqs.sum-aux}.
\end{proof}

The next lemma is a fundamental fact about counting:

\begin{lemma}
\label{lem.function-count}Let $W$ be a finite set. Let $\left(  C_{1}%
,C_{2},\ldots,C_{k}\right)  $ be a list of all elements of $W$ which contains
each of these elements exactly once. Let $Y$ be any set. Then, the map
\begin{align*}
Y^{W}  &  \rightarrow Y^{k},\\
f  &  \mapsto\left(  f\left(  C_{1}\right)  ,f\left(  C_{2}\right)
,\ldots,f\left(  C_{k}\right)  \right)
\end{align*}
is a bijection.
\end{lemma}

\begin{proof}
[Proof of Lemma \ref{lem.function-count}.]Let $\Phi$ denote the map%
\begin{align*}
Y^{W}  &  \rightarrow Y^{k},\\
f  &  \mapsto\left(  f\left(  C_{1}\right)  ,f\left(  C_{2}\right)
,\ldots,f\left(  C_{k}\right)  \right)  .
\end{align*}
We shall show that the map $\Phi$ is a bijection.

First, we notice that the map $\Phi$ is injective\footnote{\textit{Proof.} Let
$f$ and $g$ be two elements of $Y^{W}$ such that $\Phi\left(  f\right)
=\Phi\left(  g\right)  $. We shall show that $f=g$.
\par
Let $w\in W$. Recall that $\left(  C_{1},C_{2},\ldots,C_{k}\right)  $ is a
list of all elements of $W$. Thus, $W=\left\{  C_{1},C_{2},\ldots
,C_{k}\right\}  $. Now, $w\in W=\left\{  C_{1},C_{2},\ldots,C_{k}\right\}  $.
Hence, there exists some $i\in\left\{  1,2,\ldots,k\right\}  $ such that
$w=C_{i}$. Consider this $i$.
\par
Now, the definition of $\Phi\left(  f\right)  $ yields $\Phi\left(  f\right)
=\left(  f\left(  C_{1}\right)  ,f\left(  C_{2}\right)  ,\ldots,f\left(
C_{k}\right)  \right)  $. Hence,%
\begin{align*}
&  \left(  \text{the }i\text{-th entry of }\underbrace{\Phi\left(  f\right)
}_{=\left(  f\left(  C_{1}\right)  ,f\left(  C_{2}\right)  ,\ldots,f\left(
C_{k}\right)  \right)  }\right) \\
&  =\left(  \text{the }i\text{-th entry of }\left(  f\left(  C_{1}\right)
,f\left(  C_{2}\right)  ,\ldots,f\left(  C_{k}\right)  \right)  \right)
=f\left(  \underbrace{C_{i}}_{=w}\right)  =f\left(  w\right)  .
\end{align*}
The same argument (applied to $g$ instead of $f$) yields%
\[
\left(  \text{the }i\text{-th entry of }\Phi\left(  g\right)  \right)
=g\left(  w\right)  .
\]
Hence, $f\left(  w\right)  =\left(  \text{the }i\text{-th entry of
}\underbrace{\Phi\left(  f\right)  }_{=\Phi\left(  g\right)  }\right)
=\left(  \text{the }i\text{-th entry of }\Phi\left(  g\right)  \right)
=g\left(  w\right)  $.
\par
Now, let us forget that we fixed $w$. We thus have proven that $f\left(
w\right)  =g\left(  w\right)  $ for every $w\in W$. In other words, $f=g$.
\par
Let us now forget that we fixed $f$ and $g$. We thus have shown that if $f$
and $g$ are two elements of $Y^{W}$ such that $\Phi\left(  f\right)
=\Phi\left(  g\right)  $, then $f=g$. In other words, the map $\Phi$ is
injective. Qed.}. Now, let $s\in Y^{k}$ be arbitrary. Let us write $s$ in the
form $\left(  y_{1},y_{2},\ldots,y_{k}\right)  $. Thus, $s=\left(  y_{1}%
,y_{2},\ldots,y_{k}\right)  $.

We now define a map $f:W\rightarrow Y$ as follows: Let $w\in W$. Then, there
exists a unique $i\in\left\{  1,2,\ldots,k\right\}  $ such that $w=C_{i}$
(since $\left(  C_{1},C_{2},\ldots,C_{k}\right)  $ is a list of all elements
of $W$ which contains each of these elements exactly once). Consider this $i$.
Then, we set $f\left(  w\right)  =y_{i}$.

Thus, we have defined a map $f:W\rightarrow Y$. It is clear that if $w\in W$,
and if $i\in\left\{  1,2,\ldots,k\right\}  $ is such that $w=C_{i}$, then%
\begin{equation}
f\left(  w\right)  =y_{i} \label{pf.lem.function-count.1}%
\end{equation}
(by the definition of $f\left(  w\right)  $).

Now, every $i\in\left\{  1,2,\ldots,k\right\}  $ satisfies $f\left(
C_{i}\right)  =y_{i}$ (by (\ref{pf.lem.function-count.1}), applied to
$w=C_{i}$). Hence, $\left(  f\left(  C_{1}\right)  ,f\left(  C_{2}\right)
,\ldots,f\left(  C_{k}\right)  \right)  =\left(  y_{1},y_{2},\ldots
,y_{k}\right)  $. Now, the definition of $\Phi$ yields $\Phi\left(  f\right)
=\left(  f\left(  C_{1}\right)  ,f\left(  C_{2}\right)  ,\ldots,f\left(
C_{k}\right)  \right)  =\left(  y_{1},y_{2},\ldots,y_{k}\right)  =s$. Thus,
$s=\Phi\left(  \underbrace{f}_{\in Y^{W}}\right)  \in\Phi\left(  Y^{W}\right)
$.

Now, let us forget that we fixed $s$. We thus have proven that $s\in
\Phi\left(  Y^{W}\right)  $ for every $s\in Y^{k}$. In other words,
$Y^{k}\subseteq\Phi\left(  Y^{W}\right)  $. In other words, the map $\Phi$ is surjective.

Since the map $\Phi$ is both injective and surjective, we see that the map
$\Phi$ is bijective. In other words, the map $\Phi$ is a bijection. In other
words, the map
\begin{align*}
Y^{W}  &  \rightarrow Y^{k},\\
f  &  \mapsto\left(  f\left(  C_{1}\right)  ,f\left(  C_{2}\right)
,\ldots,f\left(  C_{k}\right)  \right)
\end{align*}
is a bijection (since this map is $\Phi$). Lemma \ref{lem.function-count} is
thus proven.
\end{proof}
\end{verlong}

\begin{lemma}
\label{lem.Eqs.sum}Let $\left(  V,B\right)  $ be a finite graph. Then,%
\[
\sum_{\substack{f:V\rightarrow\mathbb{N}_{+};\\B\subseteq\operatorname*{Eqs}%
f}}\mathbf{x}_{f}=p_{\lambda\left(  V,B\right)  }.
\]
(Here, $\mathbf{x}_{f}$ is defined as in Definition \ref{def.chromsym}
\textbf{(a)}, and the expression $\lambda\left(  V,B\right)  $ is understood
according to Definition \ref{def.connectedness} \textbf{(b)}.)
\end{lemma}

\begin{vershort}
\begin{proof}
[Proof of Lemma \ref{lem.Eqs.sum}.]Let $\left(  C_{1},C_{2},\ldots
,C_{k}\right)  $ be a list of all connected components of $\left(  V,B\right)
$, ordered such that $\left\vert C_{1}\right\vert \geq\left\vert
C_{2}\right\vert \geq\cdots\geq\left\vert C_{k}\right\vert $%
.\ \ \ \ \footnote{Every connected component of $\left(  V,B\right)  $ should
appear exactly once in this list.} Then, $\lambda\left(  V,B\right)  =\left(
\left\vert C_{1}\right\vert ,\left\vert C_{2}\right\vert ,\ldots,\left\vert
C_{k}\right\vert \right)  $ (by the definition of $\lambda\left(  V,B\right)
$). Hence, (\ref{eq.def.powersum2.finite-expression}) (applied to
$\lambda\left(  V,B\right)  $ and $\left\vert C_{i}\right\vert $ instead of
$\lambda$ and $\lambda_{i}$) shows that{}%
\begin{equation}
p_{\lambda\left(  V,B\right)  }=p_{\left\vert C_{1}\right\vert }p_{\left\vert
C_{2}\right\vert }\cdots p_{\left\vert C_{k}\right\vert }=\prod_{i=1}%
^{k}p_{\left\vert C_{i}\right\vert }. \label{pf.lem.Eqs.sum.short.p1}%
\end{equation}

However, for every $i\in\left\{  1,2,\ldots,k\right\}  $, we have
$p_{\left\vert C_{i}\right\vert }=\sum_{s\in\mathbb{N}_{+}}x_{s}^{\left\vert
C_{i}\right\vert }$ (by the definition of $p_{\left\vert C_{i}\right\vert }$).
Hence, (\ref{pf.lem.Eqs.sum.short.p1}) becomes%
\begin{align}
p_{\lambda\left(  V,B\right)  }  &  =\prod_{i=1}^{k}\underbrace{p_{\left\vert
C_{i}\right\vert }}_{=\sum_{s\in\mathbb{N}_{+}}x_{s}^{\left\vert
C_{i}\right\vert }}=\prod_{i=1}^{k}\ \ \sum_{s\in\mathbb{N}_{+}}%
x_{s}^{\left\vert C_{i}\right\vert }\nonumber\\
&  =\sum_{\left(  s_{1},s_{2},\ldots,s_{k}\right)  \in\left(  \mathbb{N}%
_{+}\right)  ^{k}}\ \ \prod_{i=1}^{k}x_{s_{i}}^{\left\vert C_{i}\right\vert }
\label{pf.lem.Eqs.sum.short.p2}%
\end{align}
(by the product rule).

The list $\left(  C_{1},C_{2},\ldots,C_{k}\right)  $ contains all connected
components of $\left(  V,B\right)  $, each exactly once. Thus, $V=\bigsqcup
_{i=1}^{k}C_{i}$.

We now define a map%
\[
\Phi:\left(  \mathbb{N}_{+}\right)  ^{k}\rightarrow\left\{  f:V\rightarrow
\mathbb{N}_{+}\ \mid\ B\subseteq\operatorname*{Eqs}f\right\}
\]
as follows: Given any $\left(  s_{1},s_{2},\ldots,s_{k}\right)  \in\left(
\mathbb{N}_{+}\right)  ^{k}$, we let $\Phi\left(  s_{1},s_{2},\ldots
,s_{k}\right)  $ be the map $V\rightarrow\mathbb{N}_{+}$ which sends every
$v\in V$ to $s_{i}$, where $i\in\left\{  1,2,\ldots,k\right\}  $ is such that
$v\in C_{i}$. (This is well-defined, because for every $v\in V$, there exists
a unique $i\in\left\{  1,2,\ldots,k\right\}  $ such that $v\in C_{i}$; this
follows from $V=\bigsqcup_{i=1}^{k}C_{i}$.) This map $\Phi$ is well-defined,
because for every $\left(  s_{1},s_{2},\ldots,s_{k}\right)  \in\left(
\mathbb{N}_{+}\right)  ^{k}$, the map $\Phi\left(  s_{1},s_{2},\ldots
,s_{k}\right)  $ actually belongs to $\left\{  f:V\rightarrow\mathbb{N}%
_{+}\ \mid\ B\subseteq\operatorname*{Eqs}f\right\}  $%
\ \ \ \ \footnote{\textit{Proof.} We just need to check that $B\subseteq
\operatorname*{Eqs}\left(  \Phi\left(  s_{1},s_{2},\ldots,s_{k}\right)
\right)  $. But this is easy: For every $\left\{  u,v\right\}  \in B$, the
vertices $u$ and $v$ of $\left(  V,B\right)  $ lie in one and the same
connected component $C_{i}$ of the graph $\left(  V,B\right)  $, and thus (by
the definition of $\Phi\left(  s_{1},s_{2},\ldots,s_{k}\right)  $) the map
$\Phi\left(  s_{1},s_{2},\ldots,s_{k}\right)  $ sends both of them to $s_{i}$;
but this shows that $\left\{  u,v\right\}  \in\operatorname*{Eqs}\left(
\Phi\left(  s_{1},s_{2},\ldots,s_{k}\right)  \right)  $.}.

A moment's thought reveals that the map $\Phi$ is injective\footnote{In fact,
we can reconstruct $\left(  s_{1},s_{2},\ldots,s_{k}\right)  \in\left(
\mathbb{N}_{+}\right)  ^{k}$ from its image $\Phi\left(  s_{1},s_{2}%
,\ldots,s_{k}\right)  $, because each $s_{i}$ is the image of any element of
$C_{i}$ under $\Phi\left(  s_{1},s_{2},\ldots,s_{k}\right)  $ (and this allows
us to compute $s_{i}$, since $C_{i}$ is nonempty).}. Let us now show that the
map $\Phi$ is surjective.

In order to show this, we must prove that every map $f:V\rightarrow
\mathbb{N}_{+}$ satisfying $B\subseteq\operatorname*{Eqs}f$ has the form
$\Phi\left(  s_{1},s_{2},\ldots,s_{k}\right)  $ for some $\left(  s_{1}%
,s_{2},\ldots,s_{k}\right)  \in\left(  \mathbb{N}_{+}\right)  ^{k}$. So let us
fix a map $f:V\rightarrow\mathbb{N}_{+}$ satisfying $B\subseteq
\operatorname*{Eqs}f$. We must find some $\left(  s_{1},s_{2},\ldots
,s_{k}\right)  \in\left(  \mathbb{N}_{+}\right)  ^{k}$ such that
$f=\Phi\left(  s_{1},s_{2},\ldots,s_{k}\right)  $.

We have $B\subseteq\operatorname*{Eqs}f$. Thus, for every $\left\{
s,t\right\}  \in B$, we have $\left\{  s,t\right\}  \in B\subseteq
\operatorname*{Eqs}f$ and thus%
\begin{equation}
f\left(  s\right)  =f\left(  t\right)  . \label{pf.lem.Eqs.sum.short.surj.1}%
\end{equation}

Now, if $x$ and $y$ are two elements of $V$ lying in the same connected
component of $\left(  V,B\right)  $, then%
\begin{equation}
f\left(  x\right)  =f\left(  y\right)  \label{pf.lem.Eqs.sum.short.surj.2}%
\end{equation}
\footnote{\textit{Proof of (\ref{pf.lem.Eqs.sum.short.surj.2}):} Let $x$ and
$y$ be two elements of $V$ lying in the same connected component of $\left(
V,B\right)  $. Then, the vertices $x$ and $y$ are connected by a walk in the
graph $\left(  V,B\right)  $ (by the definition of a \textquotedblleft
connected component\textquotedblright). Let $\left(  v_{0},v_{1},\ldots
,v_{j}\right)  $ be this walk (regarded as a sequence of vertices); thus,
$v_{0}=x$ and $v_{j}=y$. For every $i\in\left\{  0,1,\ldots,j-1\right\}  $, we
have $\left\{  v_{i},v_{i+1}\right\}  \in B$ (since $\left(  v_{0}%
,v_{1},\ldots,v_{j}\right)  $ is a walk in the graph $\left(  V,B\right)  $)
and thus $f\left(  v_{i}\right)  =f\left(  v_{i+1}\right)  $ (by
(\ref{pf.lem.Eqs.sum.short.surj.1}), applied to $\left(  s,t\right)  =\left(
v_{i},v_{i+1}\right)  $). In other words, $f\left(  v_{0}\right)  =f\left(
v_{1}\right)  =\cdots=f\left(  v_{j}\right)  $. Hence, $f\left(  v_{0}\right)
=f\left(  v_{j}\right)  $, so that $f\left(  \underbrace{x}_{=v_{0}}\right)
=f\left(  v_{0}\right)  =f\left(  \underbrace{v_{j}}_{=y}\right)  =f\left(
y\right)  $, qed.}. In other words, the map $f$ is constant on each connected
component of $\left(  V,B\right)  $. Thus, the map $f$ is constant on $C_{i}$
for each $i\in\left\{  1,2,\ldots,k\right\}  $ (since $C_{i}$ is a connected
component of $\left(  V,B\right)  $). Hence, for each $i\in\left\{
1,2,\ldots,k\right\}  $, we can define a positive integer $s_{i}\in
\mathbb{N}_{+}$ to be the image of any element of $C_{i}$ under $f$ (this is
well-defined, because $f$ is constant on $C_{i}$ and thus the choice of the
element does not matter). Define $s_{i}\in\mathbb{N}_{+}$ for each
$i\in\left\{  1,2,\ldots,k\right\}  $ this way. Thus, we have defined a
$k$-tuple $\left(  s_{1},s_{2},\ldots,s_{k}\right)  \in\left(  \mathbb{N}%
_{+}\right)  ^{k}$. Now, $f=\Phi\left(  s_{1},s_{2},\ldots,s_{k}\right)  $
(this follows immediately by recalling the definitions of $\Phi$ and $s_{i}$).

Let us now forget that we fixed $f$. We thus have shown that for every map
$f:V\rightarrow\mathbb{N}_{+}$ satisfying $B\subseteq\operatorname*{Eqs}f$,
there exists some $\left(  s_{1},s_{2},\ldots,s_{k}\right)  \in\left(
\mathbb{N}_{+}\right)  ^{k}$ such that $f=\Phi\left(  s_{1},s_{2},\ldots
,s_{k}\right)  $. In other words, the map $\Phi$ is surjective. Since $\Phi$
is both injective and surjective, we conclude that $\Phi$ is a bijection.

Moreover, it is straightforward to see that every $k$-tuple $\left(
s_{1},s_{2},\ldots,s_{k}\right)  \in\left(  \mathbb{N}_{+}\right)  ^{k}$
satisfies%
\begin{equation}
\mathbf{x}_{\Phi\left(  s_{1},s_{2},\ldots,s_{k}\right)  }=\prod_{i=1}%
^{k}x_{s_{i}}^{\left\vert C_{i}\right\vert } \label{pf.lem.Eqs.sum.short.4}%
\end{equation}
(by the definitions of $\mathbf{x}_{\Phi\left(  s_{1},s_{2},\ldots
,s_{k}\right)  }$ and of $\Phi$). Now,%
\begin{align*}
&  \sum_{\substack{f:V\rightarrow\mathbb{N}_{+};\\B\subseteq
\operatorname*{Eqs}f}}\mathbf{x}_{f}\\
&  =\sum_{\left(  s_{1},s_{2},\ldots,s_{k}\right)  \in\left(  \mathbb{N}%
_{+}\right)  ^{k}}\underbrace{\mathbf{x}_{\Phi\left(  s_{1},s_{2},\ldots
,s_{k}\right)  }}_{\substack{=\prod_{i=1}^{k}x_{s_{i}}^{\left\vert
C_{i}\right\vert }\\\text{(by (\ref{pf.lem.Eqs.sum.short.4}))}}}\\
&  \ \ \ \ \ \ \ \ \ \ \ \ \ \ \ \ \ \ \ \ \left(
\begin{array}
[c]{c}%
\text{here, we have substituted }\Phi\left(  s_{1},s_{2},\ldots,s_{k}\right)
\text{ for }f\text{ in the sum,}\\
\text{since the map }\Phi:\left(  \mathbb{N}_{+}\right)  ^{k}\rightarrow
\left\{  f:V\rightarrow\mathbb{N}_{+}\ \mid\ B\subseteq\operatorname*{Eqs}%
f\right\} \\
\text{is a bijection}%
\end{array}
\right) \\
&  =\sum_{\left(  s_{1},s_{2},\ldots,s_{k}\right)  \in\left(  \mathbb{N}%
_{+}\right)  ^{k}}\ \ \prod_{i=1}^{k}x_{s_{i}}^{\left\vert C_{i}\right\vert
}=p_{\lambda\left(  V,B\right)  }\ \ \ \ \ \ \ \ \ \ \left(  \text{by
(\ref{pf.lem.Eqs.sum.short.p2})}\right)  .
\end{align*}
This proves Lemma \ref{lem.Eqs.sum}.
\end{proof}
\end{vershort}

\begin{verlong}
\begin{proof}
[Proof of Lemma \ref{lem.Eqs.sum}.]Let $\sim$ denote the equivalence relation
$\sim_{\left(  V,B\right)  }$ (defined as in Definition
\ref{def.connectedness} \textbf{(a)}). The connected components of $\left(
V,B\right)  $ are the $\sim_{\left(  V,B\right)  }$-equivalence classes
(because this is how the connected components of $\left(  V,B\right)  $ are
defined). In other words, the connected components of $\left(  V,B\right)  $
are the $\sim$-equivalence classes (since $\sim$ is the relation
$\sim_{\left(  V,B\right)  }$). In other words, the connected components of
$\left(  V,B\right)  $ are the elements of $V/\left(  \sim\right)  $ (since
the elements of $V/\left(  \sim\right)  $ are the $\sim$-equivalence classes
(by the definition of $V/\left(  \sim\right)  $)).

A set $\left(  \mathbb{N}_{+}\right)  _{\sim}^{V}$ is defined (according to
Definition \ref{def.relquot.maps} \textbf{(b)}).

Proposition \ref{prop.relquot.uniprop} \textbf{(b)} (applied to $X=V$ and
$Y=\mathbb{N}_{+}$) shows that the map\footnote{Here, the map $\pi_{V}$ is
defined as in Definition \ref{def.relquot}.}%
\[
\left(  \mathbb{N}_{+}\right)  ^{V/\left(  \sim\right)  }\rightarrow\left(
\mathbb{N}_{+}\right)  _{\sim}^{V},\ \ \ \ \ \ \ \ \ \ f\mapsto f\circ\pi_{V}%
\]
is a bijection.

For every map $f:V\rightarrow\mathbb{N}_{+}$, we have the following
equivalence:%
\begin{equation}
\left(  B\subseteq\operatorname*{Eqs}f\right)  \ \Longleftrightarrow\ \left(
f\in\left(  \mathbb{N}_{+}\right)  _{\sim}^{V}\right)
\label{pf.lem.Eqs.sum.equivalence}%
\end{equation}
(according to Lemma \ref{lem.Eqs.sum-aux}, applied to $Y=\mathbb{N}_{+}$).
Thus, we have the following equality of summation signs:
\[
\sum_{\substack{f:V\rightarrow\mathbb{N}_{+};\\B\subseteq\operatorname*{Eqs}%
f}}=\sum_{\substack{f:V\rightarrow\mathbb{N}_{+};\\f\in\left(  \mathbb{N}%
_{+}\right)  _{\sim}^{V}}}=\sum_{\substack{f\in\left(  \mathbb{N}_{+}\right)
^{V};\\f\in\left(  \mathbb{N}_{+}\right)  _{\sim}^{V}}}=\sum_{f\in\left(
\mathbb{N}_{+}\right)  _{\sim}^{V}}%
\]
(since $\left(  \mathbb{N}_{+}\right)  _{\sim}^{V}$ is a subset of $\left(
\mathbb{N}_{+}\right)  ^{V}$). Hence,%
\begin{equation}
\underbrace{\sum_{\substack{f:V\rightarrow\mathbb{N}_{+};\\B\subseteq
\operatorname*{Eqs}f}}}_{=\sum_{f\in\left(  \mathbb{N}_{+}\right)  _{\sim}%
^{V}}}\mathbf{x}_{f}=\sum_{f\in\left(  \mathbb{N}_{+}\right)  _{\sim}^{V}%
}\mathbf{x}_{f}=\sum_{f\in\left(  \mathbb{N}_{+}\right)  ^{V/\left(
\sim\right)  }}\mathbf{x}_{f\circ\pi_{V}} \label{pf.lem.Eqs.sum.1}%
\end{equation}
(here, we have substituted $f\circ\pi_{V}$ for $f$ in the sum, since the map
$\left(  \mathbb{N}_{+}\right)  ^{V/\left(  \sim\right)  }\rightarrow\left(
\mathbb{N}_{+}\right)  _{\sim}^{V},\ f\mapsto f\circ\pi_{V}$ is a bijection).

Now, let $\left(  C_{1},C_{2},\ldots,C_{k}\right)  $ be a list of all
connected components of $\left(  V,B\right)  $, ordered such that $\left\vert
C_{1}\right\vert \geq\left\vert C_{2}\right\vert \geq\cdots\geq\left\vert
C_{k}\right\vert $.\ \ \ \ \footnote{Every connected component of $\left(
V,B\right)  $ should appear exactly once in this list.} Then, $\left(
\left\vert C_{1}\right\vert ,\left\vert C_{2}\right\vert ,\ldots,\left\vert
C_{k}\right\vert \right)  $ is the list of the sizes of all connected
components of $\left(  V,B\right)  $, in weakly decreasing order (since
$\left\vert C_{1}\right\vert \geq\left\vert C_{2}\right\vert \geq\cdots
\geq\left\vert C_{k}\right\vert $). In other words, $\left(  \left\vert
C_{1}\right\vert ,\left\vert C_{2}\right\vert ,\ldots,\left\vert
C_{k}\right\vert \right)  $ is $\lambda\left(  V,B\right)  $ (since
$\lambda\left(  V,B\right)  $ is the list of the sizes of all connected
components of $\left(  V,B\right)  $, in weakly decreasing order (by the
definition of $\lambda\left(  V,B\right)  $)). In other words, $\lambda\left(
V,B\right)  =\left(  \left\vert C_{1}\right\vert ,\left\vert C_{2}\right\vert
,\ldots,\left\vert C_{k}\right\vert \right)  $. Thus,
(\ref{eq.def.powersum2.finite-expression}) (applied to $\lambda\left(
V,B\right)  $ and $\left\vert C_{i}\right\vert $ instead of $\lambda$ and
$\lambda_{i}$) shows that{}%
\begin{equation}
p_{\lambda\left(  V,B\right)  }=p_{\left\vert C_{1}\right\vert }p_{\left\vert
C_{2}\right\vert }\cdots p_{\left\vert C_{k}\right\vert }=\prod_{i=1}%
^{k}p_{\left\vert C_{i}\right\vert }. \label{pf.lem.Eqs.sum.p1}%
\end{equation}

However, for every $i\in\left\{  1,2,\ldots,k\right\}  $, we have
\begin{equation}
p_{\left\vert C_{i}\right\vert }=\sum_{s\in\mathbb{N}_{+}}x_{s}^{\left\vert
C_{i}\right\vert } \label{pf.lem.Eqs.sum.pCi=}%
\end{equation}
\footnote{\textit{Proof.} Let $i\in\left\{  1,2,\ldots,k\right\}  $. Then,
$C_{i}$ is a connected component of $\left(  V,B\right)  $ (since $\left(
C_{1},C_{2},\ldots,C_{k}\right)  $ is a list of all connected components of
$\left(  V,B\right)  $). Hence, $C_{i}$ is a nonempty subset of $V$ (since
every connected component of $\left(  V,B\right)  $ is a nonempty subset of
$V$). Hence, $\left\vert C_{i}\right\vert $ is a positive integer. Thus,
(\ref{eq.def.powersum.pn}) (applied to $n=\left\vert C_{i}\right\vert $) shows
that $p_{\left\vert C_{i}\right\vert }=\underbrace{\sum_{j\geq1}}_{=\sum
_{j\in\mathbb{N}_{+}}}x_{j}^{\left\vert C_{i}\right\vert }=\sum_{j\in
\mathbb{N}_{+}}x_{j}^{\left\vert C_{i}\right\vert }=\sum_{s\in\mathbb{N}_{+}%
}x_{s}^{\left\vert C_{i}\right\vert }$ (here, we have renamed the summation
index $j$ as $s$). Qed.}. Hence, (\ref{pf.lem.Eqs.sum.p1}) becomes%
\begin{align}
p_{\lambda\left(  V,B\right)  }  &  =\prod_{i=1}^{k}\underbrace{p_{\left\vert
C_{i}\right\vert }}_{\substack{=\sum_{s\in\mathbb{N}_{+}}x_{s}^{\left\vert
C_{i}\right\vert }\\\text{(by (\ref{pf.lem.Eqs.sum.pCi=}))}}}=\prod_{i=1}%
^{k}\ \ \sum_{s\in\mathbb{N}_{+}}x_{s}^{\left\vert C_{i}\right\vert
}\nonumber\\
&  =\sum_{\left(  s_{1},s_{2},\ldots,s_{k}\right)  \in\left(  \mathbb{N}%
_{+}\right)  ^{k}}\ \ \prod_{i=1}^{k}x_{s_{i}}^{\left\vert C_{i}\right\vert }
\label{pf.lem.Eqs.sum.p2}%
\end{align}
(by the product rule).

Recall that $\left(  C_{1},C_{2},\ldots,C_{k}\right)  $ is a list of all
connected components of $\left(  V,B\right)  $. In other words, $\left(
C_{1},C_{2},\ldots,C_{k}\right)  $ is a list of all elements of $V/\left(
\sim\right)  $ (since the elements of $V/\left(  \sim\right)  $ are the
connected components of $\left(  V,B\right)  $). Moreover, every element of
$V/\left(  \sim\right)  $ appears exactly once in this list $\left(
C_{1},C_{2},\ldots,C_{k}\right)  $ (since the entries of the list $\left(
C_{1},C_{2},\ldots,C_{k}\right)  $ are pairwise distinct\footnote{since every
connected component of $\left(  V,B\right)  $ appears exactly once in this
list}). Thus, $\left(  C_{1},C_{2},\ldots,C_{k}\right)  $ is a list of all
elements of $V/\left(  \sim\right)  $, and contains each of these elements
exactly once. Hence, the map%
\begin{align*}
\left(  \mathbb{N}_{+}\right)  ^{V/\left(  \sim\right)  }  &  \rightarrow
\left(  \mathbb{N}_{+}\right)  ^{k},\\
f  &  \mapsto\left(  f\left(  C_{1}\right)  ,f\left(  C_{2}\right)
,\ldots,f\left(  C_{k}\right)  \right)
\end{align*}
is a bijection (by Lemma \ref{lem.function-count}, applied to $W=V/\left(
\sim\right)  $ and $Y=\mathbb{N}_{+}$).

For every $\gamma\in V/\left(  \sim\right)  $, we have%
\begin{equation}
\pi_{V}^{-1}\left(  \gamma\right)  =\gamma\label{pf.lem.Eqs.sum.pi-1}%
\end{equation}
\footnote{\textit{Proof of (\ref{pf.lem.Eqs.sum.pi-1}):} Let $\gamma\in
V/\left(  \sim\right)  $. Thus, $\gamma$ is an element of $V/\left(
\sim\right)  $. In other words, $\gamma$ is an $\sim$-equivalence class (since
the elements of $V/\left(  \sim\right)  $ are the $\sim$-equivalence classes).
In particular, $\gamma\subseteq V$.
\par
Let $x\in\gamma$. Then, $x\in\gamma\subseteq V$. The $\sim$-equivalence class
of $x$ must be $\gamma$ (since $x$ lies in the $\sim$-equivalence class
$\gamma$ (since $x\in\gamma$)). In other words, $\left[  x\right]  _{\sim}$
must be $\gamma$ (since $\left[  x\right]  _{\sim}$ is the $\sim$-equivalence
class of $x$). In other words, $\left[  x\right]  _{\sim}=\gamma$. Now, the
definition of $\pi_{V}$ yields $\pi_{V}\left(  x\right)  =\left[  x\right]
_{\sim}=\gamma$. Hence, $x\in\pi_{V}^{-1}\left(  \gamma\right)  $.
\par
Let us now forget that we fixed $x$. We thus have shown that $x\in\pi_{V}%
^{-1}\left(  \gamma\right)  $ for every $x\in\gamma$. In other words,
$\gamma\subseteq\pi_{V}^{-1}\left(  \gamma\right)  $.
\par
Let now $y\in\pi_{V}^{-1}\left(  \gamma\right)  $. Thus, $y\in V$ and $\pi
_{V}\left(  y\right)  =\gamma$. The definition of $\pi_{V}$ yields $\pi
_{V}\left(  y\right)  =\left[  y\right]  _{\sim}$. Thus, $\gamma=\pi
_{V}\left(  y\right)  =\left[  y\right]  _{\sim}$. Hence, $\gamma$ is the
$\sim$-equivalence class of $y$ (since $\left[  y\right]  _{\sim}$ is the
$\sim$-equivalence class of $y$). Consequently, $y$ must belong to $\gamma$.
In other words, $y\in\gamma$.
\par
Let us now forget that we fixed $y$. We thus have shown that $y\in\gamma$ for
every $y\in\pi_{V}^{-1}\left(  \gamma\right)  $. In other words, $\pi_{V}%
^{-1}\left(  \gamma\right)  \subseteq\gamma$. Combining this with
$\gamma\subseteq\pi_{V}^{-1}\left(  \gamma\right)  $, we obtain $\pi_{V}%
^{-1}\left(  \gamma\right)  =\gamma$. This proves (\ref{pf.lem.Eqs.sum.pi-1}%
).}.

Also, the map $\left\{  1,2,\ldots,k\right\}  \rightarrow V/\left(
\sim\right)  ,\ i\mapsto C_{i}$ is a bijection (since $\left(  C_{1}%
,C_{2},\ldots,C_{k}\right)  $ is a list of all elements of $V/\left(
\sim\right)  $, and contains each of these elements exactly once).

We have%
\begin{equation}
\mathbf{x}_{f\circ\pi_{V}}=\prod_{i=1}^{k}x_{f\left(  C_{i}\right)
}^{\left\vert C_{i}\right\vert }\ \ \ \ \ \ \ \ \ \ \text{for every }%
f\in\left(  \mathbb{N}_{+}\right)  ^{V/\left(  \sim\right)  }
\label{pf.lem.Eqs.sum.3}%
\end{equation}
\footnote{\textit{Proof of (\ref{pf.lem.Eqs.sum.3}):} Let $f\in\left(
\mathbb{N}_{+}\right)  ^{V/\left(  \sim\right)  }$. Then, the definition of
$\mathbf{x}_{f\circ\pi_{V}}$ yields%
\begin{align*}
\mathbf{x}_{f\circ\pi_{V}}  &  =\prod_{v\in V}x_{\left(  f\circ\pi_{V}\right)
\left(  v\right)  }=\prod_{\gamma\in V/\left(  \sim\right)  }%
\ \ \underbrace{\prod_{\substack{v\in V;\\\pi_{V}\left(  v\right)  =\gamma}%
}}_{\substack{=\prod_{v\in\pi_{V}^{-1}\left(  \gamma\right)  }=\prod
_{v\in\gamma}\\\text{(since }\pi_{V}^{-1}\left(  \gamma\right)  =\gamma
\\\text{(by (\ref{pf.lem.Eqs.sum.pi-1})))}}}\underbrace{x_{\left(  f\circ
\pi_{V}\right)  \left(  v\right)  }}_{\substack{=x_{f\left(  \gamma\right)
}\\\text{(since }\left(  f\circ\pi_{V}\right)  \left(  v\right)  =f\left(
\pi_{V}\left(  v\right)  \right)  =f\left(  \gamma\right)  \\\text{(since }%
\pi_{V}\left(  v\right)  =\gamma\text{))}}}\\
&  \ \ \ \ \ \ \ \ \ \ \left(
\begin{array}
[c]{c}%
\text{because for every }v\in V\text{, there exists a unique }\gamma\in
V/\left(  \sim\right) \\
\text{such that }\pi_{V}\left(  v\right)  =\gamma\text{ (since }\pi_{V}\text{
is a map }V\rightarrow V/\left(  \sim\right)  \text{)}%
\end{array}
\right) \\
&  =\prod_{\gamma\in V/\left(  \sim\right)  }\ \ \underbrace{\prod_{v\in
\gamma}x_{f\left(  \gamma\right)  }}_{=x_{f\left(  \gamma\right)
}^{\left\vert \gamma\right\vert }}=\prod_{\gamma\in V/\left(  \sim\right)
}x_{f\left(  \gamma\right)  }^{\left\vert \gamma\right\vert }=\prod
_{i\in\left\{  1,2,\ldots,k\right\}  }x_{f\left(  C_{i}\right)  }^{\left\vert
C_{i}\right\vert }%
\end{align*}
(here, we have substituted $C_{i}$ for $\gamma$ in the product, since the map
$\left\{  1,2,\ldots,k\right\}  \rightarrow V/\left(  \sim\right)  ,\ i\mapsto
C_{i}$ is a bijection). Thus,%
\[
\mathbf{x}_{f\circ\pi_{V}}=\underbrace{\prod_{i\in\left\{  1,2,\ldots
,k\right\}  }}_{=\prod_{i=1}^{k}}x_{f\left(  C_{i}\right)  }^{\left\vert
C_{i}\right\vert }=\prod_{i=1}^{k}x_{f\left(  C_{i}\right)  }^{\left\vert
C_{i}\right\vert }.
\]
This proves (\ref{pf.lem.Eqs.sum.3}).}.

Now, (\ref{pf.lem.Eqs.sum.1}) becomes%
\begin{align*}
\sum_{\substack{f:V\rightarrow\mathbb{N}_{+};\\B\subseteq\operatorname*{Eqs}%
f}}\mathbf{x}_{f}  &  =\sum_{f\in\left(  \mathbb{N}_{+}\right)  ^{V/\left(
\sim\right)  }}\underbrace{\mathbf{x}_{f\circ\pi_{V}}}_{\substack{=\prod
_{i=1}^{k}x_{f\left(  C_{i}\right)  }^{\left\vert C_{i}\right\vert
}\\\text{(by (\ref{pf.lem.Eqs.sum.3}))}}}=\sum_{f\in\left(  \mathbb{N}%
_{+}\right)  ^{V/\left(  \sim\right)  }}\ \ \prod_{i=1}^{k}x_{f\left(
C_{i}\right)  }^{\left\vert C_{i}\right\vert }\\
&  =\sum_{\left(  s_{1},s_{2},\ldots,s_{k}\right)  \in\left(  \mathbb{N}%
_{+}\right)  ^{k}}\ \ \prod_{i=1}^{k}x_{s_{i}}^{\left\vert C_{i}\right\vert }%
\end{align*}
(here, we have substituted $\left(  s_{1},s_{2},\ldots,s_{k}\right)  $ for
$\left(  f\left(  C_{1}\right)  ,f\left(  C_{2}\right)  ,\ldots,f\left(
C_{k}\right)  \right)  $ in the sum, since the map $\left(  \mathbb{N}%
_{+}\right)  ^{V/\left(  \sim\right)  }\rightarrow\left(  \mathbb{N}%
_{+}\right)  ^{k},\ f\mapsto\left(  f\left(  C_{1}\right)  ,f\left(
C_{2}\right)  ,\ldots,f\left(  C_{k}\right)  \right)  $ is a bijection).
Comparing this with (\ref{pf.lem.Eqs.sum.p2}), we obtain $\sum
_{\substack{f:V\rightarrow\mathbb{N}_{+};\\B\subseteq\operatorname*{Eqs}%
f}}\mathbf{x}_{f}=p_{\lambda\left(  V,B\right)  }$. This proves Lemma
\ref{lem.Eqs.sum}.
\end{proof}
\end{verlong}

\begin{lemma}
\label{lem.BC.nonempty}Let $G=\left(  V,E\right)  $ be a finite graph. Let $X$
be a totally ordered set. Let $\ell:E\rightarrow X$ be a labeling function.
Let $K$ be a broken circuit of $G$. Then, $K\neq\varnothing$.
\end{lemma}

\begin{vershort}
\begin{proof}
[Proof of Lemma \ref{lem.BC.nonempty}.]The set $K$ is a broken circuit of $G$,
and thus is a circuit of $G$ with an edge removed (by the definition of a
broken circuit). Thus, the set $K$ contains at least $1$ edge (since every
circuit of $G$ contains at least $2$ edges). This proves Lemma
\ref{lem.BC.nonempty}.
\end{proof}
\end{vershort}

\begin{verlong}
\begin{proof}
[Proof of Lemma \ref{lem.BC.nonempty}.]The set $K$ is a broken circuit of $G$.
In other words, the set $K$ is a subset of $E$ having the form $C\setminus
\left\{  e\right\}  $, where $C$ is a circuit of $G$, and where $e$ is the
unique edge in $C$ having maximum label (among the edges in $C$%
)\ \ \ \ \footnote{because a broken circuit of $G$ is the same as a subset of
$E$ having the form $C\setminus\left\{  e\right\}  $, where $C$ is a circuit
of $G$, and where $e$ is the unique edge in $C$ having maximum label (among
the edges in $C$) (by the definition of a \textquotedblleft broken
circuit\textquotedblright)}. Consider this $C$ and this $e$. Thus, we have the
following facts:

\begin{itemize}
\item The set $C$ is a circuit of $G$.

\item The element $e$ is the unique edge in $C$ having maximum label (among
the edges in $C$).

\item We have $K=C\setminus\left\{  e\right\}  $.
\end{itemize}

Now, assume (for the sake of contradiction) that $K=\varnothing$. Consider the
map $\operatorname*{id}:V\rightarrow V$. Then, $\operatorname*{Eqs}%
\operatorname*{id}=\varnothing$\ \ \ \ \footnote{\textit{Proof.} Let
$f\in\operatorname*{Eqs}\operatorname*{id}$. Thus,%
\begin{equation}
f\in\operatorname*{Eqs}\operatorname*{id}=\left\{  \left\{  s,t\right\}
\ \mid\ \left(  s,t\right)  \in V^{2},\ s\neq t\text{ and }\operatorname*{id}%
\left(  s\right)  =\operatorname*{id}\left(  t\right)  \right\} \nonumber
\end{equation}
(by the definition of $\operatorname*{Eqs}\operatorname*{id}$). In other
words, $f$ has the form $\left\{  s,t\right\}  $ for some $\left(  s,t\right)
\in V^{2}$ satisfying $s\neq t$ and $\operatorname*{id}\left(  s\right)
=\operatorname*{id}\left(  t\right)  $. Consider this $\left(  s,t\right)  $.
We have $s=\operatorname*{id}\left(  s\right)  =\operatorname*{id}\left(
t\right)  =t$; but this contradicts $s\neq t$.
\par
Now, let us forget that we fixed $f$. We thus have obtained a contradiction
for every $f\in\operatorname*{Eqs}\operatorname*{id}$. Thus, there exists no
$f\in\operatorname*{Eqs}\operatorname*{id}$. In other words,
$\operatorname*{Eqs}\operatorname*{id}$ is the empty set. Thus,
$\operatorname*{Eqs}\operatorname*{id}=\varnothing$, qed.}. But $C\setminus
\left\{  e\right\}  =K=\varnothing\subseteq\varnothing=\operatorname*{Eqs}%
\operatorname*{id}$. Hence, Lemma \ref{lem.Eqs.circuit} (applied to $V$ and
$\operatorname*{id}$ instead of $X$ and $f$) yields $e\in E\cap
\underbrace{\operatorname*{Eqs}\operatorname*{id}}_{=\varnothing}%
=E\cap\varnothing=\varnothing$. Thus, the set $\varnothing$ has at least one
element (namely, $e$). This contradicts the fact that this set $\varnothing$
is empty. This contradiction shows that our assumption (that $K=\varnothing$)
was wrong. Hence, we cannot have $K=\varnothing$. We thus have $K\neq
\varnothing$. This proves Lemma \ref{lem.BC.nonempty}.
\end{proof}
\end{verlong}

\subsection{Alternating sums}

We shall now come to less simple lemmas.

\begin{definition}
\label{def.iverson}We shall use the so-called \emph{Iverson bracket notation}:
If $\mathcal{S}$ is any logical statement, then $\left[  \mathcal{S}\right]  $
shall mean the integer $%
\begin{cases}
1, & \text{if }\mathcal{S}\text{ is true;}\\
0, & \text{if }\mathcal{S}\text{ is false}%
\end{cases}
$.
\end{definition}

The following lemma is probably the most crucial one in this paper:

\begin{lemma}
\label{lem.NBCm.moeb}Let $G=\left(  V,E\right)  $ be a finite graph. Let $X$
be a totally ordered set. Let $\ell:E\rightarrow X$ be a labeling function.
Let $\mathfrak{K}$ be some set of broken circuits of $G$ (not necessarily
containing all of them). Let $a_{K}$ be an element of $\mathbf{k}$ for every
$K\in\mathfrak{K}$.

Let $Y$ be any set. Let $f:V\rightarrow Y$ be any map. Then,%
\[
\sum_{B\subseteq E\cap\operatorname*{Eqs}f}\left(  -1\right)  ^{\left\vert
B\right\vert }\prod_{\substack{K\in\mathfrak{K};\\K\subseteq B}}a_{K}=\left[
E\cap\operatorname*{Eqs}f=\varnothing\right]  .
\]

\end{lemma}

\begin{vershort}
\begin{proof}
[Proof of Lemma \ref{lem.NBCm.moeb}.]We WLOG assume that $E\cap
\operatorname*{Eqs}f\neq\varnothing$ (since otherwise, the claim is
obvious\footnote{In (slightly) more detail: If $E\cap\operatorname*{Eqs}%
f=\varnothing$, then the sum $\sum_{B\subseteq E\cap\operatorname*{Eqs}%
f}\left(  -1\right)  ^{\left\vert B\right\vert }\prod_{\substack{K\in
\mathfrak{K};\\K\subseteq B}}a_{K}$ has only one addend (namely, the addend
for $B=\varnothing$), and thus simplifies to
\begin{align*}
\underbrace{\left(  -1\right)  ^{\left\vert \varnothing\right\vert }%
}_{=\left(  -1\right)  ^{0}=1}\underbrace{\prod_{\substack{K\in\mathfrak{K}%
;\\K\subseteq\varnothing}}}_{=\prod_{\substack{K\in\mathfrak{K}%
;\\K=\varnothing}}}a_{K}  &  =\prod_{\substack{K\in\mathfrak{K}%
;\\K=\varnothing}}a_{K}=\left(  \text{empty product}\right)
\ \ \ \ \ \ \ \ \ \ \left(
\begin{array}
[c]{c}%
\text{since no }K\in\mathfrak{K}\text{ satisfies }K=\varnothing\\
\text{(by Lemma \ref{lem.BC.nonempty})}%
\end{array}
\right) \\
&  =1=\left[  E\cap\operatorname*{Eqs}f=\varnothing\right]  .
\end{align*}
}). Thus, $\left[  E\cap\operatorname*{Eqs}f=\varnothing\right]  =0$.

Pick any $d\in E\cap\operatorname*{Eqs}f$ with maximum $\ell\left(  d\right)
$ (among all $d\in E\cap\operatorname*{Eqs}f$). (This is clearly possible,
since $E\cap\operatorname*{Eqs}f\neq\varnothing$.) Define two subsets
$\mathcal{U}$ and $\mathcal{V}$ of $\mathcal{P}\left(  E\cap
\operatorname*{Eqs}f\right)  $ as follows:%
\begin{align*}
\mathcal{U}  &  =\left\{  F\in\mathcal{P}\left(  E\cap\operatorname*{Eqs}%
f\right)  \ \mid\ d\notin F\right\}  ;\\
\mathcal{V}  &  =\left\{  F\in\mathcal{P}\left(  E\cap\operatorname*{Eqs}%
f\right)  \ \mid\ d\in F\right\}  .
\end{align*}
Thus, we have $\mathcal{P}\left(  E\cap\operatorname*{Eqs}f\right)
=\mathcal{U}\cup\mathcal{V}$, and the sets $\mathcal{U}$ and $\mathcal{V}$ are
disjoint. Now, we define a map $\Phi:\mathcal{U}\rightarrow\mathcal{V}$ by%
\[
\left(  \Phi\left(  B\right)  =B\cup\left\{  d\right\}
\ \ \ \ \ \ \ \ \ \ \text{for every }B\in\mathcal{U}\right)  .
\]
This map $\Phi$ is well-defined (because for every $B\in\mathcal{U}$, we have
$B\cup\left\{  d\right\}  \in\mathcal{V}$\ \ \ \ \footnote{This follows from
the fact that $d\in E\cap\operatorname*{Eqs}f$.}) and a bijection\footnote{Its
inverse is the map $\Psi:\mathcal{V}\rightarrow\mathcal{U}$ defined by
$\left(  \Psi\left(  B\right)  =B\setminus\left\{  d\right\}
\ \ \ \ \ \ \ \ \ \ \text{for every }B\in\mathcal{V}\right)  $.}. Moreover,
every $B\in\mathcal{U}$ satisfies%
\begin{equation}
\left(  -1\right)  ^{\left\vert \Phi\left(  B\right)  \right\vert }=-\left(
-1\right)  ^{\left\vert B\right\vert } \label{pf.lem.NBCm.moeb.short.Phi.-1}%
\end{equation}
\footnote{\textit{Proof.} Let $B\in\mathcal{U}$. Thus, $d\notin B$ (by the
definition of $\mathcal{U}$). Now, $\left\vert \underbrace{\Phi\left(
B\right)  }_{=B\cup\left\{  d\right\}  }\right\vert =\left\vert B\cup\left\{
d\right\}  \right\vert =\left\vert B\right\vert +1$ (since $d\notin B$), so
that $\left(  -1\right)  ^{\left\vert \Phi\left(  B\right)  \right\vert
}=-\left(  -1\right)  ^{\left\vert B\right\vert }$, qed.}.

Now, we claim that, for every $B\in\mathcal{U}$ and every $K\in\mathfrak{K}$,
we have the following logical equivalence:%
\begin{equation}
\left(  K\subseteq B\right)  \ \Longleftrightarrow\ \left(  K\subseteq
\Phi\left(  B\right)  \right)  . \label{pf.lem.NBCm.moeb.short.Phi.equiv}%
\end{equation}

\textit{Proof of (\ref{pf.lem.NBCm.moeb.short.Phi.equiv}):} Let $B\in
\mathcal{U}$ and $K\in\mathfrak{K}$. We must prove the equivalence
(\ref{pf.lem.NBCm.moeb.short.Phi.equiv}). The definition of $\Phi$ yields
$\Phi\left(  B\right)  =B\cup\left\{  d\right\}  \supseteq B$, so that
$B\subseteq\Phi\left(  B\right)  $. Hence, if $K\subseteq B$, then $K\subseteq
B\subseteq\Phi\left(  B\right)  $. Therefore, the forward implication of the
equivalence (\ref{pf.lem.NBCm.moeb.short.Phi.equiv}) is proven. It thus
remains to prove the backward implication of this equivalence. In other words,
it remains to prove that if $K\subseteq\Phi\left(  B\right)  $, then
$K\subseteq B$. So let us assume that $K\subseteq\Phi\left(  B\right)  $.

We want to prove that $K\subseteq B$. Assume the contrary. Thus,
$K\not \subseteq B$. We have $K\in\mathfrak{K}$. Thus, $K$ is a broken circuit
of $G$ (since $\mathfrak{K}$ is a set of broken circuits of $G$). In other
words, $K$ is a subset of $E$ having the form $C\setminus\left\{  e\right\}
$, where $C$ is a circuit of $G$, and where $e$ is the unique edge in $C$
having maximum label (among the edges in $C$) (because this is how a broken
circuit is defined). Consider these $C$ and $e$. Thus, $K=C\setminus\left\{
e\right\}  $.

The element $e$ is the unique edge in $C$ having maximum label (among the
edges in $C$). Thus, if $e^{\prime}$ is any edge in $C$ satisfying
$\ell\left(  e^{\prime}\right)  \geq\ell\left(  e\right)  $, then%
\begin{equation}
e^{\prime}=e. \label{pf.lem.NBCm.moeb.short.Phi.equiv.pf.e'}%
\end{equation}

But $\underbrace{K}_{\subseteq\Phi\left(  B\right)  =B\cup\left\{  d\right\}
}\setminus\left\{  d\right\}  \subseteq\left(  B\cup\left\{  d\right\}
\right)  \setminus\left\{  d\right\}  \subseteq B$.

If we had $d\notin K$, then we would have $K\setminus\left\{  d\right\}  =K$
and therefore $K=K\setminus\left\{  d\right\}  \subseteq B$; this would
contradict $K\not \subseteq B$. Hence, we cannot have $d\notin K$. We thus
must have $d\in K$. Hence, $d\in K=C\setminus\left\{  e\right\}  $. Hence,
$d\in C$ and $d\neq e$.

But $C\setminus\left\{  e\right\}  =K\subseteq\Phi\left(  B\right)  \subseteq
E\cap\operatorname*{Eqs}f$ (since $\Phi\left(  B\right)  \in\mathcal{P}\left(
E\cap\operatorname*{Eqs}f\right)  $), so that $C\setminus\left\{  e\right\}
\subseteq E\cap\operatorname*{Eqs}f\subseteq\operatorname*{Eqs}f$. Hence,
Lemma \ref{lem.Eqs.circuit} (applied to $Y$ instead of $X$) shows that $e\in
E\cap\operatorname*{Eqs}f$. Thus, $\ell\left(  d\right)  \geq\ell\left(
e\right)  $ (since $d$ was defined to be an element of $E\cap
\operatorname*{Eqs}f$ with maximum $\ell\left(  d\right)  $ among all $d\in
E\cap\operatorname*{Eqs}f$).

Also, $d\in C$. Since $\ell\left(  d\right)  \geq\ell\left(  e\right)  $, we
can therefore apply (\ref{pf.lem.NBCm.moeb.short.Phi.equiv.pf.e'}) to
$e^{\prime}=d$. We thus obtain $d=e$. This contradicts $d\neq e$. This
contradiction proves that our assumption was wrong. Hence, $K\subseteq B$ is
proven. Thus, we have proven the backward implication of the equivalence
(\ref{pf.lem.NBCm.moeb.short.Phi.equiv}); this completes the proof of
(\ref{pf.lem.NBCm.moeb.short.Phi.equiv}).

Now, recall that we have $\mathcal{P}\left(  E\cap\operatorname*{Eqs}f\right)
=\mathcal{U}\cup\mathcal{V}$, and the sets $\mathcal{U}$ and $\mathcal{V}$ are
disjoint. Hence, the sum $\sum_{B\subseteq E\cap\operatorname*{Eqs}f}\left(
-1\right)  ^{\left\vert B\right\vert }\prod_{\substack{K\in\mathfrak{K}%
;\\K\subseteq B}}a_{K}$ can be split into two sums as follows:
\begin{align}
&  \sum_{B\subseteq E\cap\operatorname*{Eqs}f}\left(  -1\right)  ^{\left\vert
B\right\vert }\prod_{\substack{K\in\mathfrak{K};\\K\subseteq B}}a_{K}%
\nonumber\\
&  =\sum_{B\in\mathcal{U}}\underbrace{\left(  -1\right)  ^{\left\vert
B\right\vert }}_{\substack{=-\left(  -1\right)  ^{\left\vert \Phi\left(
B\right)  \right\vert }\\\text{(by (\ref{pf.lem.NBCm.moeb.short.Phi.-1}))}%
}}\underbrace{\prod_{\substack{K\in\mathfrak{K};\\K\subseteq B}}}%
_{\substack{=\prod_{\substack{K\in\mathfrak{K};\\K\subseteq\Phi\left(
B\right)  }}\\\text{(because of the equivalence
(\ref{pf.lem.NBCm.moeb.short.Phi.equiv}))}}}a_{K}+\underbrace{\sum
_{B\in\mathcal{V}}\left(  -1\right)  ^{\left\vert B\right\vert }%
\prod_{\substack{K\in\mathfrak{K};\\K\subseteq B}}a_{K}}_{\substack{=\sum
_{B\in\mathcal{U}}\left(  -1\right)  ^{\left\vert \Phi\left(  B\right)
\right\vert }\prod_{\substack{K\in\mathfrak{K};\\K\subseteq\Phi\left(
B\right)  }}a_{K}\\\text{(here, we have substituted }\Phi\left(  B\right)
\text{ for }B\text{ in the sum,}\\\text{since the map }\Phi:\mathcal{U}%
\rightarrow\mathcal{V}\text{ is a bijection)}}}\nonumber\\
&  =\sum_{B\in\mathcal{U}}\left(  -\left(  -1\right)  ^{\left\vert \Phi\left(
B\right)  \right\vert }\right)  \prod_{\substack{K\in\mathfrak{K}%
;\\K\subseteq\Phi\left(  B\right)  }}a_{K}+\sum_{B\in\mathcal{U}}\left(
-1\right)  ^{\left\vert \Phi\left(  B\right)  \right\vert }\prod
_{\substack{K\in\mathfrak{K};\\K\subseteq\Phi\left(  B\right)  }%
}a_{K}\nonumber\\
&  =-\sum_{B\in\mathcal{U}}\left(  -1\right)  ^{\left\vert \Phi\left(
B\right)  \right\vert }\prod_{\substack{K\in\mathfrak{K};\\K\subseteq
\Phi\left(  B\right)  }}a_{K}+\sum_{B\in\mathcal{U}}\left(  -1\right)
^{\left\vert \Phi\left(  B\right)  \right\vert }\prod_{\substack{K\in
\mathfrak{K};\\K\subseteq\Phi\left(  B\right)  }}a_{K}\nonumber\\
&  =0=\left[  E\cap\operatorname*{Eqs}f=\varnothing\right]
\ \ \ \ \ \ \ \ \ \ \left(  \text{since }\left[  E\cap\operatorname*{Eqs}%
f=\varnothing\right]  =0\right)  . \label{pf.lem.NBCm.moeb.short.there}%
\end{align}
This proves Lemma \ref{lem.NBCm.moeb}.
\end{proof}
\end{vershort}

\begin{verlong}
\begin{proof}
[Proof of Lemma \ref{lem.NBCm.moeb}.]If $E\cap\operatorname*{Eqs}%
f=\varnothing$, then Lemma \ref{lem.NBCm.moeb} holds\footnote{\textit{Proof.}
Assume that $E\cap\operatorname*{Eqs}f=\varnothing$. We need to check that
Lemma \ref{lem.NBCm.moeb} holds.
\par
We have%
\[
\sum_{B\subseteq E\cap\operatorname*{Eqs}f}\left(  -1\right)  ^{\left\vert
B\right\vert }\prod_{\substack{K\in\mathfrak{K};\\K\subseteq B}}a_{K}%
=\sum_{B\subseteq\varnothing}\left(  -1\right)  ^{\left\vert B\right\vert
}\prod_{\substack{K\in\mathfrak{K};\\K\subseteq B}}a_{K}%
\ \ \ \ \ \ \ \ \ \ \left(  \text{since }E\cap\operatorname*{Eqs}%
f=\varnothing\right)  .
\]
But the only subset $B$ of $\varnothing$ is the set $\varnothing$. Thus, the
only addend of the sum $\sum_{B\subseteq\varnothing}\left(  -1\right)
^{\left\vert B\right\vert }\prod_{\substack{K\in\mathfrak{K};\\K\subseteq
B}}a_{K}$ is the addend for $B=\varnothing$. Hence,%
\[
\sum_{B\subseteq\varnothing}\left(  -1\right)  ^{\left\vert B\right\vert
}\prod_{\substack{K\in\mathfrak{K};\\K\subseteq B}}a_{K}=\underbrace{\left(
-1\right)  ^{\left\vert \varnothing\right\vert }}_{\substack{=1\\\text{(since
}\left\vert \varnothing\right\vert =0\text{ is even)}}}\prod_{\substack{K\in
\mathfrak{K};\\K\subseteq\varnothing}}a_{K}=\prod_{\substack{K\in
\mathfrak{K};\\K\subseteq\varnothing}}a_{K}.
\]
\par
But let $K\in\mathfrak{K}$ be such that $K\subseteq\varnothing$. Then, $K$ is
an element of $\mathfrak{K}$, and thus a broken circuit of $G$ (since
$\mathfrak{K}$ is a set of broken circuits of $G$). Hence, Lemma
\ref{lem.BC.nonempty} shows that $K\neq\varnothing$. But from $K\subseteq
\varnothing$, we obtain $K=\varnothing$; this contradicts $K\neq\varnothing$.
\par
Now, let us forget that we fixed $K$. Thus, we have obtained a contradiction
for every $K\in\mathfrak{K}$ satisfying $K\subseteq\varnothing$. Hence, there
exists no $K\in\mathfrak{K}$ satisfying $K\subseteq\varnothing$. Therefore,
the product $\prod_{\substack{K\in\mathfrak{K};\\K\subseteq\varnothing}}a_{K}$
is empty, and thus equals $1$. In other words, $\prod_{\substack{K\in
\mathfrak{K};\\K\subseteq\varnothing}}a_{K}=1$.
\par
Now,%
\[
\sum_{B\subseteq E\cap\operatorname*{Eqs}f}\left(  -1\right)  ^{\left\vert
B\right\vert }\prod_{\substack{K\in\mathfrak{K};\\K\subseteq B}}a_{K}%
=\sum_{B\subseteq\varnothing}\left(  -1\right)  ^{\left\vert B\right\vert
}\prod_{\substack{K\in\mathfrak{K};\\K\subseteq B}}a_{K}=\prod_{\substack{K\in
\mathfrak{K};\\K\subseteq\varnothing}}a_{K}=1=\left[  E\cap\operatorname*{Eqs}%
f=\varnothing\right]
\]
(since $\left[  E\cap\operatorname*{Eqs}f=\varnothing\right]  =1$ (since
$E\cap\operatorname*{Eqs}f=\varnothing$ holds)). Thus, Lemma
\ref{lem.NBCm.moeb} holds, qed.}. Thus, we can WLOG assume that we don't have
$E\cap\operatorname*{Eqs}f=\varnothing$. Assume this.

We don't have $E\cap\operatorname*{Eqs}f=\varnothing$. Thus, we have $\left[
E\cap\operatorname*{Eqs}f=\varnothing\right]  =0$.

We know that $V$ is finite (since the graph $\left(  V,E\right)  $ is finite).
Thus, $E$ is finite (since $E\subseteq\dbinom{V}{2}$), and therefore the set
$E\cap\operatorname*{Eqs}f$ is also finite.

We have $E\cap\operatorname*{Eqs}f\neq\varnothing$ (since we don't have
$E\cap\operatorname*{Eqs}f=\varnothing$). Thus, $E\cap\operatorname*{Eqs}f$ is
a nonempty finite set. Hence, there exists some $d\in E\cap\operatorname*{Eqs}%
f$ with maximum $\ell\left(  d\right)  $ (among all $d\in E\cap
\operatorname*{Eqs}f$). Pick such a $d$. (If there are several such $d$, then
it does not matter which one we pick.)

We have chosen $d$ to be the element of $E\cap\operatorname*{Eqs}f$ with
maximum $\ell\left(  d\right)  $ (among all $d\in E\cap\operatorname*{Eqs}f$).
Thus,%
\begin{equation}
\ell\left(  d\right)  \geq\ell\left(  e\right)  \ \ \ \ \ \ \ \ \ \ \text{for
every }e\in E\cap\operatorname*{Eqs}f. \label{pf.lem.NBCm.moeb.maxl}%
\end{equation}

As usual, we let $\mathcal{P}\left(  S\right)  $ denote the powerset of any
set $S$. We now define two subsets $\mathcal{U}$ and $\mathcal{V}$ of
$\mathcal{P}\left(  E\cap\operatorname*{Eqs}f\right)  $ as follows:%
\begin{align*}
\mathcal{U}  &  =\left\{  F\in\mathcal{P}\left(  E\cap\operatorname*{Eqs}%
f\right)  \ \mid\ d\notin F\right\}  ;\\
\mathcal{V}  &  =\left\{  F\in\mathcal{P}\left(  E\cap\operatorname*{Eqs}%
f\right)  \ \mid\ d\in F\right\}  .
\end{align*}
Every $B\in\mathcal{U}$ satisfies $B\cup\left\{  d\right\}  \in\mathcal{V}%
$\ \ \ \ \footnote{\textit{Proof.} Let $B\in\mathcal{U}$. Thus, $B\in
\mathcal{U}=\left\{  F\in\mathcal{P}\left(  E\cap\operatorname*{Eqs}f\right)
\ \mid\ d\notin F\right\}  $. In other words, $B$ is an element $F$ of
$\mathcal{P}\left(  E\cap\operatorname*{Eqs}f\right)  $ satisfying $d\notin
F$. In other words, $B$ is an element of $\mathcal{P}\left(  E\cap
\operatorname*{Eqs}f\right)  $ and satisfies $d\notin B$. We have
$B\in\mathcal{P}\left(  E\cap\operatorname*{Eqs}f\right)  $; in other words,
$B$ is a subset of $E\cap\operatorname*{Eqs}f$. Also, $\left\{  d\right\}
\subseteq E\cap\operatorname*{Eqs}f$ (since $d\in E\cap\operatorname*{Eqs}f$).
Thus, both $B$ and $\left\{  d\right\}  $ are subsets of $E\cap
\operatorname*{Eqs}f$. Hence, their union $B\cup\left\{  d\right\}  $ is a
subset of $E\cap\operatorname*{Eqs}f$. In other words, $B\cup\left\{
d\right\}  \in\mathcal{P}\left(  E\cap\operatorname*{Eqs}f\right)  $. Also,
$d\in\left\{  d\right\}  \subseteq B\cup\left\{  d\right\}  $. Hence,
$B\cup\left\{  d\right\}  $ is an element of $\mathcal{P}\left(
E\cap\operatorname*{Eqs}f\right)  $ and satisfies $d\in B\cup\left\{
d\right\}  $. In other words, $B\cup\left\{  d\right\}  $ is an element $F$ of
$\mathcal{P}\left(  E\cap\operatorname*{Eqs}f\right)  $ satisfying $d\in F$.
In other words, $B\cup\left\{  d\right\}  \in\left\{  F\in\mathcal{P}\left(
E\cap\operatorname*{Eqs}f\right)  \ \mid\ d\in F\right\}  =\mathcal{V}$,
qed.}. Thus, we can define a map $\Phi:\mathcal{U}\rightarrow\mathcal{V}$ by%
\[
\left(  \Phi\left(  B\right)  =B\cup\left\{  d\right\}
\ \ \ \ \ \ \ \ \ \ \text{for every }B\in\mathcal{U}\right)  .
\]
Consider this map $\Phi$.

Every $B\in\mathcal{V}$ satisfies $B\setminus\left\{  d\right\}
\in\mathcal{U}$\ \ \ \ \footnote{\textit{Proof.} Let $B\in\mathcal{V}$. Thus,
$B\in\mathcal{V}=\left\{  F\in\mathcal{P}\left(  E\cap\operatorname*{Eqs}%
f\right)  \ \mid\ d\in F\right\}  $. In other words, $B$ is an element $F$ of
$\mathcal{P}\left(  E\cap\operatorname*{Eqs}f\right)  $ satisfying $d\in F$.
In other words, $B$ is an element of $\mathcal{P}\left(  E\cap
\operatorname*{Eqs}f\right)  $ and satisfies $d\in B$. We have $B\in
\mathcal{P}\left(  E\cap\operatorname*{Eqs}f\right)  $; in other words, $B$ is
a subset of $E\cap\operatorname*{Eqs}f$. Hence, $B\setminus\left\{  d\right\}
$ is a subset of $E\cap\operatorname*{Eqs}f$ (since $B\setminus\left\{
d\right\}  \subseteq B$). In other words, $B\setminus\left\{  d\right\}
\in\mathcal{P}\left(  E\cap\operatorname*{Eqs}f\right)  $. Also, $d\notin
B\setminus\left\{  d\right\}  $ (since $d\in\left\{  d\right\}  $). Hence,
$B\setminus\left\{  d\right\}  $ is an element of $\mathcal{P}\left(
E\cap\operatorname*{Eqs}f\right)  $ and satisfies $d\notin B\setminus\left\{
d\right\}  $. In other words, $B\setminus\left\{  d\right\}  $ is an element
$F$ of $\mathcal{P}\left(  E\cap\operatorname*{Eqs}f\right)  $ satisfying
$d\notin F$. In other words, $B\setminus\left\{  d\right\}  \in\left\{
F\in\mathcal{P}\left(  E\cap\operatorname*{Eqs}f\right)  \ \mid\ d\notin
F\right\}  =\mathcal{U}$, qed.}. Thus, we can define a map $\Psi
:\mathcal{V}\rightarrow\mathcal{U}$ by%
\[
\left(  \Psi\left(  B\right)  =B\setminus\left\{  d\right\}
\ \ \ \ \ \ \ \ \ \ \text{for every }B\in\mathcal{V}\right)  .
\]
Consider this map $\Psi$.

We have $\Phi\circ\Psi=\operatorname*{id}$\ \ \ \ \footnote{\textit{Proof.}
Let $B\in\mathcal{V}$. We have%
\[
\left(  \Phi\circ\Psi\right)  \left(  B\right)  =\Phi\left(  \underbrace{\Psi
\left(  B\right)  }_{\substack{=B\setminus\left\{  d\right\}  \\\text{(by the
definition of }\Psi\text{)}}}\right)  =\Phi\left(  B\setminus\left\{
d\right\}  \right)  =\left(  B\setminus\left\{  d\right\}  \right)
\cup\left\{  d\right\}
\]
(by the definition of $\Phi$).
\par
We have $B\in\mathcal{V}=\left\{  F\in\mathcal{P}\left(  E\cap
\operatorname*{Eqs}f\right)  \ \mid\ d\in F\right\}  $. In other words, $B$ is
an element $F$ of $\mathcal{P}\left(  E\cap\operatorname*{Eqs}f\right)  $
satisfying $d\in F$. In other words, $B$ is an element of $\mathcal{P}\left(
E\cap\operatorname*{Eqs}f\right)  $ and satisfies $d\in B$. From $d\in B$, we
obtain $\left\{  d\right\}  \subseteq B$. Now, $\left(  \Phi\circ\Psi\right)
\left(  B\right)  =\left(  B\setminus\left\{  d\right\}  \right)  \cup\left\{
d\right\}  =B$ (since $\left\{  d\right\}  \subseteq B$). Thus, $\left(
\Phi\circ\Psi\right)  \left(  B\right)  =B=\operatorname*{id}\left(  B\right)
$.
\par
Now, let us forget that we fixed $B$. We thus have proven that $\left(
\Phi\circ\Psi\right)  \left(  B\right)  =\operatorname*{id}\left(  B\right)  $
for every $B\in\mathcal{V}$. In other words, $\Phi\circ\Psi=\operatorname*{id}%
$, qed.} and $\Psi\circ\Phi=\operatorname*{id}$%
\ \ \ \ \footnote{\textit{Proof.} Let $B\in\mathcal{U}$. We have%
\[
\left(  \Psi\circ\Phi\right)  \left(  B\right)  =\Psi\left(  \underbrace{\Phi
\left(  B\right)  }_{\substack{=B\cup\left\{  d\right\}  \\\text{(by the
definition of }\Phi\text{)}}}\right)  =\Psi\left(  B\cup\left\{  d\right\}
\right)  =\left(  B\cup\left\{  d\right\}  \right)  \setminus\left\{
d\right\}
\]
(by the definition of $\Psi$).
\par
We have $B\in\mathcal{U}=\left\{  F\in\mathcal{P}\left(  E\cap
\operatorname*{Eqs}f\right)  \ \mid\ d\notin F\right\}  $. In other words, $B$
is an element $F$ of $\mathcal{P}\left(  E\cap\operatorname*{Eqs}f\right)  $
satisfying $d\notin F$. In other words, $B$ is an element of $\mathcal{P}%
\left(  E\cap\operatorname*{Eqs}f\right)  $ and satisfies $d\notin B$. From
$d\notin B$, we see that the sets $\left\{  d\right\}  $ and $B$ are disjoint.
Now, $\left(  \Psi\circ\Phi\right)  \left(  B\right)  =\left(  B\cup\left\{
d\right\}  \right)  \setminus\left\{  d\right\}  =B$ (since the sets $\left\{
d\right\}  $ and $B$ are disjoint). Thus, $\left(  \Psi\circ\Phi\right)
\left(  B\right)  =B=\operatorname*{id}\left(  B\right)  $.
\par
Now, let us forget that we fixed $B$. We thus have proven that $\left(
\Psi\circ\Phi\right)  \left(  B\right)  =\operatorname*{id}\left(  B\right)  $
for every $B\in\mathcal{U}$. In other words, $\Psi\circ\Phi=\operatorname*{id}%
$, qed.}. Thus, the maps $\Phi$ and $\Psi$ are mutually inverse. Hence, the
map $\Phi$ is a bijection.

Moreover, for every $B\in\mathcal{U}$ and every $K\in\mathfrak{K}$, we have
the following logical equivalence:%
\begin{equation}
\left(  K\subseteq B\right)  \ \Longleftrightarrow\ \left(  K\subseteq
\Phi\left(  B\right)  \right)  . \label{pf.lem.NBCm.moeb.equiv}%
\end{equation}

\textit{Proof of (\ref{pf.lem.NBCm.moeb.equiv}):} Let $B\in\mathcal{U}$ and
$K\in\mathfrak{K}$. We need to prove the logical equivalence
(\ref{pf.lem.NBCm.moeb.equiv}).

The definition of $\Phi$ yields $\Phi\left(  B\right)  =B\cup\left\{
d\right\}  \supseteq B$, so that $B\subseteq\Phi\left(  B\right)  $.

We have $B\in\mathcal{U}=\left\{  F\in\mathcal{P}\left(  E\cap
\operatorname*{Eqs}f\right)  \ \mid\ d\notin F\right\}  $. In other words, $B$
is an element $F$ of $\mathcal{P}\left(  E\cap\operatorname*{Eqs}f\right)  $
satisfying $d\notin F$. In other words, $B$ is an element of $\mathcal{P}%
\left(  E\cap\operatorname*{Eqs}f\right)  $ and satisfies $d\notin B$. We have
$B\in\mathcal{P}\left(  E\cap\operatorname*{Eqs}f\right)  $; in other words,
$B$ is a subset of $E\cap\operatorname*{Eqs}f$.

Also, $\Phi\left(  B\right)  \in\mathcal{V}=\left\{  F\in\mathcal{P}\left(
E\cap\operatorname*{Eqs}f\right)  \ \mid\ d\in F\right\}  \subseteq
\mathcal{P}\left(  E\cap\operatorname*{Eqs}f\right)  $. In other words,
$\Phi\left(  B\right)  \subseteq E\cap\operatorname*{Eqs}f$.

We have $K\in\mathfrak{K}$. Thus, $K$ is a broken circuit of $G$ (since
$\mathfrak{K}$ is a set of broken circuits of $G$). In other words, $K$ is a
subset of $E$ having the form $C\setminus\left\{  e\right\}  $, where $C$ is a
circuit of $G$, and where $e$ is the unique edge in $C$ having maximum label
(among the edges in $C$)\ \ \ \ \footnote{because a broken circuit of $G$ is
the same as a subset of $E$ having the form $C\setminus\left\{  e\right\}  $,
where $C$ is a circuit of $G$, and where $e$ is the unique edge in $C$ having
maximum label (among the edges in $C$) (by the definition of a
\textquotedblleft broken circuit\textquotedblright)}. Consider this $C$ and
this $e$. Thus, we have the following facts:

\begin{itemize}
\item The set $C$ is a circuit of $G$.

\item The element $e$ is the unique edge in $C$ having maximum label (among
the edges in $C$).

\item We have $K=C\setminus\left\{  e\right\}  $.
\end{itemize}

The element $e$ is the unique edge in $C$ having maximum label (among the
edges in $C$). Thus, the only edge in $C$ whose label is greater or equal to
the label of $e$ is $e$ itself. In other words, if $e^{\prime}$ is any edge in
$C$ satisfying $\ell\left(  e^{\prime}\right)  \geq\ell\left(  e\right)  $,
then%
\begin{equation}
e^{\prime}=e. \label{pf.lem.NBCm.moeb.equiv.pf.maxlab}%
\end{equation}

Let us now assume that $K\subseteq\Phi\left(  B\right)  $. Thus,
$K\subseteq\Phi\left(  B\right)  =B\cup\left\{  d\right\}  $ (by the
definition of $\Phi$). Hence, $\underbrace{K}_{\subseteq B\cup\left\{
d\right\}  }\setminus\left\{  d\right\}  \subseteq\left(  B\cup\left\{
d\right\}  \right)  \setminus\left\{  d\right\}  \subseteq B$.

We shall now prove that $K\subseteq B$.

Indeed, assume the contrary. Thus, $K\not \subseteq B$. If we had $d\notin K$,
then we would have $K\setminus\left\{  d\right\}  =K$ and therefore
$K=K\setminus\left\{  d\right\}  \subseteq B$; this would contradict
$K\not \subseteq B$. Hence, we cannot have $d\notin K$. We thus must have
$d\in K$. Hence, $d\in K=C\setminus\left\{  e\right\}  $. Hence, $d\in C$ and
$d\notin\left\{  e\right\}  $. From $d\notin\left\{  e\right\}  $, we obtain
$d\neq e$.

But $C\setminus\left\{  e\right\}  =K\subseteq\Phi\left(  B\right)  \subseteq
E\cap\operatorname*{Eqs}f\subseteq\operatorname*{Eqs}f$. Hence, Lemma
\ref{lem.Eqs.circuit} (applied to $Y$ instead of $X$) shows that $e\in
E\cap\operatorname*{Eqs}f$. Thus, (\ref{pf.lem.NBCm.moeb.maxl}) shows that
$\ell\left(  d\right)  \geq\ell\left(  e\right)  $.

Also, $d\in C$. In other words, $d$ is an edge in $C$. Since $\ell\left(
d\right)  \geq\ell\left(  e\right)  $, we can therefore apply
(\ref{pf.lem.NBCm.moeb.equiv.pf.maxlab}) to $e^{\prime}=d$. We thus obtain
$d=e$. This contradicts $d\neq e$. This contradiction proves that our
assumption was wrong. Hence, $K\subseteq B$ is proven.

Now, let us forget that we assumed that $K\subseteq\Phi\left(  B\right)  $. We
thus have proven that $K\subseteq B$ under the assumption that $K\subseteq
\Phi\left(  B\right)  $. In other words, we have proven the implication%
\begin{equation}
\left(  K\subseteq\Phi\left(  B\right)  \right)  \ \Longrightarrow\ \left(
K\subseteq B\right)  . \label{pf.lem.NBCm.moeb.equiv.pf.dir1}%
\end{equation}

On the other hand, if $K\subseteq B$, then $K\subseteq B\subseteq\Phi\left(
B\right)  $. Hence, the implication%
\[
\left(  K\subseteq B\right)  \ \Longrightarrow\ \left(  K\subseteq\Phi\left(
B\right)  \right)
\]
holds. Combining this implication with (\ref{pf.lem.NBCm.moeb.equiv.pf.dir1}),
we obtain the logical equivalence $\left(  K\subseteq B\right)
\ \Longleftrightarrow\ \left(  K\subseteq\Phi\left(  B\right)  \right)  $.
Thus, (\ref{pf.lem.NBCm.moeb.equiv}) is proven.

Also, every $B\in\mathcal{U}$ satisfies%
\begin{equation}
\left(  -1\right)  ^{\left\vert B\right\vert }=-\left(  -1\right)
^{\left\vert \Phi\left(  B\right)  \right\vert }
\label{pf.lem.NBCm.moeb.phisize}%
\end{equation}
\footnote{\textit{Proof of (\ref{pf.lem.NBCm.moeb.phisize}):} Let
$B\in\mathcal{U}$. We have $B\in\mathcal{U}=\left\{  F\in\mathcal{P}\left(
E\cap\operatorname*{Eqs}f\right)  \ \mid\ d\notin F\right\}  $. In other
words, $B$ is an element $F$ of $\mathcal{P}\left(  E\cap\operatorname*{Eqs}%
f\right)  $ satisfying $d\notin F$. In other words, $B$ is an element of
$\mathcal{P}\left(  E\cap\operatorname*{Eqs}f\right)  $ and satisfies $d\notin
B$. From $d\notin B$, we see that $\left\vert B\cup\left\{  d\right\}
\right\vert =\left\vert B\right\vert +1$. Now, the definition of $\Phi$ yields
$\Phi\left(  B\right)  =B\cup\left\{  d\right\}  $. Hence, $\left\vert
\Phi\left(  B\right)  \right\vert =\left\vert B\cup\left\{  d\right\}
\right\vert =\left\vert B\right\vert +1$. Thus, $\left(  -1\right)
^{\left\vert \Phi\left(  B\right)  \right\vert }=\left(  -1\right)
^{\left\vert B\right\vert +1}=-\left(  -1\right)  ^{\left\vert B\right\vert }%
$. Therefore, $\left(  -1\right)  ^{\left\vert B\right\vert }=-\left(
-1\right)  ^{\left\vert \Phi\left(  B\right)  \right\vert }$. This proves
(\ref{pf.lem.NBCm.moeb.phisize}).}. In other words, every $B\in\mathcal{U}$
satisfies%
\begin{equation}
\left(  -1\right)  ^{\left\vert \Phi\left(  B\right)  \right\vert }=-\left(
-1\right)  ^{\left\vert B\right\vert }. \label{pf.lem.NBCm.moeb.phisize-1}%
\end{equation}

Now,%
\begin{align}
&  \sum_{B\subseteq E\cap\operatorname*{Eqs}f}\left(  -1\right)  ^{\left\vert
B\right\vert }\prod_{\substack{K\in\mathfrak{K};\\K\subseteq B}}a_{K}%
\nonumber\\
&  =\underbrace{\sum_{\substack{B\subseteq E\cap\operatorname*{Eqs}f;\\d\in
B}}}_{\substack{=\sum_{B\in\left\{  F\in\mathcal{P}\left(  E\cap
\operatorname*{Eqs}f\right)  \ \mid\ d\in F\right\}  }=\sum_{B\in\mathcal{V}%
}\\\text{(since }\left\{  F\in\mathcal{P}\left(  E\cap\operatorname*{Eqs}%
f\right)  \ \mid\ d\in F\right\}  =\mathcal{V}\text{)}}}\left(  -1\right)
^{\left\vert B\right\vert }\prod_{\substack{K\in\mathfrak{K};\\K\subseteq
B}}a_{K}+\underbrace{\sum_{\substack{B\subseteq E\cap\operatorname*{Eqs}%
f;\\d\notin B}}}_{\substack{=\sum_{B\in\left\{  F\in\mathcal{P}\left(
E\cap\operatorname*{Eqs}f\right)  \ \mid\ d\notin F\right\}  }=\sum
_{B\in\mathcal{U}}\\\text{(since }\left\{  F\in\mathcal{P}\left(
E\cap\operatorname*{Eqs}f\right)  \ \mid\ d\notin F\right\}  =\mathcal{U}%
\text{)}}}\left(  -1\right)  ^{\left\vert B\right\vert }\prod_{\substack{K\in
\mathfrak{K};\\K\subseteq B}}a_{K}\nonumber\\
&  \ \ \ \ \ \ \ \ \ \ \ \ \ \ \ \ \ \ \ \ \left(
\begin{array}
[c]{c}%
\text{here, we have split the sum into two parts:}\\
\text{one containing all addends with }d\in B\text{,}\\
\text{and one containing all addends with }d\notin B
\end{array}
\right) \nonumber\\
&  =\underbrace{\sum_{B\in\mathcal{V}}\left(  -1\right)  ^{\left\vert
B\right\vert }\prod_{\substack{K\in\mathfrak{K};\\K\subseteq B}}a_{K}%
}_{\substack{=\sum_{B\in\mathcal{U}}\left(  -1\right)  ^{\left\vert
\Phi\left(  B\right)  \right\vert }\prod_{\substack{K\in\mathfrak{K}%
;\\K\subseteq\Phi\left(  B\right)  }}a_{K}\\\text{(here, we have}%
\\\text{substituted }\Phi\left(  B\right)  \text{ for }B\text{ in the
sum,}\\\text{since the map }\Phi:\mathcal{U}\rightarrow\mathcal{V}\text{ is a
bijection)}}}+\sum_{B\in\mathcal{U}}\underbrace{\left(  -1\right)
^{\left\vert B\right\vert }}_{\substack{=-\left(  -1\right)  ^{\left\vert
\Phi\left(  B\right)  \right\vert }\\\text{(by (\ref{pf.lem.NBCm.moeb.phisize}%
))}}}\underbrace{\prod_{\substack{K\in\mathfrak{K};\\K\subseteq B}%
}}_{\substack{=\prod_{\substack{K\in\mathfrak{K};\\K\subseteq\Phi\left(
B\right)  }}\\\text{(because for every }K\in\mathfrak{K}\text{,}\\\text{the
condition }\left(  K\subseteq B\right)  \\\text{is equivalent to }\left(
K\subseteq\Phi\left(  B\right)  \right)  \\\text{(by
(\ref{pf.lem.NBCm.moeb.equiv})))}}}a_{K}\nonumber\\
&  =\sum_{B\in\mathcal{U}}\left(  -1\right)  ^{\left\vert \Phi\left(
B\right)  \right\vert }\prod_{\substack{K\in\mathfrak{K};\\K\subseteq
\Phi\left(  B\right)  }}a_{K}+\sum_{B\in\mathcal{U}}\left(  -\left(
-1\right)  ^{\left\vert \Phi\left(  B\right)  \right\vert }\right)
\prod_{\substack{K\in\mathfrak{K};\\K\subseteq\Phi\left(  B\right)  }%
}a_{K}\nonumber\\
&  =\sum_{B\in\mathcal{U}}\left(  -1\right)  ^{\left\vert \Phi\left(
B\right)  \right\vert }\prod_{\substack{K\in\mathfrak{K};\\K\subseteq
\Phi\left(  B\right)  }}a_{K}-\sum_{B\in\mathcal{U}}\left(  -1\right)
^{\left\vert \Phi\left(  B\right)  \right\vert }\prod_{\substack{K\in
\mathfrak{K};\\K\subseteq\Phi\left(  B\right)  }}a_{K}\nonumber\\
&  =0=\left[  E\cap\operatorname*{Eqs}f=\varnothing\right]
\ \ \ \ \ \ \ \ \ \ \left(  \text{since }\left[  E\cap\operatorname*{Eqs}%
f=\varnothing\right]  =0\right)  . \label{pf.lem.NBCm.moeb.there}%
\end{align}
This proves Lemma \ref{lem.NBCm.moeb}.
\end{proof}
\end{verlong}

We now finally proceed to the proof of Theorem \ref{thm.chromsym.varis}:

\begin{vershort}
\begin{proof}
[Proof of Theorem \ref{thm.chromsym.varis}.]The definition of $X_{G}$ shows
that
\begin{align*}
X_{G}  &  =\sum_{\substack{f:V\rightarrow\mathbb{N}_{+}\text{ is
a}\\\text{proper }\mathbb{N}_{+}\text{-coloring of }G}}\mathbf{x}_{f}\\
&  =\sum_{f:V\rightarrow\mathbb{N}_{+}}\left[  \underbrace{f\text{ is a proper
}\mathbb{N}_{+}\text{-coloring of }G}_{\substack{\Longleftrightarrow\ \left(
\text{the }\mathbb{N}_{+}\text{-coloring }f\text{ of }G\text{ is
proper}\right)  \\\Longleftrightarrow\ \left(  E\cap\operatorname*{Eqs}%
f=\varnothing\right)  \\\text{(by Lemma \ref{lem.Eqs.proper}, applied to
}\mathbb{N}_{+}\text{ instead of }X\text{)}}}\right]  \mathbf{x}_{f}\\
&  =\sum_{f:V\rightarrow\mathbb{N}_{+}}\underbrace{\left[  E\cap
\operatorname*{Eqs}f=\varnothing\right]  }_{\substack{=\sum_{B\subseteq
E\cap\operatorname*{Eqs}f}\left(  -1\right)  ^{\left\vert B\right\vert }%
\prod_{\substack{K\in\mathfrak{K};\\K\subseteq B}}a_{K}\\\text{(by Lemma
\ref{lem.NBCm.moeb}, applied to }Y=\mathbb{N}_{+}\text{)}}}\mathbf{x}_{f}\\
&  =\sum_{f:V\rightarrow\mathbb{N}_{+}}\ \ \underbrace{\sum_{B\subseteq
E\cap\operatorname*{Eqs}f}}_{=\sum_{\substack{B\subseteq E;\\B\subseteq
\operatorname*{Eqs}f}}}\left(  -1\right)  ^{\left\vert B\right\vert }\left(
\prod_{\substack{K\in\mathfrak{K};\\K\subseteq B}}a_{K}\right)  \mathbf{x}%
_{f}\\
&  =\underbrace{\sum_{f:V\rightarrow\mathbb{N}_{+}}\ \ \sum
_{\substack{B\subseteq E;\\B\subseteq\operatorname*{Eqs}f}}}_{=\sum
_{B\subseteq E}\sum_{\substack{f:V\rightarrow\mathbb{N}_{+};\\B\subseteq
\operatorname*{Eqs}f}}}\left(  -1\right)  ^{\left\vert B\right\vert }\left(
\prod_{\substack{K\in\mathfrak{K};\\K\subseteq B}}a_{K}\right)  \mathbf{x}%
_{f}\\
&  =\sum_{B\subseteq E}\ \ \sum_{\substack{f:V\rightarrow\mathbb{N}%
_{+};\\B\subseteq\operatorname*{Eqs}f}}\left(  -1\right)  ^{\left\vert
B\right\vert }\left(  \prod_{\substack{K\in\mathfrak{K};\\K\subseteq B}%
}a_{K}\right)  \mathbf{x}_{f}=\sum_{B\subseteq E}\left(  -1\right)
^{\left\vert B\right\vert }\left(  \prod_{\substack{K\in\mathfrak{K}%
;\\K\subseteq B}}a_{K}\right)  \underbrace{\sum_{\substack{f:V\rightarrow
\mathbb{N}_{+};\\B\subseteq\operatorname*{Eqs}f}}\mathbf{x}_{f}}%
_{\substack{=p_{\lambda\left(  V,B\right)  }\\\text{(by Lemma
\ref{lem.Eqs.sum})}}}\\
&  =\sum_{B\subseteq E}\left(  -1\right)  ^{\left\vert B\right\vert }\left(
\prod_{\substack{K\in\mathfrak{K};\\K\subseteq B}}a_{K}\right)  p_{\lambda
\left(  V,B\right)  }=\sum_{F\subseteq E}\left(  -1\right)  ^{\left\vert
F\right\vert }\left(  \prod_{\substack{K\in\mathfrak{K};\\K\subseteq F}%
}a_{K}\right)  p_{\lambda\left(  V,F\right)  }%
\end{align*}
(here, we have renamed the summation index $B$ as $F$). This proves Theorem
\ref{thm.chromsym.varis}.
\end{proof}
\end{vershort}

\begin{verlong}
\begin{proof}
[Proof of Theorem \ref{thm.chromsym.varis}.]We have
\begin{equation}
X_{G}=\sum_{\substack{f:V\rightarrow\mathbb{N}_{+}\text{ is a}\\\text{proper
}\mathbb{N}_{+}\text{-coloring of }G}}\mathbf{x}_{f}
\label{pf.thm.chromsym.varis.XG-def}%
\end{equation}
(by the definition of $X_{G}$). Now, if $f:V\rightarrow\mathbb{N}_{+}$ is a
map, then we have the following logical equivalence:%
\begin{equation}
\left(  \text{the }\mathbb{N}_{+}\text{-coloring }f\text{ of }G\text{ is
proper}\right)  \ \Longleftrightarrow\ \left(  E\cap\operatorname*{Eqs}%
f=\varnothing\right)  \label{pf.thm.chromsym.varis.equiv}%
\end{equation}
(because the $\mathbb{N}_{+}$-coloring $f$ of $G$ is proper if and only if
$E\cap\operatorname*{Eqs}f=\varnothing$\ \ \ \ \footnote{by Lemma
\ref{lem.Eqs.proper} (applied to $\mathbb{N}_{+}$ instead of $X$)}). Now,%
\begin{align}
&  \sum_{f:V\rightarrow\mathbb{N}_{+}}\left[  \underbrace{E\cap
\operatorname*{Eqs}f=\varnothing}_{\substack{\Longleftrightarrow\ \left(
\text{the }\mathbb{N}_{+}\text{-coloring }f\text{ of }G\text{ is
proper}\right)  \\\text{(by (\ref{pf.thm.chromsym.varis.equiv}))}}}\right]
\mathbf{x}_{f}\nonumber\\
&  =\sum_{f:V\rightarrow\mathbb{N}_{+}}\left[  \underbrace{\text{the
}\mathbb{N}_{+}\text{-coloring }f\text{ of }G\text{ is proper}}%
_{\Longleftrightarrow\ \left(  f\text{ is a proper }\mathbb{N}_{+}%
\text{-coloring of }G\right)  }\right]  \mathbf{x}_{f}\nonumber\\
&  =\sum_{f:V\rightarrow\mathbb{N}_{+}}\left[  f\text{ is a proper }%
\mathbb{N}_{+}\text{-coloring of }G\right]  \mathbf{x}_{f}\nonumber\\
&  =\sum_{\substack{f:V\rightarrow\mathbb{N}_{+}\text{ is a}\\\text{proper
}\mathbb{N}_{+}\text{-coloring of }G}}\underbrace{\left[  f\text{ is a proper
}\mathbb{N}_{+}\text{-coloring of }G\right]  }_{\substack{=1\\\text{(since
}f\text{ is a proper }\mathbb{N}_{+}\text{-coloring of }G\text{)}}%
}\mathbf{x}_{f}\nonumber\\
&  \ \ \ \ \ \ \ \ \ \ +\sum_{\substack{f:V\rightarrow\mathbb{N}_{+}\text{ is
not a}\\\text{proper }\mathbb{N}_{+}\text{-coloring of }G}}\underbrace{\left[
f\text{ is a proper }\mathbb{N}_{+}\text{-coloring of }G\right]
}_{\substack{=0\\\text{(since }f\text{ is not a proper }\mathbb{N}%
_{+}\text{-coloring of }G\text{)}}}\mathbf{x}_{f}\nonumber\\
&  \ \ \ \ \ \ \ \ \ \ \ \ \ \ \ \ \ \ \ \ \left(  \text{since each
}f:V\rightarrow\mathbb{N}_{+}\text{ is either a proper }\mathbb{N}%
_{+}\text{-coloring of }G\text{ or not}\right) \nonumber\\
&  =\sum_{\substack{f:V\rightarrow\mathbb{N}_{+}\text{ is a}\\\text{proper
}\mathbb{N}_{+}\text{-coloring of }G}}\mathbf{x}_{f}+\underbrace{\sum
_{\substack{f:V\rightarrow\mathbb{N}_{+}\text{ is not a}\\\text{proper
}\mathbb{N}_{+}\text{-coloring of }G}}0\mathbf{x}_{f}}_{=0}=\sum
_{\substack{f:V\rightarrow\mathbb{N}_{+}\text{ is a}\\\text{proper }%
\mathbb{N}_{+}\text{-coloring of }G}}\mathbf{x}_{f}\nonumber\\
&  =X_{G} \label{pf.thm.chromsym.varis.step1}%
\end{align}
(by (\ref{pf.thm.chromsym.varis.XG-def})).

However, for every $f:V\rightarrow\mathbb{N}_{+}$, we have%
\begin{equation}
\sum_{B\subseteq E\cap\operatorname*{Eqs}f}\left(  -1\right)  ^{\left\vert
B\right\vert }\prod_{\substack{K\in\mathfrak{K};\\K\subseteq B}}a_{K}=\left[
E\cap\operatorname*{Eqs}f=\varnothing\right]
\label{pf.thm.chromsym.varis.moeb}%
\end{equation}
(by Lemma \ref{lem.NBCm.moeb} (applied to $\mathbb{N}_{+}$ instead of $Y$)).

For every $f:V\rightarrow\mathbb{N}_{+}$, we have%
\begin{equation}
\left\{  F\subseteq E\ \mid\ F\subseteq\operatorname*{Eqs}f\right\}
=\mathcal{P}\left(  E\cap\operatorname*{Eqs}f\right)
\label{pf.thm.chromsym.varis.cut}%
\end{equation}
\footnote{\textit{Proof of (\ref{pf.thm.chromsym.varis.cut}):} Let
$f:V\rightarrow\mathbb{N}_{+}$.
\par
Let $B\in\left\{  F\subseteq E\ \mid\ F\subseteq\operatorname*{Eqs}f\right\}
$. Thus, $B$ is a subset $F$ of $E$ satisfying $F\subseteq\operatorname*{Eqs}%
f$. In other words, $B$ is a subset of $E$ and satisfies $B\subseteq
\operatorname*{Eqs}f$. Since $B$ is a subset of $E$, we have $B\subseteq E$.
Combining this with $B\subseteq\operatorname*{Eqs}f$, we obtain $B\subseteq
E\cap\operatorname*{Eqs}f$. In other words, $B\in\mathcal{P}\left(
E\cap\operatorname*{Eqs}f\right)  $.
\par
Let us now forget that we fixed $B$. We thus have proven that every
$B\in\left\{  F\subseteq E\ \mid\ F\subseteq\operatorname*{Eqs}f\right\}  $
satisfies $B\in\mathcal{P}\left(  E\cap\operatorname*{Eqs}f\right)  $. In
other words,%
\begin{equation}
\left\{  F\subseteq E\ \mid\ F\subseteq\operatorname*{Eqs}f\right\}
\subseteq\mathcal{P}\left(  E\cap\operatorname*{Eqs}f\right)  .
\label{pf.thm.chromsym.varis.cut.pf.1}%
\end{equation}
\par
On the other hand, let $C\in\mathcal{P}\left(  E\cap\operatorname*{Eqs}%
f\right)  $. Thus, $C$ is a subset of $E\cap\operatorname*{Eqs}f$. Hence,
$C\subseteq E\cap\operatorname*{Eqs}f\subseteq E$, so that $C$ is a subset of
$E$. Also, $C\subseteq E\cap\operatorname*{Eqs}f\subseteq\operatorname*{Eqs}%
f$. Thus, $C$ is a subset of $E$ and satisfies $C\subseteq\operatorname*{Eqs}%
f$. In other words, $C$ is a subset $F$ of $E$ satisfying $F\subseteq
\operatorname*{Eqs}f$. In other words, $C\in\left\{  F\subseteq E\ \mid
\ F\subseteq\operatorname*{Eqs}f\right\}  $.
\par
Let us now forget that we fixed $C$. We thus have proven that every
$C\in\mathcal{P}\left(  E\cap\operatorname*{Eqs}f\right)  $ satisfies
$C\in\left\{  F\subseteq E\ \mid\ F\subseteq\operatorname*{Eqs}f\right\}  $.
In other words,%
\[
\mathcal{P}\left(  E\cap\operatorname*{Eqs}f\right)  \subseteq\left\{
F\subseteq E\ \mid\ F\subseteq\operatorname*{Eqs}f\right\}  .
\]
Combining this inclusion with (\ref{pf.thm.chromsym.varis.cut.pf.1}), we
obtain $\left\{  F\subseteq E\ \mid\ F\subseteq\operatorname*{Eqs}f\right\}
=\mathcal{P}\left(  E\cap\operatorname*{Eqs}f\right)  $. This proves
(\ref{pf.thm.chromsym.varis.cut}).}.

For every $f:V\rightarrow\mathbb{N}_{+}$, we have%
\begin{align}
&  \underbrace{\sum_{\substack{B\subseteq E;\\B\subseteq\operatorname*{Eqs}%
f}}}_{\substack{=\sum_{B\in\left\{  F\subseteq E\ \mid\ F\subseteq
\operatorname*{Eqs}f\right\}  }=\sum_{B\in\mathcal{P}\left(  E\cap
\operatorname*{Eqs}f\right)  }\\\text{(because }\left\{  F\subseteq
E\ \mid\ F\subseteq\operatorname*{Eqs}f\right\}  =\mathcal{P}\left(
E\cap\operatorname*{Eqs}f\right)  \\\text{(by (\ref{pf.thm.chromsym.varis.cut}%
)))}}}\left(  -1\right)  ^{\left\vert B\right\vert }\prod_{\substack{K\in
\mathfrak{K};\\K\subseteq B}}a_{K}\nonumber\\
&  =\underbrace{\sum_{B\in\mathcal{P}\left(  E\cap\operatorname*{Eqs}f\right)
}}_{=\sum_{B\subseteq E\cap\operatorname*{Eqs}f}}\left(  -1\right)
^{\left\vert B\right\vert }\prod_{\substack{K\in\mathfrak{K};\\K\subseteq
B}}a_{K}=\sum_{B\subseteq E\cap\operatorname*{Eqs}f}\left(  -1\right)
^{\left\vert B\right\vert }\prod_{\substack{K\in\mathfrak{K};\\K\subseteq
B}}a_{K}\nonumber\\
&  =\left[  E\cap\operatorname*{Eqs}f=\varnothing\right]
\label{pf.thm.chromsym.varis.moeb2}%
\end{align}
(by (\ref{pf.thm.chromsym.varis.moeb})).

Now, (\ref{pf.thm.chromsym.varis.step1}) yields%
\begin{align*}
X_{G}  &  =\sum_{f:V\rightarrow\mathbb{N}_{+}}\underbrace{\left[
E\cap\operatorname*{Eqs}f=\varnothing\right]  }_{\substack{=\sum
_{\substack{B\subseteq E;\\B\subseteq\operatorname*{Eqs}f}}\left(  -1\right)
^{\left\vert B\right\vert }\prod_{\substack{K\in\mathfrak{K};\\K\subseteq
B}}a_{K}\\\text{(by (\ref{pf.thm.chromsym.varis.moeb2}))}}}\mathbf{x}_{f}\\
&  =\sum_{f:V\rightarrow\mathbb{N}_{+}}\left(  \sum_{\substack{B\subseteq
E;\\B\subseteq\operatorname*{Eqs}f}}\left(  -1\right)  ^{\left\vert
B\right\vert }\prod_{\substack{K\in\mathfrak{K};\\K\subseteq B}}a_{K}\right)
\mathbf{x}_{f}=\underbrace{\sum_{f:V\rightarrow\mathbb{N}_{+}}\ \ \sum
_{\substack{B\subseteq E;\\B\subseteq\operatorname*{Eqs}f}}}_{=\sum
_{B\subseteq E}\ \ \sum_{\substack{f:V\rightarrow\mathbb{N}_{+};\\B\subseteq
\operatorname*{Eqs}f}}}\left(  -1\right)  ^{\left\vert B\right\vert }\left(
\prod_{\substack{K\in\mathfrak{K};\\K\subseteq B}}a_{K}\right)  \mathbf{x}%
_{f}\\
&  =\sum_{B\subseteq E}\ \ \sum_{\substack{f:V\rightarrow\mathbb{N}%
_{+};\\B\subseteq\operatorname*{Eqs}f}}\left(  -1\right)  ^{\left\vert
B\right\vert }\left(  \prod_{\substack{K\in\mathfrak{K};\\K\subseteq B}%
}a_{K}\right)  \mathbf{x}_{f}\\
&  =\sum_{B\subseteq E}\left(  -1\right)  ^{\left\vert B\right\vert }\left(
\prod_{\substack{K\in\mathfrak{K};\\K\subseteq B}}a_{K}\right)
\underbrace{\sum_{\substack{f:V\rightarrow\mathbb{N}_{+};\\B\subseteq
\operatorname*{Eqs}f}}\mathbf{x}_{f}}_{\substack{=p_{\lambda\left(
V,B\right)  }\\\text{(by Lemma \ref{lem.Eqs.sum}}\\\text{(since }\left(
V,B\right)  \text{ is a finite graph}\\\text{(since }V\text{ is a finite set
and }B\subseteq E\subseteq\dbinom{V}{2}\text{)))}}}\\
&  =\sum_{B\subseteq E}\left(  -1\right)  ^{\left\vert B\right\vert }\left(
\prod_{\substack{K\in\mathfrak{K};\\K\subseteq B}}a_{K}\right)  p_{\lambda
\left(  V,B\right)  }=\sum_{F\subseteq E}\left(  -1\right)  ^{\left\vert
F\right\vert }\left(  \prod_{\substack{K\in\mathfrak{K};\\K\subseteq F}%
}a_{K}\right)  p_{\lambda\left(  V,F\right)  }%
\end{align*}
(here, we have renamed the summation index $B$ as $F$). This proves Theorem
\ref{thm.chromsym.varis}.
\end{proof}
\end{verlong}

Thus, Theorem \ref{thm.chromsym.varis} is proven; as we know, this entails the
correctness of Theorem \ref{thm.chromsym.empty}, Corollary
\ref{cor.chromsym.K-free} and Corollary \ref{cor.chromsym.NBC}.

\section{\label{sec.chrompol}The chromatic polynomial}

\subsection{Definition}

We have so far studied the chromatic symmetric function. We shall now apply
the above results to the chromatic polynomial. The definition of the chromatic
polynomial rests upon the following fact:

\begin{theorem}
\label{thm.chrompol.exist}Let $G=\left(  V,E\right)  $ be a finite graph.
Then, there exists a unique polynomial $P\in\mathbb{Z}\left[  x\right]  $ such
that every $q\in\mathbb{N}$ satisfies%
\[
P\left(  q\right)  =\left(  \text{the number of all proper }\left\{
1,2,\ldots,q\right\}  \text{-colorings of }G\right)  .
\]

\end{theorem}

\begin{definition}
\label{def.chrompol}Let $G=\left(  V,E\right)  $ be a finite graph. Theorem
\ref{thm.chrompol.exist} shows that there exists a polynomial $P\in
\mathbb{Z}\left[  x\right]  $ such that every $q\in\mathbb{N}$ satisfies
$P\left(  q\right)  =\left(  \text{the number of all proper }\left\{
1,2,\ldots,q\right\}  \text{-colorings of }G\right)  $. This polynomial $P$ is
called the \emph{chromatic polynomial} of $G$, and will be denoted by
$\chi_{G}$.
\end{definition}

We shall later prove Theorem \ref{thm.chrompol.exist} (as a consequence of
something stronger that we show). First, we shall state some formulas for the
chromatic polynomial which are analogues of results proven before for the
chromatic symmetric function.

\subsection{Formulas for $\chi_{G}$}

Before we state several formulas for $\chi_{G}$, we need to introduce one more notation:

\begin{definition}
\label{def.conn}Let $G$ be a finite graph. We let $\operatorname*{conn}G$
denote the number of connected components of $G$.
\end{definition}

The following results are analogues of Theorem \ref{thm.chromsym.empty},
Theorem \ref{thm.chromsym.varis}, Corollary \ref{cor.chromsym.K-free} and
Corollary \ref{cor.chromsym.NBC}, respectively:

\begin{theorem}
\label{thm.chrompol.empty}Let $G=\left(  V,E\right)  $ be a finite graph.
Then,%
\[
\chi_{G}=\sum_{F\subseteq E}\left(  -1\right)  ^{\left\vert F\right\vert
}x^{\operatorname*{conn}\left(  V,F\right)  }.
\]
(Here, of course, the pair $\left(  V,F\right)  $ is regarded as a graph, and
the expression $\operatorname*{conn}\left(  V,F\right)  $ is understood
according to Definition \ref{def.conn}.)
\end{theorem}

\begin{theorem}
\label{thm.chrompol.varis}Let $G=\left(  V,E\right)  $ be a finite graph. Let
$X$ be a totally ordered set. Let $\ell:E\rightarrow X$ be a labeling
function. Let $\mathfrak{K}$ be some set of broken circuits of $G$ (not
necessarily containing all of them). Let $a_{K}$ be an element of $\mathbf{k}$
for every $K\in\mathfrak{K}$. Then,%
\[
\chi_{G}=\sum_{F\subseteq E}\left(  -1\right)  ^{\left\vert F\right\vert
}\left(  \prod_{\substack{K\in\mathfrak{K};\\K\subseteq F}}a_{K}\right)
x^{\operatorname*{conn}\left(  V,F\right)  }.
\]
(Here, of course, the pair $\left(  V,F\right)  $ is regarded as a graph, and
the expression $\operatorname*{conn}\left(  V,F\right)  $ is understood
according to Definition \ref{def.conn}. Moreover, the polynomial $\chi_{G}%
\in\mathbb{Z}\left[  x\right]  $ on the left-hand side is regarded as an
element of $\mathbf{k}\left[  x\right]  $ via the canonical ring morphism
$\mathbb{Z}\left[  x\right]  \rightarrow\mathbf{k}\left[  x\right]  $.)
\end{theorem}

\begin{corollary}
\label{cor.chrompol.K-free}Let $G=\left(  V,E\right)  $ be a finite graph. Let
$X$ be a totally ordered set. Let $\ell:E\rightarrow X$ be a labeling
function. Let $\mathfrak{K}$ be some set of broken circuits of $G$ (not
necessarily containing all of them). Then,%
\[
\chi_{G}=\sum_{\substack{F\subseteq E;\\F\text{ is }\mathfrak{K}\text{-free}%
}}\left(  -1\right)  ^{\left\vert F\right\vert }x^{\operatorname*{conn}\left(
V,F\right)  }.
\]

\end{corollary}

\begin{corollary}
\label{cor.chrompol.NBC}Let $G=\left(  V,E\right)  $ be a finite graph. Let
$X$ be a totally ordered set. Let $\ell:E\rightarrow X$ be a labeling
function. Then,%
\[
\chi_{G}=\sum_{\substack{F\subseteq E;\\F\text{ contains no broken}%
\\\text{circuit of }G\text{ as a subset}}}\left(  -1\right)  ^{\left\vert
F\right\vert }x^{\operatorname*{conn}\left(  V,F\right)  }.
\]

\end{corollary}

Except for Theorem \ref{thm.chrompol.varis}, these results are not new; in
fact, Corollary \ref{cor.chrompol.K-free} is a particular case of
\cite[(12)]{DohTri14}, and of course we can obtain Corollary
\ref{cor.chrompol.NBC} and Theorem \ref{thm.chrompol.empty} as particular
cases of Corollary \ref{cor.chrompol.K-free}.

\subsection{Proofs}

Nevertheless, for the sake of completeness, we shall give proofs of all the
five results above (Theorem \ref{thm.chrompol.exist}, Theorem
\ref{thm.chrompol.empty}, Theorem \ref{thm.chrompol.varis}, Corollary
\ref{cor.chrompol.K-free} and Corollary \ref{cor.chrompol.NBC}).

There are two approaches to these results (except for Theorem
\ref{thm.chrompol.exist}): One is to prove them similarly to how we proved the
analogous results about $X_{G}$; the other is to derive them from the latter.
We shall take the first approach, since it yields a proof of the classical
Theorem \ref{thm.chrompol.exist} \textquotedblleft for free\textquotedblright.
We begin with an analogue of Lemma \ref{lem.Eqs.sum}:

\begin{lemma}
\label{lem.Eqs.sum1}Let $\left(  V,B\right)  $ be a finite graph. Let
$q\in\mathbb{N}$. Then,%
\[
\sum_{\substack{f:V\rightarrow\left\{  1,2,\ldots,q\right\}  ;\\B\subseteq
\operatorname*{Eqs}f}}1=q^{\operatorname*{conn}\left(  V,B\right)  }.
\]
(Here, the expression $\operatorname*{conn}\left(  V,B\right)  $ is understood
according to Definition \ref{def.conn}.)
\end{lemma}

One way to prove Lemma \ref{lem.Eqs.sum1} is to evaluate the equality given by
Lemma \ref{lem.Eqs.sum} at $x_{k}=%
\begin{cases}
1, & \text{if }k\leq q;\\
0, & \text{if }k>q
\end{cases}
$. Another proof can be obtained by mimicking our proof of Lemma
\ref{lem.Eqs.sum}:

\begin{vershort}
\begin{proof}
[Proof of Lemma \ref{lem.Eqs.sum1}.]Define $\left(  C_{1},C_{2},\ldots
,C_{k}\right)  $ as in the proof of Lemma \ref{lem.Eqs.sum}. Thus,
$\operatorname*{conn}\left(  V,B\right)  =k$. Define a map $\Phi$ as in the
proof of Lemma \ref{lem.Eqs.sum}, but with $\mathbb{N}_{+}$ replaced by
$\left\{  1,2,\ldots,q\right\}  $. Then,
\[
\Phi:\left\{  1,2,\ldots,q\right\}  ^{k}\rightarrow\left\{  f:V\rightarrow
\left\{  1,2,\ldots,q\right\}  \ \mid\ B\subseteq\operatorname*{Eqs}f\right\}
\]
is a bijection\footnote{This can be shown in the same way as for the map
$\Phi$ in the proof of Lemma \ref{lem.Eqs.sum}; we just have to replace every
$\mathbb{N}_{+}$ by $\left\{  1,2,\ldots,q\right\}  $.}. Now,%
\begin{align*}
&  \sum_{\substack{f:V\rightarrow\left\{  1,2,\ldots,q\right\}  ;\\B\subseteq
\operatorname*{Eqs}f}}1\\
&  =\sum_{\left(  s_{1},s_{2},\ldots,s_{k}\right)  \in\left\{  1,2,\ldots
,q\right\}  ^{k}}1\\
&  \ \ \ \ \ \ \ \ \ \ \ \ \ \ \ \ \ \ \ \ \left(
\begin{array}
[c]{c}%
\text{here, we have substituted }\Phi\left(  s_{1},s_{2},\ldots,s_{k}\right)
\text{ for }f\text{ in the sum, since}\\
\text{the map }\Phi:\left\{  1,2,\ldots,q\right\}  ^{k}\rightarrow\left\{
f:V\rightarrow\left\{  1,2,\ldots,q\right\}  \ \mid\ B\subseteq
\operatorname*{Eqs}f\right\} \\
\text{is a bijection}%
\end{array}
\right) \\
&  =\left(  \text{the number of all }\left(  s_{1},s_{2},\ldots,s_{k}\right)
\in\left\{  1,2,\ldots,q\right\}  ^{k}\right) \\
&  =q^{k}=q^{\operatorname*{conn}\left(  V,B\right)  }%
\ \ \ \ \ \ \ \ \ \ \left(  \text{since }k=\operatorname*{conn}\left(
V,B\right)  \right)  .
\end{align*}
This proves Lemma \ref{lem.Eqs.sum1}.
\end{proof}
\end{vershort}

\begin{verlong}
\begin{proof}
[Proof of Lemma \ref{lem.Eqs.sum1}.]Let $\sim$ denote the equivalence relation
$\sim_{\left(  V,B\right)  }$ (defined as in Definition
\ref{def.connectedness} \textbf{(a)}). The connected components of $\left(
V,B\right)  $ are the $\sim_{\left(  V,B\right)  }$-equivalence classes
(because this is how the connected components of $\left(  V,B\right)  $ are
defined). In other words, the connected components of $\left(  V,B\right)  $
are the $\sim$-equivalence classes (since $\sim$ is the relation
$\sim_{\left(  V,B\right)  }$). In other words, the connected components of
$\left(  V,B\right)  $ are the elements of $V/\left(  \sim\right)  $ (since
the elements of $V/\left(  \sim\right)  $ are the $\sim$-equivalence classes
(by the definition of $V/\left(  \sim\right)  $)).

Let $Y$ be the set $\left\{  1,2,\ldots,q\right\}  $. Thus, $\left\vert
Y\right\vert =q$. A set $Y_{\sim}^{V}$ is defined (according to Definition
\ref{def.relquot.maps} \textbf{(b)}).

Proposition \ref{prop.relquot.uniprop} \textbf{(b)} (applied to $X=V$) shows
that the map%
\[
Y^{V/\left(  \sim\right)  }\rightarrow Y_{\sim}^{V}%
,\ \ \ \ \ \ \ \ \ \ f\mapsto f\circ\pi_{V}%
\]
is a bijection. Thus, there exists a bijection $Y^{V/\left(  \sim\right)
}\rightarrow Y_{\sim}^{V}$ (namely, this map). Hence, $\left\vert Y_{\sim}%
^{V}\right\vert =\left\vert Y^{V/\left(  \sim\right)  }\right\vert $.

For every map $f:V\rightarrow Y$, we have the following equivalence:%
\begin{equation}
\left(  B\subseteq\operatorname*{Eqs}f\right)  \ \Longleftrightarrow\ \left(
f\in Y_{\sim}^{V}\right)  \label{pf.lem.Eqs.sum1.equivalence}%
\end{equation}
(according to Lemma \ref{lem.Eqs.sum-aux}). Thus, we have the following
equality of summation signs:
\[
\sum_{\substack{f:V\rightarrow Y;\\B\subseteq\operatorname*{Eqs}f}%
}=\sum_{\substack{f:V\rightarrow Y;\\f\in Y_{\sim}^{V}}}=\sum_{\substack{f\in
Y^{V};\\f\in Y_{\sim}^{V}}}=\sum_{f\in Y_{\sim}^{V}}%
\]
(since $Y_{\sim}^{V}$ is a subset of $Y^{V}$). Hence,%
\begin{equation}
\underbrace{\sum_{\substack{f:V\rightarrow Y;\\B\subseteq\operatorname*{Eqs}%
f}}}_{=\sum_{f\in Y_{\sim}^{V}}}1=\sum_{f\in Y_{\sim}^{V}}1=\left\vert
Y_{\sim}^{V}\right\vert \cdot1=\left\vert Y_{\sim}^{V}\right\vert =\left\vert
Y^{V/\left(  \sim\right)  }\right\vert . \label{pf.lem.Eqs.sum1.1}%
\end{equation}

Now, let $\left(  C_{1},C_{2},\ldots,C_{k}\right)  $ be a list of all
connected components of $\left(  V,B\right)  $.\ \ \ \ \footnote{Every
connected component of $\left(  V,B\right)  $ should appear exactly once in
this list.} Thus, $k$ is the number of connected components of $\left(
V,B\right)  $. In other words, $k$ is $\operatorname*{conn}\left(  V,B\right)
$ (since $\operatorname*{conn}\left(  V,B\right)  $ is the number of connected
components of $\left(  V,B\right)  $ (by the definition of
$\operatorname*{conn}\left(  V,B\right)  $)). In other words,
$k=\operatorname*{conn}\left(  V,B\right)  $.

Recall that $\left(  C_{1},C_{2},\ldots,C_{k}\right)  $ is a list of all
connected components of $\left(  V,B\right)  $. In other words, $\left(
C_{1},C_{2},\ldots,C_{k}\right)  $ is a list of all elements of $V/\left(
\sim\right)  $ (since the elements of $V/\left(  \sim\right)  $ are the
connected components of $\left(  V,B\right)  $). Moreover, every element of
$V/\left(  \sim\right)  $ appears exactly once in this list $\left(
C_{1},C_{2},\ldots,C_{k}\right)  $ (since the entries of the list $\left(
C_{1},C_{2},\ldots,C_{k}\right)  $ are pairwise distinct\footnote{since every
connected component of $\left(  V,B\right)  $ appears exactly once in this
list}). Thus, $\left(  C_{1},C_{2},\ldots,C_{k}\right)  $ is a list of all
elements of $V/\left(  \sim\right)  $, and contains each of these elements
exactly once. Hence, the map%
\begin{align*}
Y^{V/\left(  \sim\right)  }  &  \rightarrow Y^{k},\\
f  &  \mapsto\left(  f\left(  C_{1}\right)  ,f\left(  C_{2}\right)
,\ldots,f\left(  C_{k}\right)  \right)
\end{align*}
is a bijection (by Lemma \ref{lem.function-count}, applied to $W=V/\left(
\sim\right)  $). Hence, there exists a bijection $Y^{V/\left(  \sim\right)
}\rightarrow Y^{k}$ (namely, this map). Thus, $\left\vert Y^{k}\right\vert
=\left\vert Y^{V/\left(  \sim\right)  }\right\vert $.

Comparing this with (\ref{pf.lem.Eqs.sum1.1}), we obtain%
\begin{align*}
\sum_{\substack{f:V\rightarrow Y;\\B\subseteq\operatorname*{Eqs}f}}1  &
=\left\vert Y^{k}\right\vert =\left\vert Y\right\vert ^{k}=q^{k}%
\ \ \ \ \ \ \ \ \ \ \left(  \text{since }\left\vert Y\right\vert =q\right) \\
&  =q^{\operatorname*{conn}\left(  V,B\right)  }\ \ \ \ \ \ \ \ \ \ \left(
\text{since }k=\operatorname*{conn}\left(  V,B\right)  \right)  .
\end{align*}
This proves Lemma \ref{lem.Eqs.sum1}.
\end{proof}
\end{verlong}

We shall now show a weaker version of Theorem \ref{thm.chrompol.varis} (as a
stepping stone to the actual theorem):

\begin{lemma}
\label{lem.chrompol.weak.varis}Let $G=\left(  V,E\right)  $ be a finite graph.
Let $X$ be a totally ordered set. Let $\ell:E\rightarrow X$ be a labeling
function. Let $\mathfrak{K}$ be some set of broken circuits of $G$ (not
necessarily containing all of them). Let $a_{K}$ be an element of $\mathbf{k}$
for every $K\in\mathfrak{K}$. Let $q\in\mathbb{N}$. Then,%
\begin{align*}
&  \left(  \text{the number of all proper }\left\{  1,2,\ldots,q\right\}
\text{-colorings of }G\right) \\
&  =\sum_{F\subseteq E}\left(  -1\right)  ^{\left\vert F\right\vert }\left(
\prod_{\substack{K\in\mathfrak{K};\\K\subseteq F}}a_{K}\right)
q^{\operatorname*{conn}\left(  V,F\right)  }.
\end{align*}
(Here, of course, the pair $\left(  V,F\right)  $ is regarded as a graph, and
the expression $\operatorname*{conn}\left(  V,F\right)  $ is understood
according to Definition \ref{def.conn}.)
\end{lemma}

\begin{vershort}
\begin{proof}
[Proof of Lemma \ref{lem.chrompol.weak.varis}.]We have\footnote{We are again
using the Iverson bracket notation, as defined in Definition \ref{def.iverson}%
.}%
\begin{align*}
&  \left(  \text{the number of all proper }\left\{  1,2,\ldots,q\right\}
\text{-colorings of }G\right) \\
&  =\sum_{f:V\rightarrow\left\{  1,2,\ldots,q\right\}  }\left[
\underbrace{f\text{ is a proper }\left\{  1,2,\ldots,q\right\}
\text{-coloring of }G}_{\substack{\Longleftrightarrow\ \left(  \text{the
}\left\{  1,2,\ldots,q\right\}  \text{-coloring }f\text{ of }G\text{ is
proper}\right)  \\\Longleftrightarrow\ \left(  E\cap\operatorname*{Eqs}%
f=\varnothing\right)  \\\text{(by Lemma \ref{lem.Eqs.proper}, applied to
}\left\{  1,2,\ldots,q\right\}  \text{ instead of }X\text{)}}}\right] \\
&  =\sum_{f:V\rightarrow\left\{  1,2,\ldots,q\right\}  }\underbrace{\left[
E\cap\operatorname*{Eqs}f=\varnothing\right]  }_{\substack{=\sum_{B\subseteq
E\cap\operatorname*{Eqs}f}\left(  -1\right)  ^{\left\vert B\right\vert }%
\prod_{\substack{K\in\mathfrak{K};\\K\subseteq B}}a_{K}\\\text{(by Lemma
\ref{lem.NBCm.moeb}, applied to }Y=\left\{  1,2,\ldots,q\right\}  \text{)}}}\\
&  =\sum_{f:V\rightarrow\left\{  1,2,\ldots,q\right\}  }\ \ \underbrace{\sum
_{B\subseteq E\cap\operatorname*{Eqs}f}}_{=\sum_{\substack{B\subseteq
E;\\B\subseteq\operatorname*{Eqs}f}}}\left(  -1\right)  ^{\left\vert
B\right\vert }\left(  \prod_{\substack{K\in\mathfrak{K};\\K\subseteq B}%
}a_{K}\right) \\
&  =\underbrace{\sum_{f:V\rightarrow\left\{  1,2,\ldots,q\right\}  }%
\ \ \sum_{\substack{B\subseteq E;\\B\subseteq\operatorname*{Eqs}f}}}%
_{=\sum_{B\subseteq E}\ \ \sum_{\substack{f:V\rightarrow\left\{
1,2,\ldots,q\right\}  ;\\B\subseteq\operatorname*{Eqs}f}}}\left(  -1\right)
^{\left\vert B\right\vert }\left(  \prod_{\substack{K\in\mathfrak{K}%
;\\K\subseteq B}}a_{K}\right) \\
&  =\sum_{B\subseteq E}\ \ \sum_{\substack{f:V\rightarrow\left\{
1,2,\ldots,q\right\}  ;\\B\subseteq\operatorname*{Eqs}f}}\left(  -1\right)
^{\left\vert B\right\vert }\left(  \prod_{\substack{K\in\mathfrak{K}%
;\\K\subseteq B}}a_{K}\right)  =\sum_{B\subseteq E}\left(  -1\right)
^{\left\vert B\right\vert }\left(  \prod_{\substack{K\in\mathfrak{K}%
;\\K\subseteq B}}a_{K}\right)  \underbrace{\sum_{\substack{f:V\rightarrow
\left\{  1,2,\ldots,q\right\}  ;\\B\subseteq\operatorname*{Eqs}f}%
}1}_{\substack{=q^{\operatorname*{conn}\left(  V,B\right)  }\\\text{(by Lemma
\ref{lem.Eqs.sum1})}}}\\
&  =\sum_{B\subseteq E}\left(  -1\right)  ^{\left\vert B\right\vert }\left(
\prod_{\substack{K\in\mathfrak{K};\\K\subseteq B}}a_{K}\right)
q^{\operatorname*{conn}\left(  V,B\right)  }=\sum_{F\subseteq E}\left(
-1\right)  ^{\left\vert F\right\vert }\left(  \prod_{\substack{K\in
\mathfrak{K};\\K\subseteq F}}a_{K}\right)  q^{\operatorname*{conn}\left(
V,F\right)  }%
\end{align*}
(here, we have renamed the summation index $B$ as $F$). This proves Lemma
\ref{lem.chrompol.weak.varis}.
\end{proof}
\end{vershort}

\begin{verlong}
\begin{proof}
[Proof of Lemma \ref{lem.chrompol.weak.varis}.]Let $Q=\left\{  1,2,\ldots
,q\right\}  $. If $f:V\rightarrow Q$ is a map, then we have the following
logical equivalence:%
\begin{equation}
\left(  \text{the }Q\text{-coloring }f\text{ of }G\text{ is proper}\right)
\ \Longleftrightarrow\ \left(  E\cap\operatorname*{Eqs}f=\varnothing\right)
\label{pf.lem.chrompol.weak.varis.equiv}%
\end{equation}
(because the $Q$-coloring $f$ of $G$ is proper if and only if $E\cap
\operatorname*{Eqs}f=\varnothing$\ \ \ \ \footnote{by Lemma
\ref{lem.Eqs.proper} (applied to $Q$ instead of $X$)}). Now,\footnote{We are
again using the Iverson bracket notation, as defined in Definition
\ref{def.iverson}.}%
\begin{align}
&  \sum_{f:V\rightarrow Q}\left[  \underbrace{E\cap\operatorname*{Eqs}%
f=\varnothing}_{\substack{\Longleftrightarrow\ \left(  \text{the
}Q\text{-coloring }f\text{ of }G\text{ is proper}\right)  \\\text{(by
(\ref{pf.lem.chrompol.weak.varis.equiv}))}}}\right] \nonumber\\
&  =\sum_{f:V\rightarrow Q}\left[  \underbrace{\text{the }Q\text{-coloring
}f\text{ of }G\text{ is proper}}_{\Longleftrightarrow\ \left(  f\text{ is a
proper }Q\text{-coloring of }G\right)  }\right] \nonumber\\
&  =\sum_{f:V\rightarrow Q}\left[  f\text{ is a proper }Q\text{-coloring of
}G\right] \nonumber\\
&  =\sum_{\substack{f:V\rightarrow Q\text{ is a}\\\text{proper }%
Q\text{-coloring of }G}}\underbrace{\left[  f\text{ is a proper }%
Q\text{-coloring of }G\right]  }_{\substack{=1\\\text{(since }f\text{ is a
proper }Q\text{-coloring of }G\text{)}}}\nonumber\\
&  \ \ \ \ \ \ \ \ \ \ +\sum_{\substack{f:V\rightarrow Q\text{ is not
a}\\\text{proper }Q\text{-coloring of }G}}\underbrace{\left[  f\text{ is a
proper }Q\text{-coloring of }G\right]  }_{\substack{=0\\\text{(since }f\text{
is not a proper }Q\text{-coloring of }G\text{)}}}\nonumber\\
&  =\sum_{\substack{f:V\rightarrow Q\text{ is a}\\\text{proper }%
Q\text{-coloring of }G}}1+\underbrace{\sum_{\substack{f:V\rightarrow Q\text{
is not a}\\\text{proper }Q\text{-coloring of }G}}0}_{=0}=\sum
_{\substack{f:V\rightarrow Q\text{ is a}\\\text{proper }Q\text{-coloring of
}G}}1\nonumber\\
&  =\left(  \text{the number of all proper }Q\text{-colorings of }G\right)
\cdot1\nonumber\\
&  =\left(  \text{the number of all proper }Q\text{-colorings of }G\right)
\nonumber\\
&  =\left(  \text{the number of all proper }\left\{  1,2,\ldots,q\right\}
\text{-colorings of }G\right)  \label{pf.lem.chrompol.weak.varis.1}%
\end{align}
(since $Q=\left\{  1,2,\ldots,q\right\}  $).

However, for every $f:V\rightarrow Q$, we have%
\begin{equation}
\sum_{B\subseteq E\cap\operatorname*{Eqs}f}\left(  -1\right)  ^{\left\vert
B\right\vert }\prod_{\substack{K\in\mathfrak{K};\\K\subseteq B}}a_{K}=\left[
E\cap\operatorname*{Eqs}f=\varnothing\right]
\label{pf.lem.chrompol.weak.varis.moeb}%
\end{equation}
(by Lemma \ref{lem.NBCm.moeb} (applied to $Q$ instead of $Y$)).

For every $f:V\rightarrow Q$, we have%
\begin{equation}
\left\{  F\subseteq E\ \mid\ F\subseteq\operatorname*{Eqs}f\right\}
=\mathcal{P}\left(  E\cap\operatorname*{Eqs}f\right)
\label{pf.lem.chrompol.weak.varis.cut}%
\end{equation}
\footnote{\textit{Proof of (\ref{pf.lem.chrompol.weak.varis.cut}):} Let
$f:V\rightarrow Q$.
\par
Let $B\in\left\{  F\subseteq E\ \mid\ F\subseteq\operatorname*{Eqs}f\right\}
$. Thus, $B$ is a subset $F$ of $E$ satisfying $F\subseteq\operatorname*{Eqs}%
f$. In other words, $B$ is a subset of $E$ and satisfies $B\subseteq
\operatorname*{Eqs}f$. Since $B$ is a subset of $E$, we have $B\subseteq E$.
Combining this with $B\subseteq\operatorname*{Eqs}f$, we obtain $B\subseteq
E\cap\operatorname*{Eqs}f$. In other words, $B\in\mathcal{P}\left(
E\cap\operatorname*{Eqs}f\right)  $.
\par
Let us now forget that we fixed $B$. We thus have proven that every
$B\in\left\{  F\subseteq E\ \mid\ F\subseteq\operatorname*{Eqs}f\right\}  $
satisfies $B\in\mathcal{P}\left(  E\cap\operatorname*{Eqs}f\right)  $. In
other words,%
\begin{equation}
\left\{  F\subseteq E\ \mid\ F\subseteq\operatorname*{Eqs}f\right\}
\subseteq\mathcal{P}\left(  E\cap\operatorname*{Eqs}f\right)  .
\label{pf.lem.chrompol.weak.varis.cut.pf.1}%
\end{equation}
\par
On the other hand, let $C\in\mathcal{P}\left(  E\cap\operatorname*{Eqs}%
f\right)  $. Thus, $C$ is a subset of $E\cap\operatorname*{Eqs}f$. Hence,
$C\subseteq E\cap\operatorname*{Eqs}f\subseteq E$, so that $C$ is a subset of
$E$. Also, $C\subseteq E\cap\operatorname*{Eqs}f\subseteq\operatorname*{Eqs}%
f$. Thus, $C$ is a subset of $E$ and satisfies $C\subseteq\operatorname*{Eqs}%
f$. In other words, $C$ is a subset $F$ of $E$ satisfying $F\subseteq
\operatorname*{Eqs}f$. In other words, $C\in\left\{  F\subseteq E\ \mid
\ F\subseteq\operatorname*{Eqs}f\right\}  $.
\par
Let us now forget that we fixed $C$. We thus have proven that every
$C\in\mathcal{P}\left(  E\cap\operatorname*{Eqs}f\right)  $ satisfies
$C\in\left\{  F\subseteq E\ \mid\ F\subseteq\operatorname*{Eqs}f\right\}  $.
In other words,%
\[
\mathcal{P}\left(  E\cap\operatorname*{Eqs}f\right)  \subseteq\left\{
F\subseteq E\ \mid\ F\subseteq\operatorname*{Eqs}f\right\}  .
\]
Combining this inclusion with (\ref{pf.lem.chrompol.weak.varis.cut.pf.1}), we
obtain $\left\{  F\subseteq E\ \mid\ F\subseteq\operatorname*{Eqs}f\right\}
=\mathcal{P}\left(  E\cap\operatorname*{Eqs}f\right)  $. This proves
(\ref{pf.lem.chrompol.weak.varis.cut}).}.

For every $f:V\rightarrow Q$, we have%
\begin{align}
&  \underbrace{\sum_{\substack{B\subseteq E;\\B\subseteq\operatorname*{Eqs}%
f}}}_{\substack{=\sum_{B\in\left\{  F\subseteq E\ \mid\ F\subseteq
\operatorname*{Eqs}f\right\}  }=\sum_{B\in\mathcal{P}\left(  E\cap
\operatorname*{Eqs}f\right)  }\\\text{(because }\left\{  F\subseteq
E\ \mid\ F\subseteq\operatorname*{Eqs}f\right\}  =\mathcal{P}\left(
E\cap\operatorname*{Eqs}f\right)  \\\text{(by
(\ref{pf.lem.chrompol.weak.varis.cut})))}}}\left(  -1\right)  ^{\left\vert
B\right\vert }\prod_{\substack{K\in\mathfrak{K};\\K\subseteq B}}a_{K}%
\nonumber\\
&  =\underbrace{\sum_{B\in\mathcal{P}\left(  E\cap\operatorname*{Eqs}f\right)
}}_{=\sum_{B\subseteq E\cap\operatorname*{Eqs}f}}\left(  -1\right)
^{\left\vert B\right\vert }\prod_{\substack{K\in\mathfrak{K};\\K\subseteq
B}}a_{K}=\sum_{B\subseteq E\cap\operatorname*{Eqs}f}\left(  -1\right)
^{\left\vert B\right\vert }\prod_{\substack{K\in\mathfrak{K};\\K\subseteq
B}}a_{K}\nonumber\\
&  =\left[  E\cap\operatorname*{Eqs}f=\varnothing\right]
\label{pf.lem.chrompol.weak.varis.moeb2}%
\end{align}
(by (\ref{pf.lem.chrompol.weak.varis.moeb})).

Now, (\ref{pf.lem.chrompol.weak.varis.1}) yields%
\begin{align*}
&  \left(  \text{the number of all proper }\left\{  1,2,\ldots,q\right\}
\text{-colorings of }G\right) \\
&  =\sum_{f:V\rightarrow Q}\underbrace{\left[  E\cap\operatorname*{Eqs}%
f=\varnothing\right]  }_{\substack{=\sum_{\substack{B\subseteq E;\\B\subseteq
\operatorname*{Eqs}f}}\left(  -1\right)  ^{\left\vert B\right\vert }%
\prod_{\substack{K\in\mathfrak{K};\\K\subseteq B}}a_{K}\\\text{(by
(\ref{pf.lem.chrompol.weak.varis.moeb2}))}}}\\
&  =\sum_{f:V\rightarrow Q}\left(  \sum_{\substack{B\subseteq E;\\B\subseteq
\operatorname*{Eqs}f}}\left(  -1\right)  ^{\left\vert B\right\vert }%
\prod_{\substack{K\in\mathfrak{K};\\K\subseteq B}}a_{K}\right)
=\underbrace{\sum_{f:V\rightarrow Q}\ \ \sum_{\substack{B\subseteq
E;\\B\subseteq\operatorname*{Eqs}f}}}_{=\sum_{B\subseteq E}\ \ \sum
_{\substack{f:V\rightarrow Q;\\B\subseteq\operatorname*{Eqs}f}}}\left(
-1\right)  ^{\left\vert B\right\vert }\underbrace{\left(  \prod
_{\substack{K\in\mathfrak{K};\\K\subseteq B}}a_{K}\right)  }_{=\left(
\prod_{\substack{K\in\mathfrak{K};\\K\subseteq B}}a_{K}\right)  1}\\
&  =\sum_{B\subseteq E}\ \ \sum_{\substack{f:V\rightarrow Q;\\B\subseteq
\operatorname*{Eqs}f}}\left(  -1\right)  ^{\left\vert B\right\vert }\left(
\prod_{\substack{K\in\mathfrak{K};\\K\subseteq B}}a_{K}\right)  1=\sum
_{B\subseteq E}\ \ \sum_{\substack{f:V\rightarrow\left\{  1,2,\ldots
,q\right\}  ;\\B\subseteq\operatorname*{Eqs}f}}\left(  -1\right)  ^{\left\vert
B\right\vert }\left(  \prod_{\substack{K\in\mathfrak{K};\\K\subseteq B}%
}a_{K}\right)  1\\
&  \ \ \ \ \ \ \ \ \ \ \ \ \ \ \ \ \ \ \ \ \left(  \text{since }Q=\left\{
1,2,\ldots,q\right\}  \right) \\
&  =\sum_{B\subseteq E}\left(  -1\right)  ^{\left\vert B\right\vert }\left(
\prod_{\substack{K\in\mathfrak{K};\\K\subseteq B}}a_{K}\right)
\underbrace{\sum_{\substack{f:V\rightarrow\left\{  1,2,\ldots,q\right\}
;\\B\subseteq\operatorname*{Eqs}f}}1}_{\substack{=q^{\operatorname*{conn}%
\left(  V,B\right)  }\\\text{(by Lemma \ref{lem.Eqs.sum1}}\\\text{(since
}\left(  V,B\right)  \text{ is a finite graph}\\\text{(since }V\text{ is a
finite set and }B\subseteq E\subseteq\dbinom{V}{2}\text{)))}}}\\
&  =\sum_{B\subseteq E}\left(  -1\right)  ^{\left\vert B\right\vert }\left(
\prod_{\substack{K\in\mathfrak{K};\\K\subseteq B}}a_{K}\right)
q^{\operatorname*{conn}\left(  V,B\right)  }=\sum_{F\subseteq E}\left(
-1\right)  ^{\left\vert F\right\vert }\left(  \prod_{\substack{K\in
\mathfrak{K};\\K\subseteq F}}a_{K}\right)  q^{\operatorname*{conn}\left(
V,F\right)  }%
\end{align*}
(here, we have renamed the summation index $B$ as $F$). This proves Lemma
\ref{lem.chrompol.weak.varis}.
\end{proof}
\end{verlong}

From Lemma \ref{lem.chrompol.weak.varis}, we obtain the following consequence:

\begin{lemma}
\label{lem.chrompol.weak.empty}Let $G=\left(  V,E\right)  $ be a finite graph.
Let $q\in\mathbb{N}$. Then,%
\begin{align*}
&  \left(  \text{the number of all proper }\left\{  1,2,\ldots,q\right\}
\text{-colorings of }G\right) \\
&  =\sum_{F\subseteq E}\left(  -1\right)  ^{\left\vert F\right\vert
}q^{\operatorname*{conn}\left(  V,F\right)  }.
\end{align*}
(Here, of course, the pair $\left(  V,F\right)  $ is regarded as a graph, and
the expression $\operatorname*{conn}\left(  V,F\right)  $ is understood
according to Definition \ref{def.conn}.)
\end{lemma}

\begin{vershort}
\begin{proof}
[Proof of Lemma \ref{lem.chrompol.weak.empty}.]This is derived from Lemma
\ref{lem.chrompol.weak.varis} in the same way as Theorem
\ref{thm.chromsym.empty} was derived from Theorem \ref{thm.chromsym.varis}.
\end{proof}
\end{vershort}

\begin{verlong}
\begin{proof}
[Proof of Lemma \ref{lem.chrompol.weak.empty}.]Let $X$ be the totally ordered
set $\left\{  1\right\}  $ (equipped with the only possible order on this
set). Let $\ell:E\rightarrow X$ be the function sending each $e\in E$ to $1\in
X$. Let $\mathfrak{K}$ be the empty set. Clearly, $\mathfrak{K}$ is a set of
broken circuits of $G$. Lemma \ref{lem.chrompol.weak.varis} (applied to $0$
instead of $a_{K}$) yields%
\begin{align*}
&  \left(  \text{the number of all proper }\left\{  1,2,\ldots,q\right\}
\text{-colorings of }G\right) \\
&  =\sum_{F\subseteq E}\left(  -1\right)  ^{\left\vert F\right\vert
}\underbrace{\left(  \prod_{\substack{K\in\mathfrak{K};\\K\subseteq
F}}0\right)  }_{\substack{=\left(  \text{empty product}\right)  \\\text{(since
}\mathfrak{K}\text{ is the empty set)}}}q^{\operatorname*{conn}\left(
V,F\right)  }\\
&  =\sum_{F\subseteq E}\left(  -1\right)  ^{\left\vert F\right\vert
}\underbrace{\left(  \text{empty product}\right)  }_{=1}%
q^{\operatorname*{conn}\left(  V,F\right)  }=\sum_{F\subseteq E}\left(
-1\right)  ^{\left\vert F\right\vert }q^{\operatorname*{conn}\left(
V,F\right)  }.
\end{align*}
This proves Lemma \ref{lem.chrompol.weak.empty}.
\end{proof}
\end{verlong}

Next, we recall a classical fact about polynomials over fields: Namely, if a
polynomial (in one variable) over a field has infinitely many roots, then this
polynomial is $0$. Let us state this more formally:

\begin{proposition}
\label{prop.poly0.field}Let $K$ be a field. Let $P\in K\left[  x\right]  $ be
a polynomial over $K$. Assume that there are infinitely many $\lambda\in K$
satisfying $P\left(  \lambda\right)  =0$. Then, $P=0$.
\end{proposition}

We shall use the following two consequences of this proposition:

\begin{corollary}
\label{cor.poly0.intdom}Let $R$ be an integral domain. Assume that the
canonical ring homomorphism from the ring $\mathbb{Z}$ to the ring $R$ is
injective. Let $P\in R\left[  x\right]  $ be a polynomial over $R$. Assume
that $P\left(  q\cdot1_{R}\right)  =0$ for every $q\in\mathbb{N}$ (where
$1_{R}$ denotes the unity of $R$). Then, $P=0$.
\end{corollary}

\begin{vershort}
\begin{proof}
[Proof of Corollary \ref{cor.poly0.intdom}.]Let $K$ denote the fraction field
of the integral domain $R$. We regard $R$ and $R\left[  x\right]  $ as
subrings of $K$ and $K\left[  x\right]  $, respectively. By assumption, we
have $P\left(  q\cdot1_{R}\right)  =0$ for every $q\in\mathbb{N}$. But the
elements $q\cdot1_{R}$ of $R$ for $q\in\mathbb{N}$ are pairwise distinct
(since the canonical ring homomorphism from the ring $\mathbb{Z}$ to the ring
$R$ is injective). Hence, there are infinitely many $\lambda\in K$ satisfying
$P\left(  \lambda\right)  =0$ (namely, $\lambda=q\cdot1_{R}$ for all
$q\in\mathbb{N}$). Thus, Proposition \ref{prop.poly0.field} shows that $P=0$.
This proves Corollary \ref{cor.poly0.intdom}.
\end{proof}
\end{vershort}

\begin{verlong}
\begin{proof}
[Proof of Corollary \ref{cor.poly0.intdom}.]Let $K$ denote the fraction field
of the integral domain $R$. Then, there is a canonical injective ring
homomorphism $R\rightarrow K$. We use this homomorphism to regard $R$ as a
subring of $K$. Consequently, $R\left[  x\right]  $ will be regarded as a
subring of $K\left[  x\right]  $. In particular, the polynomial $P\in R\left[
x\right]  $ will thus be regarded as a polynomial in $K\left[  x\right]  $.

Let $\iota:\mathbb{Z}\rightarrow R$ be the canonical ring homomorphism from
the ring $\mathbb{Z}$ to the ring $R$. (Thus, $\iota$ sends every
$q\in\mathbb{Z}$ to $q\cdot1_{R}\in R$.)

We have assumed that the canonical ring homomorphism from the ring
$\mathbb{Z}$ to the ring $R$ is injective. In other words, the map $\iota$ is
injective (since the map $\iota$ is the canonical ring homomorphism from the
ring $\mathbb{Z}$ to the ring $R$). Hence, $\left\vert \iota\left(
\mathbb{N}\right)  \right\vert =\left\vert \mathbb{N}\right\vert =\infty$.
Moreover, every $\lambda\in\iota\left(  \mathbb{N}\right)  $ satisfies
$P\left(  \lambda\right)  =0$\ \ \ \ \footnote{\textit{Proof.} Let $\lambda
\in\iota\left(  \mathbb{N}\right)  $. Thus, there exists some $h\in\mathbb{N}$
such that $\lambda=\iota\left(  h\right)  $. Consider this $h$.
\par
We have assumed that $P\left(  q\cdot1_{R}\right)  =0$ for every
$q\in\mathbb{N}$. Applying this to $q=h$, we obtain $P\left(  h\cdot
1_{R}\right)  =0$. But $\lambda=\iota\left(  h\right)  =h\cdot1_{R}$ (by the
definition of $\iota$). Hence, $P\left(  \underbrace{\lambda}_{=h\cdot1_{R}%
}\right)  =P\left(  h\cdot1_{R}\right)  =0$, qed.}. Hence, there are
infinitely many $\lambda\in K$ satisfying $P\left(  \lambda\right)  =0$
(because there are infinitely many $\lambda\in\iota\left(  \mathbb{N}\right)
$ (since $\left\vert \iota\left(  \mathbb{N}\right)  \right\vert =\infty$),
and because they all are elements of $K$ (since $\iota\left(  \mathbb{N}%
\right)  \subseteq R\subseteq K$)). Hence, Proposition \ref{prop.poly0.field}
shows that $P=0$. This proves Corollary \ref{cor.poly0.intdom}.
\end{proof}
\end{verlong}

\begin{corollary}
\label{cor.poly0.P1=P2}Let $R$ be an integral domain such that $\mathbb{Z}$ is
a subring of $R$. Let $P_{1}\in R\left[  x\right]  $ and $P_{2}\in R\left[
x\right]  $ be two polynomials over $R$. Assume that every $q\in\mathbb{N}$
satisfies%
\begin{equation}
P_{1}\left(  q\right)  =P_{2}\left(  q\right)  .
\label{eq.cor.poly0.P1=P2.ass}%
\end{equation}
Then, $P_{1}=P_{2}$.
\end{corollary}

\begin{vershort}
\begin{proof}
[Proof of Corollary \ref{cor.poly0.P1=P2}.]For every $q\in\mathbb{N}$, we have
$\left(  P_{1}-P_{2}\right)  \left(  q\right)  =P_{1}\left(  q\right)
-P_{2}\left(  q\right)  =0$ (by (\ref{eq.cor.poly0.P1=P2.ass})). Hence,
Corollary \ref{cor.poly0.intdom} (applied to $P=P_{1}-P_{2}$) yields that
$P_{1}-P_{2}=0$. In other words, $P_{1}=P_{2}$. This proves Corollary
\ref{cor.poly0.P1=P2}.
\end{proof}
\end{vershort}

\begin{verlong}
\begin{proof}
[Proof of Corollary \ref{cor.poly0.P1=P2}.]We have assumed that $\mathbb{Z}$
is a subring of $R$. Hence, the canonical ring homomorphism from the ring
$\mathbb{Z}$ to the ring $R$ is just the inclusion map $\mathbb{Z}\rightarrow
R$, and thus is injective.

Every $q\in\mathbb{N}$ satisfies
\[
\left(  P_{1}-P_{2}\right)  \left(  \underbrace{q\cdot1_{R}}_{=q}\right)
=\left(  P_{1}-P_{2}\right)  \left(  q\right)  =\underbrace{P_{1}\left(
q\right)  }_{\substack{=P_{2}\left(  q\right)  \\\text{(by
(\ref{eq.cor.poly0.P1=P2.ass}))}}}-P_{2}\left(  q\right)  =P_{2}\left(
q\right)  -P_{2}\left(  q\right)  =0.
\]
In other words, we have $\left(  P_{1}-P_{2}\right)  \left(  q\cdot
1_{R}\right)  =0$ for every $q\in\mathbb{N}$. Hence, Corollary
\ref{cor.poly0.intdom} (applied to $P=P_{1}-P_{2}$) yields that $P_{1}%
-P_{2}=0$. In other words, $P_{1}=P_{2}$. This proves Corollary
\ref{cor.poly0.P1=P2}.
\end{proof}
\end{verlong}

We can now prove the classical Theorem \ref{thm.chrompol.exist}:

\begin{vershort}
\begin{proof}
[Proof of Theorem \ref{thm.chrompol.exist}.]We need to show that there exists
a unique polynomial $P\in\mathbb{Z}\left[  x\right]  $ such that every
$q\in\mathbb{N}$ satisfies%
\[
P\left(  q\right)  =\left(  \text{the number of all proper }\left\{
1,2,\ldots,q\right\}  \text{-colorings of }G\right)  .
\]
To see that such a polynomial exists, we notice that $P=\sum_{F\subseteq
E}\left(  -1\right)  ^{\left\vert F\right\vert }x^{\operatorname*{conn}\left(
V,F\right)  }$ is such a polynomial (by Lemma \ref{lem.chrompol.weak.empty}).
It remains to prove that such a polynomial is unique. But this follows
directly from Corollary \ref{cor.poly0.P1=P2} (applied to $R=\mathbb{Z}$).
Theorem \ref{thm.chrompol.exist} is therefore proven.
\end{proof}
\end{vershort}

\begin{verlong}
\begin{proof}
[Proof of Theorem \ref{thm.chrompol.exist}.]We need to show that there exists
a unique polynomial $P\in\mathbb{Z}\left[  x\right]  $ such that every
$q\in\mathbb{N}$ satisfies%
\begin{equation}
P\left(  q\right)  =\left(  \text{the number of all proper }\left\{
1,2,\ldots,q\right\}  \text{-colorings of }G\right)  .
\label{pf.thm.chrompol.exist.P(q)=}%
\end{equation}

Let us first show that there exists at most one such polynomial. Indeed, let
$P_{1}$ and $P_{2}$ be two polynomials $P\in\mathbb{Z}\left[  x\right]  $ such
that every $q\in\mathbb{N}$ satisfies (\ref{pf.thm.chrompol.exist.P(q)=}). We
shall show that $P_{1}=P_{2}$.

We know that $P_{1}$ is a polynomial $P\in\mathbb{Z}\left[  x\right]  $ such
that every $q\in\mathbb{N}$ satisfies (\ref{pf.thm.chrompol.exist.P(q)=}). In
other words, $P_{1}$ is a polynomial in $\mathbb{Z}\left[  x\right]  $ and
every $q\in\mathbb{N}$ satisfies
\begin{equation}
P_{1}\left(  q\right)  =\left(  \text{the number of all proper }\left\{
1,2,\ldots,q\right\}  \text{-colorings of }G\right)  .
\label{pf.thm.chrompol.exist.P1(q)=}%
\end{equation}
The same argument (applied to $P_{2}$ instead of $P_{1}$) shows that $P_{2}$
is a polynomial in $\mathbb{Z}\left[  x\right]  $ and every $q\in\mathbb{N}$
satisfies
\begin{equation}
P_{2}\left(  q\right)  =\left(  \text{the number of all proper }\left\{
1,2,\ldots,q\right\}  \text{-colorings of }G\right)  .
\label{pf.thm.chrompol.exist.P2(q)=}%
\end{equation}

The ring $\mathbb{Z}$ is clearly a subring of $\mathbb{Z}$. Furthermore, every
$q\in\mathbb{N}$ satisfies%
\begin{align}
P_{1}\left(  q\right)   &  =\left(  \text{the number of all proper }\left\{
1,2,\ldots,q\right\}  \text{-colorings of }G\right) \nonumber\\
&  \ \ \ \ \ \ \ \ \ \ \ \ \ \ \ \ \ \ \ \ \left(  \text{by
(\ref{pf.thm.chrompol.exist.P1(q)=})}\right) \nonumber\\
&  =P_{2}\left(  q\right)  \ \ \ \ \ \ \ \ \ \ \left(  \text{by
(\ref{pf.thm.chrompol.exist.P2(q)=})}\right)  .
\label{pf.thm.chrompol.exist.4}%
\end{align}
Hence, Corollary \ref{cor.poly0.P1=P2} (applied to $\mathbb{Z}$ instead of
$R$) shows that $P_{1}=P_{2}$.

Now, let us forget that we fixed $P_{1}$ and $P_{2}$. We thus have shown that
if $P_{1}$ and $P_{2}$ are two polynomials $P\in\mathbb{Z}\left[  x\right]  $
such that every $q\in\mathbb{N}$ satisfies (\ref{pf.thm.chrompol.exist.P(q)=}%
), then $P_{1}=P_{2}$. In other words, there exists \textbf{at most one}
polynomial $P\in\mathbb{Z}\left[  x\right]  $ such that every $q\in\mathbb{N}$
satisfies (\ref{pf.thm.chrompol.exist.P(q)=}).

Let us now prove that there exists at least one such polynomial. Indeed,
define a polynomial $Q\in\mathbb{Z}\left[  x\right]  $ by%
\begin{equation}
Q=\sum_{F\subseteq E}\left(  -1\right)  ^{\left\vert F\right\vert
}x^{\operatorname*{conn}\left(  V,F\right)  }.
\label{pf.thm.chrompol.exist.Q=}%
\end{equation}
Then, every $q\in\mathbb{N}$ satisfies%
\begin{align*}
Q\left(  q\right)   &  =\sum_{F\subseteq E}\left(  -1\right)  ^{\left\vert
F\right\vert }q^{\operatorname*{conn}\left(  V,F\right)  }%
\ \ \ \ \ \ \ \ \ \ \left(  \text{here, we have substituted }q\text{ for
}x\text{ in (\ref{pf.thm.chrompol.exist.Q=})}\right) \\
&  =\left(  \text{the number of all proper }\left\{  1,2,\ldots,q\right\}
\text{-colorings of }G\right)
\end{align*}
(by Lemma \ref{lem.chrompol.weak.empty}). In other words, every $q\in
\mathbb{N}$ satisfies (\ref{pf.thm.chrompol.exist.P(q)=}) for $P=Q$. Thus, $Q$
is a polynomial $P\in\mathbb{Z}\left[  x\right]  $ such that every
$q\in\mathbb{N}$ satisfies (\ref{pf.thm.chrompol.exist.P(q)=}). Hence, there
exists \textbf{at least} one polynomial $P\in\mathbb{Z}\left[  x\right]  $
such that every $q\in\mathbb{N}$ satisfies (\ref{pf.thm.chrompol.exist.P(q)=})
(namely, $P=Q$). Consequently, there exists \textbf{exactly} one polynomial
$P\in\mathbb{Z}\left[  x\right]  $ such that every $q\in\mathbb{N}$ satisfies
(\ref{pf.thm.chrompol.exist.P(q)=}) (because we have already shown that there
exists \textbf{at most} one such polynomial). This proves Theorem
\ref{thm.chrompol.exist}.
\end{proof}
\end{verlong}

Next, it is the turn of Theorem \ref{thm.chrompol.varis}:

\begin{vershort}
\begin{proof}
[Proof of Theorem \ref{thm.chrompol.varis}.]Let $R$ be the polynomial ring
$\mathbb{Z}\left[  y_{K}\ \mid\ K\in\mathfrak{K}\right]  $, where $y_{K}$ is a
new indeterminate for each $K\in\mathfrak{K}$.

The claim of Theorem \ref{thm.chrompol.varis} is a polynomial identity in the
elements $a_{K}$ of $\mathbf{k}$. Hence, we can WLOG assume that
$\mathbf{k}=R$ and $a_{K}=y_{K}$ for each $K\in\mathfrak{K}$. Assume this.
Thus, $\mathbf{k}$ is an integral domain, and the ring $\mathbb{Z}$ is a
subring of $\mathbf{k}$.

For every $q\in\mathbb{N}$, we have%
\begin{align}
\chi_{G}\left(  q\right)   &  =\left(  \text{the number of all proper
}\left\{  1,2,\ldots,q\right\}  \text{-colorings of }G\right) \nonumber\\
&  \ \ \ \ \ \ \ \ \ \ \ \ \ \ \ \ \ \ \ \ \left(  \text{by the definition of
the chromatic polynomial }\chi_{G}\right) \nonumber\\
&  =\sum_{F\subseteq E}\left(  -1\right)  ^{\left\vert F\right\vert }\left(
\prod_{\substack{K\in\mathfrak{K};\\K\subseteq F}}a_{K}\right)
q^{\operatorname*{conn}\left(  V,F\right)  }
\label{pf.thm.chrompol.varis.short.eq1}%
\end{align}
(by Lemma \ref{lem.chrompol.weak.varis}). Define a polynomial $P\in
\mathbf{k}\left[  x\right]  $ by%
\begin{equation}
P=\sum_{F\subseteq E}\left(  -1\right)  ^{\left\vert F\right\vert }\left(
\prod_{\substack{K\in\mathfrak{K};\\K\subseteq F}}a_{K}\right)
x^{\operatorname*{conn}\left(  V,F\right)  }.
\label{pf.thm.chrompol.varis.short.P=}%
\end{equation}
Then, for every $q\in\mathbb{N}$, we have%
\[
P\left(  q\right)  =\sum_{F\subseteq E}\left(  -1\right)  ^{\left\vert
F\right\vert }\left(  \prod_{\substack{K\in\mathfrak{K};\\K\subseteq F}%
}a_{K}\right)  q^{\operatorname*{conn}\left(  V,F\right)  }=\chi_{G}\left(
q\right)  \ \ \ \ \ \ \ \ \ \ \left(  \text{by
(\ref{pf.thm.chrompol.varis.short.eq1})}\right)  .
\]
Thus, Corollary \ref{cor.poly0.P1=P2} (applied to $R=\mathbf{k}$ and $P_{1}=P$
and $P_{2}=\chi_{G}$) shows that $P=\chi_{G}$. Comparing this with
(\ref{pf.thm.chrompol.varis.short.P=}), we obtain%
\[
\chi_{G}=\sum_{F\subseteq E}\left(  -1\right)  ^{\left\vert F\right\vert
}\left(  \prod_{\substack{K\in\mathfrak{K};\\K\subseteq F}}a_{K}\right)
x^{\operatorname*{conn}\left(  V,F\right)  }.
\]
This proves Theorem \ref{thm.chrompol.varis}.
\end{proof}
\end{vershort}

\begin{verlong}
\begin{proof}
[Proof of Theorem \ref{thm.chrompol.varis}.]Let $R$ be the polynomial ring
$\mathbb{Z}\left[  y_{K}\ \mid\ K\in\mathfrak{K}\right]  $, where $y_{K}$ is a
new indeterminate for each $K\in\mathfrak{K}$. Clearly, $R$ is an integral
domain (since $R$ is a polynomial ring over $\mathbb{Z}$).

Moreover, $\mathbb{Z}$ is a subring of $R$ (since $R$ is a polynomial ring
over $\mathbb{Z}$). We therefore regard $\mathbb{Z}\left[  x\right]  $ as a
subring of $R\left[  x\right]  $.

The chromatic polynomial $\chi_{G}$ is the unique polynomial $P\in
\mathbb{Z}\left[  x\right]  $ such that every $q\in\mathbb{N}$ satisfies
$P\left(  q\right)  =\left(  \text{the number of all proper }\left\{
1,2,\ldots,q\right\}  \text{-colorings of }G\right)  $ (because this is how
$\chi_{G}$ is defined). Thus, $\chi_{G}$ is a polynomial in $\mathbb{Z}\left[
x\right]  $ and has the property that every $q\in\mathbb{N}$ satisfies%
\begin{equation}
\chi_{G}\left(  q\right)  =\left(  \text{the number of all proper }\left\{
1,2,\ldots,q\right\}  \text{-colorings of }G\right)  .
\label{pf.thm.chrompol.varis.chiG(q)=}%
\end{equation}

Lemma \ref{lem.chrompol.weak.varis} (applied to $R$ and $y_{K}$ instead of
$\mathbf{k}$ and $a_{K}$) yields that%
\begin{align*}
&  \left(  \text{the number of all proper }\left\{  1,2,\ldots,q\right\}
\text{-colorings of }G\right) \\
&  =\sum_{F\subseteq E}\left(  -1\right)  ^{\left\vert F\right\vert }\left(
\prod_{\substack{K\in\mathfrak{K};\\K\subseteq F}}y_{K}\right)
q^{\operatorname*{conn}\left(  V,F\right)  }%
\end{align*}
for every $q\in\mathbb{N}$. Thus, for every $q\in\mathbb{N}$, we have%
\begin{align}
&  \sum_{F\subseteq E}\left(  -1\right)  ^{\left\vert F\right\vert }\left(
\prod_{\substack{K\in\mathfrak{K};\\K\subseteq F}}y_{K}\right)
q^{\operatorname*{conn}\left(  V,F\right)  }\nonumber\\
&  =\left(  \text{the number of all proper }\left\{  1,2,\ldots,q\right\}
\text{-colorings of }G\right) \nonumber\\
&  =\chi_{G}\left(  q\right)  \ \ \ \ \ \ \ \ \ \ \left(  \text{by
(\ref{pf.thm.chrompol.varis.chiG(q)=})}\right)  .
\label{pf.thm.chrompol.varis.2}%
\end{align}

We have $\chi_{G}\in\mathbb{Z}\left[  x\right]  \subseteq R\left[  x\right]  $
(since $\mathbb{Z}\left[  x\right]  $ is a subring of $R\left[  x\right]  $).

On the other hand, let us define a polynomial $\widetilde{P}\in R\left[
x\right]  $ by%
\begin{equation}
\widetilde{P}=\sum_{F\subseteq E}\left(  -1\right)  ^{\left\vert F\right\vert
}\left(  \prod_{\substack{K\in\mathfrak{K};\\K\subseteq F}}y_{K}\right)
x^{\operatorname*{conn}\left(  V,F\right)  }.
\label{pf.thm.chrompol.varis.Ptilde=}%
\end{equation}
Every $q\in\mathbb{N}$ satisfies%
\begin{align}
\widetilde{P}\left(  q\right)   &  =\sum_{F\subseteq E}\left(  -1\right)
^{\left\vert F\right\vert }\left(  \prod_{\substack{K\in\mathfrak{K}%
;\\K\subseteq F}}y_{K}\right)  q^{\operatorname*{conn}\left(  V,F\right)
}\nonumber\\
&  \ \ \ \ \ \ \ \ \ \ \ \ \ \ \ \ \ \ \ \ \left(  \text{here, we have
substituted }q\text{ for }x\text{ in (\ref{pf.thm.chrompol.varis.Ptilde=}%
)}\right) \nonumber\\
&  =\chi_{G}\left(  q\right)  \ \ \ \ \ \ \ \ \ \ \left(  \text{by
(\ref{pf.thm.chrompol.varis.2})}\right)  .
\label{pf.thm.chrompol.varis.Ptildeq=}%
\end{align}
Thus, Corollary \ref{cor.poly0.P1=P2} (applied to $P_{1}=\widetilde{P}$ and
$P_{2}=\chi_{G}$) shows that $\widetilde{P}=\chi_{G}$. In other words,
$\chi_{G}=\widetilde{P}$.

Let $\mathbf{a}$ denote the family $\left(  a_{K}\right)  _{K\in\mathfrak{K}%
}\in\mathbf{k}^{\mathfrak{K}}$ of elements of $\mathbf{k}$.

Now, recall that $R$ is the polynomial ring $\mathbb{Z}\left[  y_{K}%
\ \mid\ K\in\mathfrak{K}\right]  $. Hence, $R$ satisfies the following
universal property (the well-known universal property of a polynomial ring):
For any commutative $\mathbb{Z}$-algebra $B$ and any family $\mathbf{b}%
=\left(  b_{K}\right)  _{K\in\mathfrak{K}}\in B^{\mathfrak{K}}$ of elements of
$B$, there exists a unique $\mathbb{Z}$-algebra homomorphism $\psi
:R\rightarrow B$ satisfying%
\[
\left(  \psi\left(  y_{K}\right)  =b_{K}\ \ \ \ \ \ \ \ \ \ \text{for every
}K\in\mathfrak{K}\right)  .
\]
This $\mathbb{Z}$-algebra homomorphism $\psi$ is denoted by
$\operatorname*{ev}\nolimits_{\mathbf{b}}$, and is called the \emph{evaluation
homomorphism} at the family $\mathbf{b}$.

Thus we have constructed a $\mathbb{Z}$-algebra homomorphism
$\operatorname*{ev}\nolimits_{\mathbf{b}}:R\rightarrow B$ for every
commutative $\mathbb{Z}$-algebra $B$ and every family $\mathbf{b}=\left(
b_{K}\right)  _{K\in\mathfrak{K}}\in B^{\mathfrak{K}}$ of elements of $B$.
Applying this construction to $B=\mathbf{k}$, $\mathbf{b}=\mathbf{a}$ and
$b_{K}=a_{K}$, we obtain a $\mathbb{Z}$-algebra homomorphism
$\operatorname*{ev}\nolimits_{\mathbf{a}}:R\rightarrow\mathbf{k}$. This
homomorphism, in turn, induces a $\mathbb{Z}\left[  x\right]  $-algebra
homomorphism $\operatorname*{ev}\nolimits_{\mathbf{a}}\left[  x\right]
:R\left[  x\right]  \rightarrow\mathbf{k}\left[  x\right]  $ satisfying
$\left(  \operatorname*{ev}\nolimits_{\mathbf{a}}\left[  x\right]  \right)
\left(  x\right)  =x$.\ \ \ \ \footnote{This homomorphism $\operatorname*{ev}%
\nolimits_{\mathbf{a}}\left[  x\right]  $ is explicitly given by the formula%
\[
\left(  \operatorname*{ev}\nolimits_{\mathbf{a}}\left[  x\right]  \right)
\left(  \sum_{i=0}^{\infty}r_{i}x^{i}\right)  =\sum_{i=0}^{\infty
}\operatorname*{ev}\nolimits_{\mathbf{a}}\left(  r_{i}\right)  \cdot x^{i}%
\]
for every sequence $\left(  r_{0},r_{1},r_{2},\ldots\right)  \in R^{\infty}$
of elements of $R$ which satisfies $r_{i}=0$ for all sufficiently high $i$.}.
Consider this homomorphism $\operatorname*{ev}\nolimits_{\mathbf{a}}\left[
x\right]  $.

Recall that $\mathbf{a}=\left(  a_{K}\right)  _{K\in\mathfrak{K}}$. Hence,
$\operatorname*{ev}\nolimits_{\mathbf{a}}$ is the unique $\mathbb{Z}$-algebra
homomorphism $\psi:R\rightarrow\mathbf{k}$ satisfying%
\[
\left(  \psi\left(  y_{K}\right)  =a_{K}\ \ \ \ \ \ \ \ \ \ \text{for every
}K\in\mathfrak{K}\right)
\]
(by the definition of $\operatorname*{ev}\nolimits_{\mathbf{a}}$). Thus,
$\operatorname*{ev}\nolimits_{\mathbf{a}}$ is a $\mathbb{Z}$-algebra
homomorphism and satisfies%
\begin{equation}
\operatorname*{ev}\nolimits_{\mathbf{a}}\left(  y_{K}\right)  =a_{K}%
\ \ \ \ \ \ \ \ \ \ \text{for every }K\in\mathfrak{K}.
\label{pf.thm.chrompol.varis.eva}%
\end{equation}

The construction of $\operatorname*{ev}\nolimits_{\mathbf{a}}\left[  x\right]
$ shows that
\begin{equation}
\left(  \operatorname*{ev}\nolimits_{\mathbf{a}}\left[  x\right]  \right)
\left(  u\right)  =\operatorname*{ev}\nolimits_{\mathbf{a}}\left(  u\right)
\ \ \ \ \ \ \ \ \ \ \text{for every }u\in R.
\label{pf.thm.chrompol.varis.eva(u)}%
\end{equation}
Now, for every $K\in\mathfrak{K}$, we have%
\begin{align}
\left(  \operatorname*{ev}\nolimits_{\mathbf{a}}\left[  x\right]  \right)
\left(  y_{K}\right)   &  =\operatorname*{ev}\nolimits_{\mathbf{a}}\left(
y_{K}\right)  \ \ \ \ \ \ \ \ \ \ \left(  \text{by
(\ref{pf.thm.chrompol.varis.eva(u)}), applied to }u=y_{K}\right) \nonumber\\
&  =a_{K}\ \ \ \ \ \ \ \ \ \ \left(  \text{by (\ref{pf.thm.chrompol.varis.eva}%
)}\right)  . \label{pf.thm.chrompol.varis.evax}%
\end{align}

By its definition, the homomorphism $\operatorname*{ev}\nolimits_{\mathbf{a}%
}\left[  x\right]  $ preserves $\mathbb{Z}\left[  x\right]  $ (or, rather,
sends every element of $\mathbb{Z}\left[  x\right]  \subseteq R\left[
x\right]  $ to the corresponding element of $\mathbb{Z}\left[  x\right]
\subseteq\mathbf{k}\left[  x\right]  $). In other words, $\left(
\operatorname*{ev}\nolimits_{\mathbf{a}}\left[  x\right]  \right)  \left(
Q\right)  =Q$ for every $Q\in\mathbb{Z}\left[  x\right]  $. Applying this to
$Q=\chi_{G}$, we obtain%
\[
\left(  \operatorname*{ev}\nolimits_{\mathbf{a}}\left[  x\right]  \right)
\left(  \chi_{G}\right)  =\chi_{G}%
\]
(since $\chi_{G}\in\mathbb{Z}\left[  x\right]  $). Thus,%
\begin{align*}
\chi_{G}  &  =\left(  \operatorname*{ev}\nolimits_{\mathbf{a}}\left[
x\right]  \right)  \left(  \underbrace{\chi_{G}}_{=\widetilde{P}%
=\sum_{F\subseteq E}\left(  -1\right)  ^{\left\vert F\right\vert }\left(
\prod_{\substack{K\in\mathfrak{K};\\K\subseteq F}}y_{K}\right)
x^{\operatorname*{conn}\left(  V,F\right)  }}\right) \\
&  =\left(  \operatorname*{ev}\nolimits_{\mathbf{a}}\left[  x\right]  \right)
\left(  \sum_{F\subseteq E}\left(  -1\right)  ^{\left\vert F\right\vert
}\left(  \prod_{\substack{K\in\mathfrak{K};\\K\subseteq F}}y_{K}\right)
x^{\operatorname*{conn}\left(  V,F\right)  }\right) \\
&  =\sum_{F\subseteq E}\left(  -1\right)  ^{\left\vert F\right\vert }\left(
\prod_{\substack{K\in\mathfrak{K};\\K\subseteq F}}\underbrace{\left(
\operatorname*{ev}\nolimits_{\mathbf{a}}\left[  x\right]  \right)  \left(
y_{K}\right)  }_{\substack{=a_{K}\\\text{(by (\ref{pf.thm.chrompol.varis.evax}%
))}}}\right)  \left(  \underbrace{\left(  \operatorname*{ev}%
\nolimits_{\mathbf{a}}\left[  x\right]  \right)  \left(  x\right)  }%
_{=x}\right)  ^{\operatorname*{conn}\left(  V,F\right)  }\\
&  \ \ \ \ \ \ \ \ \ \ \ \ \ \ \ \ \ \ \ \ \left(  \text{since }%
\operatorname*{ev}\nolimits_{\mathbf{a}}\left[  x\right]  \text{ is a
}\mathbb{Z}\text{-algebra homomorphism}\right) \\
&  =\sum_{F\subseteq E}\left(  -1\right)  ^{\left\vert F\right\vert }\left(
\prod_{\substack{K\in\mathfrak{K};\\K\subseteq F}}a_{K}\right)
x^{\operatorname*{conn}\left(  V,F\right)  }.
\end{align*}
This proves Theorem \ref{thm.chrompol.varis}.
\end{proof}
\end{verlong}

\begin{vershort}
Now that Theorem \ref{thm.chrompol.varis} is proven, we could derive Theorem
\ref{thm.chrompol.empty}, Corollary \ref{cor.chrompol.K-free} and Corollary
\ref{cor.chrompol.NBC} from it in the same way as we have derived Theorem
\ref{thm.chromsym.empty}, Corollary \ref{cor.chromsym.K-free} and Corollary
\ref{cor.chromsym.NBC} from Theorem \ref{thm.chromsym.varis}. We leave the
details to the reader.
\end{vershort}

\begin{verlong}
Now that Theorem \ref{thm.chrompol.varis} is proven, we can derive Theorem
\ref{thm.chrompol.empty}, Corollary \ref{cor.chrompol.K-free} and Corollary
\ref{cor.chrompol.NBC} from it in the same way as we have derived Theorem
\ref{thm.chromsym.empty}, Corollary \ref{cor.chromsym.K-free} and Corollary
\ref{cor.chromsym.NBC} from Theorem \ref{thm.chromsym.varis}. Here are the details:
\end{verlong}

\begin{verlong}
\begin{proof}
[Proof of Corollary \ref{cor.chrompol.K-free}.]We can apply Theorem
\ref{thm.chrompol.varis} to $0$ instead of $a_{K}$. As a result, we obtain%
\begin{equation}
\chi_{G}=\sum_{F\subseteq E}\left(  -1\right)  ^{\left\vert F\right\vert
}\left(  \prod_{\substack{K\in\mathfrak{K};\\K\subseteq F}}0\right)
x^{\operatorname*{conn}\left(  V,F\right)  }. \label{pf.cor.chrompol.K-free.0}%
\end{equation}
Now, if $F$ is any subset of $E$, then%
\begin{equation}
\prod_{\substack{K\in\mathfrak{K};\\K\subseteq F}}0=%
\begin{cases}
1, & \text{if }F\text{ is }\mathfrak{K}\text{-free;}\\
0, & \text{if }F\text{ is not }\mathfrak{K}\text{-free}%
\end{cases}
\label{pf.cor.chrompol.K-free.1}%
\end{equation}
\footnote{\textit{Proof of (\ref{pf.cor.chrompol.K-free.1}):} The equality
(\ref{pf.cor.chrompol.K-free.1}) has already been proven in our proof of
Corollary \ref{cor.chromsym.K-free}.}. Thus, (\ref{pf.cor.chrompol.K-free.0})
becomes%
\begin{align*}
\chi_{G}  &  =\sum_{F\subseteq E}\left(  -1\right)  ^{\left\vert F\right\vert
}\underbrace{\left(  \prod_{\substack{K\in\mathfrak{K};\\K\subseteq
F}}0\right)  }_{\substack{=%
\begin{cases}
1, & \text{if }F\text{ is }\mathfrak{K}\text{-free;}\\
0, & \text{if }F\text{ is not }\mathfrak{K}\text{-free}%
\end{cases}
\\\text{(by (\ref{pf.cor.chrompol.K-free.1}))}}}x^{\operatorname*{conn}\left(
V,F\right)  }\\
&  =\sum_{F\subseteq E}\left(  -1\right)  ^{\left\vert F\right\vert }%
\begin{cases}
1, & \text{if }F\text{ is }\mathfrak{K}\text{-free;}\\
0, & \text{if }F\text{ is not }\mathfrak{K}\text{-free}%
\end{cases}
\ \ x^{\operatorname*{conn}\left(  V,F\right)  }\\
&  =\sum_{\substack{F\subseteq E;\\F\text{ is }\mathfrak{K}\text{-free}%
}}\left(  -1\right)  ^{\left\vert F\right\vert }\underbrace{%
\begin{cases}
1, & \text{if }F\text{ is }\mathfrak{K}\text{-free;}\\
0, & \text{if }F\text{ is not }\mathfrak{K}\text{-free}%
\end{cases}
}_{\substack{=1\\\text{(since }F\text{ is }\mathfrak{K}\text{-free)}%
}}\ \ x^{\operatorname*{conn}\left(  V,F\right)  }\\
&  \ \ \ \ \ \ \ \ \ \ +\sum_{\substack{F\subseteq E;\\F\text{ is not
}\mathfrak{K}\text{-free}}}\left(  -1\right)  ^{\left\vert F\right\vert
}\underbrace{%
\begin{cases}
1, & \text{if }F\text{ is }\mathfrak{K}\text{-free;}\\
0, & \text{if }F\text{ is not }\mathfrak{K}\text{-free}%
\end{cases}
}_{\substack{=0\\\text{(since }F\text{ is not }\mathfrak{K}\text{-free)}%
}}\ \ x^{\operatorname*{conn}\left(  V,F\right)  }\\
&  =\sum_{\substack{F\subseteq E;\\F\text{ is }\mathfrak{K}\text{-free}%
}}\left(  -1\right)  ^{\left\vert F\right\vert }x^{\operatorname*{conn}\left(
V,F\right)  }+\underbrace{\sum_{\substack{F\subseteq E;\\F\text{ is not
}\mathfrak{K}\text{-free}}}\left(  -1\right)  ^{\left\vert F\right\vert
}0x^{\operatorname*{conn}\left(  V,F\right)  }}_{=0}\\
&  =\sum_{\substack{F\subseteq E;\\F\text{ is }\mathfrak{K}\text{-free}%
}}\left(  -1\right)  ^{\left\vert F\right\vert }x^{\operatorname*{conn}\left(
V,F\right)  }.
\end{align*}
This proves Corollary \ref{cor.chrompol.K-free}.
\end{proof}
\end{verlong}

\begin{verlong}
\begin{proof}
[Proof of Corollary \ref{cor.chrompol.NBC}.]Let $\mathfrak{K}$ be the set of
all broken circuits of $G$. Thus, the elements of $\mathfrak{K}$ are the
broken circuits of $G$.

Now, for every subset $F$ of $E$, we have the following equivalence of
statements:%
\begin{align*}
&  \ \left(  F\text{ is }\mathfrak{K}\text{-free}\right) \\
&  \Longleftrightarrow\ \left(  F\text{ contains no }K\in\mathfrak{K}\text{ as
a subset}\right) \\
&  \ \ \ \ \ \ \ \ \ \ \left(
\begin{array}
[c]{c}%
\text{because }F\text{ is }\mathfrak{K}\text{-free if and only if }F\text{
contains no }K\in\mathfrak{K}\text{ as a subset}\\
\text{(by the definition of \textquotedblleft}\mathfrak{K}%
\text{-free\textquotedblright)}%
\end{array}
\right) \\
&  \Longleftrightarrow\ \left(  F\text{ contains no element of }%
\mathfrak{K}\text{ as a subset}\right) \\
&  \Longleftrightarrow\ \left(  F\text{ contains no broken circuit of }G\text{
as a subset}\right)
\end{align*}
(since the elements of $\mathfrak{K}$ are the broken circuits of $G$). Hence,
$\sum_{\substack{F\subseteq E;\\F\text{ is }\mathfrak{K}\text{-free}}%
}=\sum_{\substack{F\subseteq E;\\F\text{ contains no broken}\\\text{circuit of
}G\text{ as a subset}}}$ (an equality between summation signs). Now, Corollary
\ref{cor.chrompol.K-free} yields%
\[
\chi_{G}=\underbrace{\sum_{\substack{F\subseteq E;\\F\text{ is }%
\mathfrak{K}\text{-free}}}}_{=\sum_{\substack{F\subseteq E;\\F\text{ contains
no broken}\\\text{circuit of }G\text{ as a subset}}}}\left(  -1\right)
^{\left\vert F\right\vert }x^{\operatorname*{conn}\left(  V,F\right)  }%
=\sum_{\substack{F\subseteq E;\\F\text{ contains no broken}\\\text{circuit of
}G\text{ as a subset}}}\left(  -1\right)  ^{\left\vert F\right\vert
}x^{\operatorname*{conn}\left(  V,F\right)  }.
\]
This proves Corollary \ref{cor.chrompol.NBC}.
\end{proof}
\end{verlong}

\begin{verlong}
\begin{proof}
[Proof of Theorem \ref{thm.chrompol.empty}.]Let $X$ be the totally ordered set
$\left\{  1\right\}  $ (equipped with the only possible order on this set).
Let $\ell:E\rightarrow X$ be the function sending each $e\in E$ to $1\in X$.
Let $\mathfrak{K}$ be the empty set. Clearly, $\mathfrak{K}$ is a set of
broken circuits of $G$. Theorem \ref{thm.chrompol.varis} (applied to $0$
instead of $a_{K}$) yields%
\begin{align*}
\chi_{G}  &  =\sum_{F\subseteq E}\left(  -1\right)  ^{\left\vert F\right\vert
}\underbrace{\left(  \prod_{\substack{K\in\mathfrak{K};\\K\subseteq
F}}0\right)  }_{\substack{=\left(  \text{empty product}\right)  \\\text{(since
}\mathfrak{K}\text{ is the empty set)}}}x^{\operatorname*{conn}\left(
V,F\right)  }\\
&  =\sum_{F\subseteq E}\left(  -1\right)  ^{\left\vert F\right\vert
}\underbrace{\left(  \text{empty product}\right)  }_{=1}%
x^{\operatorname*{conn}\left(  V,F\right)  }=\sum_{F\subseteq E}\left(
-1\right)  ^{\left\vert F\right\vert }x^{\operatorname*{conn}\left(
V,F\right)  }.
\end{align*}
This proves Theorem \ref{thm.chrompol.empty}.
\end{proof}
\end{verlong}

\subsection{Special case: Whitney's Broken-Circuit Theorem}

Corollary \ref{cor.chrompol.NBC} is commonly stated in the following
simplified (if less general) form:

\Needspace{4cm}

\begin{corollary}
\label{cor.chrompol.NBCfor}Let $G=\left(  V,E\right)  $ be a finite graph. Let
$X$ be a totally ordered set. Let $\ell:E\rightarrow X$ be an injective
labeling function. Then,%
\[
\chi_{G}=\sum_{\substack{F\subseteq E;\\F\text{ contains no broken}%
\\\text{circuit of }G\text{ as a subset}}}\left(  -1\right)  ^{\left\vert
F\right\vert }x^{\left\vert V\right\vert -\left\vert F\right\vert }.
\]

\end{corollary}

Corollary \ref{cor.chrompol.NBCfor} is known as \emph{Whitney's Broken-Circuit
theorem} (see, e.g., \cite{BlaSag86}). In his original 1932 paper
\cite[\S 7]{Whitne32}, Whitney stated its claim as \textquotedblleft the
$x^{\left\vert V\right\vert -i}$-coefficient of $\chi_{G}$ is $\left(
-1\right)  ^{i}$ times the number of $i$-element subsets of $E$ that contain
no broken circuit as a subset\textquotedblright, which is easily seen to be
equivalent to our formulation.

Notice that $\ell$ is required to be injective in Corollary
\ref{cor.chrompol.NBCfor}; the purpose of this requirement is to ensure that
every circuit of $G$ has a unique edge $e$ with maximum $\ell\left(  e\right)
$, and thus induces a broken circuit of $G$. The proof of Corollary
\ref{cor.chrompol.NBCfor} relies on the following standard result:

\begin{lemma}
\label{lem.conn.forest}Let $\left(  V,F\right)  $ be a finite graph. Assume
that $\left(  V,F\right)  $ has no circuits. Then, $\operatorname*{conn}%
\left(  V,F\right)  =\left\vert V\right\vert -\left\vert F\right\vert $.
\end{lemma}

(A graph that has no circuits is commonly known as a \emph{forest}.)

Lemma \ref{lem.conn.forest} is both extremely elementary and well-known; for
example, it appears in \cite[Proposition 10.6]{Bona17}, in \cite[\S I.2,
Corollary 6]{Bollobas} and in \cite[Theorem 6.3.15 \textbf{(e)}]{21f-lec6}.
Let us now see how it entails Corollary \ref{cor.chrompol.NBCfor}:

\begin{vershort}
\begin{proof}
[Proof of Corollary \ref{cor.chrompol.NBCfor}.]Corollary
\ref{cor.chrompol.NBCfor} follows from Corollary \ref{cor.chrompol.NBC}.
Indeed, the injectivity of $\ell$ shows that every circuit of $G$ has a unique
edge $e$ with maximum $\ell\left(  e\right)  $, and thus contains a broken
circuit of $G$ as a subset. Therefore, if a subset $F$ of $E$ contains no
broken circuit of $G$ as a subset, then $F$ contains no circuit of $G$ either,
and therefore the graph $\left(  V,F\right)  $ has no circuits; but this
entails that $\operatorname*{conn}\left(  V,F\right)  =\left\vert V\right\vert
-\left\vert F\right\vert $ (by Lemma \ref{lem.conn.forest}). Hence, Corollary
\ref{cor.chrompol.NBC} immediately yields Corollary \ref{cor.chrompol.NBCfor}.
\end{proof}
\end{vershort}

\begin{verlong}
\begin{proof}
[Proof of Corollary \ref{cor.chrompol.NBCfor}.]We first claim that if $F$ is a
subset of $E$ such that $F$ contains no broken circuit of $G$ as a subset,
then%
\begin{equation}
\operatorname*{conn}\left(  V,F\right)  =\left\vert V\right\vert -\left\vert
F\right\vert . \label{pf.cor.chrompol.NBCfor.1}%
\end{equation}

\textit{Proof of (\ref{pf.cor.chrompol.NBCfor.1}):} Let $F$ be a subset of $E$
such that $F$ contains no broken circuit of $G$ as a subset. We shall now show
that the graph $\left(  V,F\right)  $ has no circuits.

Indeed, assume the contrary (for the sake of contradiction). Thus, the graph
$\left(  V,F\right)  $ has a circuit. In other words, there exists a circuit
$D$ of the graph $\left(  V,F\right)  $. Consider this $D$.

The set $D$ is a circuit of $\left(  V,F\right)  $. In other words, the set
$D$ has the form $\left\{  \left\{  v_{1},v_{2}\right\}  ,\left\{  v_{2}%
,v_{3}\right\}  ,\ldots,\left\{  v_{m},v_{m+1}\right\}  \right\}  $, where
$\left(  v_{1},v_{2},\ldots,v_{m+1}\right)  $ is a cycle of $\left(
V,F\right)  $ (by the definition of a \textquotedblleft
circuit\textquotedblright). Consider this cycle $\left(  v_{1},v_{2}%
,\ldots,v_{m+1}\right)  $. We thus have
\begin{equation}
D=\left\{  \left\{  v_{1},v_{2}\right\}  ,\left\{  v_{2},v_{3}\right\}
,\ldots,\left\{  v_{m},v_{m+1}\right\}  \right\}  .
\label{pf.cor.chrompol.NBCfor.5}%
\end{equation}

The list $\left(  v_{1},v_{2},\ldots,v_{m+1}\right)  $ is a cycle of $\left(
V,F\right)  $. According to the definition of a \textquotedblleft
cycle\textquotedblright, this means that this list is a list of elements of
$V$ satisfying the following four properties:

\begin{itemize}
\item We have $m>2$.

\item We have $v_{m+1}=v_{1}$.

\item The vertices $v_{1},v_{2},\ldots,v_{m}$ are pairwise distinct.

\item We have $\left\{  v_{i},v_{i+1}\right\}  \in F$ for every $i\in\left\{
1,2,\ldots,m\right\}  $.
\end{itemize}

Thus, $\left(  v_{1},v_{2},\ldots,v_{m+1}\right)  $ is a list of elements of
$V$ satisfying the four properties that we have just mentioned. Notice that%
\begin{align*}
D  &  =\left\{  \left\{  v_{1},v_{2}\right\}  ,\left\{  v_{2},v_{3}\right\}
,\ldots,\left\{  v_{m},v_{m+1}\right\}  \right\} \\
&  =\left\{  \left\{  v_{i},v_{i+1}\right\}  \ \mid\ i\in\left\{
1,2,\ldots,m\right\}  \right\}  \subseteq F
\end{align*}
(since $\left\{  v_{i},v_{i+1}\right\}  \in F$ for every $i\in\left\{
1,2,\ldots,m\right\}  $).

Now, we have $\left\{  v_{i},v_{i+1}\right\}  \in F\subseteq E$ for every
$i\in\left\{  1,2,\ldots,m\right\}  $. Hence, $\left(  v_{1},v_{2}%
,\ldots,v_{m+1}\right)  $ is a list of elements of $V$ satisfying the
following four properties:

\begin{itemize}
\item We have $m>2$.

\item We have $v_{m+1}=v_{1}$.

\item The vertices $v_{1},v_{2},\ldots,v_{m}$ are pairwise distinct.

\item We have $\left\{  v_{i},v_{i+1}\right\}  \in E$ for every $i\in\left\{
1,2,\ldots,m\right\}  $.
\end{itemize}

According to the definition of a \textquotedblleft cycle\textquotedblright,
this means that the list $\left(  v_{1},v_{2},\ldots,v_{m+1}\right)  $ is a
cycle of $\left(  V,E\right)  $.

Thus, we conclude that the list $\left(  v_{1},v_{2},\ldots,v_{m+1}\right)  $
is a cycle of $\left(  V,E\right)  $. In other words, the list $\left(
v_{1},v_{2},\ldots,v_{m+1}\right)  $ is a cycle of $G$ (since $G=\left(
V,E\right)  $). Thus, the set $\left\{  \left\{  v_{1},v_{2}\right\}
,\left\{  v_{2},v_{3}\right\}  ,\ldots,\left\{  v_{m},v_{m+1}\right\}
\right\}  $ is a circuit of $G$ (by the definition of a \textquotedblleft
circuit\textquotedblright). In view of (\ref{pf.cor.chrompol.NBCfor.5}), this
rewrites as follows: The set $D$ is a circuit of $G$.

No two distinct edges in $D$ have the same label\footnote{\textit{Proof.}
Assume the contrary. Thus, two distinct edges in $D$ have the same label. In
other words, there exist two distinct edges $e$ and $e^{\prime}$ in $D$ such
that $e$ and $e^{\prime}$ have the same label. Consider these $e$ and
$e^{\prime}$.
\par
The edges $e$ and $e^{\prime}$ have the same label. In other words, the label
of $e$ equals the label of $e^{\prime}$. In other words, $\ell\left(
e\right)  $ equals $\ell\left(  e^{\prime}\right)  $ (since the label of $e$
is $\ell\left(  e\right)  $ (by the definition of \textquotedblleft
label\textquotedblright), whereas the label of $e^{\prime}$ is $\ell\left(
e^{\prime}\right)  $ (by the definition of \textquotedblleft
label\textquotedblright)). In other words, $\ell\left(  e\right)  =\ell\left(
e^{\prime}\right)  $. Since the map $\ell$ is injective, this shows that
$e=e^{\prime}$. This contradicts the assumption that $e$ and $e^{\prime}$ are
distinct. This contradiction proves that our assumption was wrong. Qed.}.

From $m>2>1$, we obtain%
\[
\left\{  v_{1},v_{2}\right\}  \in\left\{  \left\{  v_{1},v_{2}\right\}
,\left\{  v_{2},v_{3}\right\}  ,\ldots,\left\{  v_{m},v_{m+1}\right\}
\right\}  =D
\]
(by (\ref{pf.cor.chrompol.NBCfor.5})). Hence, the set $D$ is nonempty (since
it contains $\left\{  v_{1},v_{2}\right\}  $). Thus, $D$ is a nonempty finite
set. Hence, there exists an edge in $D$ having maximum label. Let $f$ be this
edge. Clearly, $D\setminus\left\{  f\right\}  \subseteq D\subseteq F\subseteq
E$. Thus, $D\setminus\left\{  f\right\}  $ is a subset of $E$.

The edge $f$ is an edge in $D$ having maximum label. Since no other edge in
$D$ has the same label as $f$ (because no two distinct edges in $D$ have the
same label), this shows that the edge $f$ is the \textbf{unique} edge in $D$
having maximum label. Therefore, $D\setminus\left\{  f\right\}  $ is a subset
of $E$ having the form $C\setminus\left\{  e\right\}  $, where $C$ is a
circuit of $G$, and where $e$ is the unique edge in $C$ having maximum label
(among the edges in $C$)\ \ \ \ \footnote{Namely, $D\setminus\left\{
f\right\}  $ has this form for $C=D$ and $e=f$.}. In other words,
$D\setminus\left\{  f\right\}  $ is a broken circuit of $G$ (since
$D\setminus\left\{  f\right\}  $ is a broken circuit of $G$ if and only if
$D\setminus\left\{  f\right\}  $ is a subset of $E$ having the form
$C\setminus\left\{  e\right\}  $, where $C$ is a circuit of $G$, and where $e$
is the unique edge in $C$ having maximum label (among the edges in
$C$)\ \ \ \ \footnote{by the definition of a \textquotedblleft broken
circuit\textquotedblright}). This broken circuit $D\setminus\left\{
f\right\}  $ satisfies $D\setminus\left\{  f\right\}  \subseteq F$. Thus,
there exists a broken circuit $B$ of $G$ such that $B\subseteq F$ (namely,
$B=D\setminus\left\{  f\right\}  $).

But the set $F$ contains no broken circuit of $G$ as a subset. In other words,
there exists no broken circuit $B$ of $G$ such that $B\subseteq F$. This
contradicts the fact that there exists a broken circuit $B$ of $G$ such that
$B\subseteq F$. This contradiction proves that our assumption was wrong.
Hence, the graph $\left(  V,F\right)  $ has no circuits. Lemma
\ref{lem.conn.forest} thus shows that $\operatorname*{conn}\left(  V,F\right)
=\left\vert V\right\vert -\left\vert F\right\vert $. This proves
(\ref{pf.cor.chrompol.NBCfor.1}).

Now, Corollary \ref{cor.chrompol.NBC} yields%
\begin{align*}
\chi_{G}  &  =\sum_{\substack{F\subseteq E;\\F\text{ contains no
broken}\\\text{circuit of }G\text{ as a subset}}}\left(  -1\right)
^{\left\vert F\right\vert }\underbrace{x^{\operatorname*{conn}\left(
V,F\right)  }}_{\substack{=x^{\left\vert V\right\vert -\left\vert F\right\vert
}\\\text{(since }\operatorname*{conn}\left(  V,F\right)  =\left\vert
V\right\vert -\left\vert F\right\vert \\\text{(by
(\ref{pf.cor.chrompol.NBCfor.1})))}}}\\
&  =\sum_{\substack{F\subseteq E;\\F\text{ contains no broken}\\\text{circuit
of }G\text{ as a subset}}}\left(  -1\right)  ^{\left\vert F\right\vert
}x^{\left\vert V\right\vert -\left\vert F\right\vert }.
\end{align*}
This proves Corollary \ref{cor.chrompol.NBCfor}.
\end{proof}
\end{verlong}

\section{\label{sec.trans}Application: Transitive directed graphs}

We shall now see an application of Corollary \ref{cor.chrompol.K-free} to
graphs which are obtained from certain directed graphs by \textquotedblleft
forgetting the directions of the edges\textquotedblright. Let us first
introduce the notations involved:

\begin{definition}
\label{def.digraph}\textbf{(a)} A \emph{digraph} means a pair $\left(
V,A\right)  $, where $V$ is a set, and where $A$ is a subset of $V^{2}=V\times
V$. Digraphs are also called \emph{directed graphs}. A digraph $\left(
V,A\right)  $ is said to be \emph{finite} if the set $V$ is finite. If
$D=\left(  V,A\right)  $ is a digraph, then the elements of $V$ are called the
\emph{vertices} of the digraph $D$, while the elements of $A$ are called the
\emph{arcs} (or the \emph{directed edges}) of the digraph $D$. If $a=\left(
v,w\right)  $ is an arc of a digraph $D$, then $v$ is called the \emph{source}
of $a$, whereas $w$ is called the \emph{target} of $a$.

\textbf{(b)} A digraph $\left(  V,A\right)  $ is said to be \emph{loopless} if
every $v\in V$ satisfies $\left(  v,v\right)  \notin A$. (In other words, a
digraph is loopless if and only if it has no arc whose source and target are identical.)

\textbf{(c)} A digraph $\left(  V,A\right)  $ is said to be \emph{transitive}
if it has the following property: For any $u\in V$, $v\in V$ and $w\in V$
satisfying $\left(  u,v\right)  \in A$ and $\left(  v,w\right)  \in A$, we
have $\left(  u,w\right)  \in A$.

\textbf{(d)} A digraph $\left(  V,A\right)  $ is said to be $2$%
\emph{-step-free} if there exist no three elements $u$, $v$ and $w$ of $V$
satisfying $\left(  u,v\right)  \in A$ and $\left(  v,w\right)  \in A$.

\textbf{(e)} Let $D=\left(  V,A\right)  $ be a loopless digraph. Define a map
$\operatorname*{set}:A\rightarrow\dbinom{V}{2}$ by setting%
\[
\left(  \operatorname*{set}\left(  v,w\right)  =\left\{  v,w\right\}
\ \ \ \ \ \ \ \ \ \ \text{for every }\left(  v,w\right)  \in A\right)  .
\]
(It is easy to see that $\operatorname*{set}$ is well-defined, because
$\left(  V,A\right)  $ is loopless.) The graph $\left(  V,\operatorname*{set}%
A\right)  $ will be denoted by $\underline{D}$. (Here, $\operatorname*{set}A$
means the subset $\left\{  \operatorname*{set}a\ \mid\ a\in A\right\}  $ of
$\dbinom{V}{2}$.)
\end{definition}

\begin{example}
\textbf{(a)} The digraph $D=\left(  V,A\right)  $ with $V=\left\{
1,2,3\right\}  $ and $A=\left\{  \left(  1,2\right)  ,\ \left(  2,1\right)
,\ \left(  2,3\right)  ,\ \left(  3,3\right)  \right\}  $ is not loopless
(since the vertex $v=3$ does not satisfy $\left(  v,v\right)  \notin A$).

\textbf{(b)} The digraph $D=\left(  V,A\right)  $ with $V=\left\{
1,2,3\right\}  $ and $A=\left\{  \left(  1,2\right)  ,\ \left(  2,1\right)
,\ \left(  2,3\right)  \right\}  $ is loopless. The corresponding (undirected)
graph $\underline{D}$ is $\underline{D}=\left(  V,\operatorname*{set}A\right)
$ with $\operatorname*{set}A=\left\{  \left\{  1,2\right\}  ,\ \left\{
2,3\right\}  \right\}  $. (Note that the two distinct arcs $\left(
1,2\right)  $ and $\left(  2,1\right)  $ of $D$ yield the same edge $\left\{
1,2\right\}  $ of $\underline{D}$.) Note that this digraph $D$ is not
transitive, because the three vertices $u=1$, $v=2$ and $w=1$ satisfy $\left(
u,v\right)  \in A$ and $\left(  v,w\right)  \in A$ but don't satisfy $\left(
u,w\right)  \in A$.

\textbf{(c)} The digraph $D=\left(  V,A\right)  $ with $V=\left\{
1,2,3,4\right\}  $ and $A=\left\{  \left(  1,2\right)  ,\ \left(  2,3\right)
,\ \left(  1,3\right)  ,\ \left(  3,4\right)  \right\}  $ is not transitive,
since the three vertices $u=2$, $v=3$ and $w=4$ satisfy $\left(  u,v\right)
\in A$ and $\left(  v,w\right)  \in A$ but don't satisfy $\left(  u,w\right)
\in A$.

\textbf{(d)} The digraph $D=\left(  V,A\right)  $ with $V=\left\{
1,2,3,4\right\}  $ and $A=\left\{  \left(  1,2\right)  ,\ \left(  2,3\right)
,\ \left(  1,3\right)  ,\ \left(  4,2\right)  ,\ \left(  4,3\right)  \right\}
$ is loopless and transitive. It is not $2$-step-free, since the three
elements $u=1$, $v=2$ and $w=3$ satisfy $\left(  u,v\right)  \in A$ and
$\left(  v,w\right)  \in A$.

\textbf{(e)} The digraph $D=\left(  V,A\right)  $ with $V=\left\{
1,2,3,4\right\}  $ and $A=\left\{  \left(  1,3\right)  ,\ \left(  2,3\right)
,\ \left(  1,4\right)  ,\ \left(  2,4\right)  \right\}  $ is loopless,
transitive and $2$-step-free. (Actually, any $2$-step-free digraph is
transitive, for vacuous reasons.)
\end{example}

\begin{remark}
A transitive loopless digraph cannot have any (directed) cycles. We omit the
easy proof of this fact, as we will not use it in what follows, but it
illuminates some of the arguments below.
\end{remark}

\begin{remark}
A transitive loopless digraph is more or less the same as a poset (i.e.,
partially ordered set). Indeed:

\begin{itemize}
\item If $\left(  V,A\right)  $ is a transitive loopless digraph, then we can
equip the set $V$ with a (strict) partial order $<$ defined by
\[
\left(  u<v\right)  \ \Longleftrightarrow\ \left(  \left(  u,v\right)  \in
A\right)  ,
\]
which turns $V$ into a poset.

\item Conversely, if $V$ is a poset, then we obtain a transitive loopless
digraph $\left(  V,A\right)  $ by setting $A:=\left\{  \left(  u,v\right)  \in
V^{2}\ \mid\ u<v\right\}  $.
\end{itemize}

We find the language of digraphs to be more convenient, but the reader should
be aware of the possibility of restating everything in terms of posets.
\end{remark}

We can now state our application of Corollary \ref{cor.chrompol.K-free},
answering a question suggested by Alexander Postnikov:

\begin{proposition}
\label{prop.digraph.2pf-chrom}Let $D=\left(  V,A\right)  $ be a finite
transitive loopless digraph. Then,%
\[
\chi_{\underline{D}}=\sum_{\substack{F\subseteq A;\\\text{the digraph }\left(
V,F\right)  \text{ is }2\text{-step-free}}}\left(  -1\right)  ^{\left\vert
F\right\vert }x^{\operatorname*{conn}\left(  V,\operatorname*{set}F\right)
}.
\]
(Here, $\operatorname*{set}F$ means the subset $\left\{  \operatorname*{set}%
f\ \mid\ f\in F\right\}  $ of $\dbinom{V}{2}$.)
\end{proposition}

Note that the graph $\left(  V,\operatorname*{set}F\right)  $ in Proposition
\ref{prop.digraph.2pf-chrom} can also be rewritten as $\underline{\left(
V,F\right)  }$.

\begin{vershort}
\begin{proof}
[Proof of Proposition \ref{prop.digraph.2pf-chrom}.]Let $E=\operatorname*{set}%
A$. Then, the definition of $\underline{D}$ yields $\underline{D}=\left(
V,\underbrace{\operatorname*{set}A}_{=E}\right)  =\left(  V,E\right)  $.

The map $\operatorname*{set}:A\rightarrow\dbinom{V}{2}$ (which sends every arc
$\left(  v,w\right)  \in A$ to $\left\{  v,w\right\}  \in\dbinom{V}{2}$)
restricts to a surjection $A\rightarrow E$ (since $E=\operatorname*{set}A$).
Let us denote this surjection by $\pi$. Thus, $\pi$ is a map from $A$ to $E$
sending each arc $\left(  v,w\right)  \in A$ to $\left\{  v,w\right\}  \in E$.
We shall soon see that $\pi$ is a bijection.

We define a partial order on the set $V$ as follows: For $i\in V$ and $j\in
V$, we set $i<j$ if and only if $\left(  i,j\right)  \in A$ (that is, if and
only if there is an arc from $i$ to $j$ in $D$). This is a well-defined strict
partial order\footnote{Indeed, the relation $<$ that we have just defined is
transitive (since the digraph $\left(  V,A\right)  $ is transitive) and
irreflexive (since the digraph $\left(  V,A\right)  $ is loopless). But any
such relation is a strict partial order.}. Thus, $V$ becomes a poset. For
every $i\in V$ and $j\in V$ satisfying $i\leq j$, we let $\left[  i,j\right]
$ denote the interval $\left\{  k\in V\ \mid\ i\leq k\leq j\right\}  $ of the
poset $V$.

There exist no $i,j\in V$ such that both $\left(  i,j\right)  $ and $\left(
j,i\right)  $ belong to $A$ (because if such $i$ and $j$ would exist, then
they would satisfy $i<j$ and $j<i$, but this would contradict the fact that
$V$ is a poset). Hence, the projection $\pi:A\rightarrow E$ is injective, and
thus bijective (since we already know that $\pi$ is surjective). Hence, its
inverse map $\pi^{-1}:E\rightarrow A$ is well-defined. For every subset $F$ of
$E$, we have%
\begin{align}
F  &  =\pi\left(  \pi^{-1}\left(  F\right)  \right)
\ \ \ \ \ \ \ \ \ \ \left(  \text{since }\pi\text{ is bijective}\right)
\nonumber\\
&  =\operatorname*{set}\left(  \pi^{-1}\left(  F\right)  \right)
\label{pf.prop.digraph.2pf-chrom.short.F=set}%
\end{align}
(since $\pi$ is a restriction of the map $\operatorname*{set}$).

For any $\left(  u,v\right)  \in A$ and any subset $F$ of $E$, we have the
following logical equivalence:%
\begin{equation}
\left(  \left\{  u,v\right\}  \in F\right)  \ \Longleftrightarrow\ \left(
\left(  u,v\right)  \in\pi^{-1}\left(  F\right)  \right)
\label{pf.prop.digraph.2pf-chrom.short.equiv}%
\end{equation}
\footnote{\textit{Proof of (\ref{pf.prop.digraph.2pf-chrom.short.equiv}):} Let
$\left(  u,v\right)  \in A$, and let $F$ be a subset of $E$. We need to prove
the equivalence (\ref{pf.prop.digraph.2pf-chrom.short.equiv}).
\par
From $\left(  u,v\right)  \in A$, we see that $\pi\left(  u,v\right)  $ is
well-defined. The definition of $\pi$ shows that $\pi\left(  u,v\right)
=\left\{  u,v\right\}  $. Hence, we have the following chain of equivalences:%
\[
\left(  \underbrace{\left\{  u,v\right\}  }_{=\pi\left(  u,v\right)  }\in
F\right)  \ \Longleftrightarrow\ \left(  \pi\left(  u,v\right)  \in F\right)
\ \Longleftrightarrow\ \left(  \left(  u,v\right)  \in\pi^{-1}\left(
F\right)  \right)  .
\]
This proves (\ref{pf.prop.digraph.2pf-chrom.short.equiv}).}.

Define a function $\ell^{\prime}:A\rightarrow\mathbb{N}$ by%
\[
\ell^{\prime}\left(  i,j\right)  =\left\vert \left[  i,j\right]  \right\vert
\ \ \ \ \ \ \ \ \ \ \text{for all }\left(  i,j\right)  \in A.
\]
Define a labeling function $\ell:E\rightarrow\mathbb{N}$ by $\ell=\ell
^{\prime}\circ\pi^{-1}$. Thus, $\ell\circ\pi=\ell^{\prime}$. Therefore,%
\begin{equation}
\ell\left(  \underbrace{\left\{  i,j\right\}  }_{=\pi\left(  i,j\right)
}\right)  =\underbrace{\left(  \ell\circ\pi\right)  }_{=\ell^{\prime}}\left(
i,j\right)  =\ell^{\prime}\left(  i,j\right)  =\left\vert \left[  i,j\right]
\right\vert \label{pf.prop.digraph.2pf-chrom.short.l}%
\end{equation}
for all $\left(  i,j\right)  \in A$.

Let $\mathfrak{K}$ be the set
\[
\left\{  \left\{  \left\{  i,k\right\}  ,\left\{  k,j\right\}  \right\}
\ \mid\ \left(  i,k\right)  \in A\text{ and }\left(  k,j\right)  \in
A\right\}  .
\]
Each $K\in\mathfrak{K}$ is a broken circuit of $\underline{D}$%
\ \ \ \ \footnote{\textit{Proof.} Let $K\in\mathfrak{K}$. Then, $K=\left\{
\left\{  i,k\right\}  ,\left\{  k,j\right\}  \right\}  $ for some $\left(
i,k\right)  \in A$ and $\left(  k,j\right)  \in A$ (by the definition of
$\mathfrak{K}$). Consider these $\left(  i,k\right)  $ and $\left(
k,j\right)  $. Since $\left(  V,A\right)  $ is transitive, we have $\left(
i,j\right)  \in A$. Thus, $\left\{  i,k\right\}  $, $\left\{  k,j\right\}  $
and $\left\{  i,j\right\}  $ are edges of $\underline{D}$. These edges form a
circuit of $\underline{D}$. In particular, $i$, $j$ and $k$ are pairwise
distinct (since $D$ is loopless).
\par
Applications of (\ref{pf.prop.digraph.2pf-chrom.short.l}) yield $\ell\left(
\left\{  i,j\right\}  \right)  =\left\vert \left[  i,j\right]  \right\vert $,
$\ell\left(  \left\{  i,k\right\}  \right)  =\left\vert \left[  i,k\right]
\right\vert $ and $\ell\left(  \left\{  k,j\right\}  \right)  =\left\vert
\left[  k,j\right]  \right\vert $.
\par
But we have $i<k$ (since $\left(  i,k\right)  \in A$) and $k<j$ (since
$\left(  k,j\right)  \in A$). Hence, $\left[  i,k\right]  $ is a proper subset
of $\left[  i,j\right]  $. (It is proper because it does not contain $j$,
whereas $\left[  i,j\right]  $ does.) Hence, $\left\vert \left[  i,k\right]
\right\vert <\left\vert \left[  i,j\right]  \right\vert $. Thus, $\ell\left(
\left\{  i,j\right\}  \right)  =\left\vert \left[  i,j\right]  \right\vert
>\left\vert \left[  i,k\right]  \right\vert =\ell\left(  \left\{  i,k\right\}
\right)  $. Similarly, $\ell\left(  \left\{  i,j\right\}  \right)
>\ell\left(  \left\{  k,j\right\}  \right)  $. The last two inequalities show
that $\left\{  i,j\right\}  $ is the unique edge of the circuit $\left\{
\left\{  i,k\right\}  ,\left\{  k,j\right\}  ,\left\{  i,j\right\}  \right\}
$ having maximum label. Hence, $\left\{  \left\{  i,k\right\}  ,\left\{
k,j\right\}  ,\left\{  i,j\right\}  \right\}  \setminus\left\{  \left\{
i,j\right\}  \right\}  $ is a broken circuit of $\underline{D}$. Since%
\begin{align*}
\left\{  \left\{  i,k\right\}  ,\left\{  k,j\right\}  ,\left\{  i,j\right\}
\right\}  \setminus\left\{  \left\{  i,j\right\}  \right\}   &  =\left\{
\left\{  i,k\right\}  ,\left\{  k,j\right\}  \right\}
\ \ \ \ \ \ \ \ \ \ \left(  \text{since }i\text{, }j\text{ and }k\text{ are
pairwise distinct}\right) \\
&  =K,
\end{align*}
this shows that $K$ is a broken circuit of $\underline{D}$, qed.}. Thus,
$\mathfrak{K}$ is a set of broken circuits of $\underline{D}$.

A subset $F$ of $E$ is $\mathfrak{K}$-free if and only if the digraph $\left(
V,\pi^{-1}\left(  F\right)  \right)  $ is $2$%
-step-free\footnote{\textit{Proof.} Let $F$ be a subset of $E$. Then, we have
the following equivalence of statements:%
\begin{align*}
&  \ \left(  F\text{ is }\mathfrak{K}\text{-free}\right) \\
&  \Longleftrightarrow\ \left(  \left\{  \left\{  i,k\right\}  ,\left\{
k,j\right\}  \right\}  \not \subseteq F\text{ whenever }\left(  i,k\right)
\in A\text{ and }\left(  k,j\right)  \in A\right) \\
&  \ \ \ \ \ \ \ \ \ \ \left(  \text{by the definition of }\mathfrak{K}\right)
\\
&  \Longleftrightarrow\ \left(  \text{no }\left(  i,k\right)  \in A\text{ and
}\left(  k,j\right)  \in A\text{ satisfy }\left\{  \left\{  i,k\right\}
,\left\{  k,j\right\}  \right\}  \subseteq F\right) \\
&  \Longleftrightarrow\ \left(  \text{no }\left(  i,k\right)  \in A\text{ and
}\left(  k,j\right)  \in A\text{ satisfy }\left\{  i,k\right\}  \in F\text{
and }\left\{  k,j\right\}  \in F\right) \\
&  \Longleftrightarrow\ \left(  \text{no }\left(  i,k\right)  \in A\text{ and
}\left(  k,j\right)  \in A\text{ satisfy }\left(  i,k\right)  \in\pi
^{-1}\left(  F\right)  \text{ and }\left\{  k,j\right\}  \in F\right) \\
&  \ \ \ \ \ \ \ \ \ \ \left(
\begin{array}
[c]{c}%
\text{because for }\left(  i,k\right)  \in A\text{, we have }\left\{
i,k\right\}  \in F\text{ if and only if }\left(  i,k\right)  \in\pi
^{-1}\left(  F\right) \\
\text{(by (\ref{pf.prop.digraph.2pf-chrom.short.equiv}), applied to }u=i\text{
and }v=k\text{)}%
\end{array}
\right) \\
&  \Longleftrightarrow\ \left(  \text{no }\left(  i,k\right)  \in A\text{ and
}\left(  k,j\right)  \in A\text{ satisfy }\left(  i,k\right)  \in\pi
^{-1}\left(  F\right)  \text{ and }\left(  k,j\right)  \in\pi^{-1}\left(
F\right)  \right) \\
&  \ \ \ \ \ \ \ \ \ \ \left(
\begin{array}
[c]{c}%
\text{because for }\left(  k,j\right)  \in A\text{, we have }\left\{
k,j\right\}  \in F\text{ if and only if }\left(  k,j\right)  \in\pi
^{-1}\left(  F\right) \\
\text{(by (\ref{pf.prop.digraph.2pf-chrom.short.equiv}), applied to }u=k\text{
and }v=j\text{)}%
\end{array}
\right) \\
&  \Longleftrightarrow\ \left(  \text{the digraph }\left(  V,\pi^{-1}\left(
F\right)  \right)  \text{ is }2\text{-step-free}\right)
\ \ \ \ \ \ \ \ \ \ \left(  \text{by the definition of \textquotedblleft%
}2\text{-step-free\textquotedblright}\right)  ,
\end{align*}
qed.}. Now, Corollary \ref{cor.chrompol.K-free} (applied to $X=\mathbb{N}$ and
$G=\underline{D}$) shows that%
\begin{align*}
\chi_{\underline{D}}  &  =\underbrace{\sum_{\substack{F\subseteq E;\\F\text{
is }\mathfrak{K}\text{-free}}}}_{\substack{=\sum_{\substack{F\subseteq
E;\\\text{the digraph }\left(  V,\pi^{-1}\left(  F\right)  \right)  \text{ is
}2\text{-step-free}}}\\\text{(since we have just shown that}\\\text{a subset
}F\text{ of }E\text{ is }\mathfrak{K}\text{-free if and only if}\\\text{the
digraph }\left(  V,\pi^{-1}\left(  F\right)  \right)  \text{ is }%
2\text{-step-free)}}}\underbrace{\left(  -1\right)  ^{\left\vert F\right\vert
}}_{\substack{=\left(  -1\right)  ^{\left\vert \pi^{-1}\left(  F\right)
\right\vert }\\\text{(since }\pi\text{ is bijective)}}%
}\underbrace{x^{\operatorname*{conn}\left(  V,F\right)  }}%
_{\substack{=x^{\operatorname*{conn}\left(  V,\operatorname*{set}\left(
\pi^{-1}\left(  F\right)  \right)  \right)  }\\\text{(by
(\ref{pf.prop.digraph.2pf-chrom.short.F=set}))}}}\\
&  =\sum_{\substack{F\subseteq E;\\\text{the digraph }\left(  V,\pi
^{-1}\left(  F\right)  \right)  \text{ is }2\text{-step-free}}}\left(
-1\right)  ^{\left\vert \pi^{-1}\left(  F\right)  \right\vert }%
x^{\operatorname*{conn}\left(  V,\operatorname*{set}\left(  \pi^{-1}\left(
F\right)  \right)  \right)  }\\
&  =\sum_{\substack{B\subseteq A;\\\text{the digraph }\left(  V,B\right)
\text{ is }2\text{-step-free}}}\left(  -1\right)  ^{\left\vert B\right\vert
}x^{\operatorname*{conn}\left(  V,\operatorname*{set}B\right)  }\\
&  \ \ \ \ \ \ \ \ \ \ \left(
\begin{array}
[c]{c}%
\text{here, we have substituted }B\text{ for }\pi^{-1}\left(  F\right)  \text{
in the sum,}\\
\text{since the map }\pi:A\rightarrow E\text{ is bijective and thus induces}\\
\text{a bijection from the subsets of }E\text{ to the subsets of }A\\
\text{sending each }F\subseteq E\text{ to }\pi^{-1}\left(  F\right)
\end{array}
\right) \\
&  =\sum_{\substack{F\subseteq A;\\\text{the digraph }\left(  V,F\right)
\text{ is }2\text{-step-free}}}\left(  -1\right)  ^{\left\vert F\right\vert
}x^{\operatorname*{conn}\left(  V,\operatorname*{set}F\right)  }%
\end{align*}
(here, we have renamed the summation index $B$ as $F$). This proves
Proposition \ref{prop.digraph.2pf-chrom}.
\end{proof}
\end{vershort}

\begin{verlong}
We prepare for the proof of this proposition by stating two simple lemmas:

\begin{lemma}
\label{lem.digraph.2pf-chrom.3cyc}Let $G=\left(  V,E\right)  $ be a graph. Let
$u$, $v$ and $w$ be three elements of $V$ such that $\left\{  u,v\right\}  \in
E$, $\left\{  v,w\right\}  \in E$ and $\left\{  u,w\right\}  \in E$. Let $C$
be the set $\left\{  \left\{  u,v\right\}  ,\left\{  v,w\right\}  ,\left\{
u,w\right\}  \right\}  $. Then, the set $C$ is a circuit of $G$ and satisfies
$C\setminus\left\{  \left\{  u,w\right\}  \right\}  =\left\{  \left\{
u,v\right\}  ,\left\{  v,w\right\}  \right\}  $.
\end{lemma}

\begin{proof}
[Proof of Lemma \ref{lem.digraph.2pf-chrom.3cyc}.]We have $E\subseteq
\dbinom{V}{2}$ (since $\left(  V,E\right)  $ is a graph). For any $a\in V$ and
$b\in V$ satisfying $\left\{  a,b\right\}  \in E$, we have%
\begin{equation}
a\neq b \label{pf.lem.digraph.2pf-chrom.3cyc.1}%
\end{equation}
\footnote{\textit{Proof of (\ref{pf.lem.digraph.2pf-chrom.3cyc.1}):} Let $a\in
V$ and $b\in V$ be such that $\left\{  a,b\right\}  \in E$. We need to prove
that $a\neq b$.
\par
Assume the contrary (for the sake of contradiction). Thus, $a=b$. Hence,
$\left\{  a,b\right\}  =\left\{  b,b\right\}  =\left\{  b\right\}  $, so that
$\left\vert \left\{  a,b\right\}  \right\vert =\left\vert \left\{  b\right\}
\right\vert =1$.
\par
But $\left\{  a,b\right\}  \in E\subseteq\dbinom{V}{2}$. In other words,
$\left\{  a,b\right\}  $ is a $2$-element subset of $V$ (since $\dbinom{V}{2}$
is the set of all $2$-element subsets of $V$). Thus, $\left\vert \left\{
a,b\right\}  \right\vert =2$. This contradicts $\left\vert \left\{
a,b\right\}  \right\vert =1<2$. This contradiction proves that our assumption
was wrong. Hence, $a\neq b$ holds. This proves
(\ref{pf.lem.digraph.2pf-chrom.3cyc.1}).}.

Now, (\ref{pf.lem.digraph.2pf-chrom.3cyc.1}) (applied to $a=u$ and $b=v$)
yields $u\neq v$ (since $\left\{  u,v\right\}  \in E$). Also,
(\ref{pf.lem.digraph.2pf-chrom.3cyc.1}) (applied to $a=v$ and $b=w$) yields
$v\neq w$ (since $\left\{  v,w\right\}  \in E$). Also,
(\ref{pf.lem.digraph.2pf-chrom.3cyc.1}) (applied to $a=u$ and $b=w$) yields
$u\neq w$ (since $\left\{  u,w\right\}  \in E$).

Now, set $m=3$. Thus, $m+1=4$. Now, $u$, $v$, $w$ and $u$ are elements of $V$.
Hence, $\left(  u,v,w,u\right)  \in V^{4}=V^{m+1}$ (since $4=m+1$). In other
words, $\left(  u,v,w,u\right)  $ is an $\left(  m+1\right)  $-tuple of
elements of $V$. Denote this $\left(  m+1\right)  $-tuple by $\left(
v_{1},v_{2},\ldots,v_{m+1}\right)  $. Hence, $\left(  v_{1},v_{2}%
,\ldots,v_{m+1}\right)  =\left(  u,v,w,u\right)  $. Therefore, $v_{1}=u$,
$v_{2}=v$, $v_{3}=w$ and $v_{4}=u$. Now, we have $m=3>2>1$ and%
\begin{align*}
v_{m+1}  &  =v_{4}\ \ \ \ \ \ \ \ \ \ \left(  \text{since }m+1=4\right) \\
&  =u=v_{1}.
\end{align*}
Also, the vertices $v_{1},v_{2},\ldots,v_{m}$ are pairwise
distinct\footnote{\textit{Proof.} We have $m=3$, and thus $\left(  v_{1}%
,v_{2},\ldots,v_{m}\right)  =\left(  \underbrace{v_{1}}_{=u},\underbrace{v_{2}%
}_{=v},\underbrace{v_{3}}_{=w}\right)  =\left(  u,v,w\right)  $.
\par
The vertices $u,v,w$ are pairwise distinct (since $u\neq v$, $u\neq w$ and
$v\neq w$). Since $\left(  u,v,w\right)  =\left(  v_{1},v_{2},\ldots
,v_{m}\right)  $, this rewrites as follows: The vertices $v_{1},v_{2}%
,\ldots,v_{m}$ are pairwise distinct. Qed.}. Finally, we have $\left\{
v_{i},v_{i+1}\right\}  \in E$ for every $i\in\left\{  1,2,\ldots,m\right\}
$\ \ \ \ \footnote{\textit{Proof.} Let $i\in\left\{  1,2,\ldots,m\right\}  $.
We want to prove that $\left\{  v_{i},v_{i+1}\right\}  \in E$. We have%
\begin{align*}
i  &  \in\left\{  1,2,\ldots,m\right\}  =\left\{  1,2,\ldots,3\right\}
\ \ \ \ \ \ \ \ \ \ \left(  \text{since }m=3\right) \\
&  =\left\{  1,2,3\right\}  .
\end{align*}
In other words, $i=1$ or $i=2$ or $i=3$. We are thus in one of the following
three cases:
\par
\textit{Case 1:} We have $i=1$.
\par
\textit{Case 2:} We have $i=2$.
\par
\textit{Case 3:} We have $i=3$.
\par
Let us first consider Case 1. In this case, we have $i=1$. Hence, $v_{i}%
=v_{1}=u$ and $v_{i+1}=v_{1+1}=v_{2}=v$. Now, $\left\{  \underbrace{v_{i}%
}_{=u},\underbrace{v_{i+1}}_{=v}\right\}  =\left\{  u,v\right\}  \in E$. Thus,
$\left\{  v_{i},v_{i+1}\right\}  \in E$ holds in Case 1.
\par
Let us next consider Case 2. In this case, we have $i=2$. Hence, $v_{i}%
=v_{2}=v$ and $v_{i+1}=v_{2+1}=v_{3}=w$. Now, $\left\{  \underbrace{v_{i}%
}_{=v},\underbrace{v_{i+1}}_{=w}\right\}  =\left\{  v,w\right\}  \in E$. Thus,
$\left\{  v_{i},v_{i+1}\right\}  \in E$ holds in Case 2.
\par
Let us finally consider Case 3. In this case, we have $i=3$. Hence,
$v_{i}=v_{3}=w$ and $v_{i+1}=v_{3+1}=v_{4}=u$. Now, $\left\{
\underbrace{v_{i}}_{=w},\underbrace{v_{i+1}}_{=u}\right\}  =\left\{
w,u\right\}  =\left\{  u,w\right\}  \in E$. Thus, $\left\{  v_{i}%
,v_{i+1}\right\}  \in E$ holds in Case 3.
\par
We have now proven $\left\{  v_{i},v_{i+1}\right\}  \in E$ in each of the
three Cases 1, 2 and 3. Thus, $\left\{  v_{i},v_{i+1}\right\}  \in E$ always
holds (since the Cases 1, 2 and 3 cover all possibilities). This completes our
proof.}.

Altogether, we now know that $\left(  v_{1},v_{2},\ldots,v_{m+1}\right)  $ is
a list of elements of $V$ satisfying the following four properties:

\begin{itemize}
\item We have $m>2$.

\item We have $v_{m+1}=v_{1}$.

\item The vertices $v_{1},v_{2},\ldots,v_{m}$ are pairwise distinct.

\item We have $\left\{  v_{i},v_{i+1}\right\}  \in E$ for every $i\in\left\{
1,2,\ldots,m\right\}  $.
\end{itemize}

According to the definition of a \textquotedblleft cycle\textquotedblright,
this means that the list $\left(  v_{1},v_{2},\ldots,v_{m+1}\right)  $ is a
cycle of $\left(  V,E\right)  $.

Thus, we conclude that the list $\left(  v_{1},v_{2},\ldots,v_{m+1}\right)  $
is a cycle of $\left(  V,E\right)  $. In other words, the list $\left(
v_{1},v_{2},\ldots,v_{m+1}\right)  $ is a cycle of $G$ (since $G=\left(
V,E\right)  $). Thus, the set $\left\{  \left\{  v_{1},v_{2}\right\}
,\left\{  v_{2},v_{3}\right\}  ,\ldots,\left\{  v_{m},v_{m+1}\right\}
\right\}  $ is a circuit of $G$ (by the definition of a \textquotedblleft
circuit\textquotedblright). Since%
\begin{align*}
&  \left\{  \left\{  v_{1},v_{2}\right\}  ,\left\{  v_{2},v_{3}\right\}
,\ldots,\left\{  v_{m},v_{m+1}\right\}  \right\} \\
&  =\left\{  \left\{  v_{1},v_{2}\right\}  ,\left\{  v_{2},v_{3}\right\}
,\ldots,\left\{  v_{3},v_{3+1}\right\}  \right\}  \ \ \ \ \ \ \ \ \ \ \left(
\text{since }m=3\right) \\
&  =\left\{  \left\{  \underbrace{v_{1}}_{=u},\underbrace{v_{2}}_{=v}\right\}
,\left\{  \underbrace{v_{2}}_{=v},\underbrace{v_{3}}_{=w}\right\}  ,\left\{
\underbrace{v_{3}}_{=w},\underbrace{v_{4}}_{=u}\right\}  \right\} \\
&  =\left\{  \left\{  u,v\right\}  ,\left\{  v,w\right\}
,\underbrace{\left\{  w,u\right\}  }_{=\left\{  u,w\right\}  }\right\}
=\left\{  \left\{  u,v\right\}  ,\left\{  v,w\right\}  ,\left\{  u,w\right\}
\right\}  =C
\end{align*}
(because $C=\left\{  \left\{  u,v\right\}  ,\left\{  v,w\right\}  ,\left\{
u,w\right\}  \right\}  $), this rewrites as follows: The set $C$ is a circuit
of $G$.

It remains to prove that $C\setminus\left\{  \left\{  u,w\right\}  \right\}
=\left\{  \left\{  u,v\right\}  ,\left\{  v,w\right\}  \right\}  $. Indeed, we
have $\left\{  u,w\right\}  \notin\left\{  \left\{  u,v\right\}  ,\left\{
v,w\right\}  \right\}  $\ \ \ \ \footnote{\textit{Proof.} Assume the contrary.
Thus, $\left\{  u,w\right\}  \in\left\{  \left\{  u,v\right\}  ,\left\{
v,w\right\}  \right\}  $. In other words, $\left\{  u,w\right\}  $ equals
either $\left\{  u,v\right\}  $ or $\left\{  v,w\right\}  $. In other words,
we must be in one of the following two cases:
\par
\textit{Case 1:} We have $\left\{  u,w\right\}  =\left\{  u,v\right\}  $.
\par
\textit{Case 2:} We have $\left\{  u,w\right\}  =\left\{  v,w\right\}  $.
\par
Let us first consider Case 1. In this case, we have $\left\{  u,w\right\}
=\left\{  u,v\right\}  $. Hence, $w\in\left\{  u,w\right\}  =\left\{
u,v\right\}  $. Combining this with $w\neq u$ (since $u\neq w$), we obtain
$w\in\left\{  u,v\right\}  \setminus\left\{  u\right\}  \subseteq\left\{
v\right\}  $. In other words, $w=v$. Hence, $v=w$. But this contradicts $v\neq
w$. Thus, we have found a contradiction in Case 1.
\par
Let us now consider Case 2. In this case, we have $\left\{  u,w\right\}
=\left\{  v,w\right\}  $. Hence, $u\in\left\{  u,w\right\}  =\left\{
v,w\right\}  $. Combining this with $u\neq v$, we obtain $u\in\left\{
v,w\right\}  \setminus\left\{  v\right\}  \subseteq\left\{  w\right\}  $. In
other words, $u=w$. This contradicts $u\neq w$. Thus, we have found a
contradiction in Case 2.
\par
We have now found a contradiction in each of the two Cases 1 and 2. Since
these two Cases cover all possibilities, we thus always have a contradiction.
This contradiction shows that our assumption was wrong, qed.}. Now,%
\[
C=\left\{  \left\{  u,v\right\}  ,\left\{  v,w\right\}  ,\left\{  u,w\right\}
\right\}  =\left\{  \left\{  u,v\right\}  ,\left\{  v,w\right\}  \right\}
\cup\left\{  \left\{  u,w\right\}  \right\}  .
\]
Hence,%
\begin{align*}
&  \underbrace{C}_{=\left\{  \left\{  u,v\right\}  ,\left\{  v,w\right\}
\right\}  \cup\left\{  \left\{  u,w\right\}  \right\}  }\setminus\left\{
\left\{  u,w\right\}  \right\} \\
&  =\left(  \left\{  \left\{  u,v\right\}  ,\left\{  v,w\right\}  \right\}
\cup\left\{  \left\{  u,w\right\}  \right\}  \right)  \setminus\left\{
\left\{  u,w\right\}  \right\} \\
&  =\left\{  \left\{  u,v\right\}  ,\left\{  v,w\right\}  \right\}
\ \ \ \ \ \ \ \ \ \ \left(  \text{since }\left\{  u,w\right\}  \notin\left\{
\left\{  u,v\right\}  ,\left\{  v,w\right\}  \right\}  \right)  .
\end{align*}
This completes the proof of Lemma \ref{lem.digraph.2pf-chrom.3cyc}.
\end{proof}

\begin{lemma}
\label{lem.digraph.2pf-chrom.K-free}Let $D=\left(  V,A\right)  $ be a finite
transitive loopless digraph. Let $E=\operatorname*{set}A$. Every $a\in A$
satisfies $\operatorname*{set}\underbrace{a}_{\in A}\in\operatorname*{set}%
A=E$. Hence, we can define a map $\pi:A\rightarrow E$ by%
\[
\left(  \pi\left(  a\right)  =\operatorname*{set}%
a\ \ \ \ \ \ \ \ \ \ \text{for every }a\in A\right)  .
\]
Consider this map $\pi$.

\textbf{(a)} The map $\pi:A\rightarrow E$ is bijective. Thus, an inverse map
$\pi^{-1}:E\rightarrow A$ is well-defined.

\textbf{(b)} Let $Z$ be the set $\left\{  \left(  i,k,j\right)  \in
V^{3}\ \mid\ \left(  i,k\right)  \in A\text{ and }\left(  k,j\right)  \in
A\right\}  $. Define a set $\mathfrak{K}$ by
\[
\mathfrak{K}=\left\{  \left\{  \left\{  i,k\right\}  ,\left\{  k,j\right\}
\right\}  \ \mid\ \left(  i,k,j\right)  \in Z\right\}  .
\]
Then, $\mathfrak{K}\subseteq\mathcal{P}\left(  E\right)  $.

\textbf{(c)} Let $G$ be the graph $\left(  V,E\right)  $. For every $\left(
i,j\right)  \in V^{2}$, define a subset $\operatorname*{inter}\left(
i,j\right)  $ of $V$ by%
\[
\operatorname*{inter}\left(  i,j\right)  =\left\{  k\in V\ \mid\ \left(
i,k\right)  \in A\text{ and }\left(  k,j\right)  \in A\right\}  .
\]
Define a map $\ell^{\prime}:A\rightarrow\mathbb{N}$ by setting%
\[
\left(  \ell^{\prime}\left(  i,j\right)  =\left\vert \operatorname*{inter}%
\left(  i,j\right)  \right\vert \ \ \ \ \ \ \ \ \ \ \text{for every }\left(
i,j\right)  \in A\right)  .
\]
Recall that a map $\pi^{-1}:E\rightarrow A$ is well-defined (by Lemma
\ref{lem.digraph.2pf-chrom.K-free} \textbf{(a)}). Define a map $\ell
:E\rightarrow\mathbb{N}$ by $\ell=\ell^{\prime}\circ\pi^{-1}$. We shall refer
to $\ell$ as the \emph{labeling function}. For every edge $e$ of $G$, we shall
refer to $\ell\left(  e\right)  $ as the \emph{label} of $e$. If $u$, $v$ and
$w$ are three elements of $V$ satisfying $\left(  u,v\right)  \in A$ and
$\left(  v,w\right)  \in A$, then%
\begin{align}
\left\{  u,w\right\}   &  \in E,\ \ \ \ \ \ \ \ \ \ \left\{  u,v\right\}  \in
E,\ \ \ \ \ \ \ \ \ \ \left\{  v,w\right\}  \in
E,\label{eq.lem.digraph.2pf-chrom.K-free.c.1}\\
\ell\left(  \left\{  u,w\right\}  \right)   &  >\ell\left(  \left\{
u,v\right\}  \right)  \ \ \ \ \ \ \ \ \ \ \text{and}\ \ \ \ \ \ \ \ \ \ \ell
\left(  \left\{  u,w\right\}  \right)  >\ell\left(  \left\{  v,w\right\}
\right)  . \label{eq.lem.digraph.2pf-chrom.K-free.c.2}%
\end{align}

\textbf{(d)} Consider the set $\mathfrak{K}$ defined in Lemma
\ref{lem.digraph.2pf-chrom.K-free} \textbf{(b)}. Consider the labeling
function $\ell$ defined in Lemma \ref{lem.digraph.2pf-chrom.K-free}
\textbf{(c)}. Definition \ref{def.BC} (applied to $X=\mathbb{N}$) shows that
the notion of a broken circuit of $G$ is well-defined (since a labeling
function $\ell:E\rightarrow\mathbb{N}$ is given).

Every element of $\mathfrak{K}$ is a broken circuit of $G$.

\textbf{(e)} Consider the set $\mathfrak{K}$ defined in Lemma
\ref{lem.digraph.2pf-chrom.K-free} \textbf{(b)}. Let $F$ be a subset of $A$.
Then, we have the following logical equivalence:%
\[
\ \left(  \text{the digraph }\left(  V,F\right)  \text{ is }2\text{-step-free}%
\right)  \ \Longleftrightarrow\ \left(  \text{the set }\pi\left(  F\right)
\text{ is }\mathfrak{K}\text{-free}\right)  .
\]

\end{lemma}

\begin{proof}
[Proof of Lemma \ref{lem.digraph.2pf-chrom.K-free}.]We have $A\subseteq V^{2}$
(since $\left(  V,A\right)  $ is a digraph). The set $V$ is finite (since the
digraph $\left(  V,A\right)  $ is finite).

Recall that the digraph $\left(  V,A\right)  $ is transitive if and only if,
for any $u\in V$, $v\in V$ and $w\in V$ satisfying $\left(  u,v\right)  \in A$
and $\left(  v,w\right)  \in A$, we have $\left(  u,w\right)  \in A$ (by the
definition of \textquotedblleft transitive\textquotedblright). Thus, for any
$u\in V$, $v\in V$ and $w\in V$ satisfying $\left(  u,v\right)  \in A$ and
$\left(  v,w\right)  \in A$, we have
\begin{equation}
\left(  u,w\right)  \in A \label{pf.lem.digraph.2pf-chrom.K-free.transitive}%
\end{equation}
(because the digraph $\left(  V,A\right)  $ is transitive).

Recall that the digraph $\left(  V,A\right)  $ is loopless if and only if
every $v\in V$ satisfies $\left(  v,v\right)  \notin A$. Thus, every $v\in V$
satisfies%
\begin{equation}
\left(  v,v\right)  \notin A \label{pf.lem.digraph.2pf-chrom.K-free.loopless}%
\end{equation}
(since the digraph $\left(  V,A\right)  $ is loopless).

\textbf{(a)} We have
\[
\pi\left(  A\right)  =\left\{  \underbrace{\pi\left(  a\right)  }%
_{\substack{=\operatorname*{set}a\\\text{(by the definition of }\pi\text{)}%
}}\ \mid\ a\in A\right\}  =\left\{  \operatorname*{set}a\ \mid\ a\in
A\right\}  =\operatorname*{set}A=E.
\]
Thus, the map $\pi$ is surjective.

Furthermore, if $a$ and $b$ are two elements of $A$ such that $\pi\left(
a\right)  =\pi\left(  b\right)  $, then $a=b$\ \ \ \ \footnote{\textit{Proof.}
Let $a$ and $b$ be two elements of $A$ such that $\pi\left(  a\right)
=\pi\left(  b\right)  $. We must show that $a=b$.
\par
Assume the contrary. Thus, $a\neq b$.
\par
We have $a\in A\subseteq V^{2}$. Thus, we can write $a$ in the form $a=\left(
i,j\right)  $ with $i\in V$ and $j\in V$. Consider these $i$ and $j$.
\par
We have $b\in A\subseteq V^{2}$. Thus, we can write $b$ in the form $b=\left(
i^{\prime},j^{\prime}\right)  $ with $i^{\prime}\in V$ and $j^{\prime}\in V$.
Consider these $i^{\prime}$ and $j^{\prime}$.
\par
Applying (\ref{pf.lem.digraph.2pf-chrom.K-free.loopless}) to $v=i$, we obtain
$\left(  i,i\right)  \notin A$. If we had $i=j$, then we would have $\left(
i,\underbrace{i}_{=j}\right)  =\left(  i,j\right)  =a\in A$, which would
contradict $\left(  i,i\right)  \notin A$. Thus, we cannot have $i=j$. Hence,
we have $i\neq j$, so that $j\neq i$.
\par
Applying the map $\pi$ to both sides of the equality $a=\left(  i,j\right)  $,
we obtain
\begin{align*}
\pi\left(  a\right)   &  =\pi\left(  i,j\right)  =\operatorname*{set}\left(
i,j\right)  \ \ \ \ \ \ \ \ \ \ \left(  \text{by the definition of }\pi\right)
\\
&  =\left\{  i,j\right\}  \ \ \ \ \ \ \ \ \ \ \left(  \text{by the definition
of the map }\operatorname*{set}\right)  .
\end{align*}
The same argument (but applied to $b$, $i^{\prime}$ and $j^{\prime}$ instead
of $a$, $i$ and $j$) shows that $\pi\left(  b\right)  =\left\{  i^{\prime
},j^{\prime}\right\}  $. Hence, $\left\{  i,j\right\}  =\pi\left(  a\right)
=\pi\left(  b\right)  =\left\{  i^{\prime},j^{\prime}\right\}  $.
\par
We have $i\in\left\{  i,j\right\}  =\left\{  i^{\prime},j^{\prime}\right\}  $.
In other words, either $i=i^{\prime}$ or $i=j^{\prime}$. In other words, we
are in one of the following two cases:
\par
\textit{Case 1:} We have $i=i^{\prime}$.
\par
\textit{Case 2:} We have $i=j^{\prime}$.
\par
Let us first consider Case 1. In this case, we have $i=i^{\prime}$. Now,
$j\in\left\{  i,j\right\}  =\left\{  i^{\prime},j^{\prime}\right\}  $.
Combining this with $j\neq i=i^{\prime}$, we obtain $j\in\left\{  i^{\prime
},j^{\prime}\right\}  \setminus\left\{  i^{\prime}\right\}  \subseteq\left\{
j^{\prime}\right\}  $. Thus, $j=j^{\prime}$. Now, $a=\left(  \underbrace{i}%
_{=i^{\prime}},\underbrace{j}_{=j^{\prime}}\right)  =\left(  i^{\prime
},j^{\prime}\right)  =b$; this contradicts $a\neq b$. Hence, we have found a
contradiction in Case 1.
\par
Let us now consider Case 2. In this case, we have $i=j^{\prime}$. Now,
$j\in\left\{  i,j\right\}  =\left\{  i^{\prime},j^{\prime}\right\}  $.
Combining this with $j\neq i=j^{\prime}$, we obtain $j\in\left\{  i^{\prime
},j^{\prime}\right\}  \setminus\left\{  j^{\prime}\right\}  \subseteq\left\{
i^{\prime}\right\}  $. Thus, $j=i^{\prime}$. Hence, $\left(  \underbrace{j}%
_{=i^{\prime}},\underbrace{i}_{=j^{\prime}}\right)  =\left(  i^{\prime
},j^{\prime}\right)  =b\in A$. Also, $\left(  i,j\right)  =a\in A$. Hence,
(\ref{pf.lem.digraph.2pf-chrom.K-free.transitive}) (applied to $u=i$, $v=j$
and $w=i$) yields $\left(  i,i\right)  \in A$. This contradicts $\left(
i,i\right)  \notin A$. Thus, we have found a contradiction in Case 2.
\par
We have now found a contradiction in each of the two Cases 1 and 2. Hence, we
always have a contradiction (since Cases 1 and 2 cover all possibilities).
Thus, our assumption was wrong. Hence, $a=b$ is proven. Qed.}. In other words,
the map $\pi$ is injective. Now, the map $\pi$ is bijective (since $\pi$ is
surjective and injective). Thus, an inverse map $\pi^{-1}:E\rightarrow A$ is
well-defined. This proves Lemma \ref{lem.digraph.2pf-chrom.K-free}
\textbf{(a)}.

\textbf{(b)} The definition of $Z$ yields%
\begin{align*}
Z  &  =\left\{  \left(  i,k,j\right)  \in V^{3}\ \mid\ \left(  i,k\right)  \in
A\text{ and }\left(  k,j\right)  \in A\right\} \\
&  =\left\{  \left(  u,v,w\right)  \in V^{3}\ \mid\ \left(  u,v\right)  \in
A\text{ and }\left(  v,w\right)  \in A\right\}
\end{align*}
(here, we have renamed the index $\left(  i,k,j\right)  $ as $\left(
u,v,w\right)  $). Now, let $K\in\mathfrak{K}$. We shall prove that
$K\in\mathcal{P}\left(  E\right)  $.

We have $K\in\mathfrak{K}=\left\{  \left\{  \left\{  i,k\right\}  ,\left\{
k,j\right\}  \right\}  \ \mid\ \left(  i,k,j\right)  \in Z\right\}  $. Hence,
$K=\left\{  \left\{  i,k\right\}  ,\left\{  k,j\right\}  \right\}  $ for some
$\left(  i,k,j\right)  \in Z$. Consider this $\left(  i,k,j\right)  $. We
have
\[
\left(  i,k,j\right)  \in Z=\left\{  \left(  u,v,w\right)  \in V^{3}%
\ \mid\ \left(  u,v\right)  \in A\text{ and }\left(  v,w\right)  \in
A\right\}  .
\]
In other words, $\left(  i,k,j\right)  $ is an $\left(  u,v,w\right)  \in
V^{3}$ satisfying $\left(  u,v\right)  \in A$ and $\left(  v,w\right)  \in A$.
In other words, $\left(  i,k,j\right)  $ is an element of $V^{3}$ and
satisfies $\left(  i,k\right)  \in A$ and $\left(  k,j\right)  \in A$.

Now, $\left(  i,k\right)  \in A$. The definition of $\pi$ therefore yields
$\pi\left(  i,k\right)  =\operatorname*{set}\left(  i,k\right)  =\left\{
i,k\right\}  $ (by the definition of the map $\operatorname*{set}$). Hence,
$\left\{  i,k\right\}  =\pi\underbrace{\left(  i,k\right)  }_{\in A}\in
\pi\left(  A\right)  \subseteq E$.

Also, $\left(  k,j\right)  \in A$. The definition of $\pi$ therefore yields
$\pi\left(  k,j\right)  =\operatorname*{set}\left(  k,j\right)  =\left\{
k,j\right\}  $ (by the definition of the map $\operatorname*{set}$). Hence,
$\left\{  k,j\right\}  =\pi\underbrace{\left(  k,j\right)  }_{\in A}\in
\pi\left(  A\right)  \subseteq E$.

Now, we know that both $\left\{  i,k\right\}  $ and $\left\{  k,j\right\}  $
belong to $E$. Hence, $\left\{  \left\{  i,k\right\}  ,\left\{  k,j\right\}
\right\}  \subseteq E$. Hence, $K=\left\{  \left\{  i,k\right\}  ,\left\{
k,j\right\}  \right\}  \subseteq E$, so that $K\in\mathcal{P}\left(  E\right)
$.

Now, forget that we fixed $K$. We thus have shown that $K\in\mathcal{P}\left(
E\right)  $ for every $K\in\mathfrak{K}$. In other words, $\mathfrak{K}%
\subseteq\mathcal{P}\left(  E\right)  $. This proves Lemma
\ref{lem.digraph.2pf-chrom.K-free} \textbf{(b)}.

\textbf{(c)} We only need to prove that if $u$, $v$ and $w$ are three elements
of $V$ satisfying $\left(  u,v\right)  \in A$ and $\left(  v,w\right)  \in A$,
then (\ref{eq.lem.digraph.2pf-chrom.K-free.c.1}) and
(\ref{eq.lem.digraph.2pf-chrom.K-free.c.2}) hold.

So let $u$, $v$ and $w$ be three elements of $V$ satisfying $\left(
u,v\right)  \in A$ and $\left(  v,w\right)  \in A$. We must prove
(\ref{eq.lem.digraph.2pf-chrom.K-free.c.1}) and
(\ref{eq.lem.digraph.2pf-chrom.K-free.c.2}).

The definition of $\operatorname*{inter}\left(  u,v\right)  $ yields%
\[
\operatorname*{inter}\left(  u,v\right)  =\left\{  k\in V\ \mid\ \left(
u,k\right)  \in A\text{ and }\left(  k,v\right)  \in A\right\}  .
\]
The definition of $\operatorname*{inter}\left(  v,w\right)  $ yields%
\[
\operatorname*{inter}\left(  v,w\right)  =\left\{  k\in V\ \mid\ \left(
v,k\right)  \in A\text{ and }\left(  k,w\right)  \in A\right\}  .
\]
The definition of $\operatorname*{inter}\left(  u,w\right)  $ yields%
\[
\operatorname*{inter}\left(  u,w\right)  =\left\{  k\in V\ \mid\ \left(
u,k\right)  \in A\text{ and }\left(  k,w\right)  \in A\right\}  .
\]

We have $v\in\operatorname*{inter}\left(  u,w\right)  $%
\ \ \ \ \footnote{\textit{Proof.} We know that $v$ is an element of $V$ and
satisfies $\left(  u,v\right)  \in A$ and $\left(  v,w\right)  \in A$. In
other words, $v$ is an element $k$ of $V$ satisfying $\left(  u,k\right)  \in
A$ and $\left(  k,w\right)  \in A$. Hence,%
\[
v\in\left\{  k\in V\ \mid\ \left(  u,k\right)  \in A\text{ and }\left(
k,w\right)  \in A\right\}  =\operatorname*{inter}\left(  u,w\right)  .
\]
Qed.}.

All three sets $\operatorname*{inter}\left(  u,v\right)  $,
$\operatorname*{inter}\left(  v,w\right)  $ and $\operatorname*{inter}\left(
u,w\right)  $ are subsets of the finite set $V$, and thus are finite.

Now, $\operatorname*{inter}\left(  u,v\right)  $ is a proper subset of
$\operatorname*{inter}\left(  u,w\right)  $\ \ \ \ \footnote{\textit{Proof.}
Let $g\in\operatorname*{inter}\left(  u,v\right)  $. Thus,%
\[
g\in\operatorname*{inter}\left(  u,v\right)  =\left\{  k\in V\ \mid\ \left(
u,k\right)  \in A\text{ and }\left(  k,v\right)  \in A\right\}  .
\]
In other words, $g$ is an element $k$ of $V$ satisfying $\left(  u,k\right)
\in A$ and $\left(  k,v\right)  \in A$. In other words, $g$ is an element of
$V$ and satisfies $\left(  u,g\right)  \in A$ and $\left(  g,v\right)  \in A$.
\par
Now, (\ref{pf.lem.digraph.2pf-chrom.K-free.transitive}) (applied to $g$
instead of $u$) yields $\left(  g,w\right)  \in A$ (since $\left(  g,v\right)
\in A$ and $\left(  v,w\right)  \in A$). Hence, we now know that $g$ is an
element of $V$ and satisfies $\left(  u,g\right)  \in A$ and $\left(
g,w\right)  \in A$. In other words, $g$ is an element $k$ of $V$ satisfying
$\left(  u,k\right)  \in A$ and $\left(  k,w\right)  \in A$. Hence,%
\[
g\in\left\{  k\in V\ \mid\ \left(  u,k\right)  \in A\text{ and }\left(
k,w\right)  \in A\right\}  =\operatorname*{inter}\left(  u,w\right)  .
\]
\par
Now, forget that we fixed $g$. We therefore have proven that $g\in
\operatorname*{inter}\left(  u,w\right)  $ for every $g\in
\operatorname*{inter}\left(  u,v\right)  $. In other words,
$\operatorname*{inter}\left(  u,v\right)  \subseteq\operatorname*{inter}%
\left(  u,w\right)  $.
\par
Next, we shall prove that $\operatorname*{inter}\left(  u,v\right)
\neq\operatorname*{inter}\left(  u,w\right)  $. Indeed, assume the contrary.
Thus, $\operatorname*{inter}\left(  u,v\right)  =\operatorname*{inter}\left(
u,w\right)  $. Now,%
\[
v\in\operatorname*{inter}\left(  u,w\right)  =\operatorname*{inter}\left(
u,v\right)  =\left\{  k\in V\ \mid\ \left(  u,k\right)  \in A\text{ and
}\left(  k,v\right)  \in A\right\}  .
\]
In other words, $v$ is an element $k$ of $V$ satisfying $\left(  u,k\right)
\in A$ and $\left(  k,v\right)  \in A$. In other words, $v$ is an element of
$V$ and satisfies $\left(  u,v\right)  \in A$ and $\left(  v,v\right)  \in A$.
But (\ref{pf.lem.digraph.2pf-chrom.K-free.loopless}) yields $\left(
v,v\right)  \notin A$. This contradicts $\left(  v,v\right)  \in A$. This
contradiction proves that our assumption was false. Hence,
$\operatorname*{inter}\left(  u,v\right)  \neq\operatorname*{inter}\left(
u,w\right)  $ is proven. Combining this with $\operatorname*{inter}\left(
u,v\right)  \subseteq\operatorname*{inter}\left(  u,w\right)  $, we conclude
that $\operatorname*{inter}\left(  u,v\right)  $ is a proper subset of
$\operatorname*{inter}\left(  u,w\right)  $. Qed.}. Hence,
\begin{equation}
\left\vert \operatorname*{inter}\left(  u,v\right)  \right\vert <\left\vert
\operatorname*{inter}\left(  u,w\right)  \right\vert
\label{pf.lem.digraph.2pf-chrom.K-free.c.ieq1}%
\end{equation}
(since $\operatorname*{inter}\left(  u,v\right)  $ and $\operatorname*{inter}%
\left(  u,w\right)  $ are finite sets).

Furthermore, $\operatorname*{inter}\left(  v,w\right)  $ is a proper subset of
$\operatorname*{inter}\left(  u,w\right)  $\ \ \ \ \footnote{\textit{Proof.}
Let $g\in\operatorname*{inter}\left(  v,w\right)  $. Thus,%
\[
g\in\operatorname*{inter}\left(  v,w\right)  =\left\{  k\in V\ \mid\ \left(
v,k\right)  \in A\text{ and }\left(  k,w\right)  \in A\right\}  .
\]
In other words, $g$ is an element $k$ of $V$ satisfying $\left(  v,k\right)
\in A$ and $\left(  k,w\right)  \in A$. In other words, $g$ is an element of
$V$ and satisfies $\left(  v,g\right)  \in A$ and $\left(  g,w\right)  \in A$.
\par
Now, (\ref{pf.lem.digraph.2pf-chrom.K-free.transitive}) (applied to $g$
instead of $w$) yields $\left(  u,g\right)  \in A$ (since $\left(  u,v\right)
\in A$ and $\left(  v,g\right)  \in A$). Hence, we now know that $g$ is an
element of $V$ and satisfies $\left(  u,g\right)  \in A$ and $\left(
g,w\right)  \in A$. In other words, $g$ is an element $k$ of $V$ satisfying
$\left(  u,k\right)  \in A$ and $\left(  k,w\right)  \in A$. Hence,%
\[
g\in\left\{  k\in V\ \mid\ \left(  u,k\right)  \in A\text{ and }\left(
k,w\right)  \in A\right\}  =\operatorname*{inter}\left(  u,w\right)  .
\]
\par
Now, forget that we fixed $g$. We therefore have proven that $g\in
\operatorname*{inter}\left(  u,w\right)  $ for every $g\in
\operatorname*{inter}\left(  v,w\right)  $. In other words,
$\operatorname*{inter}\left(  v,w\right)  \subseteq\operatorname*{inter}%
\left(  u,w\right)  $.
\par
Next, we shall prove that $\operatorname*{inter}\left(  v,w\right)
\neq\operatorname*{inter}\left(  u,w\right)  $. Indeed, assume the contrary.
Thus, $\operatorname*{inter}\left(  v,w\right)  =\operatorname*{inter}\left(
u,w\right)  $. Now,%
\[
v\in\operatorname*{inter}\left(  u,w\right)  =\operatorname*{inter}\left(
v,w\right)  =\left\{  k\in V\ \mid\ \left(  v,k\right)  \in A\text{ and
}\left(  k,w\right)  \in A\right\}  .
\]
In other words, $v$ is an element $k$ of $V$ satisfying $\left(  v,k\right)
\in A$ and $\left(  k,w\right)  \in A$. In other words, $v$ is an element of
$V$ and satisfies $\left(  v,v\right)  \in A$ and $\left(  v,w\right)  \in A$.
But (\ref{pf.lem.digraph.2pf-chrom.K-free.loopless}) yields $\left(
v,v\right)  \notin A$. This contradicts $\left(  v,v\right)  \in A$. This
contradiction proves that our assumption was false. Hence,
$\operatorname*{inter}\left(  v,w\right)  \neq\operatorname*{inter}\left(
u,w\right)  $ is proven. Combining this with $\operatorname*{inter}\left(
v,w\right)  \subseteq\operatorname*{inter}\left(  u,w\right)  $, we conclude
that $\operatorname*{inter}\left(  v,w\right)  $ is a proper subset of
$\operatorname*{inter}\left(  u,w\right)  $. Qed.}. Hence,
\begin{equation}
\left\vert \operatorname*{inter}\left(  v,w\right)  \right\vert <\left\vert
\operatorname*{inter}\left(  u,w\right)  \right\vert
\label{pf.lem.digraph.2pf-chrom.K-free.c.ieq2}%
\end{equation}
(since $\operatorname*{inter}\left(  v,w\right)  $ and $\operatorname*{inter}%
\left(  u,w\right)  $ are finite sets).

From (\ref{pf.lem.digraph.2pf-chrom.K-free.transitive}), we obtain $\left(
u,w\right)  \in A$. The definition of $\pi$ thus shows that%
\[
\pi\left(  u,w\right)  =\operatorname*{set}\left(  u,w\right)  =\left\{
u,w\right\}  \ \ \ \ \ \ \ \ \ \ \left(  \text{by the definition of the map
}\operatorname*{set}\right)  .
\]
Hence, $\left\{  u,w\right\}  =\pi\underbrace{\left(  u,w\right)  }_{\in A}%
\in\pi\left(  A\right)  \subseteq E$. Thus, $\ell\left(  \left\{  u,w\right\}
\right)  $ is well-defined.

Also, recall that $\left(  u,v\right)  \in A$. The definition of $\pi$ thus
shows that%
\[
\pi\left(  u,v\right)  =\operatorname*{set}\left(  u,v\right)  =\left\{
u,v\right\}  \ \ \ \ \ \ \ \ \ \ \left(  \text{by the definition of the map
}\operatorname*{set}\right)  .
\]
Hence, $\left\{  u,v\right\}  =\pi\underbrace{\left(  u,v\right)  }_{\in A}%
\in\pi\left(  A\right)  \subseteq E$. Thus, $\ell\left(  \left\{  u,v\right\}
\right)  $ is well-defined.

Also, recall that $\left(  v,w\right)  \in A$. The definition of $\pi$ thus
shows that%
\[
\pi\left(  v,w\right)  =\operatorname*{set}\left(  v,w\right)  =\left\{
v,w\right\}  \ \ \ \ \ \ \ \ \ \ \left(  \text{by the definition of the map
}\operatorname*{set}\right)  .
\]
Hence, $\left\{  v,w\right\}  =\pi\underbrace{\left(  v,w\right)  }_{\in A}%
\in\pi\left(  A\right)  \subseteq E$. Thus, $\ell\left(  \left\{  v,w\right\}
\right)  $ is well-defined. We have also proven
(\ref{eq.lem.digraph.2pf-chrom.K-free.c.1}) by now.

Now,%
\begin{align*}
\underbrace{\ell}_{=\ell^{\prime}\circ\pi^{-1}}\left(  \left\{  u,w\right\}
\right)   &  =\left(  \ell^{\prime}\circ\pi^{-1}\right)  \left(  \left\{
u,w\right\}  \right)  =\ell^{\prime}\left(  \underbrace{\pi^{-1}\left(
\left\{  u,w\right\}  \right)  }_{\substack{=\left(  u,w\right)
\\\text{(since }\left\{  u,w\right\}  =\pi\left(  u,w\right)  \text{)}%
}}\right)  =\ell^{\prime}\left(  u,w\right) \\
&  =\left\vert \operatorname*{inter}\left(  u,w\right)  \right\vert
\ \ \ \ \ \ \ \ \ \ \left(  \text{by the definition of }\ell^{\prime}\right)
.
\end{align*}
Also,%
\begin{align*}
\underbrace{\ell}_{=\ell^{\prime}\circ\pi^{-1}}\left(  \left\{  u,v\right\}
\right)   &  =\left(  \ell^{\prime}\circ\pi^{-1}\right)  \left(  \left\{
u,v\right\}  \right)  =\ell^{\prime}\left(  \underbrace{\pi^{-1}\left(
\left\{  u,v\right\}  \right)  }_{\substack{=\left(  u,v\right)
\\\text{(since }\left\{  u,v\right\}  =\pi\left(  u,v\right)  \text{)}%
}}\right)  =\ell^{\prime}\left(  u,v\right) \\
&  =\left\vert \operatorname*{inter}\left(  u,v\right)  \right\vert
\ \ \ \ \ \ \ \ \ \ \left(  \text{by the definition of }\ell^{\prime}\right)
,
\end{align*}
and%
\begin{align*}
\underbrace{\ell}_{=\ell^{\prime}\circ\pi^{-1}}\left(  \left\{  v,w\right\}
\right)   &  =\left(  \ell^{\prime}\circ\pi^{-1}\right)  \left(  \left\{
v,w\right\}  \right)  =\ell^{\prime}\left(  \underbrace{\pi^{-1}\left(
\left\{  v,w\right\}  \right)  }_{\substack{=\left(  v,w\right)
\\\text{(since }\left\{  v,w\right\}  =\pi\left(  v,w\right)  \text{)}%
}}\right)  =\ell^{\prime}\left(  v,w\right) \\
&  =\left\vert \operatorname*{inter}\left(  v,w\right)  \right\vert
\ \ \ \ \ \ \ \ \ \ \left(  \text{by the definition of }\ell^{\prime}\right)
.
\end{align*}
Now,%
\begin{align*}
\ell\left(  \left\{  u,w\right\}  \right)   &  =\left\vert
\operatorname*{inter}\left(  u,w\right)  \right\vert >\left\vert
\operatorname*{inter}\left(  u,v\right)  \right\vert
\ \ \ \ \ \ \ \ \ \ \left(  \text{by
(\ref{pf.lem.digraph.2pf-chrom.K-free.c.ieq1})}\right) \\
&  =\ell\left(  \left\{  u,v\right\}  \right)
\end{align*}
and%
\begin{align*}
\ell\left(  \left\{  u,w\right\}  \right)   &  =\left\vert
\operatorname*{inter}\left(  u,w\right)  \right\vert >\left\vert
\operatorname*{inter}\left(  v,w\right)  \right\vert
\ \ \ \ \ \ \ \ \ \ \left(  \text{by
(\ref{pf.lem.digraph.2pf-chrom.K-free.c.ieq2})}\right) \\
&  =\ell\left(  \left\{  v,w\right\}  \right)  .
\end{align*}
Thus, (\ref{eq.lem.digraph.2pf-chrom.K-free.c.2}) holds. This proves Lemma
\ref{lem.digraph.2pf-chrom.K-free} \textbf{(c)}.

\textbf{(d)} Let $K\in\mathfrak{K}$. We shall show that $K$ is a broken
circuit of $G$.

We have
\begin{align*}
K  &  \in\mathfrak{K}=\left\{  \left\{  \left\{  i,k\right\}  ,\left\{
k,j\right\}  \right\}  \ \mid\ \left(  i,k,j\right)  \in Z\right\} \\
&  =\left\{  \left\{  \left\{  u,v\right\}  ,\left\{  v,w\right\}  \right\}
\ \mid\ \left(  u,v,w\right)  \in Z\right\}
\end{align*}
(here, we renamed the index $\left(  i,k,j\right)  $ as $\left(  u,v,w\right)
$). Hence, $K=\left\{  \left\{  u,v\right\}  ,\left\{  v,w\right\}  \right\}
$ for some $\left(  u,v,w\right)  \in Z$. Consider this $\left(  u,v,w\right)
$. We have%
\[
\left(  u,v,w\right)  \in Z=\left\{  \left(  i,k,j\right)  \in V^{3}%
\ \mid\ \left(  i,k\right)  \in A\text{ and }\left(  k,j\right)  \in
A\right\}  .
\]
In other words, $\left(  u,v,w\right)  $ is an $\left(  i,k,j\right)  \in
V^{3}$ satisfying $\left(  i,k\right)  \in A$ and $\left(  k,j\right)  \in A$.
In other words, $\left(  u,v,w\right)  $ is an element of $V^{3}$ and
satisfies $\left(  u,v\right)  \in A$ and $\left(  v,w\right)  \in A$. Hence,
Lemma \ref{lem.digraph.2pf-chrom.K-free} \textbf{(c)} shows that we have%
\begin{align*}
\left\{  u,w\right\}   &  \in E,\ \ \ \ \ \ \ \ \ \ \left\{  u,v\right\}  \in
E,\ \ \ \ \ \ \ \ \ \ \left\{  u,w\right\}  \in E,\\
\ell\left(  \left\{  u,w\right\}  \right)   &  >\ell\left(  \left\{
u,v\right\}  \right)  \ \ \ \ \ \ \ \ \ \ \text{and}\ \ \ \ \ \ \ \ \ \ \ell
\left(  \left\{  u,w\right\}  \right)  >\ell\left(  \left\{  v,w\right\}
\right)  .
\end{align*}
Let $D$ be the set $\left\{  \left\{  u,v\right\}  ,\left\{  v,w\right\}
,\left\{  u,w\right\}  \right\}  $. Now, Lemma
\ref{lem.digraph.2pf-chrom.3cyc} (applied to $C=D$) shows that the set $D$ is
a circuit of $G$ and satisfies $D\setminus\left\{  \left\{  u,w\right\}
\right\}  =\left\{  \left\{  u,v\right\}  ,\left\{  v,w\right\}  \right\}  $.
Hence, $K=\left\{  \left\{  u,v\right\}  ,\left\{  v,w\right\}  \right\}
=D\setminus\left\{  \left\{  u,w\right\}  \right\}  $. Moreover, the edge
$\left\{  u,w\right\}  $ is the unique edge in $D$ having maximum label (among
the edges in $D$)\ \ \ \ \footnote{\textit{Proof.} We have $D=\left\{
\left\{  u,v\right\}  ,\left\{  v,w\right\}  ,\left\{  u,w\right\}  \right\}
$. Now, $\left\{  u,w\right\}  \in\left\{  \left\{  u,v\right\}  ,\left\{
v,w\right\}  ,\left\{  u,w\right\}  \right\}  =D$. Thus, $\left\{
u,w\right\}  $ is an edge in $D$.
\par
Now, we need to show that the edge $\left\{  u,w\right\}  $ is the unique edge
in $D$ having maximum label (among the edges in $D$). Indeed, assume the
contrary (for the sake of contradiction). Thus, $\left\{  u,w\right\}  $ is
\textbf{not} the unique edge in $D$ having maximum label (among the edges in
$D$). In other words, there exists an edge $e\in D$ distinct from $\left\{
u,w\right\}  $ such that the label of $e$ is greater or equal to the label of
$\left\{  u,w\right\}  $ (because we already know that $\left\{  u,w\right\}
$ is an edge in $D$). Consider such an $e$.
\par
The label of $e$ is greater or equal to the label of $\left\{  u,w\right\}  $.
In other words, $\ell\left(  e\right)  $ is greater or equal to $\ell\left(
\left\{  u,w\right\}  \right)  $ (since the label of $e$ is $\ell\left(
e\right)  $ (by the definition of the \textquotedblleft
label\textquotedblright) and since the label of $\left\{  u,w\right\}  $ is
$\ell\left(  \left\{  u,w\right\}  \right)  $ (by the definition of the
\textquotedblleft label\textquotedblright)). In other words, $\ell\left(
e\right)  \geq\ell\left(  \left\{  u,w\right\}  \right)  $.
\par
We have $e\in D$ and $e\notin\left\{  \left\{  u,w\right\}  \right\}  $ (since
$e$ is distinct from $\left\{  u,w\right\}  $). Thus, $e\in D\setminus\left\{
\left\{  u,w\right\}  \right\}  =\left\{  \left\{  u,v\right\}  ,\left\{
v,w\right\}  \right\}  $.
\par
But $\ell\left(  e\right)  \geq\ell\left(  \left\{  u,w\right\}  \right)
>\ell\left(  \left\{  u,v\right\}  \right)  $, so that $\ell\left(  e\right)
\neq\ell\left(  \left\{  u,v\right\}  \right)  $ and thus $e\neq\left\{
u,v\right\}  $. Also, $\ell\left(  e\right)  \geq\ell\left(  \left\{
u,w\right\}  \right)  >\ell\left(  \left\{  v,w\right\}  \right)  $, so that
$\ell\left(  e\right)  \neq\ell\left(  \left\{  v,w\right\}  \right)  $ and
thus $e\neq\left\{  v,w\right\}  $. Thus, $e$ equals neither of the two edges
$\left\{  u,v\right\}  $ and $\left\{  v,w\right\}  $ (since $e\neq\left\{
u,v\right\}  $ and $e\neq\left\{  v,w\right\}  $). In other words,
$e\notin\left\{  \left\{  u,v\right\}  ,\left\{  v,w\right\}  \right\}  $.
This contradicts $e\in\left\{  \left\{  u,v\right\}  ,\left\{  v,w\right\}
\right\}  $. This contradiction shows that our assumption was wrong. Hence, we
have shown that the edge $\left\{  u,w\right\}  $ is the unique edge in $D$
having maximum label (among the edges in $D$). Qed.}.

Also, $K=\left\{  \left\{  u,v\right\}  ,\left\{  v,w\right\}  \right\}
=\underbrace{\left\{  \left\{  u,v\right\}  \right\}  }_{\substack{\subseteq
E\\\text{(since }\left\{  u,v\right\}  \in E\text{)}}}\cup\underbrace{\left\{
\left\{  v,w\right\}  \right\}  }_{\substack{\subseteq E\\\text{(since
}\left\{  v,w\right\}  \in E\text{)}}}\subseteq E\cup E=E$. Thus, $K$ is a
subset of $E$.

Altogether, we have now shown that $K$ is a subset of $E$ and satisfies
$K=D\setminus\left\{  \left\{  u,w\right\}  \right\}  $; we also know that $D$
is a circuit of $G$, and that $\left\{  u,w\right\}  $ is the unique edge in
$D$ having maximum label (among the edges in $D$). Hence, $K$ is a subset of
$E$ having the form $C\setminus\left\{  e\right\}  $, where $C$ is a circuit
of $G$, and where $e$ is the unique edge in $C$ having maximum label (among
the edges in $C$)\ \ \ \ \footnote{Namely, $K$ has this form for $C=D$ and
$e=\left\{  u,w\right\}  $.}. In other words, $K$ is a broken circuit of $G$
(since $K$ is a broken circuit of $G$ if and only if $K$ is a subset of $E$
having the form $C\setminus\left\{  e\right\}  $, where $C$ is a circuit of
$G$, and where $e$ is the unique edge in $C$ having maximum label (among the
edges in $C$)\ \ \ \ \footnote{by the definition of a \textquotedblleft broken
circuit\textquotedblright})).

Now, forget that we fixed $K$. We thus have shown that every $K\in
\mathfrak{K}$ is a broken circuit of $G$. In other words, every element of
$\mathfrak{K}$ is a broken circuit of $G$. This proves Lemma
\ref{lem.digraph.2pf-chrom.K-free} \textbf{(d)}.

\textbf{(e)} The map $\pi:A\rightarrow E$ is bijective (by Lemma
\ref{lem.digraph.2pf-chrom.K-free} \textbf{(a)}). In other words, the map
$\pi$ is surjective and injective.

We first observe a simple fact: If $u$ and $v$ are two elements of $V$ such
that $\left(  u,v\right)  \in A$ and $\left\{  u,v\right\}  \in\pi\left(
F\right)  $, then%
\begin{equation}
\left(  u,v\right)  \in F \label{pf.lem.digraph.2pf-chrom.K-free.e.lem1}%
\end{equation}
\footnote{\textit{Proof of (\ref{pf.lem.digraph.2pf-chrom.K-free.e.lem1}):}
Let $u$ and $v$ be two elements of $V$ such that $\left(  u,v\right)  \in A$
and $\left\{  u,v\right\}  \in\pi\left(  F\right)  $. We must show that
$\left(  u,v\right)  \in F$.
\par
We have $\left\{  u,v\right\}  \in\pi\left(  F\right)  $. In other words,
there exists some $f\in F$ such that $\left\{  u,v\right\}  =\pi\left(
f\right)  $. Consider this $f$.
\par
But $\left(  u,v\right)  \in A$; therefore, the definition of $\pi$ yields%
\begin{align*}
\pi\left(  u,v\right)   &  =\operatorname*{set}\left(  u,v\right)  =\left\{
u,v\right\}  \ \ \ \ \ \ \ \ \ \ \left(  \text{by the definition of the map
}\operatorname*{set}\right) \\
&  =\pi\left(  f\right)  .
\end{align*}
Since $\pi$ is injective, this entails that $\left(  u,v\right)  =f\in F$.
This proves (\ref{pf.lem.digraph.2pf-chrom.K-free.e.lem1}).}.

Also, $\pi\left(  \underbrace{F}_{\subseteq A}\right)  \subseteq\pi\left(
A\right)  \subseteq E$. Thus, $\pi\left(  F\right)  $ is a subset of $E$.

Let us now prove the implication%
\begin{equation}
\ \left(  \text{the digraph }\left(  V,F\right)  \text{ is }2\text{-step-free}%
\right)  \ \Longrightarrow\ \left(  \text{the set }\pi\left(  F\right)  \text{
is }\mathfrak{K}\text{-free}\right)  .
\label{pf.lem.digraph.2pf-chrom.K-free.e.imp1}%
\end{equation}

\textit{Proof of (\ref{pf.lem.digraph.2pf-chrom.K-free.e.imp1}):} Assume that
the digraph $\left(  V,F\right)  $ is $2$-step-free. We shall show that the
set $\pi\left(  F\right)  $ is $\mathfrak{K}$-free.

Indeed, let $K\in\mathfrak{K}$ be such that $K\subseteq\pi\left(  F\right)  $.
We shall obtain a contradiction.

The digraph $\left(  V,F\right)  $ is $2$-step-free if and only if there exist
no three elements $u$, $v$ and $w$ of $V$ satisfying $\left(  u,v\right)  \in
F$ and $\left(  v,w\right)  \in F$ (by the definition of \textquotedblleft%
$2$-step-free\textquotedblright). Therefore, there exist no three elements
$u$, $v$ and $w$ of $V$ satisfying $\left(  u,v\right)  \in F$ and $\left(
v,w\right)  \in F$ (since we know that the digraph $\left(  V,F\right)  $ is
$2$-step-free). In other words, if $u$, $v$ and $w$ are three elements of $V$,
then%
\begin{equation}
\left(  \left(  u,v\right)  \in F\text{ and }\left(  v,w\right)  \in F\right)
\text{ is false.} \label{pf.lem.digraph.2pf-chrom.K-free.e.imp1.pf.2}%
\end{equation}

But%
\begin{align*}
K  &  \in\mathfrak{K}=\left\{  \left\{  \left\{  i,k\right\}  ,\left\{
k,j\right\}  \right\}  \ \mid\ \left(  i,k,j\right)  \in Z\right\} \\
&  =\left\{  \left\{  \left\{  u,v\right\}  ,\left\{  v,w\right\}  \right\}
\ \mid\ \left(  u,v,w\right)  \in Z\right\}
\end{align*}
(here, we renamed the index $\left(  i,k,j\right)  $ as $\left(  u,v,w\right)
$). Hence, $K=\left\{  \left\{  u,v\right\}  ,\left\{  v,w\right\}  \right\}
$ for some $\left(  u,v,w\right)  \in Z$. Consider this $\left(  u,v,w\right)
$. We have%
\[
\left(  u,v,w\right)  \in Z=\left\{  \left(  i,k,j\right)  \in V^{3}%
\ \mid\ \left(  i,k\right)  \in A\text{ and }\left(  k,j\right)  \in
A\right\}  .
\]
In other words, $\left(  u,v,w\right)  $ is an $\left(  i,k,j\right)  \in
V^{3}$ satisfying $\left(  i,k\right)  \in A$ and $\left(  k,j\right)  \in A$.
In other words, $\left(  u,v,w\right)  $ is an element of $V^{3}$ and
satisfies $\left(  u,v\right)  \in A$ and $\left(  v,w\right)  \in A$.

Now, $\left(  u,v\right)  \in A$ and $\left\{  u,v\right\}  \in\left\{
\left\{  u,v\right\}  ,\left\{  v,w\right\}  \right\}  =K\subseteq\pi\left(
F\right)  $. Hence, (\ref{pf.lem.digraph.2pf-chrom.K-free.e.lem1}) shows that
$\left(  u,v\right)  \in F$.

Also, $\left(  v,w\right)  \in A$ and $\left\{  v,w\right\}  \in\left\{
\left\{  u,v\right\}  ,\left\{  v,w\right\}  \right\}  =K\subseteq\pi\left(
F\right)  $. Hence, (\ref{pf.lem.digraph.2pf-chrom.K-free.e.lem1}) (applied to
$v$ and $w$ instead of $u$ and $v$) shows that $\left(  v,w\right)  \in F$.
Thus, we have $\left(  \left(  u,v\right)  \in F\text{ and }\left(
v,w\right)  \in F\right)  $.

But (\ref{pf.lem.digraph.2pf-chrom.K-free.e.imp1.pf.2}) shows that $\left(
\left(  u,v\right)  \in F\text{ and }\left(  v,w\right)  \in F\right)  $ is
false. This contradicts the fact that we have $\left(  \left(  u,v\right)  \in
F\text{ and }\left(  v,w\right)  \in F\right)  $.

Now, let us forget that we fixed $K$. We thus have obtained a contradiction
for each $K\in\mathfrak{K}$ satisfying $K\subseteq\pi\left(  F\right)  $.
Hence, there exists no $K\in\mathfrak{K}$ satisfying $K\subseteq\pi\left(
F\right)  $. In other words, the set $\pi\left(  F\right)  $ contains no
$K\in\mathfrak{K}$ as a subset.

The subset $\pi\left(  F\right)  $ of $E$ is $\mathfrak{K}$-free if and only
if $\pi\left(  F\right)  $ contains no $K\in\mathfrak{K}$ as a subset (by the
definition of \textquotedblleft$\mathfrak{K}$-free\textquotedblright). Thus,
the subset $\pi\left(  F\right)  $ of $E$ is $\mathfrak{K}$-free (since
$\pi\left(  F\right)  $ contains no $K\in\mathfrak{K}$ as a subset).

Now, let us forget that we assumed that the digraph $\left(  V,F\right)  $ is
$2$-step-free. We thus have shown that the set $\pi\left(  F\right)  $ is
$\mathfrak{K}$-free if the digraph $\left(  V,F\right)  $ is $2$-step-free. In
other words, we have proven the implication
(\ref{pf.lem.digraph.2pf-chrom.K-free.e.imp1}).

Next, we shall prove the following implication:%
\begin{equation}
\ \left(  \text{the set }\pi\left(  F\right)  \text{ is }\mathfrak{K}%
\text{-free}\right)  \ \Longrightarrow\ \left(  \text{the digraph }\left(
V,F\right)  \text{ is }2\text{-step-free}\right)  .
\label{pf.lem.digraph.2pf-chrom.K-free.e.imp2}%
\end{equation}

\textit{Proof of (\ref{pf.lem.digraph.2pf-chrom.K-free.e.imp2}):} Assume that
the set $\pi\left(  F\right)  $ is $\mathfrak{K}$-free. We shall show that the
digraph $\left(  V,F\right)  $ is $2$-step-free.

Let $u$, $v$ and $w$ be three elements of $V$ satisfying $\left(  u,v\right)
\in F$ and $\left(  v,w\right)  \in F$. We shall derive a contradiction.

Clearly, $\left(  u,v\right)  \in F\subseteq A$ and $\left(  v,w\right)  \in
F\subseteq A$.

The subset $\pi\left(  F\right)  $ of $E$ is $\mathfrak{K}$-free if and only
if $\pi\left(  F\right)  $ contains no $K\in\mathfrak{K}$ as a subset (by the
definition of \textquotedblleft$\mathfrak{K}$-free\textquotedblright). Thus,
$\pi\left(  F\right)  $ contains no $K\in\mathfrak{K}$ as a subset (since the
subset $\pi\left(  F\right)  $ of $E$ is $\mathfrak{K}$-free). In other words,
there exists no $K\in\mathfrak{K}$ such that $K\subseteq\pi\left(  F\right)  $.

But $\left(  u,v,w\right)  $ is an element of $V^{3}$ and satisfies $\left(
u,v\right)  \in A$ and $\left(  v,w\right)  \in A$. In other words, $\left(
u,v,w\right)  $ is an $\left(  i,k,j\right)  \in V^{3}$ satisfying $\left(
i,k\right)  \in A$ and $\left(  k,j\right)  \in A$. Hence,%
\[
\left(  u,v,w\right)  \in\left\{  \left(  i,k,j\right)  \in V^{3}%
\ \mid\ \left(  i,k\right)  \in A\text{ and }\left(  k,j\right)  \in
A\right\}  =Z.
\]
Thus, $\left\{  \left\{  u,v\right\}  ,\left\{  v,w\right\}  \right\}  $ is a
set of the form $\left\{  \left\{  i,k\right\}  ,\left\{  k,j\right\}
\right\}  $ for some $\left(  i,k,j\right)  \in Z$ (namely, for $\left(
i,k,j\right)  =\left(  u,v,w\right)  $). Hence,%
\begin{equation}
\left\{  \left\{  u,v\right\}  ,\left\{  v,w\right\}  \right\}  \in\left\{
\left\{  \left\{  i,k\right\}  ,\left\{  k,j\right\}  \right\}  \ \mid
\ \left(  i,k,j\right)  \in Z\right\}  =\mathfrak{K}.
\label{pf.lem.digraph.2pf-chrom.K-free.e.imp2.pf.2}%
\end{equation}

But $\left\{  u,v\right\}  \in\pi\left(  F\right)  $%
\ \ \ \ \footnote{\textit{Proof.} We have $\left(  u,v\right)  \in A$;
therefore, the definition of $\pi$ yields%
\[
\pi\left(  u,v\right)  =\operatorname*{set}\left(  u,v\right)  =\left\{
u,v\right\}  \ \ \ \ \ \ \ \ \ \ \left(  \text{by the definition of the map
}\operatorname*{set}\right)  .
\]
Hence, $\left\{  u,v\right\}  =\pi\underbrace{\left(  u,v\right)  }_{\in F}%
\in\pi\left(  F\right)  $, qed.} and $\left\{  v,w\right\}  \in\pi\left(
F\right)  $\ \ \ \ \footnote{\textit{Proof.} We have $\left(  v,w\right)  \in
A$; therefore, the definition of $\pi$ yields%
\[
\pi\left(  v,w\right)  =\operatorname*{set}\left(  v,w\right)  =\left\{
v,w\right\}  \ \ \ \ \ \ \ \ \ \ \left(  \text{by the definition of the map
}\operatorname*{set}\right)  .
\]
Hence, $\left\{  v,w\right\}  =\pi\underbrace{\left(  v,w\right)  }_{\in F}%
\in\pi\left(  F\right)  $, qed.}. Thus, both $\left\{  u,v\right\}  $ and
$\left\{  v,w\right\}  $ belong to the set $\pi\left(  F\right)  $. Therefore,%
\[
\left\{  \left\{  u,v\right\}  ,\left\{  v,w\right\}  \right\}  \subseteq
\pi\left(  F\right)  .
\]
Combining this with (\ref{pf.lem.digraph.2pf-chrom.K-free.e.imp2.pf.2}), we
see that the set $\left\{  \left\{  u,v\right\}  ,\left\{  v,w\right\}
\right\}  \in\mathfrak{K}$ satisfies $\left\{  \left\{  u,v\right\}  ,\left\{
v,w\right\}  \right\}  \subseteq\pi\left(  F\right)  $. Thus, there exists an
$K\in\mathfrak{K}$ such that $K\subseteq\pi\left(  F\right)  $ (namely,
$K=\left\{  \left\{  u,v\right\}  ,\left\{  v,w\right\}  \right\}  $). This
contradicts the fact that there exists no $K\in\mathfrak{K}$ such that
$K\subseteq\pi\left(  F\right)  $.

Now, forget that we fixed $u$, $v$ and $w$. We thus have obtained a
contradiction for every three elements $u$, $v$ and $w$ of $V$ satisfying
$\left(  u,v\right)  \in F$ and $\left(  v,w\right)  \in F$. Hence, there
exist no three elements $u$, $v$ and $w$ of $V$ satisfying $\left(
u,v\right)  \in F$ and $\left(  v,w\right)  \in F$.

But the digraph $\left(  V,F\right)  $ is $2$-step-free if and only if there
exist no three elements $u$, $v$ and $w$ of $V$ satisfying $\left(
u,v\right)  \in F$ and $\left(  v,w\right)  \in F$ (by the definition of
\textquotedblleft$2$-step-free\textquotedblright). Thus, the digraph $\left(
V,F\right)  $ is $2$-step-free (since there exist no three elements $u$, $v$
and $w$ of $V$ satisfying $\left(  u,v\right)  \in F$ and $\left(  v,w\right)
\in F$).

Now, let us forget that we have assumed that the set $\pi\left(  F\right)  $
is $\mathfrak{K}$-free. We thus have shown that if the set $\pi\left(
F\right)  $ is $\mathfrak{K}$-free, then the digraph $\left(  V,F\right)  $ is
$2$-step-free. In other words, we have proven the implication
(\ref{pf.lem.digraph.2pf-chrom.K-free.e.imp2}).

Now, by combining the two implications
(\ref{pf.lem.digraph.2pf-chrom.K-free.e.imp1}) and
(\ref{pf.lem.digraph.2pf-chrom.K-free.e.imp2}), we obtain the logical
equivalence%
\[
\ \left(  \text{the digraph }\left(  V,F\right)  \text{ is }2\text{-step-free}%
\right)  \ \Longleftrightarrow\ \left(  \text{the set }\pi\left(  F\right)
\text{ is }\mathfrak{K}\text{-free}\right)  .
\]
This proves Lemma \ref{lem.digraph.2pf-chrom.K-free} \textbf{(e)}.
\end{proof}

\begin{proof}
[Proof of Proposition \ref{prop.digraph.2pf-chrom}.]First, let us introduce a
general notation: If $X$ and $Y$ are two sets, and if $f:X\rightarrow Y$ is
any map, then we define a map $\mathcal{P}\left(  f\right)  :\mathcal{P}%
\left(  X\right)  \rightarrow\mathcal{P}\left(  Y\right)  $ by%
\[
\left(  \left(  \mathcal{P}\left(  f\right)  \right)  \left(  Z\right)
=f\left(  Z\right)  \ \ \ \ \ \ \ \ \ \ \text{for every }Z\in\mathcal{P}%
\left(  X\right)  \right)  .
\]
If the map $f:X\rightarrow Y$ is bijective, then%
\begin{equation}
\text{the map }\mathcal{P}\left(  f\right)  :\mathcal{P}\left(  X\right)
\rightarrow\mathcal{P}\left(  Y\right)  \text{ is bijective as well}
\label{pf.prop.digraph.2pf-chrom.P(f)bij}%
\end{equation}
(and, in fact, its inverse is $\mathcal{P}\left(  f^{-1}\right)  $).

Let $E=\operatorname*{set}A$ and $G=\underline{D}$. The definition of
$\underline{D}$ shows that $\underline{D}=\left(
V,\underbrace{\operatorname*{set}A}_{=E}\right)  =\left(  V,E\right)  $. Thus,
$G=\underline{D}=\left(  V,E\right)  $.

Define the map $\pi:A\rightarrow E$ as in Lemma
\ref{lem.digraph.2pf-chrom.K-free}. Lemma \ref{lem.digraph.2pf-chrom.K-free}
\textbf{(a)} shows that this map $\pi:A\rightarrow E$ is bijective. Hence, the
map $\mathcal{P}\left(  \pi\right)  :\mathcal{P}\left(  A\right)
\rightarrow\mathcal{P}\left(  E\right)  $ is bijective (according to
(\ref{pf.prop.digraph.2pf-chrom.P(f)bij}), applied to $X=A$, $Y=E$ and $f=\pi
$). We notice that every $F\in\mathcal{P}\left(  A\right)  $ satisfies%
\begin{align}
\left(  \mathcal{P}\left(  \pi\right)  \right)  \left(  F\right)   &
=\pi\left(  F\right)  \ \ \ \ \ \ \ \ \ \ \left(  \text{by the definition of
the map }\mathcal{P}\left(  \pi\right)  \right)
\label{pf.prop.digraph.2pf-chrom.1}\\
&  =\left\{  \underbrace{\pi\left(  a\right)  }%
_{\substack{=\operatorname*{set}a\\\text{(by the definition of }\pi\text{)}%
}}\ \mid\ a\in F\right\}  =\left\{  \operatorname*{set}a\ \mid\ a\in F\right\}
\nonumber\\
&  =\operatorname*{set}F \label{pf.prop.digraph.2pf-chrom.2}%
\end{align}
and%
\begin{equation}
\left\vert \underbrace{\left(  \mathcal{P}\left(  \pi\right)  \right)  \left(
F\right)  }_{\substack{=\pi\left(  F\right)  \\\text{(by
(\ref{pf.prop.digraph.2pf-chrom.1}))}}}\right\vert =\left\vert \pi\left(
F\right)  \right\vert =\left\vert F\right\vert
\label{pf.prop.digraph.2pf-chrom.4}%
\end{equation}
(since the map $\pi$ is injective (since the map $\pi$ is bijective)).

Define two sets $Z$ and $\mathfrak{K}$ as in Lemma
\ref{lem.digraph.2pf-chrom.K-free} \textbf{(b)}. Define a map $\ell
:E\rightarrow\mathbb{N}$ as in Lemma \ref{lem.digraph.2pf-chrom.K-free}
\textbf{(c)}. Definition \ref{def.BC} (applied to $X=\mathbb{N}$) shows that
the notion of a broken circuit of $G$ is well-defined (since a labeling
function $\ell:E\rightarrow\mathbb{N}$ is given).

Lemma \ref{lem.digraph.2pf-chrom.K-free} \textbf{(d)} shows that every element
of $\mathfrak{K}$ is a broken circuit of $G$. Thus, $\mathfrak{K}$ is a set of
broken circuits of $G$ (not necessarily containing all of them). Hence,
Corollary \ref{cor.chrompol.K-free} (applied to $X=\mathbb{N}$) shows that%
\begin{align*}
\chi_{G}  &  =\underbrace{\sum_{\substack{F\subseteq E;\\F\text{ is
}\mathfrak{K}\text{-free}}}}_{=\sum_{\substack{F\in\mathcal{P}\left(
E\right)  ;\\F\text{ is }\mathfrak{K}\text{-free}}}}\left(  -1\right)
^{\left\vert F\right\vert }x^{\operatorname*{conn}\left(  V,F\right)  }%
=\sum_{\substack{F\in\mathcal{P}\left(  E\right)  ;\\F\text{ is }%
\mathfrak{K}\text{-free}}}\left(  -1\right)  ^{\left\vert F\right\vert
}x^{\operatorname*{conn}\left(  V,F\right)  }\\
&  =\underbrace{\sum_{\substack{F\in\mathcal{P}\left(  A\right)  ;\\\left(
\mathcal{P}\left(  \pi\right)  \right)  \left(  F\right)  \text{ is
}\mathfrak{K}\text{-free}}}}_{\substack{=\sum_{\substack{F\in\mathcal{P}%
\left(  A\right)  ;\\\pi\left(  F\right)  \text{ is }\mathfrak{K}\text{-free}%
}}\\\text{(since every }F\in\mathcal{P}\left(  A\right)  \text{ satisfies}%
\\\left(  \mathcal{P}\left(  \pi\right)  \right)  \left(  F\right)
=\pi\left(  F\right)  \\\text{(by (\ref{pf.prop.digraph.2pf-chrom.1})))}%
}}\underbrace{\left(  -1\right)  ^{\left\vert \left(  \mathcal{P}\left(
\pi\right)  \right)  \left(  F\right)  \right\vert }}_{\substack{=\left(
-1\right)  ^{\left\vert F\right\vert }\\\text{(since }\left\vert \left(
\mathcal{P}\left(  \pi\right)  \right)  \left(  F\right)  \right\vert
=\left\vert F\right\vert \\\text{(by (\ref{pf.prop.digraph.2pf-chrom.4})))}%
}}\underbrace{x^{\operatorname*{conn}\left(  V,\left(  \mathcal{P}\left(
\pi\right)  \right)  \left(  F\right)  \right)  }}%
_{\substack{=x^{\operatorname*{conn}\left(  V,\operatorname*{set}F\right)
}\\\text{(since }\left(  \mathcal{P}\left(  \pi\right)  \right)  \left(
F\right)  =\operatorname*{set}F\\\text{(by (\ref{pf.prop.digraph.2pf-chrom.2}%
)))}}}\\
&  \ \ \ \ \ \ \ \ \ \ \left(
\begin{array}
[c]{c}%
\text{here, we have substituted }\left(  \mathcal{P}\left(  \pi\right)
\right)  \left(  F\right)  \text{ for }F\text{ in the sum,}\\
\text{since the map }\mathcal{P}\left(  \pi\right)  :\mathcal{P}\left(
A\right)  \rightarrow\mathcal{P}\left(  E\right)  \text{ is bijective}%
\end{array}
\right) \\
&  =\underbrace{\sum_{\substack{F\in\mathcal{P}\left(  A\right)  ;\\\pi\left(
F\right)  \text{ is }\mathfrak{K}\text{-free}}}}_{=\sum_{\substack{F\subseteq
A;\\\pi\left(  F\right)  \text{ is }\mathfrak{K}\text{-free}}}=\sum
_{\substack{F\subseteq A;\\\text{the set }\pi\left(  F\right)  \text{ is
}\mathfrak{K}\text{-free}}}}\left(  -1\right)  ^{\left\vert F\right\vert
}x^{\operatorname*{conn}\left(  V,\operatorname*{set}F\right)  }\\
&  =\underbrace{\sum_{\substack{F\subseteq A;\\\text{the set }\pi\left(
F\right)  \text{ is }\mathfrak{K}\text{-free}}}}_{\substack{=\sum
_{\substack{F\subseteq A;\\\text{the digraph }\left(  V,F\right)  \text{ is
}2\text{-step-free}}}\\\text{(because for any subset }F\text{ of }A\text{, the
condition}\\\left(  \text{the set }\pi\left(  F\right)  \text{ is
}\mathfrak{K}\text{-free}\right)  \text{ is equivalent to the}%
\\\text{condition }\left(  \text{the digraph }\left(  V,F\right)  \text{ is
}2\text{-step-free}\right)  \\\text{(by Lemma
\ref{lem.digraph.2pf-chrom.K-free} \textbf{(e)}))}}}\left(  -1\right)
^{\left\vert F\right\vert }x^{\operatorname*{conn}\left(
V,\operatorname*{set}F\right)  }\\
&  =\sum_{\substack{F\subseteq A;\\\text{the digraph }\left(  V,F\right)
\text{ is }2\text{-step-free}}}\left(  -1\right)  ^{\left\vert F\right\vert
}x^{\operatorname*{conn}\left(  V,\operatorname*{set}F\right)  }.
\end{align*}
Since $\underline{D}=G$, this rewrites as
\[
\chi_{\underline{D}}=\sum_{\substack{F\subseteq A;\\\text{the digraph }\left(
V,F\right)  \text{ is }2\text{-step-free}}}\left(  -1\right)  ^{\left\vert
F\right\vert }x^{\operatorname*{conn}\left(  V,\operatorname*{set}F\right)
}.
\]
This proves Proposition \ref{prop.digraph.2pf-chrom}.
\end{proof}
\end{verlong}

\section{\label{sec.ambi}Ambigraphs}

\subsection{Definitions of ambigraphs and proper colorings}

We now move on to study various generalizations of the chromatic symmetric function.

The first generalization replaces the finite graph $G$ by what we call an
\emph{ambigraph} (short for \textquotedblleft ambiguous
graph\textquotedblright). To our knowledge, this is a new notion, but it
serves to unify two rather well-known concepts:

\begin{itemize}
\item that of a \emph{multigraph} (see \cite[Definition 6.1.1]{21f-lec6}),
which is like a graph but allows for multiple parallel edges\footnote{More
precisely, our notion of an ambigraph generalizes \emph{loopless} multigraphs,
i.e., multigraphs with no loops. Loops would be a trivial but technically
awkward distraction in the study of chromatic polynomials, so we prefer to
leave them out of our notions of graphs.};

\item that of a \emph{hypergraph} (see \cite[Chapter 17]{Berge73}), which is
like a graph but allows its \textquotedblleft edges\textquotedblright\ to have
any number of endpoints instead of two.
\end{itemize}

In both of these settings, chromatic polynomials have been defined long ago
(for multigraphs perhaps since the introduction of the
concept\footnote{Authors often leave it vague whether their graphs are simple
graphs or multigraphs.}; for hypergraphs since Dohmen's \cite{Dohmen95}), and
it is fairly straightforward to define chromatic symmetric functions at the
same levels of generality. However, we shall instead define them for
\emph{ambigraphs}, a concept which we now introduce:

\begin{definition}
\label{def.ambigraph}\textbf{(a)} An \emph{ambigraph} shall mean a triple
$\left(  V,E,\varphi\right)  $, where $V$ and $E$ are two sets, and where
$\varphi:E\rightarrow\mathcal{P}\left(  \dbinom{V}{2}\right)  $ is a map.
(Thus, the map $\varphi$ sends each $e\in E$ to a set of $2$-element subsets
of $V$.)

\textbf{(b)} An ambigraph $\left(  V,E,\varphi\right)  $ is said to be
\emph{finite} if $V$ and $E$ are finite.

\textbf{(c)} Let $G=\left(  V,E,\varphi\right)  $ be an ambigraph. Then, the
elements of $V$ are called the \emph{vertices} of $G$, whereas the elements of
$E$ are called the \emph{edgeries} of $G$. If $e\in E$ is any edgery, then the
elements of $\varphi\left(  e\right)  $ are called the \emph{edges} of $e$.
Note that these edges are $2$-element subsets of $V$.

\textbf{(d)} Let $G=\left(  V,E,\varphi\right)  $ be an ambigraph. An edgery
$e\in E$ is said to be \emph{singleton} if it has exactly one edge (i.e., if
$\left\vert \varphi\left(  e\right)  \right\vert =1$).
\end{definition}

We view an ambigraph $\left(  V,E,\varphi\right)  $ as something akin to a
graph, except that instead of having edges, it has edgeries -- i.e., packages
of edges. (This can be equivalently viewed as an edge-colored graph, but we
eschew such an interpretation as we shall be using colors for other purposes.)

\begin{example}
\label{exa.ambigraph.1}Let $V$ be the set $\left\{  1,2,3,4,5\right\}  $. Let
$E$ be the $6$-element set $\left\{  e_{1},e_{2},e_{3},e_{4},e_{5}%
,e_{6}\right\}  $. Let $\varphi:E\rightarrow\mathcal{P}\left(  \dbinom{V}%
{2}\right)  $ be the map defined as follows:%
\begin{align*}
\varphi\left(  e_{1}\right)   &  =\left\{  \left\{  1,3\right\}  ,\ \left\{
2,5\right\}  \right\}  ,\\
\varphi\left(  e_{2}\right)   &  =\left\{  \left\{  1,2\right\}  ,\ \left\{
2,3\right\}  ,\ \left\{  3,4\right\}  \right\}  ,\\
\varphi\left(  e_{3}\right)   &  =\left\{  \left\{  2,5\right\}  \right\}  ,\\
\varphi\left(  e_{4}\right)   &  =\left\{  \left\{  1,3\right\}  ,\ \left\{
2,5\right\}  \right\}  ,\\
\varphi\left(  e_{5}\right)   &  =\left\{  {}\right\}  =\varnothing,\\
\varphi\left(  e_{6}\right)   &  =\left\{  \left\{  2,3\right\}  ,\ \left\{
3,4\right\}  \right\}  .
\end{align*}
Let $G$ be the triple $\left(  V,E,\varphi\right)  $. Then, $G$ is an
ambigraph. Its edgeries are $e_{1},e_{2},e_{3},e_{4},e_{5},e_{6}$. The edgery
$e_{3}$ is singleton, while the other edgeries are not. The edgeries $e_{1}$
and $e_{4}$ contain the same edges, namely $\left\{  1,3\right\}  $ and
$\left\{  2,5\right\}  $.
\end{example}

Both multigraphs and hypergraphs can now be encoded as ambigraphs:

\begin{itemize}
\item A multigraph can be viewed as an ambigraph whose all edgeries are
singleton\footnote{To be more precise, this is true for \emph{loopless}
multigraphs (i.e., multigraphs that have no loops). Loops can be encoded as
edgeries that have no edges.}.

\item A hypergraph can be encoded as an ambigraph by replacing each edge
$\left\{  v_{1},v_{2},\ldots,v_{k}\right\}  $ with an edgery consisting of all
edges $\left\{  v_{i},v_{j}\right\}  $ with $i<j$. (Note that this encoding
turns $1$-element edges into empty edgeries\footnote{i.e., edgeries that have
no edges}. Empty edgeries trivialize most of our results, but do not
invalidate any of our proofs, so we have no reason to exclude them.)
\end{itemize}

We can now define $X$-colorings and proper $X$-colorings for ambigraphs:

\begin{definition}
Let $G=\left(  V,E,\varphi\right)  $ be an ambigraph. Let $X$ be a set.

\textbf{(a)} An $X$\emph{-coloring} of $G$ is defined to mean a map
$V\rightarrow X$.

\textbf{(b)} If $f:V\rightarrow X$ is an $X$-coloring of $G$, and if $\left\{
s,t\right\}  $ is a $2$-element subset of $V$, then this subset $\left\{
s,t\right\}  $ is said to be $f$\emph{-dichromatic} if $f\left(  s\right)
\neq f\left(  t\right)  $.

\textbf{(c)} An $X$-coloring $f$ of $G$ is said to be \emph{proper} if each
edgery $e\in E$ has at least one $f$-dichromatic edge (i.e., for each edgery
$e\in E$, there exists at least one edge $\left\{  s,t\right\}  \in
\varphi\left(  e\right)  $ satisfying $f\left(  s\right)  \neq f\left(
t\right)  $).
\end{definition}

\begin{example}
\label{exa.ambigraph.coloring.1}Let $G=\left(  V,E,\varphi\right)  $ be the
ambigraph from Example \ref{exa.ambigraph.1}. Then, $G$ has no proper
$X$-coloring for any $X$, since the edgery $e_{5}$ will never have an
$f$-dichromatic edge, no matter what $f$ is (because $e_{5}$ has no edge to
begin with).

However, let us now modify $\varphi$ by replacing $\varphi\left(
e_{5}\right)  $ by the set $\left\{  \left\{  1,4\right\}  ,\ \left\{
2,4\right\}  ,\ \left\{  3,4\right\}  \right\}  $. Then, for example, the
$X$-coloring $f:V\rightarrow\left\{  1,2,3,4\right\}  $ given by%
\begin{align*}
f\left(  1\right)   &  =1,\ \ \ \ \ \ \ \ \ \ f\left(  2\right)
=2,\ \ \ \ \ \ \ \ \ \ f\left(  3\right)  =1,\\
f\left(  4\right)   &  =1,\ \ \ \ \ \ \ \ \ \ f\left(  5\right)  =3
\end{align*}
is proper. For instance, the edgery $e_{1}$ has the $f$-dichromatic edge
$\left\{  2,5\right\}  $, whereas the edgery $e_{2}$ has the two
$f$-dichromatic edges $\left\{  1,2\right\}  $ and $\left\{  2,3\right\}  $.
On the other hand, the $X$-coloring $f:V\rightarrow\left\{  1,2,3,4\right\}  $
given by%
\begin{align*}
f\left(  1\right)   &  =1,\ \ \ \ \ \ \ \ \ \ f\left(  2\right)
=2,\ \ \ \ \ \ \ \ \ \ f\left(  3\right)  =2,\\
f\left(  4\right)   &  =2,\ \ \ \ \ \ \ \ \ \ f\left(  5\right)  =3
\end{align*}
is not proper, since the edgery $e_{6}$ has no $f$-dichromatic edge.
\end{example}

\begin{example}
\label{exa.ambigraph.coloring.2}Let $G=\left(  V,E,\varphi\right)  $ be the
ambigraph with $V=\left\{  1,2,3,4\right\}  $ and $E=\left\{  a,b\right\}  $
and%
\[
\varphi\left(  a\right)  =\left\{  \left\{  2,3\right\}  \right\}
\ \ \ \ \ \ \ \ \ \ \text{and}\ \ \ \ \ \ \ \ \ \ \varphi\left(  b\right)
=\left\{  \left\{  1,2\right\}  ,\ \left\{  3,4\right\}  \right\}  .
\]
Let $X$ be a set. Then, a map $f:V\rightarrow X$ is a proper $X$-coloring of
$G$ if and only if it satisfies%
\[
f\left(  2\right)  \neq f\left(  3\right)  \ \ \ \ \ \ \ \ \ \ \text{and}%
\ \ \ \ \ \ \ \ \ \ \left(  f\left(  1\right)  \neq f\left(  2\right)  \text{
or }f\left(  3\right)  \neq f\left(  4\right)  \right)  .
\]
Indeed, the statement \textquotedblleft$f\left(  2\right)  \neq f\left(
3\right)  $\textquotedblright\ is saying that the edgery $a$ has an
$f$-dichromatic edge, whereas the statement \textquotedblleft$f\left(
1\right)  \neq f\left(  2\right)  $ or $f\left(  3\right)  \neq f\left(
4\right)  $\textquotedblright\ is saying that the edgery $b$ has an
$f$-dichromatic edge.
\end{example}

As Example \ref{exa.ambigraph.coloring.2} illustrates, the condition on an
$X$-coloring of $G$ to be proper is a conjunction of disjunctions of
inequalities of the form $f\left(  v\right)  \neq f\left(  w\right)  $ for
$\left(  v,w\right)  \in V^{2}$.

\begin{remark}
\label{rmk.ambigraph.Gamb}Any graph $G=\left(  V,E\right)  $ can be viewed as
an ambigraph $\left(  V,E,\varphi\right)  $ in a fairly obvious way: viz., by
setting $\varphi\left(  e\right)  =\left\{  e\right\}  $ for each edge $e\in
E$. We shall denote the latter ambigraph by $G^{\operatorname*{amb}}$. The
proper $X$-colorings of this ambigraph $G^{\operatorname*{amb}}$ are precisely
the proper $X$-colorings of the original graph $G$.
\end{remark}

\begin{remark}
\label{rmk.ambigraph.l00p}Let $G=\left(  V,E,\varphi\right)  $ be an
ambigraph, and let $X$ be a set. If there exists an edgery $e\in E$ satisfying
$\varphi\left(  e\right)  =\varnothing$, then there exists no proper
$X$-coloring $f$ of $G$ (since the edgery $e$ will never have an
$f$-dichromatic edge).
\end{remark}

\subsection{The chromatic symmetric function of an ambigraph}

We can now define the chromatic symmetric function of an ambigraph, by
imitating Definition \ref{def.chromsym}:

\begin{definition}
\label{def.ambichromsym}Let $G=\left(  V,E,\varphi\right)  $ be a finite ambigraph.

\textbf{(a)} For every $\mathbb{N}_{+}$-coloring $f:V\rightarrow\mathbb{N}%
_{+}$ of $G$, we let $\mathbf{x}_{f}$ denote the monomial $\prod_{v\in
V}x_{f\left(  v\right)  }$ in the indeterminates $x_{1},x_{2},x_{3},\ldots$.

\textbf{(b)} We define a power series $X_{G}\in\mathbf{k}\left[  \left[
x_{1},x_{2},x_{3},\ldots\right]  \right]  $ by%
\[
X_{G}=\sum_{\substack{f:V\rightarrow\mathbb{N}_{+}\text{ is a}\\\text{proper
}\mathbb{N}_{+}\text{-coloring of }G}}\mathbf{x}_{f}.
\]

This power series $X_{G}$ is called the \emph{chromatic symmetric function} of
$G$.
\end{definition}

\begin{example}
Let $G=\left(  V,E,\varphi\right)  $ be the ambigraph from Example
\ref{exa.ambigraph.coloring.2}. Then,%
\begin{align*}
X_{G}  &  =\sum_{\substack{f:V\rightarrow\mathbb{N}_{+}\text{ is
a}\\\text{proper }\mathbb{N}_{+}\text{-coloring of }G}%
}\ \ \underbrace{\mathbf{x}_{f}}_{=x_{f\left(  1\right)  }x_{f\left(
2\right)  }x_{f\left(  3\right)  }x_{f\left(  4\right)  }}\\
&  =\sum_{\substack{f:V\rightarrow\mathbb{N}_{+}\text{ is a}\\\text{proper
}\mathbb{N}_{+}\text{-coloring of }G}}x_{f\left(  1\right)  }x_{f\left(
2\right)  }x_{f\left(  3\right)  }x_{f\left(  4\right)  }\\
&  =\sum_{\substack{f:\left\{  1,2,3,4\right\}  \rightarrow\mathbb{N}%
_{+};\\f\left(  2\right)  \neq f\left(  3\right)  \text{ and }\left(  f\left(
1\right)  \neq f\left(  2\right)  \text{ or }f\left(  3\right)  \neq f\left(
4\right)  \right)  }}x_{f\left(  1\right)  }x_{f\left(  2\right)  }x_{f\left(
3\right)  }x_{f\left(  4\right)  }%
\end{align*}
(since a map $f:V\rightarrow\mathbb{N}_{+}$ is a proper $\mathbb{N}_{+}%
$-coloring of $G$ if and only if it satisfies $f\left(  2\right)  \neq
f\left(  3\right)  $ and $\left(  f\left(  1\right)  \neq f\left(  2\right)
\text{ or }f\left(  3\right)  \neq f\left(  4\right)  \right)  $). If we
re-encode each map $f:\left\{  1,2,3,4\right\}  \rightarrow\mathbb{N}_{+}$ as
the $4$-tuple $\left(  i,j,k,\ell\right)  =\left(  f\left(  1\right)
,f\left(  2\right)  ,f\left(  3\right)  ,f\left(  4\right)  \right)  $ of its
values, then we can rewrite this equality as%
\[
X_{G}=\sum_{\substack{\left(  i,j,k,\ell\right)  \in\left(  \mathbb{N}%
_{+}\right)  ^{4};\\j\neq k\text{ and }\left(  i\neq j\text{ or }k\neq
\ell\right)  }}x_{i}x_{j}x_{k}x_{\ell}.
\]

\end{example}

\begin{remark}
Let $G=\left(  V,E,\varphi\right)  $ be an ambigraph that has an edgery $e\in
E$ satisfying $\varphi\left(  e\right)  =\varnothing$. Then, there exists no
proper $\mathbb{N}_{+}$-coloring $f$ of $G$ (by Remark
\ref{rmk.ambigraph.l00p}), and thus we have $X_{G}=0$.
\end{remark}

\subsection{The union of a set of edgeries}

An ambigraph $\left(  V,E,\varphi\right)  $ can be transformed into a simple
graph $\left(  V,E^{\prime}\right)  $ by taking the union of some of its
edgeries -- i.e., by setting $E^{\prime}:=\bigcup_{e\in F}\varphi\left(
e\right)  $ for some subset $F$ of $E$. Let us give this construction a name:

\begin{definition}
\label{def.ambigraph.union}Let $G=\left(  V,E,\varphi\right)  $ be an
ambigraph. Let $F$ be a subset of $E$. Then, $\operatorname*{union}F$ shall
denote the subset $\bigcup_{e\in F}\varphi\left(  e\right)  $ of $\dbinom
{V}{2}$. Thus, we obtain a graph $\left(  V,\operatorname*{union}F\right)  $.
\end{definition}

\begin{example}
Let $G=\left(  V,E,\varphi\right)  $ be the ambigraph from Example
\ref{exa.ambigraph.1}. Then,%
\[
\operatorname*{union}\left\{  e_{2},e_{3}\right\}  =\left\{  \left\{
1,2\right\}  ,\ \left\{  2,3\right\}  ,\ \left\{  3,4\right\}  ,\ \left\{
2,5\right\}  \right\}
\]
and%
\[
\operatorname*{union}\left\{  e_{1},e_{2},e_{4}\right\}  =\left\{  \left\{
1,3\right\}  ,\ \left\{  2,5\right\}  ,\ \left\{  1,2\right\}  ,\ \left\{
2,3\right\}  ,\ \left\{  3,4\right\}  \right\}
\]
and $\operatorname*{union}\left\{  {}\right\}  =\varnothing$.
\end{example}

We can use this notion to state our first result about ambigraphs -- an
analogue to Theorem \ref{thm.chromsym.empty}. We shall prove this result at
the end of the next subsection.

\begin{theorem}
\label{thm.ambichromsym.empty}Let $G=\left(  V,E,\varphi\right)  $ be a finite
ambigraph. Then,%
\[
X_{G}=\sum_{F\subseteq E}\left(  -1\right)  ^{\left\vert F\right\vert
}p_{\lambda\left(  V,\operatorname*{union}F\right)  }.
\]
(Here, of course, the pair $\left(  V,\operatorname*{union}F\right)  $ is
regarded as a graph, and the expression $\lambda\left(
V,\operatorname*{union}F\right)  $ is understood according to Definition
\ref{def.connectedness} \textbf{(b)}.)
\end{theorem}

\subsection{Circuits and broken circuits}

Let us now define the notions of cycles, circuits and broken circuits of an ambigraph.

\begin{definition}
\label{def.ambigraph.cycle}Let $G=\left(  V,E,\varphi\right)  $ be an
ambigraph. A \emph{cycle} of $G$ denotes a list
\[
\left(  v_{1},e_{1},v_{2},e_{2},\ldots,v_{m},e_{m},v_{m+1}\right)
\]
with the following properties:

\begin{itemize}
\item The entries $v_{1},v_{2},\ldots,v_{m+1}$ at the odd positions of this
list belong to $V$, whereas the entries $e_{1},e_{2},\ldots,e_{m}$ at its even
positions belong to $E$.

\item We have $m\geq1$.

\item We have $v_{m+1}=v_{1}$.

\item The vertices $v_{1},v_{2},\ldots,v_{m}$ are pairwise distinct.

\item The edgeries $e_{1},e_{2},\ldots,e_{m}$ are pairwise distinct.

\item We have $\left\{  v_{i},v_{i+1}\right\}  \in\varphi\left(  e_{i}\right)
$ for every $i\in\left\{  1,2,\ldots,m\right\}  $.
\end{itemize}

If $\left(  v_{1},e_{1},v_{2},e_{2},\ldots,v_{m},e_{m},v_{m+1}\right)  $ is a
cycle of $G$, then the set $\left\{  e_{1},e_{2},\ldots,e_{m}\right\}  $ is
called a \emph{circuit} of $G$.
\end{definition}

\begin{example}
\label{exa.ambigraph.cycle.1}Let $G=\left(  V,E,\varphi\right)  $ be the
ambigraph from Example \ref{exa.ambigraph.1}. Then, the tuple%
\[
\left(  1,e_{2},2,e_{6},3,e_{4},1\right)
\]
is a cycle of $G$ (chiefly because $\left\{  1,2\right\}  \in\varphi\left(
e_{2}\right)  $ and $\left\{  2,3\right\}  \in\varphi\left(  e_{6}\right)  $
and $\left\{  3,1\right\}  \in\varphi\left(  e_{4}\right)  $). The circuit
corresponding to this cycle is $\left\{  e_{2},e_{6},e_{4}\right\}  $.

The tuple $\left(  1,e_{2},2,e_{6},3,e_{1},1\right)  $ is a cycle of $G$ as
well, and leads to the circuit $\left\{  e_{2},e_{6},e_{1}\right\}  $.

For comparison, the similar-looking tuple $\left(  1,e_{2},2,e_{2}%
,3,e_{4},1\right)  $ is not a cycle, since its edgeries $e_{2},e_{2},e_{4}$
are not distinct.
\end{example}

It is easy to see that any cycle $\left(  v_{1},e_{1},v_{2},e_{2},\ldots
,v_{m},e_{m},v_{m+1}\right)  $ of an ambigraph $G=\left(  V,E,\varphi\right)
$ must have $m\geq2$ (because if it had $m=1$, then we would have $v_{1}%
=v_{2}$ and thus $\left\{  v_{1}\right\}  =\left\{  v_{1},v_{2}\right\}
\in\varphi\left(  e_{1}\right)  $, contradicting the fact that $\varphi\left(
e_{1}\right)  \in\mathcal{P}\left(  \dbinom{V}{2}\right)  $ contains only
$2$-element sets).

\begin{definition}
\label{def.ambigraph.BC}Let $G=\left(  V,E,\varphi\right)  $ be an ambigraph.
Let $X$ be a totally ordered set. Let $\ell:E\rightarrow X$ be a function. We
shall refer to $\ell$ as the \emph{labeling function}. For every edgery $e$ of
$G$, we shall refer to $\ell\left(  e\right)  $ as the \emph{label} of $e$.

A \emph{broken circuit} of $G$ means a subset of $E$ having the form
$C\setminus\left\{  e\right\}  $, where $C$ is a circuit of $G$, and where $e$
is the unique singleton edgery in $C$ having maximum label (among the
singleton edgeries in $C$). Of course, the notion of a broken circuit of $G$
depends on the function $\ell$; however, we suppress the mention of $\ell$ in
our notation, since we will not consider situations where two different $\ell
$'s coexist.
\end{definition}

Thus, if $G$ is an ambigraph with a labeling function $\ell$, then any circuit
$C$ of $G$ gives rise to a broken circuit provided that

\begin{itemize}
\item at least one edgery in $C$ is singleton, and

\item among the singleton edgeries in $C$, only one attains the maximum label.
\end{itemize}

\noindent In all other cases, $C$ does not give rise to a broken circuit.
Notice that two different circuits may give rise to one and the same broken circuit.

\begin{example}
\textbf{(a)} Let $G=\left(  V,E,\varphi\right)  $ be the ambigraph from
Example \ref{exa.ambigraph.1}. Let $X$ and $\ell:E\rightarrow X$ be arbitrary.
Then, the circuit $\left\{  e_{2},e_{6},e_{4}\right\}  $ we found in Example
\ref{exa.ambigraph.cycle.1} does not give rise to a broken circuit, since it
contains no singleton edgery. However, the circuit $\left\{  e_{3}%
,e_{1}\right\}  $ (coming from the cycle $\left(  2,e_{3},5,e_{1},2\right)  $)
does give rise to a broken circuit (namely, $\left\{  e_{1}\right\}  $), since
its unique singleton edgery is $e_{3}$.

\textbf{(b)} For better examples, we can try an ambigraph having more
singleton edgeries. For instance, we can choose some graph $G$ and consider
the corresponding ambigraph $G^{\operatorname*{amb}}$ as defined in Remark
\ref{rmk.ambigraph.Gamb}. Then, the broken circuits of $G^{\operatorname*{amb}%
}$ are precisely the broken circuits of $G$.

\textbf{(c)} Here is another example: Let $G=\left(  V,E,\varphi\right)  $ be
the ambigraph with $V=\left\{  1,2,3,4\right\}  $, $E=\left\{  e_{1}%
,e_{2},e_{3}\right\}  $ and%
\[
\varphi\left(  e_{1}\right)  =\left\{  \left\{  1,2\right\}  \right\}
,\ \ \ \ \ \ \ \ \ \ \varphi\left(  e_{2}\right)  =\left\{  \left\{
2,3\right\}  \right\}  ,\ \ \ \ \ \ \ \ \ \ \varphi\left(  e_{3}\right)
=\left\{  \left\{  3,4\right\}  ,\ \left\{  1,3\right\}  \right\}  .
\]
Let $X$ and $\ell:E\rightarrow X$ be arbitrary. Then, the cycle $\left(
1,e_{1},2,e_{2},3,e_{3},1\right)  $ of $G$ gives rise to the circuit $\left\{
e_{1},e_{2},e_{3}\right\}  $. This circuit gives rise to

\begin{itemize}
\item the broken circuit $\left\{  e_{2},e_{3}\right\}  $ if $\ell\left(
e_{1}\right)  >\ell\left(  e_{2}\right)  $;

\item the broken circuit $\left\{  e_{1},e_{3}\right\}  $ if $\ell\left(
e_{1}\right)  <\ell\left(  e_{2}\right)  $;

\item no broken circuit if $\ell\left(  e_{1}\right)  =\ell\left(
e_{2}\right)  $.
\end{itemize}

\noindent Note that $\ell\left(  e_{3}\right)  $ does not matter, since the
edgery $e_{3}$ is not singleton.
\end{example}

The notion of a broken circuit always depends on a labeling function
$\ell:E\rightarrow X$. Any time we speak about broken circuits, we shall
tacitly understand that the function $\ell:E\rightarrow X$ is used as the
labeling function.

\subsection{The main results for ambigraphs}

We can now generalize Theorem \ref{thm.chromsym.varis} to ambigraphs:

\begin{theorem}
\label{thm.ambichromsym.varis}Let $G=\left(  V,E,\varphi\right)  $ be a finite
ambigraph. Let $X$ be a totally ordered set. Let $\ell:E\rightarrow X$ be a
labeling function. Let $\mathfrak{K}$ be some set of broken circuits of $G$
(not necessarily containing all of them). Let $a_{K}$ be an element of
$\mathbf{k}$ for every $K\in\mathfrak{K}$. Then,%
\[
X_{G}=\sum_{F\subseteq E}\left(  -1\right)  ^{\left\vert F\right\vert }\left(
\prod_{\substack{K\in\mathfrak{K};\\K\subseteq F}}a_{K}\right)  p_{\lambda
\left(  V,\operatorname*{union}F\right)  }.
\]
(Here, of course, the pair $\left(  V,\operatorname*{union}F\right)  $ is
regarded as a graph, and the expression $\lambda\left(
V,\operatorname*{union}F\right)  $ is understood according to Definition
\ref{def.connectedness} \textbf{(b)}.)
\end{theorem}

This theorem generalizes Theorem \ref{thm.chromsym.varis} (in fact, the latter
is easily obtained by applying the former to $G^{\operatorname*{amb}}$ instead
of $G$). Before we prove it, let us first explore some particular cases. Using
Definition \ref{def.K-free}, we can obtain the following consequences of
Theorem \ref{thm.ambichromsym.varis}:

\begin{corollary}
\label{cor.ambichromsym.K-free}Let $G=\left(  V,E,\varphi\right)  $ be a
finite ambigraph. Let $X$ be a totally ordered set. Let $\ell:E\rightarrow X$
be a labeling function. Let $\mathfrak{K}$ be some set of broken circuits of
$G$ (not necessarily containing all of them). Then,%
\[
X_{G}=\sum_{\substack{F\subseteq E;\\F\text{ is }\mathfrak{K}\text{-free}%
}}\left(  -1\right)  ^{\left\vert F\right\vert }p_{\lambda\left(
V,\operatorname*{union}F\right)  }.
\]

\end{corollary}

\begin{corollary}
\label{cor.ambichromsym.NBC}Let $G=\left(  V,E,\varphi\right)  $ be a finite
ambigraph. Let $X$ be a totally ordered set. Let $\ell:E\rightarrow X$ be a
labeling function. Then,%
\[
X_{G}=\sum_{\substack{F\subseteq E;\\F\text{ contains no broken}%
\\\text{circuit of }G\text{ as a subset}}}\left(  -1\right)  ^{\left\vert
F\right\vert }p_{\lambda\left(  V,\operatorname*{union}F\right)  }.
\]

\end{corollary}

\subsection{Proofs}

Our proof of Theorem \ref{thm.ambichromsym.varis} is mostly similar to our
above proof of Theorem \ref{thm.chromsym.varis}, but there are some
complications due to the possibility of non-singleton edgeries.

We shall use the Iverson bracket notation (Definition \ref{def.iverson}). We
begin with a basic cancellation lemma (see, e.g., \cite[Proposition
7.8.10]{mps}):

\begin{lemma}
\label{lem.cancel}Let $S$ be a finite set. Then, $\sum_{I\subseteq S}\left(
-1\right)  ^{\left\vert I\right\vert }=\left[  S=\varnothing\right]  $.
\end{lemma}

In Definition \ref{def.Eqs}, we defined a set $\operatorname*{Eqs}f$ for any
map $f:V\rightarrow X$. This set helped us find the edges of a graph whose
endpoints received the same color under a coloring $f$. We shall now introduce
a similar notion for ambigraphs:

\begin{definition}
\label{def.ambiEQS}Let $G=\left(  V,E,\varphi\right)  $ be an ambigraph. Let
$X$ be a set. Let $f:V\rightarrow X$ be a map. We let $\operatorname*{EQS}%
\left(  G,f\right)  $ denote the subset%
\[
\left\{  e\in E\ \mid\ \varphi\left(  e\right)  \subseteq\operatorname*{Eqs}%
f\right\}
\]
of $E$.
\end{definition}

\begin{example}
Let $G=\left(  V,E,\varphi\right)  $ be the ambigraph with $V=\left\{
1,2,3,4,5,6\right\}  $ and $E=\left\{  a,b,c\right\}  $ and%
\begin{align*}
\varphi\left(  a\right)   &  =\left\{  \left\{  1,3\right\}  ,\ \left\{
2,4\right\}  ,\ \left\{  3,6\right\}  \right\}  ,\\
\varphi\left(  b\right)   &  =\left\{  \left\{  1,3\right\}  ,\ \left\{
2,6\right\}  \right\}  ,\\
\varphi\left(  c\right)   &  =\varnothing.
\end{align*}
Let $X=\mathbb{N}$, and let $f:V\rightarrow X$ be the map that sends
$1,2,3,4,5,6$ to $1,2,1,2,1,2$, respectively. Then,
\[
\operatorname*{Eqs}f=\left\{  \left\{  1,3\right\}  ,\ \left\{  1,5\right\}
,\ \left\{  3,5\right\}  ,\ \left\{  2,4\right\}  ,\ \left\{  2,6\right\}
,\ \left\{  4,6\right\}  \right\}
\]
and $\operatorname*{EQS}\left(  G,f\right)  =\left\{  b,c\right\}  $. Indeed,
we have $b\in\operatorname*{EQS}\left(  G,f\right)  $ since $\varphi\left(
b\right)  =\left\{  \left\{  1,3\right\}  ,\ \left\{  2,6\right\}  \right\}
\subseteq\operatorname*{Eqs}f$, and we have $c\in\operatorname*{EQS}\left(
G,f\right)  $ since $\varphi\left(  c\right)  =\varnothing\subseteq
\operatorname*{Eqs}f$. On the other hand, $a\notin\operatorname*{EQS}\left(
G,f\right)  $ since $\varphi\left(  a\right)  \not \subseteq
\operatorname*{Eqs}f$ (because $\left\{  3,6\right\}  $ belongs to
$\varphi\left(  a\right)  $ but not to $\operatorname*{Eqs}f$).
\end{example}

\begin{remark}
\label{rmk.ambiEQS.rewrite}Let $G=\left(  V,E,\varphi\right)  $ be an
ambigraph. Let $X$ be a set. Let $f:V\rightarrow X$ be a map. The definition
of $\operatorname*{EQS}\left(  G,f\right)  $ yields%
\begin{align}
\operatorname*{EQS}\left(  G,f\right)   &  =\left\{  e\in E\ \mid
\ \varphi\left(  e\right)  \subseteq\operatorname*{Eqs}f\right\}
\label{eq.rmk.ambiEQS.rewrite.1}\\
&  =\left\{  d\in E\ \mid\ \varphi\left(  d\right)  \subseteq
\operatorname*{Eqs}f\right\} \label{eq.rmk.ambiEQS.rewrite.2}\\
&  \ \ \ \ \ \ \ \ \ \ \ \ \ \ \ \ \ \ \ \ \left(  \text{here, we have renamed
the index }e\text{ as }d\right) \nonumber\\
&  =\left\{  d\in E\ \mid\ \text{no edge of }d\text{ is }f\text{-dichromatic}%
\right\}  . \label{eq.rmk.ambiEQS.rewrite.3}%
\end{align}
(The last equality is easy to check from the definitions.)
\end{remark}

\begin{verlong}
\begin{proof}
[Proof of Remark \ref{rmk.ambiEQS.rewrite}.]The equalities
(\ref{eq.rmk.ambiEQS.rewrite.1}) and (\ref{eq.rmk.ambiEQS.rewrite.2}) are
obvious. It remains to prove (\ref{eq.rmk.ambiEQS.rewrite.3}).

Indeed, let $d\in E$. Then, the edges of $d$ are defined to be the elements of
$\varphi\left(  d\right)  $. Thus, the elements of $\varphi\left(  d\right)  $
are precisely the edges of $d$. Hence, it is easy to prove the implication%
\[
\left(  \varphi\left(  d\right)  \subseteq\operatorname*{Eqs}f\right)
\ \Longrightarrow\ \left(  \text{no edge of }d\text{ is }f\text{-dichromatic}%
\right)
\]
\footnote{\textit{Proof.} Assume that $\varphi\left(  d\right)  \subseteq
\operatorname*{Eqs}f$. We must prove that no edge of $d$ is $f$-dichromatic.
\par
Indeed, let $g$ be an edge of $d$. Thus, $g\in\varphi\left(  d\right)  $
(since the edges of $d$ are the elements of $\varphi\left(  d\right)  $).
Hence,
\[
g\in\varphi\left(  d\right)  \subseteq\operatorname*{Eqs}f=\left\{  \left\{
s,t\right\}  \ \mid\ \left(  s,t\right)  \in V^{2},\ s\neq t\text{ and
}f\left(  s\right)  =f\left(  t\right)  \right\}  .
\]
In other words, $g$ can be written as $g=\left\{  s,t\right\}  $ for some pair
$\left(  s,t\right)  \in V^{2}$ satisfying $s\neq t$ and $f\left(  s\right)
=f\left(  t\right)  $. Consider this pair $\left(  s,t\right)  $. Then,
$\left\{  s,t\right\}  $ is a $2$-element subset of $V$ (since $\left(
s,t\right)  \in V^{2}$ and $s\neq t$), and satisfies $f\left(  s\right)
=f\left(  t\right)  $. Hence, this subset $\left\{  s,t\right\}  $ is not
$f$-dichromatic (because if it was $f$-dichromatic, then it would satisfy
$f\left(  s\right)  \neq f\left(  t\right)  $ (by the definition of
\textquotedblleft$f$-dichromatic\textquotedblright), but this would contradict
$f\left(  s\right)  =f\left(  t\right)  $). In other words, $g$ is not
$f$-dichromatic (since $g=\left\{  s,t\right\}  $).
\par
Forget that we fixed $g$. We thus have shown that if $g$ is any edge of $d$,
then $g$ is not $f$-dichromatic. In other words, no edge of $d$ is
$f$-dichromatic. Qed.} and the implication%
\[
\left(  \text{no edge of }d\text{ is }f\text{-dichromatic}\right)
\ \Longrightarrow\ \left(  \varphi\left(  d\right)  \subseteq
\operatorname*{Eqs}f\right)
\]
\footnote{\textit{Proof.} Assume that no edge of $d$ is $f$-dichromatic. We
must prove that $\varphi\left(  d\right)  \subseteq\operatorname*{Eqs}f$.
\par
Indeed, let $g\in\varphi\left(  d\right)  $. Thus, $g$ is an edge of $d$
(since the edges of $d$ are the elements of $\varphi\left(  d\right)  $).
Hence, $g$ is not $f$-dichromatic (since no edge of $d$ is $f$-dichromatic).
\par
However, $g\in\varphi\left(  d\right)  \subseteq\dbinom{V}{2}$. In other
words, $g$ is a $2$-element subset of $V$. Thus, we can write $g$ as
$g=\left\{  x,y\right\}  $ for two distinct elements $x$ and $y$ of $V$.
Consider these $x$ and $y$. Then, $\left(  x,y\right)  \in V^{2}$ and $x\neq
y$ (since $x$ and $y$ are distinct). Thus, $\left\{  x,y\right\}  $ is a
$2$-element subset of $V$.
\par
Recall that $g$ is not $f$-dichromatic. In other words, $\left\{  x,y\right\}
$ is not $f$-dichromatic (since $g=\left\{  x,y\right\}  $).
\par
The $2$-element subset $\left\{  x,y\right\}  $ of $V$ is $f$-dichromatic if
and only if $f\left(  x\right)  \neq f\left(  y\right)  $ (by the definition
of \textquotedblleft$f$-dichromatic\textquotedblright). Hence, we don't have
$f\left(  x\right)  \neq f\left(  y\right)  $ (since $\left\{  x,y\right\}  $
is not $f$-dichromatic). In other words, we have $f\left(  x\right)  =f\left(
y\right)  $.
\par
Hence, $\left\{  x,y\right\}  $ is a set of the form $\left\{  s,t\right\}  $
for some pair $\left(  s,t\right)  \in V^{2}$ satisfying $s\neq t$ and
$f\left(  s\right)  =f\left(  t\right)  $ (namely, for $\left(  s,t\right)
=\left(  x,y\right)  $). In other words,%
\[
g\in\left\{  \left\{  s,t\right\}  \ \mid\ \left(  s,t\right)  \in
V^{2},\ s\neq t\text{ and }f\left(  s\right)  =f\left(  t\right)  \right\}
=\operatorname*{Eqs}f
\]
(since $\operatorname*{Eqs}f$ is defined to be $\left\{  \left\{  s,t\right\}
\ \mid\ \left(  s,t\right)  \in V^{2},\ s\neq t\text{ and }f\left(  s\right)
=f\left(  t\right)  \right\}  $).
\par
Forget that we fixed $g$. We thus have shown that $g\in\operatorname*{Eqs}f$
for each $g\in\varphi\left(  d\right)  $. In other words, $\varphi\left(
d\right)  \subseteq\operatorname*{Eqs}f$. Qed.}. Combining these two
implications, we obtain the equivalence%
\[
\left(  \varphi\left(  d\right)  \subseteq\operatorname*{Eqs}f\right)
\ \Longleftrightarrow\ \left(  \text{no edge of }d\text{ is }%
f\text{-dichromatic}\right)  .
\]

Forget that we fixed $d$. We thus have proved the equivalence%
\[
\left(  \varphi\left(  d\right)  \subseteq\operatorname*{Eqs}f\right)
\ \Longleftrightarrow\ \left(  \text{no edge of }d\text{ is }%
f\text{-dichromatic}\right)
\]
for each $d\in E$. Therefore,%
\[
\left\{  d\in E\ \mid\ \varphi\left(  d\right)  \subseteq\operatorname*{Eqs}%
f\right\}  =\left\{  d\in E\ \mid\ \text{no edge of }d\text{ is }%
f\text{-dichromatic}\right\}  .
\]
Hence, (\ref{eq.rmk.ambiEQS.rewrite.2}) can be rewritten as%
\[
\operatorname*{EQS}\left(  G,f\right)  =\left\{  d\in E\ \mid\ \text{no edge
of }d\text{ is }f\text{-dichromatic}\right\}  .
\]
This proves (\ref{eq.rmk.ambiEQS.rewrite.3}). Thus, the proof of Remark
\ref{rmk.ambiEQS.rewrite} is complete.
\end{proof}
\end{verlong}

In analogy to Lemma \ref{lem.Eqs.proper}, we can use $\operatorname*{EQS}%
\left(  G,f\right)  $ to characterize when an $X$-coloring $f$ is proper:

\begin{lemma}
\label{lem.ambiEqs.proper}Let $G=\left(  V,E,\varphi\right)  $ be an
ambigraph. Let $X$ be a set. Let $f:V\rightarrow X$ be a map. Then, the
$X$-coloring $f$ of $G$ is proper if and only if $\operatorname*{EQS}\left(
G,f\right)  =\varnothing$.
\end{lemma}

\begin{vershort}
\begin{proof}
[Proof of Lemma \ref{lem.ambiEqs.proper}.]An exercise in unfolding definitions
and applying de Morgan's laws.
\end{proof}
\end{vershort}

\begin{verlong}
\begin{proof}
[Proof of Lemma \ref{lem.ambiEqs.proper}.]The definition of
$\operatorname*{Eqs}f$ shows that%
\begin{align}
\operatorname*{Eqs}f  &  =\left\{  \left\{  s,t\right\}  \ \mid\ \left(
s,t\right)  \in V^{2},\ s\neq t\text{ and }f\left(  s\right)  =f\left(
t\right)  \right\} \nonumber\\
&  =\left\{  \left\{  x,y\right\}  \ \mid\ \left(  x,y\right)  \in
V^{2},\ x\neq y\text{ and }f\left(  x\right)  =f\left(  y\right)  \right\}
\label{pf.lem.ambiEqs.proper.0}%
\end{align}
(here, we renamed the index $\left(  s,t\right)  $ as $\left(  x,y\right)  $).

The definition of $\operatorname*{EQS}\left(  G,f\right)  $ yields
$\operatorname*{EQS}\left(  G,f\right)  =\left\{  e\in E\ \mid\ \varphi\left(
e\right)  \subseteq\operatorname*{Eqs}f\right\}  $. Thus, an edgery $e\in E$
belongs to $\operatorname*{EQS}\left(  G,f\right)  $ if and only if it
satisfies $\varphi\left(  e\right)  \subseteq\operatorname*{Eqs}f$. In other
words, for any edgery $e\in E$, we have the logical equivalence%
\begin{equation}
\left(  e\in\operatorname*{EQS}\left(  G,f\right)  \right)
\ \Longleftrightarrow\ \left(  \varphi\left(  e\right)  \subseteq
\operatorname*{Eqs}f\right)  . \label{pf.lem.ambiEqs.proper.EQS}%
\end{equation}

We shall first prove the logical implication%
\begin{equation}
\left(  \text{the }X\text{-coloring }f\text{ of }G\text{ is proper}\right)
\ \Longrightarrow\ \left(  \operatorname*{EQS}\left(  G,f\right)
=\varnothing\right)  . \label{pf.lem.ambiEqs.proper.im1}%
\end{equation}

\textit{Proof of (\ref{pf.lem.ambiEqs.proper.im1}):} Assume that the
$X$-coloring $f$ of $G$ is proper. We must show that $\operatorname*{EQS}%
\left(  G,f\right)  =\varnothing$.

Recall that the $X$-coloring $f$ of $G$ is proper if and only if each edgery
$e\in E$ has at least one $f$-dichromatic edge (by the definition of
\textquotedblleft proper\textquotedblright). Thus, each edgery $e\in E$ has at
least one $f$-dichromatic edge (since the $X$-coloring $f$ of $G$ is proper).

Now, let $d\in\operatorname*{EQS}\left(  G,f\right)  $. Thus,%
\[
d\in\operatorname*{EQS}\left(  G,f\right)  =\left\{  e\in E\ \mid
\ \varphi\left(  e\right)  \subseteq\operatorname*{Eqs}f\right\}
\]
(by the definition of $\operatorname*{EQS}\left(  G,f\right)  $). In other
words, $d$ is an $e\in E$ satisfying $\varphi\left(  e\right)  \subseteq
\operatorname*{Eqs}f$. In other words, $d\in E$ and $\varphi\left(  d\right)
\subseteq\operatorname*{Eqs}f$.

However, recall that each edgery $e\in E$ has at least one $f$-dichromatic
edge. Applying this to $e=d$, we conclude that $d$ has at least one
$f$-dichromatic edge. Let $g$ be this edge. Thus, $g\in\varphi\left(
d\right)  $ (since $g$ is an edge of $d$). Hence,%
\[
g\in\varphi\left(  d\right)  \subseteq\operatorname*{Eqs}f=\left\{  \left\{
x,y\right\}  \ \mid\ \left(  x,y\right)  \in V^{2},\ x\neq y\text{ and
}f\left(  x\right)  =f\left(  y\right)  \right\}
\]
(by (\ref{pf.lem.ambiEqs.proper.0})). In other words, $g$ has the form
$\left\{  x,y\right\}  $ for some pair $\left(  x,y\right)  \in V^{2}$
satisfying $x\neq y$ and $f\left(  x\right)  =f\left(  y\right)  $. Consider
this pair $\left(  x,y\right)  $. Thus, $g=\left\{  x,y\right\}  $.

We have $f\left(  x\right)  =f\left(  y\right)  $. In other words, we don't
have $f\left(  x\right)  \neq f\left(  y\right)  $. Moreover, $\left\{
x,y\right\}  $ is a $2$-element subset of $V$ (since $\left(  x,y\right)  \in
V^{2}$ and $x\neq y$).

Now, recall that the subset $\left\{  x,y\right\}  $ of $V$ is $f$-dichromatic
if and only if $f\left(  x\right)  \neq f\left(  y\right)  $ (by the
definition of \textquotedblleft$f$-dichromatic\textquotedblright). Thus, this
subset $\left\{  x,y\right\}  $ is not $f$-dichromatic (since we don't have
$f\left(  x\right)  \neq f\left(  y\right)  $). In other words, $g$ is not
$f$-dichromatic (since $g=\left\{  x,y\right\}  $). This contradicts the fact
that $g$ is $f$-dichromatic.

Forget that we fixed $d$. We thus have obtained a contradiction for each
$d\in\operatorname*{EQS}\left(  G,f\right)  $. Hence, there exists no
$d\in\operatorname*{EQS}\left(  G,f\right)  $. In other words, the set
$\operatorname*{EQS}\left(  G,f\right)  $ is empty. In other words,
$\operatorname*{EQS}\left(  G,f\right)  =\varnothing$. Thus, the implication
(\ref{pf.lem.ambiEqs.proper.im1}) is proven.

Now, we shall prove the implication%
\begin{equation}
\left(  \operatorname*{EQS}\left(  G,f\right)  =\varnothing\right)
\ \Longrightarrow\ \left(  \text{the }X\text{-coloring }f\text{ of }G\text{ is
proper}\right)  . \label{pf.lem.ambiEqs.proper.im2}%
\end{equation}

\textit{Proof of (\ref{pf.lem.ambiEqs.proper.im2}):} Assume that
$\operatorname*{EQS}\left(  G,f\right)  =\varnothing$. We have to show that
the $X$-coloring $f$ of $G$ is proper.

Let $e\in E$ be an edgery. We shall show that $e$ has at least one
$f$-dichromatic edge.

Indeed, assume the contrary. Thus, $e$ has no $f$-dichromatic edge. In other
words, an edge of $e$ cannot be $f$-dichromatic.

Now, let $g\in\varphi\left(  e\right)  $ be arbitrary. Thus, $g$ is an edge of
$e$. Hence, $g$ is not $f$-dichromatic (since an edge of $e$ cannot be
$f$-dichromatic). But $g\in\varphi\left(  e\right)  \subseteq\dbinom{V}{2}$,
so that $g$ is a $2$-element subset of $V$. In other words, we can write $g$
as $g=\left\{  s,t\right\}  $ for two distinct elements $s$ and $t$ of $V$.
Consider these two elements $s$ and $t$. We have $s\neq t$ (since $s$ and $t$
are distinct).

We have $g=\left\{  s,t\right\}  $. Thus, $\left\{  s,t\right\}  $ is not
$f$-dichromatic (since $g$ is not $f$-dichromatic).

However, $\left\{  s,t\right\}  $ is $f$-dichromatic if and only if $f\left(
s\right)  \neq f\left(  t\right)  $ (by the definition of \textquotedblleft%
$f$-dichromatic\textquotedblright). Hence, we cannot have $f\left(  s\right)
\neq f\left(  t\right)  $ (since $g$ is not $f$-dichromatic). In other words,
we have $f\left(  s\right)  =f\left(  t\right)  $. Hence, $\left\{
s,t\right\}  $ is a set of the form $\left\{  x,y\right\}  $ for some $\left(
x,y\right)  \in V^{2}$ satisfying $x\neq y$ and $f\left(  x\right)  =f\left(
y\right)  $ (namely, for $\left(  x,y\right)  =\left(  s,t\right)  $).

Now,%
\begin{align*}
g  &  =\left\{  s,t\right\} \\
&  \in\left\{  \left\{  x,y\right\}  \ \mid\ \left(  x,y\right)  \in
V^{2},\ x\neq y\text{ and }f\left(  x\right)  =f\left(  y\right)  \right\} \\
&  \ \ \ \ \ \ \ \ \ \ \ \ \ \ \ \ \ \ \ \ \left(
\begin{array}
[c]{c}%
\text{since }\left\{  s,t\right\}  \text{ is a set of the form }\left\{
x,y\right\}  \text{ for}\\
\text{some }\left(  x,y\right)  \in V^{2}\text{ satisfying }x\neq y\text{ and
}f\left(  x\right)  =f\left(  y\right)
\end{array}
\right) \\
&  =\operatorname*{Eqs}f\ \ \ \ \ \ \ \ \ \ \left(  \text{by
(\ref{pf.lem.ambiEqs.proper.0})}\right)  .
\end{align*}

Forget that we fixed $g$. We thus have shown that $g\in\operatorname*{Eqs}f$
for each $g\in\varphi\left(  e\right)  $. In other words, $\varphi\left(
e\right)  \subseteq\operatorname*{Eqs}f$. According to the equivalence
(\ref{pf.lem.ambiEqs.proper.EQS}), this entails that $e\in\operatorname*{EQS}%
\left(  G,f\right)  $ (since $e\in E$). In other words, $e\in\varnothing$
(since $\operatorname*{EQS}\left(  G,f\right)  =\varnothing$). But this is
absurd (since the empty set $\varnothing$ has no elements). This contradiction
shows that our assumption was false. Hence, $e$ has at least one
$f$-dichromatic edge.

Forget that we fixed $e$. We thus have shown that each edgery $e\in E$ has at
least one $f$-dichromatic edge. In other words, the $X$-coloring $f$ of $G$ is
proper (by the definition of \textquotedblleft proper\textquotedblright). This
proves the implication (\ref{pf.lem.ambiEqs.proper.im2}).

Now we have proven the two implications (\ref{pf.lem.ambiEqs.proper.im1}) and
(\ref{pf.lem.ambiEqs.proper.im2}). Combining these two implications, we obtain
the equivalence%
\[
\left(  \text{the }X\text{-coloring }f\text{ of }G\text{ is proper}\right)
\ \Longleftrightarrow\ \left(  \operatorname*{EQS}\left(  G,f\right)
=\varnothing\right)  .
\]
This proves Lemma \ref{lem.ambiEqs.proper}.
\end{proof}
\end{verlong}

The following simple lemma connects $\operatorname*{EQS}\left(  G,f\right)  $
with the $\operatorname*{union}F$ construction from Definition
\ref{def.ambigraph.union}:

\begin{lemma}
\label{lem.ambiEqs.union}Let $G=\left(  V,E,\varphi\right)  $ be an ambigraph.
Let $X$ be a set. Let $f:V\rightarrow X$ be a map. Let $B$ be a subset of $E$.
Then, $B\subseteq\operatorname*{EQS}\left(  G,f\right)  $ holds if and only if
$\operatorname*{union}B\subseteq\operatorname*{Eqs}f$.
\end{lemma}

\begin{proof}
[Proof of Lemma \ref{lem.ambiEqs.union}.]The definition of
$\operatorname*{union}B$ yields $\operatorname*{union}B=\bigcup_{e\in
B}\varphi\left(  e\right)  $. Hence, we have the following chain of logical
equivalences:%
\begin{align}
&  \ \left(  \operatorname*{union}B\subseteq\operatorname*{Eqs}f\right)
\nonumber\\
&  \Longleftrightarrow\ \left(  \bigcup_{e\in B}\varphi\left(  e\right)
\subseteq\operatorname*{Eqs}f\right) \nonumber\\
&  \Longleftrightarrow\ \left(  \varphi\left(  e\right)  \subseteq
\operatorname*{Eqs}f\text{ for each }e\in B\right)  .
\label{pf.lem.ambiEqs.union.0}%
\end{align}

\begin{vershort}
However, for an edgery $e\in E$, the condition $\varphi\left(  e\right)
\subseteq\operatorname*{Eqs}f$ is equivalent to $e\in\operatorname*{EQS}%
\left(  G,f\right)  $ (by the definition of $\operatorname*{EQS}\left(
G,f\right)  $). Hence, we can rewrite the equivalence
(\ref{pf.lem.ambiEqs.union.0}) as follows:%
\begin{align*}
\left(  \operatorname*{union}B\subseteq\operatorname*{Eqs}f\right)  \  &
\Longleftrightarrow\ \left(  e\in\operatorname*{EQS}\left(  G,f\right)  \text{
for each }e\in B\right) \\
&  \Longleftrightarrow\ \left(  B\subseteq\operatorname*{EQS}\left(
G,f\right)  \right)  .
\end{align*}

\end{vershort}

\begin{verlong}
The definition of $\operatorname*{EQS}\left(  G,f\right)  $ yields
$\operatorname*{EQS}\left(  G,f\right)  =\left\{  e\in E\ \mid\ \varphi\left(
e\right)  \subseteq\operatorname*{Eqs}f\right\}  $. Thus, an edgery $e\in E$
belongs to $\operatorname*{EQS}\left(  G,f\right)  $ if and only if it
satisfies $\varphi\left(  e\right)  \subseteq\operatorname*{Eqs}f$. In other
words, for any edgery $e\in E$, we have the logical equivalence%
\begin{equation}
\left(  e\in\operatorname*{EQS}\left(  G,f\right)  \right)
\ \Longleftrightarrow\ \left(  \varphi\left(  e\right)  \subseteq
\operatorname*{Eqs}f\right)  . \label{pf.lem.ambiEqs.union.1}%
\end{equation}

Now, we have the following chain of logical equivalences:%
\begin{align*}
&  \ \left(  B\subseteq\operatorname*{EQS}\left(  G,f\right)  \right) \\
&  \Longleftrightarrow\ \left(  \underbrace{e\in\operatorname*{EQS}\left(
G,f\right)  }_{\substack{\Longleftrightarrow\ \left(  \varphi\left(  e\right)
\subseteq\operatorname*{Eqs}f\right)  \\\text{(by
(\ref{pf.lem.ambiEqs.union.1}))}}}\text{ for each }e\in B\right) \\
&  \Longleftrightarrow\ \left(  \varphi\left(  e\right)  \subseteq
\operatorname*{Eqs}f\text{ for each }e\in B\right) \\
&  \Longleftrightarrow\ \left(  \operatorname*{union}B\subseteq
\operatorname*{Eqs}f\right)  \ \ \ \ \ \ \ \ \ \ \left(  \text{by
(\ref{pf.lem.ambiEqs.union.0})}\right)  .
\end{align*}

\end{verlong}

In other words, $B\subseteq\operatorname*{EQS}\left(  G,f\right)  $ holds if
and only if $\operatorname*{union}B\subseteq\operatorname*{Eqs}f$. This proves
Lemma \ref{lem.ambiEqs.union}.
\end{proof}

Next, let us show an analogue of Lemma \ref{lem.Eqs.circuit}:

\begin{lemma}
\label{lem.ambiEqs.circuit}Let $G=\left(  V,E,\varphi\right)  $ be an
ambigraph. Let $X$ be a set. Let $f:V\rightarrow X$ be a map. Let $C$ be a
circuit of $G$. Let $e\in C$ be a singleton edgery such that $C\setminus
\left\{  e\right\}  \subseteq\operatorname*{EQS}\left(  G,f\right)  $. Then,
$e\in\operatorname*{EQS}\left(  G,f\right)  $.
\end{lemma}

\begin{vershort}
\begin{proof}
[Proof of Lemma \ref{lem.ambiEqs.circuit}.]We have assumed that $e$ is
singleton. In other words, $e$ has exactly one edge. In other words,
$\varphi\left(  e\right)  $ is a $1$-element set.

The set $C$ is a circuit of $G$. Hence, we can write $C$ in the form%
\[
C=\left\{  e_{1},e_{2},\ldots,e_{m}\right\}
\]
for some cycle $\left(  v_{1},e_{1},v_{2},e_{2},\ldots,v_{m},e_{m}%
,v_{m+1}\right)  $ of $G$. Consider this cycle \newline$\left(  v_{1}%
,e_{1},v_{2},e_{2},\ldots,v_{m},e_{m},v_{m+1}\right)  $. According to the
definition of a \textquotedblleft cycle\textquotedblright, the cycle $\left(
v_{1},e_{1},v_{2},e_{2},\ldots,v_{m},e_{m},v_{m+1}\right)  $ is a list having
the following properties:

\begin{itemize}
\item The entries $v_{1},v_{2},\ldots,v_{m+1}$ at the odd positions of this
list belong to $V$, whereas the entries $e_{1},e_{2},\ldots,e_{m}$ at its even
positions belong to $E$.

\item We have $m\geq1$.

\item We have $v_{m+1}=v_{1}$.

\item The vertices $v_{1},v_{2},\ldots,v_{m}$ are pairwise distinct.

\item The edgeries $e_{1},e_{2},\ldots,e_{m}$ are pairwise distinct.

\item We have $\left\{  v_{i},v_{i+1}\right\}  \in\varphi\left(  e_{i}\right)
$ for every $i\in\left\{  1,2,\ldots,m\right\}  $.
\end{itemize}

Recall that $e\in C=\left\{  e_{1},e_{2},\ldots,e_{m}\right\}  $. We can thus
WLOG assume that $e=e_{m}$ (since otherwise, we can simply cyclically relabel
the vertices and the edgeries along our cycle). Assume this. Since $C=\left\{
e_{1},e_{2},\ldots,e_{m}\right\}  $ and $e=e_{m}$, we have%
\[
C\setminus\left\{  e\right\}  =\left\{  e_{1},e_{2},\ldots,e_{m-1}\right\}
\]
(since the $m$ edgeries $e_{1},e_{2},\ldots,e_{m}$ are distinct). For every
$i\in\left\{  1,2,\ldots,m-1\right\}  $, we thus have $e_{i}\in C\setminus
\left\{  e\right\}  \subseteq\operatorname*{EQS}\left(  G,f\right)  $ and
therefore $\varphi\left(  e_{i}\right)  \subseteq\operatorname*{Eqs}f$ (by the
definition of $\operatorname*{EQS}\left(  G,f\right)  $), so that $\left\{
v_{i},v_{i+1}\right\}  \in\varphi\left(  e_{i}\right)  \subseteq
\operatorname*{Eqs}f$ and therefore $f\left(  v_{i}\right)  =f\left(
v_{i+1}\right)  $ (by the definition of $\operatorname*{Eqs}f$). Hence,
$f\left(  v_{1}\right)  =f\left(  v_{2}\right)  =\cdots=f\left(  v_{m}\right)
$, so that $f\left(  v_{m}\right)  =f\left(  v_{1}\right)  $. Thus, $\left\{
v_{m},v_{1}\right\}  \in\operatorname*{Eqs}f$ (since $v_{m}\neq v_{1}$
(because the vertices $v_{1},v_{2},\ldots,v_{m}$ are pairwise distinct, and we
have $m\geq2$)).

However, recall again that $\left\{  v_{i},v_{i+1}\right\}  \in\varphi\left(
e_{i}\right)  $ for every $i\in\left\{  1,2,\ldots,m\right\}  $. Applying this
to $i=m$, we obtain $\left\{  v_{m},v_{m+1}\right\}  \in\varphi\left(
e_{m}\right)  $. Since $v_{m+1}=v_{1}$ and $e_{m}=e$, we can rewrite this as
$\left\{  v_{m},v_{1}\right\}  \in\varphi\left(  e\right)  $. Since
$\varphi\left(  e\right)  $ is a $1$-element set, this entails that
$\varphi\left(  e\right)  =\left\{  \left\{  v_{m},v_{1}\right\}  \right\}
\subseteq\operatorname*{Eqs}f$ (since $\left\{  v_{m},v_{1}\right\}
\in\operatorname*{Eqs}f$). In other words, $e\in\operatorname*{EQS}\left(
G,f\right)  $ (by the definition of $\operatorname*{EQS}\left(  G,f\right)
$). This proves Lemma \ref{lem.ambiEqs.circuit}.
\end{proof}
\end{vershort}

\begin{verlong}
\begin{proof}
[Proof of Lemma \ref{lem.ambiEqs.circuit}.]We have assumed that the edgery $e$
is singleton. In other words, $e$ has exactly one edge (by the definition of
\textquotedblleft singleton\textquotedblright). In other words, $\varphi
\left(  e\right)  $ is a $1$-element set (since the edges of $e$ are defined
to be the elements of $\varphi\left(  e\right)  $).

The set $C$ is a circuit of $G$. In other words, the set $C$ has the form
$\left\{  e_{1},e_{2},\ldots,e_{m}\right\}  $, where $\left(  v_{1}%
,e_{1},v_{2},e_{2},\ldots,v_{m},e_{m},v_{m+1}\right)  $ is a cycle of $G$ (by
the definition of a \textquotedblleft circuit\textquotedblright). Consider
this cycle $\left(  v_{1},e_{1},v_{2},e_{2},\ldots,v_{m},e_{m},v_{m+1}\right)
$. We thus have
\begin{equation}
C=\left\{  e_{1},e_{2},\ldots,e_{m}\right\}  .
\label{pf.lem.ambiEqs.circuit.1}%
\end{equation}

The list $\left(  v_{1},e_{1},v_{2},e_{2},\ldots,v_{m},e_{m},v_{m+1}\right)  $
is a cycle of $G$. According to the definition of a \textquotedblleft
cycle\textquotedblright, this means that this list is a list satisfying the
following six properties:

\begin{itemize}
\item The entries $v_{1},v_{2},\ldots,v_{m+1}$ at the odd positions of this
list belong to $V$, whereas the entries $e_{1},e_{2},\ldots,e_{m}$ at its even
positions belong to $E$.

\item We have $m\geq1$.

\item We have $v_{m+1}=v_{1}$.

\item The vertices $v_{1},v_{2},\ldots,v_{m}$ are pairwise distinct.

\item The edgeries $e_{1},e_{2},\ldots,e_{m}$ are pairwise distinct.

\item We have $\left\{  v_{i},v_{i+1}\right\}  \in\varphi\left(  e_{i}\right)
$ for every $i\in\left\{  1,2,\ldots,m\right\}  $.
\end{itemize}

Thus, $\left(  v_{1},e_{1},v_{2},e_{2},\ldots,v_{m},e_{m},v_{m+1}\right)  $ is
a list satisfying the six properties that we have just mentioned. In
particular, $v_{m+1}=v_{1}$. Also,%
\begin{equation}
\left\{  v_{i},v_{i+1}\right\}  \in\varphi\left(  e_{i}\right)
\ \ \ \ \ \ \ \ \ \ \text{for every }i\in\left\{  1,2,\ldots,m\right\}  .
\label{pf.lem.ambiEqs.circuit.inphi}%
\end{equation}

We have $e\in C=\left\{  e_{1},e_{2},\ldots,e_{m}\right\}  $. Thus, $e=e_{i}$
for some $i\in\left\{  1,2,\ldots,m\right\}  $. Consider this $i$.

Now, we have%
\begin{equation}
f\left(  v_{j}\right)  =f\left(  v_{j+1}\right)  \ \ \ \ \ \ \ \ \ \ \text{for
every }j\in\left\{  1,2,\ldots,m\right\}  \setminus\left\{  i\right\}
\label{pf.lem.ambiEqs.circuit.2}%
\end{equation}
\footnote{\textit{Proof of (\ref{pf.lem.ambiEqs.circuit.2}):} Let
$j\in\left\{  1,2,\ldots,m\right\}  \setminus\left\{  i\right\}  $. Thus,
$j\in\left\{  1,2,\ldots,m\right\}  $ and $j\notin\left\{  i\right\}  $. From
$j\notin\left\{  i\right\}  $, we obtain $j\neq i$. Thus, $e_{j}\neq e_{i}$
(since the edgeries $e_{1},e_{2},\ldots,e_{m}$ are pairwise distinct). In
other words, $e_{j}\neq e$ (since $e=e_{i}$).
\par
Now, $e_{j}\in\left\{  e_{1},e_{2},\ldots,e_{m}\right\}  =C$ (by
(\ref{pf.lem.ambiEqs.circuit.1})). Combining this with $e_{j}\neq e$, we
obtain
\[
e_{j}\in C\setminus\left\{  e\right\}  \subseteq\operatorname*{EQS}\left(
G,f\right)  =\left\{  d\in E\ \mid\ \varphi\left(  d\right)  \subseteq
\operatorname*{Eqs}f\right\}
\]
(by (\ref{eq.rmk.ambiEQS.rewrite.2})). In other words, $e_{j}$ is a $d\in E$
satisfying $\varphi\left(  d\right)  \subseteq\operatorname*{Eqs}f$. In other
words, $e_{j}\in E$ and $\varphi\left(  e_{j}\right)  \subseteq
\operatorname*{Eqs}f$.
\par
However, (\ref{pf.lem.ambiEqs.circuit.inphi}) (applied to $j$ instead of $i$)
yields
\begin{align*}
\left\{  v_{j},v_{j+1}\right\}   &  \in\varphi\left(  e_{j}\right)
\subseteq\operatorname*{Eqs}f\\
&  =\left\{  \left\{  s,t\right\}  \ \mid\ \left(  s,t\right)  \in
V^{2},\ s\neq t\text{ and }f\left(  s\right)  =f\left(  t\right)  \right\}
\end{align*}
(by the definition of $\operatorname*{Eqs}f$). In other words, the set
$\left\{  v_{j},v_{j+1}\right\}  $ has the form $\left\{  s,t\right\}  $ for
some $\left(  s,t\right)  \in V^{2}$ satisfying $s\neq t$ and $f\left(
s\right)  =f\left(  t\right)  $. Consider this $\left(  s,t\right)  $. Thus,
$\left\{  v_{j},v_{j+1}\right\}  =\left\{  s,t\right\}  $.
\par
We have $f\left(  s\right)  =f\left(  t\right)  $. Therefore, set $g=f\left(
s\right)  =f\left(  t\right)  $. We have%
\[
f\left(  \left\{  s,t\right\}  \right)  =\left\{  \underbrace{f\left(
s\right)  }_{=g},\underbrace{f\left(  t\right)  }_{=g}\right\}  =\left\{
g,g\right\}  =\left\{  g\right\}  .
\]
Now, $v_{j}\in\left\{  v_{j},v_{j+1}\right\}  =\left\{  s,t\right\}  $, and
thus $f\left(  \underbrace{v_{j}}_{\in\left\{  s,t\right\}  }\right)  \in
f\left(  \left\{  s,t\right\}  \right)  =\left\{  g\right\}  $. In other
words, $f\left(  v_{j}\right)  =g$. Also, $v_{j+1}\in\left\{  v_{j}%
,v_{j+1}\right\}  =\left\{  s,t\right\}  $, and thus $f\left(
\underbrace{v_{j+1}}_{\in\left\{  s,t\right\}  }\right)  \in f\left(  \left\{
s,t\right\}  \right)  =\left\{  g\right\}  $. In other words, $f\left(
v_{j+1}\right)  =g$. Comparing this with $f\left(  v_{j}\right)  =g$, we
obtain $f\left(  v_{j}\right)  =f\left(  v_{j+1}\right)  $. This proves
(\ref{pf.lem.ambiEqs.circuit.2}).}. Hence,%
\begin{equation}
f\left(  v_{1}\right)  =f\left(  v_{i}\right)
\label{pf.lem.ambiEqs.circuit.3a}%
\end{equation}
\footnote{\textit{Proof of (\ref{pf.lem.ambiEqs.circuit.3a}):} Let
$j\in\left\{  1,2,\ldots,i-1\right\}  $. Thus, $j\in\left\{  1,2,\ldots
,i-1\right\}  \subseteq\left\{  1,2,\ldots,m\right\}  $. Combining this with
$j\neq i$ (since $j<i$ (since $j\in\left\{  1,2,\ldots,i-1\right\}  $)), we
obtain $j\in\left\{  1,2,\ldots,m\right\}  \setminus\left\{  i\right\}  $.
Hence, $f\left(  v_{j}\right)  =f\left(  v_{j+1}\right)  $ (by
(\ref{pf.lem.ambiEqs.circuit.2})).
\par
Now, let us forget that we fixed $j$. We thus have proven that $f\left(
v_{j}\right)  =f\left(  v_{j+1}\right)  $ for every $j\in\left\{
1,2,\ldots,i-1\right\}  $. In other words, $f\left(  v_{1}\right)  =f\left(
v_{2}\right)  =\cdots=f\left(  v_{i}\right)  $. Hence, $f\left(  v_{1}\right)
=f\left(  v_{i}\right)  $. This proves (\ref{pf.lem.ambiEqs.circuit.3a}).}.
Also,%
\begin{equation}
f\left(  v_{i+1}\right)  =f\left(  v_{m+1}\right)
\label{pf.lem.ambiEqs.circuit.3b}%
\end{equation}
\footnote{\textit{Proof of (\ref{pf.lem.ambiEqs.circuit.3b}):} Let
$j\in\left\{  i+1,i+2,\ldots,m\right\}  $. Thus, $j\in\left\{  i+1,i+2,\ldots
,m\right\}  \subseteq\left\{  1,2,\ldots,m\right\}  $. Combining this with
$j\neq i$ (since $j>i$ (since $j\in\left\{  i+1,i+2,\ldots,m\right\}  $)), we
obtain $j\in\left\{  1,2,\ldots,m\right\}  \setminus\left\{  i\right\}  $.
Hence, $f\left(  v_{j}\right)  =f\left(  v_{j+1}\right)  $ (by
(\ref{pf.lem.ambiEqs.circuit.2})).
\par
Now, let us forget that we fixed $j$. We thus have proven that $f\left(
v_{j}\right)  =f\left(  v_{j+1}\right)  $ for every $j\in\left\{
i+1,i+2,\ldots,m\right\}  $. In other words, $f\left(  v_{i+1}\right)
=f\left(  v_{i+2}\right)  =\cdots=f\left(  v_{m+1}\right)  $. Hence, $f\left(
v_{i+1}\right)  =f\left(  v_{m+1}\right)  $. This proves
(\ref{pf.lem.ambiEqs.circuit.3b}).}. Now, (\ref{pf.lem.ambiEqs.circuit.3a})
yields $f\left(  v_{i}\right)  =f\left(  v_{m+1}\right)  =f\left(
v_{1}\right)  $ (since $v_{m+1}=v_{1}$), so that
\[
f\left(  v_{i}\right)  =f\left(  v_{m+1}\right)  =f\left(  v_{i+1}\right)
\ \ \ \ \ \ \ \ \ \ \left(  \text{by (\ref{pf.lem.ambiEqs.circuit.3b}%
)}\right)  .
\]

Moreover, $v_{i}$ and $v_{i+1}$ are elements of $V$ (since $v_{1},v_{2}%
,\ldots,v_{m+1}$ are elements of $V$). In other words, $v_{i}\in V$ and
$v_{i+1}\in V$. Hence, $\left(  v_{i},v_{i+1}\right)  \in V^{2}$.

Furthermore, $v_{i}\neq v_{i+1}$\ \ \ \ \footnote{\textit{Proof.} Assume the
contrary. Thus, $v_{i}=v_{i+1}$. Thus, $\left\{  v_{i},v_{i+1}\right\}
=\left\{  v_{i+1},v_{i+1}\right\}  =\left\{  v_{i+1}\right\}  $, so that
$\left\vert \left\{  v_{i},v_{i+1}\right\}  \right\vert =\left\vert \left\{
v_{i+1}\right\}  \right\vert =1$.
\par
However, (\ref{pf.lem.ambiEqs.circuit.inphi}) yields that $\left\{
v_{i},v_{i+1}\right\}  \in\varphi\left(  e_{i}\right)  \subseteq\dbinom{V}{2}%
$. Thus, $\left\{  v_{i},v_{i+1}\right\}  $ is a $2$-element subset of $V$. In
particular, this entails $\left\vert \left\{  v_{i},v_{i+1}\right\}
\right\vert =2$. But this contradicts $\left\vert \left\{  v_{i}%
,v_{i+1}\right\}  \right\vert =1\neq2$. This contradiction shows that our
assumption was wrong. Qed.}.

Now, the definition of $\operatorname*{Eqs}f$ shows that%
\begin{equation}
\operatorname*{Eqs}f=\left\{  \left\{  s,t\right\}  \ \mid\ \left(
s,t\right)  \in V^{2},\ s\neq t\text{ and }f\left(  s\right)  =f\left(
t\right)  \right\}  . \label{pf.lem.ambiEqs.circuit.Eqs-def}%
\end{equation}

But we have $\left(  v_{i},v_{i+1}\right)  \in V^{2}$, $v_{i}\neq v_{i+1}$ and
$f\left(  v_{i}\right)  =f\left(  v_{i+1}\right)  $. Hence, the set $\left\{
v_{i},v_{i+1}\right\}  $ has the form $\left\{  s,t\right\}  $ for some
$\left(  s,t\right)  \in V^{2}$ satisfying $s\neq t$ and $f\left(  s\right)
=f\left(  t\right)  $ (namely, for $\left(  s,t\right)  =\left(  v_{i}%
,v_{i+1}\right)  $). Thus,%
\[
\left\{  v_{i},v_{i+1}\right\}  \in\left\{  \left\{  s,t\right\}
\ \mid\ \left(  s,t\right)  \in V^{2},\ s\neq t\text{ and }f\left(  s\right)
=f\left(  t\right)  \right\}  =\operatorname*{Eqs}f
\]
(by (\ref{pf.lem.ambiEqs.circuit.Eqs-def})).

Now, (\ref{pf.lem.ambiEqs.circuit.inphi}) yields that $\left\{  v_{i}%
,v_{i+1}\right\}  \in\varphi\left(  e_{i}\right)  $. We can rewrite this as
$\left\{  v_{i},v_{i+1}\right\}  \in\varphi\left(  e\right)  $ (since
$e=e_{i}$). However, we also know that $\varphi\left(  e\right)  $ is a
$1$-element set. In other words, $\varphi\left(  e\right)  =\left\{
u\right\}  $ for some element $u$. Consider this $u$. From $\left\{
v_{i},v_{i+1}\right\}  \in\varphi\left(  e\right)  =\left\{  u\right\}  $, we
obtain $\left\{  v_{i},v_{i+1}\right\}  =u$, so that $u=\left\{  v_{i}%
,v_{i+1}\right\}  \in\operatorname*{Eqs}f$. Hence, $\left\{  u\right\}
\subseteq\operatorname*{Eqs}f$.

Thus, altogether, $\varphi\left(  e\right)  =\left\{  u\right\}
\subseteq\operatorname*{Eqs}f$. Thus, $e\in E$ and $\varphi\left(  e\right)
\subseteq\operatorname*{Eqs}f$. In other words, $e$ is an element $d\in E$
satisfying $\varphi\left(  d\right)  \subseteq\operatorname*{Eqs}f$. Hence,%
\[
e\in\left\{  d\in E\ \mid\ \varphi\left(  d\right)  \subseteq
\operatorname*{Eqs}f\right\}  =\operatorname*{EQS}\left(  G,f\right)
\]
(by (\ref{eq.rmk.ambiEQS.rewrite.2})). This proves Lemma
\ref{lem.ambiEqs.circuit}.
\end{proof}
\end{verlong}

Our next lemma will play a role in our proof of Theorem
\ref{thm.ambichromsym.varis} that is similar to the role of Lemma
\ref{lem.BC.nonempty} in the proof of Theorem \ref{thm.chromsym.varis}
(although it is different in its claim).

\begin{lemma}
\label{lem.ambiEqs.no-bc}Let $G=\left(  V,E,\varphi\right)  $ be an ambigraph.
Let $X$ be a totally ordered set. Let $\ell:E\rightarrow X$ be a labeling function.

Let $Y$ be any set. Let $f:V\rightarrow Y$ be any map. Assume that the set
$\operatorname*{EQS}\left(  G,f\right)  $ contains no singleton edgery. Then,
there exists no broken circuit $K$ of $G$ satisfying $K\subseteq
\operatorname*{EQS}\left(  G,f\right)  $.
\end{lemma}

\begin{proof}
[Proof of Lemma \ref{lem.ambiEqs.no-bc}.]Assume the contrary. Thus, there
exists some broken circuit $K$ of $G$ satisfying $K\subseteq
\operatorname*{EQS}\left(  G,f\right)  $.

\begin{vershort}
The set $K$ is a broken circuit of $G$. According to the definition of a
broken circuit, this means that $K$ can be written as $K=C\setminus\left\{
e\right\}  $, where $C$ is a circuit of $G$, and where $e$ is the unique
singleton edgery in $C$ having maximum label (among the singleton edgeries in
$C$). Consider these $C$ and $e$.
\end{vershort}

\begin{verlong}
We recall that a broken circuit of $G$ means a subset of $E$ having the form
$C\setminus\left\{  e\right\}  $, where $C$ is a circuit of $G$, and where $e$
is the unique singleton edgery in $C$ having maximum label (among the
singleton edgeries in $C$). Thus, in particular, $K$ is a subset of $E$ having
this form (since $K$ is a broken circuit of $G$). In other words, we can write
$K$ as $K=C\setminus\left\{  e\right\}  $, where $C$ is a circuit of $G$, and
where $e$ is the unique singleton edgery in $C$ having maximum label (among
the singleton edgeries in $C$). Consider these $C$ and $e$.
\end{verlong}

We have $C\setminus\left\{  e\right\}  =K\subseteq\operatorname*{EQS}\left(
G,f\right)  $. Thus, Lemma \ref{lem.ambiEqs.circuit} (applied to $Y$ instead
of $X$) yields that $e\in\operatorname*{EQS}\left(  G,f\right)  $. Thus, the
set $\operatorname*{EQS}\left(  G,f\right)  $ contains a singleton edgery
(namely, $e$). But this contradicts the fact that the set $\operatorname*{EQS}%
\left(  G,f\right)  $ contains no singleton edgery.

This contradiction shows that our assumption was false. Hence, Lemma
\ref{lem.ambiEqs.no-bc} is proven.
\end{proof}

We are now ready to prove the keystone lemma, which of course is an analogue
of Lemma \ref{lem.NBCm.moeb}:

\begin{lemma}
\label{lem.ambiNBCm.moeb}Let $G=\left(  V,E,\varphi\right)  $ be a finite
ambigraph. Let $X$ be a totally ordered set. Let $\ell:E\rightarrow X$ be a
labeling function. Let $\mathfrak{K}$ be some set of broken circuits of $G$
(not necessarily containing all of them). Let $a_{K}$ be an element of
$\mathbf{k}$ for every $K\in\mathfrak{K}$.

Let $Y$ be any set. Let $f:V\rightarrow Y$ be any map. Then,%
\[
\sum_{B\subseteq\operatorname*{EQS}\left(  G,f\right)  }\left(  -1\right)
^{\left\vert B\right\vert }\prod_{\substack{K\in\mathfrak{K};\\K\subseteq
B}}a_{K}=\left[  \operatorname*{EQS}\left(  G,f\right)  =\varnothing\right]
.
\]

\end{lemma}

\begin{proof}
[Proof of Lemma \ref{lem.ambiNBCm.moeb}.]We are in one of the following two cases:

\textit{Case 1:} The set $\operatorname*{EQS}\left(  G,f\right)  $ contains no
singleton edgery.

\textit{Case 2:} The set $\operatorname*{EQS}\left(  G,f\right)  $ contains at
least one singleton edgery.

Let us first consider Case 1. In this case, the set $\operatorname*{EQS}%
\left(  G,f\right)  $ contains no singleton edgery. Hence, using Lemma
\ref{lem.ambiEqs.no-bc}, it is easy to see that every subset $B$ of
$\operatorname*{EQS}\left(  G,f\right)  $ satisfies%
\begin{equation}
\prod_{\substack{K\in\mathfrak{K};\\K\subseteq B}}a_{K}=1.
\label{pf.lem.ambiNBCm.moeb.c1.1}%
\end{equation}

\begin{vershort}
\textit{Proof of (\ref{pf.lem.ambiNBCm.moeb.c1.1}):} Let $B$ be a subset of
$\operatorname*{EQS}\left(  G,f\right)  $. Then, Lemma \ref{lem.ambiEqs.no-bc}
yields that there exists no broken circuit $K$ of $G$ satisfying
$K\subseteq\operatorname*{EQS}\left(  G,f\right)  $. Hence, there exists no
$K\in\mathfrak{K}$ satisfying $K\subseteq\operatorname*{EQS}\left(
G,f\right)  $ (since each $K\in\mathfrak{K}$ is a broken circuit of $G$).
Therefore, there exists no $K\in\mathfrak{K}$ satisfying $K\subseteq B$ either
(since $K\subseteq B$ would entail $K\subseteq B\subseteq\operatorname*{EQS}%
\left(  G,f\right)  $, which would contradict the previous sentence). Hence,
the product $\prod_{\substack{K\in\mathfrak{K};\\K\subseteq B}}a_{K}$ is
empty, and thus equals $1$ by definition. This proves
(\ref{pf.lem.ambiNBCm.moeb.c1.1}).
\end{vershort}

\begin{verlong}
\textit{Proof of (\ref{pf.lem.ambiNBCm.moeb.c1.1}):} Let $B$ be a subset of
$\operatorname*{EQS}\left(  G,f\right)  $. Then, Lemma \ref{lem.ambiEqs.no-bc}
yields that there exists no broken circuit $K$ of $G$ satisfying
$K\subseteq\operatorname*{EQS}\left(  G,f\right)  $. Hence, there exists no
$K\in\mathfrak{K}$ satisfying $K\subseteq B$\ \ \ \ \footnote{\textit{Proof.}
Assume the contrary. Thus, there exists some $K\in\mathfrak{K}$ satisfying
$K\subseteq B$. Consider such a $K$, and denote it by $L$. Thus, $L$ is a
$K\in\mathfrak{K}$ satisfying $K\subseteq B$. In other words, $L\in
\mathfrak{K}$ and $L\subseteq B$. From $L\in\mathfrak{K}$, we see that $L$ is
a broken circuit of $G$ (because $\mathfrak{K}$ is a set of broken circuits of
$G$). However, $B$ is a subset of $\operatorname*{EQS}\left(  G,f\right)  $;
in other words, $B\subseteq\operatorname*{EQS}\left(  G,f\right)  $. Thus,
$L\subseteq B\subseteq\operatorname*{EQS}\left(  G,f\right)  $.
\par
Hence, $L$ is a broken circuit $K$ of $G$ satisfying $K\subseteq
\operatorname*{EQS}\left(  G,f\right)  $ (since $L$ is a broken circuit of $G$
and since $L\subseteq\operatorname*{EQS}\left(  G,f\right)  $). Thus, there
exists a broken circuit $K$ of $G$ satisfying $K\subseteq\operatorname*{EQS}%
\left(  G,f\right)  $ (namely, $L$). But this contradicts the fact that there
exists no broken circuit $K$ of $G$ satisfying $K\subseteq\operatorname*{EQS}%
\left(  G,f\right)  $. This contradiction shows that our assumption was false.
Qed.}. Therefore, the product $\prod_{\substack{K\in\mathfrak{K};\\K\subseteq
B}}a_{K}$ is empty. Thus, $\prod_{\substack{K\in\mathfrak{K};\\K\subseteq
B}}a_{K}=\left(  \text{empty product}\right)  =1$. This proves
(\ref{pf.lem.ambiNBCm.moeb.c1.1}).
\end{verlong}

\begin{vershort}
Now,
\begin{align*}
\sum_{B\subseteq\operatorname*{EQS}\left(  G,f\right)  }\left(  -1\right)
^{\left\vert B\right\vert }\underbrace{\prod_{\substack{K\in\mathfrak{K}%
;\\K\subseteq B}}a_{K}}_{\substack{=1\\\text{(by
(\ref{pf.lem.ambiNBCm.moeb.c1.1}))}}}  &  =\sum_{B\subseteq\operatorname*{EQS}%
\left(  G,f\right)  }\left(  -1\right)  ^{\left\vert B\right\vert }%
=\sum_{I\subseteq\operatorname*{EQS}\left(  G,f\right)  }\left(  -1\right)
^{\left\vert I\right\vert }\\
&  =\left[  \operatorname*{EQS}\left(  G,f\right)  =\varnothing\right]
\end{align*}
(by Lemma \ref{lem.cancel}, applied to $S=\operatorname*{EQS}\left(
G,f\right)  $). Thus, Lemma \ref{lem.ambiNBCm.moeb} is proved in Case 1.
\end{vershort}

\begin{verlong}
The set $E$ is finite (since $\left(  V,E,\varphi\right)  $ is finite). Hence,
its subset $\operatorname*{EQS}\left(  G,f\right)  $ is also finite. Now,
\begin{align*}
&  \sum_{B\subseteq\operatorname*{EQS}\left(  G,f\right)  }\left(  -1\right)
^{\left\vert B\right\vert }\underbrace{\prod_{\substack{K\in\mathfrak{K}%
;\\K\subseteq B}}a_{K}}_{\substack{=1\\\text{(by
(\ref{pf.lem.ambiNBCm.moeb.c1.1}))}}}\\
&  =\sum_{B\subseteq\operatorname*{EQS}\left(  G,f\right)  }\left(  -1\right)
^{\left\vert B\right\vert }\\
&  =\sum_{I\subseteq\operatorname*{EQS}\left(  G,f\right)  }\left(  -1\right)
^{\left\vert I\right\vert }\ \ \ \ \ \ \ \ \ \ \left(
\begin{array}
[c]{c}%
\text{here, we have renamed the summation}\\
\text{index }B\text{ as }I
\end{array}
\right) \\
&  =\left[  \operatorname*{EQS}\left(  G,f\right)  =\varnothing\right]
\ \ \ \ \ \ \ \ \ \ \left(  \text{by Lemma \ref{lem.cancel}, applied to
}S=\operatorname*{EQS}\left(  G,f\right)  \right)  .
\end{align*}
Thus, Lemma \ref{lem.ambiNBCm.moeb} is proved in Case 1.
\end{verlong}

\begin{vershort}
Let us now consider Case 2. In this case, the set $\operatorname*{EQS}\left(
G,f\right)  $ contains at least one singleton edgery. Pick any such singleton
edgery $d\in\operatorname*{EQS}\left(  G,f\right)  $ with maximum $\ell\left(
d\right)  $ (among all such singleton edgeries).

Define two subsets $\mathcal{U}$ and $\mathcal{V}$ of $\mathcal{P}\left(
\operatorname*{EQS}\left(  G,f\right)  \right)  $ as follows:%
\begin{align*}
\mathcal{U}  &  =\left\{  F\in\mathcal{P}\left(  \operatorname*{EQS}\left(
G,f\right)  \right)  \ \mid\ d\notin F\right\}  ;\\
\mathcal{V}  &  =\left\{  F\in\mathcal{P}\left(  \operatorname*{EQS}\left(
G,f\right)  \right)  \ \mid\ d\in F\right\}  .
\end{align*}
Thus, we have $\mathcal{P}\left(  \operatorname*{EQS}\left(  G,f\right)
\right)  =\mathcal{U}\cup\mathcal{V}$, and the sets $\mathcal{U}$ and
$\mathcal{V}$ are disjoint. Now, we define a map $\Phi:\mathcal{U}%
\rightarrow\mathcal{V}$ by%
\[
\left(  \Phi\left(  B\right)  =B\cup\left\{  d\right\}
\ \ \ \ \ \ \ \ \ \ \text{for every }B\in\mathcal{U}\right)  .
\]
As in our above proof of Lemma \ref{lem.NBCm.moeb}, we can see that this map
$\Phi$ is well-defined and a bijection, and that every $B\in\mathcal{U}$
satisfies%
\begin{equation}
\left(  -1\right)  ^{\left\vert \Phi\left(  B\right)  \right\vert }=-\left(
-1\right)  ^{\left\vert B\right\vert }.
\label{pf.lem.ambiNBCm.moeb.short.Phi.-1}%
\end{equation}

Furthermore, we claim that, for every $B\in\mathcal{U}$ and every
$K\in\mathfrak{K}$, we have the following logical equivalence:%
\begin{equation}
\left(  K\subseteq B\right)  \ \Longleftrightarrow\ \left(  K\subseteq
\Phi\left(  B\right)  \right)  . \label{pf.lem.ambiNBCm.moeb.short.Phi.equiv}%
\end{equation}

Indeed, our above proof of (\ref{pf.lem.NBCm.moeb.short.Phi.equiv}) can be
easily transformed into a proof of (\ref{pf.lem.ambiNBCm.moeb.short.Phi.equiv}%
) by some simple changes\footnote{Here are the changes that we need to make to
the above proof: We need to replace \textquotedblleft edge\textquotedblright%
\ by \textquotedblleft singleton edgery\textquotedblright; replace both sets
$E\cap\operatorname*{Eqs}f$ and $\operatorname*{Eqs}f$ by $\operatorname*{EQS}%
\left(  G,f\right)  $; and replace the reference to Lemma
\ref{lem.Eqs.circuit} by a reference to Lemma \ref{lem.ambiEqs.circuit}.}.
Thus, (\ref{pf.lem.ambiNBCm.moeb.short.Phi.equiv}) holds.

Now, just as in the proof of Lemma \ref{lem.NBCm.moeb}, we can conclude that
\begin{equation}
\sum_{B\subseteq\operatorname*{EQS}\left(  G,f\right)  }\left(  -1\right)
^{\left\vert B\right\vert }\prod_{\substack{K\in\mathfrak{K};\\K\subseteq
B}}a_{K}=0. \label{pf.lem.ambiNBCm.moeb.short.Phi.c2.at}%
\end{equation}

However, the set $\operatorname*{EQS}\left(  G,f\right)  $ contains at least
one singleton edgery, and thus is nonempty. Hence, $\left[
\operatorname*{EQS}\left(  G,f\right)  =\varnothing\right]  =0$. Comparing
this with (\ref{pf.lem.ambiNBCm.moeb.short.Phi.c2.at}), we obtain
$\sum_{B\subseteq\operatorname*{EQS}\left(  G,f\right)  }\left(  -1\right)
^{\left\vert B\right\vert }\prod_{\substack{K\in\mathfrak{K};\\K\subseteq
B}}a_{K}=\left[  \operatorname*{EQS}\left(  G,f\right)  =\varnothing\right]
$. Thus, Lemma \ref{lem.ambiNBCm.moeb} is proved in Case 2.
\end{vershort}

\begin{verlong}
Let us now consider Case 2. In this case, the set $\operatorname*{EQS}\left(
G,f\right)  $ contains at least one singleton edgery.

The set $E$ is finite (since $\left(  V,E,\varphi\right)  $ is finite). Hence,
its subset $\operatorname*{EQS}\left(  G,f\right)  $ is also finite. Moreover,
this subset $\operatorname*{EQS}\left(  G,f\right)  $ is nonempty (since it
contains at least one singleton edgery). In other words, $\operatorname*{EQS}%
\left(  G,f\right)  \neq\varnothing$. Hence, $\left[  \operatorname*{EQS}%
\left(  G,f\right)  =\varnothing\right]  =0$.

We know that the set $\operatorname*{EQS}\left(  G,f\right)  $ contains at
least one singleton edgery. Pick any such singleton edgery $d\in
\operatorname*{EQS}\left(  G,f\right)  $ with maximum $\ell\left(  d\right)  $
(among all such singleton edgeries). (If there are several such $d$, then it
does not matter which one we pick.)

We have chosen $d$ to be the singleton edgery in $\operatorname*{EQS}\left(
G,f\right)  $ with maximum $\ell\left(  d\right)  $ (among all singleton
edgeries in $\operatorname*{EQS}\left(  G,f\right)  $). Thus, for any
singleton edgery $e\in\operatorname*{EQS}\left(  G,f\right)  $, we have%
\begin{equation}
\ell\left(  d\right)  \geq\ell\left(  e\right)  .
\label{pf.lem.ambiNBCm.moeb.maxl}%
\end{equation}

As usual, we let $\mathcal{P}\left(  S\right)  $ denote the powerset of any
set $S$. We now define two subsets $\mathcal{U}$ and $\mathcal{V}$ of
$\mathcal{P}\left(  \operatorname*{EQS}\left(  G,f\right)  \right)  $ as
follows:%
\begin{align*}
\mathcal{U}  &  =\left\{  F\in\mathcal{P}\left(  \operatorname*{EQS}\left(
G,f\right)  \right)  \ \mid\ d\notin F\right\}  ;\\
\mathcal{V}  &  =\left\{  F\in\mathcal{P}\left(  \operatorname*{EQS}\left(
G,f\right)  \right)  \ \mid\ d\in F\right\}  .
\end{align*}
Every $B\in\mathcal{U}$ satisfies $B\cup\left\{  d\right\}  \in\mathcal{V}%
$\ \ \ \ \footnote{\textit{Proof.} Let $B\in\mathcal{U}$. Thus, $B\in
\mathcal{U}=\left\{  F\in\mathcal{P}\left(  \operatorname*{EQS}\left(
G,f\right)  \right)  \ \mid\ d\notin F\right\}  $. In other words, $B$ is an
element $F$ of $\mathcal{P}\left(  \operatorname*{EQS}\left(  G,f\right)
\right)  $ satisfying $d\notin F$. In other words, $B$ is an element of
$\mathcal{P}\left(  \operatorname*{EQS}\left(  G,f\right)  \right)  $ and
satisfies $d\notin B$. We have $B\in\mathcal{P}\left(  \operatorname*{EQS}%
\left(  G,f\right)  \right)  $; in other words, $B$ is a subset of
$\operatorname*{EQS}\left(  G,f\right)  $. Also, $\left\{  d\right\}
\subseteq\operatorname*{EQS}\left(  G,f\right)  $ (since $d\in
\operatorname*{EQS}\left(  G,f\right)  $). Thus, both $B$ and $\left\{
d\right\}  $ are subsets of $\operatorname*{EQS}\left(  G,f\right)  $. Hence,
their union $B\cup\left\{  d\right\}  $ is a subset of $\operatorname*{EQS}%
\left(  G,f\right)  $. In other words, $B\cup\left\{  d\right\}
\in\mathcal{P}\left(  \operatorname*{EQS}\left(  G,f\right)  \right)  $. Also,
$d\in\left\{  d\right\}  \subseteq B\cup\left\{  d\right\}  $. Hence,
$B\cup\left\{  d\right\}  $ is an element of $\mathcal{P}\left(
\operatorname*{EQS}\left(  G,f\right)  \right)  $ and satisfies $d\in
B\cup\left\{  d\right\}  $. In other words, $B\cup\left\{  d\right\}  $ is an
element $F$ of $\mathcal{P}\left(  \operatorname*{EQS}\left(  G,f\right)
\right)  $ satisfying $d\in F$. In other words, $B\cup\left\{  d\right\}
\in\left\{  F\in\mathcal{P}\left(  \operatorname*{EQS}\left(  G,f\right)
\right)  \ \mid\ d\in F\right\}  =\mathcal{V}$, qed.}. Thus, we can define a
map $\Phi:\mathcal{U}\rightarrow\mathcal{V}$ by%
\[
\left(  \Phi\left(  B\right)  =B\cup\left\{  d\right\}
\ \ \ \ \ \ \ \ \ \ \text{for every }B\in\mathcal{U}\right)  .
\]
Consider this map $\Phi$.

Every $B\in\mathcal{V}$ satisfies $B\setminus\left\{  d\right\}
\in\mathcal{U}$\ \ \ \ \footnote{\textit{Proof.} Let $B\in\mathcal{V}$. Thus,
$B\in\mathcal{V}=\left\{  F\in\mathcal{P}\left(  \operatorname*{EQS}\left(
G,f\right)  \right)  \ \mid\ d\in F\right\}  $. In other words, $B$ is an
element $F$ of $\mathcal{P}\left(  \operatorname*{EQS}\left(  G,f\right)
\right)  $ satisfying $d\in F$. In other words, $B$ is an element of
$\mathcal{P}\left(  \operatorname*{EQS}\left(  G,f\right)  \right)  $ and
satisfies $d\in B$. We have $B\in\mathcal{P}\left(  \operatorname*{EQS}\left(
G,f\right)  \right)  $; in other words, $B$ is a subset of
$\operatorname*{EQS}\left(  G,f\right)  $. Hence, $B\setminus\left\{
d\right\}  $ is a subset of $\operatorname*{EQS}\left(  G,f\right)  $ (since
$B\setminus\left\{  d\right\}  \subseteq B$). In other words, $B\setminus
\left\{  d\right\}  \in\mathcal{P}\left(  \operatorname*{EQS}\left(
G,f\right)  \right)  $. Also, $d\notin B\setminus\left\{  d\right\}  $ (since
$d\in\left\{  d\right\}  $). Hence, $B\setminus\left\{  d\right\}  $ is an
element of $\mathcal{P}\left(  \operatorname*{EQS}\left(  G,f\right)  \right)
$ and satisfies $d\notin B\setminus\left\{  d\right\}  $. In other words,
$B\setminus\left\{  d\right\}  $ is an element $F$ of $\mathcal{P}\left(
\operatorname*{EQS}\left(  G,f\right)  \right)  $ satisfying $d\notin F$. In
other words, $B\setminus\left\{  d\right\}  \in\left\{  F\in\mathcal{P}\left(
\operatorname*{EQS}\left(  G,f\right)  \right)  \ \mid\ d\notin F\right\}
=\mathcal{U}$, qed.}. Thus, we can define a map $\Psi:\mathcal{V}%
\rightarrow\mathcal{U}$ by%
\[
\left(  \Psi\left(  B\right)  =B\setminus\left\{  d\right\}
\ \ \ \ \ \ \ \ \ \ \text{for every }B\in\mathcal{V}\right)  .
\]
Consider this map $\Psi$.

We have $\Phi\circ\Psi=\operatorname*{id}$\ \ \ \ \footnote{\textit{Proof.}
Let $B\in\mathcal{V}$. We have%
\[
\left(  \Phi\circ\Psi\right)  \left(  B\right)  =\Phi\left(  \underbrace{\Psi
\left(  B\right)  }_{\substack{=B\setminus\left\{  d\right\}  \\\text{(by the
definition of }\Psi\text{)}}}\right)  =\Phi\left(  B\setminus\left\{
d\right\}  \right)  =\left(  B\setminus\left\{  d\right\}  \right)
\cup\left\{  d\right\}
\]
(by the definition of $\Phi$).
\par
We have $B\in\mathcal{V}=\left\{  F\in\mathcal{P}\left(  \operatorname*{EQS}%
\left(  G,f\right)  \right)  \ \mid\ d\in F\right\}  $. In other words, $B$ is
an element $F$ of $\mathcal{P}\left(  \operatorname*{EQS}\left(  G,f\right)
\right)  $ satisfying $d\in F$. In other words, $B$ is an element of
$\mathcal{P}\left(  \operatorname*{EQS}\left(  G,f\right)  \right)  $ and
satisfies $d\in B$. From $d\in B$, we obtain $\left\{  d\right\}  \subseteq
B$. Now, $\left(  \Phi\circ\Psi\right)  \left(  B\right)  =\left(
B\setminus\left\{  d\right\}  \right)  \cup\left\{  d\right\}  =B$ (since
$\left\{  d\right\}  \subseteq B$). Thus, $\left(  \Phi\circ\Psi\right)
\left(  B\right)  =B=\operatorname*{id}\left(  B\right)  $.
\par
Now, let us forget that we fixed $B$. We thus have proven that $\left(
\Phi\circ\Psi\right)  \left(  B\right)  =\operatorname*{id}\left(  B\right)  $
for every $B\in\mathcal{V}$. In other words, $\Phi\circ\Psi=\operatorname*{id}%
$, qed.} and $\Psi\circ\Phi=\operatorname*{id}$%
\ \ \ \ \footnote{\textit{Proof.} Let $B\in\mathcal{U}$. We have%
\[
\left(  \Psi\circ\Phi\right)  \left(  B\right)  =\Psi\left(  \underbrace{\Phi
\left(  B\right)  }_{\substack{=B\cup\left\{  d\right\}  \\\text{(by the
definition of }\Phi\text{)}}}\right)  =\Psi\left(  B\cup\left\{  d\right\}
\right)  =\left(  B\cup\left\{  d\right\}  \right)  \setminus\left\{
d\right\}
\]
(by the definition of $\Psi$).
\par
We have $B\in\mathcal{U}=\left\{  F\in\mathcal{P}\left(  \operatorname*{EQS}%
\left(  G,f\right)  \right)  \ \mid\ d\notin F\right\}  $. In other words, $B$
is an element $F$ of $\mathcal{P}\left(  \operatorname*{EQS}\left(
G,f\right)  \right)  $ satisfying $d\notin F$. In other words, $B$ is an
element of $\mathcal{P}\left(  \operatorname*{EQS}\left(  G,f\right)  \right)
$ and satisfies $d\notin B$. Now, $\left(  \Psi\circ\Phi\right)  \left(
B\right)  =\left(  B\cup\left\{  d\right\}  \right)  \setminus\left\{
d\right\}  =B$ (since $d\notin B$). Thus, $\left(  \Psi\circ\Phi\right)
\left(  B\right)  =B=\operatorname*{id}\left(  B\right)  $.
\par
Now, let us forget that we fixed $B$. We thus have proven that $\left(
\Psi\circ\Phi\right)  \left(  B\right)  =\operatorname*{id}\left(  B\right)  $
for every $B\in\mathcal{U}$. In other words, $\Psi\circ\Phi=\operatorname*{id}%
$, qed.}. Thus, the maps $\Phi$ and $\Psi$ are mutually inverse. Hence, the
map $\Phi$ is a bijection.

Moreover, for every $B\in\mathcal{U}$ and every $K\in\mathfrak{K}$, we have
the following logical equivalence:%
\begin{equation}
\left(  K\subseteq B\right)  \ \Longleftrightarrow\ \left(  K\subseteq
\Phi\left(  B\right)  \right)  . \label{pf.lem.ambiNBCm.moeb.equiv}%
\end{equation}

\textit{Proof of (\ref{pf.lem.ambiNBCm.moeb.equiv}):} Let $B\in\mathcal{U}$
and $K\in\mathfrak{K}$. We need to prove the logical equivalence
(\ref{pf.lem.ambiNBCm.moeb.equiv}).

The definition of $\Phi$ yields $\Phi\left(  B\right)  =B\cup\left\{
d\right\}  \supseteq B$, so that $B\subseteq\Phi\left(  B\right)  $.

We have $B\in\mathcal{U}=\left\{  F\in\mathcal{P}\left(  \operatorname*{EQS}%
\left(  G,f\right)  \right)  \ \mid\ d\notin F\right\}  $. In other words, $B$
is an element $F$ of $\mathcal{P}\left(  \operatorname*{EQS}\left(
G,f\right)  \right)  $ satisfying $d\notin F$. In other words, $B$ is an
element of $\mathcal{P}\left(  \operatorname*{EQS}\left(  G,f\right)  \right)
$ and satisfies $d\notin B$. We have $B\in\mathcal{P}\left(
\operatorname*{EQS}\left(  G,f\right)  \right)  $; in other words, $B$ is a
subset of $\operatorname*{EQS}\left(  G,f\right)  $.

Also, $\Phi\left(  B\right)  \in\mathcal{V}=\left\{  F\in\mathcal{P}\left(
\operatorname*{EQS}\left(  G,f\right)  \right)  \ \mid\ d\in F\right\}
\subseteq\mathcal{P}\left(  \operatorname*{EQS}\left(  G,f\right)  \right)  $.
In other words, $\Phi\left(  B\right)  \subseteq\operatorname*{EQS}\left(
G,f\right)  $.

We have $K\in\mathfrak{K}$. Thus, $K$ is a broken circuit of $G$ (since
$\mathfrak{K}$ is a set of broken circuits of $G$). In other words, $K$ is a
subset of $E$ having the form $C\setminus\left\{  e\right\}  $, where $C$ is a
circuit of $G$, and where $e$ is the unique singleton edgery in $C$ having
maximum label (among the singleton edgeries in $C$)\ \ \ \ \footnote{because a
broken circuit of $G$ is the same as a subset of $E$ having the form
$C\setminus\left\{  e\right\}  $, where $C$ is a circuit of $G$, and where $e$
is the unique singleton edgery in $C$ having maximum label (among the
singleton edgeries in $C$) (by the definition of a \textquotedblleft broken
circuit\textquotedblright)}. Consider this $C$ and this $e$. Thus, we have the
following facts:

\begin{itemize}
\item The set $C$ is a circuit of $G$.

\item The element $e$ is the unique singleton edgery in $C$ having maximum
label (among the singleton edgeries in $C$).

\item We have $K=C\setminus\left\{  e\right\}  $.
\end{itemize}

The element $e$ is the unique singleton edgery in $C$ having maximum label
(among the singleton edgeries in $C$). Thus, the only singleton edgery in $C$
whose label is greater or equal to the label of $e$ is $e$ itself. In other
words, if $e^{\prime}$ is any singleton edgery in $C$ satisfying $\ell\left(
e^{\prime}\right)  \geq\ell\left(  e\right)  $, then%
\begin{equation}
e^{\prime}=e. \label{pf.lem.ambiNBCm.moeb.equiv.pf.maxlab}%
\end{equation}

Let us now assume that $K\subseteq\Phi\left(  B\right)  $. Thus,
$K\subseteq\Phi\left(  B\right)  =B\cup\left\{  d\right\}  $. Hence,
$\underbrace{K}_{\subseteq B\cup\left\{  d\right\}  }\setminus\left\{
d\right\}  \subseteq\left(  B\cup\left\{  d\right\}  \right)  \setminus
\left\{  d\right\}  \subseteq B$.

We shall now prove that $K\subseteq B$.

Indeed, assume the contrary. Thus, $K\not \subseteq B$. If we had $d\notin K$,
then we would have $K\setminus\left\{  d\right\}  =K$ and therefore
$K=K\setminus\left\{  d\right\}  \subseteq B$; this would contradict
$K\not \subseteq B$. Hence, we cannot have $d\notin K$. We thus must have
$d\in K$. Hence, $d\in K=C\setminus\left\{  e\right\}  $. Hence, $d\in C$ and
$d\notin\left\{  e\right\}  $. From $d\notin\left\{  e\right\}  $, we obtain
$d\neq e$.

But $C\setminus\left\{  e\right\}  =K\subseteq\Phi\left(  B\right)
\subseteq\operatorname*{EQS}\left(  G,f\right)  $. Hence, Lemma
\ref{lem.ambiEqs.circuit} (applied to $Y$ instead of $X$) shows that
$e\in\operatorname*{EQS}\left(  G,f\right)  $. Thus,
(\ref{pf.lem.ambiNBCm.moeb.maxl}) shows that $\ell\left(  d\right)  \geq
\ell\left(  e\right)  $ (since $e$ is a singleton edgery).

Also, $d\in C$. Hence, $d$ is a singleton edgery in $C$ (since $d$ is a
singleton edgery). Since $\ell\left(  d\right)  \geq\ell\left(  e\right)  $,
we can therefore apply (\ref{pf.lem.ambiNBCm.moeb.equiv.pf.maxlab}) to
$e^{\prime}=d$. We thus obtain $d=e$. This contradicts $d\neq e$. This
contradiction proves that our assumption was wrong. Hence, $K\subseteq B$ is proven.

Now, let us forget that we assumed that $K\subseteq\Phi\left(  B\right)  $. We
thus have proven that $K\subseteq B$ under the assumption that $K\subseteq
\Phi\left(  B\right)  $. In other words, we have proven the implication%
\begin{equation}
\left(  K\subseteq\Phi\left(  B\right)  \right)  \ \Longrightarrow\ \left(
K\subseteq B\right)  . \label{pf.lem.ambiNBCm.moeb.equiv.pf.dir1}%
\end{equation}

On the other hand, if $K\subseteq B$, then $K\subseteq B\subseteq\Phi\left(
B\right)  $. Hence, the implication%
\[
\left(  K\subseteq B\right)  \ \Longrightarrow\ \left(  K\subseteq\Phi\left(
B\right)  \right)
\]
holds. Combining this implication with
(\ref{pf.lem.ambiNBCm.moeb.equiv.pf.dir1}), we obtain the logical equivalence
$\left(  K\subseteq B\right)  \ \Longleftrightarrow\ \left(  K\subseteq
\Phi\left(  B\right)  \right)  $. Thus, (\ref{pf.lem.ambiNBCm.moeb.equiv}) is proven.

Also, every $B\in\mathcal{U}$ satisfies%
\begin{equation}
\left(  -1\right)  ^{\left\vert B\right\vert }=-\left(  -1\right)
^{\left\vert \Phi\left(  B\right)  \right\vert }
\label{pf.lem.ambiNBCm.moeb.phisize}%
\end{equation}
\footnote{\textit{Proof of (\ref{pf.lem.ambiNBCm.moeb.phisize}):} Let
$B\in\mathcal{U}$. We have $B\in\mathcal{U}=\left\{  F\in\mathcal{P}\left(
\operatorname*{EQS}\left(  G,f\right)  \right)  \ \mid\ d\notin F\right\}  $.
In other words, $B$ is an element $F$ of $\mathcal{P}\left(
\operatorname*{EQS}\left(  G,f\right)  \right)  $ satisfying $d\notin F$. In
other words, $B$ is an element of $\mathcal{P}\left(  \operatorname*{EQS}%
\left(  G,f\right)  \right)  $ and satisfies $d\notin B$. From $d\notin B$, we
see that $\left\vert B\cup\left\{  d\right\}  \right\vert =\left\vert
B\right\vert +1$. Now, the definition of $\Phi$ yields $\Phi\left(  B\right)
=B\cup\left\{  d\right\}  $. Hence, $\left\vert \Phi\left(  B\right)
\right\vert =\left\vert B\cup\left\{  d\right\}  \right\vert =\left\vert
B\right\vert +1$. Thus, $\left(  -1\right)  ^{\left\vert \Phi\left(  B\right)
\right\vert }=\left(  -1\right)  ^{\left\vert B\right\vert +1}=-\left(
-1\right)  ^{\left\vert B\right\vert }$. Therefore, $\left(  -1\right)
^{\left\vert B\right\vert }=-\left(  -1\right)  ^{\left\vert \Phi\left(
B\right)  \right\vert }$. This proves (\ref{pf.lem.ambiNBCm.moeb.phisize}).}.

Now,%
\begin{align*}
&  \sum_{B\subseteq\operatorname*{EQS}\left(  G,f\right)  }\left(  -1\right)
^{\left\vert B\right\vert }\prod_{\substack{K\in\mathfrak{K};\\K\subseteq
B}}a_{K}\\
&  =\underbrace{\sum_{\substack{B\subseteq\operatorname*{EQS}\left(
G,f\right)  ;\\d\in B}}}_{\substack{=\sum_{B\in\left\{  F\in\mathcal{P}\left(
\operatorname*{EQS}\left(  G,f\right)  \right)  \ \mid\ d\in F\right\}  }%
=\sum_{B\in\mathcal{V}}\\\text{(since }\left\{  F\in\mathcal{P}\left(
\operatorname*{EQS}\left(  G,f\right)  \right)  \ \mid\ d\in F\right\}
=\mathcal{V}\text{)}}}\left(  -1\right)  ^{\left\vert B\right\vert }%
\prod_{\substack{K\in\mathfrak{K};\\K\subseteq B}}a_{K}\\
&  \ \ \ \ \ \ \ \ \ \ +\underbrace{\sum_{\substack{B\subseteq
\operatorname*{EQS}\left(  G,f\right)  ;\\d\notin B}}}_{\substack{=\sum
_{B\in\left\{  F\in\mathcal{P}\left(  \operatorname*{EQS}\left(  G,f\right)
\right)  \ \mid\ d\notin F\right\}  }=\sum_{B\in\mathcal{U}}\\\text{(since
}\left\{  F\in\mathcal{P}\left(  \operatorname*{EQS}\left(  G,f\right)
\right)  \ \mid\ d\notin F\right\}  =\mathcal{U}\text{)}}}\left(  -1\right)
^{\left\vert B\right\vert }\prod_{\substack{K\in\mathfrak{K};\\K\subseteq
B}}a_{K}\\
&  \ \ \ \ \ \ \ \ \ \ \ \ \ \ \ \ \ \ \ \ \left(
\begin{array}
[c]{c}%
\text{here, we have split the sum into two parts:}\\
\text{one containing all addends with }d\in B\text{,}\\
\text{and one containing all addends with }d\notin B
\end{array}
\right) \\
&  =\underbrace{\sum_{B\in\mathcal{V}}\left(  -1\right)  ^{\left\vert
B\right\vert }\prod_{\substack{K\in\mathfrak{K};\\K\subseteq B}}a_{K}%
}_{\substack{=\sum_{B\in\mathcal{U}}\left(  -1\right)  ^{\left\vert
\Phi\left(  B\right)  \right\vert }\prod_{\substack{K\in\mathfrak{K}%
;\\K\subseteq\Phi\left(  B\right)  }}a_{K}\\\text{(here, we have}%
\\\text{substituted }\Phi\left(  B\right)  \text{ for }B\text{ in the
sum,}\\\text{since the map }\Phi:\mathcal{U}\rightarrow\mathcal{V}\text{ is a
bijection)}}}+\sum_{B\in\mathcal{U}}\underbrace{\left(  -1\right)
^{\left\vert B\right\vert }}_{\substack{=-\left(  -1\right)  ^{\left\vert
\Phi\left(  B\right)  \right\vert }\\\text{(by
(\ref{pf.lem.ambiNBCm.moeb.phisize}))}}}\underbrace{\prod_{\substack{K\in
\mathfrak{K};\\K\subseteq B}}}_{\substack{=\prod_{\substack{K\in
\mathfrak{K};\\K\subseteq\Phi\left(  B\right)  }}\\\text{(because for every
}K\in\mathfrak{K}\text{,}\\\text{the condition }\left(  K\subseteq B\right)
\\\text{is equivalent to }\left(  K\subseteq\Phi\left(  B\right)  \right)
\\\text{(by (\ref{pf.lem.ambiNBCm.moeb.equiv})))}}}a_{K}\\
&  =\sum_{B\in\mathcal{U}}\left(  -1\right)  ^{\left\vert \Phi\left(
B\right)  \right\vert }\prod_{\substack{K\in\mathfrak{K};\\K\subseteq
\Phi\left(  B\right)  }}a_{K}+\sum_{B\in\mathcal{U}}\left(  -\left(
-1\right)  ^{\left\vert \Phi\left(  B\right)  \right\vert }\right)
\prod_{\substack{K\in\mathfrak{K};\\K\subseteq\Phi\left(  B\right)  }}a_{K}\\
&  =\sum_{B\in\mathcal{U}}\left(  -1\right)  ^{\left\vert \Phi\left(
B\right)  \right\vert }\prod_{\substack{K\in\mathfrak{K};\\K\subseteq
\Phi\left(  B\right)  }}a_{K}-\sum_{B\in\mathcal{U}}\left(  -1\right)
^{\left\vert \Phi\left(  B\right)  \right\vert }\prod_{\substack{K\in
\mathfrak{K};\\K\subseteq\Phi\left(  B\right)  }}a_{K}\\
&  =0=\left[  \operatorname*{EQS}\left(  G,f\right)  =\varnothing\right]
\ \ \ \ \ \ \ \ \ \ \left(  \text{since }\left[  \operatorname*{EQS}\left(
G,f\right)  =\varnothing\right]  =0\right)  .
\end{align*}
Thus, Lemma \ref{lem.ambiNBCm.moeb} is proved in Case 2.
\end{verlong}

We have now proved Lemma \ref{lem.ambiNBCm.moeb} in both Cases 1 and 2. Since
these two cases cover all possibilities, we thus have proved Lemma
\ref{lem.ambiNBCm.moeb}.
\end{proof}

We are now ready to prove Theorem \ref{thm.ambichromsym.varis} and Corollaries
\ref{cor.ambichromsym.K-free} and \ref{cor.ambichromsym.NBC} as well as
Theorem \ref{thm.ambichromsym.empty}:

\begin{vershort}
\begin{proof}
[Proof of Theorem \ref{thm.ambichromsym.varis}.]The definition of $X_{G}$
shows that
\begin{align*}
X_{G}  &  =\sum_{\substack{f:V\rightarrow\mathbb{N}_{+}\text{ is
a}\\\text{proper }\mathbb{N}_{+}\text{-coloring of }G}}\mathbf{x}_{f}\\
&  =\sum_{f:V\rightarrow\mathbb{N}_{+}}\left[  \underbrace{f\text{ is a proper
}\mathbb{N}_{+}\text{-coloring of }G}_{\substack{\Longleftrightarrow\ \left(
\text{the }\mathbb{N}_{+}\text{-coloring }f\text{ of }G\text{ is
proper}\right)  \\\Longleftrightarrow\ \left(  \operatorname*{EQS}\left(
G,f\right)  =\varnothing\right)  \\\text{(by Lemma \ref{lem.ambiEqs.proper},
applied to }\mathbb{N}_{+}\text{ instead of }X\text{)}}}\right]
\mathbf{x}_{f}\\
&  =\sum_{f:V\rightarrow\mathbb{N}_{+}}\underbrace{\left[  \operatorname*{EQS}%
\left(  G,f\right)  =\varnothing\right]  }_{\substack{=\sum_{B\subseteq
\operatorname*{EQS}\left(  G,f\right)  }\left(  -1\right)  ^{\left\vert
B\right\vert }\prod_{\substack{K\in\mathfrak{K};\\K\subseteq B}}a_{K}%
\\\text{(by Lemma \ref{lem.ambiNBCm.moeb}, applied to }Y=\mathbb{N}%
_{+}\text{)}}}\mathbf{x}_{f}\\
&  =\sum_{f:V\rightarrow\mathbb{N}_{+}}\ \ \underbrace{\sum_{B\subseteq
\operatorname*{EQS}\left(  G,f\right)  }}_{\substack{=\sum
_{\substack{B\subseteq E;\\B\subseteq\operatorname*{EQS}\left(  G,f\right)
}}\\\text{(since }\operatorname*{EQS}\left(  G,f\right)  \text{ is a subset of
}E\text{)}}}\left(  -1\right)  ^{\left\vert B\right\vert }\left(
\prod_{\substack{K\in\mathfrak{K};\\K\subseteq B}}a_{K}\right)  \mathbf{x}%
_{f}\\
&  =\underbrace{\sum_{f:V\rightarrow\mathbb{N}_{+}}\ \ \sum
_{\substack{B\subseteq E;\\B\subseteq\operatorname*{EQS}\left(  G,f\right)
}}}_{=\sum_{B\subseteq E}\ \ \sum_{\substack{f:V\rightarrow\mathbb{N}%
_{+};\\B\subseteq\operatorname*{EQS}\left(  G,f\right)  }}}\left(  -1\right)
^{\left\vert B\right\vert }\left(  \prod_{\substack{K\in\mathfrak{K}%
;\\K\subseteq B}}a_{K}\right)  \mathbf{x}_{f}\\
&  =\sum_{B\subseteq E}\ \ \sum_{\substack{f:V\rightarrow\mathbb{N}%
_{+};\\B\subseteq\operatorname*{EQS}\left(  G,f\right)  }}\left(  -1\right)
^{\left\vert B\right\vert }\left(  \prod_{\substack{K\in\mathfrak{K}%
;\\K\subseteq B}}a_{K}\right)  \mathbf{x}_{f}%
\end{align*}%
\begin{align*}
&  =\sum_{B\subseteq E}\left(  -1\right)  ^{\left\vert B\right\vert }\left(
\prod_{\substack{K\in\mathfrak{K};\\K\subseteq B}}a_{K}\right)
\underbrace{\sum_{\substack{f:V\rightarrow\mathbb{N}_{+};\\B\subseteq
\operatorname*{EQS}\left(  G,f\right)  }}}_{\substack{=\sum
_{\substack{f:V\rightarrow\mathbb{N}_{+};\\\operatorname*{union}%
B\subseteq\operatorname*{Eqs}f}}\\\text{(since Lemma \ref{lem.ambiEqs.union}%
}\\\text{(applied to }X=\mathbb{N}_{+}\text{) yields that}\\\text{the
condition }B\subseteq\operatorname*{EQS}\left(  G,f\right)  \text{
(on}\\\text{a map }f:V\rightarrow\mathbb{N}_{+}\text{) is}\\\text{equivalent
to }\operatorname*{union}B\subseteq\operatorname*{Eqs}f\text{)}}%
}\mathbf{x}_{f}\\
&  =\sum_{B\subseteq E}\left(  -1\right)  ^{\left\vert B\right\vert }\left(
\prod_{\substack{K\in\mathfrak{K};\\K\subseteq B}}a_{K}\right)
\underbrace{\sum_{\substack{f:V\rightarrow\mathbb{N}_{+}%
;\\\operatorname*{union}B\subseteq\operatorname*{Eqs}f}}\mathbf{x}_{f}%
}_{\substack{=p_{\lambda\left(  V,\operatorname*{union}B\right)  }\\\text{(by
Lemma \ref{lem.Eqs.sum}}\\\text{(applied to }\operatorname*{union}B\text{
instead of }B\text{))}}}\\
&  =\sum_{B\subseteq E}\left(  -1\right)  ^{\left\vert B\right\vert }\left(
\prod_{\substack{K\in\mathfrak{K};\\K\subseteq B}}a_{K}\right)  p_{\lambda
\left(  V,\operatorname*{union}B\right)  }=\sum_{F\subseteq E}\left(
-1\right)  ^{\left\vert F\right\vert }\left(  \prod_{\substack{K\in
\mathfrak{K};\\K\subseteq F}}a_{K}\right)  p_{\lambda\left(
V,\operatorname*{union}F\right)  }%
\end{align*}
(here, we have renamed the summation index $B$ as $F$). This proves Theorem
\ref{thm.ambichromsym.varis}.
\end{proof}
\end{vershort}

\begin{verlong}
\begin{proof}
[Proof of Theorem \ref{thm.ambichromsym.varis}.]We have%
\begin{equation}
X_{G}=\sum_{\substack{f:V\rightarrow\mathbb{N}_{+}\text{ is a}\\\text{proper
}\mathbb{N}_{+}\text{-coloring of }G}}\mathbf{x}_{f}
\label{pf.thm.ambichromsym.varis.XG-def}%
\end{equation}
(by the definition of $X_{G}$). Now, if $f:V\rightarrow\mathbb{N}_{+}$ is a
map, then we have the following logical equivalence:%
\begin{equation}
\left(  \text{the }\mathbb{N}_{+}\text{-coloring }f\text{ of }G\text{ is
proper}\right)  \ \Longleftrightarrow\ \left(  \operatorname*{EQS}\left(
G,f\right)  =\varnothing\right)  \label{pf.thm.ambichromsym.varis.equiv}%
\end{equation}
(because the $\mathbb{N}_{+}$-coloring $f$ of $G$ is proper if and only if
$\operatorname*{EQS}\left(  G,f\right)  =\varnothing$\ \ \ \ \footnote{by
Lemma \ref{lem.ambiEqs.proper} (applied to $\mathbb{N}_{+}$ instead of $X$)}).
Now,%
\begin{align}
&  \sum_{f:V\rightarrow\mathbb{N}_{+}}\left[  \underbrace{\operatorname*{EQS}%
\left(  G,f\right)  =\varnothing}_{\substack{\Longleftrightarrow\ \left(
\text{the }\mathbb{N}_{+}\text{-coloring }f\text{ of }G\text{ is
proper}\right)  \\\text{(by (\ref{pf.thm.ambichromsym.varis.equiv}))}%
}}\right]  \mathbf{x}_{f}\nonumber\\
&  =\sum_{f:V\rightarrow\mathbb{N}_{+}}\left[  \underbrace{\text{the
}\mathbb{N}_{+}\text{-coloring }f\text{ of }G\text{ is proper}}%
_{\Longleftrightarrow\ \left(  f\text{ is a proper }\mathbb{N}_{+}%
\text{-coloring of }G\right)  }\right]  \mathbf{x}_{f}\nonumber\\
&  =\sum_{f:V\rightarrow\mathbb{N}_{+}}\left[  f\text{ is a proper }%
\mathbb{N}_{+}\text{-coloring of }G\right]  \mathbf{x}_{f}\nonumber\\
&  =\sum_{\substack{f:V\rightarrow\mathbb{N}_{+}\text{ is a}\\\text{proper
}\mathbb{N}_{+}\text{-coloring of }G}}\underbrace{\left[  f\text{ is a proper
}\mathbb{N}_{+}\text{-coloring of }G\right]  }_{\substack{=1\\\text{(since
}f\text{ is a proper }\mathbb{N}_{+}\text{-coloring of }G\text{)}}%
}\mathbf{x}_{f}\nonumber\\
&  \ \ \ \ \ \ \ \ \ \ +\sum_{\substack{f:V\rightarrow\mathbb{N}_{+}\text{ is
not a}\\\text{proper }\mathbb{N}_{+}\text{-coloring of }G}}\underbrace{\left[
f\text{ is a proper }\mathbb{N}_{+}\text{-coloring of }G\right]
}_{\substack{=0\\\text{(since }f\text{ is not a proper }\mathbb{N}%
_{+}\text{-coloring of }G\text{)}}}\mathbf{x}_{f}\nonumber\\
&  \ \ \ \ \ \ \ \ \ \ \ \ \ \ \ \ \ \ \ \ \left(  \text{since each
}f:V\rightarrow\mathbb{N}_{+}\text{ either is a proper }\mathbb{N}%
_{+}\text{-coloring of }G\text{ or not}\right) \nonumber\\
&  =\sum_{\substack{f:V\rightarrow\mathbb{N}_{+}\text{ is a}\\\text{proper
}\mathbb{N}_{+}\text{-coloring of }G}}\mathbf{x}_{f}+\underbrace{\sum
_{\substack{f:V\rightarrow\mathbb{N}_{+}\text{ is not a}\\\text{proper
}\mathbb{N}_{+}\text{-coloring of }G}}0\mathbf{x}_{f}}_{=0}=\sum
_{\substack{f:V\rightarrow\mathbb{N}_{+}\text{ is a}\\\text{proper }%
\mathbb{N}_{+}\text{-coloring of }G}}\mathbf{x}_{f}\nonumber\\
&  =X_{G} \label{pf.thm.ambichromsym.varis.step1}%
\end{align}
(by (\ref{pf.thm.ambichromsym.varis.XG-def})).

However, for every $f:V\rightarrow\mathbb{N}_{+}$, we have%
\begin{equation}
\sum_{B\subseteq\operatorname*{EQS}\left(  G,f\right)  }\left(  -1\right)
^{\left\vert B\right\vert }\prod_{\substack{K\in\mathfrak{K};\\K\subseteq
B}}a_{K}=\left[  \operatorname*{EQS}\left(  G,f\right)  =\varnothing\right]
\label{pf.thm.ambichromsym.varis.moeb}%
\end{equation}
(by Lemma \ref{lem.ambiNBCm.moeb} (applied to $\mathbb{N}_{+}$ instead of $Y$)).

For every $f:V\rightarrow\mathbb{N}_{+}$, we have%
\begin{equation}
\left\{  F\subseteq E\ \mid\ \operatorname*{union}F\subseteq
\operatorname*{Eqs}f\right\}  =\mathcal{P}\left(  \operatorname*{EQS}\left(
G,f\right)  \right)  \label{pf.thm.ambichromsym.varis.cut}%
\end{equation}
\footnote{\textit{Proof of (\ref{pf.thm.ambichromsym.varis.cut}):} Let
$f:V\rightarrow\mathbb{N}_{+}$.
\par
Let $B\in\left\{  F\subseteq E\ \mid\ \operatorname*{union}F\subseteq
\operatorname*{Eqs}f\right\}  $. Thus, $B$ is a subset $F$ of $E$ satisfying
$\operatorname*{union}F\subseteq\operatorname*{Eqs}f$. In other words, $B$ is
a subset of $E$ and satisfies $\operatorname*{union}B\subseteq
\operatorname*{Eqs}f$. However, Lemma \ref{lem.ambiEqs.union} yields that
$B\subseteq\operatorname*{EQS}\left(  G,f\right)  $ holds if and only if
$\operatorname*{union}B\subseteq\operatorname*{Eqs}f$. Hence, we have
$B\subseteq\operatorname*{EQS}\left(  G,f\right)  $ (since we have
$\operatorname*{union}B\subseteq\operatorname*{Eqs}f$). In other words,
$B\in\mathcal{P}\left(  \operatorname*{EQS}\left(  G,f\right)  \right)  $.
\par
Let us now forget that we fixed $B$. We thus have proven that every
$B\in\left\{  F\subseteq E\ \mid\ \operatorname*{union}F\subseteq
\operatorname*{Eqs}f\right\}  $ satisfies $B\in\mathcal{P}\left(
\operatorname*{EQS}\left(  G,f\right)  \right)  $. In other words,%
\begin{equation}
\left\{  F\subseteq E\ \mid\ \operatorname*{union}F\subseteq
\operatorname*{Eqs}f\right\}  \subseteq\mathcal{P}\left(  \operatorname*{EQS}%
\left(  G,f\right)  \right)  . \label{pf.thm.ambichromsym.varis.cut.pf.1}%
\end{equation}
\par
On the other hand, let $C\in\mathcal{P}\left(  \operatorname*{EQS}\left(
G,f\right)  \right)  $. Thus, $C$ is a subset of $\operatorname*{EQS}\left(
G,f\right)  $. Hence, $C\subseteq\operatorname*{EQS}\left(  G,f\right)
\subseteq E$, so that $C$ is a subset of $E$. Hence, Lemma
\ref{lem.ambiEqs.union} (applied to $B=C$) yields that $C\subseteq
\operatorname*{EQS}\left(  G,f\right)  $ holds if and only if
$\operatorname*{union}C\subseteq\operatorname*{Eqs}f$. Hence, we have
$\operatorname*{union}C\subseteq\operatorname*{Eqs}f$ (since we have
$C\subseteq\operatorname*{EQS}\left(  G,f\right)  $). Thus, $C$ is a subset of
$E$ and satisfies $\operatorname*{union}C\subseteq\operatorname*{Eqs}f$. In
other words, $C$ is a subset $F$ of $E$ satisfying $\operatorname*{union}%
F\subseteq\operatorname*{Eqs}f$. In other words, $C\in\left\{  F\subseteq
E\ \mid\ \operatorname*{union}F\subseteq\operatorname*{Eqs}f\right\}  $.
\par
Let us now forget that we fixed $C$. We thus have proven that every
$C\in\mathcal{P}\left(  \operatorname*{EQS}\left(  G,f\right)  \right)  $
satisfies $C\in\left\{  F\subseteq E\ \mid\ \operatorname*{union}%
F\subseteq\operatorname*{Eqs}f\right\}  $. In other words,%
\[
\mathcal{P}\left(  \operatorname*{EQS}\left(  G,f\right)  \right)
\subseteq\left\{  F\subseteq E\ \mid\ \operatorname*{union}F\subseteq
\operatorname*{Eqs}f\right\}  .
\]
Combining this inclusion with (\ref{pf.thm.ambichromsym.varis.cut.pf.1}), we
obtain $\left\{  F\subseteq E\ \mid\ \operatorname*{union}F\subseteq
\operatorname*{Eqs}f\right\}  =\mathcal{P}\left(  \operatorname*{EQS}\left(
G,f\right)  \right)  $. This proves (\ref{pf.thm.ambichromsym.varis.cut}).}.

For every $f:V\rightarrow\mathbb{N}_{+}$, we have%
\begin{align}
&  \underbrace{\sum_{\substack{B\subseteq E;\\\operatorname*{union}%
B\subseteq\operatorname*{Eqs}f}}}_{\substack{=\sum_{B\in\left\{  F\subseteq
E\ \mid\ \operatorname*{union}F\subseteq\operatorname*{Eqs}f\right\}  }%
=\sum_{B\in\mathcal{P}\left(  \operatorname*{EQS}\left(  G,f\right)  \right)
}\\\text{(because }\left\{  F\subseteq E\ \mid\ \operatorname*{union}%
F\subseteq\operatorname*{Eqs}f\right\}  =\mathcal{P}\left(
\operatorname*{EQS}\left(  G,f\right)  \right)  \\\text{(by
(\ref{pf.thm.ambichromsym.varis.cut})))}}}\left(  -1\right)  ^{\left\vert
B\right\vert }\prod_{\substack{K\in\mathfrak{K};\\K\subseteq B}}a_{K}%
\nonumber\\
&  =\underbrace{\sum_{B\in\mathcal{P}\left(  \operatorname*{EQS}\left(
G,f\right)  \right)  }}_{=\sum_{B\subseteq\operatorname*{EQS}\left(
G,f\right)  }}\left(  -1\right)  ^{\left\vert B\right\vert }\prod
_{\substack{K\in\mathfrak{K};\\K\subseteq B}}a_{K}=\sum_{B\subseteq
\operatorname*{EQS}\left(  G,f\right)  }\left(  -1\right)  ^{\left\vert
B\right\vert }\prod_{\substack{K\in\mathfrak{K};\\K\subseteq B}}a_{K}%
\nonumber\\
&  =\left[  \operatorname*{EQS}\left(  G,f\right)  =\varnothing\right]
\label{pf.thm.ambichromsym.varis.moeb2}%
\end{align}
(by (\ref{pf.thm.ambichromsym.varis.moeb})).

Now, (\ref{pf.thm.ambichromsym.varis.step1}) yields%
\begin{align}
X_{G}  &  =\sum_{f:V\rightarrow\mathbb{N}_{+}}\underbrace{\left[
\operatorname*{EQS}\left(  G,f\right)  =\varnothing\right]  }_{\substack{=\sum
_{\substack{B\subseteq E;\\\operatorname*{union}B\subseteq\operatorname*{Eqs}%
f}}\left(  -1\right)  ^{\left\vert B\right\vert }\prod_{\substack{K\in
\mathfrak{K};\\K\subseteq B}}a_{K}\\\text{(by
(\ref{pf.thm.ambichromsym.varis.moeb2}))}}}\mathbf{x}_{f}\nonumber\\
&  =\sum_{f:V\rightarrow\mathbb{N}_{+}}\left(  \sum_{\substack{B\subseteq
E;\\\operatorname*{union}B\subseteq\operatorname*{Eqs}f}}\left(  -1\right)
^{\left\vert B\right\vert }\prod_{\substack{K\in\mathfrak{K};\\K\subseteq
B}}a_{K}\right)  \mathbf{x}_{f}\nonumber\\
&  =\underbrace{\sum_{f:V\rightarrow\mathbb{N}_{+}}\ \ \sum
_{\substack{B\subseteq E;\\\operatorname*{union}B\subseteq\operatorname*{Eqs}%
f}}}_{=\sum_{B\subseteq E}\ \ \sum_{\substack{f:V\rightarrow\mathbb{N}%
_{+};\\\operatorname*{union}B\subseteq\operatorname*{Eqs}f}}}\left(
-1\right)  ^{\left\vert B\right\vert }\left(  \prod_{\substack{K\in
\mathfrak{K};\\K\subseteq B}}a_{K}\right)  \mathbf{x}_{f}\nonumber\\
&  =\sum_{B\subseteq E}\ \ \sum_{\substack{f:V\rightarrow\mathbb{N}%
_{+};\\\operatorname*{union}B\subseteq\operatorname*{Eqs}f}}\left(  -1\right)
^{\left\vert B\right\vert }\left(  \prod_{\substack{K\in\mathfrak{K}%
;\\K\subseteq B}}a_{K}\right)  \mathbf{x}_{f}\nonumber\\
&  =\sum_{B\subseteq E}\left(  -1\right)  ^{\left\vert B\right\vert }\left(
\prod_{\substack{K\in\mathfrak{K};\\K\subseteq B}}a_{K}\right)  \sum
_{\substack{f:V\rightarrow\mathbb{N}_{+};\\\operatorname*{union}%
B\subseteq\operatorname*{Eqs}f}}\mathbf{x}_{f}.
\label{pf.thm.ambichromsym.varis.step4}%
\end{align}

However, if $B$ is a subset of $E$, then the pair $\left(
V,\operatorname*{union}B\right)  $ is a finite graph (since $V$ is a finite
set and since $\operatorname*{union}B\subseteq\dbinom{V}{2}$), and thus we
have%
\begin{equation}
\sum_{\substack{f:V\rightarrow\mathbb{N}_{+};\\\operatorname*{union}%
B\subseteq\operatorname*{Eqs}f}}\mathbf{x}_{f}=p_{\lambda\left(
V,\operatorname*{union}B\right)  } \label{pf.thm.ambichromsym.varis.p}%
\end{equation}
(by Lemma \ref{lem.Eqs.sum}, applied to $\operatorname*{union}B$ instead of
$B$).

Hence, (\ref{pf.thm.ambichromsym.varis.step4}) becomes%
\begin{align*}
X_{G}  &  =\sum_{B\subseteq E}\left(  -1\right)  ^{\left\vert B\right\vert
}\left(  \prod_{\substack{K\in\mathfrak{K};\\K\subseteq B}}a_{K}\right)
\underbrace{\sum_{\substack{f:V\rightarrow\mathbb{N}_{+}%
;\\\operatorname*{union}B\subseteq\operatorname*{Eqs}f}}\mathbf{x}_{f}%
}_{\substack{=p_{\lambda\left(  V,\operatorname*{union}B\right)  }\\\text{(by
(\ref{pf.thm.ambichromsym.varis.p}))}}}\\
&  =\sum_{B\subseteq E}\left(  -1\right)  ^{\left\vert B\right\vert }\left(
\prod_{\substack{K\in\mathfrak{K};\\K\subseteq B}}a_{K}\right)  p_{\lambda
\left(  V,\operatorname*{union}B\right)  }=\sum_{F\subseteq E}\left(
-1\right)  ^{\left\vert F\right\vert }\left(  \prod_{\substack{K\in
\mathfrak{K};\\K\subseteq F}}a_{K}\right)  p_{\lambda\left(
V,\operatorname*{union}F\right)  }%
\end{align*}
(here, we have renamed the summation index $B$ as $F$). This proves Theorem
\ref{thm.ambichromsym.varis}.
\end{proof}
\end{verlong}

\begin{vershort}
\begin{proof}
[Proof of Corollary \ref{cor.ambichromsym.K-free}.]Analogous to the proof of
Corollary \ref{cor.chromsym.K-free}.
\end{proof}
\end{vershort}

\begin{verlong}
\begin{proof}
[Proof of Corollary \ref{cor.ambichromsym.K-free}.]We can apply Theorem
\ref{thm.ambichromsym.varis} to $0$ instead of $a_{K}$. As a result, we obtain%
\begin{equation}
X_{G}=\sum_{F\subseteq E}\left(  -1\right)  ^{\left\vert F\right\vert }\left(
\prod_{\substack{K\in\mathfrak{K};\\K\subseteq F}}0\right)  p_{\lambda\left(
V,\operatorname*{union}F\right)  }. \label{pf.cor.ambichromsym.K-free.0}%
\end{equation}
Now, if $F$ is any subset of $E$, then%
\begin{equation}
\prod_{\substack{K\in\mathfrak{K};\\K\subseteq F}}0=%
\begin{cases}
1, & \text{if }F\text{ is }\mathfrak{K}\text{-free;}\\
0, & \text{if }F\text{ is not }\mathfrak{K}\text{-free}%
\end{cases}
\label{pf.cor.ambichromsym.K-free.1}%
\end{equation}
\footnote{\textit{Proof of (\ref{pf.cor.ambichromsym.K-free.1}):} The proof of
(\ref{pf.cor.ambichromsym.K-free.1}) is completely analogous to the proof of
(\ref{pf.cor.chromsym.K-free.1}).}.

Thus, (\ref{pf.cor.ambichromsym.K-free.0}) becomes%
\begin{align*}
X_{G}  &  =\sum_{F\subseteq E}\left(  -1\right)  ^{\left\vert F\right\vert
}\underbrace{\left(  \prod_{\substack{K\in\mathfrak{K};\\K\subseteq
F}}0\right)  }_{\substack{=%
\begin{cases}
1, & \text{if }F\text{ is }\mathfrak{K}\text{-free;}\\
0, & \text{if }F\text{ is not }\mathfrak{K}\text{-free}%
\end{cases}
\\\text{(by (\ref{pf.cor.ambichromsym.K-free.1}))}}}p_{\lambda\left(
V,\operatorname*{union}F\right)  }\\
&  =\sum_{F\subseteq E}\left(  -1\right)  ^{\left\vert F\right\vert }%
\begin{cases}
1, & \text{if }F\text{ is }\mathfrak{K}\text{-free;}\\
0, & \text{if }F\text{ is not }\mathfrak{K}\text{-free}%
\end{cases}
\ \ p_{\lambda\left(  V,\operatorname*{union}F\right)  }\\
&  =\sum_{\substack{F\subseteq E;\\F\text{ is }\mathfrak{K}\text{-free}%
}}\left(  -1\right)  ^{\left\vert F\right\vert }\underbrace{%
\begin{cases}
1, & \text{if }F\text{ is }\mathfrak{K}\text{-free;}\\
0, & \text{if }F\text{ is not }\mathfrak{K}\text{-free}%
\end{cases}
}_{\substack{=1\\\text{(since }F\text{ is }\mathfrak{K}\text{-free)}%
}}\ \ p_{\lambda\left(  V,\operatorname*{union}F\right)  }\\
&  \ \ \ \ \ \ \ \ \ \ +\sum_{\substack{F\subseteq E;\\F\text{ is not
}\mathfrak{K}\text{-free}}}\left(  -1\right)  ^{\left\vert F\right\vert
}\underbrace{%
\begin{cases}
1, & \text{if }F\text{ is }\mathfrak{K}\text{-free;}\\
0, & \text{if }F\text{ is not }\mathfrak{K}\text{-free}%
\end{cases}
}_{\substack{=0\\\text{(since }F\text{ is not }\mathfrak{K}\text{-free)}%
}}\ \ p_{\lambda\left(  V,\operatorname*{union}F\right)  }\\
&  \ \ \ \ \ \ \ \ \ \ \ \ \ \ \ \ \ \ \ \ \left(  \text{since each subset
}F\text{ of }E\text{ either is }\mathfrak{K}\text{-free or is not}\right) \\
&  =\sum_{\substack{F\subseteq E;\\F\text{ is }\mathfrak{K}\text{-free}%
}}\left(  -1\right)  ^{\left\vert F\right\vert }p_{\lambda\left(
V,\operatorname*{union}F\right)  }+\underbrace{\sum_{\substack{F\subseteq
E;\\F\text{ is not }\mathfrak{K}\text{-free}}}\left(  -1\right)  ^{\left\vert
F\right\vert }0p_{\lambda\left(  V,\operatorname*{union}F\right)  }}_{=0}\\
&  =\sum_{\substack{F\subseteq E;\\F\text{ is }\mathfrak{K}\text{-free}%
}}\left(  -1\right)  ^{\left\vert F\right\vert }p_{\lambda\left(
V,\operatorname*{union}F\right)  }.
\end{align*}
This proves Corollary \ref{cor.ambichromsym.K-free}.
\end{proof}
\end{verlong}

\begin{vershort}
\begin{proof}
[Proof of Corollary \ref{cor.ambichromsym.NBC}.]Corollary
\ref{cor.ambichromsym.NBC} follows from Corollary
\ref{cor.ambichromsym.K-free} when $\mathfrak{K}$ is set to be the set of
\textbf{all} broken circuits of $G$.
\end{proof}
\end{vershort}

\begin{verlong}
\begin{proof}
[Proof of Corollary \ref{cor.ambichromsym.NBC}.]Let $\mathfrak{K}$ be the set
of all broken circuits of $G$.

Now, just as in the proof of Corollary \ref{cor.chromsym.NBC}, we can prove
the following equality:
\[
\sum_{\substack{F\subseteq E;\\F\text{ is }\mathfrak{K}\text{-free}}%
}=\sum_{\substack{F\subseteq E;\\F\text{ contains no broken}\\\text{circuit of
}G\text{ as a subset}}}
\]
(an equality between summation signs). Now, Corollary
\ref{cor.ambichromsym.K-free} yields%
\[
X_{G}=\underbrace{\sum_{\substack{F\subseteq E;\\F\text{ is }\mathfrak{K}%
\text{-free}}}}_{=\sum_{\substack{F\subseteq E;\\F\text{ contains no
broken}\\\text{circuit of }G\text{ as a subset}}}}\left(  -1\right)
^{\left\vert F\right\vert }p_{\lambda\left(  V,\operatorname*{union}F\right)
}=\sum_{\substack{F\subseteq E;\\F\text{ contains no broken}\\\text{circuit of
}G\text{ as a subset}}}\left(  -1\right)  ^{\left\vert F\right\vert
}p_{\lambda\left(  V,\operatorname*{union}F\right)  }.
\]
This proves Corollary \ref{cor.ambichromsym.NBC}.
\end{proof}
\end{verlong}

\begin{vershort}
\begin{proof}
[Proof of Theorem \ref{thm.ambichromsym.empty}.]This follows from Theorem
\ref{thm.ambichromsym.varis} in the same way as Theorem
\ref{thm.chromsym.empty} follows from Theorem \ref{thm.chromsym.varis}.
\end{proof}
\end{vershort}

\begin{verlong}
\begin{proof}
[Proof of Theorem \ref{thm.ambichromsym.empty}.]Let $X$ be the totally ordered
set $\left\{  1\right\}  $ (equipped with the only possible order on this
set). Let $\ell:E\rightarrow X$ be the function sending each $e\in E$ to $1\in
X$. Let $\mathfrak{K}$ be the empty set. Clearly, $\mathfrak{K}$ is a set of
broken circuits of $G$. Theorem \ref{thm.ambichromsym.varis} (applied to $0$
instead of $a_{K}$) yields%
\begin{align*}
X_{G}  &  =\sum_{F\subseteq E}\left(  -1\right)  ^{\left\vert F\right\vert
}\underbrace{\left(  \prod_{\substack{K\in\mathfrak{K};\\K\subseteq
F}}0\right)  }_{\substack{=\left(  \text{empty product}\right)  \\\text{(since
}\mathfrak{K}\text{ is the empty set)}}}p_{\lambda\left(
V,\operatorname*{union}F\right)  }\\
&  =\sum_{F\subseteq E}\left(  -1\right)  ^{\left\vert F\right\vert
}\underbrace{\left(  \text{empty product}\right)  }_{=1}p_{\lambda\left(
V,\operatorname*{union}F\right)  }=\sum_{F\subseteq E}\left(  -1\right)
^{\left\vert F\right\vert }p_{\lambda\left(  V,\operatorname*{union}F\right)
}.
\end{align*}
This proves Theorem \ref{thm.ambichromsym.empty}.
\end{proof}
\end{verlong}

\subsection{The chromatic polynomial}

We have thus proved analogues of Theorems \ref{thm.chromsym.empty} and
\ref{thm.chromsym.varis} and Corollaries \ref{cor.chromsym.K-free} and
\ref{cor.chromsym.NBC} for ambigraphs. We can just as easily prove analogues
of Theorem \ref{thm.chrompol.exist}, Definition \ref{def.chrompol}, Theorems
\ref{thm.chrompol.empty} and \ref{thm.chrompol.varis} and Corollaries
\ref{cor.chrompol.K-free} and \ref{cor.chrompol.NBC}. Here they are, in the
order in which we have just mentioned them:

\begin{theorem}
\label{thm.ambichrompol.exist}Let $G=\left(  V,E,\varphi\right)  $ be a finite
ambigraph. Then, there exists a unique polynomial $P\in\mathbb{Z}\left[
x\right]  $ such that every $q\in\mathbb{N}$ satisfies%
\[
P\left(  q\right)  =\left(  \text{the number of all proper }\left\{
1,2,\ldots,q\right\}  \text{-colorings of }G\right)  .
\]

\end{theorem}

\begin{definition}
\label{def.ambichrompol}Let $G=\left(  V,E,\varphi\right)  $ be a finite
ambigraph. Theorem \ref{thm.ambichrompol.exist} shows that there exists a
polynomial $P\in\mathbb{Z}\left[  x\right]  $ such that every $q\in\mathbb{N}$
satisfies $P\left(  q\right)  =\left(  \text{the number of all proper
}\left\{  1,2,\ldots,q\right\}  \text{-colorings of }G\right)  $. This
polynomial $P$ is called the \emph{chromatic polynomial} of $G$, and will be
denoted by $\chi_{G}$.
\end{definition}

\begin{theorem}
\label{thm.ambichrompol.empty}Let $G=\left(  V,E,\varphi\right)  $ be a finite
ambigraph. Then,%
\[
\chi_{G}=\sum_{F\subseteq E}\left(  -1\right)  ^{\left\vert F\right\vert
}x^{\operatorname*{conn}\left(  V,\operatorname*{union}F\right)  }.
\]
(Here, of course, the pair $\left(  V,\operatorname*{union}F\right)  $ is
regarded as a graph, and the expression $\operatorname*{conn}\left(
V,\operatorname*{union}F\right)  $ is understood according to Definition
\ref{def.conn}.)
\end{theorem}

\begin{theorem}
\label{thm.ambichrompol.varis}Let $G=\left(  V,E,\varphi\right)  $ be a finite
ambigraph. Let $X$ be a totally ordered set. Let $\ell:E\rightarrow X$ be a
labeling function. Let $\mathfrak{K}$ be some set of broken circuits of $G$
(not necessarily containing all of them). Let $a_{K}$ be an element of
$\mathbf{k}$ for every $K\in\mathfrak{K}$. Then,%
\[
\chi_{G}=\sum_{F\subseteq E}\left(  -1\right)  ^{\left\vert F\right\vert
}\left(  \prod_{\substack{K\in\mathfrak{K};\\K\subseteq F}}a_{K}\right)
x^{\operatorname*{conn}\left(  V,\operatorname*{union}F\right)  }.
\]
(Here, of course, the pair $\left(  V,\operatorname*{union}F\right)  $ is
regarded as a graph, and the expression $\operatorname*{conn}\left(
V,\operatorname*{union}F\right)  $ is understood according to Definition
\ref{def.conn}. Moreover, the polynomial $\chi_{G}\in\mathbb{Z}\left[
x\right]  $ on the left-hand side is regarded as an element of $\mathbf{k}%
\left[  x\right]  $ via the canonical ring morphism $\mathbb{Z}\left[
x\right]  \rightarrow\mathbf{k}\left[  x\right]  $.)
\end{theorem}

\begin{corollary}
\label{cor.ambichrompol.K-free}Let $G=\left(  V,E,\varphi\right)  $ be a
finite ambigraph. Let $X$ be a totally ordered set. Let $\ell:E\rightarrow X$
be a labeling function. Let $\mathfrak{K}$ be some set of broken circuits of
$G$ (not necessarily containing all of them). Then,%
\[
\chi_{G}=\sum_{\substack{F\subseteq E;\\F\text{ is }\mathfrak{K}\text{-free}%
}}\left(  -1\right)  ^{\left\vert F\right\vert }x^{\operatorname*{conn}\left(
V,\operatorname*{union}F\right)  }.
\]

\end{corollary}

\begin{corollary}
\label{cor.ambichrompol.NBC}Let $G=\left(  V,E,\varphi\right)  $ be a finite
ambigraph. Let $X$ be a totally ordered set. Let $\ell:E\rightarrow X$ be a
labeling function. Then,%
\[
\chi_{G}=\sum_{\substack{F\subseteq E;\\F\text{ contains no broken}%
\\\text{circuit of }G\text{ as a subset}}}\left(  -1\right)  ^{\left\vert
F\right\vert }x^{\operatorname*{conn}\left(  V,\operatorname*{union}F\right)
}.
\]

\end{corollary}

The proofs of all these results are analogous to the proofs of the
corresponding results from Section \ref{sec.chrompol}, so we leave them all to
the reader.

One may reasonably wonder whether Corollary \ref{cor.chrompol.NBCfor} has an
analogue for ambigraphs as well, i.e., whether one can replace the exponent
$\operatorname*{conn}\left(  V,\operatorname*{union}F\right)  $ in Corollary
\ref{cor.ambichrompol.NBC} by something simpler when $\ell$ is injective. In
the case of a hypergraph, Dohmen has obtained such a result (\cite[Theorem
2.1]{Dohmen95}) under the additional condition that each cycle of $G$ have at
least one singleton edgery. Unfortunately, for ambigraphs, such a
simplification does not appear possible (even under a condition like Dohmen's).

\section{\label{sec.weigh}Weighted and noncommutative versions}

In the recent decades, the chromatic symmetric function of a graph has been
generalized in several directions. Two of them are the \emph{chromatic
symmetric function of a weighted graph} as defined by Crew and Spirkl
(\cite[\S 3]{CreSpi19}), and the \emph{noncommutative chromatic symmetric
function} of Gebhard and Sagan (\cite[\S 3]{GebSag01}). In this section, we
will recall the definitions of both of these generalizations, and extend our
results to them. (The extensions will be fairly mechanical, as all the hard
work has already been done.)

\subsection{Weighted graphs and their chromatic symmetric functions}

For us, a \emph{weighted graph} will just mean a pair consisting of a graph
$G=\left(  V,E\right)  $ and a weight function on $V$. Weight functions are
defined as follows:

\begin{definition}
\label{def.weight}Let $V$ be a set. A \emph{weight function} on $V$ means a
function $w:V\rightarrow\mathbb{N}_{+}$. If $w:V\rightarrow\mathbb{N}_{+}$ is
a weight function on $V$, then the \emph{weight} of an element $v\in V$ is
defined to be the positive integer $w\left(  v\right)  \in\mathbb{N}_{+}$.
\end{definition}

Thus, a weight function on a set $V$ just assigns a \textquotedblleft
weight\textquotedblright\ (a positive integer) to each element of $V$. Given
such a weight function for a graph $G=\left(  V,E\right)  $, we can define a
\textquotedblleft weighted chromatic symmetric function\textquotedblright:

\begin{definition}
\label{def.wchromsym}Let $G=\left(  V,E\right)  $ be a finite graph. Let
$w:V\rightarrow\mathbb{N}_{+}$ be a weight function on $V$.

\textbf{(a)} For every $\mathbb{N}_{+}$-coloring $f:V\rightarrow\mathbb{N}%
_{+}$ of $G$, we let $\mathbf{x}_{f,w}$ denote the monomial $\prod_{v\in
V}x_{f\left(  v\right)  }^{w\left(  v\right)  }$ in the indeterminates
$x_{1},x_{2},x_{3},\ldots$.

\textbf{(b)} We define a power series $X_{G,w}\in\mathbf{k}\left[  \left[
x_{1},x_{2},x_{3},\ldots\right]  \right]  $ by%
\[
X_{G,w}=\sum_{\substack{f:V\rightarrow\mathbb{N}_{+}\text{ is a}\\\text{proper
}\mathbb{N}_{+}\text{-coloring of }G}}\mathbf{x}_{f,w}.
\]

This power series $X_{G,w}$ is called the \emph{chromatic symmetric function}
of $\left(  G,w\right)  $.
\end{definition}

This chromatic symmetric function $X_{G,w}$ was introduced by Crew and Spirkl
in \cite[(1)]{CreSpi19} (where it was denoted $X_{\left(  G,w\right)  }$). It
generalizes the original chromatic symmetric function $X_{G}$, which is
obtained when all the weights $w\left(  v\right)  $ are $1$:

\begin{example}
\label{exa.wchromsym.w=1}Let $G=\left(  V,E\right)  $ be a finite graph. Let
$w:V\rightarrow\mathbb{N}_{+}$ be the weight function that sends each $v\in V$
to $1$. Then, for every $\mathbb{N}_{+}$-coloring $f:V\rightarrow
\mathbb{N}_{+}$ of $G$, we have%
\begin{align}
\mathbf{x}_{f,w}  &  =\prod_{v\in V}\underbrace{x_{f\left(  v\right)
}^{w\left(  v\right)  }}_{\substack{=x_{f\left(  v\right)  }\\\text{(since
}w\left(  v\right)  =1\\\text{(by the definition of }w\text{))}}%
}\ \ \ \ \ \ \ \ \ \ \left(  \text{by Definition \ref{def.wchromsym}
\textbf{(a)}}\right) \nonumber\\
&  =\prod_{v\in V}x_{f\left(  v\right)  }=\mathbf{x}_{f}
\label{eq.exa.wchromsym.w=1.1}%
\end{align}
(see Definition \ref{def.chromsym} \textbf{(a)} for the definition of
$\mathbf{x}_{f}$). Thus, Definition \ref{def.wchromsym} \textbf{(b)} yields%
\[
X_{G,w}=\sum_{\substack{f:V\rightarrow\mathbb{N}_{+}\text{ is a}\\\text{proper
}\mathbb{N}_{+}\text{-coloring of }G}}\underbrace{\mathbf{x}_{f,w}%
}_{\substack{=\mathbf{x}_{f}\\\text{(by (\ref{eq.exa.wchromsym.w=1.1}))}%
}}=\sum_{\substack{f:V\rightarrow\mathbb{N}_{+}\text{ is a}\\\text{proper
}\mathbb{N}_{+}\text{-coloring of }G}}\mathbf{x}_{f}=X_{G}%
\]
(by Definition \ref{def.chromsym} \textbf{(b)}).
\end{example}

\subsection{The weight of a subset and the partition $\lambda\left(
G,w\right)  $}

To state Whitney-like formulas for $X_{G,w}$, we need to adapt the partition
$\lambda\left(  G\right)  $ defined in Definition \ref{def.connectedness}
\textbf{(b)} to the case of a weighted graph. This adaptation consists in
replacing the size of each connected component by its \emph{weight}. Here, the
weight of a subset of $V$ is defined as follows:

\begin{definition}
\label{def.weight-of-subset}Let $V$ be a finite set. Let $w:V\rightarrow
\mathbb{N}_{+}$ be a weight function on $V$. Let $S$ be a subset of $V$. Then,
the \emph{weight} of $S$ means the nonnegative integer $\sum_{v\in S}w\left(
v\right)  $. This weight is denoted by $w\left(  S\right)  $.
\end{definition}

Note that if the subset $S$ in this definition is nonempty, then its weight
$w\left(  S\right)  $ is a positive integer, since it is defined as the
nonempty sum $\sum_{v\in S}w\left(  v\right)  $ of the positive weights
$w\left(  v\right)  $.

For example, $w\left(  \left\{  2,5,6\right\}  \right)  =w\left(  2\right)
+w\left(  5\right)  +w\left(  6\right)  $ (if $\left\{  2,5,6\right\}  $ is a
subset of $V$).

Now, we can define the analogue of the partition $\lambda\left(  G\right)  $
for a weighted graph:

\begin{definition}
\label{def.wlambda}Let $G=\left(  V,E\right)  $ be a finite graph. Let
$w:V\rightarrow\mathbb{N}_{+}$ be a weight function on $V$. We let
$\lambda\left(  G,w\right)  $ denote the list of the weights of all connected
components of $G$, in weakly decreasing order. (Each connected component
should contribute only one entry to the list.) We view $\lambda\left(
G,w\right)  $ as a partition (since $\lambda\left(  G,w\right)  $ is a weakly
decreasing finite list of positive integers).
\end{definition}

\begin{example}
Let $G=\left(  V,E\right)  $ be the finite graph with $V=\left\{
1,2,3,4,5\right\}  $ and $E=\left\{  \left\{  1,3\right\}  ,\ \left\{
2,4\right\}  ,\ \left\{  2,5\right\}  \right\}  $. Then, the connected
components of $G$ are $A=\left\{  1,3\right\}  $ and $B=\left\{
2,4,5\right\}  $.

Let $w:V\rightarrow\mathbb{N}_{+}$ be the weight function given by $w\left(
1\right)  =6$ and $w\left(  2\right)  =9$ and $w\left(  3\right)  =6$ and
$w\left(  4\right)  =1$ and $w\left(  5\right)  =3$. Then, the weights of the
connected components $A$ and $B$ are
\begin{align*}
w\left(  A\right)   &  =w\left(  \left\{  1,3\right\}  \right)  =w\left(
1\right)  +w\left(  3\right)  =6+6=12\ \ \ \ \ \ \ \ \ \ \text{and}\\
w\left(  B\right)   &  =w\left(  \left\{  2,4,5\right\}  \right)  =w\left(
2\right)  +w\left(  4\right)  +w\left(  5\right)  =9+1+3=13.
\end{align*}
Hence, the partition $\lambda\left(  G,w\right)  $ is the list of these two
weights $12$ and $13$, in weakly decreasing order. In other words,
$\lambda\left(  G,w\right)  =\left(  13,12\right)  $.
\end{example}

\subsection{Formulas for $X_{G,w}$}

We are now ready to state analogues of Theorems \ref{thm.chromsym.empty} and
\ref{thm.chromsym.varis} and Corollaries \ref{cor.chromsym.K-free} and
\ref{cor.chromsym.NBC} for weighted graphs:

\begin{theorem}
\label{thm.wchromsym.empty}Let $G=\left(  V,E\right)  $ be a finite graph. Let
$w:V\rightarrow\mathbb{N}_{+}$ be a weight function on $V$. Then,%
\[
X_{G,w}=\sum_{F\subseteq E}\left(  -1\right)  ^{\left\vert F\right\vert
}p_{\lambda\left(  \left(  V,F\right)  ,w\right)  }.
\]
(Here, of course, the pair $\left(  V,F\right)  $ is regarded as a graph, and
the expression $\lambda\left(  \left(  V,F\right)  ,w\right)  $ is understood
according to Definition \ref{def.wlambda}.)
\end{theorem}

\begin{theorem}
\label{thm.wchromsym.varis}Let $G=\left(  V,E\right)  $ be a finite graph. Let
$w:V\rightarrow\mathbb{N}_{+}$ be a weight function on $V$. Let $X$ be a
totally ordered set. Let $\ell:E\rightarrow X$ be a labeling function. Let
$\mathfrak{K}$ be some set of broken circuits of $G$ (not necessarily
containing all of them). Let $a_{K}$ be an element of $\mathbf{k}$ for every
$K\in\mathfrak{K}$. Then,%
\[
X_{G,w}=\sum_{F\subseteq E}\left(  -1\right)  ^{\left\vert F\right\vert
}\left(  \prod_{\substack{K\in\mathfrak{K};\\K\subseteq F}}a_{K}\right)
p_{\lambda\left(  \left(  V,F\right)  ,w\right)  }.
\]
(Here, of course, the pair $\left(  V,F\right)  $ is regarded as a graph, and
the expression $\lambda\left(  \left(  V,F\right)  ,w\right)  $ is understood
according to Definition \ref{def.wlambda}.)
\end{theorem}

\begin{corollary}
\label{cor.wchromsym.K-free}Let $G=\left(  V,E\right)  $ be a finite graph.
Let $w:V\rightarrow\mathbb{N}_{+}$ be a weight function on $V$. Let $X$ be a
totally ordered set. Let $\ell:E\rightarrow X$ be a labeling function. Let
$\mathfrak{K}$ be some set of broken circuits of $G$ (not necessarily
containing all of them). Then,%
\[
X_{G,w}=\sum_{\substack{F\subseteq E;\\F\text{ is }\mathfrak{K}\text{-free}%
}}\left(  -1\right)  ^{\left\vert F\right\vert }p_{\lambda\left(  \left(
V,F\right)  ,w\right)  }.
\]

\end{corollary}

\begin{corollary}
\label{cor.wchromsym.NBC}Let $G=\left(  V,E\right)  $ be a finite graph. Let
$w:V\rightarrow\mathbb{N}_{+}$ be a weight function on $V$. Let $X$ be a
totally ordered set. Let $\ell:E\rightarrow X$ be a labeling function. Then,%
\[
X_{G,w}=\sum_{\substack{F\subseteq E;\\F\text{ contains no broken}%
\\\text{circuit of }G\text{ as a subset}}}\left(  -1\right)  ^{\left\vert
F\right\vert }p_{\lambda\left(  \left(  V,F\right)  ,w\right)  }.
\]

\end{corollary}

Note that Theorem \ref{thm.wchromsym.empty} is a result by Crew and Spirkl
(namely, \cite[Lemma 3]{CreSpi19}).

\subsection{Proofs}

In this section, we shall prove Theorem \ref{thm.wchromsym.varis}, Corollary
\ref{cor.wchromsym.K-free}, Corollary \ref{cor.wchromsym.NBC} and Theorem
\ref{thm.wchromsym.empty}. This will be fairly easy, as many of our above
lemmas can be reused without any change. However, we need the following
weighted version of Lemma \ref{lem.Eqs.sum}:

\begin{lemma}
\label{lem.wEqs.sum}Let $\left(  V,B\right)  $ be a finite graph. Let
$w:V\rightarrow\mathbb{N}_{+}$ be a weight function on $V$. Then,%
\[
\sum_{\substack{f:V\rightarrow\mathbb{N}_{+};\\B\subseteq\operatorname*{Eqs}%
f}}\mathbf{x}_{f,w}=p_{\lambda\left(  \left(  V,B\right)  ,w\right)  }.
\]
(Here, $\mathbf{x}_{f,w}$ is defined as in Definition \ref{def.wchromsym}
\textbf{(a)}, and the expression $\lambda\left(  \left(  V,B\right)
,w\right)  $ is understood according to Definition \ref{def.wlambda}.)
\end{lemma}

\begin{vershort}
\begin{proof}
[Proof of Lemma \ref{lem.wEqs.sum}.]This is almost completely analogous to the
proof of Lemma \ref{lem.Eqs.sum} that we gave long ago. (Replace each size
$\left\vert C_{i}\right\vert $ by the weight $w\left(  C_{i}\right)  $;
replace all monomials $\mathbf{x}_{g}$ by their weighted analogues
$\mathbf{x}_{g,w}$; and of course, replace $\lambda\left(  V,B\right)  $ by
$\lambda\left(  \left(  V,B\right)  ,w\right)  $.)
\end{proof}
\end{vershort}

\begin{verlong}
\begin{proof}
[Proof of Lemma \ref{lem.wEqs.sum}.]Let $\sim$ denote the equivalence relation
$\sim_{\left(  V,B\right)  }$ (defined as in Definition
\ref{def.connectedness} \textbf{(a)}). Then, the connected components of
$\left(  V,B\right)  $ are the elements of $V/\left(  \sim\right)  $. (This
has already been shown in our proof of Lemma \ref{lem.Eqs.sum} above.)

A set $\left(  \mathbb{N}_{+}\right)  _{\sim}^{V}$ is defined (according to
Definition \ref{def.relquot.maps} \textbf{(b)}).

Proposition \ref{prop.relquot.uniprop} \textbf{(b)} (applied to $X=V$ and
$Y=\mathbb{N}_{+}$) shows that the map\footnote{Here, the map $\pi_{V}$ is
defined as in Definition \ref{def.relquot}.}%
\[
\left(  \mathbb{N}_{+}\right)  ^{V/\left(  \sim\right)  }\rightarrow\left(
\mathbb{N}_{+}\right)  _{\sim}^{V},\ \ \ \ \ \ \ \ \ \ f\mapsto f\circ\pi_{V}%
\]
is a bijection.

We have the following equality of summation signs:
\[
\sum_{\substack{f:V\rightarrow\mathbb{N}_{+};\\B\subseteq\operatorname*{Eqs}%
f}}=\sum_{f\in\left(  \mathbb{N}_{+}\right)  _{\sim}^{V}}%
\]
(indeed, this has already been shown in our proof of Lemma \ref{lem.Eqs.sum}
above). Hence,%
\begin{equation}
\underbrace{\sum_{\substack{f:V\rightarrow\mathbb{N}_{+};\\B\subseteq
\operatorname*{Eqs}f}}}_{=\sum_{f\in\left(  \mathbb{N}_{+}\right)  _{\sim}%
^{V}}}\mathbf{x}_{f,w}=\sum_{f\in\left(  \mathbb{N}_{+}\right)  _{\sim}^{V}%
}\mathbf{x}_{f,w}=\sum_{f\in\left(  \mathbb{N}_{+}\right)  ^{V/\left(
\sim\right)  }}\mathbf{x}_{f\circ\pi_{V},w} \label{pf.lem.wEqs.sum.1}%
\end{equation}
(here, we have substituted $f\circ\pi_{V}$ for $f$ in the sum, since the map
$\left(  \mathbb{N}_{+}\right)  ^{V/\left(  \sim\right)  }\rightarrow\left(
\mathbb{N}_{+}\right)  _{\sim}^{V},\ f\mapsto f\circ\pi_{V}$ is a bijection).

Now, let $\left(  C_{1},C_{2},\ldots,C_{k}\right)  $ be a list of all
connected components of $\left(  V,B\right)  $, ordered such that $w\left(
C_{1}\right)  \geq w\left(  C_{2}\right)  \geq\cdots\geq w\left(
C_{k}\right)  $.\ \ \ \ \footnote{Every connected component of $\left(
V,B\right)  $ should appear exactly once in this list.} Then, $\left(
w\left(  C_{1}\right)  ,w\left(  C_{2}\right)  ,\ldots,w\left(  C_{k}\right)
\right)  $ is the list of the weights of all connected components of $\left(
V,B\right)  $, in weakly decreasing order (since $w\left(  C_{1}\right)  \geq
w\left(  C_{2}\right)  \geq\cdots\geq w\left(  C_{k}\right)  $). In other
words, $\left(  w\left(  C_{1}\right)  ,w\left(  C_{2}\right)  ,\ldots
,w\left(  C_{k}\right)  \right)  $ is $\lambda\left(  \left(  V,B\right)
,w\right)  $ (since $\lambda\left(  \left(  V,B\right)  ,w\right)  $ is the
list of the weights of all connected components of $\left(  V,B\right)  $, in
weakly decreasing order (by the definition of $\lambda\left(  \left(
V,B\right)  ,w\right)  $)). In other words,
\[
\lambda\left(  \left(  V,B\right)  ,w\right)  =\left(  w\left(  C_{1}\right)
,w\left(  C_{2}\right)  ,\ldots,w\left(  C_{k}\right)  \right)  .
\]
Thus, (\ref{eq.def.powersum2.finite-expression}) (applied to $\lambda\left(
\left(  V,B\right)  ,w\right)  $ and $w\left(  C_{i}\right)  $ instead of
$\lambda$ and $\lambda_{i}$) shows that{}%
\begin{equation}
p_{\lambda\left(  \left(  V,B\right)  ,w\right)  }=p_{w\left(  C_{1}\right)
}p_{w\left(  C_{2}\right)  }\cdots p_{w\left(  C_{k}\right)  }=\prod_{i=1}%
^{k}p_{w\left(  C_{i}\right)  }. \label{pf.lem.wEqs.sum.p1}%
\end{equation}

However, for every $i\in\left\{  1,2,\ldots,k\right\}  $, we have%
\begin{equation}
p_{w\left(  C_{i}\right)  }=\sum_{s\in\mathbb{N}_{+}}x_{s}^{w\left(
C_{i}\right)  } \label{pf.lem.wEqs.sum.pwCi=}%
\end{equation}
\footnote{\textit{Proof.} Let $i\in\left\{  1,2,\ldots,k\right\}  $. Then,
$C_{i}$ is a connected component of $\left(  V,B\right)  $ (since $\left(
C_{1},C_{2},\ldots,C_{k}\right)  $ is a list of all connected components of
$\left(  V,B\right)  $). Hence, $C_{i}$ is a nonempty subset of $V$ (since
every connected component of $\left(  V,B\right)  $ is a nonempty subset of
$V$). Hence, $w\left(  C_{i}\right)  $ is a positive integer (since $w\left(
S\right)  $ is a positive integer whenever $S$ is a nonempty subset of $V$).
Thus, (\ref{eq.def.powersum.pn}) (applied to $n=w\left(  C_{i}\right)  $)
shows that $p_{w\left(  C_{i}\right)  }=\underbrace{\sum_{j\geq1}}%
_{=\sum_{j\in\mathbb{N}_{+}}}x_{j}^{w\left(  C_{i}\right)  }=\sum
_{j\in\mathbb{N}_{+}}x_{j}^{w\left(  C_{i}\right)  }=\sum_{s\in\mathbb{N}_{+}%
}x_{s}^{w\left(  C_{i}\right)  }$ (here, we have renamed the summation index
$j$ as $s$). Qed.}. Hence, (\ref{pf.lem.wEqs.sum.p1}) becomes%
\begin{align}
p_{\lambda\left(  \left(  V,B\right)  ,w\right)  }  &  =\prod_{i=1}%
^{k}\underbrace{p_{w\left(  C_{i}\right)  }}_{\substack{=\sum_{s\in
\mathbb{N}_{+}}x_{s}^{w\left(  C_{i}\right)  }\\\text{(by
(\ref{pf.lem.wEqs.sum.pwCi=}))}}}=\prod_{i=1}^{k}\ \ \sum_{s\in\mathbb{N}_{+}%
}x_{s}^{w\left(  C_{i}\right)  }\nonumber\\
&  =\sum_{\left(  s_{1},s_{2},\ldots,s_{k}\right)  \in\left(  \mathbb{N}%
_{+}\right)  ^{k}}\ \ \prod_{i=1}^{k}x_{s_{i}}^{w\left(  C_{i}\right)  }
\label{pf.lem.wEqs.sum.p2}%
\end{align}
(by the product rule).

Recall that $\left(  C_{1},C_{2},\ldots,C_{k}\right)  $ is a list of all
connected components of $\left(  V,B\right)  $. In other words, $\left(
C_{1},C_{2},\ldots,C_{k}\right)  $ is a list of all elements of $V/\left(
\sim\right)  $ (since the elements of $V/\left(  \sim\right)  $ are the
connected components of $\left(  V,B\right)  $). Moreover, every element of
$V/\left(  \sim\right)  $ appears exactly once in this list $\left(
C_{1},C_{2},\ldots,C_{k}\right)  $ (since the entries of the list $\left(
C_{1},C_{2},\ldots,C_{k}\right)  $ are pairwise distinct\footnote{since every
connected component of $\left(  V,B\right)  $ appears exactly once in this
list}). Thus, $\left(  C_{1},C_{2},\ldots,C_{k}\right)  $ is a list of all
elements of $V/\left(  \sim\right)  $, and contains each of these elements
exactly once. Hence, the map%
\begin{align*}
\left(  \mathbb{N}_{+}\right)  ^{V/\left(  \sim\right)  }  &  \rightarrow
\left(  \mathbb{N}_{+}\right)  ^{k},\\
f  &  \mapsto\left(  f\left(  C_{1}\right)  ,f\left(  C_{2}\right)
,\ldots,f\left(  C_{k}\right)  \right)
\end{align*}
is a bijection (by Lemma \ref{lem.function-count}, applied to $W=V/\left(
\sim\right)  $ and $Y=\mathbb{N}_{+}$).

For every $\gamma\in V/\left(  \sim\right)  $, we have%
\begin{equation}
\pi_{V}^{-1}\left(  \gamma\right)  =\gamma. \label{pf.lem.wEqs.sum.pi-1}%
\end{equation}
(This has already been shown in our proof of Lemma \ref{lem.Eqs.sum} above.)

Also, the map $\left\{  1,2,\ldots,k\right\}  \rightarrow V/\left(
\sim\right)  ,\ i\mapsto C_{i}$ is a bijection (since $\left(  C_{1}%
,C_{2},\ldots,C_{k}\right)  $ is a list of all elements of $V/\left(
\sim\right)  $, and contains each of these elements exactly once).

We have%
\begin{equation}
\mathbf{x}_{f\circ\pi_{V},w}=\prod_{i=1}^{k}x_{f\left(  C_{i}\right)
}^{w\left(  C_{i}\right)  }\ \ \ \ \ \ \ \ \ \ \text{for every }f\in\left(
\mathbb{N}_{+}\right)  ^{V/\left(  \sim\right)  } \label{pf.lem.wEqs.sum.3}%
\end{equation}
\footnote{\textit{Proof of (\ref{pf.lem.wEqs.sum.3}):} Let $f\in\left(
\mathbb{N}_{+}\right)  ^{V/\left(  \sim\right)  }$. Then, the definition of
$\mathbf{x}_{f\circ\pi_{V},w}$ yields%
\begin{align*}
\mathbf{x}_{f\circ\pi_{V},w}  &  =\prod_{v\in V}x_{\left(  f\circ\pi
_{V}\right)  \left(  v\right)  }^{w\left(  v\right)  }=\prod_{\gamma\in
V/\left(  \sim\right)  }\ \ \underbrace{\prod_{\substack{v\in V;\\\pi
_{V}\left(  v\right)  =\gamma}}}_{\substack{=\prod_{v\in\pi_{V}^{-1}\left(
\gamma\right)  }=\prod_{v\in\gamma}\\\text{(since }\pi_{V}^{-1}\left(
\gamma\right)  =\gamma\\\text{(by (\ref{pf.lem.wEqs.sum.pi-1})))}%
}}\underbrace{x_{\left(  f\circ\pi_{V}\right)  \left(  v\right)  }^{w\left(
v\right)  }}_{\substack{=x_{f\left(  \gamma\right)  }^{w\left(  v\right)
}\\\text{(since }\left(  f\circ\pi_{V}\right)  \left(  v\right)  =f\left(
\pi_{V}\left(  v\right)  \right)  =f\left(  \gamma\right)  \\\text{(since }%
\pi_{V}\left(  v\right)  =\gamma\text{))}}}\\
&  \ \ \ \ \ \ \ \ \ \ \left(
\begin{array}
[c]{c}%
\text{because for every }v\in V\text{, there exists a unique }\gamma\in
V/\left(  \sim\right) \\
\text{such that }\pi_{V}\left(  v\right)  =\gamma\text{ (since }\pi_{V}\text{
is a map }V\rightarrow V/\left(  \sim\right)  \text{)}%
\end{array}
\right) \\
&  =\prod_{\gamma\in V/\left(  \sim\right)  }\ \ \underbrace{\prod_{v\in
\gamma}x_{f\left(  \gamma\right)  }^{w\left(  v\right)  }}%
_{\substack{=x_{f\left(  \gamma\right)  }^{\sum_{v\in\gamma}w\left(  v\right)
}=x_{f\left(  \gamma\right)  }^{w\left(  \gamma\right)  }\\\text{(because
}\sum_{v\in\gamma}w\left(  v\right)  =w\left(  \gamma\right)  \\\text{(since
the weight }w\left(  \gamma\right)  \text{ of the subset }\gamma\\\text{is
defined to be }\sum_{v\in\gamma}w\left(  v\right)  \text{))}}}=\prod
_{\gamma\in V/\left(  \sim\right)  }x_{f\left(  \gamma\right)  }^{w\left(
\gamma\right)  }=\prod_{i\in\left\{  1,2,\ldots,k\right\}  }x_{f\left(
C_{i}\right)  }^{w\left(  C_{i}\right)  }%
\end{align*}
(here, we have substituted $C_{i}$ for $\gamma$ in the product, since the map
$\left\{  1,2,\ldots,k\right\}  \rightarrow V/\left(  \sim\right)  ,\ i\mapsto
C_{i}$ is a bijection). Thus,%
\[
\mathbf{x}_{f\circ\pi_{V},w}=\underbrace{\prod_{i\in\left\{  1,2,\ldots
,k\right\}  }}_{=\prod_{i=1}^{k}}x_{f\left(  C_{i}\right)  }^{w\left(
C_{i}\right)  }=\prod_{i=1}^{k}x_{f\left(  C_{i}\right)  }^{w\left(
C_{i}\right)  }.
\]
This proves (\ref{pf.lem.wEqs.sum.3}).}.

Now, (\ref{pf.lem.wEqs.sum.1}) becomes%
\begin{align*}
\sum_{\substack{f:V\rightarrow\mathbb{N}_{+};\\B\subseteq\operatorname*{Eqs}%
f}}\mathbf{x}_{f,w}  &  =\sum_{f\in\left(  \mathbb{N}_{+}\right)  ^{V/\left(
\sim\right)  }}\underbrace{\mathbf{x}_{f\circ\pi_{V},w}}_{\substack{=\prod
_{i=1}^{k}x_{f\left(  C_{i}\right)  }^{w\left(  C_{i}\right)  }\\\text{(by
(\ref{pf.lem.wEqs.sum.3}))}}}=\sum_{f\in\left(  \mathbb{N}_{+}\right)
^{V/\left(  \sim\right)  }}\ \ \prod_{i=1}^{k}x_{f\left(  C_{i}\right)
}^{w\left(  C_{i}\right)  }\\
&  =\sum_{\left(  s_{1},s_{2},\ldots,s_{k}\right)  \in\left(  \mathbb{N}%
_{+}\right)  ^{k}}\ \ \prod_{i=1}^{k}x_{s_{i}}^{w\left(  C_{i}\right)  }%
\end{align*}
(here, we have substituted $\left(  s_{1},s_{2},\ldots,s_{k}\right)  $ for
$\left(  f\left(  C_{1}\right)  ,f\left(  C_{2}\right)  ,\ldots,f\left(
C_{k}\right)  \right)  $ in the sum, since the map $\left(  \mathbb{N}%
_{+}\right)  ^{V/\left(  \sim\right)  }\rightarrow\left(  \mathbb{N}%
_{+}\right)  ^{k},\ f\mapsto\left(  f\left(  C_{1}\right)  ,f\left(
C_{2}\right)  ,\ldots,f\left(  C_{k}\right)  \right)  $ is a bijection).
Comparing this with (\ref{pf.lem.wEqs.sum.p2}), we obtain $\sum
_{\substack{f:V\rightarrow\mathbb{N}_{+};\\B\subseteq\operatorname*{Eqs}%
f}}\mathbf{x}_{f,w}=p_{\lambda\left(  \left(  V,B\right)  ,w\right)  }$. This
proves Lemma \ref{lem.wEqs.sum}.
\end{proof}
\end{verlong}

\begin{vershort}
It is now straightforward to adapt our above proofs of Theorem
\ref{thm.chromsym.varis}, Corollary \ref{cor.chromsym.K-free}, Corollary
\ref{cor.chromsym.NBC} and Theorem \ref{thm.chromsym.empty} to obtain proofs
of Theorem \ref{thm.wchromsym.varis}, Corollary \ref{cor.wchromsym.K-free},
Corollary \ref{cor.wchromsym.NBC} and Theorem \ref{thm.wchromsym.empty}. Of
course, Lemma \ref{lem.wEqs.sum} needs to be used instead of Lemma
\ref{lem.Eqs.sum}, but everything else stays almost completely unchanged. We
leave the details to the reader.
\end{vershort}

\begin{verlong}
It is now straightforward to adapt our above proofs of Theorem
\ref{thm.chromsym.varis}, Corollary \ref{cor.chromsym.K-free}, Corollary
\ref{cor.chromsym.NBC} and Theorem \ref{thm.chromsym.empty} to obtain proofs
of Theorem \ref{thm.wchromsym.varis}, Corollary \ref{cor.wchromsym.K-free},
Corollary \ref{cor.wchromsym.NBC} and Theorem \ref{thm.wchromsym.empty}:

\begin{proof}
[Proof of Theorem \ref{thm.wchromsym.varis}.]We have
\begin{equation}
X_{G,w}=\sum_{\substack{f:V\rightarrow\mathbb{N}_{+}\text{ is a}\\\text{proper
}\mathbb{N}_{+}\text{-coloring of }G}}\mathbf{x}_{f,w}
\label{pf.thm.wchromsym.varis.XG-def}%
\end{equation}
(by the definition of $X_{G,w}$). Now, if $f:V\rightarrow\mathbb{N}_{+}$ is a
map, then we have the following logical equivalence:%
\begin{equation}
\left(  \text{the }\mathbb{N}_{+}\text{-coloring }f\text{ of }G\text{ is
proper}\right)  \ \Longleftrightarrow\ \left(  E\cap\operatorname*{Eqs}%
f=\varnothing\right)  \label{pf.thm.wchromsym.varis.equiv}%
\end{equation}
(because the $\mathbb{N}_{+}$-coloring $f$ of $G$ is proper if and only if
$E\cap\operatorname*{Eqs}f=\varnothing$\ \ \ \ \footnote{by Lemma
\ref{lem.Eqs.proper} (applied to $\mathbb{N}_{+}$ instead of $X$)}). Now,%
\begin{align}
&  \sum_{f:V\rightarrow\mathbb{N}_{+}}\left[  \underbrace{E\cap
\operatorname*{Eqs}f=\varnothing}_{\substack{\Longleftrightarrow\ \left(
\text{the }\mathbb{N}_{+}\text{-coloring }f\text{ of }G\text{ is
proper}\right)  \\\text{(by (\ref{pf.thm.wchromsym.varis.equiv}))}}}\right]
\mathbf{x}_{f,w}\nonumber\\
&  =\sum_{f:V\rightarrow\mathbb{N}_{+}}\left[  \underbrace{\text{the
}\mathbb{N}_{+}\text{-coloring }f\text{ of }G\text{ is proper}}%
_{\Longleftrightarrow\ \left(  f\text{ is a proper }\mathbb{N}_{+}%
\text{-coloring of }G\right)  }\right]  \mathbf{x}_{f,w}\nonumber\\
&  =\sum_{f:V\rightarrow\mathbb{N}_{+}}\left[  f\text{ is a proper }%
\mathbb{N}_{+}\text{-coloring of }G\right]  \mathbf{x}_{f,w}\nonumber\\
&  =\sum_{\substack{f:V\rightarrow\mathbb{N}_{+}\text{ is a}\\\text{proper
}\mathbb{N}_{+}\text{-coloring of }G}}\underbrace{\left[  f\text{ is a proper
}\mathbb{N}_{+}\text{-coloring of }G\right]  }_{\substack{=1\\\text{(since
}f\text{ is a proper }\mathbb{N}_{+}\text{-coloring of }G\text{)}}%
}\mathbf{x}_{f,w}\nonumber\\
&  \ \ \ \ \ \ \ \ \ \ +\sum_{\substack{f:V\rightarrow\mathbb{N}_{+}\text{ is
not a}\\\text{proper }\mathbb{N}_{+}\text{-coloring of }G}}\underbrace{\left[
f\text{ is a proper }\mathbb{N}_{+}\text{-coloring of }G\right]
}_{\substack{=0\\\text{(since }f\text{ is not a proper }\mathbb{N}%
_{+}\text{-coloring of }G\text{)}}}\mathbf{x}_{f,w}\nonumber\\
&  \ \ \ \ \ \ \ \ \ \ \ \ \ \ \ \ \ \ \ \ \left(  \text{since each
}f:V\rightarrow\mathbb{N}_{+}\text{ is either a proper }\mathbb{N}%
_{+}\text{-coloring of }G\text{ or not}\right) \nonumber\\
&  =\sum_{\substack{f:V\rightarrow\mathbb{N}_{+}\text{ is a}\\\text{proper
}\mathbb{N}_{+}\text{-coloring of }G}}\mathbf{x}_{f,w}+\underbrace{\sum
_{\substack{f:V\rightarrow\mathbb{N}_{+}\text{ is not a}\\\text{proper
}\mathbb{N}_{+}\text{-coloring of }G}}0\mathbf{x}_{f,w}}_{=0}=\sum
_{\substack{f:V\rightarrow\mathbb{N}_{+}\text{ is a}\\\text{proper }%
\mathbb{N}_{+}\text{-coloring of }G}}\mathbf{x}_{f,w}\nonumber\\
&  =X_{G,w} \label{pf.thm.wchromsym.varis.step1}%
\end{align}
(by (\ref{pf.thm.wchromsym.varis.XG-def})).

However, for every $f:V\rightarrow\mathbb{N}_{+}$, we have%
\begin{equation}
\sum_{\substack{B\subseteq E;\\B\subseteq\operatorname*{Eqs}f}}\left(
-1\right)  ^{\left\vert B\right\vert }\prod_{\substack{K\in\mathfrak{K}%
;\\K\subseteq B}}a_{K}=\left[  E\cap\operatorname*{Eqs}f=\varnothing\right]
\label{pf.thm.wchromsym.varis.moeb2}%
\end{equation}
(indeed, this can be shown just as in our above proof of Theorem
\ref{thm.chromsym.varis}).

Now, (\ref{pf.thm.wchromsym.varis.step1}) yields%
\begin{align*}
X_{G,w}  &  =\sum_{f:V\rightarrow\mathbb{N}_{+}}\underbrace{\left[
E\cap\operatorname*{Eqs}f=\varnothing\right]  }_{\substack{=\sum
_{\substack{B\subseteq E;\\B\subseteq\operatorname*{Eqs}f}}\left(  -1\right)
^{\left\vert B\right\vert }\prod_{\substack{K\in\mathfrak{K};\\K\subseteq
B}}a_{K}\\\text{(by (\ref{pf.thm.wchromsym.varis.moeb2}))}}}\mathbf{x}_{f,w}\\
&  =\sum_{f:V\rightarrow\mathbb{N}_{+}}\left(  \sum_{\substack{B\subseteq
E;\\B\subseteq\operatorname*{Eqs}f}}\left(  -1\right)  ^{\left\vert
B\right\vert }\prod_{\substack{K\in\mathfrak{K};\\K\subseteq B}}a_{K}\right)
\mathbf{x}_{f,w}=\underbrace{\sum_{f:V\rightarrow\mathbb{N}_{+}}%
\ \ \sum_{\substack{B\subseteq E;\\B\subseteq\operatorname*{Eqs}f}}}%
_{=\sum_{B\subseteq E}\ \ \sum_{\substack{f:V\rightarrow\mathbb{N}%
_{+};\\B\subseteq\operatorname*{Eqs}f}}}\left(  -1\right)  ^{\left\vert
B\right\vert }\left(  \prod_{\substack{K\in\mathfrak{K};\\K\subseteq B}%
}a_{K}\right)  \mathbf{x}_{f,w}\\
&  =\sum_{B\subseteq E}\ \ \sum_{\substack{f:V\rightarrow\mathbb{N}%
_{+};\\B\subseteq\operatorname*{Eqs}f}}\left(  -1\right)  ^{\left\vert
B\right\vert }\left(  \prod_{\substack{K\in\mathfrak{K};\\K\subseteq B}%
}a_{K}\right)  \mathbf{x}_{f,w}\\
&  =\sum_{B\subseteq E}\left(  -1\right)  ^{\left\vert B\right\vert }\left(
\prod_{\substack{K\in\mathfrak{K};\\K\subseteq B}}a_{K}\right)
\underbrace{\sum_{\substack{f:V\rightarrow\mathbb{N}_{+};\\B\subseteq
\operatorname*{Eqs}f}}\mathbf{x}_{f,w}}_{\substack{=p_{\lambda\left(  \left(
V,B\right)  ,w\right)  }\\\text{(by Lemma \ref{lem.wEqs.sum}}\\\text{(since
}\left(  V,B\right)  \text{ is a finite graph}\\\text{(since }V\text{ is a
finite set and }B\subseteq E\subseteq\dbinom{V}{2}\text{)))}}}\\
&  =\sum_{B\subseteq E}\left(  -1\right)  ^{\left\vert B\right\vert }\left(
\prod_{\substack{K\in\mathfrak{K};\\K\subseteq B}}a_{K}\right)  p_{\lambda
\left(  \left(  V,B\right)  ,w\right)  }=\sum_{F\subseteq E}\left(  -1\right)
^{\left\vert F\right\vert }\left(  \prod_{\substack{K\in\mathfrak{K}%
;\\K\subseteq F}}a_{K}\right)  p_{\lambda\left(  \left(  V,F\right)
,w\right)  }%
\end{align*}
(here, we have renamed the summation index $B$ as $F$). This proves Theorem
\ref{thm.wchromsym.varis}.
\end{proof}

\begin{proof}
[Proof of Corollary \ref{cor.wchromsym.K-free}.]We can apply Theorem
\ref{thm.wchromsym.varis} to $0$ instead of $a_{K}$. As a result, we obtain%
\begin{equation}
X_{G,w}=\sum_{F\subseteq E}\left(  -1\right)  ^{\left\vert F\right\vert
}\left(  \prod_{\substack{K\in\mathfrak{K};\\K\subseteq F}}0\right)
p_{\lambda\left(  \left(  V,F\right)  ,w\right)  }.
\label{pf.cor.wchromsym.K-free.0}%
\end{equation}
Now, if $F$ is any subset of $E$, then%
\begin{equation}
\prod_{\substack{K\in\mathfrak{K};\\K\subseteq F}}0=%
\begin{cases}
1, & \text{if }F\text{ is }\mathfrak{K}\text{-free;}\\
0, & \text{if }F\text{ is not }\mathfrak{K}\text{-free}%
\end{cases}
\label{pf.cor.wchromsym.K-free.1}%
\end{equation}
(indeed, this was already shown in our above proof of Corollary
\ref{cor.chromsym.K-free}).

Thus, (\ref{pf.cor.wchromsym.K-free.0}) becomes%
\begin{align*}
X_{G,w}  &  =\sum_{F\subseteq E}\left(  -1\right)  ^{\left\vert F\right\vert
}\underbrace{\left(  \prod_{\substack{K\in\mathfrak{K};\\K\subseteq
F}}0\right)  }_{\substack{=%
\begin{cases}
1, & \text{if }F\text{ is }\mathfrak{K}\text{-free;}\\
0, & \text{if }F\text{ is not }\mathfrak{K}\text{-free}%
\end{cases}
\\\text{(by (\ref{pf.cor.wchromsym.K-free.1}))}}}p_{\lambda\left(  \left(
V,F\right)  ,w\right)  }\\
&  =\sum_{F\subseteq E}\left(  -1\right)  ^{\left\vert F\right\vert }%
\begin{cases}
1, & \text{if }F\text{ is }\mathfrak{K}\text{-free;}\\
0, & \text{if }F\text{ is not }\mathfrak{K}\text{-free}%
\end{cases}
\ \ p_{\lambda\left(  \left(  V,F\right)  ,w\right)  }\\
&  =\sum_{\substack{F\subseteq E;\\F\text{ is }\mathfrak{K}\text{-free}%
}}\left(  -1\right)  ^{\left\vert F\right\vert }\underbrace{%
\begin{cases}
1, & \text{if }F\text{ is }\mathfrak{K}\text{-free;}\\
0, & \text{if }F\text{ is not }\mathfrak{K}\text{-free}%
\end{cases}
}_{\substack{=1\\\text{(since }F\text{ is }\mathfrak{K}\text{-free)}%
}}\ \ p_{\lambda\left(  \left(  V,F\right)  ,w\right)  }\\
&  \ \ \ \ \ \ \ \ \ \ +\sum_{\substack{F\subseteq E;\\F\text{ is not
}\mathfrak{K}\text{-free}}}\left(  -1\right)  ^{\left\vert F\right\vert
}\underbrace{%
\begin{cases}
1, & \text{if }F\text{ is }\mathfrak{K}\text{-free;}\\
0, & \text{if }F\text{ is not }\mathfrak{K}\text{-free}%
\end{cases}
}_{\substack{=0\\\text{(since }F\text{ is not }\mathfrak{K}\text{-free)}%
}}\ \ p_{\lambda\left(  \left(  V,F\right)  ,w\right)  }\\
&  \ \ \ \ \ \ \ \ \ \ \ \ \ \ \ \ \ \ \ \ \left(  \text{since each subset
}F\text{ of }E\text{ either is }\mathfrak{K}\text{-free or is not}\right) \\
&  =\sum_{\substack{F\subseteq E;\\F\text{ is }\mathfrak{K}\text{-free}%
}}\left(  -1\right)  ^{\left\vert F\right\vert }p_{\lambda\left(  \left(
V,F\right)  ,w\right)  }+\underbrace{\sum_{\substack{F\subseteq E;\\F\text{ is
not }\mathfrak{K}\text{-free}}}\left(  -1\right)  ^{\left\vert F\right\vert
}0p_{\lambda\left(  \left(  V,F\right)  ,w\right)  }}_{=0}\\
&  =\sum_{\substack{F\subseteq E;\\F\text{ is }\mathfrak{K}\text{-free}%
}}\left(  -1\right)  ^{\left\vert F\right\vert }p_{\lambda\left(  \left(
V,F\right)  ,w\right)  }.
\end{align*}
This proves Corollary \ref{cor.wchromsym.K-free}.
\end{proof}

\begin{proof}
[Proof of Corollary \ref{cor.wchromsym.NBC}.]Let $\mathfrak{K}$ be the set of
all broken circuits of $G$. Then, we have the following equality between
summation signs:
\[
\sum_{\substack{F\subseteq E;\\F\text{ is }\mathfrak{K}\text{-free}}%
}=\sum_{\substack{F\subseteq E;\\F\text{ contains no broken}\\\text{circuit of
}G\text{ as a subset}}}
\]
(indeed, this was already shown in our above proof of Corollary
\ref{cor.chromsym.NBC}). Now, Corollary \ref{cor.wchromsym.K-free} yields%
\[
X_{G,w}=\underbrace{\sum_{\substack{F\subseteq E;\\F\text{ is }\mathfrak{K}%
\text{-free}}}}_{=\sum_{\substack{F\subseteq E;\\F\text{ contains no
broken}\\\text{circuit of }G\text{ as a subset}}}}\left(  -1\right)
^{\left\vert F\right\vert }p_{\lambda\left(  \left(  V,F\right)  ,w\right)
}=\sum_{\substack{F\subseteq E;\\F\text{ contains no broken}\\\text{circuit of
}G\text{ as a subset}}}\left(  -1\right)  ^{\left\vert F\right\vert
}p_{\lambda\left(  \left(  V,F\right)  ,w\right)  }.
\]
This proves Corollary \ref{cor.wchromsym.NBC}.
\end{proof}

\begin{proof}
[Proof of Theorem \ref{thm.wchromsym.empty}.]Let $X$ be the totally ordered
set $\left\{  1\right\}  $ (equipped with the only possible order on this
set). Let $\ell:E\rightarrow X$ be the function sending each $e\in E$ to $1\in
X$. Let $\mathfrak{K}$ be the empty set. Clearly, $\mathfrak{K}$ is a set of
broken circuits of $G$. Theorem \ref{thm.wchromsym.varis} (applied to $0$
instead of $a_{K}$) yields%
\begin{align*}
X_{G,w}  &  =\sum_{F\subseteq E}\left(  -1\right)  ^{\left\vert F\right\vert
}\underbrace{\left(  \prod_{\substack{K\in\mathfrak{K};\\K\subseteq
F}}0\right)  }_{\substack{=\left(  \text{empty product}\right)  \\\text{(since
}\mathfrak{K}\text{ is the empty set)}}}p_{\lambda\left(  \left(  V,F\right)
,w\right)  }\\
&  =\sum_{F\subseteq E}\left(  -1\right)  ^{\left\vert F\right\vert
}\underbrace{\left(  \text{empty product}\right)  }_{=1}p_{\lambda\left(
\left(  V,F\right)  ,w\right)  }=\sum_{F\subseteq E}\left(  -1\right)
^{\left\vert F\right\vert }p_{\lambda\left(  \left(  V,F\right)  ,w\right)  }.
\end{align*}
This proves Theorem \ref{thm.wchromsym.empty}.
\end{proof}
\end{verlong}

\subsection{Ambigraphs redux}

Just as we have imposed weights on the vertices of a graph, we can do the same
to the vertices of an ambigraph. This leads to a generalization of the
chromatic symmetric function $X_{G}$ we introduced in Definition
\ref{def.ambichromsym}:

\begin{definition}
\label{def.wambichromsym}Let $G=\left(  V,E,\varphi\right)  $ be a finite
ambigraph. Let $w:V\rightarrow\mathbb{N}_{+}$ be a weight function on $V$.

\textbf{(a)} For every $\mathbb{N}_{+}$-coloring $f:V\rightarrow\mathbb{N}%
_{+}$ of $G$, we let $\mathbf{x}_{f,w}$ denote the monomial $\prod_{v\in
V}x_{f\left(  v\right)  }^{w\left(  v\right)  }$ in the indeterminates
$x_{1},x_{2},x_{3},\ldots$.

\textbf{(b)} We define a power series $X_{G,w}\in\mathbf{k}\left[  \left[
x_{1},x_{2},x_{3},\ldots\right]  \right]  $ by%
\[
X_{G,w}=\sum_{\substack{f:V\rightarrow\mathbb{N}_{+}\text{ is a}\\\text{proper
}\mathbb{N}_{+}\text{-coloring of }G}}\mathbf{x}_{f,w}.
\]

This power series $X_{G,w}$ is called the \emph{chromatic symmetric function}
of $\left(  G,w\right)  $.
\end{definition}

We can now state generalizations of Theorems \ref{thm.ambichromsym.empty} and
\ref{thm.ambichromsym.varis} and Corollaries \ref{cor.ambichromsym.K-free} and
\ref{cor.ambichromsym.NBC}:

\begin{theorem}
\label{thm.wambichromsym.empty}Let $G=\left(  V,E,\varphi\right)  $ be a
finite ambigraph. Let $w:V\rightarrow\mathbb{N}_{+}$ be a weight function on
$V$. Then,%
\[
X_{G,w}=\sum_{F\subseteq E}\left(  -1\right)  ^{\left\vert F\right\vert
}p_{\lambda\left(  \left(  V,\operatorname*{union}F\right)  ,w\right)  }.
\]
(Here, of course, the pair $\left(  V,\operatorname*{union}F\right)  $ is
regarded as a graph, and the expression $\lambda\left(  \left(
V,\operatorname*{union}F\right)  ,w\right)  $ is understood according to
Definition \ref{def.wlambda}.)
\end{theorem}

\begin{theorem}
\label{thm.wambichromsym.varis}Let $G=\left(  V,E,\varphi\right)  $ be a
finite ambigraph. Let $w:V\rightarrow\mathbb{N}_{+}$ be a weight function on
$V$. Let $X$ be a totally ordered set. Let $\ell:E\rightarrow X$ be a labeling
function. Let $\mathfrak{K}$ be some set of broken circuits of $G$ (not
necessarily containing all of them). Let $a_{K}$ be an element of $\mathbf{k}$
for every $K\in\mathfrak{K}$. Then,%
\[
X_{G,w}=\sum_{F\subseteq E}\left(  -1\right)  ^{\left\vert F\right\vert
}\left(  \prod_{\substack{K\in\mathfrak{K};\\K\subseteq F}}a_{K}\right)
p_{\lambda\left(  \left(  V,\operatorname*{union}F\right)  ,w\right)  }.
\]
(Here, of course, the pair $\left(  V,\operatorname*{union}F\right)  $ is
regarded as a graph, and the expression $\lambda\left(  \left(
V,\operatorname*{union}F\right)  ,w\right)  $ is understood according to
Definition \ref{def.wlambda}.)
\end{theorem}

\begin{corollary}
\label{cor.wambichromsym.K-free}Let $G=\left(  V,E,\varphi\right)  $ be a
finite ambigraph. Let $w:V\rightarrow\mathbb{N}_{+}$ be a weight function on
$V$. Let $X$ be a totally ordered set. Let $\ell:E\rightarrow X$ be a labeling
function. Let $\mathfrak{K}$ be some set of broken circuits of $G$ (not
necessarily containing all of them). Then,%
\[
X_{G,w}=\sum_{\substack{F\subseteq E;\\F\text{ is }\mathfrak{K}\text{-free}%
}}\left(  -1\right)  ^{\left\vert F\right\vert }p_{\lambda\left(  \left(
V,\operatorname*{union}F\right)  ,w\right)  }.
\]

\end{corollary}

\begin{corollary}
\label{cor.wambichromsym.NBC}Let $G=\left(  V,E,\varphi\right)  $ be a finite
ambigraph. Let $w:V\rightarrow\mathbb{N}_{+}$ be a weight function on $V$. Let
$X$ be a totally ordered set. Let $\ell:E\rightarrow X$ be a labeling
function. Then,%
\[
X_{G,w}=\sum_{\substack{F\subseteq E;\\F\text{ contains no broken}%
\\\text{circuit of }G\text{ as a subset}}}\left(  -1\right)  ^{\left\vert
F\right\vert }p_{\lambda\left(  \left(  V,\operatorname*{union}F\right)
,w\right)  }.
\]

\end{corollary}

\begin{vershort}
The proofs of these four results proceed precisely like their non-weighted
counterparts, again using Lemma \ref{lem.wEqs.sum} instead of Lemma
\ref{lem.Eqs.sum}.
\end{vershort}

\begin{verlong}
The proofs of these four results proceed precisely like their non-weighted
counterparts, again using Lemma \ref{lem.wEqs.sum} instead of Lemma
\ref{lem.Eqs.sum}. For the sake of completeness, here they are in detail:

\begin{proof}
[Proof of Theorem \ref{thm.wambichromsym.varis}.]We have%
\begin{equation}
X_{G,w}=\sum_{\substack{f:V\rightarrow\mathbb{N}_{+}\text{ is a}\\\text{proper
}\mathbb{N}_{+}\text{-coloring of }G}}\mathbf{x}_{f,w}
\label{pf.thm.wambichromsym.varis.XG-def}%
\end{equation}
(by the definition of $X_{G,w}$). Now, if $f:V\rightarrow\mathbb{N}_{+}$ is a
map, then we have the following logical equivalence:%
\begin{equation}
\left(  \text{the }\mathbb{N}_{+}\text{-coloring }f\text{ of }G\text{ is
proper}\right)  \ \Longleftrightarrow\ \left(  \operatorname*{EQS}\left(
G,f\right)  =\varnothing\right)  \label{pf.thm.wambichromsym.varis.equiv}%
\end{equation}
(because the $\mathbb{N}_{+}$-coloring $f$ of $G$ is proper if and only if
$\operatorname*{EQS}\left(  G,f\right)  =\varnothing$\ \ \ \ \footnote{by
Lemma \ref{lem.ambiEqs.proper} (applied to $\mathbb{N}_{+}$ instead of $X$)}).
Now,%
\begin{align}
&  \sum_{f:V\rightarrow\mathbb{N}_{+}}\left[  \underbrace{\operatorname*{EQS}%
\left(  G,f\right)  =\varnothing}_{\substack{\Longleftrightarrow\ \left(
\text{the }\mathbb{N}_{+}\text{-coloring }f\text{ of }G\text{ is
proper}\right)  \\\text{(by (\ref{pf.thm.wambichromsym.varis.equiv}))}%
}}\right]  \mathbf{x}_{f,w}\nonumber\\
&  =\sum_{f:V\rightarrow\mathbb{N}_{+}}\left[  \underbrace{\text{the
}\mathbb{N}_{+}\text{-coloring }f\text{ of }G\text{ is proper}}%
_{\Longleftrightarrow\ \left(  f\text{ is a proper }\mathbb{N}_{+}%
\text{-coloring of }G\right)  }\right]  \mathbf{x}_{f,w}\nonumber\\
&  =\sum_{f:V\rightarrow\mathbb{N}_{+}}\left[  f\text{ is a proper }%
\mathbb{N}_{+}\text{-coloring of }G\right]  \mathbf{x}_{f,w}\nonumber\\
&  =\sum_{\substack{f:V\rightarrow\mathbb{N}_{+}\text{ is a}\\\text{proper
}\mathbb{N}_{+}\text{-coloring of }G}}\underbrace{\left[  f\text{ is a proper
}\mathbb{N}_{+}\text{-coloring of }G\right]  }_{\substack{=1\\\text{(since
}f\text{ is a proper }\mathbb{N}_{+}\text{-coloring of }G\text{)}}%
}\mathbf{x}_{f,w}\nonumber\\
&  \ \ \ \ \ \ \ \ \ \ +\sum_{\substack{f:V\rightarrow\mathbb{N}_{+}\text{ is
not a}\\\text{proper }\mathbb{N}_{+}\text{-coloring of }G}}\underbrace{\left[
f\text{ is a proper }\mathbb{N}_{+}\text{-coloring of }G\right]
}_{\substack{=0\\\text{(since }f\text{ is not a proper }\mathbb{N}%
_{+}\text{-coloring of }G\text{)}}}\mathbf{x}_{f,w}\nonumber\\
&  \ \ \ \ \ \ \ \ \ \ \ \ \ \ \ \ \ \ \ \ \left(  \text{since each
}f:V\rightarrow\mathbb{N}_{+}\text{ either is a proper }\mathbb{N}%
_{+}\text{-coloring of }G\text{ or not}\right) \nonumber\\
&  =\sum_{\substack{f:V\rightarrow\mathbb{N}_{+}\text{ is a}\\\text{proper
}\mathbb{N}_{+}\text{-coloring of }G}}\mathbf{x}_{f,w}+\underbrace{\sum
_{\substack{f:V\rightarrow\mathbb{N}_{+}\text{ is not a}\\\text{proper
}\mathbb{N}_{+}\text{-coloring of }G}}0\mathbf{x}_{f,w}}_{=0}=\sum
_{\substack{f:V\rightarrow\mathbb{N}_{+}\text{ is a}\\\text{proper }%
\mathbb{N}_{+}\text{-coloring of }G}}\mathbf{x}_{f,w}\nonumber\\
&  =X_{G,w} \label{pf.thm.wambichromsym.varis.step1}%
\end{align}
(by (\ref{pf.thm.wambichromsym.varis.XG-def})).

However, for every $f:V\rightarrow\mathbb{N}_{+}$, we have%
\begin{equation}
\sum_{\substack{B\subseteq E;\\\operatorname*{union}B\subseteq
\operatorname*{Eqs}f}}\left(  -1\right)  ^{\left\vert B\right\vert }%
\prod_{\substack{K\in\mathfrak{K};\\K\subseteq B}}a_{K}=\left[
\operatorname*{EQS}\left(  G,f\right)  =\varnothing\right]
\label{pf.thm.wambichromsym.varis.moeb2}%
\end{equation}
(indeed, we have already shown this in the above proof of Theorem
\ref{thm.ambichromsym.varis}).

Now, (\ref{pf.thm.wambichromsym.varis.step1}) yields%
\begin{align}
X_{G,w}  &  =\sum_{f:V\rightarrow\mathbb{N}_{+}}\underbrace{\left[
\operatorname*{EQS}\left(  G,f\right)  =\varnothing\right]  }_{\substack{=\sum
_{\substack{B\subseteq E;\\\operatorname*{union}B\subseteq\operatorname*{Eqs}%
f}}\left(  -1\right)  ^{\left\vert B\right\vert }\prod_{\substack{K\in
\mathfrak{K};\\K\subseteq B}}a_{K}\\\text{(by
(\ref{pf.thm.wambichromsym.varis.moeb2}))}}}\mathbf{x}_{f,w}\nonumber\\
&  =\sum_{f:V\rightarrow\mathbb{N}_{+}}\left(  \sum_{\substack{B\subseteq
E;\\\operatorname*{union}B\subseteq\operatorname*{Eqs}f}}\left(  -1\right)
^{\left\vert B\right\vert }\prod_{\substack{K\in\mathfrak{K};\\K\subseteq
B}}a_{K}\right)  \mathbf{x}_{f,w}\nonumber\\
&  =\underbrace{\sum_{f:V\rightarrow\mathbb{N}_{+}}\ \ \sum
_{\substack{B\subseteq E;\\\operatorname*{union}B\subseteq\operatorname*{Eqs}%
f}}}_{=\sum_{B\subseteq E}\ \ \sum_{\substack{f:V\rightarrow\mathbb{N}%
_{+};\\\operatorname*{union}B\subseteq\operatorname*{Eqs}f}}}\left(
-1\right)  ^{\left\vert B\right\vert }\left(  \prod_{\substack{K\in
\mathfrak{K};\\K\subseteq B}}a_{K}\right)  \mathbf{x}_{f,w}\nonumber\\
&  =\sum_{B\subseteq E}\ \ \sum_{\substack{f:V\rightarrow\mathbb{N}%
_{+};\\\operatorname*{union}B\subseteq\operatorname*{Eqs}f}}\left(  -1\right)
^{\left\vert B\right\vert }\left(  \prod_{\substack{K\in\mathfrak{K}%
;\\K\subseteq B}}a_{K}\right)  \mathbf{x}_{f,w}\nonumber\\
&  =\sum_{B\subseteq E}\left(  -1\right)  ^{\left\vert B\right\vert }\left(
\prod_{\substack{K\in\mathfrak{K};\\K\subseteq B}}a_{K}\right)  \sum
_{\substack{f:V\rightarrow\mathbb{N}_{+};\\\operatorname*{union}%
B\subseteq\operatorname*{Eqs}f}}\mathbf{x}_{f,w}.
\label{pf.thm.wambichromsym.varis.step4}%
\end{align}

However, if $B$ is a subset of $E$, then the pair $\left(
V,\operatorname*{union}B\right)  $ is a finite graph (since $V$ is a finite
set and since $\operatorname*{union}B\subseteq\dbinom{V}{2}$), and thus we
have%
\begin{equation}
\sum_{\substack{f:V\rightarrow\mathbb{N}_{+};\\\operatorname*{union}%
B\subseteq\operatorname*{Eqs}f}}\mathbf{x}_{f,w}=p_{\lambda\left(  \left(
V,\operatorname*{union}B\right)  ,w\right)  }
\label{pf.thm.wambichromsym.varis.p}%
\end{equation}
(by Lemma \ref{lem.wEqs.sum}, applied to $\operatorname*{union}B$ instead of
$B$).

Hence, (\ref{pf.thm.wambichromsym.varis.step4}) becomes%
\begin{align*}
X_{G,w}  &  =\sum_{B\subseteq E}\left(  -1\right)  ^{\left\vert B\right\vert
}\left(  \prod_{\substack{K\in\mathfrak{K};\\K\subseteq B}}a_{K}\right)
\underbrace{\sum_{\substack{f:V\rightarrow\mathbb{N}_{+}%
;\\\operatorname*{union}B\subseteq\operatorname*{Eqs}f}}\mathbf{x}_{f,w}%
}_{\substack{=p_{\lambda\left(  \left(  V,\operatorname*{union}B\right)
,w\right)  }\\\text{(by (\ref{pf.thm.wambichromsym.varis.p}))}}}\\
&  =\sum_{B\subseteq E}\left(  -1\right)  ^{\left\vert B\right\vert }\left(
\prod_{\substack{K\in\mathfrak{K};\\K\subseteq B}}a_{K}\right)  p_{\lambda
\left(  \left(  V,\operatorname*{union}B\right)  ,w\right)  }\\
&  =\sum_{F\subseteq E}\left(  -1\right)  ^{\left\vert F\right\vert }\left(
\prod_{\substack{K\in\mathfrak{K};\\K\subseteq F}}a_{K}\right)  p_{\lambda
\left(  \left(  V,\operatorname*{union}F\right)  ,w\right)  }%
\end{align*}
(here, we have renamed the summation index $B$ as $F$). This proves Theorem
\ref{thm.wambichromsym.varis}.
\end{proof}

\begin{proof}
[Proof of Corollary \ref{cor.wambichromsym.K-free}.]We can apply Theorem
\ref{thm.wambichromsym.varis} to $0$ instead of $a_{K}$. As a result, we
obtain%
\begin{equation}
X_{G,w}=\sum_{F\subseteq E}\left(  -1\right)  ^{\left\vert F\right\vert
}\left(  \prod_{\substack{K\in\mathfrak{K};\\K\subseteq F}}0\right)
p_{\lambda\left(  \left(  V,\operatorname*{union}F\right)  ,w\right)  }.
\label{pf.cor.wambichromsym.K-free.0}%
\end{equation}
Now, if $F$ is any subset of $E$, then%
\begin{equation}
\prod_{\substack{K\in\mathfrak{K};\\K\subseteq F}}0=%
\begin{cases}
1, & \text{if }F\text{ is }\mathfrak{K}\text{-free;}\\
0, & \text{if }F\text{ is not }\mathfrak{K}\text{-free}%
\end{cases}
\label{pf.cor.wambichromsym.K-free.1}%
\end{equation}
(indeed, this was already shown in our above proof of Corollary
\ref{cor.ambichromsym.K-free}).

Thus, (\ref{pf.cor.wambichromsym.K-free.0}) becomes%
\begin{align*}
X_{G,w}  &  =\sum_{F\subseteq E}\left(  -1\right)  ^{\left\vert F\right\vert
}\underbrace{\left(  \prod_{\substack{K\in\mathfrak{K};\\K\subseteq
F}}0\right)  }_{\substack{=%
\begin{cases}
1, & \text{if }F\text{ is }\mathfrak{K}\text{-free;}\\
0, & \text{if }F\text{ is not }\mathfrak{K}\text{-free}%
\end{cases}
\\\text{(by (\ref{pf.cor.wambichromsym.K-free.1}))}}}p_{\lambda\left(  \left(
V,\operatorname*{union}F\right)  ,w\right)  }\\
&  =\sum_{F\subseteq E}\left(  -1\right)  ^{\left\vert F\right\vert }%
\begin{cases}
1, & \text{if }F\text{ is }\mathfrak{K}\text{-free;}\\
0, & \text{if }F\text{ is not }\mathfrak{K}\text{-free}%
\end{cases}
\ \ p_{\lambda\left(  \left(  V,\operatorname*{union}F\right)  ,w\right)  }\\
&  =\sum_{\substack{F\subseteq E;\\F\text{ is }\mathfrak{K}\text{-free}%
}}\left(  -1\right)  ^{\left\vert F\right\vert }\underbrace{%
\begin{cases}
1, & \text{if }F\text{ is }\mathfrak{K}\text{-free;}\\
0, & \text{if }F\text{ is not }\mathfrak{K}\text{-free}%
\end{cases}
}_{\substack{=1\\\text{(since }F\text{ is }\mathfrak{K}\text{-free)}%
}}\ \ p_{\lambda\left(  \left(  V,\operatorname*{union}F\right)  ,w\right)
}\\
&  \ \ \ \ \ \ \ \ \ \ +\sum_{\substack{F\subseteq E;\\F\text{ is not
}\mathfrak{K}\text{-free}}}\left(  -1\right)  ^{\left\vert F\right\vert
}\underbrace{%
\begin{cases}
1, & \text{if }F\text{ is }\mathfrak{K}\text{-free;}\\
0, & \text{if }F\text{ is not }\mathfrak{K}\text{-free}%
\end{cases}
}_{\substack{=0\\\text{(since }F\text{ is not }\mathfrak{K}\text{-free)}%
}}\ \ p_{\lambda\left(  \left(  V,\operatorname*{union}F\right)  ,w\right)
}\\
&  \ \ \ \ \ \ \ \ \ \ \ \ \ \ \ \ \ \ \ \ \left(  \text{since each subset
}F\text{ of }E\text{ either is }\mathfrak{K}\text{-free or is not}\right) \\
&  =\sum_{\substack{F\subseteq E;\\F\text{ is }\mathfrak{K}\text{-free}%
}}\left(  -1\right)  ^{\left\vert F\right\vert }p_{\lambda\left(  \left(
V,\operatorname*{union}F\right)  ,w\right)  }+\underbrace{\sum
_{\substack{F\subseteq E;\\F\text{ is not }\mathfrak{K}\text{-free}}}\left(
-1\right)  ^{\left\vert F\right\vert }0p_{\lambda\left(  \left(
V,\operatorname*{union}F\right)  ,w\right)  }}_{=0}\\
&  =\sum_{\substack{F\subseteq E;\\F\text{ is }\mathfrak{K}\text{-free}%
}}\left(  -1\right)  ^{\left\vert F\right\vert }p_{\lambda\left(  \left(
V,\operatorname*{union}F\right)  ,w\right)  }.
\end{align*}
This proves Corollary \ref{cor.wambichromsym.K-free}.
\end{proof}

\begin{proof}
[Proof of Corollary \ref{cor.wambichromsym.NBC}.]Let $\mathfrak{K}$ be the set
of all broken circuits of $G$.

Now, just as in the proof of Corollary \ref{cor.chromsym.NBC}, we can prove
the following equality:
\[
\sum_{\substack{F\subseteq E;\\F\text{ is }\mathfrak{K}\text{-free}}%
}=\sum_{\substack{F\subseteq E;\\F\text{ contains no broken}\\\text{circuit of
}G\text{ as a subset}}}
\]
(an equality between summation signs). Now, Corollary
\ref{cor.wambichromsym.K-free} yields%
\begin{align*}
X_{G,w}  &  =\underbrace{\sum_{\substack{F\subseteq E;\\F\text{ is
}\mathfrak{K}\text{-free}}}}_{=\sum_{\substack{F\subseteq E;\\F\text{ contains
no broken}\\\text{circuit of }G\text{ as a subset}}}}\left(  -1\right)
^{\left\vert F\right\vert }p_{\lambda\left(  \left(  V,\operatorname*{union}%
F\right)  ,w\right)  }\\
&  =\sum_{\substack{F\subseteq E;\\F\text{ contains no broken}\\\text{circuit
of }G\text{ as a subset}}}\left(  -1\right)  ^{\left\vert F\right\vert
}p_{\lambda\left(  \left(  V,\operatorname*{union}F\right)  ,w\right)  }.
\end{align*}
This proves Corollary \ref{cor.wambichromsym.NBC}.
\end{proof}

\begin{proof}
[Proof of Theorem \ref{thm.wambichromsym.empty}.]Let $X$ be the totally
ordered set $\left\{  1\right\}  $ (equipped with the only possible order on
this set). Let $\ell:E\rightarrow X$ be the function sending each $e\in E$ to
$1\in X$. Let $\mathfrak{K}$ be the empty set. Clearly, $\mathfrak{K}$ is a
set of broken circuits of $G$. Theorem \ref{thm.wambichromsym.varis} (applied
to $0$ instead of $a_{K}$) yields%
\begin{align*}
X_{G,w}  &  =\sum_{F\subseteq E}\left(  -1\right)  ^{\left\vert F\right\vert
}\underbrace{\left(  \prod_{\substack{K\in\mathfrak{K};\\K\subseteq
F}}0\right)  }_{\substack{=\left(  \text{empty product}\right)  \\\text{(since
}\mathfrak{K}\text{ is the empty set)}}}p_{\lambda\left(  \left(
V,\operatorname*{union}F\right)  ,w\right)  }\\
&  =\sum_{F\subseteq E}\left(  -1\right)  ^{\left\vert F\right\vert
}\underbrace{\left(  \text{empty product}\right)  }_{=1}p_{\lambda\left(
\left(  V,\operatorname*{union}F\right)  ,w\right)  }=\sum_{F\subseteq
E}\left(  -1\right)  ^{\left\vert F\right\vert }p_{\lambda\left(  \left(
V,\operatorname*{union}F\right)  ,w\right)  }.
\end{align*}
This proves Theorem \ref{thm.wambichromsym.empty}.
\end{proof}
\end{verlong}

\subsection{Noncommutative chromatic symmetric functions}

The \emph{noncommutative chromatic symmetric function} $Y_{G}$ of a graph $G$
has been introduced by Gebhard and Sagan in \cite[\S 3]{GebSag01} as a lift of
the chromatic symmetric function $X_{G}$ to a noncommutative polynomial ring.
In order to define it, we need to lift the monomials $\mathbf{x}_{f}$ to
noncommutative monomials, which requires fixing a list of the vertices of $G$
(since a noncommutative product is only defined if its factors appear in a
chosen order). Unlike \cite{GebSag01}, we shall not require this list to
contain each vertex of $G$ exactly once; thus, we obtain a more general notion
that refines not only Stanley's original $X_{G}$ but also its weighted version
$X_{G,w}$ discussed above.

We consider the $\mathbf{k}$-algebra $\mathbf{k}\left\langle \left\langle
X_{1},X_{2},X_{3},\ldots\right\rangle \right\rangle $ of noncommutative power
series in countably many distinct indeterminates $X_{1},X_{2},X_{3},\ldots$
over $\mathbf{k}$. It is a topological $\mathbf{k}$-algebra\footnote{Its
topology is defined in the same way as the topology on $\mathbf{k}\left[
\left[  x_{1},x_{2},x_{3},\ldots\right]  \right]  $ (but of course the
monomials are now noncommutative monomials).}. A noncommutative power series
$P\in\mathbf{k}\left\langle \left\langle X_{1},X_{2},X_{3},\ldots\right\rangle
\right\rangle $ is said to be \emph{bounded-degree} if there exists an
$N\in\mathbb{N}$ such that every noncommutative monomial of degree $>N$
appears with coefficient $0$ in $P$. A noncommutative power series
$P\in\mathbf{k}\left\langle \left\langle X_{1},X_{2},X_{3},\ldots\right\rangle
\right\rangle $ is said to be \emph{symmetric} if and only if $P$ is invariant
under any permutation of the indeterminates. We let $\Lambda
_{\operatorname*{NC}}$ be the subset of $\mathbf{k}\left\langle \left\langle
X_{1},X_{2},X_{3},\ldots\right\rangle \right\rangle $ consisting of all
symmetric bounded-degree power series $P\in\mathbf{k}\left\langle \left\langle
X_{1},X_{2},X_{3},\ldots\right\rangle \right\rangle $. This subset
$\Lambda_{\operatorname*{NC}}$ is a $\mathbf{k}$-subalgebra of $\mathbf{k}%
\left\langle \left\langle X_{1},X_{2},X_{3},\ldots\right\rangle \right\rangle
$, and is called the $\mathbf{k}$\emph{-algebra of symmetric functions in
noncommutative indeterminates} over $\mathbf{k}$.

This $\mathbf{k}$-algebra $\Lambda_{\operatorname*{NC}}$ is called $\Pi\left(
\mathbf{x}\right)  $ in \cite{RosSag04}, and should not be mistaken for the
algebra $\operatorname*{NSym}$ of noncommutative symmetric functions (which is
studied, e.g., in \cite[\S 5.4]{Reiner}).\footnote{The latter algebra
$\operatorname*{NSym}$ can too be viewed as a subalgebra of $\mathbf{k}%
\left\langle \left\langle X_{1},X_{2},X_{3},\ldots\right\rangle \right\rangle
$, but it does not consist of symmetric power series (despite its name).}

We can now define noncommutative chromatic symmetric functions of ambigraphs:

\begin{definition}
\label{def.ncambichromsym}Let $G=\left(  V,E,\varphi\right)  $ be a finite
ambigraph. Let $\mathbf{t}=\left(  t_{1},t_{2},\ldots,t_{N}\right)  $ be a
finite list of elements of $V$ that contains each element of $V$ at least once.

\textbf{(a)} For every $\mathbb{N}_{+}$-coloring $f:V\rightarrow\mathbb{N}%
_{+}$ of $G$, we let $\mathbf{X}_{f,\mathbf{t}}$ denote the noncommutative
monomial $X_{f\left(  t_{1}\right)  }X_{f\left(  t_{2}\right)  }\cdots
X_{f\left(  t_{N}\right)  }$ in the indeterminates $X_{1},X_{2},X_{3},\ldots$.
(This does not actually depend on the ambigraph $G$, but only depends on the
set $V$.)

\textbf{(b)} We define a noncommutative power series $Y_{G,\mathbf{t}}%
\in\mathbf{k}\left\langle \left\langle X_{1},X_{2},X_{3},\ldots\right\rangle
\right\rangle $ by%
\begin{equation}
Y_{G,\mathbf{t}}=\sum_{\substack{f:V\rightarrow\mathbb{N}_{+}\text{ is
a}\\\text{proper }\mathbb{N}_{+}\text{-coloring of }G}}\mathbf{X}%
_{f,\mathbf{t}}. \label{eq.def.ncambichromsym.YGf=}%
\end{equation}

This power series $Y_{G,\mathbf{t}}$ is called the \emph{noncommutative
chromatic symmetric function} of $\left(  G,\mathbf{t}\right)  $.
\end{definition}

\begin{remark}
Why did we require the list $\mathbf{t}$ to contain each element of $V$ at
least once in Definition \ref{def.ncambichromsym}?

Otherwise, there could be a vertex $v\in V$ that does not appear in
$\mathbf{t}$. In that case, the monomials $\mathbf{X}_{f,\mathbf{t}}$ would be
independent of the color $f\left(  v\right)  $ of this vertex, and thus we
would obtain the same monomial $\mathbf{X}_{f,\mathbf{t}}$ for infinitely many
different proper $\mathbb{N}_{+}$-colorings $f$ (since we could arbitrarily
change the color $f\left(  v\right)  $ without affecting the monomial
$\mathbf{X}_{f,\mathbf{t}}$, as long as the $\mathbb{N}_{+}$-coloring $f$
remains proper). Hence, the sum on the right-hand side of
(\ref{eq.def.ncambichromsym.YGf=}) would contain infinitely many identical
monomials, and this would render $Y_{G,\mathbf{t}}$ undefined.
\end{remark}

The noncommutative chromatic symmetric function $Y_{G,\mathbf{t}}$ is a lift
of the weighted chromatic symmetric function $X_{G,w}$ introduced in
Definition \ref{def.wambichromsym} \textbf{(b)}. This can be made precise as follows:

\begin{remark}
\label{rmk.ncambichromsym.lift}Let $G=\left(  V,E,\varphi\right)  $ be a
finite ambigraph. Let $\mathbf{t}=\left(  t_{1},t_{2},\ldots,t_{N}\right)  $
be a finite list of elements of $V$ that contains each element of $V$ at least
once. Let $w:V\rightarrow\mathbb{N}_{+}$ be the weight function on $V$ that is
defined by%
\begin{align*}
w\left(  v\right)   &  =\left(  \text{number of times that }v\text{ appears in
the list }\mathbf{t}\right) \\
&  =\left(  \text{number of all }i\in\left\{  1,2,\ldots,N\right\}  \text{
such that }t_{i}=v\right)  \ \ \ \ \ \ \ \ \ \ \text{for each }v\in V.
\end{align*}

Let $\pi:\mathbf{k}\left\langle \left\langle X_{1},X_{2},X_{3},\ldots
\right\rangle \right\rangle \rightarrow\mathbf{k}\left[  \left[  x_{1}%
,x_{2},x_{3},\ldots\right]  \right]  $ be the topological $\mathbf{k}$-algebra
homomorphism that sends the noncommuting indeterminates $X_{1},X_{2}%
,X_{3},\ldots$ to the respective commuting indeterminates $x_{1},x_{2}%
,x_{3},\ldots$. Then:

\textbf{(a)} For every $\mathbb{N}_{+}$-coloring $f:V\rightarrow\mathbb{N}%
_{+}$ of $G$, we have $\pi\left(  \mathbf{X}_{f,\mathbf{t}}\right)
=\mathbf{x}_{f,w}$. (See Definition \ref{def.wambichromsym} \textbf{(a)} and
Definition \ref{def.ncambichromsym} \textbf{(a)} for the meanings of
$\mathbf{x}_{f,w}$ and $\mathbf{X}_{f,\mathbf{t}}$.)

\textbf{(b)} We have $\pi\left(  Y_{G,\mathbf{t}}\right)  =X_{G,w}$.
\end{remark}

We omit the simple proofs of these claims.

Our goal is now to state noncommutative analogues of our main theorems for
ambigraphs (specifically, Theorems \ref{thm.wambichromsym.empty} and
\ref{thm.wambichromsym.varis} and Corollaries \ref{cor.wambichromsym.K-free}
and \ref{cor.wambichromsym.NBC}). To do so, we need a noncommutative analogue
of the power-sum symmetric functions $p_{\lambda}$. In the commutative case,
we defined $p_{\lambda}$ as the product $\prod_{i\geq1}p_{\lambda_{i}}$ (see
Definition \ref{def.powersum2}). The noncommutative case, however, will not be
such a product, so we need to define it differently. This will require some preparations.

We begin by recalling the notion of a set partition:

\begin{definition}
Let $X$ be a set.

\textbf{(a)} A \emph{set partition} of $X$ means a set $\mathbf{P}$ of
disjoint nonempty subsets of $X$ such that $\bigcup_{S\in\mathbf{P}}S=X$.

\textbf{(b)} If $\mathbf{P}$ is a set partition of $X$, then the sets
$S\in\mathbf{P}$ are called the \emph{blocks} of $\mathbf{P}$.
\end{definition}

For example:

\begin{itemize}
\item The set $\left\{  \left\{  1,3\right\}  ,\ \left\{  2,4,5\right\}
\right\}  $ is a set partition of the set $\left\{  1,2,3,4,5\right\}  $,
since $\left\{  1,3\right\}  $ and $\left\{  2,4,5\right\}  $ are two disjoint
nonempty subsets of $\left\{  1,2,3,4,5\right\}  $ whose union is $\left\{
1,3\right\}  \cup\left\{  2,4,5\right\}  =\left\{  1,2,3,4,5\right\}  $.

\item The set $\left\{  \left\{  1,4\right\}  ,\ \left\{  2,5\right\}
,\ \left\{  3,6\right\}  \right\}  $ is a set partition of the set $\left\{
1,2,3,4,5,6\right\}  $. The blocks of this set partition are $\left\{
1,4\right\}  $ and $\left\{  2,5\right\}  $ and $\left\{  3,6\right\}  $.

\item The set $\left\{  \left\{  1,2,3,4,5\right\}  \right\}  $ is a set
partition of the set $\left\{  1,2,3,4,5\right\}  $. It has only one block,
namely $\left\{  1,2,3,4,5\right\}  $.

\item The set $\left\{  \left\{  1\right\}  ,\ \left\{  2\right\}  ,\ \left\{
3\right\}  ,\ \left\{  4\right\}  ,\ \left\{  5\right\}  \right\}  $ is a set
partition of the set $\left\{  1,2,3,4,5\right\}  $. It has five blocks,
namely $\left\{  1\right\}  ,\ \left\{  2\right\}  ,\ \left\{  3\right\}
,\ \left\{  4\right\}  ,\ \left\{  5\right\}  $.
\end{itemize}

There is a well-known relation (actually a one-to-one correspondence) between
the set partitions of a given set $X$ and the equivalence relations on $X$. It
can be summarized in the following theorem:

\begin{theorem}
\label{thm.setpar.eqrel}Let $X$ be a set.

\begin{vershort}
\textbf{(a)} If $\sim$ is an equivalence relation on $X$, then the set%
\[
X/\left(  \sim\right)  :=\left\{  \text{all }\sim\text{-equivalence
classes}\right\}
\]
is a set partition of $X$.
\end{vershort}

\begin{verlong}
\textbf{(a)} If $\sim$ is an equivalence relation on $X$, then the set%
\[
X/\left(  \sim\right)  =\left\{  \text{all }\sim\text{-equivalence
classes}\right\}
\]
(see Definition \ref{def.relquot} for its definition) is a set partition of
$X$.
\end{verlong}

\textbf{(b)} If $\mathbf{P}$ is a set partition of $X$, then we can define an
equivalence relation $\sim$ on the set $X$ as follows: For any two elements
$a$ and $b$ of $X$, we shall have $a\sim b$ if and only if the elements $a$
and $b$ belong to the same block of $\mathbf{P}$. This relation $\sim$ will be
called $\sim_{\mathbf{P}}$.

\textbf{(c)} The maps%
\begin{align*}
\left\{  \text{equivalence relations on }X\right\}   &  \rightarrow\left\{
\text{set partitions of }X\right\}  ,\\
\left(  \sim\right)   &  \mapsto X/\left(  \sim\right)
\end{align*}
and%
\begin{align*}
\left\{  \text{set partitions of }X\right\}   &  \rightarrow\left\{
\text{equivalence relations on }X\right\}  ,\\
\mathbf{P}  &  \mapsto\left(  \sim_{\mathbf{P}}\right)
\end{align*}
are mutually inverse bijections.
\end{theorem}

\begin{example}
Let $X$ be the set $\left\{  1,2,3,4,5,6\right\}  $.

\textbf{(a)} If $\sim$ is the equivalence relation on $X$ given by
\[
\left(  a\sim b\right)  \ \Longleftrightarrow\ \left(  a\equiv
b\operatorname{mod}2\right)  ,
\]
then the corresponding set partition $X/\left(  \sim\right)  $ of $X$ is
$\left\{  \left\{  1,3,5\right\}  ,\ \left\{  2,4,6\right\}  \right\}  $.

\textbf{(b)} If $\mathbf{P}$ is the set partition $\left\{  \left\{
1,2\right\}  ,\ \left\{  3,4,6\right\}  ,\ \left\{  5\right\}  \right\}  $ of
$X$, then the corresponding equivalence relation $\sim_{\mathbf{P}}$ on $X$ is
given by%
\[
1\sim_{\mathbf{P}}2,\ \ \ \ \ \ \ \ \ \ 3\sim_{\mathbf{P}}4\sim_{\mathbf{P}}6
\]
and no further relations (except, of course, for the ones that follow from the
relations just given by reflexivity, symmetry and transitivity).
\end{example}

We can now define the noncommutative analogues of the power-sum symmetric
functions $p_{\lambda}$. These are indexed not by integer partitions $\lambda$
but by set partitions $\mathbf{P}$:

\begin{definition}
Let $N\in\mathbb{N}$. Let $\mathbf{P}$ be a set partition of the set $\left\{
1,2,\ldots,N\right\}  $. Recall the relation $\sim_{\mathbf{P}}$ defined in
Theorem \ref{thm.setpar.eqrel} \textbf{(b)}.

Then, we define a noncommutative power series $P_{\mathbf{P}}\in
\mathbf{k}\left\langle \left\langle X_{1},X_{2},X_{3},\ldots\right\rangle
\right\rangle $ by%
\[
P_{\mathbf{P}}=\sum_{\substack{\left(  i_{1},i_{2},\ldots,i_{N}\right)
\in\left(  \mathbb{N}_{+}\right)  ^{N};\\i_{a}=i_{b}\text{ whenever }%
a\sim_{\mathbf{P}}b}}X_{i_{1}}X_{i_{2}}\cdots X_{i_{N}}.
\]
Here, the condition \textquotedblleft$i_{a}=i_{b}$ whenever $a\sim
_{\mathbf{P}}b$\textquotedblright\ under the summation sign is shorthand for
\textquotedblleft$i_{a}=i_{b}$ for any two elements $a,b\in\left\{
1,2,\ldots,N\right\}  $ that satisfy $a\sim_{\mathbf{P}}b$\textquotedblright.

This power series $P_{\mathbf{P}}$ is called the \emph{power-sum symmetric
function in noncommutative variables} corresponding to the set partition
$\mathbf{P}$. It is not hard to see that it belongs to $\Lambda
_{\operatorname*{NC}}$.
\end{definition}

Note that $P_{\mathbf{P}}$ is called $p_{\mathbf{P}}$ in \cite[\S 2]{GebSag01}
and in \cite{RosSag04}.

\begin{example}
\textbf{(a)} If $\mathbf{P}$ is the set partition $\left\{  \left\{
1,3\right\}  ,\ \left\{  2,4,5\right\}  \right\}  $ of $\left\{
1,2,3,4,5\right\}  $, then
\begin{align*}
P_{\mathbf{P}}  &  =\sum_{\substack{\left(  i_{1},i_{2},i_{3},i_{4}%
,i_{5}\right)  \in\left(  \mathbb{N}_{+}\right)  ^{5};\\i_{a}=i_{b}\text{
whenever }a\sim_{\mathbf{P}}b}}X_{i_{1}}X_{i_{2}}X_{i_{3}}X_{i_{4}}X_{i_{5}}\\
&  =\sum_{\substack{\left(  i_{1},i_{2},i_{3},i_{4},i_{5}\right)  \in\left(
\mathbb{N}_{+}\right)  ^{5};\\i_{1}=i_{3}\text{ and }i_{2}=i_{4}=i_{5}%
}}X_{i_{1}}X_{i_{2}}X_{i_{3}}X_{i_{4}}X_{i_{5}}\\
&  =\sum_{\left(  u,v\right)  \in\left(  \mathbb{N}_{+}\right)  ^{2}}%
X_{u}X_{v}X_{u}X_{v}X_{v}%
\end{align*}
(here, we have substituted $\left(  u,v,u,v,v\right)  $ for the summation
index $\left(  i_{1},i_{2},i_{3},i_{4},i_{5}\right)  $, since the condition
\textquotedblleft$i_{1}=i_{3}$ and $i_{2}=i_{4}=i_{5}$\textquotedblright\ is
saying precisely that the $5$-tuple $\left(  i_{1},i_{2},i_{3},i_{4}%
,i_{5}\right)  $ can be written in the form $\left(  u,v,u,v,v\right)  $).
This is a noncommutative power series that contains terms such as $X_{4}%
X_{7}X_{4}X_{7}X_{7}$ or $X_{5}X_{5}X_{5}X_{5}X_{5}$ (we are allowed to have
$u=v$ in the above sum), but not terms such as $X_{1}X_{2}X_{2}X_{1}X_{1}$
(since the indeterminates don't commute).

\textbf{(b)} If $\mathbf{P}$ is the set partition $\left\{  \left\{
1,4\right\}  ,\ \left\{  2,5\right\}  ,\ \left\{  3,6\right\}  \right\}  $ of
$\left\{  1,2,3,4,5,6\right\}  $, then
\begin{align*}
P_{\mathbf{P}}  &  =\sum_{\substack{\left(  i_{1},i_{2},i_{3},i_{4}%
,i_{5},i_{6}\right)  \in\left(  \mathbb{N}_{+}\right)  ^{6};\\i_{a}%
=i_{b}\text{ whenever }a\sim_{\mathbf{P}}b}}X_{i_{1}}X_{i_{2}}X_{i_{3}%
}X_{i_{4}}X_{i_{5}}X_{i_{6}}\\
&  =\sum_{\substack{\left(  i_{1},i_{2},i_{3},i_{4},i_{5},i_{6}\right)
\in\left(  \mathbb{N}_{+}\right)  ^{6};\\i_{1}=i_{4}\text{ and }i_{2}%
=i_{5}\text{ and }i_{3}=i_{6}}}X_{i_{1}}X_{i_{2}}X_{i_{3}}X_{i_{4}}X_{i_{5}%
}X_{i_{6}}\\
&  =\sum_{\left(  u,v,w\right)  \in\left(  \mathbb{N}_{+}\right)  ^{3}%
}\underbrace{X_{u}X_{v}X_{w}X_{u}X_{v}X_{w}}_{=\left(  X_{u}X_{v}X_{w}\right)
^{2}}=\sum_{\left(  u,v,w\right)  \in\left(  \mathbb{N}_{+}\right)  ^{3}%
}\left(  X_{u}X_{v}X_{w}\right)  ^{2}.
\end{align*}
Note that this cannot be simplified to $\left(  \sum_{u\in\mathbb{N}_{+}}%
X_{u}^{2}\right)  ^{3}$, since the indeterminates don't commute.

\textbf{(c)} If $\mathbf{P}$ is the set partition $\left\{  \left\{
1,2,3,4\right\}  \right\}  $ of $\left\{  1,2,3,4\right\}  $, then
\begin{align*}
P_{\mathbf{P}}  &  =\sum_{\substack{\left(  i_{1},i_{2},i_{3},i_{4}\right)
\in\left(  \mathbb{N}_{+}\right)  ^{4};\\i_{a}=i_{b}\text{ whenever }%
a\sim_{\mathbf{P}}b}}X_{i_{1}}X_{i_{2}}X_{i_{3}}X_{i_{4}}\\
&  =\sum_{\substack{\left(  i_{1},i_{2},i_{3},i_{4}\right)  \in\left(
\mathbb{N}_{+}\right)  ^{4};\\i_{1}=i_{2}=i_{3}=i_{4}}}X_{i_{1}}X_{i_{2}%
}X_{i_{3}}X_{i_{4}}=\sum_{u\in\mathbb{N}_{+}}X_{u}X_{u}X_{u}X_{u}=\sum
_{u\in\mathbb{N}_{+}}X_{u}^{4}.
\end{align*}

\textbf{(d)} If $\mathbf{P}$ is the set partition $\varnothing$ of $\left\{
{}\right\}  $ (so we have $N=0$), then%
\begin{align*}
P_{\mathbf{P}}  &  =\sum_{\substack{\left(  {}\right)  \in\left(
\mathbb{N}_{+}\right)  ^{0};\\i_{a}=i_{b}\text{ whenever }a\sim_{\mathbf{P}}%
b}}\left(  \text{empty product}\right) \\
&  =\left(  \text{empty product}\right)  \ \ \ \ \ \ \ \ \ \ \left(
\text{since there is only one }0\text{-tuple}\right) \\
&  =1.
\end{align*}

\end{example}

We shall now assign an equivalence relation to any finite graph $\left(
V,E\right)  $ and any list $\mathbf{t}$ of its vertices:

\begin{proposition}
\label{prop.graph-setpar}Let $\left(  V,B\right)  $ be a finite graph. Then,
according to Definition \ref{def.connectedness} \textbf{(a)}, an equivalence
relation $\sim_{\left(  V,B\right)  }$ is defined on the set $V$.

Let $\mathbf{t}=\left(  t_{1},t_{2},\ldots,t_{N}\right)  $ be a finite list of
elements of $V$. Let $\approx$ be the relation on the set $\left\{
1,2,\ldots,N\right\}  $ defined as follows: Two elements $i$ and $j$ of
$\left\{  1,2,\ldots,N\right\}  $ shall satisfy $i\approx j$ if and only if
$t_{i}\sim_{\left(  V,B\right)  }t_{j}$.

Then, this relation $\approx$ is an equivalence relation.
\end{proposition}

\begin{vershort}
\begin{proof}
Straightforward and easy.
\end{proof}
\end{vershort}

\begin{verlong}
\begin{proof}
[Proof of Proposition \ref{prop.graph-setpar}.]This is a straightforward and
easy argument, but we give it nevertheless for the sake of completeness.

We shall show that the relation $\approx$ is reflexive, symmetric and transitive.

\textit{Proof of reflexivity:} Let $i\in\left\{  1,2,\ldots,N\right\}  $. We
shall show that $i\approx i$.

The relation $\sim_{\left(  V,B\right)  }$ is an equivalence relation, and
thus is reflexive. Hence, we have $t_{i}\sim_{\left(  V,B\right)  }t_{i}$.
However, $i\approx i$ holds if and only if we have $t_{i}\sim_{\left(
V,B\right)  }t_{i}$ (by the definition of the relation $\approx$). Hence,
$i\approx i$ holds (since we have $t_{i}\sim_{\left(  V,B\right)  }t_{i}$).

Forget that we fixed $i$. We thus have shown that $i\approx i$ for each
$i\in\left\{  1,2,\ldots,N\right\}  $. In other words, the relation $\approx$
is reflexive.

\textit{Proof of symmetry:} Let $i$ and $j$ be two elements of $\left\{
1,2,\ldots,N\right\}  $ satisfying $i\approx j$. We shall show that $j\approx
i$.

We have $i\approx j$ if and only if we have $t_{i}\sim_{\left(  V,B\right)
}t_{j}$ (by the definition of the relation $\approx$). Hence, we have
$t_{i}\sim_{\left(  V,B\right)  }t_{j}$ (since we have $i\approx j$).

The relation $\sim_{\left(  V,B\right)  }$ is an equivalence relation, and
thus is symmetric. Hence, from $t_{i}\sim_{\left(  V,B\right)  }t_{j}$, we
obtain $t_{j}\sim_{\left(  V,B\right)  }t_{i}$. However, $j\approx i$ holds if
and only if we have $t_{j}\sim_{\left(  V,B\right)  }t_{i}$ (by the definition
of the relation $\approx$). Hence, $j\approx i$ holds (since we have
$t_{j}\sim_{\left(  V,B\right)  }t_{i}$).

Forget that we fixed $i$ and $j$. We thus have shown that if $i$ and $j$ are
two elements of $\left\{  1,2,\ldots,N\right\}  $ satisfying $i\approx j$,
then $j\approx i$. In other words, the relation $\approx$ is symmetric.

\textit{Proof of transitivity:} Let $i$, $j$ and $k$ be three elements of
$\left\{  1,2,\ldots,N\right\}  $ satisfying $i\approx j$ and $j\approx k$. We
shall show that $i\approx k$.

We have $i\approx j$ if and only if we have $t_{i}\sim_{\left(  V,B\right)
}t_{j}$ (by the definition of the relation $\approx$). Hence, we have
$t_{i}\sim_{\left(  V,B\right)  }t_{j}$ (since we have $i\approx j$).

We have $j\approx k$ if and only if we have $t_{j}\sim_{\left(  V,B\right)
}t_{k}$ (by the definition of the relation $\approx$). Hence, we have
$t_{j}\sim_{\left(  V,B\right)  }t_{k}$ (since we have $j\approx k$).

The relation $\sim_{\left(  V,B\right)  }$ is an equivalence relation, and
thus is transitive. Hence, from $t_{i}\sim_{\left(  V,B\right)  }t_{j}$ and
$t_{j}\sim_{\left(  V,B\right)  }t_{k}$, we obtain $t_{i}\sim_{\left(
V,B\right)  }t_{k}$. However, $i\approx k$ holds if and only if we have
$t_{i}\sim_{\left(  V,B\right)  }t_{k}$ (by the definition of the relation
$\approx$). Hence, $i\approx k$ holds (since we have $t_{i}\sim_{\left(
V,B\right)  }t_{k}$).

Forget that we fixed $i$, $j$ and $k$. We thus have shown that if $i$, $j$ and
$k$ are three elements of $\left\{  1,2,\ldots,N\right\}  $ satisfying
$i\approx j$ and $j\approx k$, then $i\approx k$. In other words, the relation
$\approx$ is transitive.

We have now shown that the relation $\approx$ is reflexive, symmetric and
transitive. In other words, this relation $\approx$ is an equivalence relation
(because an equivalence relation is defined to be a binary relation that is
reflexive, symmetric and transitive). This proves Proposition
\ref{prop.graph-setpar}.
\end{proof}
\end{verlong}

As we know from Theorem \ref{thm.setpar.eqrel}, an equivalence relation is
\textquotedblleft essentially the same as\textquotedblright\ a set partition.
Thus, in particular, we can turn the equivalence relation defined in
Proposition \ref{prop.graph-setpar} into a set partition:

\begin{definition}
\label{def.graph-setpar}Let $\left(  V,B\right)  $ be a finite graph. Let
$\mathbf{t}=\left(  t_{1},t_{2},\ldots,t_{N}\right)  $ be a finite list of
elements of $V$.

\textbf{(a)} Let $\approx_{\left(  V,B,\mathbf{t}\right)  }$ be the relation
$\approx$ on the set $\left\{  1,2,\ldots,N\right\}  $ defined in Proposition
\ref{prop.graph-setpar}. As we know from Proposition \ref{prop.graph-setpar},
this relation $\approx$ is an equivalence relation. In other words, the
relation $\approx_{\left(  V,B,\mathbf{t}\right)  }$ is an equivalence relation.

\textbf{(b)} Therefore, Theorem \ref{thm.setpar.eqrel} \textbf{(a)} (applied
to $X=\left\{  1,2,\ldots,N\right\}  $ and $\left(  \sim\right)  =\left(
\approx_{\left(  V,B,\mathbf{t}\right)  }\right)  $) shows that the set%
\[
\left\{  1,2,\ldots,N\right\}  /\left(  \approx_{\left(  V,B,\mathbf{t}%
\right)  }\right)  =\left\{  \text{all }\approx_{\left(  V,B,\mathbf{t}%
\right)  }\text{-equivalence classes}\right\}
\]
is a set partition of $\left\{  1,2,\ldots,N\right\}  $. We shall denote this
set partition by $\mathbf{P}\left(  V,B,\mathbf{t}\right)  $.
\end{definition}

\begin{example}
Let $\left(  V,B\right)  $ be the finite graph with $V=\left\{
u,v,w,x,y\right\}  $ and $B=\left\{  \left\{  u,v\right\}  ,\ \left\{
v,w\right\}  ,\ \left\{  x,y\right\}  \right\}  $. Then, the equivalence
relation $\sim_{\left(  V,B\right)  }$ from Definition \ref{def.connectedness}
\textbf{(a)} satisfies $u\sim_{\left(  V,B\right)  }v\sim_{\left(  V,B\right)
}w$ and $x\sim_{\left(  V,B\right)  }y$.

Let $\mathbf{t}=\left(  t_{1},t_{2},\ldots,t_{N}\right)  $ be the list
$\left(  u,v,x,y,x,u\right)  $ of elements of $V$ (so that $N=6$ and $t_{1}=u$
and $t_{2}=v$ and $t_{3}=x$ and $t_{4}=y$ and $t_{5}=x$ and $t_{6}=u$). Then,
the equivalence relation $\approx_{\left(  V,B,\mathbf{t}\right)  }$ from
Definition \ref{def.graph-setpar} \textbf{(a)} satisfies%
\begin{align*}
1  &  \approx_{\left(  V,B,\mathbf{t}\right)  }2\approx_{\left(
V,B,\mathbf{t}\right)  }6\ \ \ \ \ \ \ \ \ \ \left(  \text{since }t_{1}%
\sim_{\left(  V,B\right)  }t_{2}\sim_{\left(  V,B\right)  }t_{6}\right)
\ \ \ \ \ \ \ \ \ \ \text{and}\\
3  &  \approx_{\left(  V,B,\mathbf{t}\right)  }4\approx_{\left(
V,B,\mathbf{t}\right)  }5\ \ \ \ \ \ \ \ \ \ \left(  \text{since }t_{3}%
\sim_{\left(  V,B\right)  }t_{4}\sim_{\left(  V,B\right)  }t_{5}\right)
\end{align*}
and no further relations (except for the ones that follow from the relations
just given using reflexivity, symmetry and transitivity). Thus, the set
partition $\mathbf{P}\left(  V,B,\mathbf{t}\right)  $ from Definition
\ref{def.graph-setpar} \textbf{(b)} is%
\[
\left\{  \left\{  1,2,6\right\}  ,\ \left\{  3,4,5\right\}  \right\}  .
\]

\end{example}

We now have all notations in place to state noncommutative analogues of
Theorems \ref{thm.wambichromsym.empty} and \ref{thm.wambichromsym.varis} and
Corollaries \ref{cor.wambichromsym.K-free} and \ref{cor.wambichromsym.NBC}:

\begin{theorem}
\label{thm.ncambichromsym.empty}Let $G=\left(  V,E,\varphi\right)  $ be a
finite ambigraph. Let $\mathbf{t}=\left(  t_{1},t_{2},\ldots,t_{N}\right)  $
be a finite list of elements of $V$ that contains each element of $V$ at least
once. Then,%
\[
Y_{G,\mathbf{t}}=\sum_{F\subseteq E}\left(  -1\right)  ^{\left\vert
F\right\vert }P_{\mathbf{P}\left(  V,\operatorname*{union}F,\mathbf{t}\right)
}.
\]
(Here, of course, the pair $\left(  V,\operatorname*{union}F\right)  $ is
regarded as a graph, and the expression $\mathbf{P}\left(
V,\operatorname*{union}F,\mathbf{t}\right)  $ is understood according to
Definition \ref{def.graph-setpar}.)
\end{theorem}

\begin{theorem}
\label{thm.ncambichromsym.varis}Let $G=\left(  V,E,\varphi\right)  $ be a
finite ambigraph. Let $\mathbf{t}=\left(  t_{1},t_{2},\ldots,t_{N}\right)  $
be a finite list of elements of $V$ that contains each element of $V$ at least
once. Let $X$ be a totally ordered set. Let $\ell:E\rightarrow X$ be a
labeling function. Let $\mathfrak{K}$ be some set of broken circuits of $G$
(not necessarily containing all of them). Let $a_{K}$ be an element of
$\mathbf{k}$ for every $K\in\mathfrak{K}$. Then,%
\[
Y_{G,\mathbf{t}}=\sum_{F\subseteq E}\left(  -1\right)  ^{\left\vert
F\right\vert }\left(  \prod_{\substack{K\in\mathfrak{K};\\K\subseteq F}%
}a_{K}\right)  P_{\mathbf{P}\left(  V,\operatorname*{union}F,\mathbf{t}%
\right)  }.
\]

\end{theorem}

\begin{corollary}
\label{cor.ncambichromsym.K-free}Let $G=\left(  V,E,\varphi\right)  $ be a
finite ambigraph. Let $\mathbf{t}=\left(  t_{1},t_{2},\ldots,t_{N}\right)  $
be a finite list of elements of $V$ that contains each element of $V$ at least
once. Let $X$ be a totally ordered set. Let $\ell:E\rightarrow X$ be a
labeling function. Let $\mathfrak{K}$ be some set of broken circuits of $G$
(not necessarily containing all of them). Then,%
\[
Y_{G,\mathbf{t}}=\sum_{\substack{F\subseteq E;\\F\text{ is }\mathfrak{K}%
\text{-free}}}\left(  -1\right)  ^{\left\vert F\right\vert }P_{\mathbf{P}%
\left(  V,\operatorname*{union}F,\mathbf{t}\right)  }.
\]

\end{corollary}

\begin{corollary}
\label{cor.ncambichromsym.NBC}Let $G=\left(  V,E,\varphi\right)  $ be a finite
ambigraph. Let $\mathbf{t}=\left(  t_{1},t_{2},\ldots,t_{N}\right)  $ be a
finite list of elements of $V$ that contains each element of $V$ at least
once. Let $X$ be a totally ordered set. Let $\ell:E\rightarrow X$ be a
labeling function. Then,%
\[
Y_{G,\mathbf{t}}=\sum_{\substack{F\subseteq E;\\F\text{ contains no
broken}\\\text{circuit of }G\text{ as a subset}}}\left(  -1\right)
^{\left\vert F\right\vert }P_{\mathbf{P}\left(  V,\operatorname*{union}%
F,\mathbf{t}\right)  }.
\]

\end{corollary}

Since any graph or loopless multigraph can be viewed as an ambigraph, it is
easy to see that Theorem \ref{thm.ncambichromsym.empty} and Corollary
\ref{cor.ncambichromsym.NBC} generalize \cite[Theorem 3.6]{GebSag01} and
\cite[Theorem 3.8]{GebSag01}, respectively.

In order to prove these four results, we proceed similarly to the commutative
case, which we have studied to exhaustion. Instead of Lemma \ref{lem.wEqs.sum}%
, we need the following noncommutative analogue:

\begin{lemma}
\label{lem.ncEqs.sum}Let $\left(  V,B\right)  $ be a finite graph. Let
$\mathbf{t}=\left(  t_{1},t_{2},\ldots,t_{N}\right)  $ be a finite list of
elements of $V$ that contains each element of $V$ at least once. Then,%
\[
\sum_{\substack{f:V\rightarrow\mathbb{N}_{+};\\B\subseteq\operatorname*{Eqs}%
f}}\mathbf{X}_{f,\mathbf{t}}=P_{\mathbf{P}\left(  V,B,\mathbf{t}\right)  }.
\]

\end{lemma}

\begin{proof}
[Proof of Lemma \ref{lem.ncEqs.sum}.]The definition of $P_{\mathbf{P}\left(
V,B,\mathbf{t}\right)  }$ yields%
\begin{equation}
P_{\mathbf{P}\left(  V,B,\mathbf{t}\right)  }=\sum_{\substack{\left(
i_{1},i_{2},\ldots,i_{N}\right)  \in\left(  \mathbb{N}_{+}\right)
^{N};\\i_{a}=i_{b}\text{ whenever }a\sim_{\mathbf{P}\left(  V,B,\mathbf{t}%
\right)  }b}}X_{i_{1}}X_{i_{2}}\cdots X_{i_{N}} \label{pf.lem.ncEqs.sum.P=}%
\end{equation}
(where the condition \textquotedblleft$i_{a}=i_{b}$ whenever $a\sim
_{\mathbf{P}\left(  V,B,\mathbf{t}\right)  }b$\textquotedblright\ is shorthand
for \textquotedblleft$i_{a}=i_{b}$ for any two elements $a,b\in\left\{
1,2,\ldots,N\right\}  $ that satisfy $a\sim_{\mathbf{P}\left(  V,B,\mathbf{t}%
\right)  }b$\textquotedblright).

\begin{vershort}
However, we know from Theorem \ref{thm.setpar.eqrel} that every equivalence
relation $\sim$ on a set $X$ can be canonically transformed into a set
partition of this set (namely, the set partition $X/\left(  \sim\right)  $,
which consists of the $\sim$-equivalence classes), and conversely, every set
partition $\mathbf{P}$ of a set $X$ can be canonically transformed into an
equivalence relation $\sim_{\mathbf{P}}$ on this set. These two
transformations are mutually inverse; in particular, if an equivalence
relation $\sim$ gives rise to a set partition $X/\left(  \sim\right)  $, then
the equivalence relation $\sim_{X/\left(  \sim\right)  }$ constructed from the
latter set partition is again the original relation $\sim$.

Thus, it follows that the relation $\sim_{\mathbf{P}\left(  V,B,\mathbf{t}%
\right)  }$ is precisely the relation $\approx_{\left(  V,B,\mathbf{t}\right)
}$ (because the relation $\sim_{\mathbf{P}\left(  V,B,\mathbf{t}\right)  }$ is
constructed from the set partition $\mathbf{P}\left(  V,B,\mathbf{t}\right)
$, but the latter set partition $\mathbf{P}\left(  V,B,\mathbf{t}\right)  $ is
in turn constructed from the equivalence relation $\approx_{\left(
V,B,\mathbf{t}\right)  }$).

On the other hand, the relation $\approx_{\left(  V,B,\mathbf{t}\right)  }$ is
defined as the relation $\approx$ on the set $\left\{  1,2,\ldots,N\right\}  $
for which two elements satisfy $i\approx j$ if and only if $t_{i}\sim_{\left(
V,B\right)  }t_{j}$. Hence, for any two elements $a,b\in\left\{
1,2,\ldots,N\right\}  $, we have the equivalence%
\[
\left(  a\approx_{\left(  V,B,\mathbf{t}\right)  }b\right)
\ \Longleftrightarrow\ \left(  t_{a}\sim_{\left(  V,B\right)  }t_{b}\right)
.
\]
Since the relation $\sim_{\mathbf{P}\left(  V,B,\mathbf{t}\right)  }$ is
precisely the relation $\approx_{\left(  V,B,\mathbf{t}\right)  }$, we can
rewrite this as follows: For any two elements $a,b\in\left\{  1,2,\ldots
,N\right\}  $, we have the equivalence
\[
\left(  a\sim_{\mathbf{P}\left(  V,B,\mathbf{t}\right)  }b\right)
\ \Longleftrightarrow\ \left(  t_{a}\sim_{\left(  V,B\right)  }t_{b}\right)
.
\]

Therefore, we can replace the condition \textquotedblleft$a\sim_{\mathbf{P}%
\left(  V,B,\mathbf{t}\right)  }b$\textquotedblright\ under the summation sign
in (\ref{pf.lem.ncEqs.sum.P=}) by \textquotedblleft$t_{a}\sim_{\left(
V,B\right)  }t_{b}$\textquotedblright. As a result, (\ref{pf.lem.ncEqs.sum.P=}%
) rewrites as follows:%
\begin{equation}
P_{\mathbf{P}\left(  V,B,\mathbf{t}\right)  }=\sum_{\substack{\left(
i_{1},i_{2},\ldots,i_{N}\right)  \in\left(  \mathbb{N}_{+}\right)
^{N};\\i_{a}=i_{b}\text{ whenever }t_{a}\sim_{\left(  V,B\right)  }t_{b}%
}}X_{i_{1}}X_{i_{2}}\cdots X_{i_{N}}. \label{pf.lem.ncEqs.sum.short.P=2}%
\end{equation}

\end{vershort}

\begin{verlong}
However, for any two elements $a,b\in\left\{  1,2,\ldots,N\right\}  $, we have
the equivalence
\begin{equation}
\left(  a\sim_{\mathbf{P}\left(  V,B,\mathbf{t}\right)  }b\right)
\ \Longleftrightarrow\ \left(  t_{a}\sim_{\left(  V,B\right)  }t_{b}\right)
\label{pf.lem.ncEqs.sum.equivalence}%
\end{equation}
\footnote{\textit{Proof.} Let $X=\left\{  1,2,\ldots,N\right\}  $. Then,
$\mathbf{P}\left(  V,B,\mathbf{t}\right)  $ is a set partition of $X$ (since
$\mathbf{P}\left(  V,B,\mathbf{t}\right)  $ is a set partition of $\left\{
1,2,\ldots,N\right\}  $).
\par
Let us recall how the relation $\sim_{\mathbf{P}\left(  V,B,\mathbf{t}\right)
}$ is defined (according to Theorem \ref{thm.setpar.eqrel} \textbf{(b)}): This
relation is the relation $\sim$ on the set $X$ such that for any two elements
$a$ and $b$ of $X$, we have $a\sim b$ if and only if the elements $a$ and $b$
belong to the same block of $\mathbf{P}\left(  V,B,\mathbf{t}\right)  $.
Hence, for any two elements $a$ and $b$ of $X$, we have $a\sim_{\mathbf{P}%
\left(  V,B,\mathbf{t}\right)  }b$ if and only if the elements $a$ and $b$
belong to the same block of $\mathbf{P}\left(  V,B,\mathbf{t}\right)  $. In
other words, for any two elements $a$ and $b$ of $X$, we have the equivalence%
\begin{align}
&  \ \left(  a\sim_{\mathbf{P}\left(  V,B,\mathbf{t}\right)  }b\right)
\nonumber\\
&  \Longleftrightarrow\ \left(  \text{the elements }a\text{ and }b\text{
belong to the same block of }\mathbf{P}\left(  V,B,\mathbf{t}\right)  \right)
. \label{pf.lem.ncEqs.sum.equivalence.pf.1}%
\end{align}
\par
Let $\approx_{\left(  V,B,\mathbf{t}\right)  }$ be the relation $\approx$ on
the set $\left\{  1,2,\ldots,N\right\}  $ defined in Proposition
\ref{prop.graph-setpar}. As we know from Proposition \ref{prop.graph-setpar},
this relation $\approx$ is an equivalence relation. In other words, the
relation $\approx_{\left(  V,B,\mathbf{t}\right)  }$ is an equivalence
relation.
\par
Now, recall how the set partition $\mathbf{P}\left(  V,B,\mathbf{t}\right)  $
is defined (according to Definition \ref{def.graph-setpar}): It is defined by%
\[
\mathbf{P}\left(  V,B,\mathbf{t}\right)  =\left\{  1,2,\ldots,N\right\}
/\left(  \approx_{\left(  V,B,\mathbf{t}\right)  }\right)  =\left\{  \text{all
}\approx_{\left(  V,B,\mathbf{t}\right)  }\text{-equivalence classes}\right\}
.
\]
Hence, the blocks of $\mathbf{P}\left(  V,B,\mathbf{t}\right)  $ are the
$\approx_{\left(  V,B,\mathbf{t}\right)  }$-equivalence classes.
\par
Now, let $a$ and $b$ be two elements of $\left\{  1,2,\ldots,N\right\}  $.
Thus, $a$ and $b$ are two elements of $X$ (since $X=\left\{  1,2,\ldots
,N\right\}  $). Therefore, we have the following chain of equivalences:%
\begin{align}
&  \ \left(  a\sim_{\mathbf{P}\left(  V,B,\mathbf{t}\right)  }b\right)
\nonumber\\
&  \Longleftrightarrow\ \left(  \text{the elements }a\text{ and }b\text{
belong to the same block of }\mathbf{P}\left(  V,B,\mathbf{t}\right)  \right)
\nonumber\\
&  \ \ \ \ \ \ \ \ \ \ \ \ \ \ \ \ \ \ \ \ \left(  \text{by
(\ref{pf.lem.ncEqs.sum.equivalence.pf.1})}\right) \nonumber\\
&  \Longleftrightarrow\ \left(  \text{the elements }a\text{ and }b\text{
belong to the same }\approx_{\left(  V,B,\mathbf{t}\right)  }%
\text{-equivalence class}\right) \nonumber\\
&  \ \ \ \ \ \ \ \ \ \ \ \ \ \ \ \ \ \ \ \ \left(  \text{since the blocks of
}\mathbf{P}\left(  V,B,\mathbf{t}\right)  \text{ are the }\approx_{\left(
V,B,\mathbf{t}\right)  }\text{-equivalence classes}\right) \nonumber\\
&  \Longleftrightarrow\ \left(  a\approx_{\left(  V,B,\mathbf{t}\right)
}b\right)  \label{pf.lem.ncEqs.sum.equivalence.pf.2}%
\end{align}
(because two elements of $\left\{  1,2,\ldots,N\right\}  $ belong to the same
$\approx_{\left(  V,B,\mathbf{t}\right)  }$-equivalence class if and only if
they are related by the relation $\approx_{\left(  V,B,\mathbf{t}\right)  }$).
\par
However, the relation $\approx_{\left(  V,B,\mathbf{t}\right)  }$ was defined
as the relation $\approx$ on the set $\left\{  1,2,\ldots,N\right\}  $ defined
in Proposition \ref{prop.graph-setpar}. But the latter relation $\approx$ was
defined as follows: Two elements $i$ and $j$ of $\left\{  1,2,\ldots
,N\right\}  $ shall satisfy $i\approx j$ if and only if $t_{i}\sim_{\left(
V,B\right)  }t_{j}$.
\par
Since we have denoted this relation $\approx$ by $\approx_{\left(
V,B,\mathbf{t}\right)  }$, we can rewrite this fact as follows: Two elements
$i$ and $j$ of $\left\{  1,2,\ldots,N\right\}  $ shall satisfy $i\approx
_{\left(  V,B,\mathbf{t}\right)  }j$ if and only if $t_{i}\sim_{\left(
V,B\right)  }t_{j}$. In other words, for any two elements $i$ and $j$ of
$\left\{  1,2,\ldots,N\right\}  $, we have the equivalence $\left(
i\approx_{\left(  V,B,\mathbf{t}\right)  }j\right)  \ \Longleftrightarrow
\ \left(  t_{i}\sim_{\left(  V,B\right)  }t_{j}\right)  $.
\par
Applying this to $i=a$ and $j=b$, we obtain the equivalence $\left(
a\approx_{\left(  V,B,\mathbf{t}\right)  }b\right)  \ \Longleftrightarrow
\ \left(  t_{a}\sim_{\left(  V,B\right)  }t_{b}\right)  $. Hence, we have the
equivalence
\begin{align*}
\left(  a\sim_{\mathbf{P}\left(  V,B,\mathbf{t}\right)  }b\right)  \  &
\Longleftrightarrow\ \left(  a\approx_{\left(  V,B,\mathbf{t}\right)
}b\right)  \ \ \ \ \ \ \ \ \ \ \left(  \text{by
(\ref{pf.lem.ncEqs.sum.equivalence.pf.2})}\right) \\
&  \Longleftrightarrow\ \left(  t_{a}\sim_{\left(  V,B\right)  }t_{b}\right)
.
\end{align*}
This proves (\ref{pf.lem.ncEqs.sum.equivalence}).}.
\end{verlong}

\begin{vershort}
Now, we define two sets%
\[
\mathcal{F}=\left\{  g:V\rightarrow\mathbb{N}_{+}\text{ is a map }\mid\text{
}B\subseteq\operatorname*{Eqs}g\right\}
\]
and%
\[
\mathcal{I}=\left\{  \left(  i_{1},i_{2},\ldots,i_{N}\right)  \in\left(
\mathbb{N}_{+}\right)  ^{N}\ \mid\ i_{a}=i_{b}\text{ whenever }t_{a}%
\sim_{\left(  V,B\right)  }t_{b}\right\}  .
\]

\end{vershort}

\begin{verlong}
Now, we define two sets%
\[
\mathcal{F}=\left\{  g:V\rightarrow\mathbb{N}_{+}\text{ is a map }\mid\text{
}B\subseteq\operatorname*{Eqs}g\right\}
\]
and%
\[
\mathcal{I}=\left\{  \left(  i_{1},i_{2},\ldots,i_{N}\right)  \in\left(
\mathbb{N}_{+}\right)  ^{N}\ \mid\ i_{a}=i_{b}\text{ whenever }a\sim
_{\mathbf{P}\left(  V,B,\mathbf{t}\right)  }b\right\}  .
\]

\end{verlong}

We claim the following:

\begin{statement}
\textit{Claim 1:} For any $f\in\mathcal{F}$, we have $\left(  f\left(
t_{1}\right)  ,f\left(  t_{2}\right)  ,\ldots,f\left(  t_{N}\right)  \right)
\in\mathcal{I}$.
\end{statement}

\begin{vershort}
[\textit{Proof of Claim 1:} Let $f\in\mathcal{F}$. Thus, by the definition of
$\mathcal{F}$, we conclude that $f:V\rightarrow\mathbb{N}_{+}$ is a map
satisfying $B\subseteq\operatorname*{Eqs}f$.

We need to show that $\left(  f\left(  t_{1}\right)  ,f\left(  t_{2}\right)
,\ldots,f\left(  t_{N}\right)  \right)  \in\mathcal{I}$. By the definition of
$\mathcal{I}$, this requires us to show that $f\left(  t_{a}\right)  =f\left(
t_{b}\right)  $ whenever $t_{a}\sim_{\left(  V,B\right)  }t_{b}$ (that is,
whenever $a,b\in\left\{  1,2,\ldots,N\right\}  $ are two elements satisfying
$t_{a}\sim_{\left(  V,B\right)  }t_{b}$). So let us show this.

Let $a,b\in\left\{  1,2,\ldots,N\right\}  $ be two elements satisfying
$t_{a}\sim_{\left(  V,B\right)  }t_{b}$. We must show that $f\left(
t_{a}\right)  =f\left(  t_{b}\right)  $.

We have $B\subseteq\operatorname*{Eqs}f$. Hence, if $x$ and $y$ are two
elements of $V$ lying in the same connected component of $\left(  V,B\right)
$, then%
\begin{equation}
f\left(  x\right)  =f\left(  y\right)  .
\label{pf.lem.ncEqs.sum.short.c1.pf.1}%
\end{equation}
(Indeed, this can be shown in the same way as we established
(\ref{pf.lem.Eqs.sum.short.surj.2}) during our proof of Lemma
\ref{lem.Eqs.sum}.)

However, we have $t_{a}\sim_{\left(  V,B\right)  }t_{b}$. In other words, the
elements $t_{a}$ and $t_{b}$ lie in the same connected component of $\left(
V,B\right)  $ (since the connected components of $\left(  V,B\right)  $ are
the $\sim_{\left(  V,B\right)  }$-equivalence classes). Thus,
(\ref{pf.lem.ncEqs.sum.short.c1.pf.1}) (applied to $x=t_{a}$ and $y=t_{b}$)
yields that $f\left(  t_{a}\right)  =f\left(  t_{b}\right)  $. As we
explained, this completes the proof of Claim 1.]
\end{vershort}

\begin{verlong}
[\textit{Proof of Claim 1:} Let $f\in\mathcal{F}$. Then,%
\[
f\in\mathcal{F}=\left\{  g:V\rightarrow\mathbb{N}_{+}\text{ is a map }%
\mid\text{ }B\subseteq\operatorname*{Eqs}g\right\}  .
\]
In other words, $f$ is a map $g:V\rightarrow\mathbb{N}_{+}$ satisfying
$B\subseteq\operatorname*{Eqs}g$. In other words, $f$ is a map from $V$ to
$\mathbb{N}_{+}$ and satisfies $B\subseteq\operatorname*{Eqs}f$.

Let $a,b\in\left\{  1,2,\ldots,N\right\}  $ be two elements that satisfy
$a\sim_{\mathbf{P}\left(  V,B,\mathbf{t}\right)  }b$. We shall show that
$f\left(  t_{a}\right)  =f\left(  t_{b}\right)  $.

From (\ref{pf.lem.ncEqs.sum.equivalence}), we know that the statements
\textquotedblleft$a\sim_{\mathbf{P}\left(  V,B,\mathbf{t}\right)  }%
b$\textquotedblright\ and \textquotedblleft$t_{a}\sim_{\left(  V,B\right)
}t_{b}$\textquotedblright\ are equivalent. Hence, we have $t_{a}\sim_{\left(
V,B\right)  }t_{b}$ (since we have $a\sim_{\mathbf{P}\left(  V,B,\mathbf{t}%
\right)  }b$).

Now, let $\sim$ denote the equivalence relation $\sim_{\left(  V,B\right)  }$.
Thus, we have $t_{a}\sim t_{b}$ (since $t_{a}\sim_{\left(  V,B\right)  }t_{b}$).

Let $Y=\mathbb{N}_{+}$. Thus, $f:V\rightarrow Y$ is a map (since
$f:V\rightarrow\mathbb{N}_{+}$ is a map). A set $Y_{\sim}^{V}$ is defined
(according to Definition \ref{def.relquot.maps} \textbf{(b)}). Lemma
\ref{lem.Eqs.sum-aux} yields that we have the following logical equivalence of
statements:%
\[
\left(  B\subseteq\operatorname*{Eqs}f\right)  \ \Longleftrightarrow\ \left(
f\in Y_{\sim}^{V}\right)  .
\]
Hence, we have $f\in Y_{\sim}^{V}$ (since we have $B\subseteq
\operatorname*{Eqs}f$). Therefore,%
\[
f\in Y_{\sim}^{V}=\left\{  g\in Y^{V}\ \mid\ g\left(  x\right)  =g\left(
y\right)  \text{ for any }x\in V\text{ and }y\in V\text{ satisfying }x\sim
y\right\}
\]
(by the definition of $Y_{\sim}^{V}$). In other words, $f$ is a $g\in Y^{V}$
satisfying
\[
g\left(  x\right)  =g\left(  y\right)  \text{ for any }x\in V\text{ and }y\in
V\text{ satisfying }x\sim y.
\]
In other words, $f$ is an element of $Y^{V}$ and satisfies%
\begin{equation}
f\left(  x\right)  =f\left(  y\right)  \text{ for any }x\in V\text{ and }y\in
V\text{ satisfying }x\sim y. \label{pf.lem.ncEqs.sum.c1.pf.5}%
\end{equation}

We can apply (\ref{pf.lem.ncEqs.sum.c1.pf.5}) to $x=t_{a}$ and $y=t_{b}$
(since $t_{a}\in V$ and $t_{b}\in V$ and $t_{a}\sim t_{b}$). Thus, we obtain
$f\left(  t_{a}\right)  =f\left(  t_{b}\right)  $.

Now, forget that we fixed $a$ and $b$. We thus have shown that $f\left(
t_{a}\right)  =f\left(  t_{b}\right)  $ for any two elements $a,b\in\left\{
1,2,\ldots,N\right\}  $ that satisfy $a\sim_{\mathbf{P}\left(  V,B,\mathbf{t}%
\right)  }b$. In other words, we have $f\left(  t_{a}\right)  =f\left(
t_{b}\right)  $ whenever $a\sim_{\mathbf{P}\left(  V,B,\mathbf{t}\right)  }b$.

Hence, the $N$-tuple $\left(  f\left(  t_{1}\right)  ,f\left(  t_{2}\right)
,\ldots,f\left(  t_{N}\right)  \right)  $ is an $N$-tuple $\left(  i_{1}%
,i_{2},\ldots,i_{N}\right)  \in\left(  \mathbb{N}_{+}\right)  ^{N}$ satisfying
$i_{a}=i_{b}$ whenever $a\sim_{\mathbf{P}\left(  V,B,\mathbf{t}\right)  }b$.
In other words,%
\begin{align*}
\left(  f\left(  t_{1}\right)  ,f\left(  t_{2}\right)  ,\ldots,f\left(
t_{N}\right)  \right)   &  \in\left\{  \left(  i_{1},i_{2},\ldots
,i_{N}\right)  \in\left(  \mathbb{N}_{+}\right)  ^{N}\ \mid\ i_{a}=i_{b}\text{
whenever }a\sim_{\mathbf{P}\left(  V,B,\mathbf{t}\right)  }b\right\} \\
&  =\mathcal{I}\ \ \ \ \ \ \ \ \ \ \left(  \text{by the definition of
}\mathcal{I}\right)  .
\end{align*}
This proves Claim 1.]
\end{verlong}

Thanks to Claim 1, we can define a map%
\begin{align*}
\Psi:\mathcal{F}  &  \rightarrow\mathcal{I},\\
f  &  \mapsto\left(  f\left(  t_{1}\right)  ,f\left(  t_{2}\right)
,\ldots,f\left(  t_{N}\right)  \right)  .
\end{align*}
Consider this map $\Psi$. We claim the following:

\begin{statement}
\textit{Claim 2:} The map $\Psi$ is injective.
\end{statement}

\begin{vershort}
[\textit{Proof of Claim 2:} The list $\left(  t_{1},t_{2},\ldots,t_{N}\right)
$ contains each element of $V$ at least once (according to the hypotheses of
Lemma \ref{lem.ncEqs.sum}). Thus, if $f\in\mathcal{F}$ is arbitrary, then the
list $\left(  f\left(  t_{1}\right)  ,f\left(  t_{2}\right)  ,\ldots,f\left(
t_{N}\right)  \right)  $ contains each value of $f$ at least once. Therefore,
any $f\in\mathcal{F}$ can be uniquely reconstructed from this list $\left(
f\left(  t_{1}\right)  ,f\left(  t_{2}\right)  ,\ldots,f\left(  t_{N}\right)
\right)  $. In other words, any $f\in\mathcal{F}$ can be uniquely
reconstructed from $\Psi\left(  f\right)  $ (since the definition of $\Psi$
yields $\Psi\left(  f\right)  =\left(  f\left(  t_{1}\right)  ,f\left(
t_{2}\right)  ,\ldots,f\left(  t_{N}\right)  \right)  $). In other words, the
map $\Psi$ is injective. This proves Claim 2.]
\end{vershort}

\begin{verlong}
[\textit{Proof of Claim 2:} Let $f$ and $g$ be two elements of $\mathcal{F}$
that satisfy $\Psi\left(  f\right)  =\Psi\left(  g\right)  $. We shall show
that $f=g$.

Note that both $f$ and $g$ are elements of $\mathcal{F}$, and thus are maps
$V\rightarrow\mathbb{N}_{+}$.

The definition of $\Psi$ yields $\Psi\left(  f\right)  =\left(  f\left(
t_{1}\right)  ,f\left(  t_{2}\right)  ,\ldots,f\left(  t_{N}\right)  \right)
$ and $\Psi\left(  g\right)  =\left(  g\left(  t_{1}\right)  ,g\left(
t_{2}\right)  ,\ldots,g\left(  t_{N}\right)  \right)  $. Hence,%
\[
\left(  f\left(  t_{1}\right)  ,f\left(  t_{2}\right)  ,\ldots,f\left(
t_{N}\right)  \right)  =\Psi\left(  f\right)  =\Psi\left(  g\right)  =\left(
g\left(  t_{1}\right)  ,g\left(  t_{2}\right)  ,\ldots,g\left(  t_{N}\right)
\right)  .
\]
In other words,%
\begin{equation}
f\left(  t_{i}\right)  =g\left(  t_{i}\right)  \ \ \ \ \ \ \ \ \ \ \text{for
each }i\in\left\{  1,2,\ldots,N\right\}  . \label{pf.lem.ncEqs.sum.c2.pf.1}%
\end{equation}

Now, let $v\in V$. We shall show that $f\left(  v\right)  =g\left(  v\right)
$.

The list $\left(  t_{1},t_{2},\ldots,t_{N}\right)  $ contains each element of
$V$ at least once (according to the hypotheses of Lemma \ref{lem.ncEqs.sum}).
In particular, this list contains $v$ at least once (since $v$ is an element
of $V$). In other words, there exists an $i\in\left\{  1,2,\ldots,N\right\}  $
such that $t_{i}=v$. Consider this $i$. From (\ref{pf.lem.ncEqs.sum.c2.pf.1}),
we obtain $f\left(  t_{i}\right)  =g\left(  t_{i}\right)  $. In view of
$t_{i}=v$, we can rewrite this as $f\left(  v\right)  =g\left(  v\right)  $.

Forget that we fixed $v$. We thus have shown that $f\left(  v\right)
=g\left(  v\right)  $ for each $v\in V$. In other words, $f=g$ (since both $f$
and $g$ are maps $V\rightarrow\mathbb{N}_{+}$).

Forget that we fixed $f$ and $g$. We thus have shown that if $f$ and $g$ are
two elements of $\mathcal{F}$ that satisfy $\Psi\left(  f\right)  =\Psi\left(
g\right)  $, then $f=g$. In other words, the map $\Psi$ is injective. This
proves Claim 2.]
\end{verlong}

\begin{statement}
\textit{Claim 3:} The map $\Psi$ is surjective.
\end{statement}

[\textit{Proof of Claim 3:} Let $\mathbf{i}\in\mathcal{I}$. We shall construct
an $f\in\mathcal{F}$ satisfying $\Psi\left(  f\right)  =\mathbf{i}$.

\begin{vershort}
Indeed, $\mathbf{i}\in\mathcal{I}$. By the definition of $\mathcal{I}$, this
means that $\mathbf{i}$ has the form $\mathbf{i}=\left(  i_{1},i_{2}%
,\ldots,i_{N}\right)  $ for some $N$-tuple $\left(  i_{1},i_{2},\ldots
,i_{N}\right)  \in\left(  \mathbb{N}_{+}\right)  ^{N}$ that satisfies
\begin{equation}
i_{a}=i_{b}\text{ whenever }t_{a}\sim_{\left(  V,B\right)  }t_{b}.
\label{pf.lem.ncEqs.sum.short.c3.pf.1}%
\end{equation}
Consider this $N$-tuple $\left(  i_{1},i_{2},\ldots,i_{N}\right)  $.
\end{vershort}

\begin{verlong}
Indeed, $\mathbf{i}\in\mathcal{I}=\left\{  \left(  i_{1},i_{2},\ldots
,i_{N}\right)  \in\left(  \mathbb{N}_{+}\right)  ^{N}\ \mid\ i_{a}=i_{b}\text{
whenever }a\sim_{\mathbf{P}\left(  V,B,\mathbf{t}\right)  }b\right\}  $ (by
the definition of $\mathcal{I}$). In other words, $\mathbf{i}$ has the form
$\mathbf{i}=\left(  i_{1},i_{2},\ldots,i_{N}\right)  $ for some $N$-tuple
$\left(  i_{1},i_{2},\ldots,i_{N}\right)  \in\left(  \mathbb{N}_{+}\right)
^{N}$ that satisfies $i_{a}=i_{b}$ whenever $a\sim_{\mathbf{P}\left(
V,B,\mathbf{t}\right)  }b$. Consider this $N$-tuple $\left(  i_{1}%
,i_{2},\ldots,i_{N}\right)  $.

We have $i_{a}=i_{b}$ whenever $a\sim_{\mathbf{P}\left(  V,B,\mathbf{t}%
\right)  }b$. In other words, we have
\begin{equation}
i_{a}=i_{b} \label{pf.lem.ncEqs.sum.c3.pf.1}%
\end{equation}
for any two elements $a,b\in\left\{  1,2,\ldots,N\right\}  $ that satisfy
$a\sim_{\mathbf{P}\left(  V,B,\mathbf{t}\right)  }b$ (since this is what
\textquotedblleft$i_{a}=i_{b}$ whenever $a\sim_{\mathbf{P}\left(
V,B,\mathbf{t}\right)  }b$\textquotedblright\ means).
\end{verlong}

We shall now define a map $f:V\rightarrow\mathbb{N}_{+}$ as follows:

Let $v\in V$. The list $\left(  t_{1},t_{2},\ldots,t_{N}\right)  $ contains
each element of $V$ at least once (according to the hypotheses of Lemma
\ref{lem.ncEqs.sum}). In particular, this list contains $v$ at least once
(since $v$ is an element of $V$). In other words, there exists a $k\in\left\{
1,2,\ldots,N\right\}  $ such that $t_{k}=v$. Pick the smallest such $k$, and
set $f\left(  v\right)  :=i_{k}$.

\begin{vershort}
Thus, we have defined a positive integer $f\left(  v\right)  \in\mathbb{N}%
_{+}$ for each $v\in V$. In other words, we have defined a map $f:V\rightarrow
\mathbb{N}_{+}$.
\end{vershort}

\begin{verlong}
Thus, we have defined a positive integer $f\left(  v\right)  \in\mathbb{N}%
_{+}$ for each $v\in V$. In other words, we have defined a map $f:V\rightarrow
\mathbb{N}_{+}$. According to its definition, this map can be computed as
follows: If $v\in V$ is any vertex, then%
\begin{equation}
f\left(  v\right)  =i_{k}, \label{pf.lem.ncEqs.sum.c3.pf.fv=}%
\end{equation}
where $k$ is the smallest element of $\left\{  1,2,\ldots,N\right\}  $ such
that $t_{k}=v$.
\end{verlong}

We shall now show that $B\subseteq\operatorname*{Eqs}f$.

\begin{vershort}
Indeed, let $e\in B$ be arbitrary. Then, $e\in B\subseteq\dbinom{V}{2}$, so
that $e=\left\{  x,y\right\}  $ for two distinct vertices $x,y\in V$. Consider
these $x,y$. Thus, $\left(  x,e,y\right)  $ is a walk from $x$ to $y$ in the
graph $\left(  V,B\right)  $. Therefore, $x$ is connected to $y$ in this
graph. In other words, $x\sim_{\left(  V,B\right)  }y$.

The definition of $f$ shows that $f\left(  x\right)  =i_{a}$, where $a$ is the
smallest element of $\left\{  1,2,\ldots,N\right\}  $ such that $t_{a}=x$.
Similarly, $f\left(  y\right)  =i_{b}$, where $b$ is the smallest element of
$\left\{  1,2,\ldots,N\right\}  $ such that $t_{b}=y$. Consider these $a$ and
$b$.

However, we have $x\sim_{\left(  V,B\right)  }y$. In other words, $t_{a}%
\sim_{\left(  V,B\right)  }t_{b}$ (since $t_{a}=x$ and $t_{b}=y$). Hence, from
(\ref{pf.lem.ncEqs.sum.short.c3.pf.1}), we obtain $i_{a}=i_{b}$. In other
words, $f\left(  x\right)  =f\left(  y\right)  $ (since $f\left(  x\right)
=i_{a}$ and $f\left(  y\right)  =i_{b}$). In other words, $\left\{
x,y\right\}  \in\operatorname*{Eqs}f$ (by the definition of
$\operatorname*{Eqs}f$). Hence, $e=\left\{  x,y\right\}  \in
\operatorname*{Eqs}f$. Now, forget that we fixed $e$. We thus have shown that
$e\in\operatorname*{Eqs}f$ for each $e\in B$. In other words, $B\subseteq
\operatorname*{Eqs}f$.
\end{vershort}

\begin{verlong}
Indeed, let $\sim$ denote the equivalence relation $\sim_{\left(  V,B\right)
}$. We shall first show that%
\begin{equation}
f\left(  x\right)  =f\left(  y\right)  \text{ for any }x\in V\text{ and }y\in
V\text{ satisfying }x\sim y. \label{pf.lem.ncEqs.sum.c3.pf.3}%
\end{equation}

[\textit{Proof of (\ref{pf.lem.ncEqs.sum.c3.pf.3}):} Let $x\in V$ and $y\in V$
be such that $x\sim y$. We must show that $f\left(  x\right)  =f\left(
y\right)  $.

We know that the list $\left(  t_{1},t_{2},\ldots,t_{N}\right)  $ contains
each element of $V$ at least once. In particular, this list contains $x$ at
least once (since $x$ is an element of $V$). In other words, there exists an
$a\in\left\{  1,2,\ldots,N\right\}  $ such that $t_{a}=x$. Pick the smallest
such $a$. Thus, $a$ is the smallest element of $\left\{  1,2,\ldots,N\right\}
$ such that $t_{a}=x$. Hence, $f\left(  x\right)  =i_{a}$ (by
(\ref{pf.lem.ncEqs.sum.c3.pf.fv=}), applied to $v=x$ and $k=a$).

We know that the list $\left(  t_{1},t_{2},\ldots,t_{N}\right)  $ contains
each element of $V$ at least once. In particular, this list contains $y$ at
least once (since $y$ is an element of $V$). In other words, there exists a
$b\in\left\{  1,2,\ldots,N\right\}  $ such that $t_{b}=y$. Pick the smallest
such $b$. Thus, $b$ is the smallest element of $\left\{  1,2,\ldots,N\right\}
$ such that $t_{b}=y$. Hence, $f\left(  y\right)  =i_{b}$ (by
(\ref{pf.lem.ncEqs.sum.c3.pf.fv=}), applied to $v=y$ and $k=b$).

We have $x\sim y$. In other words, $x\sim_{\left(  V,B\right)  }y$ (since
$\sim$ is the relation $\sim_{\left(  V,B\right)  }$). In other words,
$t_{a}\sim_{\left(  V,B\right)  }t_{b}$ (since $t_{a}=x$ and $t_{b}=y$).
However, from (\ref{pf.lem.ncEqs.sum.equivalence}), we know that the
statements \textquotedblleft$a\sim_{\mathbf{P}\left(  V,B,\mathbf{t}\right)
}b$\textquotedblright\ and \textquotedblleft$t_{a}\sim_{\left(  V,B\right)
}t_{b}$\textquotedblright\ are equivalent. Hence, we have $a\sim
_{\mathbf{P}\left(  V,B,\mathbf{t}\right)  }b$ (since we have $t_{a}%
\sim_{\left(  V,B\right)  }t_{b}$). Thus, from (\ref{pf.lem.ncEqs.sum.c3.pf.1}%
), we conclude that $i_{a}=i_{b}$. In view of $f\left(  x\right)  =i_{a}$ and
$f\left(  y\right)  =i_{b}$, we can rewrite this as $f\left(  x\right)
=f\left(  y\right)  $. Thus, (\ref{pf.lem.ncEqs.sum.c3.pf.3}) is proved.]

Now, let $Y=\mathbb{N}_{+}$. Thus, $f:V\rightarrow Y$ is a map (since
$f:V\rightarrow\mathbb{N}_{+}$ is a map). A set $Y_{\sim}^{V}$ is defined
(according to Definition \ref{def.relquot.maps} \textbf{(b)}). Its definition
yields that
\[
Y_{\sim}^{V}=\left\{  g\in Y^{V}\ \mid\ g\left(  x\right)  =g\left(  y\right)
\text{ for any }x\in V\text{ and }y\in V\text{ satisfying }x\sim y\right\}  .
\]

However, $f$ is a $g\in Y^{V}$ satisfying
\[
g\left(  x\right)  =g\left(  y\right)  \text{ for any }x\in V\text{ and }y\in
V\text{ satisfying }x\sim y
\]
(since (\ref{pf.lem.ncEqs.sum.c3.pf.3}) shows that $f\left(  x\right)
=f\left(  y\right)  $ for any $x\in V$ and $y\in V$ satisfying $x\sim y$). In
other words,%
\begin{align*}
f  &  \in\left\{  g\in Y^{V}\ \mid\ g\left(  x\right)  =g\left(  y\right)
\text{ for any }x\in V\text{ and }y\in V\text{ satisfying }x\sim y\right\} \\
&  =Y_{\sim}^{V}%
\end{align*}
(since $Y_{\sim}^{V}=\left\{  g\in Y^{V}\ \mid\ g\left(  x\right)  =g\left(
y\right)  \text{ for any }x\in V\text{ and }y\in V\text{ satisfying }x\sim
y\right\}  $).

However, Lemma \ref{lem.Eqs.sum-aux} yields that we have the following logical
equivalence of statements:%
\[
\left(  B\subseteq\operatorname*{Eqs}f\right)  \ \Longleftrightarrow\ \left(
f\in Y_{\sim}^{V}\right)  .
\]
Hence, we have $B\subseteq\operatorname*{Eqs}f$ (since we have $f\in Y_{\sim
}^{V}$).
\end{verlong}

\begin{vershort}
Thus, we know that $f$ is a map $V\rightarrow\mathbb{N}_{+}$ and satisfies
$B\subseteq\operatorname*{Eqs}f$. In other words, $f\in\mathcal{F}$ (by the
definition of $\mathcal{F}$).
\end{vershort}

\begin{verlong}
Thus, we know that $f$ is a map $V\rightarrow\mathbb{N}_{+}$ and satisfies
$B\subseteq\operatorname*{Eqs}f$. In other words, $f$ is a map $g:V\rightarrow
\mathbb{N}_{+}$ satisfying $B\subseteq\operatorname*{Eqs}g$. In other words,
\[
f\in\left\{  g:V\rightarrow\mathbb{N}_{+}\text{ is a map }\mid\text{
}B\subseteq\operatorname*{Eqs}g\right\}  =\mathcal{F}%
\]
(by the definition of $\mathcal{F}$).
\end{verlong}

We shall now show that $\Psi\left(  f\right)  =\mathbf{i}$.

Indeed, the definition of $\Psi$ yields $\Psi\left(  f\right)  =\left(
f\left(  t_{1}\right)  ,f\left(  t_{2}\right)  ,\ldots,f\left(  t_{N}\right)
\right)  $.

Now, let $j\in\left\{  1,2,\ldots,N\right\}  $. We shall show that $f\left(
t_{j}\right)  =i_{j}$.

\begin{vershort}
Indeed, the definition of $f$ shows that $f\left(  t_{j}\right)  =i_{k}$,
where $k$ is the smallest element of $\left\{  1,2,\ldots,N\right\}  $ such
that $t_{k}=t_{j}$. Consider this $k$. From $t_{k}=t_{j}$, we obtain
$t_{k}\sim_{\left(  V,B\right)  }t_{j}$ (since the relation $\sim_{\left(
V,B\right)  }$ is an equivalence relation). Hence,
(\ref{pf.lem.ncEqs.sum.short.c3.pf.1}) (applied to $a=k$ and $b=j$) yields
$i_{k}=i_{j}$. Thus, $f\left(  t_{j}\right)  =i_{k}=i_{j}$.
\end{vershort}

\begin{verlong}
Indeed, there exists a $k\in\left\{  1,2,\ldots,N\right\}  $ such that
$t_{k}=t_{j}$ (for example, $k=j$ qualifies). Pick the smallest such $k$.
Thus, $k$ is the smallest element of $\left\{  1,2,\ldots,N\right\}  $ such
that $t_{k}=t_{j}$. Hence, (\ref{pf.lem.ncEqs.sum.c3.pf.fv=}) (applied to
$v=t_{j}$) yields $f\left(  t_{j}\right)  =i_{k}$. However, the relation
$\sim_{\left(  V,B\right)  }$ is an equivalence relation, and thus is
reflexive. Hence, $t_{k}\sim_{\left(  V,B\right)  }t_{k}$. In other words,
$t_{k}\sim_{\left(  V,B\right)  }t_{j}$ (since $t_{k}=t_{j}$). However,
(\ref{pf.lem.ncEqs.sum.equivalence}) (applied to $a=k$ and $b=j$) shows that
we have the equivalence
\[
\left(  k\sim_{\mathbf{P}\left(  V,B,\mathbf{t}\right)  }j\right)
\ \Longleftrightarrow\ \left(  t_{k}\sim_{\left(  V,B\right)  }t_{j}\right)
.
\]
Thus, we have $k\sim_{\mathbf{P}\left(  V,B,\mathbf{t}\right)  }j$ (since we
have $t_{k}\sim_{\left(  V,B\right)  }t_{j}$). Hence,
(\ref{pf.lem.ncEqs.sum.c3.pf.1}) (applied to $a=k$ and $b=j$) yields
$i_{k}=i_{j}$. Hence, $f\left(  t_{j}\right)  =i_{k}=i_{j}$.
\end{verlong}

Forget that we fixed $j$. We thus have shown that $f\left(  t_{j}\right)
=i_{j}$ for each $j\in\left\{  1,2,\ldots,N\right\}  $. In other words,%
\[
\left(  f\left(  t_{1}\right)  ,f\left(  t_{2}\right)  ,\ldots,f\left(
t_{N}\right)  \right)  =\left(  i_{1},i_{2},\ldots,i_{N}\right)  .
\]
In view of $\Psi\left(  f\right)  =\left(  f\left(  t_{1}\right)  ,f\left(
t_{2}\right)  ,\ldots,f\left(  t_{N}\right)  \right)  $ and $\mathbf{i}%
=\left(  i_{1},i_{2},\ldots,i_{N}\right)  $, we can rewrite this as
$\Psi\left(  f\right)  =\mathbf{i}$. Hence,%
\[
\mathbf{i}=\Psi\left(  \underbrace{f}_{\in\mathcal{F}}\right)  \in\Psi\left(
\mathcal{F}\right)  .
\]

Forget that we fixed $\mathbf{i}$. We thus have shown that $\mathbf{i}\in
\Psi\left(  \mathcal{F}\right)  $ for each $\mathbf{i}\in\mathcal{I}$. In
other words, $\mathcal{I}\subseteq\Psi\left(  \mathcal{F}\right)  $. In other
words, the map $\Psi$ is surjective. This proves Claim 3.]

We now know that the map $\Psi$ is injective (by Claim 2) and surjective (by
Claim 3). In other words, this map $\Psi$ is bijective, i.e., is a bijection.

\begin{vershort}
In other words, the map%
\begin{align*}
\mathcal{F}  &  \rightarrow\mathcal{I},\\
f  &  \mapsto\left(  f\left(  t_{1}\right)  ,f\left(  t_{2}\right)
,\ldots,f\left(  t_{N}\right)  \right)
\end{align*}
is a bijection (since the map is $\Psi$). Now,
(\ref{pf.lem.ncEqs.sum.short.P=2}) becomes%
\begin{align*}
P_{\mathbf{P}\left(  V,B,\mathbf{t}\right)  }  &  =\underbrace{\sum
_{\substack{\left(  i_{1},i_{2},\ldots,i_{N}\right)  \in\left(  \mathbb{N}%
_{+}\right)  ^{N};\\i_{a}=i_{b}\text{ whenever }t_{a}\sim_{\left(  V,B\right)
}t_{b}}}}_{\substack{=\sum_{\left(  i_{1},i_{2},\ldots,i_{N}\right)
\in\mathcal{I}}\\\text{(by the definition of }\mathcal{I}\text{)}}}X_{i_{1}%
}X_{i_{2}}\cdots X_{i_{N}}\\
&  =\sum_{\left(  i_{1},i_{2},\ldots,i_{N}\right)  \in\mathcal{I}}X_{i_{1}%
}X_{i_{2}}\cdots X_{i_{N}}\\
&  =\underbrace{\sum_{f\in\mathcal{F}}}_{\substack{=\sum
_{\substack{f:V\rightarrow\mathbb{N}_{+};\\B\subseteq\operatorname*{Eqs}%
f}}\\\text{(by the definition of }\mathcal{F}\text{)}}}X_{f\left(
t_{1}\right)  }X_{f\left(  t_{2}\right)  }\cdots X_{f\left(  t_{N}\right)  }\\
&  \ \ \ \ \ \ \ \ \ \ \ \ \ \ \ \ \ \ \ \ \left(
\begin{array}
[c]{c}%
\text{here, we have substituted }\left(  f\left(  t_{1}\right)  ,f\left(
t_{2}\right)  ,\ldots,f\left(  t_{N}\right)  \right) \\
\text{for }\left(  i_{1},i_{2},\ldots,i_{N}\right)  \text{ in the sum, since
the}\\
\text{map }\mathcal{F}\rightarrow\mathcal{I},\ f\mapsto\left(  f\left(
t_{1}\right)  ,f\left(  t_{2}\right)  ,\ldots,f\left(  t_{N}\right)  \right)
\\
\text{is a bijection}%
\end{array}
\right) \\
&  =\sum_{\substack{f:V\rightarrow\mathbb{N}_{+};\\B\subseteq
\operatorname*{Eqs}f}}\underbrace{X_{f\left(  t_{1}\right)  }X_{f\left(
t_{2}\right)  }\cdots X_{f\left(  t_{N}\right)  }}_{\substack{=\mathbf{X}%
_{f,\mathbf{t}}\\\text{(since }\mathbf{X}_{f,\mathbf{t}}\text{ was defined to
be }X_{f\left(  t_{1}\right)  }X_{f\left(  t_{2}\right)  }\cdots X_{f\left(
t_{N}\right)  }\\\text{in Definition \ref{def.ncambichromsym} \textbf{(a)})}%
}}\\
&  =\sum_{\substack{f:V\rightarrow\mathbb{N}_{+};\\B\subseteq
\operatorname*{Eqs}f}}\mathbf{X}_{f,\mathbf{t}}.
\end{align*}

\end{vershort}

\begin{verlong}
However, the map $\Psi$ is the map%
\begin{align*}
\mathcal{F}  &  \rightarrow\mathcal{I},\\
f  &  \mapsto\left(  f\left(  t_{1}\right)  ,f\left(  t_{2}\right)
,\ldots,f\left(  t_{N}\right)  \right)
\end{align*}
(by its definition). Hence, the map%
\begin{align*}
\mathcal{F}  &  \rightarrow\mathcal{I},\\
f  &  \mapsto\left(  f\left(  t_{1}\right)  ,f\left(  t_{2}\right)
,\ldots,f\left(  t_{N}\right)  \right)
\end{align*}
is a bijection (since $\Psi$ is a bijection).

Now, recall that%
\[
\mathcal{F}=\left\{  g:V\rightarrow\mathbb{N}_{+}\text{ is a map }\mid\text{
}B\subseteq\operatorname*{Eqs}g\right\}  .
\]
Hence,%
\begin{equation}
\sum_{f\in\mathcal{F}}=\sum_{f\in\left\{  g:V\rightarrow\mathbb{N}_{+}\text{
is a map }\mid\text{ }B\subseteq\operatorname*{Eqs}g\right\}  }=\sum
_{\substack{f:V\rightarrow\mathbb{N}_{+};\\B\subseteq\operatorname*{Eqs}f}}
\label{pf.lem.ncEqs.sum.sumF=}%
\end{equation}
(an equality between summation signs).

Furthermore, recall that%
\[
\mathcal{I}=\left\{  \left(  i_{1},i_{2},\ldots,i_{N}\right)  \in\left(
\mathbb{N}_{+}\right)  ^{N}\ \mid\ i_{a}=i_{b}\text{ whenever }a\sim
_{\mathbf{P}\left(  V,B,\mathbf{t}\right)  }b\right\}  .
\]
Hence,
\[
\sum_{\left(  i_{1},i_{2},\ldots,i_{N}\right)  \in\mathcal{I}}=\sum
_{\substack{\left(  i_{1},i_{2},\ldots,i_{N}\right)  \in\left(  \mathbb{N}%
_{+}\right)  ^{N};\\i_{a}=i_{b}\text{ whenever }a\sim_{\mathbf{P}\left(
V,B,\mathbf{t}\right)  }b}}
\]
(an equality between summation signs). In light of this, we can rewrite
(\ref{pf.lem.ncEqs.sum.P=}) as follows:%
\begin{align*}
P_{\mathbf{P}\left(  V,B,\mathbf{t}\right)  }  &  =\sum_{\left(  i_{1}%
,i_{2},\ldots,i_{N}\right)  \in\mathcal{I}}X_{i_{1}}X_{i_{2}}\cdots X_{i_{N}%
}\\
&  =\underbrace{\sum_{f\in\mathcal{F}}}_{\substack{=\sum
_{\substack{f:V\rightarrow\mathbb{N}_{+};\\B\subseteq\operatorname*{Eqs}%
f}}\\\text{(by (\ref{pf.lem.ncEqs.sum.sumF=}))}}}X_{f\left(  t_{1}\right)
}X_{f\left(  t_{2}\right)  }\cdots X_{f\left(  t_{N}\right)  }\\
&  \ \ \ \ \ \ \ \ \ \ \ \ \ \ \ \ \ \ \ \ \left(
\begin{array}
[c]{c}%
\text{here, we have substituted }\left(  f\left(  t_{1}\right)  ,f\left(
t_{2}\right)  ,\ldots,f\left(  t_{N}\right)  \right) \\
\text{for }\left(  i_{1},i_{2},\ldots,i_{N}\right)  \text{ in the sum, since
the}\\
\text{map }\mathcal{F}\rightarrow\mathcal{I},\ f\mapsto\left(  f\left(
t_{1}\right)  ,f\left(  t_{2}\right)  ,\ldots,f\left(  t_{N}\right)  \right)
\\
\text{is a bijection}%
\end{array}
\right) \\
&  =\sum_{\substack{f:V\rightarrow\mathbb{N}_{+};\\B\subseteq
\operatorname*{Eqs}f}}\underbrace{X_{f\left(  t_{1}\right)  }X_{f\left(
t_{2}\right)  }\cdots X_{f\left(  t_{N}\right)  }}_{\substack{=\mathbf{X}%
_{f,\mathbf{t}}\\\text{(since }\mathbf{X}_{f,\mathbf{t}}\text{ was defined to
be }X_{f\left(  t_{1}\right)  }X_{f\left(  t_{2}\right)  }\cdots X_{f\left(
t_{N}\right)  }\\\text{in Definition \ref{def.ncambichromsym} \textbf{(a)})}%
}}\\
&  =\sum_{\substack{f:V\rightarrow\mathbb{N}_{+};\\B\subseteq
\operatorname*{Eqs}f}}\mathbf{X}_{f,\mathbf{t}}.
\end{align*}

\end{verlong}

This proves Lemma \ref{lem.ncEqs.sum}.
\end{proof}

\begin{vershort}
We can now prove Theorems \ref{thm.ncambichromsym.varis} and
\ref{thm.ncambichromsym.empty} and Corollaries \ref{cor.ncambichromsym.K-free}
and \ref{cor.ncambichromsym.NBC} by making straightforward changes to the
above proofs of Theorems \ref{thm.wambichromsym.varis} and
\ref{thm.wambichromsym.empty} and Corollaries \ref{cor.wambichromsym.K-free}
and \ref{cor.wambichromsym.NBC} (replacing, in particular, the use of Lemma
\ref{lem.wEqs.sum} by a use of Lemma \ref{lem.ncEqs.sum}). We leave the
details to the reader.
\end{vershort}

\begin{verlong}
We can now prove Theorems \ref{thm.ncambichromsym.varis} and
\ref{thm.ncambichromsym.empty} and Corollaries \ref{cor.ncambichromsym.K-free}
and \ref{cor.ncambichromsym.NBC} by making straightforward changes to the
above proofs of Theorems \ref{thm.wambichromsym.varis} and
\ref{thm.wambichromsym.empty} and Corollaries \ref{cor.wambichromsym.K-free}
and \ref{cor.wambichromsym.NBC} (replacing, in particular, the use of Lemma
\ref{lem.wEqs.sum} by a use of Lemma \ref{lem.ncEqs.sum}). Here are the details:

\begin{proof}
[Proof of Theorem \ref{thm.ncambichromsym.varis}.]We have%
\begin{equation}
Y_{G,\mathbf{t}}=\sum_{\substack{f:V\rightarrow\mathbb{N}_{+}\text{ is
a}\\\text{proper }\mathbb{N}_{+}\text{-coloring of }G}}\mathbf{X}%
_{f,\mathbf{t}} \label{pf.thm.ncambichromsym.varis.XG-def}%
\end{equation}
(by the definition of $Y_{G,\mathbf{t}}$). Now, if $f:V\rightarrow
\mathbb{N}_{+}$ is a map, then we have the following logical equivalence:%
\begin{equation}
\left(  \text{the }\mathbb{N}_{+}\text{-coloring }f\text{ of }G\text{ is
proper}\right)  \ \Longleftrightarrow\ \left(  \operatorname*{EQS}\left(
G,f\right)  =\varnothing\right)  \label{pf.thm.ncambichromsym.varis.equiv}%
\end{equation}
(because the $\mathbb{N}_{+}$-coloring $f$ of $G$ is proper if and only if
$\operatorname*{EQS}\left(  G,f\right)  =\varnothing$\ \ \ \ \footnote{by
Lemma \ref{lem.ambiEqs.proper} (applied to $\mathbb{N}_{+}$ instead of $X$)}).
Now,%
\begin{align}
&  \sum_{f:V\rightarrow\mathbb{N}_{+}}\left[  \underbrace{\operatorname*{EQS}%
\left(  G,f\right)  =\varnothing}_{\substack{\Longleftrightarrow\ \left(
\text{the }\mathbb{N}_{+}\text{-coloring }f\text{ of }G\text{ is
proper}\right)  \\\text{(by (\ref{pf.thm.ncambichromsym.varis.equiv}))}%
}}\right]  \mathbf{X}_{f,\mathbf{t}}\nonumber\\
&  =\sum_{f:V\rightarrow\mathbb{N}_{+}}\left[  \underbrace{\text{the
}\mathbb{N}_{+}\text{-coloring }f\text{ of }G\text{ is proper}}%
_{\Longleftrightarrow\ \left(  f\text{ is a proper }\mathbb{N}_{+}%
\text{-coloring of }G\right)  }\right]  \mathbf{X}_{f,\mathbf{t}}\nonumber\\
&  =\sum_{f:V\rightarrow\mathbb{N}_{+}}\left[  f\text{ is a proper }%
\mathbb{N}_{+}\text{-coloring of }G\right]  \mathbf{X}_{f,\mathbf{t}%
}\nonumber\\
&  =\sum_{\substack{f:V\rightarrow\mathbb{N}_{+}\text{ is a}\\\text{proper
}\mathbb{N}_{+}\text{-coloring of }G}}\underbrace{\left[  f\text{ is a proper
}\mathbb{N}_{+}\text{-coloring of }G\right]  }_{\substack{=1\\\text{(since
}f\text{ is a proper }\mathbb{N}_{+}\text{-coloring of }G\text{)}}%
}\mathbf{X}_{f,\mathbf{t}}\nonumber\\
&  \ \ \ \ \ \ \ \ \ \ +\sum_{\substack{f:V\rightarrow\mathbb{N}_{+}\text{ is
not a}\\\text{proper }\mathbb{N}_{+}\text{-coloring of }G}}\underbrace{\left[
f\text{ is a proper }\mathbb{N}_{+}\text{-coloring of }G\right]
}_{\substack{=0\\\text{(since }f\text{ is not a proper }\mathbb{N}%
_{+}\text{-coloring of }G\text{)}}}\mathbf{X}_{f,\mathbf{t}}\nonumber\\
&  \ \ \ \ \ \ \ \ \ \ \ \ \ \ \ \ \ \ \ \ \left(  \text{since each
}f:V\rightarrow\mathbb{N}_{+}\text{ either is a proper }\mathbb{N}%
_{+}\text{-coloring of }G\text{ or not}\right) \nonumber\\
&  =\sum_{\substack{f:V\rightarrow\mathbb{N}_{+}\text{ is a}\\\text{proper
}\mathbb{N}_{+}\text{-coloring of }G}}\mathbf{X}_{f,\mathbf{t}}%
+\underbrace{\sum_{\substack{f:V\rightarrow\mathbb{N}_{+}\text{ is not
a}\\\text{proper }\mathbb{N}_{+}\text{-coloring of }G}}0\mathbf{X}%
_{f,\mathbf{t}}}_{=0}=\sum_{\substack{f:V\rightarrow\mathbb{N}_{+}\text{ is
a}\\\text{proper }\mathbb{N}_{+}\text{-coloring of }G}}\mathbf{X}%
_{f,\mathbf{t}}\nonumber\\
&  =Y_{G,\mathbf{t}} \label{pf.thm.ncambichromsym.varis.step1}%
\end{align}
(by (\ref{pf.thm.ncambichromsym.varis.XG-def})).

However, for every $f:V\rightarrow\mathbb{N}_{+}$, we have%
\begin{equation}
\sum_{\substack{B\subseteq E;\\\operatorname*{union}B\subseteq
\operatorname*{Eqs}f}}\left(  -1\right)  ^{\left\vert B\right\vert }%
\prod_{\substack{K\in\mathfrak{K};\\K\subseteq B}}a_{K}=\left[
\operatorname*{EQS}\left(  G,f\right)  =\varnothing\right]
\label{pf.thm.ncambichromsym.varis.moeb2}%
\end{equation}
(indeed, we have already shown this in the above proof of Theorem
\ref{thm.ambichromsym.varis}).

Now, (\ref{pf.thm.ncambichromsym.varis.step1}) yields%
\begin{align}
Y_{G,\mathbf{t}}  &  =\sum_{f:V\rightarrow\mathbb{N}_{+}}\underbrace{\left[
\operatorname*{EQS}\left(  G,f\right)  =\varnothing\right]  }_{\substack{=\sum
_{\substack{B\subseteq E;\\\operatorname*{union}B\subseteq\operatorname*{Eqs}%
f}}\left(  -1\right)  ^{\left\vert B\right\vert }\prod_{\substack{K\in
\mathfrak{K};\\K\subseteq B}}a_{K}\\\text{(by
(\ref{pf.thm.ncambichromsym.varis.moeb2}))}}}\mathbf{X}_{f,\mathbf{t}%
}\nonumber\\
&  =\sum_{f:V\rightarrow\mathbb{N}_{+}}\left(  \sum_{\substack{B\subseteq
E;\\\operatorname*{union}B\subseteq\operatorname*{Eqs}f}}\left(  -1\right)
^{\left\vert B\right\vert }\prod_{\substack{K\in\mathfrak{K};\\K\subseteq
B}}a_{K}\right)  \mathbf{X}_{f,\mathbf{t}}\nonumber\\
&  =\underbrace{\sum_{f:V\rightarrow\mathbb{N}_{+}}\ \ \sum
_{\substack{B\subseteq E;\\\operatorname*{union}B\subseteq\operatorname*{Eqs}%
f}}}_{=\sum_{B\subseteq E}\ \ \sum_{\substack{f:V\rightarrow\mathbb{N}%
_{+};\\\operatorname*{union}B\subseteq\operatorname*{Eqs}f}}}\left(
-1\right)  ^{\left\vert B\right\vert }\left(  \prod_{\substack{K\in
\mathfrak{K};\\K\subseteq B}}a_{K}\right)  \mathbf{X}_{f,\mathbf{t}%
}\nonumber\\
&  =\sum_{B\subseteq E}\ \ \sum_{\substack{f:V\rightarrow\mathbb{N}%
_{+};\\\operatorname*{union}B\subseteq\operatorname*{Eqs}f}}\left(  -1\right)
^{\left\vert B\right\vert }\left(  \prod_{\substack{K\in\mathfrak{K}%
;\\K\subseteq B}}a_{K}\right)  \mathbf{X}_{f,\mathbf{t}}\nonumber\\
&  =\sum_{B\subseteq E}\left(  -1\right)  ^{\left\vert B\right\vert }\left(
\prod_{\substack{K\in\mathfrak{K};\\K\subseteq B}}a_{K}\right)  \sum
_{\substack{f:V\rightarrow\mathbb{N}_{+};\\\operatorname*{union}%
B\subseteq\operatorname*{Eqs}f}}\mathbf{X}_{f,\mathbf{t}}.
\label{pf.thm.ncambichromsym.varis.step4}%
\end{align}

However, if $B$ is a subset of $E$, then the pair $\left(
V,\operatorname*{union}B\right)  $ is a finite graph (since $V$ is a finite
set and since $\operatorname*{union}B\subseteq\dbinom{V}{2}$), and thus we
have%
\begin{equation}
\sum_{\substack{f:V\rightarrow\mathbb{N}_{+};\\\operatorname*{union}%
B\subseteq\operatorname*{Eqs}f}}\mathbf{X}_{f,\mathbf{t}}=P_{\mathbf{P}\left(
V,\operatorname*{union}B,\mathbf{t}\right)  }
\label{pf.thm.ncambichromsym.varis.p}%
\end{equation}
(by Lemma \ref{lem.ncEqs.sum}, applied to $\operatorname*{union}B$ instead of
$B$).

Hence, (\ref{pf.thm.ncambichromsym.varis.step4}) becomes%
\begin{align*}
Y_{G,\mathbf{t}}  &  =\sum_{B\subseteq E}\left(  -1\right)  ^{\left\vert
B\right\vert }\left(  \prod_{\substack{K\in\mathfrak{K};\\K\subseteq B}%
}a_{K}\right)  \underbrace{\sum_{\substack{f:V\rightarrow\mathbb{N}%
_{+};\\\operatorname*{union}B\subseteq\operatorname*{Eqs}f}}\mathbf{X}%
_{f,\mathbf{t}}}_{\substack{=P_{\mathbf{P}\left(  V,\operatorname*{union}%
B,\mathbf{t}\right)  }\\\text{(by (\ref{pf.thm.ncambichromsym.varis.p}))}}}\\
&  =\sum_{B\subseteq E}\left(  -1\right)  ^{\left\vert B\right\vert }\left(
\prod_{\substack{K\in\mathfrak{K};\\K\subseteq B}}a_{K}\right)  P_{\mathbf{P}%
\left(  V,\operatorname*{union}B,\mathbf{t}\right)  }\\
&  =\sum_{F\subseteq E}\left(  -1\right)  ^{\left\vert F\right\vert }\left(
\prod_{\substack{K\in\mathfrak{K};\\K\subseteq F}}a_{K}\right)  P_{\mathbf{P}%
\left(  V,\operatorname*{union}F,\mathbf{t}\right)  }%
\end{align*}
(here, we have renamed the summation index $B$ as $F$). This proves Theorem
\ref{thm.ncambichromsym.varis}.
\end{proof}

\begin{proof}
[Proof of Corollary \ref{cor.ncambichromsym.K-free}.]We can apply Theorem
\ref{thm.ncambichromsym.varis} to $0$ instead of $a_{K}$. As a result, we
obtain%
\begin{equation}
Y_{G,\mathbf{t}}=\sum_{F\subseteq E}\left(  -1\right)  ^{\left\vert
F\right\vert }\left(  \prod_{\substack{K\in\mathfrak{K};\\K\subseteq
F}}0\right)  P_{\mathbf{P}\left(  V,\operatorname*{union}F,\mathbf{t}\right)
}. \label{pf.cor.ncambichromsym.K-free.0}%
\end{equation}
Now, if $F$ is any subset of $E$, then%
\begin{equation}
\prod_{\substack{K\in\mathfrak{K};\\K\subseteq F}}0=%
\begin{cases}
1, & \text{if }F\text{ is }\mathfrak{K}\text{-free;}\\
0, & \text{if }F\text{ is not }\mathfrak{K}\text{-free}%
\end{cases}
\label{pf.cor.ncambichromsym.K-free.1}%
\end{equation}
(indeed, this was already shown in our above proof of Corollary
\ref{cor.ambichromsym.K-free}).

Thus, (\ref{pf.cor.ncambichromsym.K-free.0}) becomes%
\begin{align*}
Y_{G,\mathbf{t}}  &  =\sum_{F\subseteq E}\left(  -1\right)  ^{\left\vert
F\right\vert }\underbrace{\left(  \prod_{\substack{K\in\mathfrak{K}%
;\\K\subseteq F}}0\right)  }_{\substack{=%
\begin{cases}
1, & \text{if }F\text{ is }\mathfrak{K}\text{-free;}\\
0, & \text{if }F\text{ is not }\mathfrak{K}\text{-free}%
\end{cases}
\\\text{(by (\ref{pf.cor.ncambichromsym.K-free.1}))}}}P_{\mathbf{P}\left(
V,\operatorname*{union}F,\mathbf{t}\right)  }\\
&  =\sum_{F\subseteq E}\left(  -1\right)  ^{\left\vert F\right\vert }%
\begin{cases}
1, & \text{if }F\text{ is }\mathfrak{K}\text{-free;}\\
0, & \text{if }F\text{ is not }\mathfrak{K}\text{-free}%
\end{cases}
\ \ P_{\mathbf{P}\left(  V,\operatorname*{union}F,\mathbf{t}\right)  }\\
&  =\sum_{\substack{F\subseteq E;\\F\text{ is }\mathfrak{K}\text{-free}%
}}\left(  -1\right)  ^{\left\vert F\right\vert }\underbrace{%
\begin{cases}
1, & \text{if }F\text{ is }\mathfrak{K}\text{-free;}\\
0, & \text{if }F\text{ is not }\mathfrak{K}\text{-free}%
\end{cases}
}_{\substack{=1\\\text{(since }F\text{ is }\mathfrak{K}\text{-free)}%
}}\ \ P_{\mathbf{P}\left(  V,\operatorname*{union}F,\mathbf{t}\right)  }\\
&  \ \ \ \ \ \ \ \ \ \ +\sum_{\substack{F\subseteq E;\\F\text{ is not
}\mathfrak{K}\text{-free}}}\left(  -1\right)  ^{\left\vert F\right\vert
}\underbrace{%
\begin{cases}
1, & \text{if }F\text{ is }\mathfrak{K}\text{-free;}\\
0, & \text{if }F\text{ is not }\mathfrak{K}\text{-free}%
\end{cases}
}_{\substack{=0\\\text{(since }F\text{ is not }\mathfrak{K}\text{-free)}%
}}\ \ P_{\mathbf{P}\left(  V,\operatorname*{union}F,\mathbf{t}\right)  }\\
&  \ \ \ \ \ \ \ \ \ \ \ \ \ \ \ \ \ \ \ \ \left(  \text{since each subset
}F\text{ of }E\text{ either is }\mathfrak{K}\text{-free or is not}\right) \\
&  =\sum_{\substack{F\subseteq E;\\F\text{ is }\mathfrak{K}\text{-free}%
}}\left(  -1\right)  ^{\left\vert F\right\vert }P_{\mathbf{P}\left(
V,\operatorname*{union}F,\mathbf{t}\right)  }+\underbrace{\sum
_{\substack{F\subseteq E;\\F\text{ is not }\mathfrak{K}\text{-free}}}\left(
-1\right)  ^{\left\vert F\right\vert }0P_{\mathbf{P}\left(
V,\operatorname*{union}F,\mathbf{t}\right)  }}_{=0}\\
&  =\sum_{\substack{F\subseteq E;\\F\text{ is }\mathfrak{K}\text{-free}%
}}\left(  -1\right)  ^{\left\vert F\right\vert }P_{\mathbf{P}\left(
V,\operatorname*{union}F,\mathbf{t}\right)  }.
\end{align*}
This proves Corollary \ref{cor.ncambichromsym.K-free}.
\end{proof}

\begin{proof}
[Proof of Corollary \ref{cor.ncambichromsym.NBC}.]Let $\mathfrak{K}$ be the
set of all broken circuits of $G$.

Now, just as in the proof of Corollary \ref{cor.chromsym.NBC}, we can prove
the following equality:
\[
\sum_{\substack{F\subseteq E;\\F\text{ is }\mathfrak{K}\text{-free}}%
}=\sum_{\substack{F\subseteq E;\\F\text{ contains no broken}\\\text{circuit of
}G\text{ as a subset}}}
\]
(an equality between summation signs). Now, Corollary
\ref{cor.ncambichromsym.K-free} yields%
\begin{align*}
Y_{G,\mathbf{t}}  &  =\underbrace{\sum_{\substack{F\subseteq E;\\F\text{ is
}\mathfrak{K}\text{-free}}}}_{=\sum_{\substack{F\subseteq E;\\F\text{ contains
no broken}\\\text{circuit of }G\text{ as a subset}}}}\left(  -1\right)
^{\left\vert F\right\vert }P_{\mathbf{P}\left(  V,\operatorname*{union}%
F,\mathbf{t}\right)  }\\
&  =\sum_{\substack{F\subseteq E;\\F\text{ contains no broken}\\\text{circuit
of }G\text{ as a subset}}}\left(  -1\right)  ^{\left\vert F\right\vert
}P_{\mathbf{P}\left(  V,\operatorname*{union}F,\mathbf{t}\right)  }.
\end{align*}
This proves Corollary \ref{cor.ncambichromsym.NBC}.
\end{proof}

\begin{proof}
[Proof of Theorem \ref{thm.ncambichromsym.empty}.]Let $X$ be the totally
ordered set $\left\{  1\right\}  $ (equipped with the only possible order on
this set). Let $\ell:E\rightarrow X$ be the function sending each $e\in E$ to
$1\in X$. Let $\mathfrak{K}$ be the empty set. Clearly, $\mathfrak{K}$ is a
set of broken circuits of $G$. Theorem \ref{thm.ncambichromsym.varis} (applied
to $0$ instead of $a_{K}$) yields%
\begin{align*}
Y_{G,\mathbf{t}}  &  =\sum_{F\subseteq E}\left(  -1\right)  ^{\left\vert
F\right\vert }\underbrace{\left(  \prod_{\substack{K\in\mathfrak{K}%
;\\K\subseteq F}}0\right)  }_{\substack{=\left(  \text{empty product}\right)
\\\text{(since }\mathfrak{K}\text{ is the empty set)}}}P_{\mathbf{P}\left(
V,\operatorname*{union}F,\mathbf{t}\right)  }\\
&  =\sum_{F\subseteq E}\left(  -1\right)  ^{\left\vert F\right\vert
}\underbrace{\left(  \text{empty product}\right)  }_{=1}P_{\mathbf{P}\left(
V,\operatorname*{union}F,\mathbf{t}\right)  }=\sum_{F\subseteq E}\left(
-1\right)  ^{\left\vert F\right\vert }P_{\mathbf{P}\left(
V,\operatorname*{union}F,\mathbf{t}\right)  }.
\end{align*}
This proves Theorem \ref{thm.ncambichromsym.empty}.
\end{proof}
\end{verlong}

\subsection{An abstract setting}

The reader will by now have realized that we have been making the same
arguments in a series of slightly different settings. In particular, the
chromatic symmetric function $X_{G}$, its weighted version $X_{G,w}$ and its
noncommutative version $Y_{G,\mathbf{t}}$ are all defined as sums over proper
$\mathbb{N}_{+}$-colorings of $G$; they differ only in the addends being
summed. We can generalize them all by allowing these addends to be arbitrary,
i.e., replacing them by arbitrary elements $\alpha_{f}$ of a $\mathbf{k}%
$-module $M$, provided that the resulting (potentially infinite) sums are
still well-defined. While we are at it, we can also replace $\mathbb{N}_{+}%
$-colorings by $Y$-colorings for an arbitrary set $Y$. Thus, we are led to the
following general setting:

\begin{definition}
\label{def.genambichrom}Let $G=\left(  V,E,\varphi\right)  $ be a finite
ambigraph. Let $Y$ be any set.

Let $M$ be a topological $\mathbf{k}$-module. Let $\alpha_{f}\in M$ be an
element for each $Y$-coloring $f:V\rightarrow Y$. Assume that the family
$\left(  \alpha_{f}\right)  _{f:V\rightarrow Y}$ of these elements is summable
(so that the sum $\sum_{f:V\rightarrow Y}\alpha_{f}$ and any of its subsums is well-defined).

Then:

\textbf{(a)} We define an element%
\[
\Xi_{G}:=\sum_{\substack{f:V\rightarrow Y\text{ is a}\\\text{proper
}Y\text{-coloring of }G}}\alpha_{f}\in M.
\]

\textbf{(b)} Furthermore, if $B$ is a subset of $\dbinom{V}{2}$, then we set%
\[
\pi_{B}:=\sum_{\substack{f:V\rightarrow Y;\\B\subseteq\operatorname*{Eqs}%
f}}\alpha_{f}\in M.
\]
(This does not actually depend on the ambigraph $G$, but only depends on the
set $V$.)
\end{definition}

Through appropriate choices of $\alpha_{f}$, we recover the previously defined
power series $X_{G}$, $X_{G,w}$ and $Y_{G,\mathbf{t}}$:

\begin{itemize}
\item If $Y=\mathbb{N}_{+}$ and $\alpha_{f}=\mathbf{x}_{f}$, then $\Xi
_{G}=X_{G}$ and $\pi_{B}=p_{\lambda\left(  V,B\right)  }$.

\item If $Y=\mathbb{N}_{+}$ and $\alpha_{f}=\mathbf{x}_{f,w}$ (for a given
weight function $w:V\rightarrow\mathbb{N}_{+}$), then $\Xi_{G}=X_{G,w}$ and
$\pi_{B}=p_{\lambda\left(  \left(  V,B\right)  ,w\right)  }$.

\item If $Y=\mathbb{N}_{+}$ and $\alpha_{f}=\mathbf{X}_{f,\mathbf{t}}$ (for a
given list $\mathbf{t}$ of elements of $V$ that contains each element at least
once), then $\Xi_{G}=Y_{G,\mathbf{t}}$ and $\pi_{B}=P_{\mathbf{P}\left(
V,B,\mathbf{t}\right)  }$.
\end{itemize}

We can now state analogues of Theorems \ref{thm.ncambichromsym.empty} and
\ref{thm.ncambichromsym.varis} and Corollaries \ref{cor.ncambichromsym.K-free}
and \ref{cor.ncambichromsym.NBC} in this general context:

\begin{theorem}
\label{thm.genambichromsym.empty}Let $G$, $V$, $E$, $\varphi$, $Y$, $M$ and
$\alpha_{f}$ be as in Definition \ref{def.genambichrom}. Then, using the
notations of Definition \ref{def.genambichrom}, we have%
\[
\Xi_{G}=\sum_{F\subseteq E}\left(  -1\right)  ^{\left\vert F\right\vert }%
\pi_{\operatorname*{union}F}.
\]

\end{theorem}

\begin{theorem}
\label{thm.genambichromsym.varis}Let $G$, $V$, $E$, $\varphi$, $Y$, $M$ and
$\alpha_{f}$ be as in Definition \ref{def.genambichrom}. Let $X$ be a totally
ordered set. Let $\ell:E\rightarrow X$ be a labeling function. Let
$\mathfrak{K}$ be some set of broken circuits of $G$ (not necessarily
containing all of them). Let $a_{K}$ be an element of $\mathbf{k}$ for every
$K\in\mathfrak{K}$. Then, using the notations of Definition
\ref{def.genambichrom}, we have%
\[
\Xi_{G}=\sum_{F\subseteq E}\left(  -1\right)  ^{\left\vert F\right\vert
}\left(  \prod_{\substack{K\in\mathfrak{K};\\K\subseteq F}}a_{K}\right)
\pi_{\operatorname*{union}F}.
\]

\end{theorem}

\begin{corollary}
\label{cor.genambichromsym.K-free}Let $G$, $V$, $E$, $\varphi$, $Y$, $M$ and
$\alpha_{f}$ be as in Definition \ref{def.genambichrom}. Let $X$ be a totally
ordered set. Let $\ell:E\rightarrow X$ be a labeling function. Let
$\mathfrak{K}$ be some set of broken circuits of $G$ (not necessarily
containing all of them). Then, using the notations of Definition
\ref{def.genambichrom}, we have%
\[
\Xi_{G}=\sum_{\substack{F\subseteq E;\\F\text{ is }\mathfrak{K}\text{-free}%
}}\left(  -1\right)  ^{\left\vert F\right\vert }\pi_{\operatorname*{union}F}.
\]

\end{corollary}

\begin{corollary}
\label{cor.genambichromsym.NBC}Let $G$, $V$, $E$, $\varphi$, $Y$, $M$ and
$\alpha_{f}$ be as in Definition \ref{def.genambichrom}. Let $X$ be a totally
ordered set. Let $\ell:E\rightarrow X$ be a labeling function. Then, using the
notations of Definition \ref{def.genambichrom}, we have%
\[
\Xi_{G}=\sum_{\substack{F\subseteq E;\\F\text{ contains no broken}%
\\\text{circuit of }G\text{ as a subset}}}\left(  -1\right)  ^{\left\vert
F\right\vert }\pi_{\operatorname*{union}F}.
\]

\end{corollary}

\begin{vershort}
The reader will have no difficulty proving these four results by following the
same well-trodden path that led us to their particular cases.
\end{vershort}

\begin{verlong}
These four results can be proved through arguments very similar to the ones we
used in earlier proofs (e.g., in the proofs of Theorems
\ref{thm.wambichromsym.varis} and \ref{thm.wambichromsym.empty} and
Corollaries \ref{cor.wambichromsym.K-free} and \ref{cor.wambichromsym.NBC}).
Here are the details:

\begin{proof}
[Proof of Theorem \ref{thm.genambichromsym.varis}.]We have%
\begin{equation}
\Xi_{G}=\sum_{\substack{f:V\rightarrow Y\text{ is a}\\\text{proper
}Y\text{-coloring of }G}}\alpha_{f}
\label{pf.thm.genambichromsym.varis.XG-def}%
\end{equation}
(by the definition of $\Xi_{G}$). Now, if $f:V\rightarrow Y$ is a map, then we
have the following logical equivalence:%
\begin{equation}
\left(  \text{the }Y\text{-coloring }f\text{ of }G\text{ is proper}\right)
\ \Longleftrightarrow\ \left(  \operatorname*{EQS}\left(  G,f\right)
=\varnothing\right)  \label{pf.thm.genambichromsym.varis.equiv}%
\end{equation}
(because the $Y$-coloring $f$ of $G$ is proper if and only if
$\operatorname*{EQS}\left(  G,f\right)  =\varnothing$\ \ \ \ \footnote{by
Lemma \ref{lem.ambiEqs.proper} (applied to $Y$ instead of $X$)}). Now,%
\begin{align}
&  \sum_{f:V\rightarrow Y}\left[  \underbrace{\operatorname*{EQS}\left(
G,f\right)  =\varnothing}_{\substack{\Longleftrightarrow\ \left(  \text{the
}Y\text{-coloring }f\text{ of }G\text{ is proper}\right)  \\\text{(by
(\ref{pf.thm.genambichromsym.varis.equiv}))}}}\right]  \alpha_{f}\nonumber\\
&  =\sum_{f:V\rightarrow Y}\left[  \underbrace{\text{the }Y\text{-coloring
}f\text{ of }G\text{ is proper}}_{\Longleftrightarrow\ \left(  f\text{ is a
proper }Y\text{-coloring of }G\right)  }\right]  \alpha_{f}\nonumber\\
&  =\sum_{f:V\rightarrow Y}\left[  f\text{ is a proper }Y\text{-coloring of
}G\right]  \alpha_{f}\nonumber\\
&  =\sum_{\substack{f:V\rightarrow Y\text{ is a}\\\text{proper }%
Y\text{-coloring of }G}}\underbrace{\left[  f\text{ is a proper }%
Y\text{-coloring of }G\right]  }_{\substack{=1\\\text{(since }f\text{ is a
proper }Y\text{-coloring of }G\text{)}}}\alpha_{f}\nonumber\\
&  \ \ \ \ \ \ \ \ \ \ +\sum_{\substack{f:V\rightarrow Y\text{ is not
a}\\\text{proper }Y\text{-coloring of }G}}\underbrace{\left[  f\text{ is a
proper }Y\text{-coloring of }G\right]  }_{\substack{=0\\\text{(since }f\text{
is not a proper }Y\text{-coloring of }G\text{)}}}\alpha_{f}\nonumber\\
&  \ \ \ \ \ \ \ \ \ \ \ \ \ \ \ \ \ \ \ \ \left(  \text{since each
}f:V\rightarrow Y\text{ either is a proper }Y\text{-coloring of }G\text{ or
not}\right) \nonumber\\
&  =\sum_{\substack{f:V\rightarrow Y\text{ is a}\\\text{proper }%
Y\text{-coloring of }G}}\alpha_{f}+\underbrace{\sum_{\substack{f:V\rightarrow
Y\text{ is not a}\\\text{proper }Y\text{-coloring of }G}}0\alpha_{f}}%
_{=0}=\sum_{\substack{f:V\rightarrow Y\text{ is a}\\\text{proper
}Y\text{-coloring of }G}}\alpha_{f}\nonumber\\
&  =\Xi_{G} \label{pf.thm.genambichromsym.varis.step1}%
\end{align}
(by (\ref{pf.thm.genambichromsym.varis.XG-def})).

However, for every $f:V\rightarrow Y$, we have%
\begin{equation}
\sum_{\substack{B\subseteq E;\\\operatorname*{union}B\subseteq
\operatorname*{Eqs}f}}\left(  -1\right)  ^{\left\vert B\right\vert }%
\prod_{\substack{K\in\mathfrak{K};\\K\subseteq B}}a_{K}=\left[
\operatorname*{EQS}\left(  G,f\right)  =\varnothing\right]
\label{pf.thm.genambichromsym.varis.moeb2}%
\end{equation}
(indeed, we have already shown this in the above proof of Theorem
\ref{thm.ambichromsym.varis}).

Now, (\ref{pf.thm.genambichromsym.varis.step1}) yields%
\begin{align*}
\Xi_{G}  &  =\sum_{f:V\rightarrow Y}\underbrace{\left[  \operatorname*{EQS}%
\left(  G,f\right)  =\varnothing\right]  }_{\substack{=\sum
_{\substack{B\subseteq E;\\\operatorname*{union}B\subseteq\operatorname*{Eqs}%
f}}\left(  -1\right)  ^{\left\vert B\right\vert }\prod_{\substack{K\in
\mathfrak{K};\\K\subseteq B}}a_{K}\\\text{(by
(\ref{pf.thm.genambichromsym.varis.moeb2}))}}}\alpha_{f}\\
&  =\sum_{f:V\rightarrow Y}\left(  \sum_{\substack{B\subseteq
E;\\\operatorname*{union}B\subseteq\operatorname*{Eqs}f}}\left(  -1\right)
^{\left\vert B\right\vert }\prod_{\substack{K\in\mathfrak{K};\\K\subseteq
B}}a_{K}\right)  \alpha_{f}\\
&  =\underbrace{\sum_{f:V\rightarrow Y}\ \ \sum_{\substack{B\subseteq
E;\\\operatorname*{union}B\subseteq\operatorname*{Eqs}f}}}_{=\sum_{B\subseteq
E}\ \ \sum_{\substack{f:V\rightarrow Y;\\\operatorname*{union}B\subseteq
\operatorname*{Eqs}f}}}\left(  -1\right)  ^{\left\vert B\right\vert }\left(
\prod_{\substack{K\in\mathfrak{K};\\K\subseteq B}}a_{K}\right)  \alpha_{f}\\
&  =\sum_{B\subseteq E}\ \ \sum_{\substack{f:V\rightarrow
Y;\\\operatorname*{union}B\subseteq\operatorname*{Eqs}f}}\left(  -1\right)
^{\left\vert B\right\vert }\left(  \prod_{\substack{K\in\mathfrak{K}%
;\\K\subseteq B}}a_{K}\right)  \alpha_{f}\\
&  =\sum_{B\subseteq E}\left(  -1\right)  ^{\left\vert B\right\vert }\left(
\prod_{\substack{K\in\mathfrak{K};\\K\subseteq B}}a_{K}\right)
\underbrace{\sum_{\substack{f:V\rightarrow Y;\\\operatorname*{union}%
B\subseteq\operatorname*{Eqs}f}}\alpha_{f}}_{\substack{=\pi
_{\operatorname*{union}B}\\\text{(since }\pi_{\operatorname*{union}B}\text{
was defined}\\\text{to be }\sum_{\substack{f:V\rightarrow
Y;\\\operatorname*{union}B\subseteq\operatorname*{Eqs}f}}\alpha_{f}\text{)}%
}}\\
&  =\sum_{B\subseteq E}\left(  -1\right)  ^{\left\vert B\right\vert }\left(
\prod_{\substack{K\in\mathfrak{K};\\K\subseteq B}}a_{K}\right)  \pi
_{\operatorname*{union}B}=\sum_{F\subseteq E}\left(  -1\right)  ^{\left\vert
F\right\vert }\left(  \prod_{\substack{K\in\mathfrak{K};\\K\subseteq F}%
}a_{K}\right)  \pi_{\operatorname*{union}F}%
\end{align*}
(here, we have renamed the summation index $B$ as $F$). This proves Theorem
\ref{thm.genambichromsym.varis}.
\end{proof}

\begin{proof}
[Proof of Corollary \ref{cor.genambichromsym.K-free}.]We can apply Theorem
\ref{thm.genambichromsym.varis} to $0$ instead of $a_{K}$. As a result, we
obtain%
\begin{equation}
\Xi_{G}=\sum_{F\subseteq E}\left(  -1\right)  ^{\left\vert F\right\vert
}\left(  \prod_{\substack{K\in\mathfrak{K};\\K\subseteq F}}0\right)
\pi_{\operatorname*{union}F}. \label{pf.cor.genambichromsym.K-free.0}%
\end{equation}
Now, if $F$ is any subset of $E$, then%
\begin{equation}
\prod_{\substack{K\in\mathfrak{K};\\K\subseteq F}}0=%
\begin{cases}
1, & \text{if }F\text{ is }\mathfrak{K}\text{-free;}\\
0, & \text{if }F\text{ is not }\mathfrak{K}\text{-free}%
\end{cases}
\label{pf.cor.genambichromsym.K-free.1}%
\end{equation}
(indeed, this was already shown in our above proof of Corollary
\ref{cor.ambichromsym.K-free}).

Thus, (\ref{pf.cor.genambichromsym.K-free.0}) becomes%
\begin{align*}
\Xi_{G}  &  =\sum_{F\subseteq E}\left(  -1\right)  ^{\left\vert F\right\vert
}\underbrace{\left(  \prod_{\substack{K\in\mathfrak{K};\\K\subseteq
F}}0\right)  }_{\substack{=%
\begin{cases}
1, & \text{if }F\text{ is }\mathfrak{K}\text{-free;}\\
0, & \text{if }F\text{ is not }\mathfrak{K}\text{-free}%
\end{cases}
\\\text{(by (\ref{pf.cor.genambichromsym.K-free.1}))}}}\pi
_{\operatorname*{union}F}\\
&  =\sum_{F\subseteq E}\left(  -1\right)  ^{\left\vert F\right\vert }%
\begin{cases}
1, & \text{if }F\text{ is }\mathfrak{K}\text{-free;}\\
0, & \text{if }F\text{ is not }\mathfrak{K}\text{-free}%
\end{cases}
\ \ \pi_{\operatorname*{union}F}\\
&  =\sum_{\substack{F\subseteq E;\\F\text{ is }\mathfrak{K}\text{-free}%
}}\left(  -1\right)  ^{\left\vert F\right\vert }\underbrace{%
\begin{cases}
1, & \text{if }F\text{ is }\mathfrak{K}\text{-free;}\\
0, & \text{if }F\text{ is not }\mathfrak{K}\text{-free}%
\end{cases}
}_{\substack{=1\\\text{(since }F\text{ is }\mathfrak{K}\text{-free)}}%
}\ \ \pi_{\operatorname*{union}F}\\
&  \ \ \ \ \ \ \ \ \ \ +\sum_{\substack{F\subseteq E;\\F\text{ is not
}\mathfrak{K}\text{-free}}}\left(  -1\right)  ^{\left\vert F\right\vert
}\underbrace{%
\begin{cases}
1, & \text{if }F\text{ is }\mathfrak{K}\text{-free;}\\
0, & \text{if }F\text{ is not }\mathfrak{K}\text{-free}%
\end{cases}
}_{\substack{=0\\\text{(since }F\text{ is not }\mathfrak{K}\text{-free)}%
}}\ \ \pi_{\operatorname*{union}F}\\
&  \ \ \ \ \ \ \ \ \ \ \ \ \ \ \ \ \ \ \ \ \left(  \text{since each subset
}F\text{ of }E\text{ either is }\mathfrak{K}\text{-free or is not}\right) \\
&  =\sum_{\substack{F\subseteq E;\\F\text{ is }\mathfrak{K}\text{-free}%
}}\left(  -1\right)  ^{\left\vert F\right\vert }\pi_{\operatorname*{union}%
F}+\underbrace{\sum_{\substack{F\subseteq E;\\F\text{ is not }\mathfrak{K}%
\text{-free}}}\left(  -1\right)  ^{\left\vert F\right\vert }0\pi
_{\operatorname*{union}F}}_{=0}\\
&  =\sum_{\substack{F\subseteq E;\\F\text{ is }\mathfrak{K}\text{-free}%
}}\left(  -1\right)  ^{\left\vert F\right\vert }\pi_{\operatorname*{union}F}.
\end{align*}
This proves Corollary \ref{cor.genambichromsym.K-free}.
\end{proof}

\begin{proof}
[Proof of Corollary \ref{cor.genambichromsym.NBC}.]Let $\mathfrak{K}$ be the
set of all broken circuits of $G$.

Now, just as in the proof of Corollary \ref{cor.chromsym.NBC}, we can prove
the following equality:
\[
\sum_{\substack{F\subseteq E;\\F\text{ is }\mathfrak{K}\text{-free}}%
}=\sum_{\substack{F\subseteq E;\\F\text{ contains no broken}\\\text{circuit of
}G\text{ as a subset}}}
\]
(an equality between summation signs). Now, Corollary
\ref{cor.genambichromsym.K-free} yields%
\[
\Xi_{G}=\underbrace{\sum_{\substack{F\subseteq E;\\F\text{ is }\mathfrak{K}%
\text{-free}}}}_{=\sum_{\substack{F\subseteq E;\\F\text{ contains no
broken}\\\text{circuit of }G\text{ as a subset}}}}\left(  -1\right)
^{\left\vert F\right\vert }\pi_{\operatorname*{union}F}=\sum
_{\substack{F\subseteq E;\\F\text{ contains no broken}\\\text{circuit of
}G\text{ as a subset}}}\left(  -1\right)  ^{\left\vert F\right\vert }%
\pi_{\operatorname*{union}F}.
\]
This proves Corollary \ref{cor.genambichromsym.NBC}.
\end{proof}

\begin{proof}
[Proof of Theorem \ref{thm.genambichromsym.empty}.]Let $X$ be the totally
ordered set $\left\{  1\right\}  $ (equipped with the only possible order on
this set). Let $\ell:E\rightarrow X$ be the function sending each $e\in E$ to
$1\in X$. Let $\mathfrak{K}$ be the empty set. Clearly, $\mathfrak{K}$ is a
set of broken circuits of $G$. Theorem \ref{thm.genambichromsym.varis}
(applied to $0$ instead of $a_{K}$) yields%
\begin{align*}
\Xi_{G}  &  =\sum_{F\subseteq E}\left(  -1\right)  ^{\left\vert F\right\vert
}\underbrace{\left(  \prod_{\substack{K\in\mathfrak{K};\\K\subseteq
F}}0\right)  }_{\substack{=\left(  \text{empty product}\right)  \\\text{(since
}\mathfrak{K}\text{ is the empty set)}}}\pi_{\operatorname*{union}F}\\
&  =\sum_{F\subseteq E}\left(  -1\right)  ^{\left\vert F\right\vert
}\underbrace{\left(  \text{empty product}\right)  }_{=1}\pi
_{\operatorname*{union}F}=\sum_{F\subseteq E}\left(  -1\right)  ^{\left\vert
F\right\vert }\pi_{\operatorname*{union}F}.
\end{align*}
This proves Theorem \ref{thm.genambichromsym.empty}.
\end{proof}
\end{verlong}

Corollary \ref{cor.genambichromsym.NBC} can be used to prove certain results
about list colorings (i.e., colorings of a graph or ambigraph that are not
allowed to use certain colors for certain vertices); in particular,
\cite[Lemma 3.2]{Erey19} follows easily from Corollary
\ref{cor.genambichromsym.NBC} (just turn the graph $G$ into an ambigraph, and
set $\alpha_{f}$ to be the Iverson bracket $\left[  f\left(  v\right)  \neq
r\left(  v\right)  \text{ for each }v\in V\right]  $).

\section{\label{sec.alt-sum}Application: A vanishing alternating sum}

Chromatic symmetric functions of different graphs are far from being linearly
independent; they satisfy several linear relations. One such relation was
observed by Dahlberg and van Willigenburg in 2018 \cite[Proposition
5]{DahWil18}:

\begin{theorem}
\label{thm.dahwil.graph}Let $G=\left(  V,E\right)  $ be a finite graph. Let
$C$ be a circuit of $G$, and let $e\in C$ be arbitrary. Then,%
\[
\sum_{F\subseteq C\setminus\left\{  e\right\}  }\left(  -1\right)
^{\left\vert F\right\vert }X_{G\setminus F}=0.
\]
Here, whenever $F$ is a subset of $E$, the notation $G\setminus F$ denotes the
graph $\left(  V,\ E\setminus F\right)  $ (that is, the graph obtained from
$G$ by removing the edges in $F$).
\end{theorem}

This was extended to noncommutative chromatic symmetric functions
$Y_{G,\mathbf{t}}$ by Dahlberg and van Willigenburg in \cite[Proposition
3.6]{DahWil19}, and to weighted chromatic symmetric functions $X_{G,w}$ by
Crew and Spirkl in \cite[Theorem 6]{CreSpi19}. Again, we shall now one-up
these results by generalizing them from graphs to ambigraphs and by moving to
the abstract setting of Definition \ref{def.genambichrom}. Thus, we claim the following:

\begin{theorem}
\label{thm.dahwil.gen}Let $G$, $V$, $E$, $\varphi$, $Y$, $M$ and $\alpha_{f}$
be as in Definition \ref{def.genambichrom}. Let $C$ be a circuit of $G$, and
let $e\in C$ be a singleton edgery. Then, using the notations of Definition
\ref{def.genambichrom} \textbf{(a)}, we have%
\[
\sum_{F\subseteq C\setminus\left\{  e\right\}  }\left(  -1\right)
^{\left\vert F\right\vert }\Xi_{G\setminus F}=0.
\]
Here, whenever $F$ is a subset of $E$, the notation $G\setminus F$ denotes the
ambigraph $\left(  V,\ E\setminus F,\ \varphi\mid_{E\setminus F}\right)  $
(that is, the ambigraph obtained from $G$ by removing the edges in $F$).
\end{theorem}

Applying this to a graph instead of an ambigraph, and setting $Y=\mathbb{N}%
_{+}$ and $\alpha_{f}=\mathbf{x}_{f}$, we recover Theorem
\ref{thm.dahwil.graph}.

Theorem \ref{thm.dahwil.gen} is quite easy to prove despite its generality; in
fact, the beautiful sign-reversing involution argument from \cite[proof of
Proposition 3.6]{DahWil19} still does the trick. However, by way of
illustration, we shall now demonstrate how Theorem \ref{thm.dahwil.gen} can be
derived from Theorem \ref{thm.genambichromsym.empty} and Corollary
\ref{cor.genambichromsym.K-free}.

\begin{proof}
[Proof of Theorem \ref{thm.dahwil.gen}.]Let us set $B:=C\setminus\left\{
e\right\}  $. Thus, $B=C\setminus\left\{  e\right\}  \subseteq C\subseteq E$.

Now, we shall show the following (using the notations of Definition
\ref{def.genambichrom} \textbf{(b)}):

\begin{statement}
\textit{Claim 1:} Let $J$ be a subset of $B$. Then,%
\[
\Xi_{G\setminus J}=\sum_{\substack{F\subseteq E;\\J\subseteq E\setminus
F}}\left(  -1\right)  ^{\left\vert F\right\vert }\pi_{\operatorname*{union}%
F}.
\]

\end{statement}

\begin{vershort}
[\textit{Proof of Claim 1:} We know that $G\setminus J=\left(  V,\ E\setminus
J,\ \varphi\mid_{E\setminus J}\right)  $ is an ambigraph. Thus, we can apply
Theorem \ref{thm.genambichromsym.empty} to $G\setminus J$, $E\setminus J$ and
$\varphi\mid_{E\setminus J}$ instead of $G$, $E$ and $\varphi$. As a result,
we obtain%
\begin{equation}
\Xi_{G\setminus J}=\sum_{F\subseteq E\setminus J}\left(  -1\right)
^{\left\vert F\right\vert }\pi_{\operatorname*{union}F}
\label{pf.thm.dahwil.gen.c1.pf.short.1}%
\end{equation}
(since the meaning of $\operatorname*{union}F$ is the same whether we consider
$F$ as a set of edgeries of $G\setminus J$ or as a set of edgeries of $G$).
However, a subset $F$ of $E\setminus J$ is the same thing as a subset $F$ of
$E$ that is disjoint from $J$, and this is in turn the same as a subset $F$ of
$E$ that satisfies $J\subseteq E\setminus F$ (since $J\subseteq B\subseteq
E$). Hence, we can replace the summation sign \textquotedblleft$\sum
_{F\subseteq E\setminus J}$\textquotedblright\ in
(\ref{pf.thm.dahwil.gen.c1.pf.short.1}) by \textquotedblleft$\sum
_{\substack{F\subseteq E;\\J\subseteq E\setminus F}}$\textquotedblright. As a
result, (\ref{pf.thm.dahwil.gen.c1.pf.short.1}) becomes%
\[
\Xi_{G\setminus J}=\sum_{\substack{F\subseteq E;\\J\subseteq E\setminus
F}}\left(  -1\right)  ^{\left\vert F\right\vert }\pi_{\operatorname*{union}%
F}.
\]
Thus, Claim 1 is proved.]
\end{vershort}

\begin{verlong}
[\textit{Proof of Claim 1:} The definition of the ambigraph $G\setminus J$
shows that $G\setminus J=\left(  V,\ E\setminus J,\ \varphi\mid_{E\setminus
J}\right)  $. Thus, we can apply Theorem \ref{thm.genambichromsym.empty} to
$G\setminus J$, $E\setminus J$ and $\varphi\mid_{E\setminus J}$ instead of
$G$, $E$ and $\varphi$. As a result, we obtain%
\begin{equation}
\Xi_{G\setminus J}=\sum_{F\subseteq E\setminus J}\left(  -1\right)
^{\left\vert F\right\vert }\pi_{\operatorname*{union}F}
\label{pf.thm.dahwil.gen.c1.pf.1}%
\end{equation}
\footnote{We need to be careful here: Theorem \ref{thm.genambichromsym.empty}
uses the notation $\operatorname*{union}F$, which implicitly depends on the
map $\varphi$ that is part of the ambigraph $G=\left(  V,E,\varphi\right)  $.
When we apply Theorem \ref{thm.genambichromsym.empty} to $G\setminus J$,
$E\setminus J$ and $\varphi\mid_{E\setminus J}$ instead of $G$, $E$ and
$\varphi$, we are thus using this notation $\operatorname*{union}F$ in a
slightly different sense than it is used in Claim 1. Namely, in Claim 1, the
notation $\operatorname*{union}F$ means $\bigcup_{e\in F}\varphi\left(
e\right)  $ (since it is defined with respect to the ambigraph $G=\left(
V,E,\varphi\right)  $), whereas in (\ref{pf.thm.dahwil.gen.c1.pf.1}) it means
$\bigcup_{e\in F}\left(  \varphi\mid_{E\setminus J}\right)  \left(  e\right)
$ (since it is defined with respect to the ambigraph $G\setminus J=\left(
V,\ E\setminus J,\ \varphi\mid_{E\setminus J}\right)  $). Fortunately,
however, these two meanings of $\operatorname*{union}F$ are equivalent, since
any subset $F$ of $E\setminus J$ satisfies%
\[
\bigcup_{e\in F}\underbrace{\left(  \varphi\mid_{E\setminus J}\right)  \left(
e\right)  }_{=\varphi\left(  e\right)  }=\bigcup_{e\in F}\varphi\left(
e\right)  .
\]
Thus, we have no need to distinguish between these two meanings of
$\operatorname*{union}F$.}. However, we have $J\subseteq B\subseteq E$. Thus,
it is easy to see that
\begin{equation}
\mathcal{P}\left(  E\setminus J\right)  =\left\{  Z\in\mathcal{P}\left(
E\right)  \ \mid\ J\subseteq E\setminus Z\right\}
\label{pf.thm.dahwil.gen.c1.pf.2}%
\end{equation}
\footnote{\textit{Proof:} Let $H\in\mathcal{P}\left(  E\setminus J\right)  $.
Thus, $H$ is a subset of $E\setminus J$. Hence, $H\subseteq E\setminus
J\subseteq E$, so that $H\in\mathcal{P}\left(  E\right)  $. Also, from
$H\subseteq E\setminus J$, we conclude that $H$ is disjoint from $J$. In other
words, $J$ is disjoint from $H$. Combining this with $J\subseteq E$, we obtain
$J\subseteq E\setminus H$.
\par
Now, we know that $H\in\mathcal{P}\left(  E\right)  $ and $J\subseteq
E\setminus H$. In other words, $H$ is a $Z\in\mathcal{P}\left(  E\right)  $
satisfying $J\subseteq E\setminus Z$. In other words, $H\in\left\{
Z\in\mathcal{P}\left(  E\right)  \ \mid\ J\subseteq E\setminus Z\right\}  $.
\par
Forget that we fixed $H$. We thus have shown that $H\in\left\{  Z\in
\mathcal{P}\left(  E\right)  \ \mid\ J\subseteq E\setminus Z\right\}  $ for
each $H\in\mathcal{P}\left(  E\setminus J\right)  $. In other words,%
\begin{equation}
\mathcal{P}\left(  E\setminus J\right)  \subseteq\left\{  Z\in\mathcal{P}%
\left(  E\right)  \ \mid\ J\subseteq E\setminus Z\right\}  .
\label{pf.thm.dahwil.gen.c1.pf.2.pf.1}%
\end{equation}
\par
On the other hand, let $U\in\left\{  Z\in\mathcal{P}\left(  E\right)
\ \mid\ J\subseteq E\setminus Z\right\}  $. Thus, $U$ is a $Z\in
\mathcal{P}\left(  E\right)  $ satisfying $J\subseteq E\setminus Z$. In other
words, $U\in\mathcal{P}\left(  E\right)  $ and $J\subseteq E\setminus U$.
\par
From $J\subseteq E\setminus U$, we conclude that $J$ is disjoint from $U$. In
other words, $U$ is disjoint from $J$. However, from $U\in\mathcal{P}\left(
E\right)  $, we conclude that $U$ is a subset of $E$. Thus, $U$ is a subset of
$E$ that is disjoint from $J$. In other words, $U$ is a subset of $E\setminus
J$. In other words, $U\in\mathcal{P}\left(  E\setminus J\right)  $.
\par
Forget that we fixed $U$. We thus have shown that $U\in\mathcal{P}\left(
E\setminus J\right)  $ for each $U\in\left\{  Z\in\mathcal{P}\left(  E\right)
\ \mid\ J\subseteq E\setminus Z\right\}  $. In other words,%
\[
\left\{  Z\in\mathcal{P}\left(  E\right)  \ \mid\ J\subseteq E\setminus
Z\right\}  \subseteq\mathcal{P}\left(  E\setminus J\right)  .
\]
Combining this with (\ref{pf.thm.dahwil.gen.c1.pf.2.pf.1}), we obtain
\[
\mathcal{P}\left(  E\setminus J\right)  =\left\{  Z\in\mathcal{P}\left(
E\right)  \ \mid\ J\subseteq E\setminus Z\right\}  .
\]
This proves (\ref{pf.thm.dahwil.gen.c1.pf.2}).}.

Now, we have the following equality between summation signs:%
\begin{align*}
\sum_{F\subseteq E\setminus J}  &  =\sum_{F\in\mathcal{P}\left(  E\setminus
J\right)  }=\sum_{F\in\left\{  Z\in\mathcal{P}\left(  E\right)  \ \mid
\ J\subseteq E\setminus Z\right\}  }\ \ \ \ \ \ \ \ \ \ \left(  \text{by
(\ref{pf.thm.dahwil.gen.c1.pf.2})}\right) \\
&  =\sum_{\substack{F\in\mathcal{P}\left(  E\right)  ;\\J\subseteq E\setminus
F}}=\sum_{\substack{F\subseteq E;\\J\subseteq E\setminus F}}.
\end{align*}
Hence, (\ref{pf.thm.dahwil.gen.c1.pf.1}) becomes%
\[
\Xi_{G\setminus J}=\underbrace{\sum_{F\subseteq E\setminus J}}_{=\sum
_{\substack{F\subseteq E;\\J\subseteq E\setminus F}}}\left(  -1\right)
^{\left\vert F\right\vert }\pi_{\operatorname*{union}F}=\sum
_{\substack{F\subseteq E;\\J\subseteq E\setminus F}}\left(  -1\right)
^{\left\vert F\right\vert }\pi_{\operatorname*{union}F}.
\]
Thus, Claim 1 is proved.]
\end{verlong}

\begin{statement}
\textit{Claim 2:} We have
\begin{equation}
\sum_{\substack{F\subseteq E;\\B\subseteq F}}\left(  -1\right)  ^{\left\vert
F\right\vert }\pi_{\operatorname*{union}F}=0. \label{pf.thm.dahwil.gen.0=}%
\end{equation}

\end{statement}

\begin{vershort}
[\textit{Proof of Claim 2:} Theorem \ref{thm.genambichromsym.empty} yields%
\begin{equation}
\Xi_{G}=\sum_{F\subseteq E}\left(  -1\right)  ^{\left\vert F\right\vert }%
\pi_{\operatorname*{union}F}. \label{pf.thm.dahwil.gen.XiG.short.1}%
\end{equation}

On the other hand, let us define a labeling function $\ell:E\rightarrow
\mathbb{N}$ by setting $\ell\left(  e\right)  =1$ and setting $\ell\left(
f\right)  =0$ for all $f\in E\setminus\left\{  e\right\}  $. Then, the edgery
$e$ is the unique singleton edgery in $C$ having maximum label. Hence,
$C\setminus\left\{  e\right\}  $ is a broken circuit of $G$. In other words,
$B$ is a broken circuit of $G$ (since $B=C\setminus\left\{  e\right\}  $).
Hence, $\left\{  B\right\}  $ is a set of broken circuits of $G$. Therefore,
Corollary \ref{cor.genambichromsym.K-free} (applied to $\mathfrak{K}=\left\{
B\right\}  $) yields
\[
\Xi_{G}=\sum_{\substack{F\subseteq E;\\F\text{ is }\left\{  B\right\}
\text{-free}}}\left(  -1\right)  ^{\left\vert F\right\vert }\pi
_{\operatorname*{union}F}=\sum_{\substack{F\subseteq E;\\B\not \subseteq
F}}\left(  -1\right)  ^{\left\vert F\right\vert }\pi_{\operatorname*{union}F}%
\]
(since the condition \textquotedblleft$F$ is $\left\{  B\right\}
$-free\textquotedblright\ is easily seen to be equivalent to \textquotedblleft%
$B\not \subseteq F$\textquotedblright). Subtracting this equality from
(\ref{pf.thm.dahwil.gen.XiG.short.1}), we obtain%
\begin{equation}
0=\sum_{F\subseteq E}\left(  -1\right)  ^{\left\vert F\right\vert }%
\pi_{\operatorname*{union}F}-\sum_{\substack{F\subseteq E;\\B\not \subseteq
F}}\left(  -1\right)  ^{\left\vert F\right\vert }\pi_{\operatorname*{union}%
F}=\sum_{\substack{F\subseteq E;\\B\subseteq F}}\left(  -1\right)
^{\left\vert F\right\vert }\pi_{\operatorname*{union}F}.\nonumber
\end{equation}
This proves Claim 2.]
\end{vershort}

\begin{verlong}
[\textit{Proof of Claim 2:} Theorem \ref{thm.genambichromsym.empty} yields%
\begin{align}
\Xi_{G}  &  =\sum_{F\subseteq E}\left(  -1\right)  ^{\left\vert F\right\vert
}\pi_{\operatorname*{union}F}\nonumber\\
&  =\sum_{\substack{F\subseteq E;\\B\subseteq F}}\left(  -1\right)
^{\left\vert F\right\vert }\pi_{\operatorname*{union}F}+\sum
_{\substack{F\subseteq E;\\B\not \subseteq F}}\left(  -1\right)  ^{\left\vert
F\right\vert }\pi_{\operatorname*{union}F}
\label{pf.thm.dahwil.gen.XiG.c2.pf.2}%
\end{align}
(since each subset $F$ of $E$ satisfies either $B\subseteq F$ or
$B\not \subseteq F$, but not both at the same time).

Let us now find a different formula for $\Xi_{G}$. We define a function
$\ell:E\rightarrow\mathbb{N}$ by setting%
\[
\ell\left(  f\right)  =\left[  f=e\right]  \ \ \ \ \ \ \ \ \ \ \text{for each
}f\in E.
\]
We shall use this function $\ell$ as our labeling function (where the role of
the totally ordered set $X$ is played by $\mathbb{N}$ equipped with the usual
total order). It is easy to see that the edgery $e$ is the unique singleton
edgery in $C$ having maximum label\footnote{\textit{Proof.} Clearly, $e$ is a
singleton edgery in $C$ (since $e\in C$ is a singleton edgery). We shall now
show that $e$ has a larger label than any other singleton edgery in $C$.
\par
Indeed, let $f$ be any singleton edgery in $C$ distinct from $e$. Then, the
definition of $\ell$ yields $\ell\left(  f\right)  =\left[  f=e\right]  =0$
(since we don't have $f=e$ (because $f$ is distinct from $e$)). On the other
hand, the definition of $\ell$ yields $\ell\left(  e\right)  =\left[
e=e\right]  =1$ (since $e=e$). Hence, $\ell\left(  e\right)  =1>0=\ell\left(
f\right)  $. In other words, $e$ has a larger label than $f$ (since the label
of $e$ is $\ell\left(  e\right)  $, whereas the label of $f$ is $\ell\left(
f\right)  $).
\par
Forget that we fixed $f$. We thus have shown that $e$ has a larger label than
$f$ whenever $f$ is any singleton edgery in $C$ distinct from $e$. In other
words, $e$ has a larger label than any other singleton edgery in $C$. Hence,
the edgery $e$ is the unique singleton edgery in $C$ having maximum label
(since $e$ itself is a singleton edgery in $C$).}. Therefore, $C\setminus
\left\{  e\right\}  $ is a broken circuit of $G$ (by the definition of a
\textquotedblleft broken circuit\textquotedblright\ in Definition
\ref{def.ambigraph.BC}). In other words, $B$ is a broken circuit of $G$ (since
$B=C\setminus\left\{  e\right\}  $). Hence, $\left\{  B\right\}  $ is a set of
broken circuits of $G$. Therefore, Corollary \ref{cor.genambichromsym.K-free}
(applied to $\mathfrak{K}=\left\{  B\right\}  $) yields
\begin{equation}
\Xi_{G}=\sum_{\substack{F\subseteq E;\\F\text{ is }\left\{  B\right\}
\text{-free}}}\left(  -1\right)  ^{\left\vert F\right\vert }\pi
_{\operatorname*{union}F}. \label{pf.thm.dahwil.gen.XiG.c2.pf.3}%
\end{equation}

However, if $F$ is a subset of $E$, then the condition \textquotedblleft$F$ is
$\left\{  B\right\}  $-free\textquotedblright\ is equivalent to
\textquotedblleft$B\not \subseteq F$\textquotedblright%
\ \ \ \ \footnote{\textit{Proof.} Let $F$ be a subset of $E$. Then, we have
the following chain of logical equivalences:%
\begin{align*}
&  \ \left(  F\text{ is }\left\{  B\right\}  \text{-free}\right) \\
&  \Longleftrightarrow\ \left(  F\text{ contains no }K\in\left\{  B\right\}
\text{ as a subset}\right) \\
&  \ \ \ \ \ \ \ \ \ \ \ \ \ \ \ \ \ \ \ \ \left(  \text{by the definition of
\textquotedblleft}\left\{  B\right\}  \text{-free\textquotedblright\ in
Definition \ref{def.K-free}}\right) \\
&  \Longleftrightarrow\ \left(  \text{there exists no }K\in\left\{  B\right\}
\text{ such that }F\text{ contains }K\text{ as a subset}\right) \\
&  \Longleftrightarrow\ \left(  \text{there exists no }K\in\left\{  B\right\}
\text{ such that }K\subseteq F\right) \\
&  \Longleftrightarrow\ \left(  \text{each }K\in\left\{  B\right\}  \text{
satisfies }K\not \subseteq F\right) \\
&  \Longleftrightarrow\ \left(  B\not \subseteq F\right)
\ \ \ \ \ \ \ \ \ \ \left(  \text{since the only }K\in\left\{  B\right\}
\text{ is }B\right)  .
\end{align*}
Hence, the condition \textquotedblleft$F$ is $\left\{  B\right\}
$-free\textquotedblright\ is equivalent to \textquotedblleft$B\not \subseteq
F$\textquotedblright. Qed.}. Hence, the summation sign \textquotedblleft%
$\sum_{\substack{F\subseteq E;\\F\text{ is }\left\{  B\right\}  \text{-free}%
}}$\textquotedblright\ can be rewritten as \textquotedblleft$\sum
_{\substack{F\subseteq E;\\B\not \subseteq F}}$\textquotedblright. Therefore,
we can rewrite (\ref{pf.thm.dahwil.gen.XiG.c2.pf.3}) as%
\begin{equation}
\Xi_{G}=\sum_{\substack{F\subseteq E;\\B\not \subseteq F}}\left(  -1\right)
^{\left\vert F\right\vert }\pi_{\operatorname*{union}F}.
\label{pf.thm.dahwil.gen.XiG.c2.pf.4}%
\end{equation}

Subtracting this equality from (\ref{pf.thm.dahwil.gen.XiG.c2.pf.2}), we
obtain%
\begin{align*}
\Xi_{G}-\Xi_{G}  &  =\left(  \sum_{\substack{F\subseteq E;\\B\subseteq
F}}\left(  -1\right)  ^{\left\vert F\right\vert }\pi_{\operatorname*{union}%
F}+\sum_{\substack{F\subseteq E;\\B\not \subseteq F}}\left(  -1\right)
^{\left\vert F\right\vert }\pi_{\operatorname*{union}F}\right)  -\sum
_{\substack{F\subseteq E;\\B\not \subseteq F}}\left(  -1\right)  ^{\left\vert
F\right\vert }\pi_{\operatorname*{union}F}\\
&  =\sum_{\substack{F\subseteq E;\\B\subseteq F}}\left(  -1\right)
^{\left\vert F\right\vert }\pi_{\operatorname*{union}F}.
\end{align*}
Comparing this with $\Xi_{G}-\Xi_{G}=0$, we obtain%
\[
\sum_{\substack{F\subseteq E;\\B\subseteq F}}\left(  -1\right)  ^{\left\vert
F\right\vert }\pi_{\operatorname*{union}F}=0.
\]
This proves Claim 2.]
\end{verlong}

\begin{vershort}
However, from $C\setminus\left\{  e\right\}  =B$, we obtain%
\begin{align}
\sum_{F\subseteq C\setminus\left\{  e\right\}  }\left(  -1\right)
^{\left\vert F\right\vert }\Xi_{G\setminus F}  &  =\sum_{F\subseteq B}\left(
-1\right)  ^{\left\vert F\right\vert }\Xi_{G\setminus F}=\sum_{J\subseteq
B}\left(  -1\right)  ^{\left\vert J\right\vert }\underbrace{\Xi_{G\setminus
J}}_{\substack{=\sum_{\substack{F\subseteq E;\\J\subseteq E\setminus
F}}\left(  -1\right)  ^{\left\vert F\right\vert }\pi_{\operatorname*{union}%
F}\\\text{(by Claim 1)}}}\nonumber\\
&  \ \ \ \ \ \ \ \ \ \ \ \ \ \ \ \ \ \ \ \ \left(
\begin{array}
[c]{c}%
\text{here, we have renamed the}\\
\text{summation index }F\text{ as }J
\end{array}
\right) \nonumber\\
&  =\sum_{J\subseteq B}\left(  -1\right)  ^{\left\vert J\right\vert }%
\sum_{\substack{F\subseteq E;\\J\subseteq E\setminus F}}\left(  -1\right)
^{\left\vert F\right\vert }\pi_{\operatorname*{union}F}\nonumber\\
&  =\underbrace{\sum_{J\subseteq B}\ \ \sum_{\substack{F\subseteq
E;\\J\subseteq E\setminus F}}}_{=\sum_{F\subseteq E}\ \ \sum
_{\substack{J\subseteq B;\\J\subseteq E\setminus F}}}\left(  -1\right)
^{\left\vert J\right\vert }\left(  -1\right)  ^{\left\vert F\right\vert }%
\pi_{\operatorname*{union}F}\nonumber\\
&  =\sum_{F\subseteq E}\ \ \underbrace{\sum_{\substack{J\subseteq
B;\\J\subseteq E\setminus F}}}_{\substack{=\sum_{J\subseteq B\cap\left(
E\setminus F\right)  }\\=\sum_{J\subseteq B\setminus F}\\\text{(since
}B\subseteq E\text{ entails}\\B\cap\left(  E\setminus F\right)  =B\setminus
F\text{)}}}\left(  -1\right)  ^{\left\vert J\right\vert }\left(  -1\right)
^{\left\vert F\right\vert }\pi_{\operatorname*{union}F}\nonumber\\
&  =\sum_{F\subseteq E}\ \ \underbrace{\sum_{J\subseteq B\setminus F}\left(
-1\right)  ^{\left\vert J\right\vert }}_{\substack{=\sum_{I\subseteq
B\setminus F}\left(  -1\right)  ^{\left\vert I\right\vert }\\=\left[
B\setminus F=\varnothing\right]  \\\text{(by Lemma \ref{lem.cancel})}}}\left(
-1\right)  ^{\left\vert F\right\vert }\pi_{\operatorname*{union}F}\nonumber\\
&  =\sum_{F\subseteq E}\underbrace{\left[  B\setminus F=\varnothing\right]
}_{=\left[  B\subseteq F\right]  }\left(  -1\right)  ^{\left\vert F\right\vert
}\pi_{\operatorname*{union}F}\nonumber\\
&  =\sum_{F\subseteq E}\left[  B\subseteq F\right]  \left(  -1\right)
^{\left\vert F\right\vert }\pi_{\operatorname*{union}F}.
\label{pf.thm.dahwil.gen.5}%
\end{align}
In the sum on the right-hand side, we can clearly remove all addends that
don't satisfy $B\subseteq F$, since the presence of the $\left[  B\subseteq
F\right]  =0$ factor renders all these addends equal to $0$. Thus, we are left
with only the addends that do satisfy $B\subseteq F$. Hence,
(\ref{pf.thm.dahwil.gen.5}) rewrites as%
\begin{align*}
\sum_{F\subseteq C\setminus\left\{  e\right\}  }\left(  -1\right)
^{\left\vert F\right\vert }\Xi_{G\setminus F}  &  =\sum_{\substack{F\subseteq
E;\\B\subseteq F}}\underbrace{\left[  B\subseteq F\right]  }%
_{\substack{=1\\\text{(since }B\subseteq F\text{)}}}\left(  -1\right)
^{\left\vert F\right\vert }\pi_{\operatorname*{union}F}\\
&  =\sum_{\substack{F\subseteq E;\\B\subseteq F}}\left(  -1\right)
^{\left\vert F\right\vert }\pi_{\operatorname*{union}F}%
=0\ \ \ \ \ \ \ \ \ \ \left(  \text{by (\ref{pf.thm.dahwil.gen.0=})}\right)  .
\end{align*}

\end{vershort}

\begin{verlong}
However, from $C\setminus\left\{  e\right\}  =B$, we obtain%
\begin{align}
\sum_{F\subseteq C\setminus\left\{  e\right\}  }\left(  -1\right)
^{\left\vert F\right\vert }\Xi_{G\setminus F}  &  =\sum_{F\subseteq B}\left(
-1\right)  ^{\left\vert F\right\vert }\Xi_{G\setminus F}=\sum_{J\subseteq
B}\left(  -1\right)  ^{\left\vert J\right\vert }\underbrace{\Xi_{G\setminus
J}}_{\substack{=\sum_{\substack{F\subseteq E;\\J\subseteq E\setminus
F}}\left(  -1\right)  ^{\left\vert F\right\vert }\pi_{\operatorname*{union}%
F}\\\text{(by Claim 1)}}}\nonumber\\
&  \ \ \ \ \ \ \ \ \ \ \ \ \ \ \ \ \ \ \ \ \left(
\begin{array}
[c]{c}%
\text{here, we have renamed the}\\
\text{summation index }F\text{ as }J
\end{array}
\right) \nonumber\\
&  =\sum_{J\subseteq B}\left(  -1\right)  ^{\left\vert J\right\vert }%
\sum_{\substack{F\subseteq E;\\J\subseteq E\setminus F}}\left(  -1\right)
^{\left\vert F\right\vert }\pi_{\operatorname*{union}F}\nonumber\\
&  =\underbrace{\sum_{J\subseteq B}\ \ \sum_{\substack{F\subseteq
E;\\J\subseteq E\setminus F}}}_{=\sum_{F\subseteq E}\ \ \sum
_{\substack{J\subseteq B;\\J\subseteq E\setminus F}}}\left(  -1\right)
^{\left\vert J\right\vert }\left(  -1\right)  ^{\left\vert F\right\vert }%
\pi_{\operatorname*{union}F}\nonumber\\
&  =\sum_{F\subseteq E}\ \ \underbrace{\sum_{\substack{J\subseteq
B;\\J\subseteq E\setminus F}}}_{\substack{=\sum_{J\subseteq B\cap\left(
E\setminus F\right)  }}}\left(  -1\right)  ^{\left\vert J\right\vert }\left(
-1\right)  ^{\left\vert F\right\vert }\pi_{\operatorname*{union}F}\nonumber\\
&  =\sum_{F\subseteq E}\ \ \underbrace{\sum_{J\subseteq B\cap\left(
E\setminus F\right)  }\left(  -1\right)  ^{\left\vert J\right\vert }%
}_{\substack{=\sum_{I\subseteq B\cap\left(  E\setminus F\right)  }\left(
-1\right)  ^{\left\vert I\right\vert }\\\text{(here, we have renamed}%
\\\text{the summation index }J\text{ as }I\text{)}}}\left(  -1\right)
^{\left\vert F\right\vert }\pi_{\operatorname*{union}F}\nonumber\\
&  =\sum_{F\subseteq E}\ \ \underbrace{\sum_{I\subseteq B\cap\left(
E\setminus F\right)  }\left(  -1\right)  ^{\left\vert I\right\vert }%
}_{\substack{=\left[  B\cap\left(  E\setminus F\right)  =\varnothing\right]
\\\text{(by Lemma \ref{lem.cancel},}\\\text{applied to }S=B\cap\left(
E\setminus F\right)  \text{)}}}\left(  -1\right)  ^{\left\vert F\right\vert
}\pi_{\operatorname*{union}F}\nonumber\\
&  =\sum_{F\subseteq E}\left[  \underbrace{B\cap\left(  E\setminus F\right)
}_{=\left(  B\cap E\right)  \setminus F}=\varnothing\right]  \left(
-1\right)  ^{\left\vert F\right\vert }\pi_{\operatorname*{union}F}\nonumber\\
&  =\sum_{F\subseteq E}\left[  \underbrace{\left(  B\cap E\right)
}_{\substack{=B\\\text{(since }B\subseteq E\text{)}}}\setminus F=\varnothing
\right]  \left(  -1\right)  ^{\left\vert F\right\vert }\pi
_{\operatorname*{union}F}\nonumber
\end{align}%
\begin{align*}
&  =\sum_{F\subseteq E}\underbrace{\left[  B\setminus F=\varnothing\right]
}_{\substack{=\left[  B\subseteq F\right]  \\\text{(since the statement
\textquotedblleft}B\setminus F=\varnothing\text{\textquotedblright}\\\text{is
equivalent to \textquotedblleft}B\subseteq F\text{\textquotedblright)}%
}}\left(  -1\right)  ^{\left\vert F\right\vert }\pi_{\operatorname*{union}F}\\
&  =\sum_{F\subseteq E}\left[  B\subseteq F\right]  \left(  -1\right)
^{\left\vert F\right\vert }\pi_{\operatorname*{union}F}\\
&  =\sum_{\substack{F\subseteq E;\\B\subseteq F}}\underbrace{\left[
B\subseteq F\right]  }_{\substack{=1\\\text{(since }B\subseteq F\text{)}%
}}\left(  -1\right)  ^{\left\vert F\right\vert }\pi_{\operatorname*{union}%
F}+\sum_{\substack{F\subseteq E;\\\text{we don't have }B\subseteq
F}}\underbrace{\left[  B\subseteq F\right]  }_{\substack{=0\\\text{(since we
don't}\\\text{have }B\subseteq F\text{)}}}\left(  -1\right)  ^{\left\vert
F\right\vert }\pi_{\operatorname*{union}F}\\
&  \ \ \ \ \ \ \ \ \ \ \ \ \ \ \ \ \ \ \ \ \left(  \text{since each subset
}F\text{ of }E\text{ either satisfies }B\subseteq F\text{ or does not}\right)
\\
&  =\sum_{\substack{F\subseteq E;\\B\subseteq F}}\left(  -1\right)
^{\left\vert F\right\vert }\pi_{\operatorname*{union}F}+\underbrace{\sum
_{\substack{F\subseteq E;\\\text{we don't have }B\subseteq F}}0\left(
-1\right)  ^{\left\vert F\right\vert }\pi_{\operatorname*{union}F}}_{=0}\\
&  =\sum_{\substack{F\subseteq E;\\B\subseteq F}}\left(  -1\right)
^{\left\vert F\right\vert }\pi_{\operatorname*{union}F}%
=0\ \ \ \ \ \ \ \ \ \ \left(  \text{by (\ref{pf.thm.dahwil.gen.0=})}\right)  .
\end{align*}

\end{verlong}

\noindent This proves Theorem \ref{thm.dahwil.gen}.
\end{proof}

\section{\label{sec.matroid}The characteristic polynomial of a matroid}

\subsection{An introduction to matroids}

We shall now present a result that can be considered as a generalization of
Theorem \ref{thm.chrompol.varis} in a different direction than Theorem
\ref{thm.chromsym.varis}: namely, a formula for the characteristic polynomial
of a matroid. Let us first recall the basic notions from the theory of
matroids that will be needed to state it.

\begin{verlong}
[TODO: Make the following proofs more detailed.]
\end{verlong}

First, we introduce some basic poset-related terminology:

\begin{definition}
\label{def.poset.max}Let $P$ be a poset.

\textbf{(a)} An element $v$ of $P$ is said to be \emph{maximal} (with respect
to $P$) if and only if every $w\in P$ satisfying $w\geq v$ must satisfy $w=v$.

\textbf{(b)} An element $v$ of $P$ is said to be \emph{minimal} (with respect
to $P$) if and only if every $w\in P$ satisfying $w\leq v$ must satisfy $w=v$.
\end{definition}

\begin{definition}
\label{def.poset.P(E)}For any set $E$, we shall regard the powerset
$\mathcal{P}\left(  E\right)  $ as a poset (with respect to inclusion). Thus,
any subset $\mathcal{S}$ of $\mathcal{P}\left(  E\right)  $ also becomes a
poset, and therefore the notions of \textquotedblleft
minimal\textquotedblright\ and \textquotedblleft maximal\textquotedblright%
\ elements in $\mathcal{S}$ make sense. Beware that these notions are not
related to size; i.e., a maximal element of $\mathcal{S}$ might not be a
maximum-size element of $\mathcal{S}$.
\end{definition}

Now, let us define the notion of \textquotedblleft matroid\textquotedblright%
\ that we will use:

\begin{definition}
\label{def.matroid}\textbf{(a)} A \emph{matroid} means a pair $\left(
E,\mathcal{I}\right)  $ consisting of a finite set $E$ and a set
$\mathcal{I}\subseteq\mathcal{P}\left(  E\right)  $ satisfying the following axioms:

\begin{itemize}
\item \textit{Matroid axiom 1:} We have $\varnothing\in\mathcal{I}$.

\item \textit{Matroid axiom 2:} If $Y\in\mathcal{I}$ and $Z\in\mathcal{P}%
\left(  E\right)  $ are such that $Z\subseteq Y$, then $Z\in\mathcal{I}$.

\item \textit{Matroid axiom 3:} If $Y\in\mathcal{I}$ and $Z\in\mathcal{I}$ are
such that $\left\vert Y\right\vert <\left\vert Z\right\vert $, then there
exists some $x\in Z\setminus Y$ such that $Y\cup\left\{  x\right\}
\in\mathcal{I}$.
\end{itemize}

\textbf{(b)} Let $\left(  E,\mathcal{I}\right)  $ be a matroid. A subset $S$
of $E$ is said to be \emph{independent} (for this matroid) if and only if
$S\in\mathcal{I}$. The set $E$ is called the \emph{ground set} of the matroid
$\left(  E,\mathcal{I}\right)  $.
\end{definition}

There are different definitions of a matroid in the literature; these
definitions are (mostly) equivalent, but not always in the obvious
way\footnote{Indeed, most of these definitions define a matroid as a pair
$\left(  E,U\right)  $ consisting of a finite set $E$ and a subset
$U\subseteq\mathcal{P}\left(  E\right)  $ satisfying a certain set of axioms,
but these sets of axioms are not always equivalent, so they define different
classes of pairs $\left(  E,U\right)  $. Thus, a matroid in the sense of one
definition is not necessarily a matroid in the sense of another definition.
However, there are canonical bijections between one type of matroids and
another (see, e.g., \cite[\S 10.2]{Schrij13}); these are commonly known as
\textquotedblleft cryptomorphisms\textquotedblright.}. Definition
\ref{def.matroid} is how a matroid is defined in \cite[\S 10.1]{Schrij13} and
in \cite[Definition 3.4.1]{Martin22} (where it is called a \textquotedblleft%
(matroid) independence system\textquotedblright). The definition of a matroid
given in Stanley's \cite[Definition 3.8]{Stanley} is directly equivalent to
Definition \ref{def.matroid}, with the only differences that

\begin{itemize}
\item Stanley replaces Matroid axiom 1 by the requirement that $\mathcal{I}%
\neq\varnothing$ (which is, of course, equivalent to Matroid axiom 1 as long
as Matroid axiom 2 is assumed), and

\item Stanley replaces Matroid axiom 3 by the requirement that for every
$T\in\mathcal{P}\left(  E\right)  $, all maximal elements of $\mathcal{I}%
\cap\mathcal{P}\left(  T\right)  $ have the same cardinality\footnote{Here, we
regard $\mathcal{I}\cap\mathcal{P}\left(  T\right)  $ as a poset with respect
to inclusion (as explained in Definition \ref{def.poset.P(E)}). Thus, an
element $Y$ of this poset is maximal if and only if there exists no
$Z\in\mathcal{I}\cap\mathcal{P}\left(  T\right)  $ such that $Y$ is a proper
subset of $Z$.} (this requirement is equivalent to Matroid axiom 3 as long as
Matroid axiom 2 is assumed).
\end{itemize}

We now introduce some terminology related to matroids:

\begin{definition}
\label{def.matroid.notions}Let $M=\left(  E,\mathcal{I}\right)  $ be a matroid.

\textbf{(a)} We define a function $r_{M}:\mathcal{P}\left(  E\right)
\rightarrow\mathbb{N}$ by setting%
\begin{equation}
r_{M}\left(  S\right)  =\max\left\{  \left\vert Z\right\vert \ \mid
\ Z\in\mathcal{I}\text{ and }Z\subseteq S\right\}
\ \ \ \ \ \ \ \ \ \ \text{for every }S\subseteq E. \label{eq.def.rank.def}%
\end{equation}
(Note that the right-hand side of (\ref{eq.def.rank.def}) is well-defined,
because there exists at least one $Z\in\mathcal{I}$ satisfying $Z\subseteq S$
(namely, $Z=\varnothing$).) If $S$ is a subset of $E$, then the nonnegative
integer $r_{M}\left(  S\right)  $ is called the \emph{rank} of $S$ (with
respect to $M$). It is clear that $r_{M}$ is a weakly increasing function from
the poset $\mathcal{P}\left(  E\right)  $ to $\mathbb{N}$.

\textbf{(b)} If $k\in\mathbb{N}$, then a $k$\emph{-flat} of $M$ means a subset
of $E$ that has rank $k$ and is maximal among all such subsets (i.e., it is
not a proper subset of any other subset having rank $k$). (Beware: Not all
$k$-flats have the same size.) A \emph{flat} of $M$ is a subset of $E$ which
is a $k$-flat for some $k\in\mathbb{N}$. We let $\operatorname*{Flats}M$
denote the set of all flats of $M$; thus, $\operatorname*{Flats}M$ is a
subposet of $\mathcal{P}\left(  E\right)  $.

\textbf{(c)} A \emph{circuit} of $M$ means a minimal element of $\mathcal{P}%
\left(  E\right)  \setminus\mathcal{I}$. (That is, a circuit of $M$ means a
subset of $E$ which is not independent (for $M$) and which is minimal among
such subsets.)

\textbf{(d)} An element $e$ of $E$ is said to be a \emph{loop} (of $M$) if
$\left\{  e\right\}  \notin\mathcal{I}$. The matroid $M$ is said to be
\emph{loopless} if no loops (of $M$) exist.
\end{definition}

Notice that the function that we called $r_{M}$ in Definition
\ref{def.matroid.notions} \textbf{(a)} is called the \emph{rank function} of
$M$, and is denoted by $\operatorname*{rk}$ in Stanley's \cite[Lecture
3]{Stanley}.

One of the most classical examples of a matroid is the \emph{graphical
matroid} of a graph:

\begin{example}
\label{exam.matroid.graphical}Let $G=\left(  V,E\right)  $ be a finite graph.
Define a subset $\mathcal{I}$ of $\mathcal{P}\left(  E\right)  $ by%
\[
\mathcal{I}=\left\{  T\in\mathcal{P}\left(  E\right)  \ \mid\ T\text{ contains
no circuit of }G\text{ as a subset}\right\}  .
\]
Then, $\left(  E,\mathcal{I}\right)  $ is a matroid; it is called the
\emph{graphical matroid} (or the \emph{cycle matroid}) of $G$. It has the
following properties:

\begin{itemize}
\item The matroid $\left(  E,\mathcal{I}\right)  $ is loopless.

\item For each $T\in\mathcal{P}\left(  E\right)  $, we have%
\[
r_{\left(  E,\mathcal{I}\right)  }\left(  T\right)  =\left\vert V\right\vert
-\operatorname*{conn}\left(  V,T\right)
\]
(where $\operatorname*{conn}\left(  V,T\right)  $ is defined as in Definition
\ref{def.conn}).

\item The circuits of the matroid $\left(  E,\mathcal{I}\right)  $ are
precisely the circuits of the graph $G$.

\item The flats of the matroid $\left(  E,\mathcal{I}\right)  $ are related to
colorings of $G$. More precisely: For each set $X$ and each $X$-coloring $f$
of $G$, the set $E\cap\operatorname*{Eqs}f$ is a flat of $\left(
E,\mathcal{I}\right)  $. Every flat of $\left(  E,\mathcal{I}\right)  $ can be
obtained in this way when $X$ is chosen large enough; but often, several
distinct $X$-colorings $f$ lead to one and the same flat $E\cap
\operatorname*{Eqs}f$.
\end{itemize}
\end{example}

We recall three basic facts that are used countless times in arguing about matroids:

\begin{lemma}
\label{lem.matroid.rI}Let $M=\left(  E,\mathcal{I}\right)  $ be a matroid. Let
$T\in\mathcal{I}$. Then, $r_{M}\left(  T\right)  =\left\vert T\right\vert $.
\end{lemma}

\begin{proof}
[Proof of Lemma \ref{lem.matroid.rI}.]We have $T\in\mathcal{I}$ and
$T\subseteq T$. Thus, $T$ is a $Z\in\mathcal{I}$ satisfying $Z\subseteq T$.
Therefore, $\left\vert T\right\vert \in\left\{  \left\vert Z\right\vert
\ \mid\ Z\in\mathcal{I}\text{ and }Z\subseteq T\right\}  $, so that%
\begin{equation}
\left\vert T\right\vert \leq\max\left\{  \left\vert Z\right\vert \ \mid
\ Z\in\mathcal{I}\text{ and }Z\subseteq T\right\}  \label{pf.lem.matroid.rI.1}%
\end{equation}
(since any element of a set of integers is smaller or equal to the maximum of
this set).

On the other hand, the definition of $r_{M}$ yields%
\[
r_{M}\left(  T\right)  =\max\left\{  \left\vert Z\right\vert \ \mid
\ Z\in\mathcal{I}\text{ and }Z\subseteq T\right\}  .
\]
Hence, (\ref{pf.lem.matroid.rI.1}) rewrites as follows:%
\[
\left\vert T\right\vert \leq r_{M}\left(  T\right)  .
\]

Also,
\begin{align*}
r_{M}\left(  T\right)   &  =\max\left\{  \left\vert Z\right\vert \ \mid
\ Z\in\mathcal{I}\text{ and }Z\subseteq T\right\}  \ \ \ \ \ \ \ \ \ \ \left(
\text{by the definition of }r_{M}\right) \\
&  \in\left\{  \left\vert Z\right\vert \ \mid\ Z\in\mathcal{I}\text{ and
}Z\subseteq T\right\}
\end{align*}
(since the maximum of any set belongs to this set). Thus, there exists a
$Z\in\mathcal{I}$ satisfying $Z\subseteq T$ and $r_{M}\left(  T\right)
=\left\vert Z\right\vert $. Consider this $Z$. From $Z\subseteq T$, we obtain
$\left\vert Z\right\vert \leq\left\vert T\right\vert $, so that $r_{M}\left(
T\right)  =\left\vert Z\right\vert \leq\left\vert T\right\vert $. Combining
this with $\left\vert T\right\vert \leq r_{M}\left(  T\right)  $, we obtain
$r_{M}\left(  T\right)  =\left\vert T\right\vert $. This proves Lemma
\ref{lem.matroid.rI}.
\end{proof}

\begin{lemma}
\label{lem.matroid.Cin}Let $M=\left(  E,\mathcal{I}\right)  $ be a matroid.
Let $Q\in\mathcal{P}\left(  E\right)  \setminus\mathcal{I}$. Then, there
exists a circuit $C$ of $M$ such that $C\subseteq Q$.
\end{lemma}

\begin{proof}
[Proof of Lemma \ref{lem.matroid.Cin}.]We have $Q\in\mathcal{P}\left(
E\right)  \setminus\mathcal{I}$. Thus, there exists at least one
$C\in\mathcal{P}\left(  E\right)  \setminus\mathcal{I}$ such that $C\subseteq
Q$ (namely, $C=Q$). Thus, there also exists a \textbf{minimal} such $C$.
Consider this minimal $C$. We know that $C$ is a minimal element of
$\mathcal{P}\left(  E\right)  \setminus\mathcal{I}$ such that $C\subseteq Q$.
In other words, $C$ is an element of $\mathcal{P}\left(  E\right)
\setminus\mathcal{I}$ satisfying $C\subseteq Q$, and moreover,%
\begin{equation}
\text{every }D\in\mathcal{P}\left(  E\right)  \setminus\mathcal{I}\text{
satisfying }D\subseteq Q\text{ and }D\subseteq C\text{ must satisfy }D=C.
\label{pf.lem.matroid.Cin.1}%
\end{equation}
Thus, $C$ is a minimal element of $\mathcal{P}\left(  E\right)  \setminus
\mathcal{I}$\ \ \ \ \footnote{\textit{Proof.} We need to show that every
$D\in\mathcal{P}\left(  E\right)  \setminus\mathcal{I}$ satisfying $D\subseteq
C$ must satisfy $D=C$ (since we already know that $C\in\mathcal{P}\left(
E\right)  \setminus\mathcal{I}$).
\par
So let $D\in\mathcal{P}\left(  E\right)  \setminus\mathcal{I}$ be such that
$D\subseteq C$. Then, $D\subseteq C\subseteq Q$. Hence,
(\ref{pf.lem.matroid.Cin.1}) shows that $D=C$. This completes our proof.}. In
other words, $C$ is a circuit of $M$ (by the definition of a \textquotedblleft
circuit\textquotedblright). This circuit $C$ satisfies $C\subseteq Q$. Thus,
we have constructed a circuit $C$ of $M$ satisfying $C\subseteq Q$. Lemma
\ref{lem.matroid.Cin} is thus proven.
\end{proof}

\begin{lemma}
\label{lem.matroid.basext}Let $M=\left(  E,\mathcal{I}\right)  $ be a matroid.
Let $T$ be a subset of $E$. Let $S\in\mathcal{I}$ be such that $S\subseteq T$.
Then, there exists an $S^{\prime}\in\mathcal{I}$ satisfying $S\subseteq
S^{\prime}\subseteq T$ and $\left\vert S^{\prime}\right\vert =r_{M}\left(
T\right)  $.
\end{lemma}

\begin{proof}
[Proof of Lemma \ref{lem.matroid.basext}.]Clearly, there exists at least one
$S^{\prime}\in\mathcal{I}$ satisfying $S\subseteq S^{\prime}\subseteq T$
(namely, $S^{\prime}=S$). Hence, there exists a \textbf{maximal} such
$S^{\prime}$. Let $Q$ be such a maximal $S^{\prime}$. Thus, $Q$ is an element
of $\mathcal{I}$ satisfying $S\subseteq Q\subseteq T$.

Recall that
\begin{align*}
r_{M}\left(  T\right)   &  =\max\left\{  \left\vert Z\right\vert \ \mid
\ Z\in\mathcal{I}\text{ and }Z\subseteq T\right\}  \ \ \ \ \ \ \ \ \ \ \left(
\text{by the definition of }r_{M}\right) \\
&  \in\left\{  \left\vert Z\right\vert \ \mid\ Z\in\mathcal{I}\text{ and
}Z\subseteq T\right\}
\end{align*}
(since the maximum of any set must belong to this set). Hence, there exists
some $Z\in\mathcal{I}$ satisfying $Z\subseteq T$ and $r_{M}\left(  T\right)
=\left\vert Z\right\vert $. Denote such a $Z$ by $W$. Thus, $W$ is an element
of $\mathcal{I}$ satisfying $W\subseteq T$ and $r_{M}\left(  T\right)
=\left\vert W\right\vert $.

We have $\left\vert Q\right\vert \in\left\{  \left\vert Z\right\vert
\ \mid\ Z\in\mathcal{I}\text{ and }Z\subseteq T\right\}  $ (since
$Q\in\mathcal{I}$ and $Q\subseteq T$). Since any element of a set is smaller
or equal to the maximum of this set, this entails that $\left\vert
Q\right\vert \leq\max\left\{  \left\vert Z\right\vert \ \mid\ Z\in
\mathcal{I}\text{ and }Z\subseteq T\right\}  =r_{M}\left(  T\right)
=\left\vert W\right\vert $.

Now, assume (for the sake of contradiction) that $\left\vert Q\right\vert
\neq\left\vert W\right\vert $. Thus, $\left\vert Q\right\vert <\left\vert
W\right\vert $ (since $\left\vert Q\right\vert \leq\left\vert W\right\vert $).
Hence, Matroid axiom 3 (applied to $Y=Q$ and $Z=W$) shows that there exists
some $x\in W\setminus Q$ such that $Q\cup\left\{  x\right\}  \in\mathcal{I}$.
Consider this $x$. We have $x\in W\setminus Q\subseteq W\subseteq T$, so that
$Q\cup\left\{  x\right\}  \subseteq T$ (since $Q\subseteq T$). Also, $x\notin
Q$ (since $x\in W\setminus Q$).

Recall that $Q$ is a \textbf{maximal} $S^{\prime}\in\mathcal{I}$ satisfying
$S\subseteq S^{\prime}\subseteq T$. Thus, if some $S^{\prime}\in\mathcal{I}$
satisfies $S\subseteq S^{\prime}\subseteq T$ and $S^{\prime}\supseteq Q$, then
$S^{\prime}=Q$. Applying this to $S^{\prime}=Q\cup\left\{  x\right\}  $, we
obtain $Q\cup\left\{  x\right\}  =Q$ (since $S\subseteq Q\subseteq
Q\cup\left\{  x\right\}  \subseteq T$ and $Q\cup\left\{  x\right\}  \supseteq
Q$). Thus, $x\in Q$. But this contradicts $x\notin Q$. This contradiction
shows that our assumption (that $\left\vert Q\right\vert \neq\left\vert
W\right\vert $) was wrong. Hence, $\left\vert Q\right\vert =\left\vert
W\right\vert =r_{M}\left(  T\right)  $. Thus, there exists an $S^{\prime}%
\in\mathcal{I}$ satisfying $S\subseteq S^{\prime}\subseteq T$ and $\left\vert
S^{\prime}\right\vert =r_{M}\left(  T\right)  $ (namely, $S^{\prime}=Q$). This
proves Lemma \ref{lem.matroid.basext}.
\end{proof}

\subsection{The lattice of flats}

We shall now show a lemma that can be regarded as an alternative criterion for
a subset of $E$ to be a flat:

\begin{lemma}
\label{lem.matroid.flat-crit}Let $M=\left(  E,\mathcal{I}\right)  $ be a
matroid. Let $T$ be a subset of $E$. Then, the following statements are equivalent:

\textit{Statement }$\mathfrak{F}_{1}$\textit{:} The set $T$ is a flat of $M$.

\textit{Statement }$\mathfrak{F}_{2}$\textit{:} If $C$ is a circuit of $M$,
and if $e\in C$ is such that $C\setminus\left\{  e\right\}  \subseteq T$, then
$C\subseteq T$.
\end{lemma}

\begin{proof}
[Proof of Lemma \ref{lem.matroid.flat-crit}.]\textit{Proof of the implication
$\mathfrak{F}_{1}\Longrightarrow\mathfrak{F}_{2}$:} Assume that Statement
$\mathfrak{F}_{1}$ holds. We must prove that Statement $\mathfrak{F}_{2}$ holds.

Let $C$ be a circuit of $M$. Let $e\in C$ be such that $C\setminus\left\{
e\right\}  \subseteq T$. We must prove that $C\subseteq T$.

Assume the contrary. Thus, $C\not \subseteq T$. Combining this with
$C\setminus\left\{  e\right\}  \subseteq T$, we obtain $e\notin T$. Hence, $T$
is a proper subset of $T\cup\left\{  e\right\}  $.

We have assumed that Statement $\mathfrak{F}_{1}$ holds. In other words, the
set $T$ is a flat of $M$. In other words, there exists some $k\in\mathbb{N}$
such that $T$ is a $k$-flat of $M$. Consider this $k$.

The set $T$ is a $k$-flat of $M$, thus a subset of $E$ that has rank $k$ and
is maximal among all such subsets. In other words, $r_{M}\left(  T\right)
=k$, but every subset $S$ of $E$ for which $T$ is a proper subset of $S$ must
satisfy%
\begin{equation}
r_{M}\left(  S\right)  \neq k. \label{pf.lem.matroid.flat-crit.1.1}%
\end{equation}
Applying (\ref{pf.lem.matroid.flat-crit.1.1}) to $S=T\cup\left\{  e\right\}
$, we obtain $r_{M}\left(  T\cup\left\{  e\right\}  \right)  \neq k$. Since
$T\cup\left\{  e\right\}  \supseteq T$ (and since the function $r_{M}%
:\mathcal{P}\left(  E\right)  \rightarrow\mathbb{N}$ is weakly increasing), we
have $r_{M}\left(  T\cup\left\{  e\right\}  \right)  \geq r_{M}\left(
T\right)  =k$. Combined with $r_{M}\left(  T\cup\left\{  e\right\}  \right)
\neq k$, this yields $r_{M}\left(  T\cup\left\{  e\right\}  \right)
>k=r_{M}\left(  T\right)  $.

Notice that $C\setminus\left\{  e\right\}  $ is a proper subset of $C$ (since
$e\in C$). The set $C$ is a circuit of $M$, thus a minimal element of
$\mathcal{P}\left(  E\right)  \setminus\mathcal{I}$ (by the definition of a
\textquotedblleft circuit\textquotedblright). Hence, no proper subset of $C$
belongs to $\mathcal{P}\left(  E\right)  \setminus\mathcal{I}$ (because $C$ is
minimal). In other words, every proper subset of $C$ belongs to $\mathcal{I}$.
Applying this to the proper subset $C\setminus\left\{  e\right\}  $ of $C$, we
conclude that $C\setminus\left\{  e\right\}  $ belongs to $\mathcal{I}$.
Hence, Lemma \ref{lem.matroid.basext} (applied to $S=C\setminus\left\{
e\right\}  $) shows that there exists an $S^{\prime}\in\mathcal{I}$ satisfying
$C\setminus\left\{  e\right\}  \subseteq S^{\prime}\subseteq T$ and
$\left\vert S^{\prime}\right\vert =r_{M}\left(  T\right)  $. Denote this
$S^{\prime}$ by $S$. Thus, $S$ is an element of $\mathcal{I}$ satisfying
$C\setminus\left\{  e\right\}  \subseteq S\subseteq T$ and $\left\vert
S\right\vert =r_{M}\left(  T\right)  $.

Furthermore, $S\subseteq T\subseteq T\cup\left\{  e\right\}  $. Thus, Lemma
\ref{lem.matroid.basext} (applied to $T\cup\left\{  e\right\}  $ instead of
$T$) shows that there exists an $S^{\prime}\in\mathcal{I}$ satisfying
$S\subseteq S^{\prime}\subseteq T\cup\left\{  e\right\}  $ and $\left\vert
S^{\prime}\right\vert =r_{M}\left(  T\cup\left\{  e\right\}  \right)  $.
Consider this $S^{\prime}$.

We have $\left\vert S^{\prime}\right\vert =r_{M}\left(  T\cup\left\{
e\right\}  \right)  >r_{M}\left(  T\right)  $. Hence, $S^{\prime
}\not \subseteq T$\ \ \ \ \footnote{\textit{Proof.} Assume the contrary. Thus,
$S^{\prime}\subseteq T$. Hence, $S^{\prime}$ is an element of $\mathcal{I}$
and satisfies $S^{\prime}\subseteq T$. Thus, $\left\vert S^{\prime}\right\vert
\in\left\{  \left\vert Z\right\vert \ \mid\ Z\in\mathcal{I}\text{ and
}Z\subseteq T\right\}  $.
\par
Now, the definition of $r_{M}$ yields%
\[
r_{M}\left(  T\right)  =\max\left\{  \left\vert Z\right\vert \ \mid
\ Z\in\mathcal{I}\text{ and }Z\subseteq T\right\}  \geq\left\vert S^{\prime
}\right\vert
\]
(since $\left\vert S^{\prime}\right\vert \in\left\{  \left\vert Z\right\vert
\ \mid\ Z\in\mathcal{I}\text{ and }Z\subseteq T\right\}  $). This contradicts
$\left\vert S^{\prime}\right\vert >r_{M}\left(  T\right)  $. This
contradiction proves that our assumption was wrong, qed.}. Combining this with
$S^{\prime}\subseteq T\cup\left\{  e\right\}  $, we obtain $e\in S^{\prime}$.
Combining this with $C\setminus\left\{  e\right\}  \subseteq S\subseteq
S^{\prime}$, we find that $\left(  C\setminus\left\{  e\right\}  \right)
\cup\left\{  e\right\}  \subseteq S^{\prime}$. Thus, $C=\left(  C\setminus
\left\{  e\right\}  \right)  \cup\left\{  e\right\}  \subseteq S^{\prime}$.
Since $S^{\prime}\in\mathcal{I}$, this entails that $C\in\mathcal{I}$ (by
Matroid axiom 2). But $C\in\mathcal{P}\left(  E\right)  \setminus\mathcal{I}$
(since $C$ is a minimal element of $\mathcal{P}\left(  E\right)
\setminus\mathcal{I}$), so that $C\notin\mathcal{I}$. This contradicts
$C\in\mathcal{I}$. This contradiction shows that our assumption was wrong.
Hence, $C\subseteq T$ is proven. Therefore, Statement $\mathfrak{F}_{2}$
holds. Thus, the implication $\mathfrak{F}_{1}\Longrightarrow\mathfrak{F}_{2}$
is proven.

\textit{Proof of the implication $\mathfrak{F}_{2}\Longrightarrow
\mathfrak{F}_{1}$:} Assume that Statement $\mathfrak{F}_{2}$ holds. We must
prove that Statement $\mathfrak{F}_{1}$ holds.

Let $k=r_{M}\left(  T\right)  $. We shall show that $T$ is a $k$-flat of $M$.

Let $W$ be a subset of $E$ that has rank $k$ and satisfies $T\subseteq W$. We
shall show that $T=W$.

Indeed, assume the contrary. Thus, $T\neq W$. Combined with $T\subseteq W$,
this shows that $T$ is a proper subset of $W$. Thus, there exists an $e\in
W\setminus T$. Consider this $e$. We have $e\notin T$ (since $e\in W\setminus
T$).

We have
\begin{align*}
k  &  =r_{M}\left(  T\right)  =\max\left\{  \left\vert Z\right\vert
\ \mid\ Z\in\mathcal{I}\text{ and }Z\subseteq T\right\}
\ \ \ \ \ \ \ \ \ \ \left(  \text{by the definition of }r_{M}\right) \\
&  \in\left\{  \left\vert Z\right\vert \ \mid\ Z\in\mathcal{I}\text{ and
}Z\subseteq T\right\}
\end{align*}
(since the maximum of a set must belong to that set). Hence, there exists some
$Z\in\mathcal{I}$ satisfying $Z\subseteq T$ and $k=\left\vert Z\right\vert $.
Denote this $Z$ by $K$. Thus, $K$ is an element of $\mathcal{I}$ and satisfies
$K\subseteq T$ and $k=\left\vert K\right\vert $. Notice that $e\notin T$, so
that $e\notin K$ (since $K\subseteq T$).

We have $r_{M}\left(  W\right)  =k$ (since $W$ has rank $k$). Hence,
$K\cup\left\{  e\right\}  \notin\mathcal{I}$\ \ \ \ \footnote{\textit{Proof.}
Assume the contrary. Thus, $K\cup\left\{  e\right\}  \in\mathcal{I}$. Thus,
$r_{M}\left(  K\cup\left\{  e\right\}  \right)  =\left\vert K\cup\left\{
e\right\}  \right\vert $ (by Lemma \ref{lem.matroid.rI}). Thus, $r_{M}\left(
K\cup\left\{  e\right\}  \right)  =\left\vert K\cup\left\{  e\right\}
\right\vert >\left\vert K\right\vert $ (since $e\notin K$).
\par
But $K\cup\left\{  e\right\}  \subseteq W$ (since $K\subseteq T\subseteq W$
and $e\in W\setminus T\subseteq W$). Since the function $r_{M}$ is weakly
increasing, this yields $r_{M}\left(  K\cup\left\{  e\right\}  \right)  \leq
r_{M}\left(  W\right)  =k=\left\vert K\right\vert $. This contradicts
$r_{M}\left(  K\cup\left\{  e\right\}  \right)  >\left\vert K\right\vert $.
This contradiction proves that our assumption was wrong, qed.}. In other
words, $K\cup\left\{  e\right\}  \in\mathcal{P}\left(  E\right)
\setminus\mathcal{I}$. Hence, Lemma \ref{lem.matroid.Cin} (applied to
$Q=K\cup\left\{  e\right\}  $) shows that there exists a circuit $C$ of $M$
such that $C\subseteq K\cup\left\{  e\right\}  $. Consider this $C$. From
$C\subseteq K\cup\left\{  e\right\}  $, we obtain $C\setminus\left\{
e\right\}  \subseteq K\subseteq T$.

From $C\setminus\left\{  e\right\}  \subseteq K$, we conclude (using Matroid
axiom 2) that $C\setminus\left\{  e\right\}  \in\mathcal{I}$ (since
$K\in\mathcal{I}$). On the other hand, $C$ is a circuit of $M$. In other
words, $C$ is a minimal element of $\mathcal{P}\left(  E\right)
\setminus\mathcal{I}$ (by the definition of a \textquotedblleft
circuit\textquotedblright). Hence, $C\in\mathcal{P}\left(  E\right)
\setminus\mathcal{I}$, so that $C\notin\mathcal{I}$. Hence, $e\in C$ (since
otherwise, we would have $C\setminus\left\{  e\right\}  =C\notin\mathcal{I}$,
which would contradict $C\setminus\left\{  e\right\}  \in\mathcal{I}$). Now,
Statement $\mathfrak{F}_{2}$ shows that $C\subseteq T$. Hence, $e\in
C\subseteq T$, which contradicts $e\notin T$.

This contradiction shows that our assumption was wrong. Hence, $T=W$ is proven.

Now, forget that we fixed $W$. Thus, we have shown that if $W$ is a subset of
$E$ that has rank $k$ and satisfies $T\subseteq W$, then $T=W$. In other
words, $T$ is a subset of $E$ that has rank $k$ and is maximal among all such
subsets (because we already know that $T$ has rank $r_{M}\left(  T\right)
=k$). In other words, $T$ is a $k$-flat of $M$ (by the definition of a
\textquotedblleft$k$-flat\textquotedblright). Thus, $T$ is a flat of $M$. In
other words, Statement $\mathfrak{F}_{1}$ holds. This proves the implication
$\mathfrak{F}_{2}\Longrightarrow\mathfrak{F}_{1}$.

We have now proven the implications $\mathfrak{F}_{1}\Longrightarrow
\mathfrak{F}_{2}$ and $\mathfrak{F}_{2}\Longrightarrow\mathfrak{F}_{1}$.
Together, these implications show that Statements $\mathfrak{F}_{1}$ and
$\mathfrak{F}_{2}$ are equivalent. This proves Lemma
\ref{lem.matroid.flat-crit}.
\end{proof}

\begin{corollary}
\label{cor.matroid.flat-cap}Let $M=\left(  E,\mathcal{I}\right)  $ be a
matroid. Let $F_{1},F_{2},\ldots,F_{k}$ be flats of $M$. Then, $F_{1}\cap
F_{2}\cap\cdots\cap F_{k}$ is a flat of $M$. (Notice that $k$ is allowed to be
$0$ here; in this case, the empty intersection $F_{1}\cap F_{2}\cap\cdots\cap
F_{k}$ is to be interpreted as $E$.)
\end{corollary}

\begin{proof}
[Proof of Corollary \ref{cor.matroid.flat-cap}.]Lemma
\ref{lem.matroid.flat-crit} gives a necessary and sufficient criterion for a
subset $T$ of $E$ to be a flat of $M$. It is easy to see that if this
criterion is satisfied for $T=F_{1}$, for $T=F_{2}$, etc., and for $T=F_{k}$,
then it is satisfied for $T=F_{1}\cap F_{2}\cap\cdots\cap F_{k}$. In other
words, if $F_{1},F_{2},\ldots,F_{k}$ are flats of $M$, then $F_{1}\cap
F_{2}\cap\cdots\cap F_{k}$ is a flat of $M$.\ \ \ \ \footnote{Here is this
argument in slightly more detail:
\par
For every $i\in\left\{  1,2,\ldots,k\right\}  $, the following statement
holds: If $C$ is a circuit of $M$, and if $e\in C$ is such that $C\setminus
\left\{  e\right\}  \subseteq F_{i}$, then%
\begin{equation}
C\subseteq F_{i}. \label{pf.cor.matroid.flat-cap.fn1.1}%
\end{equation}
\par
\textit{Proof of (\ref{pf.cor.matroid.flat-cap.fn1.1}):} Let $i\in\left\{
1,2,\ldots,k\right\}  $. Then, the set $F_{i}$ is a flat of $M$. In other
words, Statement $\mathfrak{F}_{1}$ of Lemma \ref{lem.matroid.flat-crit} is
satisfied for $T=F_{i}$. Therefore, Statement $\mathfrak{F}_{2}$ of Lemma
\ref{lem.matroid.flat-crit} must also be satisfied for $T=F_{i}$ (since Lemma
\ref{lem.matroid.flat-crit} shows that the Statements $\mathfrak{F}_{1}$ and
$\mathfrak{F}_{2}$ are equivalent). In other words, if $C$ is a circuit of
$M$, and if $e\in C$ is such that $C\setminus\left\{  e\right\}  \subseteq
F_{i}$, then $C\subseteq F_{i}$. This proves
(\ref{pf.cor.matroid.flat-cap.fn1.1}).
\par
Now, let $C$ be a circuit of $M$, and let $e\in C$ be such that $C\setminus
\left\{  e\right\}  \subseteq F_{1}\cap F_{2}\cap\cdots\cap F_{k}$. For every
$i\in\left\{  1,2,\ldots,k\right\}  $, we have $C\setminus\left\{  e\right\}
\subseteq F_{1}\cap F_{2}\cap\cdots\cap F_{k}\subseteq F_{i}$, and therefore
$C\subseteq F_{i}$ (by (\ref{pf.cor.matroid.flat-cap.fn1.1})). So we have
shown the inclusion $C\subseteq F_{i}$ for each $i\in\left\{  1,2,\ldots
,k\right\}  $. Combining these $k$ inclusions, we obtain $C\subseteq F_{1}\cap
F_{2}\cap\cdots\cap F_{k}$.
\par
Now, forget that we fixed $C$. We thus have shown that if $C$ is a circuit of
$M$, and if $e\in C$ is such that $C\setminus\left\{  e\right\}  \subseteq
F_{1}\cap F_{2}\cap\cdots\cap F_{k}$, then $C\subseteq F_{1}\cap F_{2}%
\cap\cdots\cap F_{k}$. In other words, Statement $\mathfrak{F}_{2}$ of Lemma
\ref{lem.matroid.flat-crit} is satisfied for $T=F_{1}\cap F_{2}\cap\cdots\cap
F_{k}$. Therefore, Statement $\mathfrak{F}_{1}$ of Lemma
\ref{lem.matroid.flat-crit} must also be satisfied for $T=F_{1}\cap F_{2}%
\cap\cdots\cap F_{k}$ (since Lemma \ref{lem.matroid.flat-crit} shows that the
Statements $\mathfrak{F}_{1}$ and $\mathfrak{F}_{2}$ are equivalent). In other
words, the set $F_{1}\cap F_{2}\cap\cdots\cap F_{k}$ is a flat of $M$. Qed.}
This proves Corollary \ref{cor.matroid.flat-cap}.
\end{proof}

Corollary \ref{cor.matroid.flat-cap} (a well-known fact, which is left to the
reader to prove in \cite[\S 3.1]{Stanley}) allows us to define the
\emph{closure} of a set in a matroid:

\begin{definition}
\label{def.matroid.closure}Let $M=\left(  E,\mathcal{I}\right)  $ be a
matroid. Let $T$ be a subset of $E$. The \emph{closure} of $T$ is defined to
be the intersection of all flats of $M$ which contain $T$ as a subset. In
other words, the closure of $T$ is defined to be $\bigcap_{\substack{F\in
\operatorname*{Flats}M;\\T\subseteq F}}F$. The closure of $T$ is denoted by
$\overline{T}$.
\end{definition}

The following proposition gathers some simple properties of closures in matroids:

\begin{proposition}
\label{prop.matroid.closure.props}Let $M=\left(  E,\mathcal{I}\right)  $ be a matroid.

\textbf{(a)} If $T$ is a subset of $E$, then $\overline{T}$ is a flat of $M$
satisfying $T\subseteq\overline{T}$.

\textbf{(b)} If $G$ is a flat of $M$, then $\overline{G}=G$.

\textbf{(c)} If $T$ is a subset of $E$ and if $G$ is a flat of $M$ satisfying
$T\subseteq G$, then $\overline{T}\subseteq G$.

\textbf{(d)} If $S$ and $T$ are two subsets of $E$ satisfying $S\subseteq T$,
then $\overline{S}\subseteq\overline{T}$.

\textbf{(e)} If the matroid $M$ is loopless, then $\overline{\varnothing
}=\varnothing$.

\textbf{(f)} Every subset $T$ of $E$ satisfies $r_{M}\left(  T\right)
=r_{M}\left(  \overline{T}\right)  $.

\textbf{(g)} If $T$ is a subset of $E$ and if $G$ is a flat of $M$, then the
conditions $\left(  \overline{T}\subseteq G\right)  $ and $\left(  T\subseteq
G\right)  $ are equivalent.
\end{proposition}

\begin{proof}
[Proof of Proposition \ref{prop.matroid.closure.props}.]\textbf{(a)} The set
$\operatorname*{Flats}M$ is a subset of the finite set $\mathcal{P}\left(
E\right)  $, and thus itself finite.

Let $T$ be a subset of $E$. The closure $\overline{T}$ of $T$ is defined as
$\bigcap_{\substack{F\in\operatorname*{Flats}M;\\T\subseteq F}}F$. Now,
Corollary \ref{cor.matroid.flat-cap} shows that any intersection of finitely
many flats of $M$ is a flat of $M$. Hence, $\bigcap_{\substack{F\in
\operatorname*{Flats}M;\\T\subseteq F}}F$ (being an intersection of finitely
many flats of $M$\ \ \ \ \footnote{\textquotedblleft Finitely
many\textquotedblright\ since the set $\operatorname*{Flats}M$ is finite.}) is
a flat of $M$. In other words, $\overline{T}$ is a flat of $M$ (since
$\overline{T}=\bigcap_{\substack{F\in\operatorname*{Flats}M;\\T\subseteq F}}F$).

Also, $T\subseteq F$ for every $F\in\operatorname*{Flats}M$ satisfying
$T\subseteq F$. Hence, $T\subseteq\bigcap_{\substack{F\in\operatorname*{Flats}%
M;\\T\subseteq F}}F=\overline{T}$. This completes the proof of Proposition
\ref{prop.matroid.closure.props} \textbf{(a)}.

\textbf{(c)} Let $T$ be a subset of $E$, and let $G$ be a flat of $M$
satisfying $T\subseteq G$. Then, $G$ is an element of $\operatorname*{Flats}M$
satisfying $T\subseteq G$. Hence, $G$ is one term in the intersection
$\bigcap_{\substack{F\in\operatorname*{Flats}M;\\T\subseteq F}}F$. Thus,
$\bigcap_{\substack{F\in\operatorname*{Flats}M;\\T\subseteq F}}F\subseteq G$.
But the definition of $\overline{T}$ yields $\overline{T}=\bigcap
_{\substack{F\in\operatorname*{Flats}M;\\T\subseteq F}}F\subseteq G$. This
proves Proposition \ref{prop.matroid.closure.props} \textbf{(c)}.

\textbf{(b)} Let $G$ be a flat of $M$. Proposition
\ref{prop.matroid.closure.props} \textbf{(c)} (applied to $T=G$) yields
$\overline{G}\subseteq G$ (since $G\subseteq G$). But Proposition
\ref{prop.matroid.closure.props} \textbf{(a)} (applied to $T=G$) shows that
$\overline{G}$ is a flat of $M$ satisfying $G\subseteq\overline{G}$. Combining
$G\subseteq\overline{G}$ with $\overline{G}\subseteq G$, we obtain
$\overline{G}=G$. This proves Proposition \ref{prop.matroid.closure.props}
\textbf{(b)}.

\textbf{(d)} Let $S$ and $T$ be two subsets of $E$ satisfying $S\subseteq T$.
Proposition \ref{prop.matroid.closure.props} \textbf{(a)} shows that
$\overline{T}$ is a flat of $M$ satisfying $T\subseteq\overline{T}$. Now,
$S\subseteq T\subseteq\overline{T}$. Hence, Proposition
\ref{prop.matroid.closure.props} \textbf{(c)} (applied to $S$ and
$\overline{T}$ instead of $T$ and $G$) shows $\overline{S}\subseteq
\overline{T}$. This proves Proposition \ref{prop.matroid.closure.props}
\textbf{(d)}.

\textbf{(e)} Assume that the matroid $M$ is loopless. In other words, no loops
(of $M$) exist.

The definition of $r_{M}$ quickly yields $r_{M}\left(  \varnothing\right)
=0$. In other words, the set $\varnothing$ has rank $0$. We shall now show
that $\varnothing$ is a $0$-flat of $M$.

Indeed, let $W$ be a subset of $E$ that has rank $0$ and satisfies
$\varnothing\subseteq W$. We shall show that $\varnothing=W$.

Assume the contrary. Thus, $\varnothing\neq W$. Hence, $W$ has an element $w$.
Consider this $w$. The element $w$ of $E$ is not a loop (since no loops
exist). In other words, we cannot have $\left\{  w\right\}  \notin\mathcal{I}$
(since $w$ is a loop if and only if $\left\{  w\right\}  \notin\mathcal{I}$
(by the definition of a loop)). In other words, we must have $\left\{
w\right\}  \in\mathcal{I}$. Clearly, $\left\{  w\right\}  \subseteq W$ (since
$w\in W$). Thus, $\left\{  w\right\}  $ is a $Z\in\mathcal{I}$ satisfying
$Z\subseteq W$. Thus, $\left\vert \left\{  w\right\}  \right\vert \in\left\{
\left\vert Z\right\vert \ \mid\ Z\in\mathcal{I}\text{ and }Z\subseteq
W\right\}  $.

But $W$ has rank $0$. In other words,%
\begin{align*}
0  &  =r_{M}\left(  W\right)  =\max\left\{  \left\vert Z\right\vert
\ \mid\ Z\in\mathcal{I}\text{ and }Z\subseteq W\right\}
\ \ \ \ \ \ \ \ \ \ \left(  \text{by the definition of }r_{M}\right) \\
&  \geq\left\vert \left\{  w\right\}  \right\vert \ \ \ \ \ \ \ \ \ \ \left(
\text{since }\left\vert \left\{  w\right\}  \right\vert \in\left\{  \left\vert
Z\right\vert \ \mid\ Z\in\mathcal{I}\text{ and }Z\subseteq W\right\}  \right)
\\
&  =1,
\end{align*}
which is absurd. This contradiction shows that our assumption was wrong.
Hence, $\varnothing=W$ is proven.

Let us now forget that we fixed $W$. We thus have proven that if $W$ is any
subset of $E$ that has rank $0$ and satisfies $\varnothing\subseteq W$, then
$\varnothing=W$. Thus, $\varnothing$ is a subset of $E$ that has rank $0$ and
is maximal among all such subsets (because we already know that $\varnothing$
has rank $0$). In other words, $\varnothing$ is a $0$-flat of $M$ (by the
definition of a \textquotedblleft$0$-flat\textquotedblright). Thus,
$\varnothing$ is a flat of $M$. Thus, Proposition
\ref{prop.matroid.closure.props} \textbf{(b)} (applied to $G=\varnothing$)
yields $\overline{\varnothing}=\varnothing$. This proves Proposition
\ref{prop.matroid.closure.props} \textbf{(e)}.

\textbf{(f)} Let $T$ be a subset of $E$. We have $T\subseteq\overline{T}$ (by
Proposition \ref{prop.matroid.closure.props} \textbf{(a)}), and thus
$r_{M}\left(  T\right)  \leq r_{M}\left(  \overline{T}\right)  $ (since the
function $r_{M}$ is weakly increasing).

Let $k=r_{M}\left(  T\right)  $. Thus, there exists a $Q\in\mathcal{P}\left(
E\right)  $ satisfying $T\subseteq Q$ and $k=r_{M}\left(  Q\right)  $ (namely,
$Q=T$). Hence, there exists a \textbf{maximal} such $Q$. Denote this $Q$ by
$R$. Thus, $R$ is a maximal $Q\in\mathcal{P}\left(  E\right)  $ satisfying
$T\subseteq Q$ and $k=r_{M}\left(  Q\right)  $. In particular, $R$ is an
element of $\mathcal{P}\left(  E\right)  $ and satisfies $T\subseteq R$ and
$k=r_{M}\left(  R\right)  $.

Now, $R$ is a subset of $E$ (since $R\in\mathcal{P}\left(  E\right)  $) and
has rank $r_{M}\left(  R\right)  =k$. Thus, $R$ is a subset of $E$ that has
rank $k$. Furthermore, $R$ is maximal among all such
subsets\footnote{\textit{Proof.} Let $W$ be any subset of $E$ that has rank
$k$ and satisfies $W\supseteq R$. We must prove that $W=R$.
\par
We have $W\in\mathcal{P}\left(  E\right)  $, $T\subseteq R\subseteq W$ and
$k=r_{M}\left(  W\right)  $ (since $W$ has rank $k$). Thus, $W$ is a
$Q\in\mathcal{P}\left(  E\right)  $ satisfying $T\subseteq Q$ and
$k=r_{M}\left(  Q\right)  $. But recall that $R$ is a \textbf{maximal} such
$Q$. Hence, if $W\supseteq R$, then $W=R$. Therefore, $W=R$ (since we know
that $W\supseteq R$). Qed.}. Thus, $R$ is a $k$-flat of $M$ (by the definition
of a \textquotedblleft$k$-flat\textquotedblright), and therefore a flat of
$M$. Now, Proposition \ref{prop.matroid.closure.props} \textbf{(c)} (applied
to $G=R$) shows that $\overline{T}\subseteq R$. Since the function $r_{M}$ is
weakly increasing, this yields $r_{M}\left(  \overline{T}\right)  \leq
r_{M}\left(  R\right)  =k$. Combining this with $k=r_{M}\left(  T\right)  \leq
r_{M}\left(  \overline{T}\right)  $, we obtain $r_{M}\left(  \overline
{T}\right)  =k=r_{M}\left(  T\right)  $. This proves Proposition
\ref{prop.matroid.closure.props} \textbf{(f)}.

\textbf{(g)} Let $T$ be a subset of $E$. Let $G$ be a flat of $M$. Proposition
\ref{prop.matroid.closure.props} \textbf{(a)} shows that $T\subseteq
\overline{T}$. Hence, if $\overline{T}\subseteq G$, then $T\subseteq
\overline{T}\subseteq G$. Thus, we have proven the implication $\left(
\overline{T}\subseteq G\right)  \Longrightarrow\left(  T\subseteq G\right)  $.
The reverse implication (i.e., the implication $\left(  T\subseteq G\right)
\Longrightarrow\left(  \overline{T}\subseteq G\right)  $) follows from
Proposition \ref{prop.matroid.closure.props} \textbf{(c)}. Combining these two
implications, we obtain the equivalence $\left(  \overline{T}\subseteq
G\right)  \Longleftrightarrow\left(  T\subseteq G\right)  $. This proves
Proposition \ref{prop.matroid.closure.props} \textbf{(g)}.
\end{proof}

We shall now recall a few more classical notions related to posets:

\begin{definition}
\label{def.lattice}Let $P$ be a poset.

\textbf{(a)} An element $p\in P$ is said to be a \emph{global minimum} of $P$
if every $q\in P$ satisfies $p\leq q$. Clearly, a global minimum of $P$ is
unique if it exists.

\textbf{(b)} An element $p\in P$ is said to be a \emph{global maximum} of $P$
if every $q\in P$ satisfies $p\geq q$. Clearly, a global maximum of $P$ is
unique if it exists.

\textbf{(c)} Let $x$ and $y$ be two elements of $P$. An \emph{upper bound} of
$x$ and $y$ (in $P$) means an element $z\in P$ satisfying $z\geq x$ and $z\geq
y$. A \emph{join} (or \emph{least upper bound}) of $x$ and $y$ (in $P$) means
an upper bound $z$ of $x$ and $y$ such that every upper bound $z^{\prime}$ of
$x$ and $y$ satisfies $z^{\prime}\geq z$. In other words, a join of $x$ and
$y$ is a global minimum of the subposet $\left\{  w\in P\ \mid\ w\geq x\text{
and }w\geq y\right\}  $ of $P$. Thus, a join of $x$ and $y$ is unique if it exists.

\textbf{(d)} Let $x$ and $y$ be two elements of $P$. A \emph{lower bound} of
$x$ and $y$ (in $P$) means an element $z\in P$ satisfying $z\leq x$ and $z\leq
y$. A \emph{meet} (or \emph{greatest lower bound}) of $x$ and $y$ (in $P$)
means a lower bound $z$ of $x$ and $y$ such that every lower bound $z^{\prime
}$ of $x$ and $y$ satisfies $z^{\prime}\leq z$. In other words, a meet of $x$
and $y$ is a global maximum of the subposet $\left\{  w\in P\ \mid\ w\leq
x\text{ and }w\leq y\right\}  $ of $P$. Thus, a meet of $x$ and $y$ is unique
if it exists.

\textbf{(e)} The poset $P$ is said to be a \emph{lattice} if and only if it
has a global minimum and a global maximum, and every two elements of $P$ have
a meet and a join.
\end{definition}

\begin{proposition}
\label{prop.matroid.flat-lat}Let $M=\left(  E,\mathcal{I}\right)  $ be a
matroid. The subposet $\operatorname*{Flats}M$ of the poset $\mathcal{P}%
\left(  E\right)  $ is a lattice.
\end{proposition}

\begin{proof}
[Proof of Proposition \ref{prop.matroid.flat-lat}.]By the definition of a
lattice, it suffices to check the following four claims:

\textit{Claim 1:} The poset $\operatorname*{Flats}M$ has a global minimum.

\textit{Claim 2:} The poset $\operatorname*{Flats}M$ has a global maximum.

\textit{Claim 3:} Every two elements of $\operatorname*{Flats}M$ have a meet
(in $\operatorname*{Flats}M$).

\textit{Claim 4:} Every two elements of $\operatorname*{Flats}M$ have a join
(in $\operatorname*{Flats}M$).

\textit{Proof of Claim 1:} Applying Proposition
\ref{prop.matroid.closure.props} \textbf{(a)} to $T=\varnothing$, we see that
$\overline{\varnothing}$ is a flat of $M$ satisfying $\varnothing
\subseteq\overline{\varnothing}$. In particular, $\overline{\varnothing}$ is a
flat of $M$, so that $\overline{\varnothing}\in\operatorname*{Flats}M$. If $G$
is a flat of $M$, then $\overline{\varnothing}\subseteq G$ (by Proposition
\ref{prop.matroid.closure.props} \textbf{(c)}, applied to $T=\varnothing$).
Hence, $\overline{\varnothing}$ is a global minimum of the poset
$\operatorname*{Flats}M$. Thus, the poset $\operatorname*{Flats}M$ has a
global minimum. This proves Claim 1.

\textit{Proof of Claim 2:} Applying Proposition
\ref{prop.matroid.closure.props} \textbf{(a)} to $T=E$, we see that
$\overline{E}$ is a flat of $M$ satisfying $E\subseteq\overline{E}$. From
$E\subseteq\overline{E}$, we conclude that $\overline{E}=E$. Thus, $E$ is a
flat of $M$ (since $\overline{E}$ is a flat of $M$). In other words,
$E\in\operatorname*{Flats}M$. If $G$ is a flat of $M$, then $E\supseteq G$
(obviously). Hence, $E$ is a global maximum of the poset
$\operatorname*{Flats}M$. Thus, the poset $\operatorname*{Flats}M$ has a
global maximum. This proves Claim 2.

\textit{Proof of Claim 3:} Let $F$ and $G$ be two elements of
$\operatorname*{Flats}M$. We have to prove that $F$ and $G$ have a meet.

We know that $F$ and $G$ are elements of $\operatorname*{Flats}M$, thus flats
of $M$. Hence, Corollary \ref{cor.matroid.flat-cap} shows that $F\cap G$ is a
flat of $M$. In other words, $F\cap G\in\operatorname*{Flats}M$. Clearly,
$F\cap G\subseteq F$ and $F\cap G\subseteq G$; thus, $F\cap G$ is a lower
bound of $F$ and $G$ in $\operatorname*{Flats}M$. Also, every lower bound $H$
of $F$ and $G$ in $\operatorname*{Flats}M$ satisfies $H\subseteq F\cap
G$\ \ \ \ \footnote{\textit{Proof.} Let $H$ be a lower bound of $F$ and $G$ in
$\operatorname*{Flats}M$. Thus, $H\subseteq F$ and $H\subseteq G$. Combining
these two inclusions, we obtain $H\subseteq F\cap G$, qed.}. Hence, $F\cap G$
is a meet of $F$ and $G$. Thus, $F$ and $G$ have a meet. This proves Claim 3.

\textit{Proof of Claim 4:} Let $F$ and $G$ be two elements of
$\operatorname*{Flats}M$. We have to prove that $F$ and $G$ have a join.

We know that $F$ and $G$ are elements of $\operatorname*{Flats}M$, thus flats
of $M$. Proposition \ref{prop.matroid.closure.props} \textbf{(a)} (applied to
$T=F\cup G$) shows that $\overline{F\cup G}$ is a flat of $M$ satisfying
$F\cup G\subseteq\overline{F\cup G}$. Now, $\overline{F\cup G}\in
\operatorname*{Flats}M$ (since $\overline{F\cup G}$ is a flat of $M$).
Clearly, $F\subseteq F\cup G\subseteq\overline{F\cup G}$ and $G\subseteq F\cup
G\subseteq\overline{F\cup G}$; thus, $\overline{F\cup G}$ is an upper bound of
$F$ and $G$ in $\operatorname*{Flats}M$. Also, every upper bound $H$ of $F$
and $G$ in $\operatorname*{Flats}M$ satisfies $H\supseteq\overline{F\cup G}%
$\ \ \ \ \footnote{\textit{Proof.} Let $H$ be an upper bound of $F$ and $G$ in
$\operatorname*{Flats}M$. Thus, $H\supseteq F$ and $H\supseteq G$. Combining
these two inclusions, we obtain $H\supseteq F\cup G$. But $H\in
\operatorname*{Flats}M$; thus, $H$ is a flat of $M$. Since $H$ satisfies
$F\cup G\subseteq H$, we therefore obtain $\overline{F\cup G}\subseteq H$ (by
Proposition \ref{prop.matroid.closure.props} \textbf{(c)}, applied to $F\cup
G$ and $H$ instead of $T$ and $G$). In other words, $H\supseteq\overline{F\cup
G}$, qed.}. Hence, $\overline{F\cup G}$ is a join of $F$ and $G$. Thus, $F$
and $G$ have a join. This proves Claim 4.

We have now proven all four Claims 1, 2, 3, and 4. Thus, Proposition
\ref{prop.matroid.flat-lat} is proven.
\end{proof}

\begin{definition}
\label{def.matroid.flat-lat}Let $M=\left(  E,\mathcal{I}\right)  $ be a
matroid. Proposition \ref{prop.matroid.flat-lat} shows that the subposet
$\operatorname*{Flats}M$ of the poset $\mathcal{P}\left(  E\right)  $ is a
lattice. This subposet $\operatorname*{Flats}M$ is called the \emph{lattice of
flats} of $M$. (Beware: It is a subposet, but not a sublattice of
$\mathcal{P}\left(  E\right)  $, since its join is not a restriction of the
join of $\mathcal{P}\left(  E\right)  $.)
\end{definition}

The lattice of flats $\operatorname*{Flats}M$ of a matroid $M$ is denoted by
$L\left(  M\right)  $ in \cite[\S 3.2]{Stanley}.

Next, we recall the definition of the M\"{o}bius function of a poset (see,
e.g., \cite[Definition 1.2]{Stanley} or \cite[\S 2.2]{Martin22}):

\begin{definition}
\label{def.moebius}Let $P$ be a poset.

\textbf{(a)} If $x$ and $y$ are two elements of $P$ satisfying $x\leq y$, then
the set $\left\{  z\in P\ \mid\ x\leq z\leq y\right\}  $ is denoted by
$\left[  x,y\right]  $.

\textbf{(b)} A subset of $P$ is called a \emph{closed interval} of $P$ if it
has the form $\left[  x,y\right]  $ for two elements $x$ and $y$ of $P$
satisfying $x\leq y$.

\textbf{(c)} We denote by $\operatorname*{Int}P$ the set of all closed
intervals of $P$.

\textbf{(d)} If $f:\operatorname*{Int}P\rightarrow\mathbb{Z}$ is any map, then
the image $f\left(  \left[  x,y\right]  \right)  $ of a closed interval
$\left[  x,y\right]  \in\operatorname*{Int}P$ under $f$ will be abbreviated by
$f\left(  x,y\right)  $.

\textbf{(e)} Assume that every closed interval of $P$ is finite. The
\emph{M\"{o}bius function} of the poset $P$ is defined to be the unique
function $\mu:\operatorname*{Int}P\rightarrow\mathbb{Z}$ having the following
two properties:

\begin{itemize}
\item We have
\begin{equation}
\mu\left(  x,x\right)  =1\ \ \ \ \ \ \ \ \ \ \text{for every }x\in P.
\label{eq.def.moebius.rec1}%
\end{equation}

\item We have%
\begin{align}
\mu\left(  x,y\right)   &  =-\sum_{\substack{z\in P;\\x\leq z<y}}\mu\left(
x,z\right) \label{eq.def.moebius.rec2}\\
&  \ \ \ \ \ \ \ \ \ \ \text{for all }x,y\in P\text{ satisfying }x<y.\nonumber
\end{align}

\end{itemize}

(It is easy to see that these two properties indeed determine $\mu$ uniquely.)
This M\"{o}bius function is denoted by $\mu$.
\end{definition}

We can now define the characteristic polynomial of a matroid $M$, following
\cite[(22)]{Stanley}\footnote{Our notation slightly differs from that in
\cite[(22)]{Stanley}. Namely, we use $x$ as the indeterminate, while Stanley
instead uses $t$. Stanley also denotes the global minimum $\overline
{\varnothing}$ of $\operatorname*{Flats}M$ by $\widehat{0}$.}:

\begin{definition}
\label{def.matroid.charpol}Let $M=\left(  E,\mathcal{I}\right)  $ be a
matroid. Let $m=r_{M}\left(  E\right)  $. The \emph{characteristic polynomial}
$\chi_{M}$ of the matroid $M$ is defined to be the polynomial%
\[
\sum_{F\in\operatorname*{Flats}M}\mu\left(  \overline{\varnothing},F\right)
x^{m-r_{M}\left(  F\right)  }\in\mathbb{Z}\left[  x\right]
\]
(where $\mu$ is the M\"{o}bius function of the lattice $\operatorname*{Flats}%
M$). We further define a polynomial $\widetilde{\chi}_{M}\in\mathbb{Z}\left[
x\right]  $ by $\widetilde{\chi}_{M}=\left[  \overline{\varnothing
}=\varnothing\right]  \chi_{M}$. Here, we are using the Iverson bracket
notation (as in Definition \ref{def.iverson}). If the matroid $M$ is loopless,
then%
\[
\widetilde{\chi}_{M}=\underbrace{\left[  \overline{\varnothing}=\varnothing
\right]  }_{\substack{=1\\\text{(by Proposition
\ref{prop.matroid.closure.props} \textbf{(e)})}}}\chi_{M}=\chi_{M}.
\]

\end{definition}

\begin{example}
\label{exam.matroid.charpol.graph}Let $G=\left(  V,E\right)  $ be a finite
graph. Consider the graphical matroid $\left(  E,\mathcal{I}\right)  $ defined
as in Example \ref{exam.matroid.graphical}. Then, the characteristic
polynomial $\chi_{\left(  E,\mathcal{I}\right)  }$ of this matroid is
connected to the chromatic polynomial $\chi_{G}$ of the graph $G$ as follows:%
\[
x^{\operatorname*{conn}G}\cdot\chi_{\left(  E,\mathcal{I}\right)  }\left(
x\right)  =\chi_{G}\left(  x\right)  .
\]
This equality is a classical result (see, e.g., \cite[Proposition
7.5.1]{Zaslav87}), but can also be derived from our results below
(specifically, by comparing Theorem \ref{thm.matroid.charpol.empty} with
Theorem \ref{thm.chrompol.empty}).
\end{example}

Note that Zaslavsky, in \cite[\S 7.2]{Zaslav87}, defines the \textquotedblleft
characteristic polynomial\textquotedblright\ of a matroid $M$ to be our
$\widetilde{\chi}_{M}$ instead of our $\chi_{M}$; but this makes no difference
when $M$ is the graphical matroid from Example \ref{exam.matroid.graphical},
since such a matroid $M$ is always loopless.

\subsection{Generalized Whitney formulas}

Let us next define broken circuits of a matroid $M=\left(  E,\mathcal{I}%
\right)  $. Stanley, in \cite[\S 4.1]{Stanley}, defines them in terms of a
total ordering $\mathcal{O}$ on the set $E$, whereas we shall use a
\textquotedblleft labeling function\textquotedblright\ $\ell:E\rightarrow X$
instead (as in the case of graphs); our setting is slightly more general than Stanley's.

\begin{definition}
\label{def.matroid.BC}Let $M=\left(  E,\mathcal{I}\right)  $ be a matroid. Let
$X$ be a totally ordered set. Let $\ell:E\rightarrow X$ be a function. We
shall refer to $\ell$ as the \emph{labeling function}. For every $e\in E$, we
shall refer to $\ell\left(  e\right)  $ as the \emph{label} of $e$.

A \emph{broken circuit} of $M$ means a subset of $E$ having the form
$C\setminus\left\{  e\right\}  $, where $C$ is a circuit of $M$, and where $e$
is the unique element of $C$ having maximum label (among the elements of $C$).
Of course, the notion of a broken circuit of $M$ depends on the function
$\ell$; however, we suppress the mention of $\ell$ in our notation, since we
will not consider situations where two different $\ell$'s coexist.
\end{definition}

We shall now state analogues (and, in light of Example
\ref{exam.matroid.charpol.graph}, generalizations, although we shall not
elaborate on the few minor technicalities of seeing them as such) of Theorem
\ref{thm.chrompol.varis}, Theorem \ref{thm.chrompol.empty}, Corollary
\ref{cor.chrompol.K-free}, Corollary \ref{cor.chrompol.NBC} and Corollary
\ref{cor.chrompol.NBCfor}:

\begin{theorem}
\label{thm.matroid.charpol.varis}Let $M=\left(  E,\mathcal{I}\right)  $ be a
matroid. Let $m=r_{M}\left(  E\right)  $. Let $X$ be a totally ordered set.
Let $\ell:E\rightarrow X$ be a labeling function. Let $\mathfrak{K}$ be some
set of broken circuits of $M$ (not necessarily containing all of them). Let
$a_{K}$ be an element of $\mathbf{k}$ for every $K\in\mathfrak{K}$. Then,%
\[
\widetilde{\chi}_{M}=\sum_{F\subseteq E}\left(  -1\right)  ^{\left\vert
F\right\vert }\left(  \prod_{\substack{K\in\mathfrak{K};\\K\subseteq F}%
}a_{K}\right)  x^{m-r_{M}\left(  F\right)  }.
\]
(Here, the polynomial $\widetilde{\chi}_{M}
\in\mathbb{Z}\left[  x\right]  $ on the left-hand side is regarded as an
element of $\mathbf{k}\left[  x\right]  $ via the canonical ring morphism
$\mathbb{Z}\left[  x\right]  \rightarrow\mathbf{k}\left[  x\right]  $.)
\end{theorem}

\begin{theorem}
\label{thm.matroid.charpol.empty}Let $M=\left(  E,\mathcal{I}\right)  $ be a
matroid. Let $m=r_{M}\left(  E\right)  $. Then,%
\[
\widetilde{\chi}_{M}=\sum_{F\subseteq E}\left(  -1\right)  ^{\left\vert
F\right\vert }x^{m-r_{M}\left(  F\right)  }.
\]

\end{theorem}

\begin{corollary}
\label{cor.matroid.charpol.K-free}Let $M=\left(  E,\mathcal{I}\right)  $ be a
matroid. Let $m=r_{M}\left(  E\right)  $. Let $X$ be a totally ordered set.
Let $\ell:E\rightarrow X$ be a labeling function. Let $\mathfrak{K}$ be some
set of broken circuits of $M$ (not necessarily containing all of them). Then,%
\[
\widetilde{\chi}_{M}=\sum_{\substack{F\subseteq E;\\F\text{ is }%
\mathfrak{K}\text{-free}}}\left(  -1\right)  ^{\left\vert F\right\vert
}x^{m-r_{M}\left(  F\right)  }.
\]

\end{corollary}

\begin{corollary}
\label{cor.matroid.charpol.NBC}Let $M=\left(  E,\mathcal{I}\right)  $ be a
matroid. Let $m=r_{M}\left(  E\right)  $. Let $X$ be a totally ordered set.
Let $\ell:E\rightarrow X$ be a labeling function. Then,%
\[
\widetilde{\chi}_{M}=\sum_{\substack{F\subseteq E;\\F\text{ contains no
broken}\\\text{circuit of }M\text{ as a subset}}}\left(  -1\right)
^{\left\vert F\right\vert }x^{m-r_{M}\left(  F\right)  }.
\]

\end{corollary}

\begin{corollary}
\label{cor.matroid.charpol.NBCfor}Let $M=\left(  E,\mathcal{I}\right)  $ be a
matroid. Let $m=r_{M}\left(  E\right)  $. Let $X$ be a totally ordered set.
Let $\ell:E\rightarrow X$ be an injective labeling function. Then,%
\[
\widetilde{\chi}_{M}=\sum_{\substack{F\subseteq E;\\F\text{ contains no
broken}\\\text{circuit of }M\text{ as a subset}}}\left(  -1\right)
^{\left\vert F\right\vert }x^{m-\left\vert F\right\vert }.
\]

\end{corollary}

We notice that Corollary \ref{cor.matroid.charpol.NBCfor} is equivalent to
\cite[Theorem 4.12]{Stanley} (at least when $M$ is loopless).

Before we prove these results, let us state a lemma which will serve as an
analogue of Lemma \ref{lem.NBCm.moeb}:

\begin{lemma}
\label{lem.matroid.NBCm.moeb}Let $M=\left(  E,\mathcal{I}\right)  $ be a
matroid. Let $X$ be a totally ordered set. Let $\ell:E\rightarrow X$ be a
labeling function. Let $\mathfrak{K}$ be some set of broken circuits of $M$
(not necessarily containing all of them). Let $a_{K}$ be an element of
$\mathbf{k}$ for every $K\in\mathfrak{K}$.

Let $F$ be any flat of $M$. Then,%
\begin{equation}
\sum_{B\subseteq F}\left(  -1\right)  ^{\left\vert B\right\vert }%
\prod_{\substack{K\in\mathfrak{K};\\K\subseteq B}}a_{K}=\left[  F=\varnothing
\right]  . \label{eq.lem.matroid.NBCm.moeb.1}%
\end{equation}
(Again, we are using the Iverson bracket notation as in Definition
\ref{def.iverson}.)
\end{lemma}

\begin{proof}
[Proof of Lemma \ref{lem.matroid.NBCm.moeb}.]Our proof will imitate the proof
of Lemma \ref{lem.NBCm.moeb} much of the time (with $E\cap\operatorname*{Eqs}%
f$ replaced by $F$); thus, we will allow ourselves some more brevity.

We WLOG assume that $F\neq\varnothing$ (since otherwise, the claim is
obvious\footnote{\textit{Proof.} Assume that $F=\varnothing$. We must show
that the claim is obvious.
\par
Let us first show that no $K\in\mathfrak{K}$ satisfies $K=\varnothing$.
Indeed, assume the contrary. Thus, there exists a $K\in\mathfrak{K}$
satisfying $K=\varnothing$. In other words, $\varnothing\in\mathfrak{K}$.
Thus, $\varnothing$ is a broken circuit of $M$ (since $\mathfrak{K}$ is a set
of broken circuits of $M$). Therefore, $\varnothing$ is obtained from a
circuit of $M$ by removing one element (by the definition of a broken
circuit). This latter circuit must therefore be a one-element set, i.e., it
has the form $\left\{  e\right\}  $ for some $e\in E$. Consider this $e$.
Thus, $\left\{  e\right\}  $ is a circuit of $M$.
\par
But $F$ is a flat of $M$. In other words, Statement $\mathfrak{F}_{1}$ (of
Lemma \ref{lem.matroid.flat-crit}) holds for $T=F$. Hence, Statement
$\mathfrak{F}_{2}$ (of Lemma \ref{lem.matroid.flat-crit}) also holds for $T=F$
(since Lemma \ref{lem.matroid.flat-crit} shows that these two statements are
equivalent). Applying Statement $\mathfrak{F}_{2}$ to $T=F$ and $C=\left\{
e\right\}  $, we thus obtain $\left\{  e\right\}  \subseteq F$ (because
$\left\{  e\right\}  \setminus\left\{  e\right\}  =\varnothing\subseteq F$).
Thus, $e\in\left\{  e\right\}  \subseteq F=\varnothing$, which is absurd. This
contradiction proves that our assumption was wrong.
\par
Hence, we have shown that no $K\in\mathfrak{K}$ satisfies $K=\varnothing$. But
from $F=\varnothing$, we see that the sum $\sum_{B\subseteq F}\left(
-1\right)  ^{\left\vert B\right\vert }\prod_{\substack{K\in\mathfrak{K}%
;\\K\subseteq B}}a_{K}$ has only one addend (namely, the addend for
$B=\varnothing$), and thus simplifies to
\begin{align*}
\underbrace{\left(  -1\right)  ^{\left\vert \varnothing\right\vert }%
}_{=\left(  -1\right)  ^{0}=1}\underbrace{\prod_{\substack{K\in\mathfrak{K}%
;\\K\subseteq\varnothing}}}_{=\prod_{\substack{K\in\mathfrak{K}%
;\\K=\varnothing}}}a_{K}  &  =\prod_{\substack{K\in\mathfrak{K}%
;\\K=\varnothing}}a_{K}=\left(  \text{empty product}\right)
\ \ \ \ \ \ \ \ \ \ \left(  \text{since no }K\in\mathfrak{K}\text{ satisfies
}K=\varnothing\right) \\
&  =1=\left[  F=\varnothing\right]  \ \ \ \ \ \ \ \ \ \ \left(  \text{since
}F=\varnothing\right)  .
\end{align*}
Thus, Lemma \ref{lem.matroid.NBCm.moeb} is proven.}). Thus, $\left[
F=\varnothing\right]  =0$.

Pick any $d\in F$ with maximum $\ell\left(  d\right)  $ (among all $d\in F$).
(This is clearly possible, since $F\neq\varnothing$.) Define two subsets
$\mathcal{U}$ and $\mathcal{V}$ of $\mathcal{P}\left(  F\right)  $ as follows:%
\begin{align*}
\mathcal{U}  &  =\left\{  T\in\mathcal{P}\left(  F\right)  \ \mid\ d\notin
T\right\}  ;\\
\mathcal{V}  &  =\left\{  T\in\mathcal{P}\left(  F\right)  \ \mid\ d\in
T\right\}  .
\end{align*}
Thus, we have $\mathcal{P}\left(  F\right)  =\mathcal{U}\cup\mathcal{V}$, and
the sets $\mathcal{U}$ and $\mathcal{V}$ are disjoint. Now, we define a map
$\Phi:\mathcal{U}\rightarrow\mathcal{V}$ by%
\[
\left(  \Phi\left(  B\right)  =B\cup\left\{  d\right\}
\ \ \ \ \ \ \ \ \ \ \text{for every }B\in\mathcal{U}\right)  .
\]
This map $\Phi$ is well-defined (because for every $B\in\mathcal{U}$, we have
$B\cup\left\{  d\right\}  \in\mathcal{V}$\ \ \ \ \footnote{This follows from
the fact that $d\in F$.}) and a bijection\footnote{Its inverse is the map
$\Psi:\mathcal{V}\rightarrow\mathcal{U}$ defined by $\left(  \Psi\left(
B\right)  =B\setminus\left\{  d\right\}  \ \ \ \ \ \ \ \ \ \ \text{for every
}B\in\mathcal{V}\right)  $.}. Moreover, every $B\in\mathcal{U}$ satisfies%
\begin{equation}
\left(  -1\right)  ^{\left\vert \Phi\left(  B\right)  \right\vert }=-\left(
-1\right)  ^{\left\vert B\right\vert }
\label{pf.lem.matroid.NBCm.moeb.short.Phi.-1}%
\end{equation}

\begin{vershort}
\noindent\footnote{\textit{Proof.} This is proven exactly like we proved
(\ref{pf.lem.NBCm.moeb.short.Phi.-1}).}.
\end{vershort}

\begin{verlong}
\noindent\footnote{\textit{Proof.} This is proven exactly like we proved
(\ref{pf.lem.NBCm.moeb.phisize-1}).}.
\end{verlong}

Now, we claim that, for every $B\in\mathcal{U}$ and every $K\in\mathfrak{K}$,
we have the following logical equivalence:%
\begin{equation}
\left(  K\subseteq B\right)  \ \Longleftrightarrow\ \left(  K\subseteq
\Phi\left(  B\right)  \right)  .
\label{pf.lem.matroid.NBCm.moeb.short.Phi.equiv}%
\end{equation}

\textit{Proof of (\ref{pf.lem.matroid.NBCm.moeb.short.Phi.equiv}):} Let
$B\in\mathcal{U}$ and $K\in\mathfrak{K}$. We must prove the equivalence
(\ref{pf.lem.matroid.NBCm.moeb.short.Phi.equiv}). The definition of $\Phi$
yields $\Phi\left(  B\right)  =B\cup\left\{  d\right\}  \supseteq B$, so that
$B\subseteq\Phi\left(  B\right)  $. Hence, if $K\subseteq B$, then $K\subseteq
B\subseteq\Phi\left(  B\right)  $. Therefore, the forward implication of the
equivalence (\ref{pf.lem.matroid.NBCm.moeb.short.Phi.equiv}) is proven. It
thus remains to prove the backward implication of this equivalence. In other
words, it remains to prove that if $K\subseteq\Phi\left(  B\right)  $, then
$K\subseteq B$. So let us assume that $K\subseteq\Phi\left(  B\right)  $.

We want to prove that $K\subseteq B$. Assume the contrary. Thus,
$K\not \subseteq B$. We have $K\in\mathfrak{K}$. Thus, $K$ is a broken circuit
of $M$ (since $\mathfrak{K}$ is a set of broken circuits of $M$). In other
words, $K$ is a subset of $E$ having the form $C\setminus\left\{  e\right\}
$, where $C$ is a circuit of $M$, and where $e$ is the unique element of $C$
having maximum label (among the elements of $C$) (because this is how a broken
circuit is defined). Consider these $C$ and $e$. Thus, $K=C\setminus\left\{
e\right\}  $.

The element $e$ is the unique element of $C$ having maximum label (among the
elements of $C$). Thus, if $e^{\prime}$ is any element of $C$ satisfying
$\ell\left(  e^{\prime}\right)  \geq\ell\left(  e\right)  $, then%
\begin{equation}
e^{\prime}=e. \label{pf.lem.matroid.NBCm.moeb.short.Phi.equiv.pf.e'}%
\end{equation}

But $\underbrace{K}_{\subseteq\Phi\left(  B\right)  =B\cup\left\{  d\right\}
}\setminus\left\{  d\right\}  \subseteq\left(  B\cup\left\{  d\right\}
\right)  \setminus\left\{  d\right\}  \subseteq B$.

If we had $d\notin K$, then we would have $K\setminus\left\{  d\right\}  =K$
and therefore $K=K\setminus\left\{  d\right\}  \subseteq B$; this would
contradict $K\not \subseteq B$. Hence, we cannot have $d\notin K$. We thus
must have $d\in K$. Hence, $d\in K=C\setminus\left\{  e\right\}  $. Hence,
$d\in C$ and $d\neq e$.

But $C\setminus\left\{  e\right\}  =K\subseteq\Phi\left(  B\right)  \subseteq
F$ (since $\Phi\left(  B\right)  \in\mathcal{P}\left(  F\right)  $). On the
other hand, Statement $\mathfrak{F}_{1}$ (of Lemma \ref{lem.matroid.flat-crit}%
) holds for $T=F$ (since $F$ is a flat of $M$). Hence, Statement
$\mathfrak{F}_{2}$ (of Lemma \ref{lem.matroid.flat-crit}) also holds for $T=F$
(since Lemma \ref{lem.matroid.flat-crit} shows that these two statements are
equivalent). Thus, from $C\setminus\left\{  e\right\}  \subseteq F$, we obtain
$C\subseteq F$. Thus, $e\in C\subseteq F$. Consequently, $\ell\left(
d\right)  \geq\ell\left(  e\right)  $ (since $d$ was defined to be an element
of $F$ with maximum $\ell\left(  d\right)  $ among all $d\in F$).

Also, $d\in C$. Since $\ell\left(  d\right)  \geq\ell\left(  e\right)  $, we
can therefore apply (\ref{pf.lem.matroid.NBCm.moeb.short.Phi.equiv.pf.e'}) to
$e^{\prime}=d$. We thus obtain $d=e$. This contradicts $d\neq e$. This
contradiction proves that our assumption was wrong. Hence, $K\subseteq B$ is
proven. Thus, we have proven the backward implication of the equivalence
(\ref{pf.lem.matroid.NBCm.moeb.short.Phi.equiv}); this completes the proof of
(\ref{pf.lem.matroid.NBCm.moeb.short.Phi.equiv}).

\begin{vershort}
Now, proceeding as in the proof of (\ref{pf.lem.NBCm.moeb.short.there}), we
can show that
\[
\sum_{B\subseteq F}\left(  -1\right)  ^{\left\vert B\right\vert }%
\prod_{\substack{K\in\mathfrak{K};\\K\subseteq B}}a_{K}=\left[  F=\varnothing
\right]  .
\]

\end{vershort}

\begin{verlong}
Now, proceeding as in the proof of (\ref{pf.lem.NBCm.moeb.there}), we can show
that
\[
\sum_{B\subseteq F}\left(  -1\right)  ^{\left\vert B\right\vert }%
\prod_{\substack{K\in\mathfrak{K};\\K\subseteq B}}a_{K}=\left[  F=\varnothing
\right]  .
\]

\end{verlong}

\noindent This proves Lemma \ref{lem.matroid.NBCm.moeb}.
\end{proof}

We shall furthermore use a classical and fundamental result on the M\"{o}bius
function of any finite poset:

\begin{proposition}
\label{prop.moebius.double0}Let $P$ be a finite poset. Let $\mu$ denote the
M\"{o}bius function of $P$.

\textbf{(a)} For any $x\in P$ and $y\in P$, we have%
\begin{equation}
\sum_{\substack{z\in P;\\x\leq z\leq y}}\mu\left(  x,z\right)  =\left[
x=y\right]  . \label{eq.prop.moebius.double0.a}%
\end{equation}

\textbf{(b)} For any $x\in P$ and $y\in P$, we have%
\begin{equation}
\sum_{\substack{z\in P;\\x\leq z\leq y}}\mu\left(  z,y\right)  =\left[
x=y\right]  . \label{eq.prop.moebius.double0.b}%
\end{equation}

\textbf{(c)} Let $\mathbf{k}$ be a $\mathbb{Z}$-module. Let $\left(  \beta
_{x}\right)  _{x\in P}$ be a family of elements of $\mathbf{k}$. Then, every
$z\in P$ satisfies%
\[
\beta_{z}=\sum_{\substack{y\in P;\\y\leq z}}\mu\left(  y,z\right)
\sum_{\substack{x\in P;\\x\leq y}}\beta_{x}.
\]

\end{proposition}

For the sake of completeness, let us give a self-contained proof of this
proposition (slicker arguments appear in the literature\footnote{For example,
Proposition \ref{prop.moebius.double0} \textbf{(c)} is equivalent to the
$\Longrightarrow$ implication of \cite[(2.3a)]{Martin22}.}):

\begin{vershort}
\begin{proof}
[Proof of Proposition \ref{prop.moebius.double0}.]\textbf{(a)} Let $x\in P$
and $y\in P$. We must prove the equality (\ref{eq.prop.moebius.double0.a}). We
are in one of the following three cases:

\textit{Case 1:} We have $x=y$.

\textit{Case 2:} We have $x<y$.

\textit{Case 3:} We have neither $x=y$ nor $x<y$.

Let us first consider Case 1. In this case, we have $x=y$. Hence, the sum
$\sum_{\substack{z\in P;\\x\leq z\leq y}}\mu\left(  x,z\right)  $ contains
only one addend -- namely, the addend for $z=x$. Thus,%
\begin{align*}
\sum_{\substack{z\in P;\\x\leq z\leq y}}\mu\left(  x,z\right)   &  =\mu\left(
x,x\right)  =1\ \ \ \ \ \ \ \ \ \ \left(  \text{by the definition of the
M\"{o}bius function}\right) \\
&  =\left[  x=y\right]  \ \ \ \ \ \ \ \ \ \ \left(  \text{since }x=y\right)  .
\end{align*}
Thus, (\ref{eq.prop.moebius.double0.a}) is proven in Case 1.

Let us now consider Case 2. In this case, we have $x<y$. Hence, $x\neq y$, so
that $\left[  x=y\right]  =0$. Now, $y$ is an element of $P$ satisfying $x\leq
y\leq y$. Thus, the sum $\sum_{\substack{z\in P;\\x\leq z\leq y}}\mu\left(
x,z\right)  $ contains an addend for $z=y$. Splitting off this addend, we
obtain%
\begin{align*}
\sum_{\substack{z\in P;\\x\leq z\leq y}}\mu\left(  x,z\right)   &
=\underbrace{\sum_{\substack{z\in P;\\x\leq z\leq y;\ z\neq y}}}%
_{=\sum_{\substack{z\in P;\\x\leq z<y}}}\mu\left(  x,z\right)
+\underbrace{\mu\left(  x,y\right)  }_{\substack{=-\sum_{\substack{z\in
P;\\x\leq z<y}}\mu\left(  x,z\right)  \\\text{(by (\ref{eq.def.moebius.rec2}%
))}}}\\
&  =\sum_{\substack{z\in P;\\x\leq z<y}}\mu\left(  x,z\right)  +\left(
-\sum_{\substack{z\in P;\\x\leq z<y}}\mu\left(  x,z\right)  \right)
=0=\left[  x=y\right]  .
\end{align*}
Hence, (\ref{eq.prop.moebius.double0.a}) is proven in Case 2.

Finally, let us consider Case 3. In this case, we have neither $x=y$ nor
$x<y$. Thus, we do not have $x\leq y$. Hence, there exists no $z\in P$
satisfying $x\leq z\leq y$. Thus,
\[
\sum_{\substack{z\in P;\\x\leq z\leq y}}\mu\left(  x,z\right)  =\left(
\text{empty sum}\right)  =0=\left[  x=y\right]
\]
(since we do not have $x=y$). Thus, (\ref{eq.prop.moebius.double0.a}) is
proven in Case 3.

Hence, (\ref{eq.prop.moebius.double0.a}) is proven in all three cases. This
proves Proposition \ref{prop.moebius.double0} \textbf{(a)}.

\textbf{(b)} For any two elements $u$ and $v$ of $P$, we define a subset
$\left[  u,v\right]  $ of $P$ by%
\[
\left[  u,v\right]  =\left\{  w\in P\ \mid\ u\leq w\leq v\right\}  .
\]
This subset $\left[  u,v\right]  $ is finite (since $P$ is finite), and thus
its size $\left\vert \left[  u,v\right]  \right\vert $ is a nonnegative integer.

We shall now prove Proposition \ref{prop.moebius.double0} \textbf{(b)} by
strong induction on $\left\vert \left[  x,y\right]  \right\vert $:

\textit{Induction step:} Let $N\in\mathbb{N}$. Assume that Proposition
\ref{prop.moebius.double0} \textbf{(b)} holds whenever $\left\vert \left[
x,y\right]  \right\vert <N$. We must now prove that Proposition
\ref{prop.moebius.double0} \textbf{(b)} holds whenever $\left\vert \left[
x,y\right]  \right\vert =N$.

We have assumed that Proposition \ref{prop.moebius.double0} \textbf{(b)} holds
whenever $\left\vert \left[  x,y\right]  \right\vert <N$. In other words, we
have assumed the following claim:

\begin{statement}
\textit{Claim 1:} For any $x\in P$ and $y\in P$ satisfying $\left\vert \left[
x,y\right]  \right\vert <N$, we have%
\[
\sum_{\substack{z\in P;\\x\leq z\leq y}}\mu\left(  z,y\right)  =\left[
x=y\right]  .
\]

\end{statement}

Now, let $x$ and $y$ be two elements of $P$ satisfying $\left\vert \left[
x,y\right]  \right\vert =N$. We are going to prove that%
\begin{equation}
\sum_{\substack{z\in P;\\x\leq z\leq y}}\mu\left(  z,y\right)  =\left[
x=y\right]  . \label{pf.prop.moebius.double0.short.b.goal}%
\end{equation}

We are in one of the following three cases:

\textit{Case 1:} We have $x=y$.

\textit{Case 2:} We have $x<y$.

\textit{Case 3:} We have neither $x=y$ nor $x<y$.

In Case 1 and in Case 3, we can prove
(\ref{pf.prop.moebius.double0.short.b.goal}) in exactly the same way as (in
our above proof of Proposition \ref{prop.moebius.double0} \textbf{(a)}) we
have proven (\ref{eq.prop.moebius.double0.a}). Thus, it remains only to prove
(\ref{pf.prop.moebius.double0.short.b.goal}) in Case 2. In other words, we can
WLOG assume that we are in Case 2.

Assume this. Hence, $x<y$, so that $\left[  x=y\right]  =0$.

For every $t\in P$ satisfying $x\leq t<y$, we have%
\begin{equation}
\left\vert \left[  x,t\right]  \right\vert <N
\label{pf.prop.moebius.double0.short.b.smaller}%
\end{equation}
\footnote{\textit{Proof of (\ref{pf.prop.moebius.double0.short.b.smaller}):}
Let $t\in P$ be such that $x\leq t<y$. We shall proceed in several steps:
\par
\begin{itemize}
\item We have
\begin{align*}
\left[  x,t\right]   &  =\left\{  w\in P\ \mid\ x\leq w\leq t\right\}
\ \ \ \ \ \ \ \ \ \ \left(  \text{by the definition of }\left[  x,t\right]
\right) \\
&  \subseteq\left\{  w\in P\ \mid\ x\leq w\leq y\right\}
\ \ \ \ \ \ \ \ \ \ \left(
\begin{array}
[c]{c}%
\text{because every }w\in P\text{ satisfying }w\leq t\\
\text{must also satisfy }w\leq y\text{ (since }t<y\text{)}%
\end{array}
\right) \\
&  =\left[  x,y\right]  \ \ \ \ \ \ \ \ \ \ \left(  \text{by the definition of
}\left[  x,y\right]  \right)  .
\end{align*}
\par
\item We have $t<y$. Thus, we do not have $y\leq t$. Hence, we do not have
$x\leq y\leq t$. Hence, $y\notin\left[  x,t\right]  $. But $y\in\left[
x,y\right]  $ (since $x\leq y\leq y$). Hence, the sets $\left[  x,t\right]  $
and $\left[  x,y\right]  $ are distinct (since the latter contains $y$ but the
former does not). Combining this with $\left[  x,t\right]  \subseteq\left[
x,y\right]  $, we conclude that $\left[  x,t\right]  $ is a proper subset of
$\left[  x,y\right]  $. Hence, $\left\vert \left[  x,t\right]  \right\vert
<\left\vert \left[  x,y\right]  \right\vert =N$. This proves
(\ref{pf.prop.moebius.double0.short.b.smaller}).
\end{itemize}
}. Therefore, for every $t\in P$ satisfying $x\leq t<y$, we have%
\begin{equation}
\sum_{\substack{z\in P;\\x\leq z\leq t}}\mu\left(  z,t\right)  =\left[
x=t\right]  \label{pf.prop.moebius.double0.short.b.smaller2}%
\end{equation}
(by Claim 1, applied to $t$ instead of $y$). Also, for every $u\in P$ and
$v\in P$, we have%
\begin{equation}
\sum_{\substack{t\in P;\\u\leq t\leq v}}\mu\left(  u,t\right)  =\left[
u=v\right]  \label{pf.prop.moebius.double0.short.b.smaller3}%
\end{equation}
\footnote{\textit{Proof of (\ref{pf.prop.moebius.double0.short.b.smaller3}):}
Let $u\in P$ and $v\in P$. Proposition \ref{prop.moebius.double0} \textbf{(a)}
(applied to $x=u$ and $y=v$) shows that $\sum_{\substack{z\in P;\\u\leq z\leq
v}}\mu\left(  u,z\right)  =\left[  u=v\right]  $. Now,%
\begin{align*}
\sum_{\substack{t\in P;\\u\leq t\leq v}}\mu\left(  u,t\right)   &
=\sum_{\substack{z\in P;\\u\leq z\leq v}}\mu\left(  u,z\right)
\ \ \ \ \ \ \ \ \ \ \left(  \text{here, we have substituted }z\text{ for
}t\text{ in the sum}\right) \\
&  =\left[  u=v\right]  .
\end{align*}
This proves (\ref{pf.prop.moebius.double0.short.b.smaller3}).}.

Now,%
\begin{align*}
&  \underbrace{\sum_{\substack{\left(  z,t\right)  \in P^{2};\\x\leq z\leq
t\leq y}}}_{=\sum_{\substack{z\in P;\\x\leq z\leq y}}\ \ \sum_{\substack{t\in
P;\\z\leq t\leq y}}}\mu\left(  z,t\right) \\
&  =\sum_{\substack{z\in P;\\x\leq z\leq y}}\ \ \underbrace{\sum
_{\substack{t\in P;\\z\leq t\leq y}}\mu\left(  z,t\right)  }%
_{\substack{=\left[  z=y\right]  \\\text{(by
(\ref{pf.prop.moebius.double0.short.b.smaller3})}\\\text{(applied to
}u=z\text{ and }v=y\text{))}}}=\sum_{\substack{z\in P;\\x\leq z\leq y}}\left[
z=y\right] \\
&  =\sum_{\substack{z\in P;\\x\leq z\leq y\text{ and }z=y}}\underbrace{\left[
z=y\right]  }_{\substack{=1\\\text{(since }z=y\text{)}}}+\sum_{\substack{z\in
P;\\x\leq z\leq y\text{ and }z\neq y}}\underbrace{\left[  z=y\right]
}_{\substack{=0\\\text{(since }z\neq y\text{)}}}\\
&  \ \ \ \ \ \ \ \ \ \ \ \ \ \ \ \ \ \ \ \ \left(  \text{since every }z\in
P\text{ satisfies either }z=y\text{ or }z\neq y\text{ (but not both)}\right)
\\
&  =\underbrace{\sum_{\substack{z\in P;\\x\leq z\leq y\text{ and }z=y}%
}}_{=\sum_{z\in\left\{  w\in P\ \mid\ x\leq w\leq y\text{ and }w=y\right\}  }%
}1+\underbrace{\sum_{\substack{z\in P;\\x\leq z\leq y\text{ and }z\neq y}%
}0}_{=0}=\sum_{z\in\left\{  w\in P\ \mid\ x\leq w\leq y\text{ and
}w=y\right\}  }1\\
&  =\left\vert \underbrace{\left\{  w\in P\ \mid\ x\leq w\leq y\text{ and
}w=y\right\}  }_{=\left\{  y\right\}  }\right\vert =\left\vert \left\{
y\right\}  \right\vert =1.
\end{align*}
Hence,%
\begin{align*}
1  &  =\underbrace{\sum_{\substack{\left(  z,t\right)  \in P^{2};\\x\leq z\leq
t\leq y}}}_{=\sum_{\substack{t\in P;\\x\leq t\leq y}}\ \ \sum_{\substack{z\in
P;\\x\leq z\leq t}}}\mu\left(  z,t\right)  =\sum_{\substack{t\in P;\\x\leq
t\leq y}}\ \ \sum_{\substack{z\in P;\\x\leq z\leq t}}\mu\left(  z,t\right) \\
&  =\underbrace{\sum_{\substack{t\in P;\\x\leq t\leq y\text{ and }t=y}%
}}_{\substack{=\sum_{t\in\left\{  w\in P\ \mid\ x\leq w\leq y\text{ and
}w=y\right\}  }=\sum_{t\in\left\{  y\right\}  }\\\text{(since }\left\{  w\in
P\ \mid\ x\leq w\leq y\text{ and }w=y\right\}  =\left\{  y\right\}  \text{)}%
}}\sum_{\substack{z\in P;\\x\leq z\leq t}}\mu\left(  z,t\right) \\
&  \ \ \ \ \ \ \ \ \ \ +\sum_{\substack{t\in P;\\x\leq t\leq y\text{ and
}t\neq y}}\ \ \underbrace{\sum_{\substack{z\in P;\\x\leq z\leq t}}\mu\left(
z,t\right)  }_{\substack{=\left[  x=t\right]  \\\text{(by
(\ref{pf.prop.moebius.double0.short.b.smaller2})}\\\text{(since }t<y\text{
(because }t\leq y\text{ and }t\neq y\text{) and }x\leq t\text{))}}}\\
&  \ \ \ \ \ \ \ \ \ \ \ \ \ \ \ \ \ \ \ \ \left(  \text{since every }t\in
P\text{ satisfies either }t=y\text{ or }t\neq y\text{ (but not both)}\right)
\\
&  =\underbrace{\sum_{t\in\left\{  y\right\}  }\ \ \sum_{\substack{z\in
P;\\x\leq z\leq t}}\mu\left(  z,t\right)  }_{=\sum_{\substack{z\in P;\\x\leq
z\leq y}}\mu\left(  z,y\right)  }+\sum_{\substack{t\in P;\\x\leq t\leq y\text{
and }t\neq y}}\left[  x=t\right] \\
&  =\sum_{\substack{z\in P;\\x\leq z\leq y}}\mu\left(  z,y\right)
+\sum_{\substack{t\in P;\\x\leq t\leq y\text{ and }t\neq y}}\left[
x=t\right]  .
\end{align*}
Subtracting $\sum_{\substack{z\in P;\\x\leq z\leq y}}\mu\left(  z,y\right)  $
from both sides of this equality, we obtain%
\begin{align*}
&  1-\sum_{\substack{z\in P;\\x\leq z\leq y}}\mu\left(  z,y\right) \\
&  =\sum_{\substack{t\in P;\\x\leq t\leq y\text{ and }t\neq y}}\left[
x=t\right] \\
&  =\underbrace{\sum_{\substack{t\in P;\\x\leq t\leq y\text{ and }t=x\text{
and }t\neq y}}}_{\substack{=\sum_{t\in\left\{  z\in P\ \mid\ x\leq z\leq
y\text{ and }z=x\text{ and }z\neq y\right\}  }=\sum_{t\in\left\{  x\right\}
}\\\text{(since }\left\{  z\in P\ \mid\ x\leq z\leq y\text{ and }z=x\text{ and
}z\neq y\right\}  =\left\{  x\right\}  \text{)}}}\underbrace{\left[
x=t\right]  }_{\substack{=1\\\text{(since }x=t\text{)}}}\\
&  \ \ \ \ \ \ \ \ \ \ +\sum_{\substack{t\in P;\\x\leq t\leq y\text{ and
}t\neq x\text{ and }t\neq y}}\underbrace{\left[  x=t\right]  }%
_{\substack{=0\\\text{(since }x\neq t\text{)}}}\\
&  \ \ \ \ \ \ \ \ \ \ \ \ \ \ \ \ \ \ \ \ \left(  \text{since every }t\in
P\text{ satisfies either }t=x\text{ or }t\neq x\text{ (but not both)}\right)
\\
&  =\sum_{t\in\left\{  x\right\}  }1+\underbrace{\sum_{\substack{t\in
P;\\x\leq t\leq y\text{ and }t\neq x\text{ and }t\neq y}}0}_{=0}=\sum
_{t\in\left\{  x\right\}  }1=1.
\end{align*}
Solving this equality for $\sum_{\substack{z\in P;\\x\leq z\leq y}}\mu\left(
z,y\right)  $, we obtain%
\[
\sum_{\substack{z\in P;\\x\leq z\leq y}}\mu\left(  z,y\right)  =1-1=0=\left[
x=y\right]
\]
(since $x<y$). Thus, (\ref{pf.prop.moebius.double0.short.b.goal}) is proven.

Let us now forget that we fixed $x$ and $y$. We thus have proven that for any
$x\in P$ and $y\in P$ satisfying $\left\vert \left[  x,y\right]  \right\vert
=N$, we have%
\[
\sum_{\substack{z\in P;\\x\leq z\leq y}}\mu\left(  z,y\right)  =\left[
x=y\right]  .
\]
In other words, Proposition \ref{prop.moebius.double0} \textbf{(b)} holds
whenever $\left\vert \left[  x,y\right]  \right\vert =N$. This completes the
induction step. Thus, Proposition \ref{prop.moebius.double0} \textbf{(b)} is
proven by induction.

\textbf{(c)} For every $v\in P$, we have%
\begin{align*}
&  \sum_{\substack{y\in P;\\y\leq v}}\mu\left(  y,v\right)  \sum
_{\substack{x\in P;\\x\leq y}}\beta_{x}\\
&  =\sum_{\substack{z\in P;\\z\leq v}}\mu\left(  z,v\right)  \sum
_{\substack{x\in P;\\x\leq z}}\beta_{x}\ \ \ \ \ \ \ \ \ \ \left(
\begin{array}
[c]{c}%
\text{here, we have renamed the summation}\\
\text{index }y\text{ as }z\text{ in the outer sum}%
\end{array}
\right) \\
&  =\underbrace{\sum_{\substack{z\in P;\\z\leq v}}\ \ \sum_{\substack{x\in
P;\\x\leq z}}}_{=\sum_{x\in P}\ \ \sum_{\substack{z\in P;\\x\leq z\leq v}}}%
\mu\left(  z,v\right)  \beta_{x}=\sum_{x\in P}\ \ \underbrace{\sum
_{\substack{z\in P;\\x\leq z\leq v}}\mu\left(  z,v\right)  }%
_{\substack{=\left[  x=v\right]  \\\text{(by Proposition
\ref{prop.moebius.double0} \textbf{(b)}}\\\text{(applied to }y=v\text{))}%
}}\beta_{x}\\
&  =\sum_{x\in P}\left[  x=v\right]  \beta_{x}=\sum_{\substack{x\in
P;\\x=v}}\underbrace{\left[  x=v\right]  }_{\substack{=1\\\text{(since
}x=v\text{)}}}\beta_{x}+\sum_{\substack{x\in P;\\x\neq v}}\underbrace{\left[
x=v\right]  }_{\substack{=0\\\text{(since }x\neq v\text{)}}}\beta_{x}\\
&  \ \ \ \ \ \ \ \ \ \ \ \ \ \ \ \ \ \ \ \ \left(  \text{since every }x\in
P\text{ satisfies either }x=v\text{ or }x\neq v\text{ (but not both)}\right)
\\
&  =\sum_{\substack{x\in P;\\x=v}}\beta_{x}+\underbrace{\sum_{\substack{x\in
P;\\x\neq v}}0\beta_{x}}_{=0}=\sum_{\substack{x\in P;\\x=v}}\beta_{x}%
=\beta_{v}\ \ \ \ \ \ \ \ \ \ \left(  \text{since }v\in P\right)  .
\end{align*}
Renaming $v$ as $z$ in this result, we obtain precisely Proposition
\ref{prop.moebius.double0} \textbf{(c)}.
\end{proof}
\end{vershort}

\begin{verlong}
\begin{proof}
[Proof of Proposition \ref{prop.moebius.double0}.]We first notice that%
\begin{equation}
\left\{  w\in P\ \mid\ x\leq w\leq x\right\}  =\left\{  x\right\}
\label{pf.prop.moebius.double0.xx}%
\end{equation}
for every $x\in P$\ \ \ \ \footnote{\textit{Proof of
(\ref{pf.prop.moebius.double0.xx}):} Let $x\in P$.
\par
Clearly, $x$ is an element of $P$ and satisfies $x\leq x\leq x$. Thus, $x$ is
an element $w$ of $P$ satisfying $x\leq w\leq x$. In other words,
$x\in\left\{  w\in P\ \mid\ x\leq w\leq x\right\}  $. Hence, $\left\{
x\right\}  \subseteq\left\{  w\in P\ \mid\ x\leq w\leq x\right\}  $.
\par
On the other hand, let $z\in\left\{  w\in P\ \mid\ x\leq w\leq x\right\}  $.
Thus, $z$ is an element $w$ of $P$ satisfying $x\leq w\leq x$. In other words,
$z$ is an element of $P$ and satisfies $x\leq z\leq x$. Combining $x\leq z$
with $z\leq x$, we obtain $x=z$ (since the partial order on $P$ is
antisymmetric). Hence, $z=x\in\left\{  x\right\}  $.
\par
Now, forget that we fixed $z$. We thus have proven that $z\in\left\{
x\right\}  $ for every $z\in\left\{  w\in P\ \mid\ x\leq w\leq x\right\}  $.
In other words, $\left\{  w\in P\ \mid\ x\leq w\leq x\right\}  \subseteq
\left\{  x\right\}  $. Combining this with $\left\{  x\right\}  \subseteq
\left\{  w\in P\ \mid\ x\leq w\leq x\right\}  $, we obtain $\left\{  w\in
P\ \mid\ x\leq w\leq x\right\}  =\left\{  x\right\}  $. This proves
(\ref{pf.prop.moebius.double0.xx}).}.

\textbf{(a)} Let $x\in P$ and $y\in P$. We must prove the equality
(\ref{eq.prop.moebius.double0.a}). If we do not have $x\leq y$, then
(\ref{eq.prop.moebius.double0.a}) holds for obvious
reasons\footnote{\textit{Proof.} Assume that we do not have $x\leq y$. Then,
there exists no $z\in P$ satisfying $x\leq z\leq y$ (because if such a $z$
would exist, then it would satisfy $x\leq z\leq y$, which would contradict the
fact that we do not have $x\leq y$). Therefore, the sum $\sum_{\substack{z\in
P;\\x\leq z\leq y}}\mu\left(  x,z\right)  $ is an empty sum. Hence,%
\[
\sum_{\substack{z\in P;\\x\leq z\leq y}}\mu\left(  x,z\right)  =\left(
\text{empty sum}\right)  =0.
\]
But we do not have $x=y$ (because if we had $x=y$, then we would have $x\leq
y$, which would contradict the fact that we do not have $x\leq y$). Thus,
$\left[  x=y\right]  =0$. Now, $\sum_{\substack{z\in P;\\x\leq z\leq y}%
}\mu\left(  x,z\right)  =0=\left[  x=y\right]  $. Hence,
(\ref{eq.prop.moebius.double0.a}) is proven, qed.}. Hence, for the rest of our
proof of (\ref{eq.prop.moebius.double0.a}), we can WLOG assume that $x\leq y$.
Assume this.

We have $x\leq y$. Thus, either $x=y$ or $x<y$. In other words, we are in one
of the following two cases:

\textit{Case 1:} We have $x=y$.

\textit{Case 2:} We have $x<y$.

Let us first consider Case 1. In this case, we have $x=y$, so that $\left[
x=y\right]  =1$. On the other hand,
\[
\left\{  w\in P\ \mid\ x\leq w\leq\underbrace{y}_{=x}\right\}  =\left\{  w\in
P\ \mid\ x\leq w\leq x\right\}  =\left\{  x\right\}
\]
(by (\ref{pf.prop.moebius.double0.xx})). Now,%
\[
\underbrace{\sum_{\substack{z\in P;\\x\leq z\leq y}}}_{\substack{=\sum
_{z\in\left\{  w\in P\ \mid\ x\leq w\leq y\right\}  }=\sum_{z\in\left\{
x\right\}  }\\\text{(since }\left\{  w\in P\ \mid\ x\leq w\leq y\right\}
=\left\{  x\right\}  \text{)}}}\mu\left(  x,z\right)  =\sum_{z\in\left\{
x\right\}  }\mu\left(  x,z\right)  =\mu\left(  x,x\right)  =1
\]
(by (\ref{eq.def.moebius.rec1})). Comparing this with $\left[  x=y\right]
=1$, we obtain $\sum_{\substack{z\in P;\\x\leq z\leq y}}\mu\left(  x,z\right)
=\left[  x=y\right]  $. Hence, (\ref{eq.prop.moebius.double0.a}) is proven in
Case 1.

Let us now consider Case 2. In this case, we have $x<y$. Thus,
(\ref{eq.def.moebius.rec2}) yields
\begin{align*}
\mu\left(  x,y\right)   &  =-\underbrace{\sum_{\substack{z\in P;\\x\leq z<y}%
}}_{\substack{=\sum_{\substack{z\in P;\\x\leq z\text{ and }z<y}}=\sum
_{\substack{z\in P;\\x\leq z\text{ and }z\leq y\text{ and }z\neq
y}}\\\text{(since the condition }\left(  z<y\right)  \text{ is equivalent
to}\\\text{the condition }\left(  z\leq y\text{ and }z\neq y\right)  \text{)}%
}}\mu\left(  x,z\right)  =-\underbrace{\sum_{\substack{z\in P;\\x\leq z\text{
and }z\leq y\text{ and }z\neq y}}}_{=\sum_{\substack{z\in P;\\x\leq z\leq
y\text{ and }z\neq y}}}\mu\left(  x,z\right) \\
&  =-\sum_{\substack{z\in P;\\x\leq z\leq y\text{ and }z\neq y}}\mu\left(
x,z\right)  .
\end{align*}
Hence, $\sum_{\substack{z\in P;\\x\leq z\leq y\text{ and }z\neq y}}\mu\left(
x,z\right)  =-\mu\left(  x,y\right)  $.

On the other hand,
\[
\left\{  w\in P\ \mid\ x\leq w\leq y\text{ and }w=y\right\}  =\left\{
y\right\}
\]
\footnote{\textit{Proof.} Let $z\in\left\{  w\in P\ \mid\ x\leq w\leq y\text{
and }w=y\right\}  $. Thus, $z$ is an element $w$ of $P$ satisfying $x\leq
w\leq y$ and $w=y$. In other words, $z$ is an element of $P$ and satisfies
$x\leq z\leq y$ and $z=y$. Now, $z=y\in\left\{  y\right\}  $. Let us now
forget that we fixed $z$. We thus have proven that $z\in\left\{  y\right\}  $
for every $z\in\left\{  w\in P\ \mid\ x\leq w\leq y\text{ and }w=y\right\}  $.
In other words, $\left\{  w\in P\ \mid\ x\leq w\leq y\text{ and }w=y\right\}
\subseteq\left\{  y\right\}  $.
\par
On the other hand, $y$ is an element of $P$ and satisfies $x\leq y\leq y$ and
$y=y$. In other words, $y$ is an element $w$ of $P$ satisfying $x\leq w\leq y$
and $w=y$. Hence, $y\in\left\{  w\in P\ \mid\ x\leq w\leq y\text{ and
}w=y\right\}  $. Thus, $\left\{  y\right\}  \subseteq\left\{  w\in
P\ \mid\ x\leq w\leq y\text{ and }w=y\right\}  $. Combining this with
$\left\{  w\in P\ \mid\ x\leq w\leq y\text{ and }w=y\right\}  \subseteq
\left\{  y\right\}  $, we obtain $\left\{  w\in P\ \mid\ x\leq w\leq y\text{
and }w=y\right\}  =\left\{  y\right\}  $. Qed.}.

But every $z\in P$ satisfies either $z=y$ or $z\neq y$ (but not both). Thus,%
\begin{align*}
\sum_{\substack{z\in P;\\x\leq z\leq y}}\mu\left(  x,z\right)   &
=\underbrace{\sum_{\substack{z\in P;\\x\leq z\leq y\text{ and }z=y}%
}}_{\substack{=\sum_{z\in\left\{  w\in P\ \mid\ x\leq w\leq y\text{ and
}w=y\right\}  }=\sum_{z\in\left\{  y\right\}  }\\\text{(since }\left\{  w\in
P\ \mid\ x\leq w\leq y\text{ and }w=y\right\}  =\left\{  y\right\}  \text{)}%
}}\mu\left(  x,z\right)  +\underbrace{\sum_{\substack{z\in P;\\x\leq z\leq
y\text{ and }z\neq y}}\mu\left(  x,z\right)  }_{=-\mu\left(  x,y\right)  }\\
&  =\underbrace{\sum_{z\in\left\{  y\right\}  }\mu\left(  x,z\right)  }%
_{=\mu\left(  x,y\right)  }+\left(  -\mu\left(  x,y\right)  \right)
=\mu\left(  x,y\right)  +\left(  -\mu\left(  x,y\right)  \right)  =0.
\end{align*}
Thus, (\ref{eq.prop.moebius.double0.a}) is proven in Case 2.

We have now proven (\ref{eq.prop.moebius.double0.a}) in each of the two Cases
1 and 2. Thus, (\ref{eq.prop.moebius.double0.a}) always holds (since Cases 1
and 2 cover all possibilities). Proposition \ref{prop.moebius.double0}
\textbf{(a)} is thus proven.

\textbf{(b)} For any two elements $u$ and $v$ of $P$, we define a subset
$\left[  u,v\right]  $ of $P$ by%
\[
\left[  u,v\right]  =\left\{  w\in P\ \mid\ u\leq w\leq v\right\}  .
\]
This subset $\left[  u,v\right]  $ is finite (since $P$ is finite), and thus
its size $\left\vert \left[  u,v\right]  \right\vert $ is a nonnegative integer.

We shall now prove Proposition \ref{prop.moebius.double0} \textbf{(b)} by
strong induction on $\left\vert \left[  x,y\right]  \right\vert $:

\textit{Induction step:} Let $N\in\mathbb{N}$. Assume that Proposition
\ref{prop.moebius.double0} \textbf{(b)} holds whenever $\left\vert \left[
x,y\right]  \right\vert <N$. We must now prove that Proposition
\ref{prop.moebius.double0} \textbf{(b)} holds whenever $\left\vert \left[
x,y\right]  \right\vert =N$.

We have assumed that Proposition \ref{prop.moebius.double0} \textbf{(b)} holds
whenever $\left\vert \left[  x,y\right]  \right\vert <N$. In other words, we
have assumed the following claim:

\begin{statement}
\textit{Claim 1:} For any $x\in P$ and $y\in P$ satisfying $\left\vert \left[
x,y\right]  \right\vert <N$, we have%
\[
\sum_{\substack{z\in P;\\x\leq z\leq y}}\mu\left(  z,y\right)  =\left[
x=y\right]  .
\]

\end{statement}

Now, let $x$ and $y$ be two elements of $P$ satisfying $\left\vert \left[
x,y\right]  \right\vert =N$. We are going to prove that%
\begin{equation}
\sum_{\substack{z\in P;\\x\leq z\leq y}}\mu\left(  z,y\right)  =\left[
x=y\right]  . \label{pf.prop.moebius.double0.b.goal}%
\end{equation}

If we do not have $x\leq y$, then (\ref{pf.prop.moebius.double0.b.goal}) holds
for obvious reasons\footnote{\textit{Proof.} Assume that we do not have $x\leq
y$. Then, there exists no $z\in P$ satisfying $x\leq z\leq y$ (because if such
a $z$ would exist, then it would satisfy $x\leq z\leq y$, which would
contradict the fact that we do not have $x\leq y$). Therefore, the sum
$\sum_{\substack{z\in P;\\x\leq z\leq y}}\mu\left(  z,y\right)  $ is an empty
sum. Hence,%
\[
\sum_{\substack{z\in P;\\x\leq z\leq y}}\mu\left(  z,y\right)  =\left(
\text{empty sum}\right)  =0.
\]
But we do not have $x=y$ (because if we had $x=y$, then we would have $x\leq
y$, which would contradict the fact that we do not have $x\leq y$). Thus,
$\left[  x=y\right]  =0$. Now, $\sum_{\substack{z\in P;\\x\leq z\leq y}%
}\mu\left(  z,y\right)  =0=\left[  x=y\right]  $. Hence,
(\ref{pf.prop.moebius.double0.b.goal}) is proven, qed.}. Hence, for the rest
of our proof of (\ref{pf.prop.moebius.double0.b.goal}), we can WLOG assume
that $x\leq y$. Assume this.

If $x=y$, then (\ref{pf.prop.moebius.double0.b.goal}) holds for obvious
reasons\footnote{\textit{Proof.} Assume that $x=y$. Then,%
\[
\left\{  w\in P\ \mid\ x\leq w\leq\underbrace{y}_{=x}\right\}  =\left\{  w\in
P\ \mid\ x\leq w\leq x\right\}  =\left\{  x\right\}
\]
(by (\ref{pf.prop.moebius.double0.xx})). Now,
\[
\underbrace{\sum_{\substack{z\in P;\\x\leq z\leq y}}}_{\substack{=\sum
_{z\in\left\{  w\in P\ \mid\ x\leq w\leq y\right\}  }=\sum_{z\in\left\{
x\right\}  }\\\text{(since }\left\{  w\in P\ \mid\ x\leq w\leq y\right\}
=\left\{  x\right\}  \text{)}}}\mu\left(  z,y\right)  =\sum_{z\in\left\{
x\right\}  }\mu\left(  z,y\right)  =\mu\left(  x,\underbrace{y}_{=x}\right)
=\mu\left(  x,x\right)  =1
\]
(by (\ref{eq.def.moebius.rec1})). Comparing this with
\[
\left[  x=y\right]  =1\ \ \ \ \ \ \ \ \ \ \left(  \text{since }x=y\right)  ,
\]
we obtain $\sum_{\substack{z\in P;\\x\leq z\leq y}}\mu\left(  z,y\right)
=\left[  x=y\right]  $. Hence, (\ref{pf.prop.moebius.double0.b.goal}) is
proven, qed.}. Hence, for the rest of our proof of
(\ref{pf.prop.moebius.double0.b.goal}), we can WLOG assume that we don't have
$x=y$. Assume this. Thus, we don't have $x=y$. In other words, we have $x\neq
y$. Combining this with $x\leq y$, we obtain $x<y$.

Notice that $\left[  x=y\right]  =0$ (since we don't have $x=y$).

For any pair $\left(  z,t\right)  \in P^{2}$, we have the following logical
equivalence:%
\begin{equation}
\left(  x\leq z\leq y\text{ and }z\leq t\leq y\right)  \ \Longleftrightarrow
\ \left(  x\leq t\leq y\text{ and }x\leq z\leq t\right)
\label{pf.prop.moebius.double0.b.equiv}%
\end{equation}
\footnote{\textit{Proof of (\ref{pf.prop.moebius.double0.b.equiv}):} Let
$\left(  z,t\right)  \in P^{2}$. We shall first prove the logical implication%
\begin{equation}
\left(  x\leq z\leq y\text{ and }z\leq t\leq y\right)  \ \Longrightarrow
\ \left(  x\leq t\leq y\text{ and }x\leq z\leq t\right)  .
\label{pf.prop.moebius.double0.b.equiv.pf.1}%
\end{equation}
\par
\textit{Proof of (\ref{pf.prop.moebius.double0.b.equiv.pf.1}):} Assume that
$\left(  x\leq z\leq y\text{ and }z\leq t\leq y\right)  $ holds. We must prove
that $\left(  x\leq t\leq y\text{ and }x\leq z\leq t\right)  $ holds. We have
$x\leq z\leq t$, so that $x\leq t\leq y$. Also, $x\leq z\leq t$. Thus,
$\left(  x\leq t\leq y\text{ and }x\leq z\leq t\right)  $ holds. This proves
the implication (\ref{pf.prop.moebius.double0.b.equiv.pf.1}).
\par
Next, we shall prove the logical implication%
\begin{equation}
\left(  x\leq t\leq y\text{ and }x\leq z\leq t\right)  \ \Longrightarrow
\ \left(  x\leq z\leq y\text{ and }z\leq t\leq y\right)  .
\label{pf.prop.moebius.double0.b.equiv.pf.2}%
\end{equation}
\par
\textit{Proof of (\ref{pf.prop.moebius.double0.b.equiv.pf.2}):} Assume that
$\left(  x\leq t\leq y\text{ and }x\leq z\leq t\right)  $ holds. We must prove
that $\left(  x\leq z\leq y\text{ and }z\leq t\leq y\right)  $ holds. We have
$z\leq t\leq y$, thus $x\leq z\leq y$. Also, $z\leq t\leq y$. Thus, $\left(
x\leq z\leq y\text{ and }z\leq t\leq y\right)  $ holds. This proves the
implication (\ref{pf.prop.moebius.double0.b.equiv.pf.2}).
\par
Combining the two implications (\ref{pf.prop.moebius.double0.b.equiv.pf.1})
and (\ref{pf.prop.moebius.double0.b.equiv.pf.2}), we obtain the logical
equivalence%
\[
\left(  x\leq z\leq y\text{ and }z\leq t\leq y\right)  \ \Longleftrightarrow
\ \left(  x\leq t\leq y\text{ and }x\leq z\leq t\right)  .
\]
This proves (\ref{pf.prop.moebius.double0.b.equiv}).}.

For every $t\in P$ satisfying $x\leq t<y$, we have%
\begin{equation}
\left\vert \left[  x,t\right]  \right\vert <N
\label{pf.prop.moebius.double0.b.smaller}%
\end{equation}
\footnote{\textit{Proof of (\ref{pf.prop.moebius.double0.b.smaller}):} Let
$t\in P$ be such that $x\leq t<y$. We shall proceed in several steps:
\par
\begin{itemize}
\item We have $\left[  x,t\right]  =\left\{  w\in P\ \mid\ x\leq w\leq
t\right\}  $ (by the definition of $\left[  x,t\right]  $) and $\left[
x,y\right]  =\left\{  w\in P\ \mid\ x\leq w\leq y\right\}  $ (by the
definition of $\left[  x,y\right]  $).
\par
\item Let us first prove that $\left[  x,t\right]  \subseteq\left[
x,y\right]  $.
\par
Indeed, let $g\in\left[  x,t\right]  $. Then, $g\in\left[  x,t\right]
=\left\{  w\in P\ \mid\ x\leq w\leq t\right\}  $. In other words, $g$ is an
element $w$ of $P$ satisfying $x\leq w\leq t$. In other words, $g$ is an
element of $P$ and satisfies $x\leq g\leq t$. We have $g\leq t<y$, so that
$g\leq y$. Combining this with $x\leq g$, we obtain $x\leq g\leq y$. Thus, $g$
is an element of $P$ and satisfies $x\leq g\leq y$. In other words, $g$ is an
element $w$ of $P$ satisfying $x\leq w\leq y$. In other words, $g\in\left\{
w\in P\ \mid\ x\leq w\leq y\right\}  $. Since $\left[  x,y\right]  =\left\{
w\in P\ \mid\ x\leq w\leq y\right\}  $, this rewrites as $g\in\left[
x,y\right]  $.
\par
Now, forget that we fixed $g$. Thus, we have shown that $g\in\left[
x,y\right]  $ for each $g\in\left[  x,t\right]  $. In other words, $\left[
x,t\right]  \subseteq\left[  x,y\right]  $.
\par
\item Now, let us prove that $y\notin\left[  x,t\right]  $.
\par
Indeed, assume the contrary. Thus, $y\in\left[  x,t\right]  $. Hence,
$y\in\left[  x,t\right]  =\left\{  w\in P\ \mid\ x\leq w\leq t\right\}  $. In
other words, $y$ is an element $w$ of $P$ satisfying $x\leq w\leq t$. In other
words, $y$ is an element of $P$ and satisfies $x\leq y\leq t$. But $y\leq t$
contradicts $t<y$ (since $P$ is a poset). This contradiction proves that our
assumption was wrong. Hence, $y\notin\left[  x,t\right]  $ is proven.
\par
\item Next, let us prove that $y\in\left[  x,y\right]  $.
\par
Indeed, $y$ is an element of $P$ and satisfies $x\leq y\leq y$. In other
words, $y$ is an element $w$ of $P$ satisfying $x\leq w\leq y$. In other
words, $y\in\left\{  w\in P\ \mid\ x\leq w\leq y\right\}  $. Since $\left[
x,y\right]  =\left\{  w\in P\ \mid\ x\leq w\leq y\right\}  $, this rewrites as
$y\in\left[  x,y\right]  $.
\par
\item Let us now prove that $\left[  x,t\right]  \neq\left[  x,y\right]  $.
\par
Indeed, assume the contrary. Thus, $\left[  x,t\right]  =\left[  x,y\right]
$. Now, $y\notin\left[  x,t\right]  =\left[  x,y\right]  $ contradicts
$y\in\left[  x,y\right]  $. This contradiction proves that our assumption was
wrong. Hence, $\left[  x,t\right]  \neq\left[  x,y\right]  $ is proven.
\par
\item Combining $\left[  x,t\right]  \subseteq\left[  x,y\right]  $ with
$\left[  x,t\right]  \neq\left[  x,y\right]  $, we conclude that $\left[
x,t\right]  $ is a proper subset of $\left[  x,y\right]  $. Thus, $\left\vert
\left[  x,t\right]  \right\vert <\left\vert \left[  x,y\right]  \right\vert $
(since both $\left[  x,t\right]  $ and $\left[  x,y\right]  $ are finite
sets). Hence, $\left\vert \left[  x,t\right]  \right\vert <\left\vert \left[
x,y\right]  \right\vert =N$. This proves
(\ref{pf.prop.moebius.double0.b.smaller}).
\end{itemize}
}. Therefore, for every $t\in P$ satisfying $x\leq t<y$, we have%
\begin{equation}
\sum_{\substack{z\in P;\\x\leq z\leq t}}\mu\left(  z,t\right)  =\left[
x=t\right]  \label{pf.prop.moebius.double0.b.smaller2}%
\end{equation}
\footnote{\textit{Proof of (\ref{pf.prop.moebius.double0.b.smaller2}):} Let
$t\in P$ be such that $x\leq t<y$. Then, $\left\vert \left[  x,t\right]
\right\vert <N$ (by (\ref{pf.prop.moebius.double0.b.smaller})). Hence, Claim 1
(applied to $t$ instead of $y$) shows that $\sum_{\substack{z\in P;\\x\leq
z\leq t}}\mu\left(  z,t\right)  =\left[  x=t\right]  $. This proves
(\ref{pf.prop.moebius.double0.b.smaller2}).}. Also, for every $u\in P$ and
$v\in P$, we have%
\begin{equation}
\sum_{\substack{t\in P;\\u\leq t\leq v}}\mu\left(  u,t\right)  =\left[
u=v\right]  \label{pf.prop.moebius.double0.b.smaller3}%
\end{equation}
\footnote{\textit{Proof of (\ref{pf.prop.moebius.double0.b.smaller3}):} Let
$u\in P$ and $v\in P$. Proposition \ref{prop.moebius.double0} \textbf{(a)}
(applied to $x=u$ and $y=v$) shows that $\sum_{\substack{z\in P;\\u\leq z\leq
v}}\mu\left(  u,z\right)  =\left[  u=v\right]  $. Now,%
\begin{align*}
\sum_{\substack{t\in P;\\u\leq t\leq v}}\mu\left(  u,t\right)   &
=\sum_{\substack{z\in P;\\u\leq z\leq v}}\mu\left(  u,z\right)
\ \ \ \ \ \ \ \ \ \ \left(  \text{here, we have substituted }z\text{ for
}t\text{ in the sum}\right) \\
&  =\left[  u=v\right]  .
\end{align*}
This proves (\ref{pf.prop.moebius.double0.b.smaller3}).}.

Furthermore,
\begin{equation}
\left\{  z\in P\ \mid\ x\leq z\leq y\text{ and }z=y\right\}  =\left\{
y\right\}  \label{pf.prop.moebius.double0.b.set1}%
\end{equation}
\footnote{\textit{Proof of (\ref{pf.prop.moebius.double0.b.set1}):} We know
that $y$ is an element of $P$ and satisfies $x\leq y\leq y$ and $y=y$. In
other words, $y$ is an element $z$ of $P$ satisfying $x\leq z\leq y$ and
$z=y$. In other words, $y\in\left\{  z\in P\ \mid\ x\leq z\leq y\text{ and
}z=y\right\}  $. Thus, $\left\{  y\right\}  \subseteq\left\{  z\in
P\ \mid\ x\leq z\leq y\text{ and }z=y\right\}  $.
\par
On the other hand, let $g\in\left\{  z\in P\ \mid\ x\leq z\leq y\text{ and
}z=y\right\}  $ be arbitrary. Thus, $g$ is an element $z$ of $P$ satisfying
$x\leq z\leq y$ and $z=y$. In other words, $g$ is an element of $P$ and
satisfies $x\leq g\leq y$ and $g=y$. Thus, $g=y\in\left\{  y\right\}  $. Now,
forget that we fixed $g$. We thus have shown that $g\in\left\{  y\right\}  $
for every $g\in\left\{  z\in P\ \mid\ x\leq z\leq y\text{ and }z=y\right\}  $.
In other words, $\left\{  z\in P\ \mid\ x\leq z\leq y\text{ and }z=y\right\}
\subseteq\left\{  y\right\}  $. Combining this with $\left\{  y\right\}
\subseteq\left\{  z\in P\ \mid\ x\leq z\leq y\text{ and }z=y\right\}  $, we
obtain%
\[
\left\{  z\in P\ \mid\ x\leq z\leq y\text{ and }z=y\right\}  =\left\{
y\right\}  .
\]
This proves (\ref{pf.prop.moebius.double0.b.set1}).}. Thus,%
\begin{align}
&  \left\{  w\in P\ \mid\ x\leq w\leq y\text{ and }w=y\right\} \nonumber\\
&  =\left\{  z\in P\ \mid\ x\leq z\leq y\text{ and }z=y\right\} \nonumber\\
&  \ \ \ \ \ \ \ \ \ \ \ \ \ \ \ \ \ \ \ \ \left(  \text{here, we have renamed
the index }w\text{ as }z\right) \nonumber\\
&  =\left\{  y\right\}  . \label{pf.prop.moebius.double0.b.set2}%
\end{align}
Also,%
\begin{equation}
\left\{  z\in P\ \mid\ x\leq z\leq y\text{ and }z=x\text{ and }z\neq
y\right\}  =\left\{  x\right\}  \label{pf.prop.moebius.double0.b.set3}%
\end{equation}
\footnote{\textit{Proof of (\ref{pf.prop.moebius.double0.b.set3}):} We know
that $x$ is an element of $P$ and satisfies $x\leq x\leq y$ and $x=x$ and
$x\neq y$. In other words, $x$ is an element $z$ of $P$ satisfying $x\leq
z\leq y$ and $z=x$ and $z\neq y$. In other words, $x\in\left\{  z\in
P\ \mid\ x\leq z\leq y\text{ and }z=x\text{ and }z\neq y\right\}  $. Thus,
$\left\{  x\right\}  \subseteq\left\{  z\in P\ \mid\ x\leq z\leq y\text{ and
}z=x\text{ and }z\neq y\right\}  $.
\par
On the other hand, let $g\in\left\{  z\in P\ \mid\ x\leq z\leq y\text{ and
}z=x\text{ and }z\neq y\right\}  $ be arbitrary. Thus, $g$ is an element $z$
of $P$ satisfying $x\leq z\leq y$ and $z=x$ and $z\neq y$. In other words, $g$
is an element of $P$ and satisfies $x\leq g\leq y$ and $g=x$ and $g\neq y$.
Thus, $g=x\in\left\{  x\right\}  $. Now, forget that we fixed $g$. We thus
have shown that $g\in\left\{  x\right\}  $ for every $g\in\left\{  z\in
P\ \mid\ x\leq z\leq y\text{ and }z=x\text{ and }z\neq y\right\}  $. In other
words, $\left\{  z\in P\ \mid\ x\leq z\leq y\text{ and }z=x\text{ and }z\neq
y\right\}  \subseteq\left\{  x\right\}  $. Combining this with $\left\{
x\right\}  \subseteq\left\{  z\in P\ \mid\ x\leq z\leq y\text{ and }z=x\text{
and }z\neq y\right\}  $, we obtain%
\[
\left\{  z\in P\ \mid\ x\leq z\leq y\text{ and }z=x\text{ and }z\neq
y\right\}  =\left\{  x\right\}  .
\]
This proves (\ref{pf.prop.moebius.double0.b.set3}).}.

Now,%
\begin{align*}
&  \underbrace{\sum_{\substack{\left(  z,t\right)  \in P^{2};\\x\leq z\leq
y\text{ and }z\leq t\leq y}}}_{=\sum_{\substack{z\in P;\\x\leq z\leq
y}}\ \ \sum_{\substack{t\in P;\\z\leq t\leq y}}}\mu\left(  z,t\right) \\
&  =\sum_{\substack{z\in P;\\x\leq z\leq y}}\ \ \underbrace{\sum
_{\substack{t\in P;\\z\leq t\leq y}}\mu\left(  z,t\right)  }%
_{\substack{=\left[  z=y\right]  \\\text{(by
(\ref{pf.prop.moebius.double0.b.smaller3})}\\\text{(applied to }u=z\text{ and
}v=y\text{))}}}=\sum_{\substack{z\in P;\\x\leq z\leq y}}\left[  z=y\right] \\
&  =\sum_{\substack{z\in P;\\x\leq z\leq y\text{ and }z=y}}\underbrace{\left[
z=y\right]  }_{\substack{=1\\\text{(since }z=y\text{)}}}+\sum_{\substack{z\in
P;\\x\leq z\leq y\text{ and }z\neq y}}\underbrace{\left[  z=y\right]
}_{\substack{=0\\\text{(since }z=y\text{ is false}\\\text{(since }z\neq
y\text{))}}}\\
&  \ \ \ \ \ \ \ \ \ \ \ \ \ \ \ \ \ \ \ \ \left(  \text{since every }z\in
P\text{ satisfies either }z=y\text{ or }z\neq y\text{ (but not both)}\right)
\\
&  =\sum_{\substack{z\in P;\\x\leq z\leq y\text{ and }z=y}}1+\underbrace{\sum
_{\substack{z\in P;\\x\leq z\leq y\text{ and }z\neq y}}0}_{=0}=\sum
_{\substack{z\in P;\\x\leq z\leq y\text{ and }z=y}}1\\
&  =\left\vert \left\{  z\in P\ \mid\ x\leq z\leq y\text{ and }z=y\right\}
\right\vert \cdot1\\
&  =\left\vert \underbrace{\left\{  z\in P\ \mid\ x\leq z\leq y\text{ and
}z=y\right\}  }_{\substack{=\left\{  y\right\}  \\\text{(by
(\ref{pf.prop.moebius.double0.b.set1}))}}}\right\vert =\left\vert \left\{
y\right\}  \right\vert =1.
\end{align*}
Hence,%
\begin{align*}
1  &  =\underbrace{\sum_{\substack{\left(  z,t\right)  \in P^{2};\\x\leq z\leq
y\text{ and }z\leq t\leq y}}}_{\substack{=\sum_{\substack{\left(  z,t\right)
\in P^{2};\\x\leq t\leq y\text{ and }x\leq z\leq t}}\\\text{(because for any
}\left(  z,t\right)  \in P^{2}\text{, the}\\\text{condition }\left(  x\leq
z\leq y\text{ and }z\leq t\leq y\right)  \\\text{is equivalent to }\left(
x\leq t\leq y\text{ and }x\leq z\leq t\right)  \\\text{(by
(\ref{pf.prop.moebius.double0.b.equiv})))}}}\mu\left(  z,t\right)
=\underbrace{\sum_{\substack{\left(  z,t\right)  \in P^{2};\\x\leq t\leq
y\text{ and }x\leq z\leq t}}}_{=\sum_{\substack{t\in P;\\x\leq t\leq
y}}\ \ \sum_{\substack{z\in P;\\x\leq z\leq t}}}\mu\left(  z,t\right) \\
&  =\sum_{\substack{t\in P;\\x\leq t\leq y}}\ \ \sum_{\substack{z\in P;\\x\leq
z\leq t}}\mu\left(  z,t\right) \\
&  =\underbrace{\sum_{\substack{t\in P;\\x\leq t\leq y\text{ and }t=y}%
}}_{\substack{=\sum_{t\in\left\{  w\in P\ \mid\ x\leq w\leq y\text{ and
}w=y\right\}  }=\sum_{t\in\left\{  y\right\}  }\\\text{(since }\left\{  w\in
P\ \mid\ x\leq w\leq y\text{ and }w=y\right\}  =\left\{  y\right\}  \text{)}%
}}\ \ \sum_{\substack{z\in P;\\x\leq z\leq t}}\mu\left(  z,t\right) \\
&  \ \ \ \ \ \ \ \ \ \ +\sum_{\substack{t\in P;\\x\leq t\leq y\text{ and
}t\neq y}}\ \ \underbrace{\sum_{\substack{z\in P;\\x\leq z\leq t}}\mu\left(
z,t\right)  }_{\substack{=\left[  x=t\right]  \\\text{(by
(\ref{pf.prop.moebius.double0.b.smaller2})}\\\text{(since }t<y\text{ (because
}t\leq y\text{ and }t\neq y\text{) and }x\leq t\text{))}}}\\
&  \ \ \ \ \ \ \ \ \ \ \ \ \ \ \ \ \ \ \ \ \left(  \text{since every }t\in
P\text{ satisfies either }t=y\text{ or }t\neq y\text{ (but not both)}\right)
\\
&  =\underbrace{\sum_{t\in\left\{  y\right\}  }\ \ \sum_{\substack{z\in
P;\\x\leq z\leq t}}\mu\left(  z,t\right)  }_{=\sum_{\substack{z\in P;\\x\leq
z\leq y}}\mu\left(  z,y\right)  }+\sum_{\substack{t\in P;\\x\leq t\leq y\text{
and }t\neq y}}\left[  x=t\right] \\
&  =\sum_{\substack{z\in P;\\x\leq z\leq y}}\mu\left(  z,y\right)
+\sum_{\substack{t\in P;\\x\leq t\leq y\text{ and }t\neq y}}\left[
x=t\right]  .
\end{align*}
Subtracting $\sum_{\substack{z\in P;\\x\leq z\leq y}}\mu\left(  z,y\right)  $
from both sides of this equality, we obtain%
\begin{align*}
1-\sum_{\substack{z\in P;\\x\leq z\leq y}}\mu\left(  z,y\right)   &
=\sum_{\substack{t\in P;\\x\leq t\leq y\text{ and }t\neq y}}\left[  x=t\right]
\\
&  =\underbrace{\sum_{\substack{t\in P;\\x\leq t\leq y\text{ and }t=x\text{
and }t\neq y}}}_{\substack{=\sum_{t\in\left\{  z\in P\ \mid\ x\leq z\leq
y\text{ and }z=x\text{ and }z\neq y\right\}  }=\sum_{t\in\left\{  x\right\}
}\\\text{(since }\left\{  z\in P\ \mid\ x\leq z\leq y\text{ and }z=x\text{ and
}z\neq y\right\}  =\left\{  x\right\}  \text{)}}}\underbrace{\left[
x=t\right]  }_{\substack{=1\\\text{(since }x=t\\\text{(since }t=x\text{))}}}\\
&  \ \ \ \ \ \ \ \ \ \ +\sum_{\substack{t\in P;\\x\leq t\leq y\text{ and
}t\neq x\text{ and }t\neq y}}\underbrace{\left[  x=t\right]  }%
_{\substack{=0\\\text{(since }x=t\text{ is false}\\\text{(since }x\neq t\text{
(since }t\neq x\text{)))}}}\\
&  \ \ \ \ \ \ \ \ \ \ \ \ \ \ \ \ \ \ \ \ \left(
\begin{array}
[c]{c}%
\text{since every }t\in P\text{ satisfies}\\
\text{either }t=x\text{ or }t\neq x\text{ (but not both)}%
\end{array}
\right) \\
&  =\sum_{t\in\left\{  x\right\}  }1+\underbrace{\sum_{\substack{t\in
P;\\x\leq t\leq y\text{ and }t\neq x\text{ and }t\neq y}}0}_{=0}=\sum
_{t\in\left\{  x\right\}  }1=1.
\end{align*}
Solving this equality for $\sum_{\substack{z\in P;\\x\leq z\leq y}}\mu\left(
z,y\right)  $, we obtain%
\[
\sum_{\substack{z\in P;\\x\leq z\leq y}}\mu\left(  z,y\right)  =1-1=0=\left[
x=y\right]  .
\]
Thus, (\ref{pf.prop.moebius.double0.b.goal}) is proven.

Let us now forget that we fixed $x$ and $y$. We thus have proven that for any
$x\in P$ and $y\in P$ satisfying $\left\vert \left[  x,y\right]  \right\vert
=N$, we have%
\[
\sum_{\substack{z\in P;\\x\leq z\leq y}}\mu\left(  z,y\right)  =\left[
x=y\right]  .
\]
In other words, Proposition \ref{prop.moebius.double0} \textbf{(b)} holds
whenever $\left\vert \left[  x,y\right]  \right\vert =N$. This completes the
induction step. Thus, Proposition \ref{prop.moebius.double0} \textbf{(b)} is
proven by induction.

\textbf{(c)} For every $v\in P$, we have%
\begin{align*}
&  \sum_{\substack{y\in P;\\y\leq v}}\mu\left(  y,v\right)  \sum
_{\substack{x\in P;\\x\leq y}}\beta_{x}\\
&  =\sum_{\substack{z\in P;\\z\leq v}}\mu\left(  z,v\right)  \sum
_{\substack{x\in P;\\x\leq z}}\beta_{x}\ \ \ \ \ \ \ \ \ \ \left(
\begin{array}
[c]{c}%
\text{here, we have renamed the summation}\\
\text{index }y\text{ as }z\text{ in the outer sum}%
\end{array}
\right) \\
&  =\underbrace{\sum_{\substack{z\in P;\\z\leq v}}\ \ \sum_{\substack{x\in
P;\\x\leq z}}}_{=\sum_{\substack{\left(  x,z\right)  \in P^{2};\\z\leq v\text{
and }x\leq z}}=\sum_{x\in P}\ \ \sum_{\substack{z\in P;\\z\leq v\text{ and
}x\leq z}}}\mu\left(  z,v\right)  \beta_{x}\\
&  =\sum_{x\in P}\ \ \underbrace{\sum_{\substack{z\in P;\\z\leq v\text{ and
}x\leq z}}}_{\substack{=\sum_{\substack{z\in P;\\x\leq z\text{ and }z\leq
v}}\\=\sum_{\substack{z\in P;\\x\leq z\leq v}}}}\mu\left(  z,v\right)
\beta_{x}=\sum_{x\in P}\ \ \sum_{\substack{z\in P;\\x\leq z\leq v}}\mu\left(
z,v\right)  \beta_{x}=\sum_{x\in P}\underbrace{\left(  \sum_{\substack{z\in
P;\\x\leq z\leq v}}\mu\left(  z,v\right)  \right)  }_{\substack{=\left[
x=v\right]  \\\text{(by Proposition \ref{prop.moebius.double0} \textbf{(b)}%
}\\\text{(applied to }y=v\text{))}}}\beta_{x}\\
&  =\sum_{x\in P}\left[  x=v\right]  \beta_{x}=\sum_{\substack{x\in
P;\\x=v}}\underbrace{\left[  x=v\right]  }_{\substack{=1\\\text{(since
}x=v\text{)}}}\beta_{x}+\sum_{\substack{x\in P;\\x\neq v}}\underbrace{\left[
x=v\right]  }_{\substack{=0\\\text{(since }x=v\text{ is false}\\\text{(since
}x\neq v\text{))}}}\beta_{x}\\
&  \ \ \ \ \ \ \ \ \ \ \ \ \ \ \ \ \ \ \ \ \left(  \text{since every }x\in
P\text{ satisfies either }x=v\text{ or }x\neq v\text{ (but not both)}\right)
\\
&  =\sum_{\substack{x\in P;\\x=v}}\beta_{x}+\underbrace{\sum_{\substack{x\in
P;\\x\neq v}}0\beta_{x}}_{=0}=\sum_{\substack{x\in P;\\x=v}}\beta_{x}%
=\beta_{v}\ \ \ \ \ \ \ \ \ \ \left(  \text{since }v\in P\right)  .
\end{align*}
Renaming $v$ as $z$ in this statement, we obtain the following: For every
$z\in P$, we have%
\[
\sum_{\substack{y\in P;\\y\leq z}}\mu\left(  y,z\right)  \sum_{\substack{x\in
P;\\x\leq y}}\beta_{x}=\beta_{z}.
\]
This proves Proposition \ref{prop.moebius.double0} \textbf{(c)}.
\end{proof}
\end{verlong}

\begin{proof}
[Proof of Theorem \ref{thm.matroid.charpol.varis}.]If $T$ is a subset of $E$,
then $\overline{T}$ is a flat of $M$ (by Proposition
\ref{prop.matroid.closure.props} \textbf{(a)}). In other words, if $T$ is a
subset of $E$, then $\overline{T}\in\operatorname*{Flats}M$. Renaming $T$ as
$B$ in this statement, we conclude that if $B$ is a subset of $E$, then
$\overline{B}\in\operatorname*{Flats}M$.

For every $F\in\operatorname*{Flats}M$, define an element $\beta_{F}%
\in\mathbf{k}$ by%
\[
\beta_{F}=\sum_{\substack{B\subseteq E;\\\overline{B}=F}}\left(  -1\right)
^{\left\vert B\right\vert }\left(  \prod_{\substack{K\in\mathfrak{K}%
;\\K\subseteq B}}a_{K}\right)  .
\]
Now, using Lemma \ref{lem.matroid.NBCm.moeb}, we can easily see that%
\begin{equation}
\sum_{\substack{G\in\operatorname*{Flats}M;\\G\subseteq F}}\beta_{G}=\left[
F=\varnothing\right]  \ \ \ \ \ \ \ \ \ \ \text{for every }F\in
\operatorname*{Flats}M \label{pf.thm.matroid.charpol.varis.b-via-a}%
\end{equation}
\footnote{\textit{Proof of (\ref{pf.thm.matroid.charpol.varis.b-via-a}):} Let
$F\in\operatorname*{Flats}M$. Thus, $F$ is a flat of $M$.
\par
If $B$ is a subset of $E$, then the statements $\left(  \overline{B}\subseteq
F\right)  $ and $\left(  B\subseteq F\right)  $ are equivalent. (This follows
from Proposition \ref{prop.matroid.closure.props} \textbf{(g)}, applied to
$T=B$ and $G=F$.)
\par
Now,%
\begin{align*}
&  \sum_{\substack{G\in\operatorname*{Flats}M;\\G\subseteq F}%
}\underbrace{\beta_{G}}_{\substack{=\sum_{\substack{B\subseteq E;\\\overline
{B}=G}}\left(  -1\right)  ^{\left\vert B\right\vert }\left(  \prod
_{\substack{K\in\mathfrak{K};\\K\subseteq B}}a_{K}\right)  \\\text{(by the
definition of }\beta_{G}\text{)}}}\\
&  =\underbrace{\sum_{\substack{G\in\operatorname*{Flats}M;\\G\subseteq
F}}\ \ \sum_{\substack{B\subseteq E;\\\overline{B}=G}}}_{\substack{=\sum
_{\substack{B\subseteq E;\\\overline{B}\subseteq F}}\\\text{(because if
}B\text{ is a subset of }E\text{,}\\\text{then }\overline{B}\in
\operatorname*{Flats}M\text{)}}}\left(  -1\right)  ^{\left\vert B\right\vert
}\left(  \prod_{\substack{K\in\mathfrak{K};\\K\subseteq B}}a_{K}\right) \\
&  =\underbrace{\sum_{\substack{B\subseteq E;\\\overline{B}\subseteq F}%
}}_{\substack{=\sum_{\substack{B\subseteq E;\\B\subseteq F}}\\\text{(because
if }B\text{ is a subset of }E\text{, then}\\\text{the statements }\left(
\overline{B}\subseteq F\right)  \text{ and }\left(  B\subseteq F\right)
\text{ are}\\\text{equivalent)}}}\left(  -1\right)  ^{\left\vert B\right\vert
}\left(  \prod_{\substack{K\in\mathfrak{K};\\K\subseteq B}}a_{K}\right)
=\underbrace{\sum_{\substack{B\subseteq E;\\B\subseteq F}}}_{=\sum_{B\subseteq
F}}\left(  -1\right)  ^{\left\vert B\right\vert }\left(  \prod_{\substack{K\in
\mathfrak{K};\\K\subseteq B}}a_{K}\right) \\
&  =\sum_{B\subseteq F}\left(  -1\right)  ^{\left\vert B\right\vert }\left(
\prod_{\substack{K\in\mathfrak{K};\\K\subseteq B}}a_{K}\right)  =\left[
F=\varnothing\right]  \ \ \ \ \ \ \ \ \ \ \left(  \text{by
(\ref{eq.lem.matroid.NBCm.moeb.1})}\right)  .
\end{align*}
This proves (\ref{pf.thm.matroid.charpol.varis.b-via-a}).}.

Let $\mu$ be the M\"{o}bius function of the lattice $\operatorname*{Flats}M$.
The element $\overline{\varnothing}$ is the global minimum of the poset
$\operatorname*{Flats}M$.\ \ \ \ \footnote{This was proven during our proof of
Proposition \ref{prop.matroid.flat-lat}.} In particular, $\overline
{\varnothing}\in\operatorname*{Flats}M$ and $\overline{\varnothing}\subseteq
F$. Hence, $\mu\left(  \overline{\varnothing},F\right)  $ is well-defined.

Now, fix $F\in\operatorname*{Flats}M$. Proposition \ref{prop.moebius.double0}
\textbf{(c)} (applied to $P=\operatorname*{Flats}M$ and $z=F$) shows that%
\begin{align}
\beta_{F}  &  =\sum_{\substack{y\in\operatorname*{Flats}M;\\y\subseteq F}%
}\mu\left(  y,F\right)  \sum_{\substack{x\in\operatorname*{Flats}%
M;\\x\subseteq y}}\beta_{x}\nonumber\\
&  \ \ \ \ \ \ \ \ \ \ \ \ \ \ \ \ \ \ \ \ \left(  \text{since the relation
}\leq\text{ of the poset }\operatorname*{Flats}M\text{ is }\subseteq\right)
\nonumber\\
&  =\sum_{\substack{H\in\operatorname*{Flats}M;\\H\subseteq F}}\mu\left(
H,F\right)  \underbrace{\sum_{\substack{G\in\operatorname*{Flats}%
M;\\G\subseteq H}}\beta_{G}}_{\substack{=\left[  H=\varnothing\right]
\\\text{(by (\ref{pf.thm.matroid.charpol.varis.b-via-a}), applied to}\\H\text{
instead of }F\text{)}}}\nonumber\\
&  \ \ \ \ \ \ \ \ \ \ \ \ \ \ \ \ \ \ \ \ \left(  \text{here, we renamed the
summation indices }y\text{ and }x\text{ as }H\text{ and }G\right) \nonumber\\
&  =\sum_{\substack{H\in\operatorname*{Flats}M;\\H\subseteq F}}\mu\left(
H,F\right)  \left[  H=\varnothing\right] \nonumber\\
&  =\sum_{\substack{H\in\operatorname*{Flats}M;\\H\subseteq F;\\H=\varnothing
}}\mu\left(  H,F\right)  \underbrace{\left[  H=\varnothing\right]
}_{\substack{=1\\\text{(since }H=\varnothing\text{)}}}+\sum_{\substack{H\in
\operatorname*{Flats}M;\\H\subseteq F;\\H\neq\varnothing}}\mu\left(
H,F\right)  \underbrace{\left[  H=\varnothing\right]  }%
_{\substack{=0\\\text{(since }H\neq\varnothing\text{)}}}\nonumber\\
&  =\underbrace{\sum_{\substack{H\in\operatorname*{Flats}M;\\H\subseteq
F;\\H=\varnothing}}}_{\substack{=\sum_{\substack{H\in\operatorname*{Flats}%
M;\\H=\varnothing}}\\\text{(since the condition }H\subseteq F\\\text{is
automatically implied by}\\\text{the condition }H=\varnothing\text{)}}%
}\mu\left(  H,F\right) \nonumber\\
&  =\sum_{\substack{H\in\operatorname*{Flats}M;\\H=\varnothing}}\mu\left(
H,F\right)  . \label{pf.thm.matroid.charpol.varis.3}%
\end{align}
Now, we shall prove that
\begin{equation}
\beta_{F}=\left[  \overline{\varnothing}=\varnothing\right]  \mu\left(
\overline{\varnothing},F\right)  . \label{pf.thm.matroid.charpol.varis.5}%
\end{equation}

\textit{Proof of (\ref{pf.thm.matroid.charpol.varis.5}):} We are in one of the
following two cases:

\textit{Case 1:} We have $\overline{\varnothing}=\varnothing$.

\textit{Case 2:} We have $\overline{\varnothing}\neq\varnothing$.

Let us consider Case 1 first. In this case, we have $\overline{\varnothing
}=\varnothing$. Hence, $\varnothing=\overline{\varnothing}\in
\operatorname*{Flats}M$. Thus, the sum $\sum_{\substack{H\in
\operatorname*{Flats}M;\\H=\varnothing}}\mu\left(  H,F\right)  $ has exactly
one addend: namely, the addend for $H=\varnothing$. Thus, $\sum
_{\substack{H\in\operatorname*{Flats}M;\\H=\varnothing}}\mu\left(  H,F\right)
=\mu\left(  \underbrace{\varnothing}_{=\overline{\varnothing}},F\right)
=\mu\left(  \overline{\varnothing},F\right)  $. Thus,
(\ref{pf.thm.matroid.charpol.varis.3}) becomes $\beta_{F}=\sum_{\substack{H\in
\operatorname*{Flats}M;\\H=\varnothing}}\mu\left(  H,F\right)  =\mu\left(
\overline{\varnothing},F\right)  $. Comparing this with $\underbrace{\left[
\overline{\varnothing}=\varnothing\right]  }_{\substack{=1\\\text{(since
}\overline{\varnothing}=\varnothing\text{)}}}\mu\left(  \overline{\varnothing
},F\right)  =\mu\left(  \overline{\varnothing},F\right)  $, we obtain
$\beta_{F}=\left[  \overline{\varnothing}=\varnothing\right]  \mu\left(
\overline{\varnothing},F\right)  $. Thus,
(\ref{pf.thm.matroid.charpol.varis.5}) is proven in Case 1.

Let us now consider Case 2. In this case, we have $\overline{\varnothing}%
\neq\varnothing$. Thus, there exists no $H\in\operatorname*{Flats}M$ such that
$H=\varnothing$\ \ \ \ \footnote{\textit{Proof.} Assume the contrary. Thus,
there exists some $H\in\operatorname*{Flats}M$ such that $H=\varnothing$. In
other words, $\varnothing\in\operatorname*{Flats}M$. Hence, $\varnothing$ is a
flat of $M$. Proposition \ref{prop.matroid.closure.props} \textbf{(b)}
(applied to $G=\varnothing$) thus shows that $\overline{\varnothing
}=\varnothing$. This contradicts $\overline{\varnothing}\neq\varnothing$. This
contradiction proves that our assumption was wrong, qed.}. Hence, the sum
$\sum_{\substack{H\in\operatorname*{Flats}M;\\H=\varnothing}}\mu\left(
H,F\right)  $ is empty. Thus, $\sum_{\substack{H\in\operatorname*{Flats}%
M;\\H=\varnothing}}\mu\left(  H,F\right)  =\left(  \text{empty sum}\right)
=0$, so that (\ref{pf.thm.matroid.charpol.varis.3}) becomes $\beta_{F}%
=\sum_{\substack{H\in\operatorname*{Flats}M;\\H=\varnothing}}\mu\left(
H,F\right)  =0$. Comparing this with $\underbrace{\left[  \overline
{\varnothing}=\varnothing\right]  }_{\substack{=0\\\text{(since }%
\overline{\varnothing}\neq\varnothing\text{)}}}\mu\left(  \overline
{\varnothing},F\right)  =0$, we obtain $\beta_{F}=\left[  \overline
{\varnothing}=\varnothing\right]  \mu\left(  \overline{\varnothing},F\right)
$. Thus, (\ref{pf.thm.matroid.charpol.varis.5}) is proven in Case 2.

Now, we have proven (\ref{pf.thm.matroid.charpol.varis.5}) in both possible
Cases 1 and 2. Thus, (\ref{pf.thm.matroid.charpol.varis.5}) always holds.

Now, let us forget that we fixed $F$. We thus have proven
(\ref{pf.thm.matroid.charpol.varis.5}) for each $F\in\operatorname*{Flats}M$.

Now,%
\begin{align}
&  \sum_{F\subseteq E}\left(  -1\right)  ^{\left\vert F\right\vert }\left(
\prod_{\substack{K\in\mathfrak{K};\\K\subseteq F}}a_{K}\right)  x^{m-r_{M}%
\left(  F\right)  }\nonumber\\
&  =\underbrace{\sum_{B\subseteq E}}_{\substack{=\sum_{F\in
\operatorname*{Flats}M}\ \ \sum_{\substack{B\subseteq E;\\\overline{B}%
=F}}\\\text{(because if }B\text{ is a subset of }E\text{,}\\\text{then
}\overline{B}\in\operatorname*{Flats}M\text{)}}}\left(  -1\right)
^{\left\vert B\right\vert }\left(  \prod_{\substack{K\in\mathfrak{K}%
;\\K\subseteq B}}a_{K}\right)  \underbrace{x^{m-r_{M}\left(  B\right)  }%
}_{\substack{=x^{m-r_{M}\left(  \overline{B}\right)  }\\\text{(since
Proposition \ref{prop.matroid.closure.props} \textbf{(f)} (applied to
}T=B\text{)}\\\text{shows that }r_{M}\left(  B\right)  =r_{M}\left(
\overline{B}\right)  \text{)}}}\nonumber\\
&  \ \ \ \ \ \ \ \ \ \ \ \ \ \ \ \ \ \ \ \ \left(  \text{here, we have renamed
the summation index }F\text{ as }B\right) \nonumber\\
&  =\sum_{F\in\operatorname*{Flats}M}\ \ \sum_{\substack{B\subseteq
E;\\\overline{B}=F}}\left(  -1\right)  ^{\left\vert B\right\vert }\left(
\prod_{\substack{K\in\mathfrak{K};\\K\subseteq B}}a_{K}\right)
\underbrace{x^{m-r_{M}\left(  \overline{B}\right)  }}_{\substack{=x^{m-r_{M}%
\left(  F\right)  }\\\text{(since }\overline{B}=F\text{)}}}\nonumber\\
&  =\sum_{F\in\operatorname*{Flats}M}\ \ \underbrace{\sum
_{\substack{B\subseteq E;\\\overline{B}=F}}\left(  -1\right)  ^{\left\vert
B\right\vert }\left(  \prod_{\substack{K\in\mathfrak{K};\\K\subseteq B}%
}a_{K}\right)  }_{\substack{=\beta_{F}=\left[  \overline{\varnothing
}=\varnothing\right]  \mu\left(  \overline{\varnothing},F\right)  \\\text{(by
(\ref{pf.thm.matroid.charpol.varis.5}))}}}x^{m-r_{M}\left(  F\right)
}\nonumber\\
&  =\sum_{F\in\operatorname*{Flats}M}\left[  \overline{\varnothing
}=\varnothing\right]  \mu\left(  \overline{\varnothing},F\right)
x^{m-r_{M}\left(  F\right)  }\nonumber\\
&  =\left[  \overline{\varnothing}=\varnothing\right]  \sum_{F\in
\operatorname*{Flats}M}\mu\left(  \overline{\varnothing},F\right)
x^{m-r_{M}\left(  F\right)  }. \label{pf.thm.matroid.charpol.varis.9}%
\end{align}

But the definition of $\chi_{M}$ yields $\chi_{M}=\sum_{F\in
\operatorname*{Flats}M}\mu\left(  \overline{\varnothing},F\right)
x^{m-r_{M}\left(  F\right)  }$. The definition of $\widetilde{\chi}_{M}$
yields%
\begin{align*}
\widetilde{\chi}_{M}  &  =\left[  \overline{\varnothing}=\varnothing\right]
\underbrace{\chi_{M}}_{=\sum_{F\in\operatorname*{Flats}M}\mu\left(
\overline{\varnothing},F\right)  x^{m-r_{M}\left(  F\right)  }}=\left[
\overline{\varnothing}=\varnothing\right]  \sum_{F\in\operatorname*{Flats}%
M}\mu\left(  \overline{\varnothing},F\right)  x^{m-r_{M}\left(  F\right)  }\\
&  =\sum_{F\subseteq E}\left(  -1\right)  ^{\left\vert F\right\vert }\left(
\prod_{\substack{K\in\mathfrak{K};\\K\subseteq F}}a_{K}\right)  x^{m-r_{M}%
\left(  F\right)  }\ \ \ \ \ \ \ \ \ \ \left(  \text{by
(\ref{pf.thm.matroid.charpol.varis.9})}\right)  .
\end{align*}
This proves Theorem \ref{thm.matroid.charpol.varis}.
\end{proof}

\begin{proof}
[Proof of Corollary \ref{cor.matroid.charpol.K-free}.]Corollary
\ref{cor.matroid.charpol.K-free} can be derived from Theorem
\ref{thm.matroid.charpol.varis} in the same way as Corollary
\ref{cor.chromsym.K-free} was derived from Theorem \ref{thm.chromsym.varis}.
\end{proof}

\begin{proof}
[Proof of Theorem \ref{thm.matroid.charpol.empty}.]Theorem
\ref{thm.matroid.charpol.empty} can be derived from Theorem
\ref{thm.matroid.charpol.varis} in the same way as Theorem
\ref{thm.chromsym.empty} was derived from Theorem \ref{thm.chromsym.varis}.
\end{proof}

\begin{proof}
[Proof of Corollary \ref{cor.matroid.charpol.NBC}.]Corollary
\ref{cor.matroid.charpol.NBC} follows from Corollary
\ref{cor.matroid.charpol.K-free} when $\mathfrak{K}$ is set to be the set of
\textbf{all} broken circuits of $M$.
\end{proof}

\begin{proof}
[Proof of Corollary \ref{cor.matroid.charpol.NBCfor}.]If $F$ is a subset of
$E$ such that $F$ contains no broken circuit of $M$ as a subset, then%
\begin{equation}
r_{M}\left(  F\right)  =\left\vert F\right\vert
\label{pf.cor.matroid.charpol.NBCfor.1}%
\end{equation}
\footnote{\textit{Proof of (\ref{pf.cor.matroid.charpol.NBCfor.1}):} Let $F$
be a subset of $E$ such that $F$ contains no broken circuit of $M$ as a
subset.
\par
We shall show that $F\in\mathcal{I}$. Indeed, assume the contrary. Thus,
$F\notin\mathcal{I}$, so that $F\in\mathcal{P}\left(  E\right)  \setminus
\mathcal{I}$. Hence, there exists a circuit $C$ of $M$ such that $C\subseteq
F$ (according to Lemma \ref{lem.matroid.Cin}, applied to $Q=F$). Consider this
$C$. The set $C$ is a circuit, and thus nonempty (because the empty set is in
$\mathcal{I}$). Let $e$ be the unique element of $C$ having maximum label.
(This is clearly well-defined, since the labeling function $\ell$ is
injective.) Then, $C\setminus\left\{  e\right\}  $ is a broken circuit of $M$
(by the definition of a broken circuit). Thus, $F$ contains a broken circuit
of $M$ as a subset (since $C\setminus\left\{  e\right\}  \subseteq C\subseteq
F$). This contradicts the fact that $F$ contains no broken circuit of $M$ as a
subset. This contradiction shows that our assumption was wrong. Hence,
$F\in\mathcal{I}$ is proven.
\par
Thus, Lemma \ref{lem.matroid.rI} (applied to $T=F$) shows that $r_{M}\left(
F\right)  =\left\vert F\right\vert $, qed.}. Now, Corollary
\ref{cor.matroid.charpol.NBC} yields%
\[
\widetilde{\chi}_{M}=\sum_{\substack{F\subseteq E;\\F\text{ contains no
broken}\\\text{circuit of }M\text{ as a subset}}}\left(  -1\right)
^{\left\vert F\right\vert }\underbrace{x^{m-r_{M}\left(  F\right)  }%
}_{\substack{=x^{m-\left\vert F\right\vert }\\\text{(by
(\ref{pf.cor.matroid.charpol.NBCfor.1}))}}}=\sum_{\substack{F\subseteq
E;\\F\text{ contains no broken}\\\text{circuit of }M\text{ as a subset}%
}}\left(  -1\right)  ^{\left\vert F\right\vert }x^{m-\left\vert F\right\vert
}.
\]
This proves Corollary \ref{cor.matroid.charpol.NBCfor}.
\end{proof}

\subsection{\label{subsec.matroid.alt-sum}A vanishing alternating sum for
matroids}

As an application of the above, we can prove an analogue of the alternating
sum identity of Dahlberg and van Willigenburg (Theorem \ref{thm.dahwil.graph}
above) for the characteristic polynomials of matroids:

\begin{theorem}
\label{thm.dahwil.matroid}Let $M=\left(  E,\mathcal{I}\right)  $ be a matroid.
Let $m=r_{M}\left(  E\right)  $. Let $C$ be a circuit of $M$, and let $e\in C$
be arbitrary. Then,%
\[
\sum_{F\subseteq C\setminus\left\{  e\right\}  }\left(  -1\right)
^{\left\vert F\right\vert }x^{m-r_{M}\left(  E\setminus F\right)  }%
\cdot\widetilde{\chi}_{M\setminus F}=0.
\]
Here, whenever $F$ is a subset of $E$, the notation $M\setminus F$ denotes the
matroid $\left(  E\setminus F,\ \mathcal{I}\cap\mathcal{P}\left(  E\setminus
F\right)  \right)  $ (that is, the matroid whose ground set is $E\setminus F$
and whose independent sets are those subsets of $E\setminus F$ that are
independent in $M$).
\end{theorem}

\begin{vershort}
\begin{proof}
[Proof of Theorem \ref{thm.dahwil.matroid}.]Quite similar to our above proof
of Theorem \ref{thm.dahwil.gen}, but using Theorem
\ref{thm.matroid.charpol.empty} and Corollary \ref{cor.matroid.charpol.K-free}
instead of Theorem \ref{thm.genambichromsym.empty} and Corollary
\ref{cor.genambichromsym.K-free}. (We also need to observe that the rank
function $r_{M\setminus F}$ is a restriction of $r_{M}$ whenever $F$ is a
subset of $E$.) We leave all details to the reader.
\end{proof}
\end{vershort}

\begin{verlong}
\begin{proof}
[Proof of Theorem \ref{thm.dahwil.matroid}.]This proof is rather similar to
our above proof of Theorem \ref{thm.dahwil.gen}, but using Theorem
\ref{thm.matroid.charpol.empty} and Corollary \ref{cor.matroid.charpol.K-free}
instead of Theorem \ref{thm.genambichromsym.empty} and Corollary
\ref{cor.genambichromsym.K-free}. Here is the argument in detail:

Let us set $B:=C\setminus\left\{  e\right\}  $. Thus, $B=C\setminus\left\{
e\right\}  \subseteq C\subseteq E$.

Now, we shall show the following:

\begin{statement}
\textit{Claim 1:} Let $J$ be a subset of $B$. Then,%
\[
x^{m-r_{M}\left(  E\setminus J\right)  }\cdot\widetilde{\chi}_{M\setminus
J}=\sum_{\substack{F\subseteq E;\\J\subseteq E\setminus F}}\left(  -1\right)
^{\left\vert F\right\vert }x^{m-r_{M}\left(  F\right)  }.
\]

\end{statement}

[\textit{Proof of Claim 1:} First, we observe that every subset $F$ of
$E\setminus J$ satisfies
\begin{equation}
r_{M\setminus J}\left(  F\right)  =r_{M}\left(  F\right)
\label{pf.thm.dahlwil.matroid.c1.pf.1}%
\end{equation}
\footnote{\textit{Proof.} Let $F$ be a subset of $E\setminus J$. Then,
$F\subseteq E\setminus J\subseteq E$. Hence, the definition of the rank
function $r_{M}$ yields%
\begin{equation}
r_{M}\left(  F\right)  =\max\left\{  \left\vert Z\right\vert \ \mid
\ Z\in\mathcal{I}\text{ and }Z\subseteq F\right\}  .
\label{pf.thm.dahlwil.matroid.c1.pf.1.pf.1}%
\end{equation}
\par
Furthermore, the definition of the matroid $M\setminus J$ yields $M\setminus
J=\left(  E\setminus J,\ \mathcal{I}\cap\mathcal{P}\left(  E\setminus
J\right)  \right)  $. Hence, the definition of its rank function
$r_{M\setminus J}$ yields%
\begin{equation}
r_{M\setminus J}\left(  F\right)  =\max\left\{  \left\vert Z\right\vert
\ \mid\ Z\in\mathcal{I}\cap\mathcal{P}\left(  E\setminus J\right)  \text{ and
}Z\subseteq F\right\}  . \label{pf.thm.dahlwil.matroid.c1.pf.1.pf.2}%
\end{equation}
\par
However, we can make the following two observations:
\par
\begin{itemize}
\item Each $Z\in\mathcal{I}$ that satisfies $Z\subseteq F$ must automatically
satisfy $Z\in\mathcal{I}\cap\mathcal{P}\left(  E\setminus J\right)  $ (because
$Z\subseteq F\subseteq E\setminus J$ yields $Z\in\mathcal{P}\left(  E\setminus
J\right)  $, and thus we can combine $Z\in\mathcal{I}$ with $Z\in
\mathcal{P}\left(  E\setminus J\right)  $ to obtain $Z\in\mathcal{I}%
\cap\mathcal{P}\left(  E\setminus J\right)  $). Hence, each $Z\in\mathcal{I}$
that satisfies $Z\subseteq F$ must be a $Z\in\mathcal{I}\cap\mathcal{P}\left(
E\setminus J\right)  $ that satisfies $Z\subseteq F$. In other words,%
\[
\left\{  Z\in\mathcal{I}\ \mid\ Z\subseteq F\right\}  \subseteq\left\{
Z\in\mathcal{I}\cap\mathcal{P}\left(  E\setminus J\right)  \ \mid\ Z\subseteq
F\right\}  .
\]
Therefore,%
\begin{align}
&  \left\{  \left\vert Z\right\vert \ \mid\ Z\in\mathcal{I}\text{ and
}Z\subseteq F\right\} \nonumber\\
&  \subseteq\left\{  \left\vert Z\right\vert \ \mid\ Z\in\mathcal{I}%
\cap\mathcal{P}\left(  E\setminus J\right)  \text{ and }Z\subseteq F\right\}
. \label{pf.thm.dahlwil.matroid.c1.pf.1.pf.5}%
\end{align}
\par
\item Each $Z\in\mathcal{I}\cap\mathcal{P}\left(  E\setminus J\right)  $ that
satisfies $Z\subseteq F$ must automatically satisfy $Z\in\mathcal{I}$ (because
$Z\in\mathcal{I}\cap\mathcal{P}\left(  E\setminus J\right)  \subseteq
\mathcal{I}$). Hence, each $Z\in\mathcal{I}\cap\mathcal{P}\left(  E\setminus
J\right)  $ that satisfies $Z\subseteq F$ must be a $Z\in\mathcal{I}$ that
satisfies $Z\subseteq F$. In other words,%
\[
\left\{  Z\in\mathcal{I}\cap\mathcal{P}\left(  E\setminus J\right)
\ \mid\ Z\subseteq F\right\}  \subseteq\left\{  Z\in\mathcal{I}\ \mid
\ Z\subseteq F\right\}  .
\]
Therefore,%
\begin{align}
&  \left\{  \left\vert Z\right\vert \ \mid\ Z\in\mathcal{I}\cap\mathcal{P}%
\left(  E\setminus J\right)  \text{ and }Z\subseteq F\right\} \nonumber\\
&  \subseteq\left\{  \left\vert Z\right\vert \ \mid\ Z\in\mathcal{I}\text{ and
}Z\subseteq F\right\}  . \label{pf.thm.dahlwil.matroid.c1.pf.1.pf.6}%
\end{align}
\end{itemize}
\par
Combining (\ref{pf.thm.dahlwil.matroid.c1.pf.1.pf.5}) with
(\ref{pf.thm.dahlwil.matroid.c1.pf.1.pf.6}), we obtain%
\[
\left\{  \left\vert Z\right\vert \ \mid\ Z\in\mathcal{I}\text{ and }Z\subseteq
F\right\}  =\left\{  \left\vert Z\right\vert \ \mid\ Z\in\mathcal{I}%
\cap\mathcal{P}\left(  E\setminus J\right)  \text{ and }Z\subseteq F\right\}
.
\]
Hence, we can rewrite (\ref{pf.thm.dahlwil.matroid.c1.pf.1.pf.1}) as
\[
r_{M}\left(  F\right)  =\max\left\{  \left\vert Z\right\vert \ \mid
\ Z\in\mathcal{I}\cap\mathcal{P}\left(  E\setminus J\right)  \text{ and
}Z\subseteq F\right\}  .
\]
Comparing this with (\ref{pf.thm.dahlwil.matroid.c1.pf.1.pf.2}), we obtain
$r_{M\setminus J}\left(  F\right)  =r_{M}\left(  F\right)  $. This proves
(\ref{pf.thm.dahlwil.matroid.c1.pf.1}).}.

Applying this to $F=E\setminus J$, we obtain $r_{M\setminus J}\left(
E\setminus J\right)  =r_{M}\left(  E\setminus J\right)  $. In other words,
$r_{M}\left(  E\setminus J\right)  =r_{M\setminus J}\left(  E\setminus
J\right)  $.

However, the definition of the matroid $M\setminus J$ yields $M\setminus
J=\left(  E\setminus J,\ \mathcal{I}\cap\mathcal{P}\left(  E\setminus
J\right)  \right)  $. Hence, Theorem \ref{thm.matroid.charpol.empty} (applied
to $M\setminus J$, $E\setminus J$, $\mathcal{I}\cap\mathcal{P}\left(
E\setminus J\right)  $ and $r_{M}\left(  E\setminus J\right)  $ instead of
$M$, $E$, $\mathcal{I}$ and $m$) yields%
\begin{align}
\widetilde{\chi}_{M\setminus J}  &  =\sum_{F\subseteq E\setminus J}\left(
-1\right)  ^{\left\vert F\right\vert }\underbrace{x^{r_{M}\left(  E\setminus
J\right)  -r_{M\setminus J}\left(  F\right)  }}_{\substack{=x^{r_{M}\left(
E\setminus J\right)  -r_{M}\left(  F\right)  }\\\text{(since }r_{M\setminus
J}\left(  F\right)  =r_{M}\left(  F\right)  \\\text{(by
(\ref{pf.thm.dahlwil.matroid.c1.pf.1}))}}}\ \ \ \ \ \ \ \ \ \ \left(
\text{since }r_{M}\left(  E\setminus J\right)  =r_{M\setminus J}\left(
E\setminus J\right)  \right) \nonumber\\
&  =\sum_{F\subseteq E\setminus J}\left(  -1\right)  ^{\left\vert F\right\vert
}x^{r_{M}\left(  E\setminus J\right)  -r_{M}\left(  F\right)  }.
\label{pf.thm.dahwil.matroid.c1.pf.3}%
\end{align}

However, we have $J\subseteq B\subseteq E$. Thus, it is easy to see that
\begin{equation}
\mathcal{P}\left(  E\setminus J\right)  =\left\{  Z\in\mathcal{P}\left(
E\right)  \ \mid\ J\subseteq E\setminus Z\right\}
\label{pf.thm.dahwil.matroid.c1.pf.2}%
\end{equation}
\footnote{\textit{Proof:} Let $H\in\mathcal{P}\left(  E\setminus J\right)  $.
Thus, $H$ is a subset of $E\setminus J$. Hence, $H\subseteq E\setminus
J\subseteq E$, so that $H\in\mathcal{P}\left(  E\right)  $. Also, from
$H\subseteq E\setminus J$, we conclude that $H$ is disjoint from $J$. In other
words, $J$ is disjoint from $H$. Combining this with $J\subseteq E$, we obtain
$J\subseteq E\setminus H$.
\par
Now, we know that $H\in\mathcal{P}\left(  E\right)  $ and $J\subseteq
E\setminus H$. In other words, $H$ is a $Z\in\mathcal{P}\left(  E\right)  $
satisfying $J\subseteq E\setminus Z$. In other words, $H\in\left\{
Z\in\mathcal{P}\left(  E\right)  \ \mid\ J\subseteq E\setminus Z\right\}  $.
\par
Forget that we fixed $H$. We thus have shown that $H\in\left\{  Z\in
\mathcal{P}\left(  E\right)  \ \mid\ J\subseteq E\setminus Z\right\}  $ for
each $H\in\mathcal{P}\left(  E\setminus J\right)  $. In other words,%
\begin{equation}
\mathcal{P}\left(  E\setminus J\right)  \subseteq\left\{  Z\in\mathcal{P}%
\left(  E\right)  \ \mid\ J\subseteq E\setminus Z\right\}  .
\label{pf.thm.dahwil.matroid.c1.pf.2.pf.1}%
\end{equation}
\par
On the other hand, let $U\in\left\{  Z\in\mathcal{P}\left(  E\right)
\ \mid\ J\subseteq E\setminus Z\right\}  $. Thus, $U$ is a $Z\in
\mathcal{P}\left(  E\right)  $ satisfying $J\subseteq E\setminus Z$. In other
words, $U\in\mathcal{P}\left(  E\right)  $ and $J\subseteq E\setminus U$.
\par
From $J\subseteq E\setminus U$, we conclude that $J$ is disjoint from $U$. In
other words, $U$ is disjoint from $J$. However, from $U\in\mathcal{P}\left(
E\right)  $, we conclude that $U$ is a subset of $E$. Thus, $U$ is a subset of
$E$ that is disjoint from $J$. In other words, $U$ is a subset of $E\setminus
J$. In other words, $U\in\mathcal{P}\left(  E\setminus J\right)  $.
\par
Forget that we fixed $U$. We thus have shown that $U\in\mathcal{P}\left(
E\setminus J\right)  $ for each $U\in\left\{  Z\in\mathcal{P}\left(  E\right)
\ \mid\ J\subseteq E\setminus Z\right\}  $. In other words,%
\[
\left\{  Z\in\mathcal{P}\left(  E\right)  \ \mid\ J\subseteq E\setminus
Z\right\}  \subseteq\mathcal{P}\left(  E\setminus J\right)  .
\]
Combining this with (\ref{pf.thm.dahwil.matroid.c1.pf.2.pf.1}), we obtain
\[
\mathcal{P}\left(  E\setminus J\right)  =\left\{  Z\in\mathcal{P}\left(
E\right)  \ \mid\ J\subseteq E\setminus Z\right\}  .
\]
This proves (\ref{pf.thm.dahwil.matroid.c1.pf.2}).}.

Now, we have the following equality between summation signs:%
\begin{align*}
\sum_{F\subseteq E\setminus J}  &  =\sum_{F\in\mathcal{P}\left(  E\setminus
J\right)  }=\sum_{F\in\left\{  Z\in\mathcal{P}\left(  E\right)  \ \mid
\ J\subseteq E\setminus Z\right\}  }\ \ \ \ \ \ \ \ \ \ \left(  \text{by
(\ref{pf.thm.dahwil.matroid.c1.pf.2})}\right) \\
&  =\sum_{\substack{F\in\mathcal{P}\left(  E\right)  ;\\J\subseteq E\setminus
F}}=\sum_{\substack{F\subseteq E;\\J\subseteq E\setminus F}}.
\end{align*}
Hence, (\ref{pf.thm.dahwil.matroid.c1.pf.3}) becomes%
\[
\widetilde{\chi}_{M\setminus J}=\underbrace{\sum_{F\subseteq E\setminus J}%
}_{=\sum_{\substack{F\subseteq E;\\J\subseteq E\setminus F}}}\left(
-1\right)  ^{\left\vert F\right\vert }x^{r_{M}\left(  E\setminus J\right)
-r_{M}\left(  F\right)  }=\sum_{\substack{F\subseteq E;\\J\subseteq E\setminus
F}}\left(  -1\right)  ^{\left\vert F\right\vert }x^{r_{M}\left(  E\setminus
J\right)  -r_{M}\left(  F\right)  }.
\]
Multiplying both sides of this equality by $x^{m-r_{M}\left(  E\setminus
J\right)  }$, we obtain
\begin{align*}
x^{m-r_{M}\left(  E\setminus J\right)  }\cdot\widetilde{\chi}_{M\setminus J}
&  =x^{m-r_{M}\left(  E\setminus J\right)  }\cdot\sum_{\substack{F\subseteq
E;\\J\subseteq E\setminus F}}\left(  -1\right)  ^{\left\vert F\right\vert
}x^{r_{M}\left(  E\setminus J\right)  -r_{M}\left(  F\right)  }\\
&  =\sum_{\substack{F\subseteq E;\\J\subseteq E\setminus F}}\left(  -1\right)
^{\left\vert F\right\vert }\underbrace{x^{m-r_{M}\left(  E\setminus J\right)
}\cdot x^{r_{M}\left(  E\setminus J\right)  -r_{M}\left(  F\right)  }%
}_{\substack{=x^{\left(  m-r_{M}\left(  E\setminus J\right)  \right)  +\left(
r_{M}\left(  E\setminus J\right)  -r_{M}\left(  F\right)  \right)
}\\=x^{m-r_{M}\left(  F\right)  }\\\text{(since }\left(  m-r_{M}\left(
E\setminus J\right)  \right)  +\left(  r_{M}\left(  E\setminus J\right)
-r_{M}\left(  F\right)  \right)  =m-r_{M}\left(  F\right)  \text{)}}}\\
&  =\sum_{\substack{F\subseteq E;\\J\subseteq E\setminus F}}\left(  -1\right)
^{\left\vert F\right\vert }x^{m-r_{M}\left(  F\right)  }.
\end{align*}
Thus, Claim 1 is proved.]

\begin{statement}
\textit{Claim 2:} We have
\begin{equation}
\sum_{\substack{F\subseteq E;\\B\subseteq F}}\left(  -1\right)  ^{\left\vert
F\right\vert }x^{m-r_{M}\left(  F\right)  }=0.
\label{pf.thm.dahwil.matroid.0=}%
\end{equation}

\end{statement}

[\textit{Proof of Claim 2:} Theorem \ref{thm.matroid.charpol.empty} yields%
\begin{align}
\widetilde{\chi}_{M}  &  =\sum_{F\subseteq E}\left(  -1\right)  ^{\left\vert
F\right\vert }x^{m-r_{M}\left(  F\right)  }\nonumber\\
&  =\sum_{\substack{F\subseteq E;\\B\subseteq F}}\left(  -1\right)
^{\left\vert F\right\vert }x^{m-r_{M}\left(  F\right)  }+\sum
_{\substack{F\subseteq E;\\B\not \subseteq F}}\left(  -1\right)  ^{\left\vert
F\right\vert }x^{m-r_{M}\left(  F\right)  }
\label{pf.thm.dahwil.matroid.XiG.c2.pf.2}%
\end{align}
(since each subset $F$ of $E$ satisfies either $B\subseteq F$ or
$B\not \subseteq F$, but not both at the same time).

Let us now find a different formula for $\widetilde{\chi}_{M}$. We define a
function $\ell:E\rightarrow\mathbb{N}$ by setting%
\[
\ell\left(  f\right)  =\left[  f=e\right]  \ \ \ \ \ \ \ \ \ \ \text{for each
}f\in E.
\]
We shall use this function $\ell$ as our labeling function (where the role of
the totally ordered set $X$ is played by $\mathbb{N}$ equipped with the usual
total order). It is easy to see that the element $e$ is the unique element of
$C$ having maximum label\footnote{\textit{Proof.} Clearly, $e$ is an element
of $C$ (since $e\in C$). We shall now show that $e$ has a larger label than
any other element of $C$.
\par
Indeed, let $f$ be any element of $C$ distinct from $e$. Then, the definition
of $\ell$ yields $\ell\left(  f\right)  =\left[  f=e\right]  =0$ (since we
don't have $f=e$ (because $f$ is distinct from $e$)). On the other hand, the
definition of $\ell$ yields $\ell\left(  e\right)  =\left[  e=e\right]  =1$
(since $e=e$). Hence, $\ell\left(  e\right)  =1>0=\ell\left(  f\right)  $. In
other words, $e$ has a larger label than $f$ (since the label of $e$ is
$\ell\left(  e\right)  $, whereas the label of $f$ is $\ell\left(  f\right)
$).
\par
Forget that we fixed $f$. We thus have shown that $e$ has a larger label than
$f$ whenever $f$ is any element of $C$ distinct from $e$. In other words, $e$
has a larger label than any other element of $C$. Hence, $e$ is the unique
element of $C$ having maximum label (since $e$ itself is an element of $C$).}.
Therefore, $C\setminus\left\{  e\right\}  $ is a broken circuit of $M$ (by the
definition of a \textquotedblleft broken circuit\textquotedblright\ in
Definition \ref{def.matroid.BC}). In other words, $B$ is a broken circuit of
$M$ (since $B=C\setminus\left\{  e\right\}  $). Hence, $\left\{  B\right\}  $
is a set of broken circuits of $M$. Therefore, Corollary
\ref{cor.matroid.charpol.K-free} (applied to $\mathfrak{K}=\left\{  B\right\}
$) yields
\begin{equation}
\widetilde{\chi}_{M}=\sum_{\substack{F\subseteq E;\\F\text{ is }\left\{
B\right\}  \text{-free}}}\left(  -1\right)  ^{\left\vert F\right\vert
}x^{m-r_{M}\left(  F\right)  }. \label{pf.thm.dahwil.matroid.XiG.c2.pf.3}%
\end{equation}

However, if $F$ is a subset of $E$, then the condition \textquotedblleft$F$ is
$\left\{  B\right\}  $-free\textquotedblright\ is equivalent to
\textquotedblleft$B\not \subseteq F$\textquotedblright%
\ \ \ \ \footnote{\textit{Proof.} Let $F$ be a subset of $E$. Then, we have
the following chain of logical equivalences:%
\begin{align*}
&  \ \left(  F\text{ is }\left\{  B\right\}  \text{-free}\right) \\
&  \Longleftrightarrow\ \left(  F\text{ contains no }K\in\left\{  B\right\}
\text{ as a subset}\right) \\
&  \ \ \ \ \ \ \ \ \ \ \ \ \ \ \ \ \ \ \ \ \left(  \text{by the definition of
\textquotedblleft}\left\{  B\right\}  \text{-free\textquotedblright\ in
Definition \ref{def.K-free}}\right) \\
&  \Longleftrightarrow\ \left(  \text{there exists no }K\in\left\{  B\right\}
\text{ such that }F\text{ contains }K\text{ as a subset}\right) \\
&  \Longleftrightarrow\ \left(  \text{there exists no }K\in\left\{  B\right\}
\text{ such that }K\subseteq F\right) \\
&  \Longleftrightarrow\ \left(  \text{each }K\in\left\{  B\right\}  \text{
satisfies }K\not \subseteq F\right) \\
&  \Longleftrightarrow\ \left(  B\not \subseteq F\right)
\ \ \ \ \ \ \ \ \ \ \left(  \text{since the only }K\in\left\{  B\right\}
\text{ is }B\right)  .
\end{align*}
Hence, the condition \textquotedblleft$F$ is $\left\{  B\right\}
$-free\textquotedblright\ is equivalent to \textquotedblleft$B\not \subseteq
F$\textquotedblright. Qed.}. Hence, the summation sign \textquotedblleft%
$\sum_{\substack{F\subseteq E;\\F\text{ is }\left\{  B\right\}  \text{-free}%
}}$\textquotedblright\ can be rewritten as \textquotedblleft$\sum
_{\substack{F\subseteq E;\\B\not \subseteq F}}$\textquotedblright. Therefore,
we can rewrite (\ref{pf.thm.dahwil.matroid.XiG.c2.pf.3}) as%
\begin{equation}
\widetilde{\chi}_{M}=\sum_{\substack{F\subseteq E;\\B\not \subseteq F}}\left(
-1\right)  ^{\left\vert F\right\vert }x^{m-r_{M}\left(  F\right)  }.
\label{pf.thm.dahwil.matroid.XiG.c2.pf.4}%
\end{equation}

Subtracting this equality from (\ref{pf.thm.dahwil.matroid.XiG.c2.pf.2}), we
obtain%
\begin{align*}
\widetilde{\chi}_{M}-\widetilde{\chi}_{M}  &  =\left(  \sum
_{\substack{F\subseteq E;\\B\subseteq F}}\left(  -1\right)  ^{\left\vert
F\right\vert }x^{m-r_{M}\left(  F\right)  }+\sum_{\substack{F\subseteq
E;\\B\not \subseteq F}}\left(  -1\right)  ^{\left\vert F\right\vert
}x^{m-r_{M}\left(  F\right)  }\right)  -\sum_{\substack{F\subseteq
E;\\B\not \subseteq F}}\left(  -1\right)  ^{\left\vert F\right\vert
}x^{m-r_{M}\left(  F\right)  }\\
&  =\sum_{\substack{F\subseteq E;\\B\subseteq F}}\left(  -1\right)
^{\left\vert F\right\vert }x^{m-r_{M}\left(  F\right)  }.
\end{align*}
Comparing this with $\widetilde{\chi}_{M}-\widetilde{\chi}_{M}=0$, we obtain%
\[
\sum_{\substack{F\subseteq E;\\B\subseteq F}}\left(  -1\right)  ^{\left\vert
F\right\vert }x^{m-r_{M}\left(  F\right)  }=0.
\]
This proves Claim 2.]

However, from $C\setminus\left\{  e\right\}  =B$, we obtain%
\begin{align*}
&  \sum_{F\subseteq C\setminus\left\{  e\right\}  }\left(  -1\right)
^{\left\vert F\right\vert }x^{m-r_{M}\left(  E\setminus F\right)  }%
\cdot\widetilde{\chi}_{M\setminus F}\\
&  =\sum_{F\subseteq B}\left(  -1\right)  ^{\left\vert F\right\vert
}x^{m-r_{M}\left(  E\setminus F\right)  }\cdot\widetilde{\chi}_{M\setminus
F}\\
&  =\sum_{J\subseteq B}\left(  -1\right)  ^{\left\vert J\right\vert
}\underbrace{x^{m-r_{M}\left(  E\setminus J\right)  }\cdot\widetilde{\chi
}_{M\setminus J}}_{\substack{=\sum_{\substack{F\subseteq E;\\J\subseteq
E\setminus F}}\left(  -1\right)  ^{\left\vert F\right\vert }x^{m-r_{M}\left(
F\right)  }\\\text{(by Claim 1)}}}\\
&  \ \ \ \ \ \ \ \ \ \ \ \ \ \ \ \ \ \ \ \ \left(  \text{here, we have renamed
the summation index }F\text{ as }J\right) \\
&  =\sum_{J\subseteq B}\left(  -1\right)  ^{\left\vert J\right\vert }%
\sum_{\substack{F\subseteq E;\\J\subseteq E\setminus F}}\left(  -1\right)
^{\left\vert F\right\vert }x^{m-r_{M}\left(  F\right)  }\\
&  =\underbrace{\sum_{J\subseteq B}\ \ \sum_{\substack{F\subseteq
E;\\J\subseteq E\setminus F}}}_{=\sum_{F\subseteq E}\ \ \sum
_{\substack{J\subseteq B;\\J\subseteq E\setminus F}}}\left(  -1\right)
^{\left\vert J\right\vert }\left(  -1\right)  ^{\left\vert F\right\vert
}x^{m-r_{M}\left(  F\right)  }\\
&  =\sum_{F\subseteq E}\ \ \underbrace{\sum_{\substack{J\subseteq
B;\\J\subseteq E\setminus F}}}_{\substack{=\sum_{J\subseteq B\cap\left(
E\setminus F\right)  }}}\left(  -1\right)  ^{\left\vert J\right\vert }\left(
-1\right)  ^{\left\vert F\right\vert }x^{m-r_{M}\left(  F\right)  }\\
&  =\sum_{F\subseteq E}\ \ \underbrace{\sum_{J\subseteq B\cap\left(
E\setminus F\right)  }\left(  -1\right)  ^{\left\vert J\right\vert }%
}_{\substack{=\sum_{I\subseteq B\cap\left(  E\setminus F\right)  }\left(
-1\right)  ^{\left\vert I\right\vert }\\\text{(here, we have renamed}%
\\\text{the summation index }J\text{ as }I\text{)}}}\left(  -1\right)
^{\left\vert F\right\vert }x^{m-r_{M}\left(  F\right)  }\\
&  =\sum_{F\subseteq E}\ \ \underbrace{\sum_{I\subseteq B\cap\left(
E\setminus F\right)  }\left(  -1\right)  ^{\left\vert I\right\vert }%
}_{\substack{=\left[  B\cap\left(  E\setminus F\right)  =\varnothing\right]
\\\text{(by Lemma \ref{lem.cancel},}\\\text{applied to }S=B\cap\left(
E\setminus F\right)  \text{)}}}\left(  -1\right)  ^{\left\vert F\right\vert
}x^{m-r_{M}\left(  F\right)  }\\
&  =\sum_{F\subseteq E}\left[  \underbrace{B\cap\left(  E\setminus F\right)
}_{=\left(  B\cap E\right)  \setminus F}=\varnothing\right]  \left(
-1\right)  ^{\left\vert F\right\vert }x^{m-r_{M}\left(  F\right)  }\\
&  =\sum_{F\subseteq E}\left[  \underbrace{\left(  B\cap E\right)
}_{\substack{=B\\\text{(since }B\subseteq E\text{)}}}\setminus F=\varnothing
\right]  \left(  -1\right)  ^{\left\vert F\right\vert }x^{m-r_{M}\left(
F\right)  }%
\end{align*}%
\begin{align*}
&  =\sum_{F\subseteq E}\underbrace{\left[  B\setminus F=\varnothing\right]
}_{\substack{=\left[  B\subseteq F\right]  \\\text{(since the statement
\textquotedblleft}B\setminus F=\varnothing\text{\textquotedblright}\\\text{is
equivalent to \textquotedblleft}B\subseteq F\text{\textquotedblright)}%
}}\left(  -1\right)  ^{\left\vert F\right\vert }x^{m-r_{M}\left(  F\right)
}\\
&  =\sum_{F\subseteq E}\left[  B\subseteq F\right]  \left(  -1\right)
^{\left\vert F\right\vert }x^{m-r_{M}\left(  F\right)  }\\
&  =\sum_{\substack{F\subseteq E;\\B\subseteq F}}\underbrace{\left[
B\subseteq F\right]  }_{\substack{=1\\\text{(since }B\subseteq F\text{)}%
}}\left(  -1\right)  ^{\left\vert F\right\vert }x^{m-r_{M}\left(  F\right)
}+\sum_{\substack{F\subseteq E;\\\text{we don't have }B\subseteq
F}}\underbrace{\left[  B\subseteq F\right]  }_{\substack{=0\\\text{(since we
don't}\\\text{have }B\subseteq F\text{)}}}\left(  -1\right)  ^{\left\vert
F\right\vert }x^{m-r_{M}\left(  F\right)  }\\
&  \ \ \ \ \ \ \ \ \ \ \ \ \ \ \ \ \ \ \ \ \left(  \text{since each subset
}F\text{ of }E\text{ either satisfies }B\subseteq F\text{ or does not}\right)
\\
&  =\sum_{\substack{F\subseteq E;\\B\subseteq F}}\left(  -1\right)
^{\left\vert F\right\vert }x^{m-r_{M}\left(  F\right)  }+\underbrace{\sum
_{\substack{F\subseteq E;\\\text{we don't have }B\subseteq F}}0\left(
-1\right)  ^{\left\vert F\right\vert }x^{m-r_{M}\left(  F\right)  }}_{=0}\\
&  =\sum_{\substack{F\subseteq E;\\B\subseteq F}}\left(  -1\right)
^{\left\vert F\right\vert }x^{m-r_{M}\left(  F\right)  }%
=0\ \ \ \ \ \ \ \ \ \ \left(  \text{by (\ref{pf.thm.dahwil.matroid.0=}%
)}\right)  .
\end{align*}
This proves Theorem \ref{thm.dahwil.matroid}.
\end{proof}
\end{verlong}

\end{document}